\documentclass[a4paper,12pt,reqno, english]{smfbook}
\usepackage{amssymb,amsmath,mathrsfs,graphics,graphicx,mathptm,float,lscape,latexsym,pstricks,multind,multicol,stmaryrd,mathdots,xfrac,latexsym,faktor}
\usepackage[all]{xy}
\usepackage[a4paper,text ={13cm, 22cm}, centering, top=3cm]{geometry}    
\usepackage[utf8]{inputenc}

\theoremstyle{plain}
\newtheorem*{thmintro}{Theorem}
\newtheorem{thm}{Theorem}[chapter]
\newtheorem{pro}[thm]{Proposition}
\newtheorem{lem}[thm]{Lemma}
\newtheorem{cor}[thm]{Corollary}
\newtheorem{claim}[thm]{Claim}

\theoremstyle{definition}
\newtheorem*{defi}{Definition}
\newtheorem*{defis}{Definitions}
\newtheorem{eg}{Example}
\newtheorem{egs}{Examples}
\newtheorem{rem}[thm]{Remark}
\newtheorem{rems}[thm]{Remarks}

\def\og{\leavevmode\raise.3ex\hbox{$\scriptscriptstyle\langle\!\langle$~}}
\def\fg{\leavevmode\raise.3ex\hbox{~$\!\scriptscriptstyle\,\rangle\!\rangle$}}

\addtocounter{section}{0}             
\numberwithin{equation}{section}       

\newsavebox{\fmbox}
\newenvironment{fmpage}[1]
     {\begin{minipage}{#1}}
     {\end{minipage}}

\makeindex{defi}
\makeindex{not}

\begin{document}

\title{The \textsc{Cremona} group and its subgroups}

\author{Julie \textsc{D\'eserti}}

\address{Universit\'e C\^ote d'Azur, CNRS, LJAD, France}
\email{deserti@math.cnrs.fr}
\thanks{The author was partially supported by the ANR grant Fatou ANR-17-CE40-
0002-01 and the ANR grant Foliage ANR-16-CE40-0008-01.} 

\maketitle

\newpage

\begin{abstract}
We give an extensive introduction to the current
literature on the \textsc{Cremona} groups 
over the field of complex numbers, mostly
of rank $2$, with an emphasis on group 
theoretical and dynamical questions. 

After a short
introduction which explains in an informal 
style some selected results and techniques 
Chapter \ref{chap:hyperbolicspace} gives
a description of the hyperbolic space on 
which the \textsc{Cremona} group in two 
variables acts, and which has turned out 
to provide some of the key techniques 
to understand the plane \textsc{Cremona}
group. In Chapter \ref{Chapter:algebraicsubgroup}
the \textsc{Zariski} topology is described.
Chapter \ref{chapter:gen} gives an 
overview of various presentations of the 
plane \textsc{Cremona} group. Chapter
\ref{chapter:alg} treats some group 
theoretical properties of the plane 
\textsc{Cremona} group. Chapter 
\ref{chapter:finite} surveys some results 
about finite (mostly abelian) subgroups of the 
plane \textsc{Cremona} group. Chapter 
\ref{chapter:uncountable} surveys results
about various subgroups using techniques 
that rely on the base-field being uncountable.
Chapter \ref{chap:hyper} gives a big
variety of important results that can be
deduce from the action of the plane 
\textsc{Cremona} group on the hyperbolic 
space, such as the \textsc{Tits} alternative
or the non-simplicity of the group. 
Chapter \ref{chapter:dyn} gives an 
introduction to some notions from dynamics and 
their relationship to the plane 
\textsc{Cremona} group.
\end{abstract}

\setcounter{page}{4}

\newpage 

\vspace*{10cm}

\begin{flushright}
\begin{Large}
\textsl{\`A Beno\^it}
\end{Large}
\end{flushright}

\newpage

\frontmatter

\tableofcontents

%

\chapter*{Preface}

The main purpose of the present treatise
is to draw a portrait of the $n$-dimensional
Cremona group 
$\mathrm{Bir}(\mathbb{P}^n_\mathbb{C})$. 
The study of 
this group started in the XIXth century;
the subject has known a lot of developments
since the beginning of the XXIth century.
Old and new results are discussed; 
unfortunately we will not be exhaustive. 
The Cremona group is approached 
through the study of its subgroups: 
algebraic, finite, normal, nilpotent,
simple, torsion subgroups are 
evoked but also centralizers of elements,
representation of lattices, subgroups
of automorphisms of positive entropy etc

\bigskip

Let us introduce birational self maps of the 
plane and the plane Cremona group from a geometrical 
point of view.

A plane collineation is a one-to-one map 
from $\mathbb{P}^2_\mathbb{C}$ to itself
such that the images of collinear points
are themselves collinear. Such maps 
leave the projective properties of curves
unaltered. In advancing beyond such properties
let us introduce other maps of the plane to
itself that establish relations between 
curves of differents orders and possessing 
different sets of singularities. The most 
general rational map of the plane is 
defined by equations of the form
\[
\phi\colon(z_0:z_1:z_2)\dashrightarrow\big(\phi_0(z_0,z_1,z_2):\phi_1(z_0,z_1,z_2):\phi_2(z_0,z_1,z_2)\big)
\]
where $\phi_0$, $\phi_1$ and $\phi_2$ are 
homogeneous polynomials of degree $n$
without common factor of positive degree.
Such a map makes correspond to a point 
$p$ with coordinates $(p_0:p_1:p_2)$ a 
point $\phi(p)=q$ with coordinates 
$(q_0:q_1:q_2)$ where 
\begin{equation}\label{eq:crem}
\delta q_0=\phi_0(p_0,p_1,p_2),\quad
\delta q_1=\phi_1(p_0,p_1,p_2),\quad
\delta q_2=\phi_2(p_0,p_1,p_2)
\end{equation}
with $\delta$ in $\mathbb{C}^*$.

Consider the net of curves $\Lambda_\phi$
defined by the equation
\[
\alpha\phi_0+\beta\phi_1+\gamma\phi_2=0
\]
where $\alpha$, $\beta$ and $\gamma$ 
are arbitrary parameters. As $p$ 
describes a line in $\mathbb{P}^2_\mathbb{C}$,
then $q=\phi(p)$ describes a curve 
$\mathcal{C}$ of~$\Lambda_\phi$. The 
curves of the net $\Lambda_\phi$ are thus
correlated by $\phi$ with the lines 
of the plane. Conversely given any 
net $\Lambda$ of curves such as 
$\Lambda_\phi$ a linear representation
of the curves of $\Lambda$ on the lines
of the plane is equivalent to a rational
map of the plane.

The curves of $\Lambda_\phi$ may have
base-points $p_i$ common to them all.
Each such point is a common zero of 
$\phi_0$, $\phi_1$ and $\phi_2$, so the
equations (\ref{eq:crem}) to determine
its corresponding point are illusory. 
Conversely each point, termed a 
\textsl{base-point} of $\phi$, which 
renders equation (\ref{eq:crem}) 
illusory is a base-point of 
$\Lambda_\phi$. In other words

\begin{thmintro}
The base-points of any rational map 
are the base-points of the associated
net of curves.
\end{thmintro}

Any two general curves $\mathcal{C}$ 
and $\mathcal{C}'$ of $\Lambda_\phi$ 
define a pencil of curves 
$\mathcal{C}+\alpha\mathcal{C}'$
of the net. Denote by $n$ the number
of free intersections of $\mathcal{C}$
and $\mathcal{C}'$ not occuring at the 
base-points $p_i$ of $\Lambda_\phi$; denote by $r_1$, 
$r_2$, $\ldots$, $r_n$ these points. 
The integer $n$ is called the 
\textsl{grade} of $\Lambda_\phi$.

To curves of the arbitrary pencil 
$\mathcal{C}+\alpha\mathcal{C}'$
there correspond by the map $\phi$ 
lines of a pencil $L+\alpha L'$. 
Furthermore if the base-point of the 
latter pencil is $q$, then clearly 
every point $r_i$ corresponds to $q$.
Conversely if any two points of the 
plane have the same preimage $q$, 
then they belong to the same free
intersection set of some pencil 
in $\Lambda_\phi$.

\begin{thmintro}
Let $\phi$ be a rational self map
of the plane. Let $\Lambda_\phi$ 
be its associated net and let $n$ 
be the grade of $\Lambda_\phi$. 
An arbitrary point $q$ is the 
transform of $n$ points $r_1$, 
$r_2$, $\ldots$, $r_n$ which 
together form the free intersection
set of a pencil of curves of 
$\Lambda_\phi$. 
\end{thmintro}

In other words the general rational
map of the plane is a $(n,1)$ 
correspondence between the points 
$p$ and $q$. And this means that, 
when the ratios of $q_0$, $q_1$, $q_2$
are given the equations 
(\ref{eq:crem}) have in general $n$ 
distinct solutions for the ratios
of $p_0$, $p_1$ and $p_2$. If 
$n=1$, {\it i.e.} if these equations
have only one solution, $(p_0:p_1:p_2)$
are rational functions of $(q_0:q_1:q_2)$.
In this case the equations of the reverse
map will be of the form 
\begin{align*}
& \alpha p_0=\psi_0(q_0,q_1,q_2)&& \alpha p_1=\psi_1(q_0,q_1,q_2)&& \alpha p_2=\psi_2(q_0,q_1,q_2)
\end{align*}
where $\psi_0$, $\psi_1$ and $\psi_2$ 
are homogeneous polynomials of degree
$n'$.
A \textsl{Cremona map}\index{defi}{Cremona map} 
is a rational map 
whose reverse is also rational, we 
also speak about birational self map
of the plane. The \textsl{plane
Cremona group}\index{defi}{plane Cremona group} 
is the group of birational self maps of the plane.

A \textsl{homaloidal net} of curves 
in the plane is one whose grade is 
$1$. 

Equations (\ref{eq:crem}) define a 
birational map $\phi$ if 
and only if the associated net~$\Lambda_\phi$ is homaloidal. 
Conversely from any given 
homaloidal net we can derive many
birational self maps of the plane; if 
$\overline{\phi_0}$, $\overline{\phi_1}$
and $\overline{\phi_2}$ 
are three independent linear 
combinations of $\phi_0$, $\phi_1$ 
and $\phi_2$, the net 
\[
\alpha\,\phi_0+\beta\,\phi_1+\gamma\,\phi_2=0
\]
can also be expressed in the form
\[
\alpha'\,\overline{\phi_0}+\beta'\,\overline{\phi_1}+\gamma'\,\overline{\phi_2}=0
\]
and the map defined by 
\[
(z_0:z_1:z_2)\dashrightarrow\big(\overline{\phi_0}(z_0,z_1,z_2):\overline{\phi_1}(z_0,z_1,z_2):\overline{\phi_2}(z_0,z_1,z_2)\big)
\]
is based on the same net. Moreover

\begin{thmintro}
To any birational self map of the plane there is 
associated a homaloidal net of 
curves. 

Conversely any homaloidal net of curves
generates an infinity of birational 
self maps of the plane, any of which 
is the product of 
any other by a plane collineation.
\end{thmintro}

A collineation is the simplest kind of
birational self map of the plane whose 
homaloidal net is composed of the lines 
of the plane.

The \textsl{degree} of a birational
self map of the plane is the degree 
of the curves of its 
generating homaloidal net.

Let $\phi$ be a birational self map
of the plane of degree $n$. Denote by $n'$ 
the degree of its inverse~$\phi^{-1}$. If the 
number of intersections of two 
curves $\mathcal{C}$ and $\mathcal{C}'$
is denoted by $\mathcal{C}\cdot\mathcal{C}'$
and if~$L$ and $L'$ are lines, then 
\[
n=L\cdot\Lambda_\phi=\phi(L)\cdot\phi(\Lambda_\phi)=\Lambda_{\phi^{-1}}\cdot L'=n'.
\]
Hence 
\begin{thmintro}
A birational self map of the plane and its
inverse have the same degree.
\end{thmintro}

Let us finish this introduction by pointing
out that this statement is not true in higher 
dimension: 
\begin{align*}
&\mathbb{P}^3_\mathbb{C}\dashrightarrow\mathbb{P}^3_\mathbb{C} && (z_0:z_1:z_2:z_3)\dashrightarrow(z_0^2:z_0z_1:z_1z_2:z_0z_3-z_1^2)
\end{align*}
is a birational self map of 
$\mathbb{P}^3_\mathbb{C}$ of 
degree $2$ whose inverse  
\begin{align*}
&\mathbb{P}^3_\mathbb{C}\dashrightarrow\mathbb{P}^3_\mathbb{C} && (z_0:z_1:z_2:z_3)\dashrightarrow \big(z_0^2z_1:z_0z_1^2:z_0^2z_2:z_1(z_0z_3+z_1^2)\big)
\end{align*}
has degree $3$. As we will
see there are many other differences between 
the $2$-dimensional Cremona 
group and the $n$-dimensional Cremona
group, $n\geq 3$. 

\smallskip

Note that the study of 
$\mathrm{Bir}(\mathbb{P}^2_\mathbb{C})$ is
central: if $S$ is a complex rational surface, 
then its group of birational self maps is 
isomorphic to
$\mathrm{Bir}(\mathbb{P}^2_\mathbb{C})$.

\bigskip

\bigskip

We now deal with the content of the manuscript.
Chapter \ref{chapter:intro} contains introductory 
examples and the very basic techniques used to 
study birational maps of the projective plane.
This chapter explains in particular the importance
of divisors and linear systems in the study of the 
plane Cremona groups.

\smallskip

Chapter \ref{chap:hyperbolicspace} builds up on Chapter
\ref{chapter:intro} by explaining how to blow-up all 
points in~$\mathbb{P}^2_\mathbb{C}$ and subsequent 
blown-up surfaces. It gives rise to an infinite
hyperbolic space on which the Cremona group
acts. This space plays a fundamental role in the 
study of Cremona groups, as it allows to 
apply tools from geometric group theory to 
study subgroups of the Cremona group,
as well as degree growth and dynamical 
behaviours of birational maps.

\smallskip

Chapter \ref{Chapter:algebraicsubgroup} presents 
two natural topologies on the Cremona group
and their properties, and the notion of algebraic 
subgroups of the Cremona groups. The 
construction of one of the topologies - the 
Zariski topology - is defined via the 
concept of morphisms. It links to the concept 
of an algebraic group acting on a variety, 
which is discussed in this chapter as well.

\smallskip

Chapter \ref{chapter:gen} adresses a very basic
and classical interest while dealing with a 
group: finding a "nice" and generating set and 
"nice" structures of the group, such as an 
amalgamated structure. This is quite an 
important topic in research on Cremona
groups because for the plane Cremona
group there are "nice" generating sets, and 
many statements are proven by using them. In higher
dimensions no nice generating sets are known: 
this is one of the many reasons why working with
Cremona groups in higher dimensions is 
very hard.

\smallskip

Chapter \ref{chapter:alg} discusses other group 
geometric properties of plane Cremona
groups. While Chapter \ref{chap:hyperbolicspace}
presents a representation of the Cremona
group in terms of isometries of an infinite
hyperbolic space this chapter deals with linear
representations (there are none) and representations
of subgroups of $\mathrm{SL}(n,\mathbb{Z})$,
$n\geq 3$, inside the plane Cremona group.

\smallskip

Chapter \ref{chapter:finite} deals with results on finite
subgroups of the plane Cremona groups. 
They have been of much interest for a very long time, 
and a short overview of the progress made in the last 
$80$ years is given. The chapter focuses on the 
classification results of finite abelian and finite 
cyclic subgroups by Blanc and 
Dolgachev and Iskovskikh.

\smallskip

Chapter \ref{chapter:uncountable} is an extension 
of Chapter \ref{chapter:finite}; it deals with 
infinite abelian subgroups of the plane 
Cremona group. It then moves on the 
related topic of endomorphisms of Cremona
groups, subject already mentioned in 
Chapter \ref{chapter:alg}.

\smallskip

Chapter \ref{chap:hyper} picks up the topic of 
Chapter \ref{chap:hyperbolicspace} which is the action
of the plane Cremona group on an infinite
hyperbolic space by isometries. The action and its
properties have been very fruitful and has played 
a vital role in many recent results on the plane
Cremona group. 

\smallskip

Chapter \ref{chapter:dyn} has a more dynamical 
flavour. We first give three answers to the 
question "when is a birational self map 
of $\mathbb{P}^2_\mathbb{C}$ birationally 
conjugate to an automorphism ?" We then 
recall some constructions of automorphisms
of rational surfaces with positive entropy.
And then we 
realize $\mathrm{SL}(2,\mathbb{Z})$ as a subgroup
of automorphisms of a rational surface with the 
property that every element of infinite order 
has positive entropy.

\bigskip

\textbf{Acknowledgments}

I would like to thank all people working on the 
Cremona groups and around. 
I have a very special thought for 
Serge Cantat and Dominique Cerveau 
who made me discover this wonderful 
group.

I am grateful to the referees for very constructive recommendations.

Enfin merci \`a mon petit univers...


\mainmatter

\chapter{Introduction}\label{chapter:intro}

\bigskip
\bigskip

This chapter is devoted to recalls and first definitions.

In the first section morphisms
between varieties, blow-ups, Cremona groups
and bubble space are introduced, the 
Zariski theorem, base-points, indeterminacy
points are recalled, ans examples of subgroups of the 
Cremona group are given, among them the group
of automorphisms of $\mathbb{P}^n_\mathbb{C}$, 
the Jonqui\`eres group, the group of 
monomial maps. 

The second section is devoted to divisors (prime
divisors, Weil divisors, principal 
divisors, Picard group) and 
intersection theory.

The third section deals with a geometric definition
of birational maps of the complex projective plane.

\bigskip
\bigskip


\section{First definitions and examples}\label{sec:firstdef}

Denote by $\mathbb{P}^n_\mathbb{C}$ the complex projective space of dimension $n$.
A \textsl{rational map}\index{defi}{rational map} 
\[
\phi\colon V_1\subset\mathbb{P}^n_\mathbb{C}\dashrightarrow V_2\subset\mathbb{P}^k_\mathbb{C}
\]
between two smooth projective complex varieties $V_1$ and $V_2$ is
a regular map on a non-empty Zariski open subset of $V_1$ such 
that the image of the points where $\phi$ is well defined is 
contained in $V_2$. If $\phi$ is well defined on $V_1$ we say
that $\phi$ is a \textsl{morphism}\index{defi}{morphism (between two varieties)} or a
\textsl{regular map}\index{defi}{regular (map)}, otherwise we 
denote by $\mathrm{Ind}(\phi)$\index{not}{$\mathrm{Ind}(\phi)$} 
the set where $\phi$ is not defined, and call it the 
\textsl{indeterminacy set}\index{defi}{indeterminacy set (of a rational map)}
of $\phi$. A \textsl{birational map}\index{defi}{birational map} 
between $V_1$ and $V_2$ is a rational map that admits an inverse
which is rational. In other words it is an isomorphism between 
two non-empty Zariski open subsets of $V_1$ and $V_2$.

\smallskip

\begin{eg}
Let us give an example of a birational morphism. 
Let $p$ be a point on a smooth algebraic surface $S$. We say that 
$\pi\colon Y\to S$ is a \textsl{blow-up}\index{defi}{blow-up of a point}
of $p$ if 
\begin{itemize}
\item[$\diamond$] $Y$ is a smooth surface,

\item[$\diamond$] $\pi_{\vert Y\smallsetminus\{\pi^{-1}(p)\}}\colon Y\smallsetminus\{\pi^{-1}(p)\}\to S\smallsetminus \{p\}$ is an isomorphism,

\item[$\diamond$] and $\pi^{-1}(p)\simeq\mathbb{P}^1_\mathbb{C}$.
\end{itemize}
We call $\pi^{-1}(p)$ the \textsl{exceptional divisor}\index{defi}{exceptional (divisor)}.

If $\pi\colon Y\to S$ and $\pi'\colon Y'\to S$ are two blow-ups of the
same point $p$, then there exists an isomorphism $\varphi\colon Y\to Y'$
such that $\pi=\pi'\circ\varphi$. We can thus speak about \textsl{the}
blow-up of $p\in S$. 

Let us describe the blow-up of $(0:0:1)$ in $\mathbb{P}^2_\mathbb{C}$
endowed with the homogeneous coordinates $(z_0:z_1:z_2)$. Consider
the affine chart $z_2=1$, {\it i.e.} let us work in $\mathbb{C}^2$ with 
coordinates $(z_0,z_1)$. Set
\[
V=\big\{\big((z_0,z_1),(u:v)\big)\in\mathbb{C}^2\times\mathbb{P}^1_\mathbb{C}\,\vert\, z_0v=z_1u\big\}.
\]
Let $\pi\colon V\to\mathbb{C}^2$ 
be the morphism given by the first projection. Then 
\begin{itemize}
\item[$\diamond$] $\pi^{-1}(0,0)=\big\{\big((0,0),(u:v)\big)\,\vert\,(u:v)\in\mathbb{P}^1_\mathbb{C}\big\}$, 
so $\pi^{-1}(0,0)\simeq\mathbb{P}^1_\mathbb{C}$;

\item[$\diamond$] if $p=(z_0,z_1)$ is a point of 
$\mathbb{C}^2\smallsetminus\{(0,0)\}$, then 
\[
\pi^{-1}(p)=\big\{((z_0,z_1),(z_0:z_1))\big\}\in V\smallsetminus\{\pi^{-1}(0,0)\},
\]
and $\pi_{\vert V\smallsetminus\{\pi^{-1}(0,0)\}}$ is an isomorphism, the inverse being 
\[
(z_0,z_1)\mapsto\big((z_0,z_1),(z_0:z_1)\big).
\]
In other words $V=\mathrm{Bl}_{(0,0)}\mathbb{P}^2_\mathbb{C}$ is the 
surface obtained by blowing up the complex projective plane 
at $(0:0:1)$, $\pi$ is the blow up of $(0:0:1)$, and 
$\pi^{-1}(0,0)$ is the exceptional divisor.
\end{itemize}
\end{eg}

\smallskip

Let $V$ be a complex algebraic variety, and let $\mathrm{Bir}(V)$ be 
the group of birational maps of $V$. The group 
$\mathrm{Bir}(\mathbb{P}^n_\mathbb{C})$\index{not}{$\mathrm{Bir}(\mathbb{P}^n_\mathbb{C})$}
is called the \textsl{Cremona group}\index{defi}{Cremona group}. 
If we fix homogeneous coordinates $(z_0:z_1:\ldots:z_n)$ of 
$\mathbb{P}^n_\mathbb{C}$ every element $\phi\in\mathrm{Bir}(\mathbb{P}^n_\mathbb{C})$
can be described by homogeneous polynomials of the same degree $\phi_0$, 
$\phi_1$, $\ldots$, $\phi_n\in\mathbb{C}[z_0,z_1,\ldots,z_n]$ without 
common factor of positive degree:
\begin{small}
\[
\phi\colon(z_0:z_1:\ldots:z_n)\dashrightarrow\big(\phi_0(z_0,z_1,z_2,\ldots,z_n):\phi_1(z_0,z_1,z_2,\ldots,z_n):\ldots:\phi_n(z_0,z_1,z_2,\ldots,z_n)\big).
\]
\end{small}
The \textsl{degree}\index{defi}{degree (of a rational map)} of $\phi$ is
the degree of the $\phi_i$'s. In the affine chart $z_0=1$, the map $\phi$
is given by $(\varphi_1,\varphi_2,\ldots,\varphi_n)$ where for any
$1\leq i\leq n$
\[
\varphi_i=\frac{\phi_i(1,z_1,z_2,\ldots,z_n)}{\phi_0(1,z_1,z_2,\ldots,z_n)}
\in\mathbb{C}(z_1,z_2,\ldots,z_n).
\]

The subgroup of $\mathrm{Bir}(\mathbb{P}^n_\mathbb{C})$ consisting of elements 
$\phi$ such that all the $\varphi_i$ are polynomials as well as the entries of
$\phi^{-1}$ is exactly the \textsl{group} 
$\mathrm{Aut}(\mathbb{A}^n_\mathbb{C})$\index{not}{$\mathrm{Aut}(\mathbb{A}^n_\mathbb{C})$}
\textsl{of polynomial automorphisms of the affine space} $\mathbb{A}^n_\mathbb{C}$
\index{defi}{polynomial automorphisms of the affine space}.

\medskip

Let $S$ be a smooth projective surface. The
\textsl{bubble space}\index{defi}{bubble space}
$\mathcal{B}(S)$\index{not}{$\mathcal{B}(S)$} is, roughly 
speaking, the set of all points that belong to $S$, or are
infinitely near to $S$. Let us be more precise: consider all
surfaces $Y$ \textsl{above}\index{defi}{above (surface)} $S$, 
{\it i.e.} all birational morphisms
$\pi\colon Y\to S$ ; we identify $p_1\in Y_1$ and $p_2\in Y_2$
if $\pi_1^{-1}\circ\pi_2$ is a local isomorphism in a 
neighborhood of $p_2$ that maps $p_2$ onto $p_1$. The bubble
space $\mathcal{B}(S)$ is the union of all points of all 
surfaces above $S$ modulo the equivalence relation generated
by these identifications. A point $p\in\mathcal{B}(S)\cap S$ 
is a \textsl{proper point}\index{defi}{proper (point)}.
All points in $\mathcal{B}(S)$ that are not proper are 
called \textsl{infinitely near}\index{defi}{infinitely near (point)}.

Let $S$ and $S'$ be two smooth projective surfaces. Let 
$\phi\colon S\dashrightarrow S'$ be a birational map. By Zariski's 
theorem (\emph{see for instance} \cite{Beauville:book}) we can 
write $\phi=\pi_2\circ\pi_1^{-1}$ where 
$\pi_1\colon Y\to S$ and $\pi_2\colon Y\to S'$ are finite sequences 
of blow-ups. We may assume that there is 
no $(-1)$-curve in $Y$ contracted by both $\pi_1$ and $\pi_2$.
We then say that $\pi_2\circ\pi_1^{-1}$ is a 
\textsl{minimal resolution}\index{defi}{minimal (resolution)} of $\phi$. 
The \textsl{base-points}\index{defi}{base-point (birational map)}
$\mathrm{Base}(\phi)$\index{not}{$\mathrm{Base}(\phi)$} of $\phi$
are the points blown up by $\pi_1$. The proper base-points of 
$\phi$ are precisely the
\textsl{indeterminacy points}\index{defi}{indeterminacy points (of a birational map)}
of $\phi$.

A birational morphism $\pi\colon S\to S'$ induces a bijection 
$\pi_\bullet\colon\mathcal{B}(S)\to\mathcal{B}(S')\smallsetminus\mathrm{Base}(\pi^{-1})$.
A birational map of smooth projective surfaces 
$\phi\colon S\dashrightarrow S'$ induces a bijection 
\[
\phi_\bullet\colon\mathcal{B}(S)\smallsetminus\mathrm{Base}(\phi)\to\mathcal{B}(S')\smallsetminus\mathrm{Base}(\phi^{-1})
\]\index{not}{$\phi_\bullet$}
by $\phi_\bullet=(\pi_2)_\bullet\circ(\pi_1)_\bullet^{-1}$ 
where $\pi_2\circ\pi_1^{-1}$ is a minimal resolution of $\phi$.

\medskip

Let us now give some subgroups of the Cremona group:

\begin{itemize}
\item First consider the automorphism group of $\mathbb{P}^n_\mathbb{C}$. 
It is the subgroup formed by \textsl{regular maps}, {\it i.e.}
maps well defined
on $\mathbb{P}^n_\mathbb{C}$ and whose inverse is also well defined 
on $\mathbb{P}^n_\mathbb{C}$:
\[
\mathrm{Aut}(\mathbb{P}^n_\mathbb{C})=\big\{\phi\in\mathrm{Bir}(\mathbb{P}^n_\mathbb{C})\,\vert\,\mathrm{Base}(\phi)=\mathrm{Base}(\phi^{-1})=\emptyset\big\}.
\]
To any 
$M=\big(a_{i,j}\big)_{0\leq i,\,j\leq n}\in\mathrm{PGL}(n+1,\mathbb{C})$
corresponds an element of $\mathrm{Bir}(\mathbb{P}^n_\mathbb{C})$ of 
degree~$1$:
\[
(z_0:z_1:\ldots:z_n)\mapsto\left(\displaystyle\sum_{j=0}^na_{0,j}z_j:\displaystyle\sum_{j=0}^na_{1,j}z_j:\ldots:\displaystyle\sum_{j=0}^na_{n,j}z_j\right)
\]
and vice-versa. Such elements are biregular. Furthermore Bezout theorem
implies that all biregular maps are linear. We thus have the following 
isomorphism
\[
\mathrm{Aut}(\mathbb{P}^n_\mathbb{C})\simeq\mathrm{PGL}(n+1,\mathbb{C}).
\]

\item The $n$-dimensional subgroup of $\mathrm{Aut}(\mathbb{P}^n_\mathbb{C})$ 
consisting of diagonal automorphisms is denoted by 
$\mathrm{D}_n$\index{not}{$\mathrm{D}_n$}. Note that $\mathrm{D}_n$ 
is the torus of highest rank of 
$\mathrm{Bir}(\mathbb{P}^n_\mathbb{C})$\footnote{indeed according to
\cite{BialynickiBirula} if $\mathrm{G}$ is an algebraic 
subgroup (\emph{see} Chapter \ref{Chapter:algebraicsubgroup} for a definition) of
$\mathrm{Bir}(\mathbb{P}^n_\mathbb{C})$ isomorphic to 
$(\mathbb{C}^*)^k$, then $k\leq n$, and if $k=n$, then $\mathrm{G}$ is
conjugate to $\mathrm{D}_n$.}.

\medskip

\item Start with the surface 
$\mathbb{P}^1_\mathbb{C}\times\mathbb{P}^1_\mathbb{C}$
considered as a smooth quadric in $\mathbb{P}^3_\mathbb{C}$; its 
automorphism group contains 
$\mathrm{PGL}(2,\mathbb{C})\times\mathrm{PGL}(2,\mathbb{C})$. By the 
stereographic projection the quadric is birationally equivalent to 
the plane, so that $\mathrm{Bir}(\mathbb{P}^2_\mathbb{C})$ contains 
also a copy of $\mathrm{PGL}(2,\mathbb{C})\times\mathrm{PGL}(2,\mathbb{C})$.

If $\mathrm{G}$ is a semi-simple algebraic group, $\mathrm{H}$ is a parabolic 
subgroup of $\mathrm{G}$, and $V=\faktor{\mathrm{G}}{\mathrm{H}}$ is a homogeneous variety 
of dimension $n$, then $V$ is rational. Once a birational map 
$\pi\colon V\dashrightarrow\mathbb{P}^n_\mathbb{C}$ is given, 
$\pi\circ\mathrm{G}\circ\pi^{-1}$ determines an algebraic subgroup of 
$\mathrm{Bir}(\mathbb{P}^n_\mathbb{C})$.

\medskip

\item A \textsl{fibration}\index{defi}{fibration} of a surface $S$ is a 
rational map $\pi\colon S\dashrightarrow C$, where $C$ is a curve, 
such that the general fibers are one-dimensional.
Two fibrations $\pi_1\colon S\dashrightarrow C$ and 
$\pi_2\colon S\dashrightarrow C'$ are identified if there exists
an open dense subset $\mathcal{U}\subset S$ that is contained in 
the domains of $\pi_1$ and $\pi_2$ such that 
$\pi_{1\vert\mathcal{U}}$ and $\pi_{2\vert\mathcal{U}}$ 
define the same set of 
fibers. We say that a group $\mathrm{G}$ 
\textsl{preserves}\index{defi}{preserved (fibration)} a 
fibration $\pi$ if $\mathrm{G}$ permutes the fibers. 
A \textsl{rational fibration}\index{defi}{rational fibration}
of a rational surface $S$ is a rational map 
$\pi\colon S\dashrightarrow\mathbb{P}^1_\mathbb{C}$ such that
the general fiber is rational. The following statement due to
Noether and Enriques says that, up to birational maps, there 
exists only one rational fibration of $\mathbb{P}^2_\mathbb{C}$:

\begin{thm}[\cite{Beauville:book}]\label{thm:fibration}
Let $S$ be a surface.
Let $\pi\colon S\dashrightarrow C$ be a rational fibration. Then
there exists a birational map 
$\phi\colon C\times\mathbb{P}^1_\mathbb{C}\dashrightarrow S$
such that $\pi\circ\phi$ is the projection onto the first
factor.
\end{thm}

The \textsl{Jonqui\`eres subgroup}\index{defi}{Jonqui\`eres group}
$\mathcal{J}$\index{not}{$\mathcal{J}$} of $\mathrm{Bir}(\mathbb{P}^2_\mathbb{C})$ is
the subgroup of elements that preserve the pencil of lines 
through the point $(0:0:1)\in\mathbb{P}^2_\mathbb{C}$. 

Any subgroup of 
$\mathrm{Bir}(\mathbb{P}^2_\mathbb{C})$ that preserves a 
rational fibration is conjugate to a subgroup of $\mathcal{J}$
(Theorem \ref{thm:fibration}).

With respect to affine coordinates $(z_0:z_1:1)$ an element of 
$\mathcal{J}$ is of the form
\[
(z_0,z_1)\dashrightarrow\left(\frac{\alpha z_0+\beta}{\gamma z_0+\delta},\frac{A(z_0)z_1+B(z_0)}{C(z_0)z_1+D(z_0)}\right)
\]
where $\left(\begin{array}{cc}\alpha & \beta \\ \gamma & \delta\end{array}\right)$ belongs to $\mathrm{PGL}(2,\mathbb{C})$ and $\left(\begin{array}{cc} A & B\\ C & D\end{array}\right)$ to $\mathrm{PGL}(2,\mathbb{C}(z_0))$. This induces an isomorphism
\[
\mathcal{J}\simeq\mathrm{PGL}(2,\mathbb{C})\rtimes\mathrm{PGL}(2,\mathbb{C}(z_0)).
\]

\medskip

\item Let $M=(a_{i,j})_{1\leq i,\,j\leq n}\in M(n,\mathbb{Z})$ be a $n\times n$
matrix of integers. The matrix $M$ determines a rational self map of 
$\mathbb{P}^n_\mathbb{C}$ given in the affine
chart $z_0=1$ by
\[
\phi_M\colon (z_1,\ldots,z_n)\mapsto\Big(z_1^{a_{1,1}}z_2^{a_{1,2}}\ldots z_n^{a_{1,n}},z_1^{a_{2,1}}z_2^{a_{2,2}}\ldots z_n^{a_{2,n}},\ldots,z_1^{a_{n,1}}z_2^{a_{n,2}}\ldots z_n^{a_{n,n}}\Big).
\]
The map $\phi_M$ is birational if and only if $M$ belongs
to $\mathrm{GL}(n,\mathbb{Z})$. This yields an injective 
homomorphism 
$\mathrm{GL}(n,\mathbb{Z})\to\mathrm{Bir}(\mathbb{P}^n_\mathbb{C})$
whose image is called the 
\textsl{group of monomial maps}\index{defi}{monomial maps}
and is denoted $\mathrm{Mon}(n,\mathbb{C})$.

\medskip

\item The well known result of Noether and Castelnuovo states that

\begin{thm}[\cite{Castelnuovo, AlberichCarraminana}]
The group
$\mathrm{Bir}(\mathbb{P}^2_\mathbb{C})$ is generated by the 
involution
\[
\sigma_2\colon(z_0:z_1:z_2)\dashrightarrow(z_1z_2:z_0z_2:z_0z_1)
\]
and the group 
$\mathrm{Aut}(\mathbb{P}^2_\mathbb{C})=\mathrm{PGL}(3,\mathbb{C})$.
\end{thm}

For $n\geq 3$ the Cremona group is not 
generated by $\mathrm{PGL}(n+1,\mathbb{C})$ and $\mathrm{Mon}(n,\mathbb{C})$
(\emph{see} \cite{Hudson, Pan:generation}).
In other words the subgroup 
\[
\langle\mathrm{PGL}(n+1,\mathbb{C}),\,\mathrm{Mon}(n,\mathbb{C})\rangle
\]
is a strict subgroup of $\mathrm{Bir}(\mathbb{P}^n_\mathbb{C})$.
The finite index subgroup of 
$\langle\mathrm{PGL}(n+1,\mathbb{C}),\,\mathrm{Mon}(n,\mathbb{C})\rangle$ 
generated by $\mathrm{PGL}(n+1,\mathbb{C})$ and the involution
\[
\sigma_n\colon(z_0:z_1:\ldots:z_n)\dashrightarrow\left(\prod_{\stackrel{i=0}{i\not=0}}^nz_i:\prod_{\stackrel{i=0}{i\not=1}}^nz_i:\ldots:\prod_{\stackrel{i=0}{i\not=n}}^nz_i\right)
\]
has been studied in \cite{BlancHeden, Deserti:reg}. The 
group 
$\mathrm{G}(n,\mathbb{C})=\langle\sigma_n,\,\mathrm{PGL}(n+1,\mathbb{C})\rangle$
"looks like" $\mathrm{G}(2,\mathbb{C})=\mathrm{Bir}(\mathbb{P}^2_\mathbb{C})$ in
the following sense (\cite{Deserti:reg}):
\begin{itemize}
\item[$\diamond$] there is no non-trivial finite dimensional 
linear representation of $\mathrm{G}(n,\mathbb{C})$ over any
field;

\item[$\diamond$] the group $\mathrm{G}(n,\mathbb{C})$ is 
perfect, {\it i.e.} 
$\big[\mathrm{G}(n,\mathbb{C}),\mathrm{G}(n,\mathbb{C})\big]=\mathrm{G}(n,\mathbb{C})$;

\item[$\diamond$] the group $\mathrm{G}(n,\mathbb{C})$ equipped
with the Zariski topology is simple;

\item[$\diamond$] let $\varphi$ be an automorphism of 
$\mathrm{Bir}(\mathbb{P}^n_\mathbb{C})$; there exist an
automorphism~$\kappa$ of the field~$\mathbb{C}$ and a 
birational self map $\psi$ of $\mathbb{P}^n_\mathbb{C}$ such
that
\[
\varphi(\phi)={}^{\kappa}\!\,(\psi\circ\phi\circ\psi^{-1})\qquad\qquad \forall\,\phi\in\mathrm{G}(n,\mathbb{C}).
\]
\end{itemize}

We will deal with
\begin{itemize}
\item[$\diamond$] the Noether and Castelnuovo theorem in 
\S \ref{subsection:nc1} and \S \ref{subsection:nc2};

\item[$\diamond$] the Hudson and Pan theorem in \S \ref{subsection:hudsonandpan};

\item[$\diamond$] the fact that there is no non-trivial
finite dimensional linear representation of 
$\mathrm{G}(2,\mathbb{C})$ over any field in 
\S \ref{section:notlinear};

\item[$\diamond$] the fact that $\mathrm{Bir}(\mathbb{P}^2_\mathbb{C})=\mathrm{G}(2,\mathbb{C})$ is perfect in \S \ref{section:perfect};

\item[$\diamond$] the fact that $\mathrm{Bir}(\mathbb{P}^2_\mathbb{C})=\mathrm{G}(2,\mathbb{C})$ equipped with the Zariski topology is simple in \S \ref{section:closednormalsubgroups};

\item[$\diamond$] the description of $\mathrm{Aut}(\mathrm{Bir}(\mathbb{P}^2_\mathbb{C}))=\mathrm{Aut}(\mathrm{G}(2,\mathbb{C}))$ 
in \S \ref{sec:uncountableabelian}.
\end{itemize}
\end{itemize}


\section{Divisors and intersection theory}

Let $V$ be an algebraic variety.

A \textsl{prime divisor}\index{defi}{prime (divisor)}
on $V$ is an irreducible closed subset of $V$ of codimension $1$.
For instance if~$V$ is a surface, then the prime divisors of $V$
are the irreducible curves that lie on it; if $V$ is the complex
projective space, then the prime divisors are given by the zeros
locus of irreducible homogeneous polynomials.

A \textsl{Weil divisor}\index{defi}{Weil divisor} on $V$ is a 
formal finite sum of prime divisors with integer coefficients:
\[
\displaystyle\sum_{i=1}^ma_iD_i \qquad\qquad m\in\mathbb{N},\,a_i\in\mathbb{Z},\,D_i\text{ prime divisor of $V$}.
\]
Let us denote by $\mathrm{Div}(V)$\index{not}{$\mathrm{Div}(V)$}
the set of all Weil divisors of $V$. 

Let $f\in\mathbb{C}(V)^*$ be a rational function, and let $D$
be a prime divisor. The \textsl{multiplicity}\index{defi}{multiplicity (of a rational function at a prime divisor)}
$\nu_f(D)$\index{not}{$\nu_f(D)$} of $f$ at $D$ is defined by
\begin{itemize}
\item[$\diamond$] $\nu_f(D)=k>0$ if $f$ vanishes on $D$ at the order $k$;

\item[$\diamond$] $\nu_f(D)=-k$ if $f$ has a pole of order $k$ on $D$;

\item[$\diamond$] $\nu_f(D)=0$ otherwise.
\end{itemize}
To any rational function $f\in\mathbb{C}(V)^*$ we associate a 
divisor $\mathrm{div}(f)$\index{not}{$\mathrm{div}(f)$} defined by
\[
\mathrm{div}(f)=\displaystyle\sum_{\stackrel{\text{$D$ prime}}{\text{divisor}}}\nu_f(D)D.
\]

Since $\nu_f(D)$ is zero for all but finitely many $D$ the divisor
$\mathrm{div}(f)$ belongs to $\mathrm{Div}(V)$. Divisors obtained 
like that are called \textsl{principal divisors}\index{defi}{principal (divisor)}.
The set of principal divisors form a subgroup of $\mathrm{Div}(V)$;
indeed $\mathrm{div}(fg)=\mathrm{div}(f)+\mathrm{div}(g)$ for any 
$f$, $g\in\mathbb{C}(V)^*$.

Let us introduce an equivalence relation on $\mathrm{Div}(V)$. Two
divisors $D$, $D'$ are \textsl{linearly equivalent}\index{defi}{linearly equivalent (divisors)}
if $D-D'$ is a principal divisor. The set of equivalence classes 
corresponds to the quotient of $\mathrm{Div}(V)$ by the subgroup
of principal divisors. The 
\textsl{Picard group}\index{defi}{Picard group} of $V$ is the 
group of isomorphism classes of line bundles on $V$; it is 
denoted $\mathrm{Pic}(V)$\index{not}{$\mathrm{Pic}(V)$}. When $V$ 
is smooth the quotient of $\mathrm{Div}(V)$ by the subgroup
of principal divisors is isomorphic to $\mathrm{Pic}(V)$.

\begin{eg}
Let us determine $\mathrm{Pic}(\mathbb{P}^n_\mathbb{C})$.
Consider the morphism of groups 
\[
\theta\colon\mathrm{Div}(\mathbb{P}^n_\mathbb{C})\to\mathbb{Z}
\]
which associates to any divisor $D$ of degree $d$ the integer $d$.
Note that $\ker\theta$ is the subgroup of principal divisors of
$\mathbb{P}^n_\mathbb{C}$: let $D=\sum a_iD_i$ be an element of $\ker\theta$ where each 
$D_i$ is a prime divisor given by an homogeneous polynomial 
$f_i\in\mathbb{C}[z_0,z_1,\ldots,z_n]$ of some degree $d_i$. 
Since $\sum a_id_i=0$, $f=\prod f_i^{a_i}$ belongs to 
$\mathbb{C}(\mathbb{P}^n_\mathbb{C})^*$. By construction 
$D=\mathrm{div}(f)$ hence $D$ is a principal divisor. 
Conversely any principal divisor is equal to $\mathrm{div}(f)$
where $f=g/h$ for some homogeneous polynomials $g$, $h$ of the 
same degree. Thus any principal divisor belongs to $\ker\theta$.

Since $\mathrm{Pic}(\mathbb{P}^n_\mathbb{C})$ is the quotient 
of $\mathrm{Div}(\mathbb{P}^n_\mathbb{C})$ by the subgroup of 
principal divisors, we get by restricting $\theta$ to the 
quotient an isomorphism between $\mathrm{Pic}(\mathbb{P}^n_\mathbb{C})$
and $\mathbb{Z}$. As an hyperplane is sent on $1$ we obtain that 
$\mathrm{Pic}(\mathbb{P}^n_\mathbb{C})=\mathbb{Z}H$ where $H$ is
the divisor of an hyperplane.
\end{eg}

Let us now assume that $\dim V=2$; set $V=S$. We can define
the notion of intersection:

\begin{pro}[\cite{Hartshorne}]
Let $S$ be a smooth projective surface. There exists a unique 
bilinear symmetric form
\begin{align*}
&\mathrm{Div}(S)\times\mathrm{Div}(S)\to\mathbb{Z}&& (C,D)\mapsto C\cdot D
\end{align*}
such that 
\begin{itemize}
\item[$\diamond$]  if $C$ and $D$ are smooth curves with transverse 
intersections, then $C\cdot D=\#(C\cap D)$;

\item[$\diamond$] if $C$ and $C'$ are linearly equivalent, then 
$C\cdot D=C'\cdot D$ for any $D$.
\end{itemize}
In particular this yields an intersection form
\begin{align*}
&\mathrm{Pic}(S)\times\mathrm{Pic}(S)\to\mathbb{Z}&& (C,D)\mapsto C\cdot D.
\end{align*}
\end{pro}

Let $\pi\colon\mathrm{Bl}_pS\to S$ be the blow-up of the point $p\in S$. 
The morphism $\pi$ induces the map 
\begin{align*}
&\pi^*\colon\mathrm{Pic}(S)\to\mathrm{Pic}(\mathrm{Bl}_pS), && C\mapsto \pi^{-1}(C).
\end{align*}
If $C$ is an irreducible curve on $S$, the 
\textsl{strict transform}\index{defi}{strict transform (of a curve)} $\widetilde{C}$
of $C$ is $\widetilde{C}=\overline{\pi^{-1}(C\smallsetminus\{p\})}$. 

If $C\subset S$ is a curve and if $p$ is a point of $S$, let us 
define the \textsl{multiplicity $m_p(C)$}\index{not}{$m_p(C)$}
\textsl{of $C$ at $p$}\index{defi}{multiplicity (of a curve at a point)}.  
Recall that if $V$ is a quasi-projective variety, and if $q$ is a point of $V$, then~$\mathcal{O}_{q,V}$ denotes the set of equivalence classes of pairs $(\mathcal{U},\varphi)$
where $\varphi$ belongs to~$\mathbb{C}[\mathcal{U}]$, and $\mathcal{U}\subset V$
is an open subset such that $q\in\mathcal{U}$. Let $\mathfrak{m}$ be the maximal ideal 
of $\mathcal{O}_{p,S}$. If $f$ is a local equation of $C$, then $m_p(C)$
is the integer $k$ such that~$f$ belongs to $\mathfrak{m}^k\smallsetminus\mathfrak{m}^{k+1}$.

\begin{eg}
Assume that $S$ is a rational surface. There exists a neighborhood~$\mathcal{U}$
of $p$ in~$S$ with $\mathcal{U}\subset\mathbb{C}^2$. We can assume that $p=(0,0)$
in this affine neighborhood and that $C$ is a curve described by the equation 
$\displaystyle\sum_{i=1}^nP_i(z_0,z_1)=0$ where $P_i$ is an homogeneous polynomial
of degree~$i$. The multiplicity $m_p(C)$ is the lowest $i$ such that $P_i$ 
is not equal to $0$. The following properties hold:
\begin{itemize}
\item[$\diamond$] $m_p(C)\geq 0$,

\item[$\diamond$] $m_p(C)=0$ if and only if $p$ does not belong to $C$,

\item[$\diamond$] $m_p(C)=1$ if and only if $p$ is a smooth point of $C$.
\end{itemize}
\end{eg}

Assume that $C$ and $D$ are distinct curves with no common component ;
we can define an integer $(C\cdot D)_p$\index{not}{$(C\cdot D)_p$} 
which counts the intersection of $C$ and $D$ at $p$:
\begin{itemize}
\item[$\diamond$] if either $C$ or $D$ does not pass
through $p$, it is equal to $0$;

\item[$\diamond$] otherwise let $f$, resp. $g$ be some local equation 
of $C$, resp. $D$ in a neighborhood of $p$, and define $(C\cdot D)_p$  
to be the dimension of $\faktor{\mathcal{O}_{p,S}}{(f,g)}$.
\end{itemize}

This number is related to $C\cdot D$ by the following statement:

\begin{pro}[\cite{Hartshorne}]\label{pro:hart2}
If $C$ and $D$ are distinct curves without any common irreducible 
component on a smooth surface, then
\[
C\cdot D=\displaystyle\sum_{p\in C\cap D}(C\cdot D)_p.
\]
In particular $C\cdot D\geq 0$.
\end{pro}

Let $C$ be a curve on $S$, and let $p$ be a point of $S$. Take
local coordinates $z_0$, $z_1$ at $p$ such that $p=(0,0)$. Set
$k=m_p(C)$. The curve $C$ is thus given by 
\[
P_k(z_0,z_1)+P_{k+1}(z_0,z_1)+\ldots+P_r(z_0,z_1)=0
\]
where the $P_i$'s denote homogeneous polynomials of degree $i$.
The blow up of $p$ can be viewed as $(u,v)\mapsto(uv,v)$, and the 
pull-back of $C$ is given by 
\[
v^k\big(p_k(u,1)+vp_{k+1}(u,1)+\ldots+v^{r-k}p_r(u,1)\big)=0.
\]
In other words the pull-back of $C$ decomposes into $k$ times 
the exceptional divisor $E=\pi^{-1}(0,0)=(v=0)$ and the strict
transform. We can thus state:

\begin{lem}[\cite{Hartshorne}]\label{lem:tireenarriere}
Let $S$ be a smooth surface. Let $\pi\colon\mathrm{Bl}_pS\to S$
be the blow-up of a point $p\in S$. If $C$ is a curve on $S$, 
if $\widetilde{C}$ is its strict transform and if $E=\pi^{-1}(p)$
is the exceptional divisor, then 
\[
\pi^*(C)=\widetilde{C}+m_p(C)E.
\]
\end{lem}

We also have the following statement:

\begin{pro}[\cite{Hartshorne}]\label{pro:tireenarriere}
Let $S$ be a smooth surface, let $p$ be a point of $S$, and let
$\pi\colon\mathrm{Bl}_pS\to S$ be the blow-up of $p$. Denote by
$E\subset\mathrm{Bl}_pS$ the exceptional divisor 
$\pi^{-1}(p)\simeq\mathbb{P}^1_\mathbb{C}$. Then 
\[
\mathrm{Pic}(\mathrm{Bl}_pS)=\pi^*\mathrm{Pic}(S)+\mathbb{Z}E.
\]
The intersection form on $\mathrm{Bl}_pS$ is induced by the 
intersection form on $S$ via the following formulas:
\begin{itemize}
\item[$\diamond$] $\pi^*C\cdot\pi^*D=C\cdot D$ for any $C$, 
$D$ in $\mathrm{Pic}(S)$;

\item[$\diamond$] $\pi^*C\cdot E=0$ for any $C$ in 
$\mathrm{Pic}(S)$;

\item[$\diamond$] $E^2=E\cdot E=-1$;

\item[$\diamond$] $\widetilde{C}^2=C^2-1$ for any smooth 
curve $C$ passing through $p$ and where $\widetilde{C}$ is 
the strict transform of $C$.
\end{itemize}
\end{pro}

If $V$ is an algebraic variety, then the 
\textsl{nef cone}\index{defi}{nef cone}
$\mathrm{Nef}(V)$\index{not}{$\mathrm{Nef}(V)$} is the cone of
divisors $D$ such that $D\cdot C\geq 0$ for any curve $C$ in $V$.


\section{A geometric definition of birational maps}\label{sec:geodef}

Let $\phi$ be the element of $\mathrm{Bir}(\mathbb{P}^2_\mathbb{C})$ 
given by
\[
\phi\colon(z_0:z_1:z_2)\dashrightarrow\big(\phi_0(z_0,z_1,z_2):\phi_1(z_0,z_1,z_2):\phi_2(z_0,z_1,z_2)\big)
\]
where the $\phi_i$'s are homogeneous polynomials of the same degree
$\nu$, and without common factor of positive degree. The 
\textsl{linear system}\index{defi}{linear system (birational map)}
$\Lambda_\phi$\index{not}{$\Lambda_\phi$} of $\phi$ is the 
strict pull-back of the system $\mathcal{O}_{\mathbb{P}^2_\mathbb{C}}(1)$
of lines of $\mathbb{P}^2_\mathbb{C}$ by $\varphi$.

\begin{rems}
\begin{itemize}
\item[$\diamond$] If $A$ is an automorphism of $\mathbb{P}^2_\mathbb{C}$, 
then $\Lambda_\phi=\Lambda_{A\circ\phi}$.

\item[$\diamond$] The degree of the curves of $\Lambda_\phi$ is $\nu$.
\end{itemize}
\end{rems}

\begin{eg}
The linear system associated to $\sigma_2$ is the linear system 
of conics passing through $(1:0:0)$, $(0:1:0)$ and $(0:0:1)$.
\end{eg}

\begin{rem}
Let us define the linear system of a divisor and then mention 
the connection between the linear system of a divisor and the linear 
system of a birational map. Let $D$ be a divisor on a surface
$S$. Denote by $\vert D\vert$ the set of all effective divisors
on $S$ linearly equivalent to $D$. Every non-vanishing 
section of $\mathcal{O}_S(D)$ defines an element of 
$\vert D\vert$, namely its divisor of zeros; conversely 
every element of $\vert D\vert$ is the divisor of zeros of
a non-vanishing section of $\mathcal{O}_S(D)$, defined 
up to scalar multiplication. Hence $\vert D\vert$ can be 
naturally identified with the projective space associated to 
the vector space $H^0(\mathcal{O}_S(D))$. A 
linear subspace $P$ of $\vert D\vert$ is called a 
\textsl{linear system}\index{defi}{linear system} 
on~$S$; of course equivalently $P$ can be defined by a 
vector subspace of $H^0(\mathcal{O}_S(D))$. The subspace
$P$ is \textsl{complete}\index{defi}{complete (linear system)}
if $P=\vert D\vert$. The 
\textsl{dimension}\index{defi}{dimension (linear system)}
of $P$ is its dimension as a projective space. A one-dimensional
linear system is a \textsl{pencil}\index{defi}{pencil}.
A curve $C$ is a 
\textsl{fixed component}\index{defi}{fixed component (linear system)}
of $P$ if every divisor of $P$ contains $C$. The 
\textsl{fixed part}\index{defi}{fixed part (linear system)} 
of $P$ is the biggest divisor that is contained in every 
element of $P$. A point $p$ of $S$ is a 
\textsl{base-point}\index{defi}{base-point (linear system)}
of $P$ if every divisor of $P$ contains $p$. If the linear
system has no fixed part, then it has only a finite number 
of fixed points; this number is bounded by $D^2$ for 
$D\in P$. 

Let $S$ be a surface. Then there is a bijection between 
\[
\big\{\text{rational maps $\phi\colon S\dashrightarrow\mathbb{P}^n_\mathbb{C}$ such that $\phi(S)$ is contained in no hyperplane}\big\}
\]
and 
\[
\big\{\text{linear systems on $S$ without fixed part and of dimension $n$}\big\}.
\]
This correspondence is constructed as follows: to the map
$\phi$ we associate the linear system $\phi^*\vert H\vert$
where $\vert H\vert$ is the system of hyperplanes in 
$\mathbb{P}^n_\mathbb{C}$. Conversely let $P$ be a linear
system on $S$ with no fixed part; denote by 
$\widehat{P}$ the projective dual space to~$P$. Define 
a rational map $\phi\colon S\dashrightarrow\widehat{P}$
by sending $p\in S$ to the hyperplane in $P$ consisting
of the divisors passing through $p$: the map $\phi$ is
defined at $p$ if and only if $p$ is not a base-point 
of $P$. 
\end{rem}

If $p_1$ is a point of indeterminacy of $\phi$, then denote by 
$\pi_1\colon\mathrm{Bl}_{p_1}\mathbb{P}^2_\mathbb{C}\to\mathbb{P}^2_\mathbb{C}$
the blow-up of $p_1$ and by $\mathcal{E}_1$ the associated exceptional divisor.
The map $\varphi_1=\phi\circ\pi_1$ is a birational map from 
$\mathrm{Bl}_{p_1}\mathbb{P}^2_\mathbb{C}$ to $\mathbb{P}^2_\mathbb{C}$.
If $p_2$ is a point of indeterminacy of $\varphi_1$, we blow up $p_2$ via
$\pi_2\colon\mathrm{Bl}_{p_1,p_2}\mathbb{P}^2_\mathbb{C}\to\mathbb{P}^2_\mathbb{C}$,
and we set $\mathcal{E}_2=\pi_2^{-1}(p_2)$. Again the map 
$\varphi_2=\varphi_1\circ\pi_1\colon\mathrm{Bl}_{p_1,p_2}\mathbb{P}^2_\mathbb{C}\dashrightarrow\mathbb{P}^2_\mathbb{C}$
is a birational map. We iterate this processus until~$\varphi_r$ becomes a
morphism. Set $E_i=(\pi_{i+1}\circ\ldots\circ\pi_r)^*\mathcal{E}_i$ and 
$\ell=(\pi_1\circ\ldots\circ\pi_r)^*L$ where $L$ is the divisor of a line. 
Applying $r$ times Proposition \ref{pro:tireenarriere} we get
\[
\left\{
\begin{array}{lllll}
\mathrm{Pic}(\mathrm{Bl}_{p_1,p_2,\ldots,p_r}\mathbb{P}^2_\mathbb{C})=\mathbb{Z}\ell\oplus\mathbb{Z}E_1\oplus\mathbb{Z}E_2\oplus\ldots\oplus\mathbb{Z}E_r,\\
\ell^2=\ell\cdot\ell,\\ 
E_i^2=E_i\cdot E_i=-1, \\
E_i\cdot E_j=0\quad\forall\,1\leq i\not=j\leq r,\\ E_i\cdot\ell=0\quad\forall\,1\leq i\leq r.
\end{array}
\right.
\]

The curves of $\Lambda_\phi$ pass through the $p_i$'s with multiplicity
$m_{p_i}(\phi)$. Applying $r$ times Lem\-ma~\ref{lem:tireenarriere} the 
elements of $\Lambda_{\varphi_r}$ are equivalent to 
\[
\nu L-\displaystyle\sum_{i=1}^rm_{p_i}(\phi)E_i
\]
where $L$ is the pull-back of a generic line in $\mathbb{P}^2_\mathbb{C}$. 
As a result the curves of $\Lambda_{\varphi_r}$
have self intersection $\nu^2-\displaystyle\sum_{i=1}^rm_{p_i}(\phi)^2$. Note 
that all
the members of a linear system are linearly equivalent and that the dimension
of $\Lambda_{\varphi_r}$ is $2$; the self intersection has thus to be 
non-negative by Proposition \ref{pro:hart2}. As a consequence the number $r$ 
exists; in other words $\phi$
has a finite number of base-points. By construction
\[
\varphi_r\colon\mathrm{Bl}_{p_1,p_2,\ldots,p_r}\mathbb{P}^2_\mathbb{C}\to\mathbb{P}^2_\mathbb{C}
\]
is a birational morphism which is the blow-up of the base-points of $\phi^{-1}$.
Consider a general line $L$ of $\mathbb{P}^2_\mathbb{C}$ that does not pass 
through $p_1$, $p_2$, $\ldots$, $p_r$. Its pull-back $\varphi_r^{-1}(L)$ 
corresponds to a smooth curve on 
$\mathrm{Bl}_{p_1,p_2,\ldots,p_r}\mathbb{P}^2_\mathbb{C}$
which has self-intersection $1$ and genus $0$. Hence
\[
\left\{
\begin{array}{ll}
 (\varphi_r^{-1}(L))^2=1, \\
 \varphi_r^{-1}(L)\cdot K_{\mathrm{Bl}_{p_1,p_2,\ldots,p_r}\mathbb{P}^2_\mathbb{C}}=-3.
\end{array}
\right.
\]
As the elements of $\Lambda_{\varphi_r}$ are equivalent to 
$\nu L-\displaystyle\sum_{i=1}^rm_{p_i}(\phi)E_i$ and since 
\[
K_{\mathrm{Bl}_{p_1,p_2,\ldots,p_r}\mathbb{P}^2_\mathbb{C}}=-3L+\displaystyle\sum_{i=1}^rE_i 
\]
the following equalities hold:
\[
\left\{
\begin{array}{ll}
\displaystyle\sum_{i=1}^rm_{p_i}(\phi)=3(\nu-1),\\ \displaystyle\sum_{i=1}^rm_{p_i}(\phi)^2=\nu^2-1.
\end{array}
\right.
\]

\begin{egs}
\begin{itemize}
\item[$\diamond$] If $\nu=2$, then $r=3$ and $m_{p_1}(\phi)=m_{p_2}(\phi)=m_{p_3}(\phi)=~1$.

\item[$\diamond$] If $\nu=3$, then $r=5$ and $m_{p_1}(\phi)=2$, $m_{p_2}(\phi)=m_{p_3}(\phi)=m_{p_4}(\phi)=m_{p_5}(\phi)=1$.
\end{itemize}
\end{egs}


\chapter[Action of $\mathrm{Bir}(\mathbb{P}^2_\mathbb{C})$ onto $\mathbb{H}^\infty$]{An isometric action of the Cremona group on an infinite dimensional hyperbolic space}\label{chap:hyperbolicspace}

\bigskip
\bigskip

If $S$ is a projective surface, the group
$\mathrm{Bir}(S)$ of birational self maps
of $S$ acts faithfully by isometries on a 
hyperbolic space $\mathbb{H}^\infty(S)$ 
of infinite dimension.
After recalling some notions of hyperbolic
geometry in the first section of this 
chapter we describe this construction 
in the second section. Let us now give an 
outline of it before heading into details. 
Let $S$ be a projective
surface. If $\pi\colon Y\to S$ is a 
birational morphism, then one 
obtains an embedding 
$\pi^*\colon\mathrm{NS}(S)\to\mathrm{NS}(Y)$
of N\'eron-Severi groups. 
If $\pi_1\colon Y_1\to S$ and $\pi_2\colon Y_2\to S$
are two birational morphisms, then
\begin{itemize}
\item[$\diamond$] $\pi_2$ is above $\pi_1$ 
if $\pi_1^{-1}\circ\pi_2$ is a morphism, 

\item[$\diamond$] one can always find 
a third birational morphism $\pi_3\colon Y_3\to S$
that is above $\pi_1$ and $\pi_2$.
\end{itemize}

Hence the inductive limit of all groups 
$\mathrm{NS}(Y_i)$ for all surfaces $Y_i$ 
above $S$ is well-defined; this limit
$\mathcal{Z}(S)$ is the 
Picard-Manin space of 
$S$. The intersection forms on~$Y_i$ 
yield to a scalar product $\langle\,,\,\rangle$
on $\mathcal{Z}(S)$.

Consider all surfaces $Y$ above $S$, {\it i.e.}
all birational morphisms $\pi\colon Y\to S$.
We identify $p_1\in Y_1$ and $p_2\in Y_2$
if $\pi_1^{-1}\circ\pi_2$ is a local 
isomorphism in a neighborhood of~$p_2$ that 
maps $p_2$ onto $p_1$. The bubble space
$\mathcal{B}(S)$ of $S$ is the union of all 
points of all surfaces above $S$ modulo
the equivalence relation generated by these
identifications. If $p$ belongs to 
$\mathcal{B}(S)$, then we denote by $\mathbf{e}_p$ the
divisor class of the exceptional divisor
of the blow up of $p$. The equalities 
$\mathbf{e}_p\cdot\mathbf{e}_p=-1$ and 
$\mathbf{e}_p\cdot\mathbf{e}_{p'}=0$ hold
by Proposition~\ref{pro:tireenarriere}

The N\'eron-Severi group
$\mathrm{NS}(S)$ is naturally embedded as a 
subgroup of the Picard-Manin
space; this finite dimensional lattice is 
orthogonal to $\mathbf{e}_p$ for any 
$p\in~\mathcal{B}(S)$. More precisely
\[
\mathcal{Z}(S)=\mathrm{NS}(S)\displaystyle\bigoplus_{p\in\mathcal{B}(S)}\mathbb{Z}\mathbf{e}_p.
\]
As a result any element $v$ of $\mathcal{Z}(S)$
can be written as a finite sum
\[
v=w+\displaystyle\sum_{p\in\mathcal{B}(S)}m_p\mathbf{e}_p.
\]
There is a completion process for which the 
completion $\mathrm{Z}(S)$ of 
$\mathcal{Z}(S)\otimes_\mathbb{Z}\mathbb{R}$
is 
\[
\mathrm{Z}(S)=\Big\{w+\displaystyle\sum_{p\in\mathcal{B}(S)}m_p\mathbf{e}_p\,\vert\, w\in\mathrm{NS}(\mathbb{R},S),\,\displaystyle\sum_{p\in\mathcal{B}(S)}m_p^2<\infty\Big\}.
\]
The intersection form extends as a scalar product
with signature $(1,\infty)$ on this space. The 
hyperbolic space $\mathbb{H}^\infty(S)$ of $S$ is 
defined by 
\[
\mathbb{H}^\infty(S)=\big\{w\in\mathrm{Z}(S),\,\vert\, \langle w,\,w\rangle=1,\,\langle w,\,a\rangle>0\text{ for all ample classes } a\in\mathrm{NS}(S)\big\}.
\]
It is an infinite dimensional analogue of the 
classical hyperbolic space $\mathbb{H}^n$. One
can define a complete distance $\mathrm{dist}$
on $\mathbb{H}^\infty(S)$ by
\[
\cosh(\mathrm{dist}(v,w))=\langle v,\,w\rangle.
\]
Geodesics are intersection of 
$\mathbb{H}^\infty(S)$ with planes. The
projection of $\mathbb{H}^\infty(S)$ to the 
projective space $\mathbb{P}(\mathrm{Z}(S))$ 
is one to one, and the boundary of its image 
is the projection of the cone of isotropic 
vectors of $\mathrm{Z}(S)$: 
\[
\partial\mathbb{H}^\infty(S)=\big\{\mathbb{R}_+v\,\vert\,v\in\mathrm{Z}(S),\,\langle v,\,v\rangle=0,\,\langle v,\,a\rangle>0\text{ for all ample classes }a\in\mathrm{NS}(S)\big\}.
\]
The important fact is that $\mathrm{Bir}(S)$
acts faithfully on $\mathrm{Z}(S)$ by 
continuous linear endomorphisms preserving 
the intersection form, the effective cone, 
the nef cone, $\mathcal{Z}(S)$ and also 
$\mathbb{H}^\infty(S)$.

If $\phi$ is an element of $\mathrm{Bir}(S)$, 
we denote by $\phi_*$ its action on 
$\mathrm{Z}(S)$: it is a linear isometry 
with respect to the intersection form; we
also denote by $\phi_*$ the isometry of 
$\mathbb{H}^\infty(S)$ induced by this 
endomorphism of $\mathrm{Z}(S)$. Let $f$ be 
an isometry of $\mathbb{H}^\infty(S)$; the 
translation length of $f$ is 
\[
L(f)=\inf\big\{\mathrm{dist}(v,f(v))\,\vert\, v\in\mathbb{H}^\infty(S)\big\}.
\]
If this infimum is a minimum, then 
\begin{itemize}
\item[$\diamond$] either it is equal to $0$, 
$f$ has a fixed point in $\mathbb{H}^\infty(S)$,
and $f$ is elliptic; 

\item[$\diamond$] or it is positive, and $f$ is 
loxodromic.
\end{itemize}

When the infimum is not realized, $L(f)$ is 
equal to $0$, and $f$ is parabolic. 

This classification into three types holds 
for all isometries of $\mathbb{H}^\infty(S)$. 
For isometries $\phi_*$ induced by birational
maps $\phi$ of $S$ there is a dictionary between 
this classification and the geometric 
properties of $\phi$. We give this 
dictionary in the third section.

\bigskip
\bigskip


\section{Some hyperbolic geometry}

Consider a real Hilbert space $\mathcal{H}$ of dimension $n$. 
Let $\mathbf{e}_0$ be a unit vector of 
$\mathcal{H}$, and let $\mathbf{e}_0^\perp$ be the orthogonal
complement of the space $\mathbb{R}\mathbf{e}_0$. Denote by 
$(\mathbf{e}_i)_{i\in I}$ an orthonormal basis of 
$\mathbf{e}_0^\perp$. A scalar product with signature $(1,n-1)$ 
can be defined on $\mathcal{H}$ by setting
\[
\langle u,\,v\rangle=a_0b_0-\displaystyle\sum_{i\in I}a_ib_i
\]
for any two elements 
$u=a_0\mathbf{e}_0+\displaystyle\sum_{i\in I}a_i\mathbf{e}_i$
and 
$v=b_0\mathbf{e}_0+\displaystyle\sum_{i\in I}b_i\mathbf{e}_i$
of $\mathcal{H}$. The set 
\[
\big\{v\in\mathcal{H}\,\vert\,\langle v,v\rangle=1\big\}
\]
defines a hyperboloid with two connected components. Let 
$\mathbb{H}^{n-1}$ be the connected component of this 
hyperboloid that contains $\mathbf{e}_0$. A metric can be 
defined on $\mathbb{H}^{n-1}$ by 
\[
d(u,v):=\mathrm{arccosh}(\langle u,v\rangle).
\]

\begin{rem}
A useful model for $\mathbb{H}^2$ is the Poincar\'e model:
$\mathbb{H}^2$ is identified to the upper half-plane
$\big\{z\in\mathbb{C}\,\vert\,\mathrm{Im}(z)>0\big\}$ with
its Riemanniann metric given by 
$\mathrm{d}s^2=\frac{x^2+y^2}{y^2}$. Its group of orientation
preserving isometries coincides with 
$\mathrm{PSL}(2,\mathbb{R})$, acting by linear fractional 
transformations.
\end{rem}

Let $(\mathcal{H},\langle\, .,.\rangle)$ be a real 
Hilbert space of infinite dimension. Let $\mathbf{e}_0$ 
be a unit vector of $\mathcal{H}$, and let 
$\mathbf{e}_0^\perp$ be its orthogonal complement. 
Any element $v$ of $\mathcal{H}$ can be written in a 
unique way as $v=v_{\mathbf{e}_0}\mathbf{e}_0+v_{\mathbf{e}_0^\perp}$
where $v_{\mathbf{e}_0}$ belongs to $\mathbb{R}$ and 
$v_{\mathbf{e}_0^\perp}$ belongs to $\mathbf{e}_0^\perp$. 
Consider the symetric bilinear form $\mathcal{B}$ 
of $\mathcal{H}$ defined by 
\[
\mathcal{B}(x,y)=x_{\mathbf{e}_0}y_{\mathbf{e}_0}-\langle x_{\mathbf{e}_0^\perp},y_{\mathbf{e}_0^\perp}\rangle;
\]
it has signature $(1,\infty)$. Let $\mathbb{H}^\infty$
be the hyperboloid given by 
\[
\mathbb{H}^\infty=\big\{x\in\mathcal{H}\,\vert\,\mathcal{B}(x,x)=1,\,\mathcal{B}(\mathbf{e}_0,x)>0\big\}.
\]
We consider on $\mathbb{H}^\infty$ the distance $d$ 
defined by $\cosh d(x,y)=\mathcal{B}(x,y)$. 
The space $(\mathbb{H}^\infty,d)$ is a complete 
metric space of infinite dimension.

\medskip

\subsection{$\delta$-hyperbolicity and $\mathrm{CAT}$($-1$) spaces}

Let $(X,d)$ be a geodesic metric space. Let $x$, $y$, $z$
be three points of $X$. We denote by $[p,q]$ the segment 
with endpoints $p$ and $q$. A 
\textsl{geodesic triangle}\index{defi}{geodesic triangle}
with vertices $x$, $y$, $z$ is the union of three geodesic
segments $[x,y]$, $[y,z]$ and $[z,x]$. Let 
$\delta\geq 0$. If for any point $m\in[x,y]$ there is a 
point in $[y,z]\cup[z,x]$ at distance less than $\delta$
of $m$, and similarly for points on the other edges, then
the triangle is said do be 
\textsl{$\delta$-slim}\index{defi}{$\delta$-slim}.
A \textsl{$\delta$-hyperbolic space}\index{defi}{$\delta$-hyperbolic space}
is a geodesic metric space whose all of geodesic 
triangles are $\delta$-slim. A $\delta$-hyperbolic 
space is called 
\textsl{Gromov hyperbolic space}\index{defi}{Gromov hyperbolic space}.

\begin{egs}
\begin{itemize}
\item[$\diamond$] Metric trees are $0$-hyperbolic: all 
triangles are tripods.

\item[$\diamond$] The hyperbolic plane is $(-2)$-hyperbolic. In 
fact the incircle of a geodesic triangle is the circle
of largest diameter contained in the triangle, and any 
geodesic triangle lies in the interior of an ideal triangle, 
all of which are isometric with incircles of diameter 
$2\log 3$ (\emph{see} \cite{CoornaertDelzantPapadopoulos}).

\item[$\diamond$] The space $\mathbb{R}^2$ endowed with the euclidian 
metric is not $\delta$-hyperbolic (for instance because
of the existence of homotheties). 
\end{itemize}
\end{egs}

Let us now introduce $\mathrm{CAT}$($-1)$ spaces\footnote{The terminology
corresponds to the initials of E. Cartan, A. Alexandrov
and V. Toponogov.}. Let $(X,d_X)$ be a geodesic metric 
space. Consider a geodesic triangle $T$ in $X$ determined
by the three points $x$, $y$, $z$ and the data of three
geodesics between two of these three points. A 
\textsl{comparison triangle}\index{defi}{comparison triangle}
of $T$ in the metric space $(X',d_{X'})$ is a triangle $T'$ such that
\[
\left\{
\begin{array}{lll}
d_X(x,y)=d_{X'}(x',y')\\
d_X(x,z)=d_{X'}(x',z')\\
d_X(y,z)=d_{X'}(y',z')\\
\end{array}
\right.
\]
Let $p$ be a point of $[x,y]\subset T$. A point 
$p'\in[x',y']\subset T'$ is a 
\textsl{comparison point}\index{defi}{comparison point}
of $p$ if $d_{X'}(x',p')=d_X(x,p)$. 

The triangle $T$ \textsl{satisfies the $\mathrm{CAT}$}$(-1)$ \textsl{inequality}\index{defi}{$\mathrm{CAT}$($-1$) inequality}
if for any $(x,y)\in T^2$ 
\[
d_X(x,y)\leq\vert\vert x'-y'\vert\vert_{\mathbb{H}^2}
\]
where $T'$ is a comparison triangle of 
$T$ in $\mathbb{H}^2$ and $x'\in T'$ (resp. $y'\in T'$) is a comparison point 
of $x$ (resp. $y$).

The space $X$ is \textsl{$\mathrm{CAT}$}$(-1)$\index{defi}{$\mathrm{CAT}$($-1$) space}
if all its triangles satisfy the $\mathrm{CAT}$($-1$) inequality.

\begin{rem}
The $\mathrm{CAT}$($-1$) spaces are Gromov hyperbolic, but the converse
is false.
\end{rem}

Set $\mathcal{H}_{>0}=\big\{v\in\mathcal{H}\,\vert\, \langle v,v\rangle>0\big\}$.
The image of $v$ by the map  
\begin{align*}
&\eta\colon\mathcal{H}_{>0}\to\mathbb{H}^\infty && v\mapsto \frac{v}{\sqrt{\langle v,v\rangle}}
\end{align*}
is called the normalization of $v$. Geometrically
$\eta$ associates to a point $v\in\mathcal{H}_{>0}$ the intersection
of $\mathbb{H}^\infty$ with the line through $v$. Note that if the 
intersection of $\mathcal{H}$ with a vectorial subspace of dimension 
$n+1$ of $\mathcal{H}$ is not empty, then it is a copy of $\mathbb{H}^n$.
In particular there exists a unique geodesic segment between two 
points of $\mathbb{H}^\infty$ obtained as the intersection of $\mathbb{H}^\infty$
with the plane that contains these two points. Hence any triangle of 
$\mathbb{H}^\infty$ is isometric to a triangle of $\mathbb{H}^2$.
As a result $\mathbb{H}^\infty$ is $\mathrm{CAT}$($-1$) and $\delta$-hyperbolic
for the same constant $\delta $ as $\mathbb{H}^2$.

\medskip

\subsection{Boundary of $\mathbb{H}^\infty$}

Let $(X,d)$ be a geodesic metric space. Let $T$ be a 
geodesic triangle of $X$ given by $x$, $y$, $z\in X$ and
geodesic segments between two of these three points. The 
triangle $T$ \textsl{satisfies the $\mathrm{CAT}$}$(0)$ \textsl{inequality}\index{defi}{$\mathrm{CAT}$($0$) inequality}
if for any $(x,y)\in T^2$ 
\[
d_X(x,y)\leq\vert\vert x'-y'\vert\vert_{\mathbb{R}^2}
\]
where $x'\in T'$ (resp. $y'\in T'$) is a comparison point 
of $x$ (resp. $y$) and $T'$ is a comparison triangle of 
$T$ in $\mathbb{R}^2$.

The space $X$ is \textsl{$\mathrm{CAT}$}$(0)$\index{defi}{$\mathrm{CAT}$($0$) space}
if all its triangles satisfy the $\mathrm{CAT}$($0$) inequality.

\begin{rem}
A $\mathrm{CAT}$($-1$) space is a $\mathrm{CAT}$($0$) space. In particular 
$\mathbb{H}^\infty$ is a $\mathrm{CAT}$($0$) space.
\end{rem}

Since $\mathbb{H}^\infty$ is a $\mathrm{CAT}$($0$), complete metric space
there exists a notion of boundary at infinity that generalizes 
the notion of boundary of finite dimensional Riemann varieties
which are complete, simply connected and with negative curvature.
The \textsl{boundary of $\mathbb{H}^\infty$}\index{defi}{boundary of $\mathbb{H}^\infty$}
is defined by\index{not}{$\partial\mathbb{H}^\infty$} 
\[
\partial\mathbb{H}^\infty=\big\{v\in\mathcal{H}\,\vert\,\langle v,v\rangle=0,\,\langle v,\mathbf{e}_0\rangle>0\big\}.
\]

A point of $\partial\mathbb{H}^\infty$ is called 
\textsl{point at infinity}\index{defi}{point at infinity}. 

\medskip

\subsection{Isometries}\label{subsec:isometries}

Denote by $\mathrm{O}_{1,n}(\mathbb{R})$ the group of 
linear transformations of~$\mathcal{H}$ preserving the 
scalar product $\langle\,,\rangle$. The group of isometries
$\mathrm{Isom}(\mathbb{H}^n)$ coincides with the index 
$2$ subgroup $\mathrm{O}_{1,n}^+(\mathbb{R})$ of 
$\mathrm{O}(\mathcal{H})$ that preserves the chosen sheet
$\mathbb{H}^n$ of the hyperboloid
\[
\big\{u\in\mathcal{H}\,\vert\,\langle u,u\rangle=1\big\}.
\]
This group acts transitively on $\mathbb{H}^n$ and on its
unit tangent bundle.

If $h$ is an isometry of
$\mathbb{H}^n$ and $v\in\mathcal{H}$ is an eigenvector of 
$h$ with eigenvalue $\lambda$, then either $\vert\lambda\vert=1$
or $v$ is isotropic. Furthermore $\mathbb{H}^n$ is homeomorphic
to a ball, so $h$ has a least one eigenvector in 
$\mathbb{H}^n\cup\partial\mathbb{H}^n$. As a consequence 
according to \cite{BurgerIozziMonod} there are three types 
of isometries:

\begin{itemize}
\item[$\diamond$] $h$ is \textsl{elliptic}\index{defi}{elliptic (isometry)} 
if and only if $h$ fixes a point $p\in\mathbb{H}^n$. Since 
$\langle\,,\,\rangle$ is negative definite on $p^\perp$, $h$ 
fixes pointwise $\mathbb{R}p$ and acts by rotation on $p^\perp$ 
with respect to $\langle\,,\,\rangle$;

\item[$\diamond$] $h$ is \textsl{parabolic}\index{defi}{parabolic (isometry)}
if $h$ is not elliptic and fixes a vector $v$ in the isotropic
cone. The line $\mathbb{R}v$ is uniquely determined by $h$. Let 
$p$ be a point of $\mathbb{H}^n$; there exists an increasing 
sequence $(n_i)\in\mathbb{N}^\mathbb{N}$ such that 
$(h^{n_i}(p))_{i\in\mathbb{N}}$ 
converges to the boundary point determined by $v$.

\item[$\diamond$] $h$ is \textsl{loxodromic}\index{defi}{loxodromic (isometry)}
if and only if $h$ has an eigenvector $v_h^+$ with eigenvalue 
$\lambda>1$. Note that $v_h^+$ is unique up to scalar 
multiplication. There is another unique isotropic eigenline
$\mathbb{R}v_h^-$ corresponding to the eigenvalue~$\frac{1}{\lambda}$. 
On the orthogonal complement of 
$\mathbb{R}v_h^-\oplus\mathbb{R}v_h^+$ the isometry $h$ acts as a 
rotation with respect to $\langle\,,\,\rangle$. The 
boundary points determined by $v_h^-$ and~$v_h^+$ are the 
two fixed points of $h$ in $\mathbb{H}^n\cup\partial\mathbb{H}^n$;
the first one is an attracting fixed point $\alpha(h)$, the second one is 
a repelling fixed point $\omega(h)$.
\end{itemize}

\smallskip

To an isometry $h$ of $\mathbb{H}^n$ one can associate the 
\textsl{translation length}\index{defi}{translation length (of an isometry)}
of $h$:
\[
L(h)=\inf\big\{d(h(p),p)\,\vert p\in\mathbb{H}^n\big\}.
\]
The isometry $h$ is elliptic if and only if  $L(h)=0$, and the infimum
is achieved, {\it i.e.}~$h$ has a fixed point in $\mathbb{H}^n$.
The isometry $h$ is parabolic if and only if $L(h)=0$, and the 
infinimum is not achieved. The isometry $h$ is loxodromic if and
only if~$L(h)>0$. In that case 
\begin{itemize}
\item[$\diamond$] $\exp(L(h))$ is the largest eigenvalue
of $h$
\item[$\diamond$] and $d(p,h^n(p))$ grows like $nL(h)$ as 
$n$ goes to infinity for any point $p\in\mathbb{H}^n$.
\end{itemize}


\section{The isometric action of $\mathrm{Bir}(S)$ 
on an infinite dimensional hyperbolic space}

\subsection{The Picard-Manin space}

Let $S$ be a smooth, irreducible, projective, complex surface.
As we see in Chapter \ref{chapter:intro}
the Picard group $\mathrm{Pic}(S)$ is the
quotient of the abelian group of divisors by the subgroup
of principal divisors (\cite{Hartshorne}).
The intersection between curves extends to a quadratic 
form, the so-called intersection form:
\begin{align*}
& \mathrm{Pic}(S)\times\mathrm{Pic}(S)\to\mathbb{Z}, && (C,D)\mapsto C\cdot D
\end{align*}
The quotient of $\mathrm{Pic}(S)$ by the subgroup of divisors
$E$ such that $E\cdot D=0$ for all divisor classes $D$ is the
\textsl{N\'eron-Severi group}\index{defi}{N\'eron-Severi group} 
$\mathrm{NS}(S)$\index{not}{$\mathrm{NS}(S)$}. In case of 
rational surfaces we have $\mathrm{NS}(S)=\mathrm{Pic}(S)$. 
The N\'eron-Severi group is a free abelian group, and its rank, 
the \textsl{Picard number}\index{defi}{Picard number} 
is finite. 
The pull-back of a birational morphism $\pi\colon Y\to S$ 
yields an injection from $\mathrm{Pic}(S)$ into $\mathrm{Pic}(Y)$; 
we thus get an injection from $\mathrm{NS}(S)$ into
$\mathrm{NS}(Y)$. The morphism $\pi\colon Y\to S$ 
can be written as a finite sequence of blow ups. Let 
$\mathbf{e}_1$, $\mathbf{e}_2$, $\ldots$, $\mathbf{e}_k\subset Y$ 
be the class of the irreducible components of the exceptional 
divisor of $\pi$, that is the classes contracted by $\pi$. We
have the following decomposition
\begin{equation}\label{eq:decom}
\mathrm{NS}(Y)=\mathrm{NS}(S)\oplus\mathbb{Z}\mathbf{e}_1\oplus\mathbb{Z}\mathbf{e}_2\oplus\ldots\oplus\mathbb{Z}\mathbf{e}_k
\end{equation}
which is orthogonal with respect to the intersection form.

Consider $\pi_1\colon Y\to S$ and $\pi_2\colon Y'\to S$ two
birational morphisms of smooth projective surfaces. We say
that $\pi_1$ is \textsl{above}\index{defi}{above (morphism)} $\pi_2$
if $\pi_2^{-1}\circ\pi_1$ is a morphism. For any two 
birational morphisms $\pi_1\colon Y\to S$ and 
$\pi_2\colon Y'\to S$ there exists a birational morphism
$\pi_3\colon Y''\to S$ that lies above $\pi_1$ and $\pi_2$.

Let us consider the set of all birational morphisms of smooth
projective surfaces $\pi\colon Y\to S$. The corresponding 
embeddings of the N\'eron-Severi groups 
$\mathrm{NS}(S)\to\mathrm{NS}(Y)$ form a directed family;
the direct limit
\[
\mathcal{Z}(S):=\displaystyle\lim_{\pi\colon Y\to S}\mathrm{NS}(Y)
\]\index{not}{$\mathcal{Z}(S)$}
thus exists. It is called the \textsl{Picard Manin space} of $S$.
The intersection forms on the groups $\mathrm{NS}(Y)$ induce a 
quadratic form on $\mathcal{Z}(S)$ of signature $(1,\infty)$.

Let $p$ be a point of the bubble space of $S$. Denote by $\mathbf{e}_p$ the
divisor class of the exceptional divisor of the blow-up of $p$ in 
the corresponding N\'eron-Severi group. One deduces from 
(\ref{eq:decom}) the following decomposition
\[
\displaystyle\mathcal{Z}(S)=\mathrm{NS}(S)\oplus\displaystyle\bigoplus_{p\in\mathcal{B}(S)}\mathbb{Z}\mathbf{e}_p.
\]
Furthermore according to Proposition \ref{pro:tireenarriere} 
the following properties hold
\[
\left\{
\begin{array}{ll}
\mathbf{e}_p\cdot \mathbf{e}_p=-1 \\
\mathbf{e}_p\cdot \mathbf{e}_q=0 \text{ for all $p\not=q$}
\end{array}
\right.
\]

\subsection{The hyperbolic space $\mathbb{H}^\infty(S)$}

Let $S$ be a smooth projective surface, and let $\mathcal{Z}(S)$ be 
its Picard-Manin space. We define $\mathrm{Z}(S)$ to be the 
completion of the real vector space $\mathcal{Z}(S)\otimes\mathbb{R}$
\[
\mathrm{Z}(S)=\Big\{v+\displaystyle\sum_{p\in\mathcal{B}(S)}m_p\mathbf{e}_p\,\vert\, v\in\mathrm{NS}(S)\otimes\mathbb{R},\,m_p\in\mathbb{R},\,\displaystyle\sum_{p\in\mathcal{B}(S)}m_p^2<\infty\Big\}.
\]
The intersection form extends continuously to a quadratic form on 
$\mathrm{Z}(S)$ with signature $(1,\infty)$. Let 
$\mathrm{Isom}(\mathrm{Z}(S))$ be the group of isometries of 
$\mathrm{Z}(S)$ with respect to the intersection form. The
set of vectors $v\in\mathrm{Z}(S)$ such that $\langle v,v\rangle=1$
is a hyperboloid. The subset
\[
\mathbb{H}^\infty(S)=\big\{v\in\mathrm{Z}(S)\,\vert\,\langle v,v\rangle=1,\, \langle v,\mathbf{e}_0\rangle> 0\big\}
\]
is the sheet of that hyperboloid containing ample classes of 
$\mathrm{NS}(S,\mathbb{R})$. Let 
$\mathrm{Isom}(\mathbb{H}^\infty(S))$ be the subgroup of
$\mathrm{Isom}(\mathcal{Z}(S))$ that preserves 
$\mathbb{H}^\infty(S)$. The space $\mathbb{H}^\infty(S)$
equipped with the distance defined by 
\[
\cosh(d(v,v'))=\langle v,v'\rangle
\]
is isometric to a hyperbolic space $\mathbb{H}^\infty$. Let
$\partial\mathbb{H}^\infty(S)$ be the boundary of 
$\mathbb{H}^\infty(S)$. To simplify we will often write
$\mathbb{H}^\infty$\index{not}{$\mathbb{H}^\infty$} 
(resp. $\partial\mathbb{H}^\infty$\index{not}{$\mathbb{H}^\infty$}) 
instead of $\mathbb{H}^\infty(S)$ (resp. 
$\partial\mathbb{H}^\infty(S)$).

\subsection{An isometric action of $\mathrm{Bir}(S)$}

Let us now describe the action of $\mathrm{Bir}(S)$
on $\mathbb{H}^\infty$ (\emph{see} 
\cite{Manin, Cantat:annals}). Let $\phi\colon Y\to S$
be a birational morphism of smooth projective surfaces.
Denote by $p_1$, $p_2$, $\ldots$, $p_n\in\mathcal{B}(S)$
the points blown up by $\phi$. Denote by $\mathbf{e}_{p_i}$
the irreducible component of the exceptional divisor
contracted to~$p_i$. One has
\[
\mathrm{NS}(Y)=\mathrm{NS}(S)\oplus\mathbb{Z}\mathbf{e}_{p_1}\oplus\mathbb{Z}\mathbf{e}_{p_2}\oplus\ldots\oplus\mathbb{Z}\mathbf{e}_{p_n}.
\]
The morphism $\phi$ induces the isomorphism 
$\phi_*\colon\mathcal{Z}(Y)\to\mathcal{Z}(S)$ 
defined by 
\[
\left\{
\begin{array}{lll}
\phi_*(\mathbf{e}_p)=\mathbf{e}_{\phi_\bullet(p)}\qquad\forall\,p\in\mathcal{B}(Y)\smallsetminus\mathrm{Base}(\phi)\\
\phi_*(\mathbf{e}_{p_i})=\mathbf{e}_{p_i}\qquad\forall\,1\leq i\leq n\\
\phi_*(D)=D\qquad\forall\,D\in\mathrm{NS}(S)\subset\mathrm{NS}(Y)
\end{array}
\right.
\]

Let $\phi\colon Y\dashrightarrow S$ be a birational map
of smooth projective surfaces. Let $\pi_2\circ\pi_1^{-1}$
be a minimal resolution of $\phi$. The 
map $\phi$ induces an isomorphism 
$\phi_*\colon\mathcal{Z}(Y)\to\mathcal{Z}(S)$ defined 
by 
\[
\phi_*=(\pi_2)_*\circ(\pi_1)_*^{-1}.
\]

Let $S$ be a smooth projective surface. Any element $\phi$ 
of $\mathrm{Bir}(S)$ induces an isomorphism
$\phi_*\colon\mathcal{Z}(S)\to\mathcal{Z}(S)$, and $\phi_*$
yields an automorphism of $\mathcal{Z}(S)\otimes\mathbb{R}$
which extends to an automorphism of the completion 
$\mathrm{Z}(S)$ and preserves the intersection form.

\medskip

Let $\phi$ be a birational self map of $\mathbb{P}^2_\mathbb{C}$.
Assume that $\phi$ has degree $d$.
Then the base-point $\mathbf{e}_0$, {\it i.e.} the class of 
a line in $\mathbb{P}^2_\mathbb{C}$, is mapped by $\phi_*$ to 
the finite sum 
\[
d\mathbf{e}_0-\displaystyle\sum_im_i\mathbf{e}_{p_i}
\]
where each $m_i$ is a positive integer and $\mathbf{e}_{p_i}$ are 
the classes of the exceptional divisors corresponding to the 
base-points of $\phi^{-1}$. For instance if $\phi=\sigma_2$
is the standard Cremona involution, then 
\[
(\sigma_2)_*\mathbf{e}_0=2\mathbf{e}_0-\mathbf{e}_{p_1}-\mathbf{e}_{p_2}-\mathbf{e}_{p_3}
\]
where $p_1=(1:0:0)$, $p_2=(0:1:0)$ and $p_3=(0:0:1)$.

\begin{rem}
An invariant structure is given by the canonical form. The 
canonical class of $\mathbb{P}^2_\mathbb{C}$ blown up in $n$
points $p_1$, $p_2$, $\ldots$, $p_n$ is equal to 
$-3\mathbf{e}_0-\displaystyle\sum_{j=1}^n\mathbf{e}_{p_j}$. By 
taking intersection products one obtains a linear form 
$\omega_\infty$ defined by 
\begin{align*}
& \omega_\infty\colon\mathcal{Z}(\mathbb{P}^2_\mathbb{C})\to\mathbb{Z}, &&
m_0\mathbf{e}_0-\displaystyle\sum_{j=1}^nm_j\mathbf{e}_{p_j}\mapsto -3m_0+\displaystyle\sum_{j=1}^nm_j
\end{align*}
Since the isometric action of $\mathrm{Bir}(\mathbb{P}^2_\mathbb{C})$
on $\mathcal{Z}(\mathbb{P}^2_\mathbb{C})$ preserves the linear
form $\omega_\infty$ we get the following equalities already obtained
in \S \ref{sec:geodef}: if 
$\phi_*\mathbf{e}_0=d\mathbf{e}_0-\displaystyle\sum_{j=1}^nm_j\mathbf{e}_{p_j}$, 
then 
\[
\left\{
\begin{array}{ll}
d^2=1+\displaystyle\sum_{j=1}^nm_j^2\\
3d-3=\displaystyle\sum_{j=1}^nm_j
\end{array}
\right.
\]
\end{rem}

\begin{eg}
Let us understand the isometry $(\sigma_2)_*$. Denote by $p_1$, 
$p_2$ and~$p_3$ the base-points of $\sigma_2$, and set 
$S=\mathrm{Bl}_{p_1,p_2,p_3}\mathbb{P}^2_\mathbb{C}$. The involution
$\sigma_2$ lifts to an automorphism $\widetilde{\sigma_2}$ on $S$. 
The N\'eron-Severi group $\mathrm{NS}(S)$ of $S$ is the lattice of 
rank~$4$ generated by the class $\mathbf{e}_0$, coming from the class
of a line in $\mathbb{P}^2_\mathbb{C}$, and the classes 
$\mathbf{e}_i=\mathbf{e}_{p_i}$ given by the three exceptional divisors.
The action of $\widetilde{\sigma_2}$ on $\mathrm{NS}(S)$ is given by
\[
\left\{
\begin{array}{llll}
(\widetilde{\sigma_2})_*\mathbf{e}_0=2\mathbf{e}_0-\mathbf{e}_1-\mathbf{e}_2-\mathbf{e}_3\\
(\widetilde{\sigma_2})_*\mathbf{e}_1=\mathbf{e}_0-\mathbf{e}_2-\mathbf{e}_3 \\
(\widetilde{\sigma_2})_*\mathbf{e}_2=\mathbf{e}_0-\mathbf{e}_1-\mathbf{e}_3 \\
(\widetilde{\sigma_2})_*\mathbf{e}_3=\mathbf{e}_0-\mathbf{e}_1-\mathbf{e}_2
\end{array}
\right.
\]
Then $(\widetilde{\sigma_2})_*$ coincides
on $\mathrm{NS}(S)$ with the 
reflection with respect to  $\mathbf{e}_0-\mathbf{e}_1-\mathbf{e}_2-\mathbf{e}_3$:
\[
(\widetilde{\sigma_2})_*(p)=p+\langle p,\mathbf{e}_0-\mathbf{e}_1-\mathbf{e}_2-\mathbf{e}_3\rangle\qquad\qquad\forall\,p\in\mathrm{NS}(S)
\]
Let us blow up all points of $S$; we thus obtain a basis of
$\mathrm{Z}(\mathbb{P}^2_\mathbb{C})$:
\[
\mathrm{Z}(\mathbb{P}^2_\mathbb{C})=\mathrm{NS}(S)\displaystyle\bigoplus_{p\in\mathcal{B}(S)}\mathbb{Z}\mathbf{e}_p.
\]

The isometry $(\sigma_2)_*$ of $\mathcal{Z}(\mathbb{P}^2_\mathbb{C})$
acts on $\mathrm{NS}(S)$ as the reflection 
$(\widetilde{\sigma_2})_*$ and permutes
each vector $\mathbf{e}_p$ with $\mathbf{e}_{\sigma_2(p)}$.
\end{eg}


\section{Types and degree growth}\label{sec:degreegrowth}

Consider an ample class $\mathbf{h}\in\mathrm{NS}(S,\mathbb{R})$
with self-intersection $1$. The \textsl{degree} of 
$\phi\in\mathrm{Bir}(S)$ with respect to the 
polarization $\mathbf{h}$ is defined by
\[
\deg_\mathbf{h}\phi=\langle\phi_*(\mathbf{h}), \mathbf{h}\rangle=\cosh(d(\mathbf{h},\phi_*\mathbf{h})).
\]
Note that if $S=\mathbb{P}^2_\mathbb{C}$ and 
$\mathbf{h}=\mathbf{e}_0$ is the class of a line, then
$\deg_\mathbf{h}\phi$ is the degree of~$\phi$
as defined in Chapter \ref{chapter:intro}.

\medskip

A birational map $\phi$ of a projective surface $S$ is
\begin{itemize}
\item[$\diamond$] \textsl{virtually isotopic to the identity}
\index{defi}{virtually isotopic to the identity (map)}
if there is a positive iterate $\phi^n$ of $\phi$ and a 
birational map $\psi\colon Z\dashrightarrow S$ such that 
$\psi^{-1}\circ\phi^n\circ\psi$ is an element of 
$\mathrm{Aut}(Z)^0$;

\item[$\diamond$] a 
\textsl{Jonqui\`eres twist}\index{defi}{Jonqui\`eres twist}
if $\phi$ preserves a one parameter family of rational 
curves on $S$, but $\phi$ is not virtually isotopic to the 
identity;

\item[$\diamond$] a 
\textsl{Halphen twist}\index{defi}{Halphen twist} if 
$\phi$ preserves a one parameter family of genus one 
curves on $S$, but $\phi$ is not virtually isotopic to
the identity.
\end{itemize}

Furthermore the Jonqui\`eres twists
(resp. Halphen twists) preserve
a unique fibration (\cite{DillerFavre}).

\begin{rem}
If $\phi$ is a Jonqui\`eres (resp. Halphen) twist, then, 
after conjugacy by a birational map 
$\psi\colon Z\dashrightarrow S$, $\phi$ permutes the 
fibers of a rational (resp. genus one) fibration 
$\pi\colon Z\to B$. If $z$ is the divisor class of the 
generic fiber of the fibration, then $z$ is an isotropic
vector in $\mathcal{Z}(S)$ fixed by $\phi_*$. In 
particular $\phi_*$ can not be loxodromic.
\end{rem}

Let $\mathcal{C}$ and $\mathcal{C}'$ be two smooth cubic curves in the complex
projective plane. By Bezout theorem $\mathcal{C}$ and $\mathcal{C}'$ intersect
in nine points denoted $p_1$, $p_2$, $\ldots$, $p_9$.
There is a pencil of cubic curves passing through these
nine points. Let us blow up $p_1$, $p_2$, $\ldots$, $p_9$. 
We get a rational surface $S$ with a fibration 
$\pi\colon S\to\mathbb{P}^1_\mathbb{C}$ whose fibers are
genus $1$ curves. 
More generally let us consider a pencil of curves of
degree $3m$ for $m\in\mathbb{Z}_+$, blow up its base-points
and denote by $S$ the surface we get. Such a pencil of 
genus $1$ curves is called a 
\textsl{Halphen pencil}\index{defi}{Halphen pencil},
and such a surface is called a 
\textsl{Halphen surface of index $m$}
\index{defi}{Halphen surface of index $m$}.

\begin{defi}
A surface $S$ is a \textsl{Halphen} one
\index{defi}{Halphen surface} if $\vert -mK_S\vert$ 
satisfies the three following properties
\begin{itemize}
\item[$\diamond$] it is one-dimensional, 

\item[$\diamond$] it has no fixed component, 

\item[$\diamond$] it is base-point free.
\end{itemize}
\end{defi}

According to \cite{CantatDolgachev} up to birational 
conjugacy
\begin{itemize}
\item[$\diamond$] every pencil of genus $1$ curves of $\mathbb{P}^2_\mathbb{C}$ is a Halphen pencil, 

\item[$\diamond$] Halphen surfaces are the only examples of rational elliptic surfaces.
\end{itemize}

\begin{lem}[\cite{Urech:ellipticsubgroups}]\label{lem:UrechHalphen}
Let $S$ be a Halphen surface. Let $\phi$ be an element of 
$\mathrm{Bir}(S)$ that preserves the Halphen pencil. 
Then $\phi$ belongs to $\mathrm{Aut}(S)$. 
\end{lem}

Up to conjugacy by birational maps every pencil of 
genus $1$ curves of $\mathbb{P}^2_\mathbb{C}$ is a
Halphen pencil and Halphen surfaces are the only 
examples of rational elliptic surfaces 
(\cite{CantatDolgachev}) so Lemma \ref{lem:UrechHalphen}
implies:

\begin{cor}
A subgroup $\mathrm{G}$ of 
$\mathrm{Bir}(\mathbb{P}^2_\mathbb{C})$ that 
preserves a pencil of genus~$1$ curves is conjugate 
to a subgroup of the automorphism group 
of some Halphen surface.
\end{cor}

\begin{proof}[Proof of Lemma \ref{lem:UrechHalphen}]
The Halphen pencil is defined by a multiple of the class
of the anticanonical divisor $-K_S$. As a result any 
birational map of a Halphen surface that preserves the 
Halphen fibration preserves the class of the canonical
divisor $K_S$. Assume by contradiction that $\phi$ is 
not an automorphism. Take a minimal resolution of $\phi$
\[
 \xymatrix{
     & Z\ar[rd]^\eta\ar[ld]_\pi & \\
    S\ar@{-->}[rr]_\phi & & S
  }
\]
Denote by $E_i$ and $F_i$ the total pull backs of 
the exceptional curves. On the one hand
\[
K_Z=\eta^*(K_S)+\sum E_i,
\] 
and on the other hand
\[
K_Z=\pi^*(K_S)+\sum F_i.
\]
The map $\phi$ preserves $K_S$, so $\eta^*(K_S)=\pi^*(K_S)$,
and hence $\sum E_i=\sum F_i$. By assumption $\phi$ is not
an automorphism, {\it i.e.} $\sum E_i$ contains at least 
one $(-1)$-curve~$E_k$. Hence both 
\[
E_k\cdot\Big(\sum E_i\Big)=-1
\]
and 
\[
E_k\cdot\Big(\sum F_i\Big)=-1
\]
hold. This implies that $E_k$ is contained in the support
of $\sum F_i$: contradiction with the minimality of the 
resolution.
\end{proof}

\begin{rem}
The automorphism groups of Halphen surfaces are 
studied in~\cite{Gizatullin} and in \cite{CantatDolgachev}.
\end{rem}

\smallskip

On the contrary Jonqui\`eres twists are not conjugate to
automorphisms of projective surfaces 
(\cite{DillerFavre, BlancDeserti:degree}).

\smallskip

Let $S$ be a projective complex surface with a polarization
$H$. Let $\phi\colon S\dashrightarrow S$ be a birational 
map. The 
\textsl{dynamical degree}\index{defi}{dynamical degree}
of $\phi$ is defined by
\[
\lambda(\phi)=\displaystyle\lim_{n\to +\infty}\deg_H(\phi^n)^{1/n}.
\]\index{not}{$\lambda(\phi)$}

\begin{defis}
An element $\phi$ of $\mathrm{Bir}(\mathbb{P}^2_\mathbb{C})$
is called \textsl{elliptic}\index{defi}{elliptic (birational map)}, 
(resp. \textsl{parabolic}\index{defi}{parabolic (birational map)}, 
resp. \textsl{loxodromic}\index{defi}{loxodromic (birational map)})
if the corresponding isometry $\phi_*$ is elliptic (resp. parabolic,
resp. loxodromic).
.
\end{defis}

The map $\phi$ is loxodromic if and only if $\lambda(\phi)>1$.
As a consequence when $\phi\in\mathrm{Bir}(\mathbb{P}^2_\mathbb{C})$,
$\lambda(\phi)>1$, the isometry $\phi_*$ preserves a 
unique geodesic line $\mathrm{Ax}(\phi)\subset\mathbb{H}^\infty$\index{not}{$\mathrm{Ax}(\phi)$}
called the \textsl{axis}\index{defi}{axis (of a 
birational map)} of $\phi$. This line is the 
intersection of $\mathbb{H}^\infty$ with a plane 
$P_\phi\subset\mathrm{Z}(\mathbb{P}^2_\mathbb{C})$
which intersects the isotropic cone of 
$\mathrm{Z}(\mathbb{P}^2_\mathbb{C})$ in two lines
$\mathbb{R}v^+_{\phi_*}$ and $\mathbb{R}v^-_{\phi_*}$ 
such that 
\[
\phi_*(p)=\lambda(\phi)^{\pm 1}p
\]
for all $p\in\mathbb{R}v^\pm_{\phi_*}$ (the lines
$\mathbb{R}v^+_{\phi_*}$ and $\mathbb{R}v^-_{\phi_*}$
correspond to $\omega(\phi)$ and $\alpha(\phi)$ 
with the notations of \S\ref{subsec:isometries}).

Take $\alpha\in\mathbb{R}v^-_{\phi_*}$ and 
$\omega\in\mathbb{R}v^+_{\phi_*}$ normalized so 
that $\langle\alpha,\,\omega\rangle=1$. The point
$p=\frac{\alpha+\omega}{\sqrt{2}}$ lies on 
$\mathrm{Ax}(\phi)$. Since 
$\phi_*(p)=\frac{\lambda(\phi)^{-1}\alpha+\lambda(\phi)\omega}{\sqrt{2}}$
one obtains
\begin{eqnarray*}
\exp(L(\phi_*))+\frac{1}{\exp(L(\phi_*))}&=& 2\mathrm{cosh}(d(p,\phi_*(p)))\\
&=& 2\langle p,\,\phi_*(p)\rangle\\
&=&\lambda(\phi)+\frac{1}{\lambda(\phi)}.
\end{eqnarray*}
The translation length is thus equal to 
$\log\lambda(\phi)$. Consequently $\lambda(\phi)$
does not depend on the polarization and is invariant
under conjugacy.

There is a correspondence between the dynamical behavior of 
a birational map~$\phi$ of $S$, in particular its degree, and 
the type of the induced isometry on~$\mathbb{H}^\infty$:

\begin{thm}[\cite{Gizatullin, DillerFavre, Cantat2}]\label{thm:dilfav}
Let $S$ be a smooth projective complex surface with a fixed
polarization $H$. Let $\phi\colon S\dashrightarrow S$ be a 
birational map. Then one of the following holds:
\begin{itemize}
\item[$\diamond$] $\phi$ is elliptic, $(\deg_H\phi^n)_n$ is 
bounded, and $\phi$ is virtually isotopic to the identity;

\item[$\diamond$] $\phi$ is parabolic and 

either $\deg_H\phi^n\sim cn$ for some positive constant $c$
and $\phi$ is a Jonqui\`eres twist;

or $\deg_H\phi^n\sim cn^2$ for some positive constant $c$
and $\phi$ is a Halphen twist;

\item[$\diamond$] $\phi$ is loxodromic and 
$\deg_H\phi^n=c\lambda(\phi)^n+O(1)$ for some positive 
constant~$c$.
\end{itemize}
\end{thm}

\begin{egs}
\begin{itemize}
\item[$\diamond$] Any birational map of finite order is 
elliptic. Any automorphism of $\mathbb{P}^2_\mathbb{C}$
is elliptic. Any element of the group 
\[
\big\{(z_0,z_1)\mapsto(\alpha z_0+P(z_1),\beta z_1+\gamma)\,\vert\,\alpha,\,\beta\in\mathbb{C}^*,\,\gamma\in\mathbb{C},\,P\in\mathbb{C}[z_1]\big\}
\]
is elliptic.

\item[$\diamond$] Any element of $\mathcal{J}$ of the form
\[
(z_0,z_1)\dashrightarrow\left(z_0,\frac{a(z_0)z_1+b(z_0)}{c(z_0)z_1+d(z_0)}\right)
\]
with $\frac{(\mathrm{tr}\,M)^2}{\det M}\in\mathbb{C}(z_0)\smallsetminus\mathbb{C}$ where 
\[
M=\left(
\begin{array}{cc}
a(z_0) & b(z_0)\\
c(z_0) & d(z_0)
\end{array}
\right)
\] 
is a Jonqui\`eres twist (\cite{CerveauDeserti:centralisateurs}).

\item[$\diamond$] Consider the family of birational self maps of
$\mathbb{P}^2_\mathbb{C})$
given in the affine chart $z_2=1$ by
\[
\phi_\varepsilon\colon(z_0,z_1)\dashrightarrow\left(z_1+1-\varepsilon,z_0\frac{z_1-\varepsilon}{z_1+1}\right).
\]
If 
\begin{itemize}
\item[$\diamond$] $\varepsilon=-1$, then $\phi_\varepsilon$ is elliptic;

\item[$\diamond$] $\varepsilon\in\{0,\,1\}$, then $\phi_\varepsilon$ is a 
Jonqui\`eres twist;

\item[$\diamond$] $\varepsilon\in \left\{\frac{1}{2},\,\frac{1}{3}\right\}$, then $\phi_\varepsilon$ is 
a Halphen twist;

\item[$\diamond$] $\varepsilon\in\displaystyle\bigcup_{k\geq 4}\frac{1}{k}$, 
then $\phi_\varepsilon$ is loxodromic.
\end{itemize}

This family has been introduced in \cite{DillerFavre}.

\item[$\diamond$] If 
$\phi\colon(z_0,z_1)\dashrightarrow(z_1,z_0+z_1^2)$, 
then $\deg(\phi^n)=(\deg\phi)^n=2^n$. If 
$\psi\colon(z_0,z_1)\dashrightarrow(z_0^2z_1,z_0z_1)$, then 
$\deg\psi^n\sim\left(\frac{3+\sqrt{5}}{2}\right)^n$; in 
particular $\deg(\psi^n)\not=(\deg\psi)^n$.

\item[$\diamond$] Let us finish with a more geometric example.
Consider the elliptic curve $E=\faktor{\mathbb{C}}{\mathbb{Z}[\mathbf{i}]}$.
The linear action of the group $\mathrm{GL}(2,\mathbb{Z}[\mathbf{i}])$
on the complex plane preserves the lattice 
$\mathbb{Z}[\mathbf{i}]\times\mathbb{Z}[\mathbf{i}]$. 
This yields to an action of $\mathrm{GL}(2,\mathbb{Z}[\mathbf{i}])$
by regular automorphisms on the abelian surface $S=E\times E$. 
Since this action commutes with 
$(z_0,z_1)\mapsto(\mathbf{i}z_0,\mathbf{i}z_1)$ one gets a 
morphism from $\mathrm{PGL}(2,\mathbb{Z}[\mathbf{i}])$ to
$\mathrm{Aut}\left(\faktor{S}{(z_0,z_1)}\mapsto(\mathbf{i}z_0,\mathbf{i}z_1)\right)$.
As $\faktor{S}{(z_0,z_1)}\mapsto(\mathbf{i}z_0,\mathbf{i}z_1)$
is rational one obtains an embedding of $\mathrm{PGL}(2,\mathbb{Z}[\mathbf{i}])$
into $\mathrm{Bir}(\mathbb{P}^2_\mathbb{C})$. 
\end{itemize}
\end{egs}

Any element virtually isotopic to the identity is 
\textsl{regularizable}\index{defi}{regularizable (birational map)}, that is 
birationally conjugate to an automorphism. What can we say
about two birational maps virtually isotopic to the 
identity ? We will see that if they commute they are 
simultaneously regularizable. Before proving it let us 
introduce a new notion.

\begin{defis}
An element $\phi\in\mathrm{Bir}(\mathbb{P}^2_\mathbb{C})$ is 
\textsl{algebraically stable}\index{defi}{algebraically stable 
(birational map of $\mathbb{P}^2_\mathbb{C}$)} if 
$\deg\phi^n=(\deg\phi)^n$ for all $n\geq 0$.

More generally if $S$ is a compact complex surface, then 
$\phi\in\mathrm{Bir}(S)$ is 
\textsl{algebraically stable}\index{defi}{algebraically stable (birational map of a surface)} 
if $(\phi^*)^n=(\phi^n)^*$ for all $n\geq 0$.
\end{defis}

A geometric characterization of algebraically stable maps is the 
following: $\phi\in\mathrm{Bir}(S)$ is algebraically stable if 
and only if there is no curve $C\subset S$ such that 
$\phi^k(C)\in\mathrm{Ind}(\phi)$ for some integer $k$.
Let us give an idea of the fact that this geometric 
characterization is equivalent to the Definition when 
$S=\mathbb{P}^2_\mathbb{C}$. If 
$\phi^k\big(\mathcal{C}\smallsetminus\mathrm{Ind}(\phi)\big)\subset\mathrm{Ind}(\phi)$, then all the components of $\phi\circ\phi^k$
have a common factor that defines the equation of 
$\mathcal{C}$. Then $\deg(\phi\circ\phi^k)<(\deg\phi)(\deg\phi^k)$.
The converse holds.

Diller and Favre proved the following result:

\begin{pro}[\cite{DillerFavre}]\label{pro:DillerFavre}
Let $S$ be a compact complex surface, and let $\phi$ be a birational
self map of $S$. There exists a composition of finitely many
point blow-ups that lifts $\phi$ to an algebraically stable map.
\end{pro}

Before giving the proof, let us give its idea. Assume
that $\phi$ is not algebraically stable. In other words
there exist a curve $C\subset S$ and an integer $k$ such 
that $C$ is blown down onto $p_1$ and $p_k=\phi^{k-1}(p_1)$
belongs to $\mathrm{Ind}(\phi)$. 
The idea of Diller and Favre to get an algebraically 
stable map is the following: after blowing up the points
$p_i=\phi^i(p_1)$, $i=1$, $\ldots$, $k$, the orbit of~$C$ 
consists of curves. Doing this 
for any element of $\mathrm{Exc}(\phi)$ whose an 
iterate belongs to $\mathrm{Ind}(\phi)$ one gets the 
statement (note that the cardinal of $\mathrm{Exc}(\phi)$ 
is finite, so the process ends).

\begin{proof}
Let us write $\phi$ as follows
$\phi=\phi_n\circ\phi_{n-1}\circ\ldots\circ\phi_1$
where 
\begin{itemize}
\item[$\diamond$] $\phi_i\colon S_{i-1}\to S_i$;

\item[$\diamond$] $S_0=S_n=S$;

\item[$\diamond$] and 
\begin{itemize}
\item[(i)] either $\phi_i$ blows up a point 
$p_i=\mathrm{Ind}(\phi_i)\in S_i$, and we denote by
$V_{i+1}=\mathrm{Exc}(\phi_i^{-1})\subset S_{i+1}$
the exceptional divisor of $\phi_i^{-1}$;

\item[(ii)] or $\phi_i$ blows down the exceptional
divisor $E_i\subset S_i$; in this case we set 
$q_{i+1}:=\phi_i(E_i)\in S_{i+1}$.
\end{itemize}
\end{itemize}
For any $j\in\mathbb{N}$ set $S_j:=S_{\text{$j$ mod $n$}}$
and $\phi_j:=\phi_{\text{$j$ mod $n$}}$.

Assume that $\phi$ is not algebraically stable. Then there 
exist integers $1\leq i\leq N$ such that $\phi_i$ blows
down $E_i$ and 
\[
\phi_{N-1}\circ\phi_{N-2}\circ\ldots\circ\phi_i(E_i)=p_N\in\mathrm{Ind}(\phi_N).
\]
Choosing a pair $(i,N)$ of minimal length we can assume that
for all $i<j\leq N$
\[
m_j:=\phi_j\circ\phi_{j-1}\circ\ldots\circ\phi_i(E_i)=\phi_j\circ\phi_{j-1}\circ\ldots\circ\phi_{i+1}(q_{i+1})
\]
does not belong to $\mathrm{Ind}(\phi_i)\cup\mathrm{Exc}(\phi_i)$.

First blow up $S_N$ at $m_N=p_N$. Then 
\begin{itemize}
\item[$\diamond$] $\phi_N$ lifts to a biholomorphism $\widehat{\phi}_N$
of $\mathrm{Bl}_{p_N}S_N$;

\item[$\diamond$] $\widehat{\phi}_{N-1}$ either blows up the two 
distinct points $m_{N-1}$ and $p_{N-1}$ or blows up $m_{N-1}$ and 
blows down $E_{N-1}\notin m_{N-1}$;

\item[$\diamond$] $\sum\mathrm{Card}\big(\phi_j\big(\mathrm{Exc}(\phi_j)\big)\big)=\sum\mathrm{Card}\big(\widehat{\phi}_j\big(\mathrm{Exc}(\widehat{\phi}_j)\big)\big)$.
\end{itemize}
Remark that modifying $S_N$ means modifying $S_{N+n}$, $S_{N-n}$, 
$\ldots$ nevertheless blowing up a point $m_j$ does not interfere
with the behavior of the map $\phi_j$ around $m_{N+n}$, $m_{N-n}$,
$\ldots$ (indeed if $j_1=j_2$ mod $n$ but $j_1\not=j_2$, then the 
points $m_{j_1}$, $m_{j_2}$ of $S_1=S_2$ are distinct), and 
these points can be blow up independently.

Similarly blow up $m_{N-1}$, $m_{N-2}$, $\ldots$, $m_{i+2}$. At 
each step $\sum\mathrm{Card}(\phi_j(\mathrm{Exc}(\phi_j)))$ 
remains constant. Let us finish by blowing up 
$m_{i+1}=\phi_i(E_i)$; the situation is then different: $\phi_i$
becomes a biholomorphism $\widehat{\phi}_i$. The number of 
components of $\mathrm{Exc}(\phi_i)$ thus reduces from $1$ to $0$.
As a consequence
\begin{equation}\label{eq:dim}
\sum\mathrm{Card}\big(\widehat{\phi}_j\big(\mathrm{Exc}(\widehat{\phi}_j)\big)\big)=
\sum\mathrm{Card}\big(\phi_j\big(\mathrm{Exc}(\phi_j)\big)\big)-1.
\end{equation}
Repeating finitely many times the above argument either we
produce an algebraically stable map 
$\widehat{\phi}=\widehat{\phi}_N\circ\widehat{\phi}_{N-1}\circ
\ldots\circ\widehat{\phi}_1$, or thanks to (\ref{eq:dim}) we
eleminate all exceptional components of the factors of $\phi$.
In both cases we get an algebraically stable map.
\end{proof}

\begin{lem}[\cite{Deserti:IMRN}]\label{lem:commas}
Let $\phi$, $\psi$ be two birational self maps of a compact
complex surface~$S$.
Assume that $\phi$ and $\psi$ are both virtually isotopic to the
identity. Assume that $\phi$ and $\psi$ commute. 

There exist a surface $Y$ and a birational map
$\zeta\colon\ Y\dashrightarrow S$  such that
\begin{itemize}
\item[$\diamond$] $\zeta^{-1}\circ\phi^\ell\circ\zeta\in\mathrm{Aut}(Y)^0$
  for some integer $\ell$,
\item[$\diamond$] $\zeta^{-1}\circ\psi\circ\zeta$ is algebraically stable.
\end{itemize}
\end{lem}

\begin{proof}
Since $\phi$ is virtually isotopic to the identity we can assume
that up to birational conjugacy and finite index $\phi$ is an
automorphism of $S$.
Let $N(\psi)$ be the minimal number of blow-ups needed to 
make $\psi$ algebraically stable (such a $N(\psi)$ exists
according to Proposition \ref{pro:DillerFavre}). If
$N(\psi)=0$, then $\zeta=\mathrm{id}$ suits. Assume that
Lemma \ref{lem:commas} holds when $N(\psi)\leq j$. Consider
a pair $(\phi,\psi)$ of birational self maps of $S$ such that 
\begin{itemize}
\item[$\diamond$] $\phi$ and $\psi$ are both virtually isotopic 
to the identity,
\item[$\diamond$] $\phi$ and $\psi$ commute,
\item[$\diamond$] $N(\psi)=j+1$.
\end{itemize}
Since $\psi$ is not algebraically stable there exists a curve 
$C$ blown down by $\psi$ and such that~$\psi^q(C)$ is a point 
of indeterminacy $p$ of $\psi$ for some integer $q$. The maps 
$\psi$ and~$\phi$ commute, so an iterate $\phi^k$ of $\phi$ 
fixes the irreducible components of $\mathrm{Ind}(\psi)$.
Let us blow up $p$ via $\pi$. On the one hand 
$\pi^{-1}\circ\phi^k\circ\pi$ is an automorphism because 
$p$ is fixed by $\phi^k$ and on the other hand 
$N(\pi^{-1}\circ\psi\circ\pi)=j$. One can thus conclude 
by induction that there exist a surface $Y$ and a birational
map $\zeta\colon Y\dashrightarrow S$ such that
$\zeta^{-1}\circ\phi^\ell\circ\zeta\in\mathrm{Aut}(Y)^0$ for
some integer $\ell$ and $\zeta^{-1}\circ\psi\circ\zeta$ is
algebraically stable.
\end{proof}

\begin{pro}[\cite{Deserti:IMRN}]\label{pro:commas}
Let $\phi$, $\psi$ be two birational self maps of a surface $S$.
Assume that $\phi$ and $\psi$ are both virtually isotopic to the
identity. Assume that $\phi$ and~$\psi$ commute. 

Then there exist a surface $Z$ and a birational map 
$\pi\colon Z\dashrightarrow S$ such that 
\begin{itemize}
\item[$\diamond$] $\pi^{-1}\circ\phi\circ\pi$ and 
$\pi^{-1}\circ\psi\circ\pi$ belong to $\mathrm{Aut}(Z)$;

\item[$\diamond$] $\pi^{-1}\circ\phi^k\circ\pi$ and 
$\pi^{-1}\circ\psi^k\circ\pi$ belong to $\mathrm{Aut}(Z)^0$
for some integer $k$.
\end{itemize}
\end{pro}

\begin{proof}
By assumption there exist a surface $\widetilde{S}$, a birational map 
$\eta\colon\widetilde{S}\dashrightarrow S$ and an integer~$n$ such 
that $\eta^{-1}\circ\phi\circ\eta$ belongs to $\mathrm{Aut}(S)$ and 
$\eta^{-1}\circ\phi^n\circ\eta$ belongs to $\mathrm{Aut}(S)^0$.
Let us now work on $\widetilde{S}$; to simplify denote by $\phi$
the automorphism 
$\eta^{-1}\circ\phi^n\circ\eta$ and by $\psi$ the birational map 
$\eta^{-1}\circ\psi\circ\eta$. 

According to Lemma \ref{lem:commas} there exist a surface $Y$, 
a birational map $\upsilon\colon Y\dashrightarrow S$ and an integer~$\ell$ such that $\zeta^{-1}\circ\widetilde{\phi}^\ell\circ\zeta$
belongs to $\mathrm{Aut}(Y)^0$ and
$\zeta^{-1}\circ\widetilde{\psi}\circ\zeta$ is algebraically 
stable.

Set $\overline{\phi}=\zeta^{-1}\circ\widetilde{\phi}^i\circ\zeta$
and $\overline{\psi}=\zeta^{-1}\circ\widetilde{\psi}\circ\zeta$.
To get an automorphism from~$\overline{\psi}$ let us blow
down curves in $\mathrm{Exc}(\overline{\psi}^{-1})$. But 
curves blown down by $\overline{\psi}^{-1}$ are of 
self-intersection $<0$ and~$\overline{\phi}$ fixes such 
curves since $\overline{\phi}$ is isotopic to the identity.
We conclude by using the fact that 
$\mathrm{Card}\big(\mathrm{Exc}(\overline{\psi}^{-1})\big)$ is finite.
\end{proof}


\section{On the hyperbolicity of graphs associated to the Cremona group}

To reinforce the analogy between the mapping class group and the 
plane Cremona group Lonjou looked for a graph analogous to 
the curve graph and such that the Cremona group acts on it trivially 
in \cite{Lonjou:epiga}. 

A candidate is the graph introduce by Wright 
(Chapter \ref{chapter:gen} \S 
\ref{subsec:wright} and \cite{Wright}).

As we have recalled in Chapter \ref{chapter:gen} 
\S \ref{subsec:wright} the complex $C$ is a 
simplicial complex of dimension $2$ and $1$-connected on which 
$\mathrm{Bir}(\mathbb{P}^2_\mathbb{C})$ acts. Since Lonjou is 
interested in the Gromov hyperbolicity property, she is only 
interested in the $1$-skeleton of~$C$. She proved that the diameter 
of this non-locally finite graph is infinite (\cite[Corollary 
2.7]{Lonjou:epiga}). She then focuses on the following question "Is 
this graph Gromov hyperbolic ?"\footnote{Minosyan and Osin note that
if the answer to this question is yes, the results of 
\cite{DahmaniGuirardelOsin} allow
to give a new proof of the non-simplicity of
$\mathrm{Bir}(\mathbb{P}^2_\mathbb{C})$ (\emph{see}
\cite{MinasyanOsin1, MinasyanOsin2}).} The answer is no:

\begin{thm}[\cite{Lonjou:epiga}]
The Wright graph is not Gromov hyperbolic.
\end{thm}

The first point of the proof is to note that the 
Wright graph is quasi-isometric to a graph related 
to the system of generators of 
$\mathrm{Bir}(\mathbb{P}^2_\mathbb{C})$ given by 
$\mathrm{PGL}(3,\mathbb{C})$ and the Jonqui\`eres
maps. It is an analogue of the Cayley graph in the
case of a finitely generated group. The vertices of 
this graph are the elements of 
$\mathrm{Bir}(\mathbb{P}^2_\mathbb{C})$ modulo
pre-composition by an automorphism of $\mathbb{P}^2_\mathbb{C}$.
An edge connects two vertices if there exists a 
Jonqui\`eres map that permutes the two vertices. 
The distance between two vertices $\phi$, $\psi$
in $\mathrm{Bir}(\mathbb{P}^2_\mathbb{C})$
is the minimal number of Jonqui\`eres maps needed
to decompose $\psi^{-1}\circ\phi$ (in 
\cite{BlancFurter:length} Blanc and Furter called this 
integer the translation length of $\psi^{-1}\circ\phi$. 
They gave an algorithm to compute this length. 
They also got that the diameter of the Wright graph
is infinite).

The second point is to prove that this graph contains
a subgraph quasi-isometric to $\mathbb{Z}^2$
(\emph{see} \cite[Theorem 2.12]{Lonjou:epiga}).
She took two Halphen twists that commute. They generate
a subgroup isomorphic to $\mathbb{Z}^2$. Using some
results of \cite{BlancFurter:length} she established
that the action of this subgroup on one of the 
vertices of the graph induces the desired 
graph\footnote{The Cayley graph of the modular group 
of a compact surface of genus $g\geq 2$ is not Gromov
hyperbolic; indeed, this group has subgroups 
isomorphic to $\mathbb{Z}^2$ (for instance 
generated by two Dehn twists along two disjoint 
closed curves).}.

Then Lonjou constructed two graphs associated to a 
Voron\"{i} tessellation of the Cremona group
introduced in \cite{Lonjou:publmath}; she proved
that 
\begin{itemize}
\item[$\diamond$] one of these graphs is quasi-isometric
to the Wright graph;

\item[$\diamond$] the second one is Gromov hyperbolic.
\end{itemize}


\chapter{Algebraic subgroups of the Cremona group}\label{Chapter:algebraicsubgroup}

\bigskip
\bigskip

The first section of this chapter deals with the algebraic structure of the
$n$-dimensional Cremona group, the fact that it is 
not an algebraic group of infinite dimension if $n\geq 2$, 
the obstruction to this, which is of a topological nature. 
By contrast, the existence of a Euclidean 
topology on the Cremona group which extends that 
of its classical subgroups and makes it a topological 
group is recalled. More precisely
in \cite{Shafarevich} Shafarevich asked

\begin{center}
\begin{fmpage}{10cm}
"Can one introduce a universal structure of an infinite dimensional group
in the group of all automorphisms (resp. all birational automorphisms)
of arbitrary algebraic variety ?"
\end{fmpage}
\end{center}

We will see that the answer to this question is no 
(\cite{BlancFurter}). For any 
algebraic va\-riety $V$ defined over $\mathbb{C}$ 
there is a natural 
notion of families of elements of $\mathrm{Bir}(\mathbb{P}^n_\mathbb{C})$
parameterized by $V$.
These are maps 
$V(\mathbb{C})\to\mathrm{Bir}(\mathbb{P}^n_\mathbb{C})$
compatible with the 
structures of algebraic varieties. Note that 
$\mathrm{Bir}(\mathbb{P}^1_\mathbb{C})\simeq\mathrm{PGL}(2,\mathbb{C})$ 
and fa\-milies
$V\dashrightarrow\mathrm{Bir}(\mathbb{P}^1_\mathbb{C})$
correspond to morphisms of algebraic varieties. If
$n\geq 2$ the set 
$\mathrm{Bir}_d(\mathbb{P}^n_\mathbb{C})$ of all
birational maps of 
$\mathbb{P}^n_\mathbb{C}$ of degree $d$ has the structure of an algebraic 
variety defined over $\mathbb{C}$ such that the families 
$V\to\mathrm{Bir}_d(\mathbb{P}^n_\mathbb{C})$ correspond to morphisms of
algebraic varieties (\cite{BlancFurter}). So  
$\mathrm{Bir}(\mathbb{P}^n_\mathbb{C})$ decomposes into a disjoint 
infinite union of algebraic varieties, having unbounded dimension.
Blanc and Furter established the following statement:

\begin{thm}[\cite{BlancFurter}]\label{thm:bf}
Let $n\geq 2$. There is no structure of algebraic variety of infinite
dimension on $\mathrm{Bir}(\mathbb{P}^n_\mathbb{C})$ such that families
$V\to\mathrm{Bir}(\mathbb{P}^n_\mathbb{C})$ would correspond to morphisms
of algebraic varieties.
\end{thm}

The lack of structure come from the degeneration of maps of degree $d$ 
into maps of smaller degree. A family of birational self maps of 
$\mathbb{P}^2_\mathbb{C}$ of degree $d$ which depends on a parameter
$t$ may degenerate for certain values of $t$ onto a non-reduced 
expression of the type $P\,\mathrm{id}=P(z_0:z_1:z_2)$ where $P$ denotes
an homogeneous polynomial of degree $d-1$. Consider for instance the 
family 
\begin{small}
\begin{eqnarray*}
\phi_{a,b,c}&\colon&(z_0:z_1:z_2)\dashrightarrow \\
& & \hspace{0.2cm}\big(z_0(az_2^2+cz_0z_2+bz_0^2):z_1(az_2^2+(b+c)z_0z_2+(a+b)z_0^2):z_2(az_2^2+cz_0z_2+bz_0^2)\big)
\end{eqnarray*}
\end{small}
parameterized by the nodal plane cubic $a^3+b^3=abc$.
The family $(\phi_{a,b,c})$ is globally defined by formulas of degree $3$,
but each element $\phi_{a,b,c}$ has degree $\leq 2$ and there is no global
para\-meterization by homogeneous formulas of degree $2$. 
In fact the obstruction to a positive answer to Shafarevich question 
comes only from the topology:

\begin{thm}[\cite{BlancFurter}]\label{thm:bf2}
There is no $\mathbb{C}$-algebraic variety of infinite dimension that is 
homeo\-morphic to $\mathrm{Bir}(\mathbb{P}^n_\mathbb{C})$.
\end{thm}

In $2010$ in the question session of the workshop "Subgroups of the Cremona
group" in Edinburgh, Serre asked the following question

\begin{center}
\begin{fmpage}{10cm}
"Is it possible to introduce such topology on 
$\mathrm{Bir}(\mathbb{P}^2_\mathbb{C})$ that is compatible with 
$\mathrm{PGL}(3,\mathbb{C})$ and 
$\mathrm{PGL}(2,\mathbb{C})\times\mathrm{PGL}(2,\mathbb{C})$ ?"
\end{fmpage}
\end{center}

We will see that Blanc and Furter 
gave a positive answer to this question:

\begin{thm}[\cite{BlancFurter}]\label{thm:bf3}
Let $n\geq 1$ be an integer. There is a natural topology on 
$\mathrm{Bir}(\mathbb{P}^n_\mathbb{C})$, called the Euclidean topology,
such that:
\begin{itemize} 
\item[$\diamond$] $\mathrm{Bir}(\mathbb{P}^n_\mathbb{C})$, endowed with 
the Euclidean topology, is a Hausdorff topological group, 
\item[$\diamond$] the restriction of the Euclidean topology to
algebraic subgroups in particular to $\mathrm{PGL}(n+1,\mathbb{C})$ and
$\mathrm{PGL}(2,\mathbb{C})^n$ is the classical Euclidean topology.
\end{itemize}
\end{thm}

In the literature an algebraic subgroup
$\mathrm{G}$ of $\mathrm{Bir}(V)$ 
corresponds to taking an algebraic 
group~$\mathrm{G}$ and a morphism
$\mathrm{G}\to\mathrm{Bir}(V)$ that
is a group morphism and whose
schematic kernel is trivial.
We will see that in the case of 
$V=\mathbb{P}^n_\mathbb{C}$ one 
can give a more intrinsic definition
(Corollary \ref{cor:agree})
which corresponds to taking closed 
subgroups of 
$\mathrm{Bir}(\mathbb{P}^n_\mathbb{C})$
of bounded degree and that these 
two definitions agree (Lemma \ref{lem:agree}).

\smallskip

An element 
$\phi\in\mathrm{Bir}(\mathbb{P}^n_\mathbb{C})$ 
is algebraic if it is contained in an algebraic 
subgroup~$\mathrm{G}$ of 
$\mathrm{Bir}(\mathbb{P}^n_\mathbb{C})$. 
It is equivalent to say that 
the sequence $(\deg\phi^n)_{n\in\mathbb{N}}$ is 
bounded. According to \cite{BlancFurter} the 
group $\mathrm{G}$ is thus an affine algebraic group. 
As a consequence~$\phi$ decomposes as 
$\phi=\phi_s\circ\phi_u$ where $\phi_s$ is 
a semi-simple element of $\mathrm{G}$ and $\phi_u$
an unipotent element of $\mathrm{G}$. This 
decomposition does not depend on $\mathrm{G}$ 
(\emph{see} \cite{Popov}).
In particular there is a natural notion
of semi-simple and unipotent elements 
of~$\mathrm{Bir}(\mathbb{P}^n_\mathbb{C})$. 
As we will see $\mathrm{G}$ could even by chosen to 
be the abelian algebraic subgroup 
$\overline{\big\{\phi^i\,\vert\,i\in\mathbb{Z}\big\}}$ 
of $\mathrm{Bir}(\mathbb{P}^n_\mathbb{C})$. 
In all linear algebraic groups the set of 
unipotent elements is closed; Popov
asked if it is the case in the context
of the Cremona group. A natural 
and related question is the following one:
is the set 
$\mathrm{Bir}(\mathbb{P}^n_\mathbb{C})_{\text{alg}}$
of algebraic elements of
$\mathrm{Bir}(\mathbb{P}^n_\mathbb{C})$
closed ? The second section deals with the answers 
to these questions (Theorem \ref{thm:Blancelalg}).

In the third section the classification 
of maximal algebraic subgroups of the 
plane Cremona group is
given.

In the fourth section we give a sketch of the 
proof of the fact that $\mathrm{Bir}(\mathbb{P}^n_\mathbb{C})$
is topologically simple when endowed with the 
Zariski topology, {\it i.e.} it
does not contain any non-trivial closed 
normal strict subgroup. The main 
ingredients of the proof are some clever
deformation arguments.

The fifth section is devoted to a modern proof
of the regularization theorem of Weil
which says that for every rational action $\rho$ of 
an algebraic group $\mathrm{G}$ on a variety~$X$ 
there exist a variety~$Y$ with a regular
action $\mu$ of $\mathrm{G}$ and a dominant rational 
map $\phi\colon X\dashrightarrow Y$ with the following
properties: for any $(g,p)\in\mathrm{G}\times X$ 
such that
\begin{itemize}
\item[$\diamond$] $\rho$ is defined in 
$(g,p)$;

\item[$\diamond$] $\phi$ is defined 
in $p$ and $\rho(g,p)$;

\item[$\diamond$] $\mu$ is defined in 
$(g,\phi(p))$
\end{itemize}
we have $\phi(\rho(g,p))=\mu(g,\phi(p))$.

\bigskip
\bigskip


\section{Topologies and algebraic subgroups of $\mathrm{Bir}(\mathbb{P}^n_\mathbb{C})$}\label{sec:alg}

\subsection{Zariski topology}

Take an irreducible algebraic variety $V$. A \textsl{family of birational 
self maps of $\mathbb{P}^n_\mathbb{C}$ parameterized by $V$}\index{defi}
{family (of birational maps)} is a birational self map
\[
\varphi\colon V\times\mathbb{P}^n_\mathbb{C}\dashrightarrow V\times\mathbb{P}^n_\mathbb{C}
\]
such that
\begin{itemize}
\item[$\diamond$] $\varphi$ determines an isomorphism between two open 
subsets $\mathcal{U}$ and $\mathcal{V}$ of $V\times\mathbb{P}^n_\mathbb{C}$
 such that the first projection $\mathrm{pr}_1$ maps both $\mathcal{U}$ and 
$\mathcal{V}$ surjectively onto $V$,

\item[$\diamond$] $\varphi(v,x)=\big(v,\mathrm{pr}_2(\varphi(v,x))\big)$ where $\mathrm{pr}_2$ 
denotes the second projection; hence each $\varphi_v=\mathrm{pr}_2(\varphi(v,\cdot))$ 
is a birational self map of $\mathbb{P}^n_\mathbb{C}$.
\end{itemize}
The map $v\mapsto \varphi_v$ is called a \textsl{morphism}\index{defi}{morphism (from a parameter space to $\mathrm{Bir}(\mathbb{P}^n_\mathbb{C})$)}
from the parameter space $V$ to $\mathrm{Bir}(\mathbb{P}^n_\mathbb{C})$.

A subset $S\subset\mathrm{Bir}(\mathbb{P}^n_\mathbb{C})$ is 
\textsl{closed}\index{defi}{closed subset (of $\mathrm{Bir}(\mathbb{P}^n_\mathbb{C})$)} 
if for any algebraic variety $V$ and any morphism 
$V\to\mathrm{Bir}(\mathbb{P}^n_\mathbb{C})$ its preimage is closed.

This yields a topology on $\mathrm{Bir}(\mathbb{P}^n_\mathbb{C})$ 
called the \textsl{Zariski topology}\index{defi}{Zariski topology}.

\begin{rem}\label{rem:homeo}
For any $\phi\in\mathrm{Bir}(\mathbb{P}^n_\mathbb{C})$
the maps from $\mathrm{Bir}(\mathbb{P}^n_\mathbb{C})$ into itself 
given by 
\begin{align*}
&\psi\mapsto \psi\circ\varphi, && \psi\mapsto \varphi\circ\psi, 
&& \psi\mapsto\psi^{-1} 
\end{align*}
are homeomorphisms of $\mathrm{Bir}(\mathbb{P}^n_\mathbb{C})$ with 
respect to the Zariski topology.

Indeed let $V$ be an irreducible algebraic 
variety. If $f$, 
$g\colon V\times\mathbb{P}^n_\mathbb{C}\to V\times\mathbb{P}^n_\mathbb{C}$ 
are two $V$-birational maps inducing 
morphisms 
$V\to\mathrm{Bir}(\mathbb{P}^n_\mathbb{C})$,
then $f\circ g$ and $f^{-1}$ are again
$V$-birational maps that induce morphisms
$V\to\mathrm{Bir}(\mathbb{P}^n_\mathbb{C})$.
\end{rem}

Let 
$\mathrm{Bir}_{\leq d}(\mathbb{P}^n_\mathbb{C})$\index{not}{$\mathrm{Bir}_{\leq d}(\mathbb{P}^n_\mathbb{C})$}
 (resp. 
$\mathrm{Bir}_d(\mathbb{P}^n_\mathbb{C})$\index{not}{$\mathrm{Bir}_d(\mathbb{P}^n_\mathbb{C})$}) 
be the set of elements of $\mathrm{Bir}(\mathbb{P}^n_\mathbb{C})$ of degree
$\leq d$ (resp. of degree $d$); we have the following increasing sequence
\[
\mathrm{Aut}(\mathbb{P}^n_\mathbb{C})=\mathrm{Bir}_{\leq 1}(\mathbb{P}^n_\mathbb{C})\subseteq\mathrm{Bir}_{\leq 2}(\mathbb{P}^n_\mathbb{C})\subseteq\mathrm{Bir}_{\leq 3}(\mathbb{P}^n_\mathbb{C})\subseteq\ldots
\]
whose union gives the Cremona group. We will see that 
$\mathrm{Bir}_{\leq d}(\mathbb{P}^n_\mathbb{C})$ is closed in~$\mathrm{Bir}(\mathbb{P}^n_\mathbb{C})$ and the topology of 
$\mathrm{Bir}(\mathbb{P}^n_\mathbb{C})$ is the inductive
topology induced by the above sequence. As a result it 
suffices to describe the topology of 
$\mathrm{Bir}_{\leq d}(\mathbb{P}^n_\mathbb{C})$ to understand
the topology of $\mathrm{Bir}(\mathbb{P}^n_\mathbb{C})$.

Take a positive integer $d$. Let $W_d$\index{not}{$W_d$} be the set of 
equivalence classes 
of non-zero $(n+1)$-uples $(\phi_0,\phi_1,\ldots,\phi_n)$ of homogeneous
polynomials $\phi_i\in\mathbb{C}[z_0,z_1,\ldots,z_n]$ of degree $d$ where
$(\phi_0,\phi_1,\ldots,\phi_n)$ is equivalent to 
$(\lambda\phi_0,\lambda\phi_1,\ldots,\lambda\phi_n)$ for any 
$\lambda\in\mathbb{C}^*$. We denote by $(\phi_0:\phi_1:\ldots:\phi_n)$
the equivalence class of $(\phi_0,\phi_1,\ldots,\phi_n)$. Let 
$H_d\subseteq W_d$\index{not}{$H_d$} be the set of elements 
$\phi=(\phi_0:\phi_1:\ldots:\phi_n)\in W_d$ such that the rational 
map $\psi_\phi\colon\mathbb{P}^n_\mathbb{C}\dashrightarrow\mathbb{P}^n_\mathbb{C}$
given by 
\[
(z_0:z_1:\ldots:z_n)\dashrightarrow \big(\phi_0(z_0,z_1,\ldots,z_n):\phi_1(z_0,z_1,\ldots,z_n):\ldots:\phi_n(z_0,z_1,\ldots,z_n)\big)
\]
is birational. The map
\begin{align*}
&H_d\to\mathrm{Bir}(\mathbb{P}^n_\mathbb{C}) && \phi\mapsto\psi_\phi
\end{align*}
is denoted by $\pi_d$\index{not}{$\pi_d$}.

\begin{lem}[\cite{BlancFurter}]\label{lem:propbf}
The following properties hold:
\begin{itemize}
\item[$\diamond$] The set $W_d$ is isomorphic to
$\mathbb{P}^k_\mathbb{C}$ where 
$k=(n+1)\binom{d+n}{d}-1$.

\item[$\diamond$] The set $H_d$ is locally closed in 
$W_d$; thus it inherits from $W_d$ the 
structure of an algebraic variety.

\item[$\diamond$] The map $\pi_d\colon H_d\to\mathrm{Bir}(\mathbb{P}^n_\mathbb{C})$
is a morphism, and $\pi_d(H_d)$ is the set 
$\mathrm{Bir}_{\leq d}(\mathbb{P}^n_\mathbb{C})$.

\item[$\diamond$] For all $\phi$ in $\mathrm{Bir}_{\leq d}(\mathbb{P}^n_\mathbb{C})$ 
the set $\pi_d^{-1}(\phi)$ is closed in $W_d$, so in $H_d$
as well.

\item[$\diamond$] If $S\subset H_\ell$ $(\ell\geq 1)$ is closed, then 
$\pi_d^{-1}(\pi_\ell(S))$ is closed in $H_d$.
\end{itemize}
\end{lem}

Hence $W_d$ and $H_d$ are naturally algebraic varieties, 
 $\mathrm{Bir}_d(\mathbb{P}^n_\mathbb{C})$ also, but not 
$\mathrm{Bir}_{\leq d}(\mathbb{P}^n_\mathbb{C})$.

\begin{proof}[Proof of Lemma \ref{lem:propbf}]
\begin{itemize}
\item[$\diamond$] The set of homogeneous polynomials of 
degree $d$ in $(n+1)$ variables is a 
$\mathbb{C}$-vector space of dimension $\binom{d+n}{d}$;
this implies the first assertion.

\item[$\diamond$] Denote by $Y\subseteq W_{d^{n-1}}\times W_d$
the set defined by
\[
\big\{(\varphi,\phi)\in W_{d^{n-1}}\times W_d\,\vert\, \varphi\circ \phi=P\,\mathrm{id}\text{ for some $P\in\mathbb{C}[z_0,z_1,\ldots,z_n]_{d^n}$}\big\}.
\]

If $P$ is nonzero, then the rational maps $\psi_\phi$ and 
$\psi_\varphi$ are birational and inverses of each other.

If $P$ is zero, then $\psi_\phi$ contracts the entire
set $\mathbb{P}^n_\mathbb{C}$ onto a strict 
subvariety included in the set
$\big\{\varphi_1=\varphi_2=\ldots=\varphi_n=0\big\}$.

In particular for any pair $(\varphi,\phi)$ of $Y$ 
the rational map $\psi_\phi$ is birational if and 
only if its Jacobian is nonzero.

As a consequence any element $\phi\in H_d$ corresponds
to at least one pair $(\varphi,\phi)$ in $Y$ (indeed 
according
to \cite{BassConnelWright} the inverse of a birational 
self map of~$\mathbb{P}^n_\mathbb{C}$ of degree $d$ 
has degree $\leq d^{n-1}$).

The description of $Y$ shows that it is closed 
in $W_{d^{n-1}}\times W_d$. The 
image $\mathrm{pr}_2(Y)$ of $Y$ by the second
projection $\mathrm{pr}_2$ is closed in $W_d$
since $W_{d^{n-1}}$ is a complete variety and 
$\mathrm{pr}_2$ a Zariski closed morphism. 
One can write $H_d$ as 
$\mathcal{U}\cap\mathrm{pr}_2(Y)$ where 
$\mathcal{U}\subseteq W_d$ is the open set 
of elements having a nonzero Jacobian. 
As a result $H_d$ is locally closed in $W_d$ 
and closed in $\mathcal{U}$.

\item[$\diamond$] Consider the $H_d$-rational 
map $\phi$ defined by 
\begin{align*}
& f\colon H_d\times\mathbb{P}^n_\mathbb{C}\dashrightarrow H_d\times\mathbb{P}^n_\mathbb{C}&& (\varphi,z)\dashrightarrow(\varphi,\varphi(z)).
\end{align*}
Set $J=\det\left(\left(\frac{\partial \varphi_i}{\partial x_j}\right)_{0\leq i,\,j\leq n}\right)$.
Let $\mathcal{V}\subset H_d\times\mathbb{P}^n_\mathbb{C}$
be the open set where $J$ is not zero.

\begin{claim}[\cite{BlancFurter}] 
The restriction $f_{\vert\mathcal{V}}$ of 
$f$ to $\mathcal{V}$ is an open immersion.
\end{claim}

Hence $\pi_d$ is a morphism and it follows from 
the construction of $H_d$ that 
the image of $\pi_d$ is 
$\mathrm{Bir}_{\leq d}(\mathbb{P}^n_\mathbb{C})$.

\item[$\diamond$] Let $\phi$ be an element in 
$\mathrm{Bir}(\mathbb{P}^n_\mathbb{C})_{\leq d}$. 
It corresponds to a birational self map~$\psi_\phi$
of $\mathbb{P}^n_\mathbb{C}$ given by 
\begin{align*}
\psi_\phi\colon (z_0:z_1:\ldots:z_n)\dashrightarrow(\phi_0(z_0,z_1,\ldots,z_n):\phi_1(z_0,z_1,\ldots,z_n):\ldots:\phi_n(z_0,z_1,\ldots,z_n))
\end{align*}
for some homogeneous polynomials of degree $k\leq d$
having no common divisor. Then 
\[
(\pi_d)^{-1}(\psi_\phi)=\big\{(\varphi_0:\varphi_1:\ldots:\varphi_n)\in W_d\,\vert\, \varphi_i\phi_j=\varphi_j\phi_i\quad\forall\, 1\leq i<j\leq n\big\}\subset H_d.
\]
This set is thus closed in $W_d$, and so in $H_d$.

\item[$\diamond$] If $\ell$ is a positive integer 
and $F$ a closed subset of $H_\ell$, then we 
denote by~$Y_F$ the subset of $Y\times\overline{F}$
(where $Y\subset W_{d^{n-1}}\times W_d$ is as 
above and $\overline{F}$ is the closure of $F$ in
$W_\ell$) given by 
\[
Y_F=\big\{((\zeta,\phi),\varphi)\,\vert\, \text{$\phi$ and $\varphi$ yield the same map $\mathbb{P}^n_\mathbb{C}\dashrightarrow\mathbb{P}^n_\mathbb{C}$}\big\}.
\]
In other words
\[
Y_F=\big\{((\zeta,\phi),\varphi)\,\vert\, \phi_i\varphi_j=\phi_j\varphi_i\quad \forall\, i,\,j\big\}.
\]
\end{itemize}
Hence $Y_F$ is closed in $Y\times\overline{F}$ and 
also in $W_{d^{n-1}}\times W_d\times W_\ell$. The 
subset $\mathrm{pr}_2(Y_F)$ of $W_d$ is closed 
in $W_d$, and so in $\mathrm{pr}_2(Y)$; as a
result $\mathrm{pr}_2(Y_F)\cap\mathcal{U}$ is 
closed in $\mathrm{pr}_2(Y)\cap\mathcal{U}$. 
We conclude using the fact that 
$\mathrm{pr}_2(Y_F)\cap\mathcal{U}=(\pi_d)^{-1}(\pi_\ell(F))$ 
and 
$\mathrm{pr}_2(Y)\cap\mathcal{U}=H_d$.
\end{proof}

\begin{lem}[\cite{BlancFurter}]\label{lem:BlancFurter}
Let $V$ be an irreducible algebraic variety, and let 
$\upsilon\colon V\to\mathrm{Bir}(\mathbb{P}^n_\mathbb{C})$ be a morphism.
There exists an open affine covering $(\mathcal{V}_i)_{i\in I}$ of $V$ such that
for each $i$ there exist an integer $d_i$ and a morphism 
$\upsilon_i\colon \mathcal{V}_i\to H_{d_i}$ such that $\upsilon_{\vert \mathcal{V}_i}=\pi_{d_i}\circ\upsilon_i$.
\end{lem}

\begin{proof}
Consider a morphism 
$\tau\colon V\to\mathrm{Bir}(\mathbb{P}^n_\mathbb{C})$
given by a $V$-birational map
\[
\phi\colon V\times\mathbb{P}^n_\mathbb{C}\dashrightarrow V\times\mathbb{P}^n_\mathbb{C}
\]
which restricts to an open immersion on an open 
set $\mathcal{U}$. Take a point $p_0$ in $V$. 
Let $\mathcal{V}_0\subset V$ be an open affine set containing
$p_0$. Take an element $w_0=(p_0,y)$ of 
$\mathcal{U}$. Let us fix homogeneous coordinates
$(z_0:z_1:\ldots:z_n)$ on $\mathbb{P}^n_\mathbb{C}$
such that 
\begin{itemize}
\item[$\diamond$] $y=(1:0:0:\ldots:0)$,

\item[$\diamond$] $\phi(w_0)$ does not belong to the 
plane $z_0=0$.
\end{itemize}

Let us denote by 
$\mathbb{A}^n_\mathbb{C}\subset\mathbb{P}^n_\mathbb{C}$
the affine set where $z_0=1$;
\begin{align*}
&x_1=\frac{z_1}{z_0}&&x_2=\frac{z_2}{z_0} &&\ldots && x_n=\frac{z_n}{z_0} 
\end{align*}
are natural affine coordinates of $\mathbb{A}^n_\mathbb{C}$.
The map $\phi$ restricts to a rational map of
$\mathcal{V}_0\times\mathbb{P}^n_\mathbb{C}$ defined at $w_0$. 
Its composition with the projection on the $i$-th 
coordinate is a rational function on 
$\mathcal{V}_0\times\mathbb{A}^n_\mathbb{C}$ defined at $w_0$.
Hence $\phi_{\vert \mathcal{V}_0\times\mathbb{A}^n_\mathbb{C}}$
can be written in a neighborhood of $w_0$ as 
\[
(v,x_1,x_2,\ldots,x_n)\mapsto\left(\frac{R_1}{Q_1},\frac{R_2}{Q_2},\ldots,\frac{R_n}{Q_n}\right)
\]
for some $R_i$, $Q_i$ in $\mathbb{C}[V][x_1,x_2,\ldots,x_n]$
such that none of the $Q_i$ vanish at $w_0$. 
As a result $\phi$ is given in a neighborhood of 
$w_0$ by 
\[
\big(v,(z_0:z_1:\ldots:z_n)\big)\mapsto(P_0:P_1:\ldots:P_n)
\]
where the $P_i\in\mathbb{C}[\mathcal{V}_0][z_0,z_1,\ldots,z_n]$ 
are homogeneous polynomials of the same degree $d_0$
such that not all vanish at $w_0$. Denote by 
$\mathcal{U}_0$ the set of points of 
$(V\times\mathbb{P}^n_\mathbb{C})\cap\mathcal{U}$
where at least one of the $P_i$ does not vanish; 
$\mathcal{U}_0$ is an open subset of 
$V\times\mathbb{P}^n_\mathbb{C}$. Its projection 
$\mathrm{pr}_1(\mathcal{U}_0)$ on $V$ is an 
open subset of $\mathcal{V}_0$ containing $p_0$. There 
thus exists an affine open subset
$\widetilde{A_0}\subseteq\mathrm{pr}_1(\mathcal{U}_0)$
containing $p_0$. The $n$-uple $(P_0,P_1,\ldots,P_n)$
yields to a morphism 
$\upsilon_0\colon\widetilde{A_0}\to H_d$.
By construction 
$\upsilon_{\vert \widetilde{A_0}}=\pi_d\circ\upsilon_0$.
If we repeat this process for any point of $V$ we
get an affine covering.
\end{proof}

Lemma \ref{lem:BlancFurter} implies the following one:

\begin{cor}[\cite{BlancFurter}]\label{cor:BlancFurter}
\begin{itemize}
\item[$\diamond$] A set $S\subseteq\mathrm{Bir}(\mathbb{P}^n_\mathbb{C})$ 
is closed if and only if $\pi_d^{-1}(S)$ is closed in $H_d$ for any 
$d\geq 1$. 

\item[$\diamond$] For any $d$,the set $\mathrm{Bir}_{\leq d}(\mathbb{P}^n_\mathbb{C})$ 
is closed in $\mathrm{Bir}(\mathbb{P}^n_\mathbb{C})$.

\item[$\diamond$] For any $d$, the map
$\pi_d\colon H_d\to\mathrm{Bir}_{\leq d}(\mathbb{P}^n_\mathbb{C})$ 
is surjective, continuous and closed. In particular it is a topological 
quotient map.
\end{itemize}
\end{cor}

\begin{proof}
  Let us prove the first assertion. Assume that $S$ 
is closed in $\mathrm{Bir}(\mathbb{P}^n_\mathbb{C})$.
Recall that a subset of $\mathrm{Bir}(\mathbb{P}^n_\mathbb{C})$ is closed in 
$\mathrm{Bir}(\mathbb{P}^n_\mathbb{C})$ if and only if its preimage by any 
morphism is closed. Since any 
$\pi_d\colon H_d\to\mathrm{Bir}(\mathbb{P}^n_\mathbb{C})$ is a morphism
$\pi_d^{-1}(S)$ is thus closed in $H_d$. 

Conversely suppose that $\pi_d^{-1}(S)$ is closed in $H_d$ for any $d$. Let 
$V$ be an irreducible algebraic variety, and let 
$\upsilon\colon V\to\mathrm{Bir}(\mathbb{P}^n_\mathbb{C})$ be a morphism. 
According to Lemma \ref{lem:BlancFurter} there exists an open affine 
covering $(\mathcal{V}_i)_{i\in I}$ of $V$ such that for any $i$ there exist
\begin{itemize}
\item[$\diamond$] an integer $d_i$,
\item[$\diamond$] a morphism $\upsilon_i\colon \mathcal{V}_i\to H_{d_i}$
\end{itemize}
with 
$\upsilon_{\vert \mathcal{V}_i}=\pi_{d_i}\circ\upsilon_i$. As $\pi_{d_i}^{-1}(S)$ is closed
and $\upsilon^{-1}(S)\cap \mathcal{V}_i=\upsilon_i^{-1}(\pi_{d_i}^{-1}(F))$ one gets that 
$\upsilon^{-1}(S)\cap~\mathcal{V}_i$ is closed in $\mathcal{V}_i$ for any $i$. As a result 
$\upsilon^{-1}(S)$ is closed.

\medskip

We will now prove the second assertion. According to 
the first assertion it suffices to prove that 
\[
\pi_\ell^{-1}\big(\mathrm{Bir}_{\leq d}(\mathbb{P}^n_\mathbb{C})\big)=\pi_\ell^{-1}(\pi_d(H_d))
\]
is closed in $H_\ell$ for any $\ell$. This follows from 
Lemma \ref{lem:propbf}.

\medskip

Finally let us prove the third assertion. The 
surjectivity follows from the construction of $H_d$ and $\pi_d$ 
(\emph{see} \cite{BlancFurter}). Since $\pi_d$ is a morphism, 
$\pi_d$ is continuous. Let $S\subseteq H_d$ be a closed subset. 
According to 
Lemma \ref{lem:propbf} the set 
$\pi_\ell^{-1}(\pi_d(S))$ is closed in $H_\ell$ for any~$\ell$. The 
first assertion allows to conclude.
\end{proof}

The first and third assertions of Corollary \ref{cor:BlancFurter} imply:

\begin{pro}[\cite{BlancFurter}]
The Zariski topology of $\mathrm{Bir}(\mathbb{P}^n_\mathbb{C})$ is the inductive
limit topology given by the Zariski topologies of 
$\mathrm{Bir}_{\leq d}(\mathbb{P}^n_\mathbb{C})$, $d\in\mathbb{N}$, 
which are the quotient topology of 
\[
\pi_d\colon H_d\to\mathrm{Bir}_{\leq d}(\mathbb{P}^n_\mathbb{C})
\]
where $H_d$ is endowed with its Zariski topology.
\end{pro}

\subsection{Algebraic subgroups}

An \textsl{algebraic subgroup}\index{defi}{algebraic subgroup of 
$\mathrm{Bir}(\mathbb{P}^n_\mathbb{C})$} of $\mathrm{Bir}(\mathbb{P}^n_\mathbb{C})$ is a subgroup 
$\mathrm{G}\subset\mathrm{Bir}(\mathbb{P}^n_\mathbb{C})$ 
which is the image of an algebraic group~$\mathrm{H}$ by a
homomorphism~$\upsilon$ such that
$\upsilon\colon\mathrm{H}\to\mathrm{Bir}(\mathbb{P}^n_\mathbb{C})$
is a morphism; by Lemma \ref{lem:BlancFurter}
 any algebraic group is contained in some 
$\mathrm{Bir}_{\leq d}(\mathbb{P}^n_\mathbb{C})$, {\it i.e.} 
any algebraic group has 
\textsl{bounded degree}\index{defi}{bounded degree (group of)}.
Corollary \ref{cor:agree} allows to give a more intrinsic
definition of algebraic groups which corresponds to taking 
closed subgroups of 
$\mathrm{Bir}(\mathbb{P}^n_\mathbb{C})$ of bounded 
degree. Lemma \ref{lem:agree} shows that these two definitions 
agree.

\begin{pro}[\cite{BlancFurter}]\label{pro:agree}
Let $\mathrm{G}$ be a subgroup of 
$\mathrm{Bir}(\mathbb{P}^n_\mathbb{C})$. 
Assume that 
\begin{itemize}
\item[$\diamond$] $\mathrm{G}$ is closed for the 
Zariski topology;

\item[$\diamond$] $\mathrm{G}$ is connected for the
Zariski topology;

\item[$\diamond$] $\mathrm{G}\subset\mathrm{Bir}_{\leq d}(\mathbb{P}^n_\mathbb{C})$ for some integer $d$.
\end{itemize}

If $d$ is choosen minimal, then the set 
$(\pi_d)^{-1}\big(\mathrm{G}\cap\mathrm{Bir}_d(\mathbb{P}^n_\mathbb{C})\big)$
is non empty. Let us denote by~$\mathrm{K}$ the closure of 
$(\pi_d)^{-1}\big(\mathrm{G}\cap\mathrm{Bir}_d(\mathbb{P}^n_\mathbb{C})\big)$
in $H_d$. Then 
\begin{itemize}
\item[$\diamond$] $\pi_d$ induces a homeomorphism 
$\mathrm{K}\to\mathrm{G}$;

\item[$\diamond$] if $V$ is an irreducible algebraic
variety, the morphisms
$V\to\mathrm{Bir}(\mathbb{P}^n_\mathbb{C})$
having image in $\mathrm{G}$ correspond, via
$\pi_d$, to the morphisms of algebraic 
varieties $V\to\mathrm{K}$;

\item[$\diamond$] the liftings to $\mathrm{K}$
of the maps 
\begin{align*}
& \mathrm{G}\times\mathrm{G}\to\mathrm{G},\, (\varphi,\psi)
&&
\mathrm{G}\to\mathrm{G},\,\varphi\mapsto\varphi^{-1}
\end{align*}
give rise to morphisms of algebraic varieties 
$\mathrm{K}\times\mathrm{K}\to\mathrm{K}$
and $\mathrm{K}\to\mathrm{K}$.
\end{itemize}
This gives $\mathrm{G}$ a unique structure of 
algebraic group.
\end{pro}

\begin{cor}[\cite{BlancFurter}]\label{cor:agree}
Let $\mathrm{G}$ be a subgroup of 
$\mathrm{Bir}(\mathbb{P}^n_\mathbb{C})$. Assume
that $\mathrm{G}$ is 
\begin{itemize}
\item[$\diamond$] closed for the Zariski
topology,

\item[$\diamond$] of bounded degree.
\end{itemize}

Then there exist an algebraic group $\mathrm{K}$
together with a morphism 
$\mathrm{K}\to\mathrm{Bir}(\mathbb{P}^n_\mathbb{C})$
inducing a homeomorphism 
$\pi\colon\mathrm{K}\to\mathrm{G}$ such that:
\begin{itemize}
\item[$\diamond$] $\pi$ is a 
group homomorphism 
\item[$\diamond$] and for any 
irreducible algebraic variety $V$ the 
morphisms $V\to\mathrm{Bir}(\mathbb{P}^n_\mathbb{C})$
having their image in $\mathrm{G}$ correspond,
via $\pi$, to the morphisms of algebraic 
varieties $V\to\mathrm{K}$.
\end{itemize}
\end{cor}

\begin{proof}
Let us first prove that $\mathrm{G}$ has 
a finite number of irreducible components.
The group~$\mathrm{G}$ is closed in 
$\mathrm{Bir}_{\leq d}(\mathbb{P}^n_\mathbb{C})$ 
hence its preimage
$(\pi_d)^{-1}(\mathrm{G})$ is also closed
in~$H_d$. It thus has a finite number of 
irreducible components $C_1$, $C_2$, $\ldots$, 
$C_r$. The sets $\pi_d(C_1)$, $\pi_d(C_2)$, 
$\ldots$, $\pi_d(C_r)$ are closed and 
irreducible and cover $\mathrm{G}$
(third assertion of Corollary \ref{cor:BlancFurter}).
If we keep the maximal ones with respect
to inclusion we get
the irreducible components of $\mathrm{G}$.

As for algebraic groups (\cite[\S 7.3]{Humphreys}) 
one can show that:
\begin{itemize}
\item[$\diamond$] exactly one irreducible 
component of $\mathrm{G}$ passes through 
$\mathrm{id}$;

\item[$\diamond$] this irreducible component
is a closed normal subgroup of finite
index in $\mathrm{G}$ whose cosets are the 
connected as well as irreducible components
of $\mathrm{G}$. 
\end{itemize}

This allows to reduced to the connected case; 
Proposition \ref{pro:agree} allows to conclude.
\end{proof}

\begin{lem}[\cite{BlancFurter}]\label{lem:agree}
Let $\mathrm{A}$ be an algebraic group and
$\rho\colon \mathrm{A}\to\mathrm{Bir}(\mathbb{P}^n_\mathbb{C})$
be a morphism that is also a group
homomorphism.

Then the image $\mathrm{G}$ of $\mathrm{A}$
is a closed subgroup of 
$\mathrm{Bir}(\mathbb{P}^n_\mathbb{C})$ 
of bounded degree. 

If $\pi\colon\mathrm{K}\to\mathrm{G}$ is 
the homeomorphism constructed in Corollary
\ref{cor:agree}, there exists a 
unique morphism of algebraic groups
$\widetilde{\rho}\colon\mathrm{A}\to\mathrm{K}$
such that $\rho=\pi\circ\widetilde{\rho}$.
\end{lem}

\begin{proof}
Lemma \ref{lem:BlancFurter} asserts that 
$\mathrm{G}=\rho(\mathrm{A})$ has bounded
degree. The closure~$\overline{\mathrm{G}}$ 
of $\mathrm{G}$ is a subgroup of 
$\mathrm{Bir}(\mathbb{P}^n_\mathbb{C})$; 
indeed inversion being a homeomorphism
$\overline{\mathrm{G}}^{-1}=\overline{\mathrm{G}^{-1}}=\overline{\mathrm{G}}$. Similarly translation by 
$g\in\mathrm{G}$ is a homeomorphism thus
$g\overline{\mathrm{K}}=\overline{g\mathrm{K}}=\overline{\mathrm{K}}$,
that is 
$\mathrm{G}\overline{\mathrm{G}}\subset\overline{\mathrm{G}}$. 
In turn, if $g\in\overline{\mathrm{G}}$, 
then $\mathrm{G}g\subset\overline{\mathrm{G}}$,
so $\overline{\mathrm{G}}g=\overline{\mathrm{G}g}\subset\overline{\mathrm{G}}$. 
As a result $\overline{\mathrm{G}}$ is a 
subgroup of 
$\mathrm{Bir}(\mathbb{P}^n_\mathbb{C})$.

According to Corollary \ref{cor:agree} there
exist a canonical homeomorphism 
$\mathrm{K}\to\overline{\mathrm{G}}$ where 
$\mathrm{K}$ is an algebraic group and a lift
$\widetilde{\rho}\colon\mathrm{A}\to\mathrm{H}$
of the morphism 
$\rho\colon\mathrm{A}\to\mathrm{Bir}(\mathbb{P}^n_\mathbb{C})$
whose image is contained in 
$\overline{\mathrm{G}}$. As $\rho$ is a 
group homomorphism $\widetilde{\rho}$ 
is a morphism of algebraic groups 
hence its image is closed, 
so $\mathrm{im}\,\rho=\mathrm{K}$. 
Therefore, $\overline{\mathrm{G}}=\mathrm{G}$.
\end{proof}

\begin{pro}[\cite{BlancFurter}]
Any algebraic subgroup of
$\mathrm{Bir}(\mathbb{P}^n_\mathbb{C})$
is affine.
\end{pro}

\begin{proof}[Sketch of the proof]
Let $\mathrm{G}$ be an algebraic subgroup
of $\mathrm{Bir}(\mathbb{P}^n_\mathbb{C})$.
One can show that $\mathrm{G}$ is linear, 
and this reduces to the connected case.
By the regularization theorem of 
Weil (\emph{see} \S \ref{sec:WeilKraft})
the group $\mathrm{G}$ acts by 
automorphisms on some (smooth) rational
variety $V$. Assume that
$\alpha_V\colon V\to A(V)$ is the 
Albanese morphism. According
to the Nishi-Matsumura
theorem the induced action of 
$\mathrm{G}$ on $A(V)$ factors through
a morphism $A(\mathrm{G})\to A(V)$ 
with finite kernel (\emph{see} for 
instance \cite{Brion}). But~$V$ is 
rational hence $A(V)$ is trivial and
so does $A(\mathrm{G})$. The 
structure theorem of Chevalley
asserts that $\mathrm{G}$ is affine
(\emph{see for instance} \cite{Rosenlicht}).
\end{proof}

Let us finish by some examples:
\begin{itemize}
\item[$\diamond$] The Cremona group in one variable 
$\mathrm{Bir}(\mathbb{P}^1_\mathbb{C})$ coincides with the group of 
linear projective transformations $\mathrm{PGL}(2,\mathbb{C})$; it
is an algebraic group of dimension $3$.

\item[$\diamond$] In dimension $2$ the Cremona group contains the 
two following algebraic subgroups:
\begin{itemize}
\item the group $\mathrm{PGL}(3,\mathbb{C})$ of automorphisms of 
$\mathbb{P}^2_\mathbb{C}$;
 
\item the group $\mathrm{PGL}(2,\mathbb{C})\times\mathrm{PGL}(2,\mathbb{C})$
obtained as follows: the surface 
$\mathbb{P}^1_\mathbb{C}\times\mathbb{P}^1_\mathbb{C}$ can be 
considered as a smooth quadric in $\mathbb{P}^3_\mathbb{C}$ whose
automorphism group contains 
$\mathrm{PGL}(2,\mathbb{C})\times\mathrm{PGL}(2,\mathbb{C})$;
by stereographic projection the quadric is birationally equivalent
to $\mathbb{P}^2_\mathbb{C}$. Hence $\mathrm{Bir}(\mathbb{P}^2_\mathbb{C})$
also contains a copy of 
$\mathrm{PGL}(2,\mathbb{C})\times\mathrm{PGL}(2,\mathbb{C})$.
\end{itemize}

\item[$\diamond$] More generally
$\mathrm{Aut}(\mathbb{P}^n_\mathbb{C})=\mathrm{PGL}(n+1,\mathbb{C})$ is an 
algebraic subgroup of $\mathrm{Bir}(\mathbb{P}^n_\mathbb{C})$ and 
\[
\underbrace{\mathrm{PGL}(2,\mathbb{C})\times\mathrm{PGL}(2,\mathbb{C})\times\ldots\times\mathrm{PGL}(2,\mathbb{C})}_\text{$n$ times}
\]
is an algebraic subgroup of 
\[
\mathrm{Aut}(\underbrace{\mathbb{P}^1_\mathbb{C}\times\mathbb{P}^1_\mathbb{C}\times\ldots\mathbb{P}^1_\mathbb{C}}_\text{$n$ times})\subset\mathrm{Bir}(\mathbb{P}^n_\mathbb{C}).
\]

\item[$\diamond$]  If $\mathrm{G}$ is a semi-simple algebraic 
group, $\mathrm{H}$ is a parabolic subgroup of $G$ and $V=\faktor{\mathrm{G}}{\mathrm{H}}$, 
then the homogeneous variety $V$ of dimension $n$ is rational; 
$\pi\circ\mathrm{G}\circ\pi^{-1}$ determines an algebraic subgroup 
of~$\mathrm{Bir}(\mathbb{P}^n_\mathbb{C})$ for any birational map 
$\pi\colon V\dashrightarrow\mathbb{P}^n_\mathbb{C}$.
\end{itemize}

\subsection{Euclidean topology}

We can put the Euclidean topology on a 
complex algebraic group; this gives
any algebraic group the structure of a
topological group. Recall that the 
Euclidean topology is finer than the
Zariski one. 

Let $n\geq 1$ be an integer. The group
$\mathrm{Bir}(\mathbb{P}^1_\mathbb{C})=\mathrm{Aut}(\mathbb{P}^2_\mathbb{C})=\mathrm{PGL}(2,\mathbb{C})$
is obviously a topological group.
Assume now that $n\geq 2$; we will 
\begin{itemize}
\item[$\diamond$] first define the 
Euclidean topology on 
$\mathrm{Bir}_{\leq d}(\mathbb{P}^n_\mathbb{C})$
and show that the natural inclusion
$\mathrm{Bir}_{\leq d}(\mathbb{P}^n_\mathbb{C})\hookrightarrow\mathrm{Bir}_{\leq d+1}(\mathbb{P}^n_\mathbb{C}) $ 
is a closed embedding;

\item[$\diamond$] second define the 
Euclidean topology on 
$\mathrm{Bir}(\mathbb{P}^n_\mathbb{C})$
as the inductive limit topology 
induced by those of 
$\mathrm{Bir}_{\leq d}(\mathbb{P}^n_\mathbb{C})$, 
that is a subset 
$F\subset\mathrm{Bir}(\mathbb{P}^n_\mathbb{C})$
is closed if and only if 
$F\cap\mathrm{Bir}_{\leq d}(\mathbb{P}^n_\mathbb{C})$
is closed in 
$\mathrm{Bir}_{\leq d}(\mathbb{P}^n_\mathbb{C})$
for each $d$. Finally we will prove that
$\mathrm{Bir}(\mathbb{P}^n_\mathbb{C})$ 
endowed with the Euclidean topology 
is a topological group;

\item[$\diamond$] third give some 
remarks and properties.
\end{itemize}

\subsection{The Euclidean topology on $\mathrm{Bir}_{\leq d}(\mathbb{P}^n_\mathbb{C})$}

Let us recall that $W_d$ is a projective space 
and $H_d$ is locally closed in $W_d$
for the Zariski topology (Lemma
\ref{lem:propbf}). Let us put the 
Euclidean topology on $W_d$: the distance
between $(p_0:p_1:\ldots:p_n)$ and 
$(q_0:q_1:\ldots:q_n)$ is 
(\emph{see} \cite{Weyl})
\[
\frac{\displaystyle\sum_{i<j}\vert p_iq_j-p_jq_i\vert^2}{\left(\displaystyle\sum_{i}\vert p_i\vert^2\right)\left(\displaystyle\sum_{i}\vert q_i\vert^2\right)}
\]
We then put the induced topology on 
$H_d$. The behavior of the Zariski
topology on $\mathrm{Bir}(\mathbb{P}^n_\mathbb{C})$
leads to:

\begin{defi}
The \textsl{Euclidean topology}
\index{defi}{Euclidean topology on 
$\mathrm{Bir}_{\leq d}(\mathbb{P}^n_\mathbb{C})$} on 
$\mathrm{Bir}_{\leq d}(\mathbb{P}^n_\mathbb{C})$
is the quotient topology induced by the 
surjective map 
$\pi_d\colon H_d\to \mathrm{Bir}_{\leq d}(\mathbb{P}^n_\mathbb{C})$
where we put the Euclidean topology on 
$H_d$. 
\end{defi}

Recall that if $f\colon X\to Y$ is a quotient
map between topological spaces, $A$ is 
a subspace of~$X$, $A$ is open and 
$A=f^{-1}(f(A))$, then the induced map
$A\to f(A)$ is a quotient map 
(\cite[Chapter I, \S 3.6]{Bourbaki}).
Set 
\[
H_{d,d}=(\pi_d)^{-1}(\mathrm{Bir}_d(\mathbb{P}^n_\mathbb{C})).
\]
As
$(\pi_d)^{-1}(\mathrm{Bir}_{\leq d-1}(\mathbb{P}^n_\mathbb{C}))$
is closed in $H_d$, $H_{d,d}$ is 
open in $H_d$ for the 
Zariski topology and 
hence also for the Euclidean 
topology; $\pi_d$ restricts to 
a homeomorphism 
$H_{d,d}\to\mathrm{Bir}_d(\mathbb{P}^n_\mathbb{C})$
for any $d\geq 1$.

\begin{lem}[\cite{BlancFurter}]\label{lem:sequential}
Let $d\geq 1$ be an integer. The spaces
$W_d$ and $H_d$ are locally compact metric
spaces endowed with the Euclidean topology.

In particular the sets $W_d$, $H_d$ and 
$\mathrm{Bir}_{\leq d}(\mathbb{P}^n_\mathbb{C})$
are sequential spaces: a subset $F$ is closed
if the limit of every convergent sequence
with values in $F$ belongs to $F$. 
\end{lem}

\begin{proof}
The construction of the topology implies
that $W_d$ and $H_d$ are metric spaces.
As~$W_d$ is compact and $H_d$ is locally
closed in $W_d$ (Lemma \ref{lem:propbf})
the set~$H_d$ is locally compact. But 
metric spaces are sequential spaces and
quotients of sequential spaces 
are sequential (\cite{Franklin}). 
\end{proof}

We now would like to prove that the topological map 
$\pi_d\colon H_d\to\mathrm{Bir}_{\leq d}(\mathbb{P}^n_\mathbb{C})$
is proper and the topological space
$\mathrm{Bir}_{\leq d}(\mathbb{P}^n_\mathbb{C})$ is 
locally compact. 
Recall that a map $f\colon X\to Y$ between two 
topological spaces is 
\textsl{proper}\index{defi}{proper (map)}
if it is continuous and universally 
closed: for each topological space $Z$ the
map 
$f\times\mathrm{id}_Z\colon X\times Z\to Y\times Z$ 
is closed (\cite{Bourbaki}). A topological space
is \textsl{locally compact}\index{defi}{locally compact (topological space)} if it is Hausdorff 
and if each of its points has a 
compact neighborhood. If $f\colon X\to Y$
is a quotient map between topological spaces
such that $X$ is locally compact, then $f$
is proper if and only if it is closed and the
preimages of points are compact. This implies
furthermore that $Y$ is locally compact. According 
to Lemma \ref{lem:propbf} for any $\phi$ in 
$\mathrm{Bir}_{\leq d}(\mathbb{P}^n_\mathbb{C})$
the set $(\pi_d)^{-1}(\phi)$ is closed in the 
compact space $W_d$, so $(\pi_d)^{-1}(Y)$ is
compact. The topological space~$H_d$ being 
locally compact (Lemma \ref{lem:sequential}), to 
prove that $\pi_d$ is proper it suffices to 
prove that $\pi_d$ is closed.

\begin{claim}
The map 
$\pi_d\colon H_d\to\mathrm{Bir}_{\leq d}(\mathbb{P}^n_\mathbb{C})$
is proper.
\end{claim}

\begin{proof}
Let $F\subset H_d$ be a closed subset. To prove 
that $\pi_d(F)$ is closed 
in~$\mathrm{Bir}_{\leq d}(\mathbb{P}^n_\mathbb{C})$
amounts to prove that the saturated set 
$\widehat{F}=(\pi_d)^{-1}(\pi_d(F))$ is closed
in~$H_d$. Consider a sequence 
$(\varphi_i)_{i\in\mathbb{N}}$ of elements 
in $\widehat{F}$ which converges to 
$\varphi\in H_d$. Let us show that $\varphi$
belongs to $\widehat{F}$. Since $\pi_d$ is 
by construction continuous, the sequence 
$\big(\pi_d(\varphi_i)\big)_{i\in\mathbb{N}}$
converges to $\pi_d(\varphi)$ in 
$\mathrm{Bir}_{\leq d}(\mathbb{P}^n_\mathbb{C})$.
Taking a subsequence of 
$\big(\pi_d(\varphi_i)\big)_{i\in\mathbb{N}}$
if needed, we may suppose that the degree of 
all $\pi_d(\varphi_i)$ is constant equal to 
some $m\leq d$. 
\begin{itemize}
\item[$\diamond$] Assume $m=d$, then 
$(\pi_d)^{-1}(\pi_d(\varphi_i))=\{\varphi_i\}$
for each $i$. As a result each~$\varphi_i$ belongs
to $F$, so $\varphi$ belongs to 
$F\subset\widehat{F}$ as
wanted.

\item[$\diamond$] Suppose $m<d$. Set $k=d-m\geq 1$.
For any $i$ there exists a non-zero homogeneous 
polynomial $a_i\in\mathbb{C}[z_0,z_1,\ldots,z_n]$
of degree $k$ such that 
\[
\varphi_i=\big(a_if_{i,0}:a_if_{i,1}:\ldots:a_if_{i,n}\big)
\]
and $(f_{i,0}:f_{i,1}:\ldots:f_{i,n})\in W_m$ 
corresponds to a birational map of degree $m<d$.
Each $a_i$ is defined up to a constant and 
$\mathbb{P}(\mathbb{C}[z_0,z_1,\ldots,z_n])$ is 
compact, so, taking a subsequence if needed, we can 
suppose that $(a_i)_{i\in\mathbb{N}}$ converges
to a non-zero homogeneous polynomial 
$a\in\mathbb{C}[z_0,z_1,\ldots,z_n]$ of degree $k$.

Taking a subsequence if needed we can assume that 
$\{(f_{i,0}:f_{i,1}:\ldots:f_{i,n})\}_{i\in\mathbb{N}}$
converges to an element $(f_0:f_1:\ldots:f_n)$
of the projective space~$W_m$. Since 
$(\varphi_i)_{i\in\mathbb{N}}$ converges to 
$\varphi$ we get that 
$\varphi=(af_0:af_1:\ldots:af_n)$ in $H_d$.

As $\varphi_i$ belongs to 
$\widehat{F}=(\pi_d)^{-1}(\pi_d(F))$ for any $i$ 
there exists $\varphi'_i$ in $F$ such that 
$\pi_d(\varphi'_i)=\pi_d(\varphi_i)$. 
Consequently 
\[
\varphi'_i=\big(b_if_{i,0}:b_if_{i,1}:\ldots:b_if_{i,n}\big)
\]
for some non-zero homogeneous polynomial 
$b_i\in\mathbb{C}[z_0,z_1,\ldots,z_n]$ of 
degree~$k$. As before we can assume that $(b_i)_{i\in\mathbb{N}}$
converges to a non-zero homogeneous polynomial
$b\in\mathbb{C}[z_0,z_1,\ldots,z_n]$ of degree $k$. 
The sequence $(\varphi'_i)_{i\in\mathbb{N}}$ converges
to $(bf_0:bf_1:\ldots:bf_n)$ and $F$ is closed, thus 
$(bf_0:bf_1:\ldots:bf_n)$ belongs to $F$. This implies
that $\varphi=(af_0:af_1:\ldots:af_n)$ belongs to 
$\widehat{F}$.
\end{itemize}
\end{proof}

We can thus state:

\begin{lem}[\cite{BlancFurter}]\label{lem:procom}
Let $d\geq 1$ be an integer. Then
\begin{itemize}
\item[$\diamond$] the topological map
$\pi_d\colon H_d\to \mathrm{Bir}_{\leq d}(\mathbb{P}^n_\mathbb{C})$
is proper $($and closed$)$;

\item[$\diamond$] the topological space
$\mathrm{Bir}_{\leq d}(\mathbb{P}^n_\mathbb{C})$ 
is locally compact $($and Hausdorff$)$.
\end{itemize}
\end{lem}

\begin{lem}[\cite{BlancFurter}]\label{lem:top}
Let $d\geq 0$ be an integer.
The natural injection 
\[
\iota_d\colon\mathrm{Bir}_{\leq d}(\mathbb{P}^n_\mathbb{C})\to\mathrm{Bir}_{\leq d+1}(\mathbb{P}^n_\mathbb{C})
\]
is a closed embedding, that is a 
homeomorphism onto its image 
which is closed in 
$\mathrm{Bir}_{\leq d+1}(\mathbb{P}^n_\mathbb{C})$.
\end{lem}

\begin{proof}
Consider the map 
\begin{align*}
& \widehat{\iota_d}\colon H_d\to H_{d+1},&&  (f_0:f_1:\ldots:f_n)\mapsto (z_0f_0:z_0f_1:\ldots:z_0f_n).
\end{align*}
It is a morphism of algebraic varieties that 
is a closed immersion. As a result it is 
continuous and closed with respect to the 
Euclidean topology.
The diagram
\[
  \xymatrix{
 H_d \ar[r]^{\widehat{\iota_d}} \ar[d]_{\pi_d} & H_{d+1}\ar[d]^{\pi_{d+1}} \\
   \mathrm{Bir}_{\leq d}(\mathbb{P}^n_\mathbb{C})  \ar[r]_{\iota_d} & \mathrm{Bir}_{\leq d+1}(\mathbb{P}^n_\mathbb{C})
  }
\]
commutes. 

The continuity of $\widehat{\iota_d}$ implies
the continuity of $\iota_d$: let 
$\mathcal{U}$ be an open subset of 
$\mathrm{Bir}_{\leq d+1}(\mathbb{P}^n_\mathbb{C})$;
the equality 
$(\pi_d)^{-1}((\iota_d)^{-1}(\mathcal{U}))=(\pi_{d+1}\circ\widehat{\iota_d})^{-1}(\mathcal{U})$
shows that $(\pi_d)^{-1}((\iota_d)^{-1}(\mathcal{U}))$
is open in $H_d$, that is 
$(\iota_d)^{-1}(\mathcal{U})$ is open in 
$\mathrm{Bir}_{\leq d}(\mathbb{P}^n_\mathbb{C})$.
\end{proof}

\subsection{The Euclidean topology on $\mathrm{Bir}(\mathbb{P}^n_\mathbb{C})$}

Thanks to Lemma \ref{lem:top} one can put 
on $\mathrm{Bir}(\mathbb{P}^n_\mathbb{C})$
the inductive limit topology given by the 
$\mathrm{Bir}_{\leq d}(\mathbb{P}^n_\mathbb{C})$:
a subset of 
$\mathrm{Bir}(\mathbb{P}^n_\mathbb{C})$ is 
closed (resp. open) if and only if its 
intersection with any 
$\mathrm{Bir}_{\leq d}(\mathbb{P}^n_\mathbb{C})$
is closed (resp. open). In particular the 
injections 
$\mathrm{Bir}_{\leq d}(\mathbb{P}^n_\mathbb{C})\hookrightarrow\mathrm{Bir}(\mathbb{P}^n_\mathbb{C})$
are closed embeddings. This topology is 
called the \textsl{Euclidean topology}\index{defi}{Euclidean topology on $\mathrm{Bir}(\mathbb{P}^n_\mathbb{C})$}
of $\mathrm{Bir}(\mathbb{P}^n_\mathbb{C})$.
Let us now prove that 
$\mathrm{Bir}(\mathbb{P}^n_\mathbb{C})$ 
is a topological group endowed with the 
Euclidean topology.

\begin{lem}[\cite{BlancFurter}]\label{lem:bfcont}
Let $d\geq 1$ be an integer.
The map
\begin{align*}
&\mathcal{I}_d\colon\mathrm{Bir}_{\leq d}(\mathbb{P}^n_\mathbb{C})\to\mathrm{Bir}_{\leq d^{n-1}}(\mathbb{P}^n_\mathbb{C}),
&& \phi\mapsto\phi^{-1}
\end{align*}
is continuous.
\end{lem}

\begin{proof}
As in Lemma \ref{lem:propbf} we consider the 
set $Y\subset W_{d^{n-1}}\times W_d$ 
defined by
\[
Y=\big\{(\varphi,\phi)\in W_{d^{n-1}}\times W_d\,\vert\, \varphi\circ \phi=P\,\mathrm{id}\text{ for some $P\in\mathbb{C}[z_0,z_1,\ldots,z_n]_d$}\big\}.
\]
Let $\mathcal{U}\subset W_d$ (resp. 
$\mathcal{U}'\subset W_{d^{n-1}}$) be the 
set of elements having a nonzero Jacobian.
The set $Y$ is closed in $W_{d^{n-1}}\times W_d$
(see the proof of Lemma \ref{lem:propbf}) 
and the set $\mathcal{U}$ is open in $W_d$.
As a consequence 
\[
L=Y\cap(W_{d^{n-1}}\times\mathcal{U})=Y\cap(\mathcal{U}'\times\mathcal{U})
\]
is locally closed in the algebraic 
variety $W_{d^{n-1}}\times W_d$.

\smallskip

The projection on the first factor is 
a morphism $\eta_1\colon L\to H_{d^{n-1}}$
which is not surjective in general.
The projection on the second factor
induces a surjective morphism 
$\eta_2\colon L\to H_d$. By construction
the diagram 
\[
  \xymatrix{
 H_d  \ar[d]_{\pi_d} & L\subset W_{d^{n-1}}\times W_d\ar[l]_{\eta_2}\ar[r]^{\eta_1} & H_{d^{n-1}}\ar[d]^{\pi_{d-1}} \\
   \mathrm{Bir}_{\leq d}(\mathbb{P}^n_\mathbb{C})  \ar[rr]_{\mathcal{I}_d} & &\mathrm{Bir}_{\leq d^{n-1}}(\mathbb{P}^n_\mathbb{C})
  }
\]
commutes. 

\smallskip

Let us prove that $\eta_2$ is a closed map for
the Euclidean topology. The set $W_{d^{n-1}}$ 
is compact, so the second projection 
$W_{d^{n-1}}\times W_d\to W_d$ is a closed 
map. Its restriction $\eta'_2\colon Y\to W_d$
to the closed subset $Y$ of 
$W_{d^{n-1}}\times W_d$ is a closed map.
Since $L=(\eta'_2)^{-1}(H_d)$, we get that $\eta_2$
is a closed map\footnote{Let us recall 
that if $\varphi\colon A\to B$ is a continuous
closed map between topological spaces and 
$C$ is any subset of $B$, then $\varphi$ 
induces a continuous closed map 
$\varphi^{-1}(C)\to C$.}.

\smallskip

As the diagram is commutative for any 
$F\subset\mathrm{Bir}_{\leq d^{n-1}}(\mathbb{P}^n_\mathbb{C})$
we have 
\[
\eta_2\big((\pi_{d^{n-1}}\circ\eta_1)^{-1}(F)\big)=(\mathcal{I}_d\circ\pi_d)^{-1}(F);
\]
furthermore this set corresponds to elements
$(\phi_0:\phi_1:\ldots:\phi_n)\in W_d$
such that the rational map~$\psi_\phi$ 
is the inverse of an element of $F$. 
Assume that $F$ is closed 
in~$\mathrm{Bir}_{\leq d^{n-1}}(\mathbb{P}^n_\mathbb{C})$.
The maps $\eta_1$ and $\pi_{d^{n-1}}$ 
are continuous for the Euclidean topology
hence $(\pi_{d^{n-1}}\circ\eta_1)^{-1}(F)$ 
is closed in~$L$. Lemma \ref{lem:procom}
asserts that
\[
\pi_d^{-1}(\mathcal{I}_d^{-1}(F))=\eta_2\big((\pi_{d^{n-1}}\circ\eta_1)^{-1}(F)\big)
\]
is closed in $H_d$ and $\mathcal{I}_d^{-1}(F)$
is closed in 
$\mathrm{Bir}_{\leq d}(\mathbb{P}^n_\mathbb{C})$.
\end{proof}

Let us introduce the map $\mathcal{I}$ defined 
by 
\begin{align*}
&\mathcal{I}\colon\mathrm{Bir}(\mathbb{P}^n_\mathbb{C})\to \mathrm{Bir}(\mathbb{P}^n_\mathbb{C}),&&\phi\mapsto\phi^{-1}.
\end{align*} 
The degree of the inverse of a birational 
self map of $\mathbb{P}^n_\mathbb{C}$ of 
degree $d$ has degree at most $d^{n-1}$. 
Consequently for any $d\geq 1$ the map $\mathcal{I}$ 
restricts to an injective map
\[
\mathcal{I}_d\colon\mathrm{Bir}_{\leq d}(\mathbb{P}^n_\mathbb{C})\to \mathrm{Bir}_{\leq d^{n-1}}(\mathbb{P}^n_\mathbb{C}).
\]
According to Lemma \ref{lem:bfcont} the map
$\mathcal{I}_d$ is continuous. The 
definition of the topology of 
$\mathrm{Bir}(\mathbb{P}^n_\mathbb{C})$ 
implies that $\mathcal{I}$ is continuous. 
Since $\mathcal{I}=\mathcal{I}^{-1}$ 
one has:

\begin{cor}[\cite{BlancFurter}]\label{cor:bfcor1}
The map 
\begin{align*}
&\mathcal{I}\colon\mathrm{Bir}(\mathbb{P}^n_\mathbb{C})\to \mathrm{Bir}(\mathbb{P}^n_\mathbb{C}),&&\phi\mapsto\phi^{-1}
\end{align*}
is a homeomorphism.
\end{cor}

Let us now look at the composition
of two birational maps.

\begin{lem}[\cite{BlancFurter}]\label{lem:bfcont2}
For any $d$, $k\geq 1$ the map
\begin{align*}
&\chi_{d,k}\colon\mathrm{Bir}_{\leq d}(\mathbb{P}^n_\mathbb{C})\times \mathrm{Bir}_{\leq k}(\mathbb{P}^n_\mathbb{C})\to \mathrm{Bir}_{\leq dk}(\mathbb{P}^n_\mathbb{C}), && (\phi,\psi)\mapsto\phi\circ\psi
\end{align*}
is continuous.
\end{lem}

\begin{proof}
Let us consider the map 
$\widehat{\chi_{d,k}}\colon H_d\times H_k\to H_{dk}$
given by 
\begin{small}
\[
\big((\phi_0:\phi_1:\ldots:\phi_n),\,(\psi_0:\psi_1:\ldots:\psi_n)\big)\mapsto
\big(\phi_n(\psi_0,\psi_1,\ldots,\psi_n)):\ldots:\phi_n(\psi_0,\psi_1,\ldots,\psi_n)).
\]
\end{small}
The diagram
\[  \xymatrix{
 H_d\times H_k  \ar[d]_{\pi_d\times\pi_k}\ar[r]^{\widehat{\chi_{d,k}}} & H_{dk}\ar[d]^{\pi_{dk}} \\
   \mathrm{Bir}_{\leq d}(\mathbb{P}^n_\mathbb{C})\times\mathrm{Bir}_{\leq k}(\mathbb{P}^n_\mathbb{C})  \ar[r]_{\mathcal{I}_d} & \mathrm{Bir}_{\leq dk}(\mathbb{P}^n_\mathbb{C})
  }
\]
commutes.

The map $\widehat{\chi_{d,k}}$ is a morphism
of algebraic varieties, so is continuous for
the Euclidean topo\-logy. Therefore, if $F$ is 
a closed subset of 
$\mathrm{Bir}_{\leq dk}(\mathbb{P}^n_\mathbb{C})$,
then $(\pi_{dk}\circ\widehat{\chi_{d,k}})^{-1}(F)$
is closed in $H_d\times H_k$. But the diagram
is commutative, so 
\[
(\pi_d\circ\widehat{\chi_{d,k}})(F)=(\pi_d\times\pi_k)^{-1}\big((\chi_{d,k})^{-1}(F)\big).
\]
The product of two proper maps is proper
(\cite[Chapter 1,\S 10.1]{Bourbaki}); as a 
consequence $\pi_d\times\pi_k$ is proper
and hence closed. This implies that 
$\pi_d\times\pi_k$ is a quotient map. 
Hence $(\chi_{d,k})^{-1}(F)$ is closed and 
$\chi_{d,k}$ is continuous.
\end{proof}

According to Lemma \ref{lem:bfcont2} the map
\[
\chi_{d,k}\colon\mathrm{Bir}_{\leq d}(\mathbb{P}^n_\mathbb{C})\times
\mathrm{Bir}_{\leq k}(\mathbb{P}^n_\mathbb{C})\to\mathrm{Bir}_{\leq dk}(\mathbb{P}^n_\mathbb{C}) 
\] 
is continuous for each $d$, $k\geq 1$. As a consequence
by definition of the topology of~$\mathrm{Bir}(\mathbb{P}^n_\mathbb{C})$ 
we get:

\begin{cor}[\cite{BlancFurter}]\label{cor:bfcor2}
The map 
\begin{align*}
& \mathrm{Bir}(\mathbb{P}^n_\mathbb{C})\times \mathrm{Bir}(\mathbb{P}^n_\mathbb{C})\to \mathrm{Bir}(\mathbb{P}^n_\mathbb{C}),&&(\phi,\psi)\mapsto\phi\circ\psi 
\end{align*}
is continuous.
\end{cor}

Corollaries \ref{cor:bfcor1} and \ref{cor:bfcor2}
complete the proof of:

\begin{thm}[\cite{BlancFurter}]
The $n$-dimensional Cremona group 
endowed with the Euclidean topology is a topological group.
\end{thm}

Let us give a statement about the restriction of the 
topology on algebraic subgroups:

\begin{pro}[\cite{BlancFurter}]
Let $\mathrm{G}$ be a 
Zariski closed subgroup of 
$\mathrm{Bir}(\mathbb{P}^n_\mathbb{C})$ of 
bounded degree, let $\mathrm{K}$ be 
its associated algebraic group 
$($Corollary \ref{cor:agree}$)$. We put
on~$\mathrm{G}$ the restriction of the 
Euclidean topology of 
$\mathrm{Bir}(\mathbb{P}^n_\mathbb{C})$, 
we get the Euclidean topology on $\mathrm{K}$ 
via the bijection 
$\pi\colon\mathrm{K}\to\mathrm{G}$ which
becomes a homeomorphism.
\end{pro}

\subsection{Properties of the Euclidean topology of $\mathrm{Bir}(\mathbb{P}^n_\mathbb{C})$}

\begin{lem}
Any convergent sequence of 
$\mathrm{Bir}(\mathbb{P}^n_\mathbb{C})$ has 
bounded degree.
\end{lem}

\begin{proof}
If the sequence $(\varphi_i)_{i\in\mathbb{N}}$ of 
$\mathrm{Bir}(\mathbb{P}^n_\mathbb{C})$ converges
to $\varphi$, then 
$\big\{\varphi_i\,\vert\,i\in\mathbb{N}\big\}\cup\{\varphi\}$
is compact, so contained in 
$\mathrm{Bir}_{\leq d}(\mathbb{P}^n_\mathbb{C})$ for 
some $d$.
\end{proof}

\begin{lem}
The topological group $\mathrm{Bir}(\mathbb{P}^n_\mathbb{C})$ 
is Hausdorff.
\end{lem}

\begin{proof}
According to \cite[III, \S 2.5, Prop. 13]{Bourbaki} 
a topological group is Hausdorff if and only if the 
trivial one-element subgroup is closed. Any point of
$\mathrm{Bir}(\mathbb{P}^n_\mathbb{C})$ is closed 
in some $\mathrm{Bir}_{\leq d}(\mathbb{P}^n_\mathbb{C})$
(Lemma \ref{lem:propbf}), 
so is closed in $\mathrm{Bir}(\mathbb{P}^n_\mathbb{C})$.
As a result $\mathrm{Bir}(\mathbb{P}^n_\mathbb{C})$
is Hausdorff.
\end{proof}

\begin{lem}\label{lem:compboun}
Any compact subset of 
$\mathrm{Bir}(\mathbb{P}^n_\mathbb{C})$ is contained
in $\mathrm{Bir}_{\leq d}(\mathbb{P}^n_\mathbb{C})$
for some $d$.
\end{lem}

\begin{proof}
Assume by contradiction that 
$\mathrm{Bir}(\mathbb{P}^n_\mathbb{C})$ contains
a compact subset $K$ such that 
$(\varphi_i)_{i\in\mathbb{N}}$ is a sequence 
of elements of $K$ with 
$\deg\varphi_{i+1}>\deg\varphi_i$ for each~$i$.
Let us consider 
$K'=\big\{\varphi_i\,\vert\,i\in\mathbb{N}\big\}$. 
On the one hand it is a closed subset of the 
compact set $K$; hence it is compact. On the 
other hand the intersection of any subset of 
$K'$ with 
$\mathrm{Bir}_{\leq d}(\mathbb{P}^n_\mathbb{C})$
is closed, so $\mathrm{K}'$ is an infinite set
endowed with the discrete topology; in particular
it cannot be compact: contradiction.
\end{proof}

\begin{lem}
For $n\geq 2$ the topological
space $\mathrm{Bir}(\mathbb{P}^n_\mathbb{C})$ 
is not locally compact.
\end{lem}

\begin{proof}
Let $\mathcal{U}\subset\mathrm{Bir}(\mathbb{P}^n_\mathbb{C})$
be an open neighborhood of the identity. Since any 
compact subset of $\mathrm{Bir}(\mathbb{P}^n_\mathbb{C})$
is contained in $\mathrm{Bir}_{\leq d}(\mathbb{P}^n_\mathbb{C})$
(Lemma \ref{lem:compboun}) for some $d$ to prove that 
$\mathcal{U}$ is not contained in any compact subset of 
$\mathrm{Bir}(\mathbb{P}^n_\mathbb{C})$ it suffices to 
show that $\mathcal{U}$ contains elements of arbitrarily 
large degree. For any integers $m$, $k\geq 1$ let us 
consider the birational map given in the affine chart
$z_0=1$ by 
\[
f_{m,k}\colon(z_1,z_2,\ldots,z_n)\dashrightarrow\left(z_1+\frac{1}{k}z_2^m,z_2,\ldots,z_n\right).
\]
Fixing $m$ we note that the sequence 
$(f_{m,k})_{k\geq 1}$ converges to the identity; 
in parti\-cular $f_{m,k}$ belongs to $\mathcal{U}$
when $k$ is large enough.
\end{proof}

\begin{lem}
For $n\geq 2$ the topological
space $\mathrm{Bir}(\mathbb{P}^n_\mathbb{C})$ 
is not metrisable.
\end{lem}

\begin{proof}
Consider the inclusion
\begin{eqnarray*}
\mathbb{C}[z_2]&\hookrightarrow&\mathrm{Aut}(\mathbb{C}^n)\subset\mathrm{Bir}(\mathbb{P}^n_\mathbb{C})\\
P &\hookrightarrow& \big((z_1,z_2,\ldots,z_n)\dashrightarrow(z_1+P(z_2),z_2,z_3,\ldots,z_n)\big)
\end{eqnarray*}
Observe that $\mathbb{C}[z_2]$ is closed in 
$\mathrm{Bir}(\mathbb{P}^n_\mathbb{C})$ and that 
for any $d$ the induced topology 
on~$\mathbb{C}[z_2]_{\leq d}$ is the topology as
a vector space (or as an algebraic group). 
The induced topology on $\mathbb{C}[z_2]$ is
thus the inductive limit topology given by 
\[
\mathbb{C}[z_2]_{\leq 1}\subset\mathbb{C}[z_2]_{\leq 2}\subset\ldots
\]
For any sequence $\ell=(\ell_n)_{n\in\mathbb{N}}$ of 
positive integers the set 
\[
\mathcal{U}_\ell=\left\{\displaystyle\sum_{i=0}^d a_iX_i\,\vert\,\vert a_i\vert<\frac{1}{\ell_i}\right\}
\]
is open in $\mathbb{C}[z_2]$. This implies that 
$\mathbb{C}[z_2]$ is not countable and hence not
metrisable. The same holds for 
$\mathrm{Bir}(\mathbb{P}^n_\mathbb{C})$.
\end{proof}

\begin{lem}
The topological group 
$\mathrm{Bir}(\mathbb{P}^n_\mathbb{C})$ is 
compactly generated if and only if $n\leq 2$.
\end{lem}

\begin{proof}
The group 
$\mathrm{Bir}(\mathbb{P}^1_\mathbb{C})=\mathrm{PGL}(2,\mathbb{C})$
is a linear algebraic group; consequently it is 
compactly generated.

By the classical Noether and Castelnuovo Theorem
the group $\mathrm{Bir}(\mathbb{P}^2_\mathbb{C})$
is gene\-rated by 
$\mathrm{Aut}(\mathbb{P}^2_\mathbb{C})=\mathrm{PGL}(3,\mathbb{C})$
and the standard involution $\sigma_2$. The linear algebraic 
group 
$\mathrm{Aut}(\mathbb{P}^2_\mathbb{C})=\mathrm{PGL}(3,\mathbb{C})$
being compactly generated, 
$\mathrm{Bir}(\mathbb{P}^2_\mathbb{C})$
is compactly gene\-rated.

Assume $n\geq 3$. The group 
$\mathrm{Bir}(\mathbb{P}^n_\mathbb{C})$ is not generated
by $\mathrm{Bir}_{\leq d}(\mathbb{P}^n_\mathbb{C})$
for any integer~$d$ because the birational type of 
the hypersurfaces that are contracted by some 
element of 
$\mathrm{Bir}_{\leq d}(\mathbb{P}^n_\mathbb{C})$ is 
bounded (\emph{see} \cite{Pan:generation}
for more details or Chapter \ref{chapter:gen}, \S 
\ref{subsection:hudsonandpan}).
The fact that $\mathrm{Bir}(\mathbb{P}^n_\mathbb{C})$
is not compactly generated follows from 
Lemma~\ref{lem:compboun}.
\end{proof}

\begin{rem}
Theorem \ref{thm:bf} holds for any field, Theorem \ref{thm:bf2}
holds for any algebraically closed field, and Theorem
\ref{thm:bf3} holds for (locally compact) local field.
\end{rem}


\section{Algebraic elements of the Cremona group}

The goal of this section is the study of algebraic
elements; in particular we will show that the set 
of all these elements is a countable union of closed
subsets but it is not closed.

In this section the considered topology is the 
Zariski topology.

An element 
$\phi\in\mathrm{Bir}(\mathbb{P}^n_\mathbb{C})$ 
is \textsl{algebraic}\index{defi}{algebraic birational map}
if it is contained in an algebraic 
subgroup~$\mathrm{G}$ of 
$\mathrm{Bir}(\mathbb{P}^n_\mathbb{C})$. 

Let us denote by 
$\mathrm{Bir}(\mathbb{P}^n_\mathbb{C})_{\text{alg}}$
\index{not}{$\mathrm{Bir}(\mathbb{P}^n_\mathbb{C})_{\text{alg}}$}
the set of algebraic elements of
$\mathrm{Bir}(\mathbb{P}^n_\mathbb{C})$.  

\begin{thm}[\cite{Blanc:algel}]\label{thm:Blancelalg}
Let $n\geq 2$. 
\begin{itemize}
\item[$\diamond$] There are a closed subset 
$U\subset\mathrm{Bir}(\mathbb{P}^n_\mathbb{C})$ 
canonically homeomorphic to $\mathbb{A}^1_\mathbb{C}$
and a family of birational maps 
$U\to\mathrm{Bir}(\mathbb{P}^n_\mathbb{C})$ such that
algebraic elements of $U$ are unipotent and correspond
to elements of the subgroup of $(\mathbb{C},+)$ 
generated by $1$;

\item[$\diamond$] there is a closed subset 
$S\subset\mathrm{Bir}(\mathbb{P}^n_\mathbb{C})$ 
such that algebraic elements of $S$ are semi-simple
and correspond to elements of 
\[
\big\{(a,\xi)\in\mathbb{A}^1_\mathbb{C}\times(\mathbb{A}^1_\mathbb{C}\smallsetminus\{0\})\,\vert\, a=\xi^k\big\}
\]
for some $k\in\mathbb{Z}$.
\end{itemize}

In particular $\mathrm{Bir}(\mathbb{P}^n_\mathbb{C})_{\text{alg}}$ and the set of unipotent elements 
of $\mathrm{Bir}(\mathbb{P}^n_\mathbb{C})$
are not closed in 
$\mathrm{Bir}(\mathbb{P}^n_\mathbb{C})$.
\end{thm}

Furthermore we will see that $\mathrm{Bir}(\mathbb{P}^n_\mathbb{C})_{\text{alg}}$
is a countable union of closed sets 
of~$\mathrm{Bir}(\mathbb{P}^n_\mathbb{C})$.

\begin{lem}[\cite{Blanc:algel}]\label{lem:plus}
Let $\phi$ be an element of 
$\mathrm{Bir}(\mathbb{P}^n_\mathbb{C})$.
The closure of 
$\big\{\phi^k\,\vert\,k\in\mathbb{Z}\big\}$
in $\mathrm{Bir}(\mathbb{P}^n_\mathbb{C})$
is a closed abelian subgroup of 
$\mathrm{Bir}(\mathbb{P}^n_\mathbb{C})$.
\end{lem}

\begin{proof}
Let us denote by $\Omega$ the closure of 
$\big\{\phi^k\,\vert\,k\in\mathbb{Z}\big\}$
in $\mathrm{Bir}(\mathbb{P}^n_\mathbb{C})$.
For any $j\in\mathbb{Z}$ the set 
$\phi^j(\Omega)$ is a closed subset of 
$\mathrm{Bir}(\mathbb{P}^n_\mathbb{C})$.
It contains
$\big\{\phi^k\,\vert\,k\in\mathbb{Z}\big\}$; 
thus it contains $\Omega$. As a result 
$\phi^k(\Omega)=\Omega$ for any 
$k\in\mathbb{Z}$. Set
\[
M=\big\{\psi\in\mathrm{Bir}(\mathbb{P}^n_\mathbb{C})\,\vert\,\psi(\Omega)\subset\Omega\big\}=\displaystyle\bigcap_{\omega\in\Omega}\Omega\omega^{-1}.
\]
As $M$ is closed and contains 
$\big\{\phi^k\,\vert\,k\in\mathbb{Z}\big\}$, 
the set $M$ contains $\Omega$. Therefore, $M$ 
is closed under composition. Similarly the set 
$\big\{\psi^{-1}\,\vert\,\psi\in\Omega\big\}$ 
is closed in $\mathrm{Bir}(\mathbb{P}^n_\mathbb{C})$
and contains 
$\big\{\phi^k\,\vert\,k\in\mathbb{Z}\big\}$.
The set $\Omega$ is then a subgroup of 
$\mathrm{Bir}(\mathbb{P}^n_\mathbb{C})$.

Let us now prove that $\Omega$ is abelian. 
The centralizer
\[
\mathrm{Cent}(\varphi)=\big\{\psi\in\mathrm{Bir}(\mathbb{P}^n_\mathbb{C})\,\vert\,\psi\circ\varphi=\varphi\circ\psi\big\}
\]
of an element $\varphi$ of
$\mathrm{Bir}(\mathbb{P}^n_\mathbb{C})$ is 
the preimage of the identity by the continuous 
map 
\begin{align*}
& \mathrm{Bir}(\mathbb{P}^n_\mathbb{C})\to\mathrm{Bir}(\mathbb{P}^n_\mathbb{C}) && \psi\mapsto \psi\circ\varphi\circ\psi^{-1}\circ\varphi^{-1}.
\end{align*}
Since a point 
of $\mathrm{Bir}(\mathbb{P}^n_\mathbb{C})$
is closed (Lemma \ref{lem:propbf}), 
$\mathrm{Cent}(\phi)$ is closed.

The closed subgroup $\mathrm{Cent}(\phi)$
of $\mathrm{Bir}(\mathbb{P}^n_\mathbb{C})$
contains 
$\big\{\phi^j\,\vert\,j\in\mathbb{Z}\big\}$ 
hence it contains~$\Omega$. Consequently 
each element of $\Omega$ commutes with 
$\phi$. The set
\[
\big\{\psi\in\mathrm{Bir}(\mathbb{P}^n_\mathbb{C})\,\vert\,\psi\circ\omega=\omega\circ\psi\qquad\forall\,\omega\in\Omega\big\}=\displaystyle\bigcap_{\omega\in\Omega}\mathrm{Cent}(\omega)
\]
is closed and contains
$\big\{\phi^j\,\vert\,j\in\mathbb{Z}\big\}$,
so contains $\Omega$. Therefore, $\Omega$ is 
abelian.
\end{proof}

\begin{pro}[\cite{Blanc:algel}]\label{pro:2plus}
Let $\phi$ be an element of 
$\mathrm{Bir}(\mathbb{P}^n_\mathbb{C})$. 
\begin{itemize}
\item[$\diamond$] If the sequence 
$(\deg \phi^k)_{k\in\mathbb{N}}$ is unbounded, 
then $\phi$ is not contained in any algebraic 
subgroup of $\mathrm{Bir}(\mathbb{P}^n_\mathbb{C})$.

\item[$\diamond$] If the sequence 
$(\deg \phi^k)_{k\in\mathbb{N}}$ is bounded, 
then $\big\{\phi^j\,\vert\,j\in\mathbb{Z}\big\}$
is an abelian algebraic subgroup of 
$\mathrm{Bir}(\mathbb{P}^n_\mathbb{C})$.
\end{itemize}
\end{pro}

A direct consequence is the following 
result:

\begin{cor}[\cite{Blanc:algel}]\label{cor:3plus}
Let $\phi$ be a birational self map of 
$\mathbb{P}^n_\mathbb{C}$. The following 
assertions are equivalent:
\begin{itemize}
\item[$\diamond$] the map $\phi$ is algebraic;

\item[$\diamond$] the sequence 
$(\deg\phi^k)_{k\in\mathbb{N}}$ is bounded, 
{\it i.e.} $\phi$ is elliptic.
\end{itemize}
\end{cor}

\begin{proof}[Proof of Proposition \ref{pro:2plus}]
The first assertion follows from Lemma \ref{lem:agree}.

Let us now focus on the second assertion. Assume that 
the sequence $(\deg\phi^k)_{k\in\mathbb{N}}$ is 
bounded. According to \cite{BassConnelWright} one 
has for any $k$ 
\[
\deg\phi^{-k}\leq(\deg\phi^k)^{n-1}.
\]
As a consequence the set 
$\big\{\phi^j\,\vert\,j\in\mathbb{Z}\big\}$
is contained in 
$\mathrm{Bir}(\mathbb{P}^n_\mathbb{C})_{\leq d}$ 
for some $d$, and so does the closure $\Omega$ of
$\big\{\phi^j\,\vert\,j\in\mathbb{Z}\}$.
Lemma \ref{lem:plus} allows to conclude.
\end{proof}

\begin{pro}
For any $k$, $d\in\mathbb{N}$ set
\[
\mathrm{Bir}(\mathbb{P}^n_\mathbb{C})_{k,d}=\big\{\phi\in\mathrm{Bir}(\mathbb{P}^n_\mathbb{C})\,\vert\,\deg\phi^k\leq d\big\}
\]
and
\[
\mathrm{Bir}(\mathbb{P}^n_\mathbb{C})_{\infty,d}=\big\{\phi\in\mathrm{Bir}(\mathbb{P}^n_\mathbb{C})\,\vert\,\deg\phi^k\leq d\,\forall\,k\in\mathbb{N}\big\}
\]
Then 
\begin{itemize}
\item[$\diamond$] the set
$\mathrm{Bir}(\mathbb{P}^n_\mathbb{C})_{k,d}$
is closed in $\mathrm{Bir}(\mathbb{P}^n_\mathbb{C})$;

\item[$\diamond$] the set
$\mathrm{Bir}(\mathbb{P}^n_\mathbb{C})_{\infty,d}=\displaystyle\bigcap_{i\in\mathbb{N}}
\mathrm{Bir}(\mathbb{P}^n_\mathbb{C})_{i,d}$
is closed in $\mathrm{Bir}(\mathbb{P}^n_\mathbb{C})$;

\item[$\diamond$] the set $\mathrm{Bir}(\mathbb{P}^n_\mathbb{C})_{\text{alg}}$
of all algebraic elements of $\mathrm{Bir}(\mathbb{P}^n_\mathbb{C})$
coincides with the union of all 
$\mathrm{Bir}(\mathbb{P}^n_\mathbb{C})_{\infty,d}$, $d\geq 1$.
\end{itemize}
\end{pro}

\begin{proof}
The set $\mathrm{Bir}(\mathbb{P}^n_\mathbb{C})_{\leq d}$ 
is closed in $\mathrm{Bir}(\mathbb{P}^n_\mathbb{C})$ 
for any $d$ (Corollary \ref{cor:BlancFurter}), and 
the map 
\begin{align*}
& \mathrm{Bir}(\mathbb{P}^n_\mathbb{C})\to\mathrm{Bir}(\mathbb{P}^n_\mathbb{C}), && \varphi\mapsto\varphi^k
\end{align*}
is continuous (Remark \ref{rem:homeo}); 
the set 
$\mathrm{Bir}(\mathbb{P}^n_\mathbb{C})_{k,d}$
is thus closed in 
$\mathrm{Bir}(\mathbb{P}^n_\mathbb{C})$.

The first assertion clearly implies the second one.

The third assertion follows from Corollary 
\ref{cor:3plus}.
\end{proof}

Let us now deal with the first assertion of 
Theorem \ref{thm:Blancelalg}. Assume 
$n\geq 2$. Consider the morphism 
$\rho\colon\mathbb{A}^1_\mathbb{C}\to\mathrm{Bir}(\mathbb{P}^n_\mathbb{C})$ given by 
\begin{small}
\[
a\mapsto\big((z_0:z_1:\ldots:z_n)\dashrightarrow(z_0z_1:z_1(z_1+z_0):z_2(z_1+az_0):z_3z_1:z_4z_1:\ldots:z_nz_1\big).
\]
\end{small}
It is clearly injective. Let
$\widetilde{\rho}\colon\mathbb{P}^1_\mathbb{C}\to W_2$
be the closed embedding given by 
\begin{small}
\[
(\alpha:\beta)\to(\alpha z_0z_1:\alpha z_1(z_1+z_0):z_2(z_1+az_0):\alpha z_3z_1:\alpha z_4z_1:\ldots:\alpha z_nz_1).
\]
\end{small}
Note that $\widehat{\rho}((0:1))$ does not 
belong to $H_2$. However for any
$t\in\mathbb{A}^1_\mathbb{C}$ one has 
$\mathrm{pr}_2(\widehat{\rho}((1:t))=\rho(t)$.
The restriction to $\mathbb{A}^1_\mathbb{C}$;
thus it yields a closed embedding 
$\mathbb{A}^1_\mathbb{C}\to H_2$. 
According to Corollary \ref{cor:BlancFurter}
the restriction of $\pi_2$ to 
$\widehat{\rho}(\mathbb{P}^1_\mathbb{C}\smallsetminus\{(0:1)\})$ 
is an homeomorphism.

\begin{pro}
\begin{itemize}
\item[$\diamond$] For $t\in\mathbb{C}$
the following conditions are equivalent:
\begin{itemize}
\item[- ] $\rho(t)$ is algebraic,

\item[- ] $\rho(t)$ is unipotent,

\item[- ]  $\rho(t)$ is conjugate to
$\rho(0)\colon(z_1,z_2,\ldots,z_n)\to(z_1+1,z_2,\ldots,z_n)$,

\item[- ] $t$ belongs to the subgroup
of $(\mathbb{C},+)$ generated by $1$.
\end{itemize}

\item[$\diamond$] The pull-back by $\rho$ 
of the set of algebraic elements is not
closed.
\end{itemize}
\end{pro}

\begin{proof}
\begin{itemize}
\item[$\diamond$] A direct computation 
yields to 
\begin{small}
\[
\rho(a)^k\colon(z_1,z_2,\ldots,z_n)\mapsto\left(z_1+k,z_2\,\frac{(z_1+a)(z_1+a+1)\ldots(z_1+a+k-1)}{z_1(z_1+1)\ldots(z_1+m-1)},z_3,z_4,\ldots,z_n\right)
\]
\end{small}
In particular the second coordinate of
$\rho(a)^k(z_1,z_2,\ldots,z_n)$ is 
\[
z_2\,\frac{\displaystyle\prod_{i=0}^{k-1}(z_1+a+i)}{\displaystyle\prod_{i=0}^{k-1}(z_1+i)}
\]
If $a$ does not belong to the subgroup
of $(\mathbb{C},+)$ generated by $1$, then
the degree growth of $\rho(a)^k$ is 
linear which implies that $\rho(a)$ is
not algebraic.

If $a$ belongs to the subgroup of 
$(\mathbb{C},+)$ generated by $1$, then
\[
\deg\rho(a)^k\leq \vert k\vert+1\qquad\forall\,k\in\mathbb{N}.
\]
As a consequence $\rho(a)$ is algebraic.
Furthermore $\rho(a)$ is conjugate to
\[
\rho(0)\colon(z_1,z_2,\ldots,z_n)\mapsto(z_1+1,z_2,\ldots,z_n)
\]
via 
\[
(z_1,z_2,\ldots,z_n)\dashrightarrow\left(z_1,\frac{z_2}{z_1(z_1+1)\ldots(z_1+a-1)},z_3,z_4,\ldots,z_n\right)
\]
if $a>0$ or via 
\[
(z_1,z_2,\ldots,z_n)\dashrightarrow\Big(z_1,z_2z_1(z_1-1)\ldots(z_1+a),z_3,z_4,\ldots,z_n\Big)
\]
if $a<0$\footnote{Let us recall that $a$ belongs to 
the subgroup of $(\mathbb{C},+)$ generated by $1$.}. 
In particular $\rho(a)$ is unipotent.

\item[$\diamond$] The second assertion follows from the first 
one and the fact that the subgroup of $(\mathbb{C},+)$ generated 
by $1$ is not closed.
\end{itemize}
\end{proof}

Finally let us prove the second assertion
of Theorem \ref{thm:Blancelalg}. Assume 
$n\geq 2$. Consider the morphism 
\[
\rho\colon\mathbb{A}^1_\mathbb{C}\times(\mathbb{A}^1_\mathbb{C}\smallsetminus\{0\})\to\mathrm{Bir}(\mathbb{P}^n_\mathbb{C})
\]
given by 
\begin{small}
\[
(a,\xi)\mapsto\Big((z_0:z_1:\ldots:z_n)\dashrightarrow(z_0(z_1+z_0):\xi z_1(z_1+z_0):z_2(z_1+az_0):z_3(z_1+z_0):\ldots:z_n(z_1+z_0)\Big).
\]
\end{small}
It is injective. Let 
$\widehat{\rho}\colon\mathbb{P}^2_\mathbb{C}\to W_2$
be the closed embedding given by 
\begin{small}
\[
(\alpha:\beta:\gamma)\dashrightarrow\Big(\alpha z_0(z_1+z_0):\gamma z_1(z_1+z_0):z_2(\alpha z_1+\beta z_0):\alpha z_3(z_1+z_0):\ldots:\alpha z_n(z_1+z_0)\Big).
\]
\end{small}
Note that 
\begin{small}
\[
(z_0:z_1:\ldots:z_n)\dashrightarrow\Big(\alpha z_0(z_1+z_0):\gamma z_1(z_1+z_0):z_2(\alpha z_1+\beta z_0):\alpha z_3(z_1+z_0):\ldots:\alpha z_n(z_1+z_0)\Big)
\]
\end{small}
is a birational map if and only if 
$\alpha\gamma\not=0$. This yields a 
closed embedding
\begin{align*}
& \mathbb{A}^1_\mathbb{C}\times(\mathbb{A}^1_\mathbb{C}\smallsetminus\{0\})\to H_2, && (a,\xi)\mapsto\widehat{\rho}((1:a:\xi)).
\end{align*}
Furthermore 
$\mathrm{pr}_2(\widehat{\rho}(1:a:\xi))=\rho(a,\xi)$.
Proposition \ref{pro:2plus} says that
the restriction of $\pi_2$ to the 
image is a homeomorphism.

\begin{pro}
\begin{itemize}
\item[$\diamond$] For $(a,\xi)\in\mathbb{A}^1_\mathbb{C}\times(\mathbb{A}^1_\mathbb{C}\smallsetminus\{0\})$ the following conditions
are equi\-valent:
\begin{itemize}
\item[- ]  $\rho(a,\xi)$ is algebraic,

\item[- ] $\rho(a,\xi)$ is semi-simple,

\item[- ] $\rho(a,\xi)$ is conjugate to $\rho(1,\xi)\colon(z_1,z_2,\ldots,z_n)\mapsto(\xi z_1,z_2,z_3,\ldots,z_n)$,

\item[- ] there exists $k\in\mathbb{Z}$ such 
that $a=\xi^k$.
\end{itemize}

\item[$\diamond$] The pull-back by $\rho$ of 
the set of algebraic elements is not closed.
\end{itemize}
\end{pro}

\begin{proof}
\begin{itemize}
\item[$\diamond$] Note that 
\begin{small}
\[
\rho(a,\xi)^k\colon(z_1,z_2,\ldots,z_n)\dashrightarrow\left(\xi^kz_1,z_2\,\frac{(z_1+a)(\xi z_1+a)\ldots(\xi^k z_1+a)}{(z_1+1)(\xi z_1+1)\ldots(\xi^{k-1}z_1+1)},z_3,z_4,\ldots,z_n\right).
\]
\end{small}
In particular the second coordinate of $\rho(a,\xi)^k$
is 
\[
z_2\,\frac{\displaystyle\prod_{i=0}^{k-1}(\xi^iz_1+a)}{\displaystyle\prod_{i=0}^{k-1}(\xi^iz_1+1)}.
\]
If $a$ does not belong to 
$\langle\xi\rangle\subset(\mathbb {C},\cdot)$, 
then the degree growth of $\rho(a,\xi)^k$ 
is linear hence $\rho(a,\xi)$ is not algebraic.

If $a$ belongs to 
$\langle\xi\rangle\subset(\mathbb {C},\cdot)$, 
then $a=\xi^k$ for some $k\in\mathbb{Z}$ and for any 
$j\in\mathbb{N}$
\[
\deg\rho(a,\xi)^j\leq\vert k\vert+1,
\]
so $\rho(a,\xi)$ is algebraic. Remark that 
$\rho(a,\xi)$ is conjugate to $\rho(1,\xi)$
via 
\[
(z_1,z_2,\ldots,z_n)\dashrightarrow\left(z_1,\frac{z_2}{z_1(z_1+1)\ldots(z_1+a-1)},z_3,z_4,\ldots,z_n\right)
\]
if $k>0$ and via
\[
(z_1,z_2,\ldots,z_n)\dashrightarrow\Big(z_1,z_2z_1(z_1-1)\ldots(z_1+a),z_3,z_4,\ldots,z_n\Big)
\]
if $k<0$. 

\item[$\diamond$] The second assertion follows
from the first one and the fact that 
\[
\big\{(a,\xi)\in\mathbb{A}^1_\mathbb{C}\times(\mathbb{A}^1_\mathbb{C}\smallsetminus\{0\})\,\vert\, a=\xi^k\text{ for some $k\in\mathbb{Z}$}\big\}
\]
is not closed.
\end{itemize}
\end{proof}

\begin{rem}
Note that all the results of this section hold for
$\mathrm{Bir}(\mathbb{P}^n_\Bbbk)$ whose~$\Bbbk$ is an 
algebraically closed field of characteristic $0$.
\end{rem}

\section{Classification of maximal algebraic subgroups of $\mathrm{Bir}(\mathbb{P}^2_\mathbb{C})$}

In \cite{Blanc:ssgpealg} the author gives a complete classification 
of maximal algebraic subgroups of the plane Cremona group 
and provides
algebraic varieties that parametrize the conjugacy classes. The 
algebraic subgroups of $\mathrm{Bir}(\mathbb{P}^n_\mathbb{C})$ have
been studied for a long time. Enriques established in 
\cite{Enriques} the complete classification of maximal connected 
algebraic subgroups of $\mathrm{Bir}(\mathbb{P}^2_\mathbb{C})$:
every such subgroup is the conjugate of the identity component of
the automorphism group of a minimal rational surface. A modern 
proof was given in \cite{Umemura:maxconnalg1}. The case of 
$\mathrm{Bir}(\mathbb{P}^3_\mathbb{C})$ was treated by Enriques and
Fano and more recently by Umemura 
(\cite{Umemura:dim3, Umemura:maxconnalg1, Umemura:imp}). Demazure
has studied the smooth connected subgroups of 
$\mathrm{Bir}(\mathbb{P}^n_\mathbb{C})$ that contain a split torus
of dimension $n$ (\emph{see} \cite{Demazure:sousgroupesalgebriques}).
Only a few results are known for non-connected subgroups even in
dimension $2$. Nevertheless there are a lot of statements in the 
case of finite subgroups which are algebraic ones
(\cite{Wiman, BayleBeauville, deFernex, BeauvilleBlanc, Beauville, Iskovskikh, DolgachevIskovskikh, Blanc:SMF, Blanc:CRAS}) and we deal with in Chapter \ref{chapter:finite}.
But these results do not show which finite groups are maximal
algebraic subgroups. As mentioned in \cite{DolgachevIskovskikh}
there are some remaining open questions like the description of
the algebraic varieties that parameterize conjugacy classes of 
finite subgroups $\mathrm{G}$ of $\mathrm{Bir}(\mathbb{P}^2_\mathbb{C})$.
Blanc gives an answer to this question for 
\begin{itemize}
\item[$\diamond$] abelian finite subgroups $\mathrm{G}$ of 
$\mathrm{Bir}(\mathbb{P}^2_\mathbb{C})$ whose elements  
 do not fix a curve of positive genus (\cite{Blanc:abelien});

\item[$\diamond$] finite cyclic subgroups of 
$\mathrm{Bir}(\mathbb{P}^2_\mathbb{C})$ (\emph{see}
\cite{Blanc:cyclic});

\item[$\diamond$] maximal algebraic subgroups of 
$\mathrm{Bir}(\mathbb{P}^2_\mathbb{C})$ (\emph{see}
\cite{Blanc:ssgpealg}).
\end{itemize}

Before specifying Blanc results let us recall some notions. If 
$S$ is a projective smooth rational surface and $\mathrm{G}$
a subgroup of $\mathrm{Aut}(S)$ we say that $(\mathrm{G},S)$
is a \textsl{pair}\index{defi}{pair}. A birational map 
$\varphi\colon X\dashrightarrow Y$ is 
\textsl{$\mathrm{G}$-equivariant}
\index{defi}{$\mathrm{G}$-equivariant (birational map)} if 
the inclusion 
$\varphi\circ\mathrm{G}\circ\varphi^{-1}\subset\mathrm{Aut}(Y)$
holds. The pair $(\mathrm{G},S)$ is 
\textsl{minimal}\index{defi}{minimal (pair)} if every 
birational $\mathrm{G}$-equivariant morphism 
$\varphi\colon S\dashrightarrow S'$ where $S'$ is a projective,
smooth surface, is an isomorphism. A morphism 
$\pi\colon S\to \mathbb{P}^1_\mathbb{C}$ is a 
\textsl{conic bundle}\index{defi}{conic bundle} if all 
generic fibers of $\pi$ are isomorphic to $\mathbb{P}^1_\mathbb{C}$
and if there exists a finite number of singular fibers which 
are the transverse union of two curves isomorphic to 
$\mathbb{P}^1_\mathbb{C}$.

\subsection{del Pezzo surfaces and their automorphism groups}

 A \textsl{del Pezzo surface}\index{defi}{del Pezzo surface} is a smooth projective surface $S$ such that the anti-canonical divisor $-K_S$ is 
ample. Let us recall the classification of del Pezzo surfaces. The 
number $d=K_S^2$ is called the \textsl{degree} of 
$S$\index{defi}{degree (of 
a del Pezzo surface)}. By Noether's formula $1\leq d\leq 9$. For 
$d\geq 3$, the anticanonical linear system $\vert -K_S\vert$ maps 
$S$ onto a non-singular surface of degree $d$ in $\mathbb{P}^d_\mathbb{C}$.
If $d=9$, then $S\simeq\mathbb{P}^2_\mathbb{C}$. If $d=8$, then
$S\simeq \mathbb{P}^1_\mathbb{C}\times\mathbb{P}^1_\mathbb{C}$ or
$S\simeq\mathbb{F}_1$. For $d\leq 7$ a del Pezzo surface $S$
is isomorphic to the blow up of $n=9-d$ points in~$\mathbb{P}^2_\mathbb{C}$ in general position, that is 
\begin{itemize}
\item[$\diamond$] no three of them are colinear,
\item[$\diamond$] no six are on the same conic,
\item[$\diamond$] if $n=8$, then the points are not on a 
plane cubic which has one of them as its singular point.
\end{itemize}

There exist (\cite[Chapter 8]{Dolgachev}) 
\begin{itemize}
\item[$\diamond$] a unique isomorphism class of del Pezzo 
surfaces of degree $5$ (resp. $6$, resp. $7$, resp. $9$), 
\item[$\diamond$] two isomorphism classes of del Pezzo 
surfaces of degree $8$,
\item[$\diamond$] and infinitely many isomorphism classes
of del Pezzo surfaces of degree $1$, (resp. $2$, resp.~$3$,
resp. $4$).
\end{itemize}

We will see that automorphism groups of del Pezzo surfaces 
are algebraic subgroups of~$\mathrm{Bir}(\mathbb{P}^2_\mathbb{C})$
and that they are finite if and only if the degree of
the corresponding surface is~$\leq 5$. If $S$ is a del 
Pezzo surface of degree $5$, then $\mathrm{Aut}(S)=\mathfrak{S}_5$.
Automorphism groups of del Pezzo surfaces of degree $\leq 4$ 
are described in \cite[\S 10]{DolgachevIskovskikh}. In 
particular the authors got the following:

\begin{thm}[\cite{DolgachevIskovskikh}]\label{thm:648}
If the automorphism group of a del Pezzo surface
is finite, then it has order at most $648$.
\end{thm}

\begin{lem}[\cite{Urech:ellipticsubgroups}]\label{lem:te1}
If the automorphism group of a del Pezzo surface is finite, 
then it can be embedded into $\mathrm{GL}(8,\mathbb{C})$.
\end{lem}

\begin{proof}
Let $S$ be a del Pezzo surface such that $\mathrm{Aut}(S)$
is finite. Then $\deg S\leq 5$ and $S$ is isomorphic to
$\mathrm{Bl}_{p_1,p_2,\ldots,p_r}\mathbb{P}^2_\mathbb{C}$
where $4\leq r=9-\deg S\leq 8$ and $p_1$, $p_2$, $\ldots$,
$p_r$ are general points of $\mathbb{P}^2_\mathbb{C}$. 
Denote by $\mathbf{e}_0$ the pullback of the class of a line 
and by $\mathbf{e}_{p_i}$ the class of the exceptional line 
$E_{p_i}$ corresponding to the point $p_i$. The dimension
of the N\'eron-Severi space $\mathrm{NS}(S)\otimes\mathbb{R}$
is $r+1$ and $\mathbf{e}_0$, $\mathbf{e}_{p_1}$, $\mathbf{e}_{p_2}$, $\ldots$,
$\mathbf{e}_{p_r}$ is a basis of $\mathrm{NS}(S)\otimes\mathbb{R}$. 
Note that the equality $\mathbf{e}_{p_i}\cdot\mathbf{e}_{p_i}=-1$ implies 
that $E_{p_i}$ is the only representative of $\mathbf{e}_{p_i}$ 
on $S$. 

If $\varphi\in\mathrm{Aut}(S)$ acts as the identity on 
$\mathrm{NS}(S)\otimes\mathbb{R}$, then $\varphi$
preserves the exceptional lines $E_{p_i}$ for
$1\leq i\leq r$. Hence $\varphi$ induces an automorphism
of $\mathbb{P}^2_\mathbb{C}$ that fixes $p_1$, $p_2$, 
$\ldots$, $p_r$. As $r\geq 4$ and as the $p_i$ are in 
general position the induced automorphism of 
$\mathbb{P}^2_\mathbb{C}$ is the identity. The action of 
$\mathrm{Aut}(S)$ on $\mathrm{NS}(S)\otimes\mathbb{R}$
is thus faithful and we get a faithful representation
\[
\mathrm{Aut}(S)\to\mathrm{GL}(r+1,\mathbb{C}).
\]
Any element $\varphi$ of $\mathrm{Aut}(S)$ fixes $K_S$;
as a result the one-dimensional subspace 
$\mathbb{R}\cdot K_S$ of $\mathrm{NS}(S)\otimes\mathbb{R}$
is fixed. By projecting the orthogonal complement of 
$K_S$ in $\mathrm{NS}(S)\otimes\mathbb{R}$ we obtain
a faithful representation of $\mathrm{Aut}(S)$ 
into $\mathrm{GL}(r,\mathbb{C})$.
\end{proof}

A del Pezzo surface of degree $6$ is isomorphic to the 
blow up of the complex projective plane in three general
points, {\it i.e.} isomorphic to the surface
\[
S_6=\big\{\big((z_0:z_1:z_2),(a:b:c)\big)\in\mathbb{P}^2_\mathbb{C}\times\mathbb{P}^2_\mathbb{C}\,\vert\, az_0=bz_1=cz_2\big\}.
\]
The automorphism group of $S_6$ is isomorphic to 
$(\mathbb{C}^*)^2\rtimes\big(\mathfrak{S}_3\times\faktor{\mathbb{Z}}{2\mathbb{Z}}\big)$ 
where $\mathfrak{S}_3$ acts by permuting the coordinates of
the two factors simultaneously, $\faktor{\mathbb{Z}}{2\mathbb{Z}}$
exchanges the two factors and $d\in(\mathbb{C}^*)^2$ acts 
as follows
\[
d\cdot\big((z_0:z_1:z_2),(a:b:c)\big)=\big(d(z_0:z_1:z_2):d^{-1}(a:b:c)\big).
\]
In other words $\mathrm{Aut}(S_6)$ is conjugate to 
$\big(\mathfrak{S}_3\times\faktor{\mathbb{Z}}{2\mathbb{Z}}\big)\ltimes \mathrm{D}_2 \subset\mathrm{GL}(2,\mathbb{Z})\ltimes\mathrm{D}_2$.

\begin{lem}[\cite{Urech:ellipticsubgroups}]\label{lem:te2}
The group $\mathrm{Aut}(S_6)$ can be embedded
in $\mathrm{GL}(6,\mathbb{C})$. 
\end{lem}

\begin{proof}
Consider the rational map
\begin{align*}
& \phi\colon\mathbb{P}^2_\mathbb{C}\dashrightarrow\mathbb{P}^6_\mathbb{C}, 
&& (z_0:z_1:z_2)\dashrightarrow(z_0^2z_1:z_0^2z_2:z_0z_1^2:z_1^2z_2:z_0z_2^2:z_1z_2^2:z_0z_1z_2).
\end{align*}
The rational action of 
$(\mathfrak{S}_3\times\faktor{\mathbb{Z}}{2\mathbb{Z}})\ltimes \mathrm{D}_2$ 
on $\phi(\mathbb{P}^2_\mathbb{C})$ extends to a regular 
action on~$\mathbb{P}^6_\mathbb{C}$ that preserves the 
affine space given by $z_6\not=0$. This yields an 
embedding of 
$(\mathfrak{S}_3\times\faktor{\mathbb{Z}}{2\mathbb{Z}})\ltimes\mathrm{D}_2$ 
into $\mathrm{GL}(6,\mathbb{C})$.
\end{proof}

\subsection{Hirzebruch surfaces and their automorphism groups}\label{subsec:hirz}

Let us introduce the Hirzebruch surfaces. Consider the surface 
$\mathbb{F}_1$ obtained by blowing up 
$(1:0:0)\in\mathbb{P}^2_\mathbb{C}$; it is a compactification of 
$\mathbb{C}^2$ which has a natural fibration corresponding to the
lines $z_1=$ constant. The divisor at infinity is the union of two 
rational curves which intersect in one point:
\begin{itemize}
\item[$\diamond$] one of them is the strict transform of the line
at infinity in $\mathbb{P}^2_\mathbb{C}$, it is a fiber denoted by
$f_1$;

\item[$\diamond$] the other one, denoted by $s_1$, is the exceptional
divisor which is a section for the fibration.
\end{itemize}

Furthermore $f_1^2=0$ and $s_1^2=-1$.
More generally for any $n$, $\mathbb{F}_n$ is a compactification of
$\mathbb{C}^2$ with a rational fibration and such that the divisor
at infinity is the union of two transversal rational curves: a 
fiber $f_n$ and a section~$s_n$ of self-intersection $-n$. These
surfaces are called \textsl{Hirzebruch surfaces}\index{defi}{Hirzebruch
surfaces}. One can go from~$\mathbb{F}_n$ to $\mathbb{F}_{n+1}$ as 
follows. Consider the surface $\mathbb{F}_n$. Set $p=s_n\cap f_n$. Let 
$\mathfrak{p}_1$ be the blow up of~$p\in\mathbb{F}_n$ and let 
$\mathfrak{p}_2$ be the contraction of the strict transform 
$\widetilde{f_n}$ of $f_n$. One goes
from $\mathbb{F}_n$ to~$\mathbb{F}_{n+1}$ via 
$\mathfrak{p}_2\circ\mathfrak{p}_1^{-1}$. We can
also go from $\mathbb{F}_{n+1}$ to $\mathbb{F}_n$ via 
$\widetilde{\mathfrak{p}_2}\circ\widetilde{\mathfrak{p}_1}^{-1}$ where 
\begin{itemize}
\item[$\diamond$] $\widetilde{\mathfrak{p}_1}$ is the blow-up of a 
point $q$ such that $q\in f_{n+1}$, $q\not\in s_{n+1}$;

\item[$\diamond$] $\widetilde{\mathfrak{p}_2}$ is the contraction of the strict transform 
$\widetilde{f_{n+1}}$ of $f_{n+1}$.
\end{itemize}

We will say that both $\mathfrak{p}_2\circ\mathfrak{p}_1^{-1}$ and
$\widetilde{\mathfrak{p}_2}\circ\widetilde{\mathfrak{p}_1}^{-1}$ 
are \textsl{elementary transformations}\index{defi}{elementary transformation}.

The $n$-th Hirzebruch surface $\mathbb{F}_n=\mathbb{P}\big(\mathcal{O}_{\mathbb{P}^1_\mathbb{C}}\oplus\mathcal{O}_{\mathbb{P}^1_\mathbb{C}}(n)\big)$
is isomorphic to the hypersurface 
\[
\big\{([x_0,x_1],[y_0,y_1,y_2])\in\mathbb{P}^1_\mathbb{C}\times\mathbb{P}^2_\mathbb{C}\,\vert\, x_0^ny_1-x_1^ny_2=0\big\}
\]
of $\mathbb{P}^1_\mathbb{C}\times\mathbb{P}^2_\mathbb{C}$.

Their automorphism groups are 
\begin{eqnarray*}
\mathrm{Aut}(\mathbb{P}^2_\mathbb{C}\times\mathbb{P}^1_\mathbb{C})&=&(\mathrm{PGL}(2,\mathbb{C})\times\mathrm{PGL}(2,\mathbb{C}))\rtimes\langle(z_0,z_1)\mapsto(z_1,z_0)\rangle, \\
\mathrm{Aut}(\mathbb{P}^2_\mathbb{C})&=&\mathrm{PGL}(3,\mathbb{C})
\end{eqnarray*}
and
\begin{eqnarray*}
\mathrm{Aut}(\mathbb{F}_n)&=&\left\{(z_0,z_1)\mapsto\left(\frac{az_0+P(z_1)}{(\gamma z_1+\delta)^n},\frac{\alpha z_1+\beta}{\gamma z_1+\delta}\right)\,\big\vert\,\right.\\
& & \hspace{0.5cm}\left.\left(\begin{array}{cc}\alpha & \beta \\ \gamma & \delta\end{array}\right)\in\mathrm{PGL}(2,\mathbb{C}),\,a\in\mathbb{C}^*,\,P\in\mathbb{C}[z_1],\,\deg P\leq n\right\}.
\end{eqnarray*}
In other words as soon as $n\geq 2$ the group 
$\mathrm{Aut}(\mathbb{F}_n)$ is isomorphic to 
$\mathbb{C}[z_0,z_1]_n\rtimes\faktor{\mathrm{GL}(2,\mathbb{C})}{\mu_n}$
where $\mu_n\subset\mathrm{GL}(2,\mathbb{C})$ is the
subgroup of $n$-torsion elements in the center of
$\mathrm{GL}(2,\mathbb{C})$.

\begin{lem}[\cite{Urech:ellipticsubgroups}]\label{lem:ru1}
If $n\geq 2$ is even, then $\faktor{\mathrm{GL}(2,\mathbb{C})}{\mu_n}$ 
is isomorphic as an algebraic group to 
$\mathrm{PGL}(2,\mathbb{C})\times\mathbb{C}^*$.

If $n$ is odd, then $\faktor{\mathrm{GL}(2,\mathbb{C})}{\mu_n}$ is 
isomorphic as an algebraic group to $\mathrm{PGL}(2,\mathbb{C})$.

In particular all finite subgroups of $\mathrm{Aut}(\mathbb{F}_n)$ 
can be embedded into 
$\mathrm{PGL}(2,\mathbb{C})\times\mathrm{PGL}(2,\mathbb{C})$
as soon as $n\geq 2$.
\end{lem}

\subsection{Automorphism groups of exceptional conic bundles}

An \textsl{exceptional conic bundle}\index{defi}{exceptional conic bundle}
$S$ is a conic bundle with singular fiber above $2n$ points in 
$\mathbb{P}^1_\mathbb{C}$ and with two sections $s_1$ and $s_2$
of self-intersection $-n$, where $n\geq 2$ (\emph{see} 
\cite{Blanc:ssgpealg}).

\begin{lem}[\cite{Blanc:ssgpealg}]\label{lem:ru2}
Let $\pi\colon S\to\mathbb{P}^1_\mathbb{C}$ be an exceptional
conic bundle. Then $\mathrm{Aut}(S,\pi)$ is isomorphic to a 
subgroup of 
$\mathrm{PGL}(2,\mathbb{C})\times\mathrm{PGL}(2,\mathbb{C})$.
\end{lem}

\subsection{$\Big(\faktor{\mathbb{Z}}{2\mathbb{Z}}\Big)^2$-conic bundles}

A conic bundle $\pi\colon S\to\mathbb{P}^1_\mathbb{C}$ is a 
\textsl{$\Big(\faktor{\mathbb{Z}}{2\mathbb{Z}}\Big)^2$-conic bundle}\index{defi}{$\Big(\faktor{\mathbb{Z}}{2\mathbb{Z}}\Big)^2$-conic bundle}
if 
\begin{itemize}
\item[$\diamond$] the group
$\mathrm{Aut}\Big(\faktor{S}{\mathbb{P}^1_\mathbb{C}}\Big)$
is isomorphic to $\Big(\faktor{\mathbb{Z}}{2\mathbb{Z}}\Big)^2$,

\item[$\diamond$] each of the three involutions of 
$\mathrm{Aut}\Big(\faktor{S}{\mathbb{P}^1_\mathbb{C}}\Big)$ fixes
pointwise an irreducible curve $C$ such that
$\pi\colon C\to\mathbb{P}^1_\mathbb{C}$ is a 
double covering that is ramified over a positive
even number of points. 
\end{itemize}

The automorphism group $\mathrm{Aut}(S,\pi)$ of a
$\Big(\faktor{\mathbb{Z}}{2\mathbb{Z}}\Big)^2$-conic bundle is finite; 
its structure is given by the following exact sequence
(\cite{Blanc:ssgpealg})
\[
1 \longrightarrow V\longrightarrow\mathrm{Aut}(S,\pi)\longrightarrow H_V\longrightarrow 1
\]
where $V\simeq\Big(\faktor{\mathbb{Z}}{2\mathbb{Z}}\Big)^2$ 
and $H_V$ is a finite subgroup of 
$\mathrm{Aut}(\mathbb{P}^1_\mathbb{C})$.

Note that we also have the following property:

\begin{lem}[\cite{Urech:ellipticsubgroups}]\label{lem:ru3}
Let $\mathrm{G}\subset\mathrm{Bir}(\mathbb{P}^2_\mathbb{C})$
be an infinite torsion group. Assume that for any finitely
generated subgroup $\Gamma\subset\mathrm{G}$ there exists 
a $\Big(\faktor{\mathbb{Z}}{2\mathbb{Z}}\Big)^2$-conic bundle 
$S\to\mathbb{P}^1_\mathbb{C}$ such that $\Gamma$ is 
conjugate to a subgroup of $\mathrm{Aut}(S,\pi)$. Then 
any finitely generated subgroup of $\mathrm{G}$ is 
isomorphic to a subgroup of
$\mathrm{PGL}(2,\mathbb{C})\times\mathrm{PGL}(2,\mathbb{C})$.
\end{lem}

\subsection{Blanc results}

First Blanc proved:

\begin{thm}[\cite{Blanc:ssgpealg}]
Every algebraic subgroup of $\mathrm{Bir}(\mathbb{P}^2_\mathbb{C})$ 
is contained in a maximal algebraic subgroup of 
$\mathrm{Bir}(\mathbb{P}^2_\mathbb{C})$. 

The maximal algebraic subgroups of the plane Cremona
group are the 
conjugate of the groups $\mathrm{G}=\mathrm{Aut}(S,\pi)$ where 
$S$ is a rational surface and $\pi\colon S\to Y$ is a morphism 
such that
\begin{enumerate}
\item $Y$ is a point, $\mathrm{G}=\mathrm{Aut}(S)$ and $S$ is 
one of the following:
\begin{itemize}
\item[$\diamond$] $\mathbb{P}^2_\mathbb{C}$, 
$\mathbb{P}^1_\mathbb{C}\times\mathbb{P}^1_\mathbb{C}$; 
\item[$\diamond$] a del Pezzo surface of degree $1$, $4$, $5$ or 
$6$; 
\item[$\diamond$] a del Pezzo surface of degree $3$ (resp. $2$) 
such that the pair $(\mathrm{Aut}(S),S)$ is minimal and such
that the fixed points 
of the action of $\mathrm{Aut}(S)$ on~$S$ are lying on exceptional 
curves; 
\end{itemize}

\item $Y\simeq\mathbb{P}^1_\mathbb{C}$ and $\pi$ is one of the 
following conic bundles:
\begin{itemize}
\item[$\diamond$] the fibration by lines of the Hirzebruch surface 
$\mathbb{F}_n$ for $n\geq 2$; 
\item[$\diamond$] an exceptional conic bundle with at least $4$ 
singular fibers; 
\item[$\diamond$] a $\Big(\faktor{\mathbb{Z}}{2\mathbb{Z}}\Big)^2$-conic bundle such 
that $S$ is not a del Pezzo surface. 
\end{itemize}
\end{enumerate}
\end{thm}

Moreover, in all these cases, the pair $(\mathrm{G},S)$ is minimal 
and the fibration $\pi\colon S\to~Y$ is a $\mathrm{G}$-Mori fibration 
which is birationally \textsl{superrigid}\index{defi}{superrigid}. 
This means that two such groups $\mathrm{G}=\mathrm{Aut}(S,\pi)$ 
and $\mathrm{G}'=\mathrm{Aut}(S',\pi')$ are conjugate if and only 
if there exists an isomorphism $S\to S'$ which sends fibers of $\pi$ 
onto fibers of $\pi'$.

Then Blanc described more precisely the structure of these 
minimal algebraic subgroups of $\mathrm{Bir}(\mathbb{P}^2_\mathbb{C})$.
Furthermore he provides algebraic varieties that paramete\-rize
the conjugacy classes of these groups:

\begin{thm}[\cite{Blanc:ssgpealg}]\label{thm:blanc11cases}
The maximal algebraic subgroups of $\mathrm{Bir}(\mathbb{P}^2_\mathbb{C})$ belong 
up to conjugacy to one of the eleven following families:
\begin{enumerate}
\item[(1)] $\mathrm{Aut}(\mathbb{P}^2_\mathbb{C})\simeq\mathrm{PGL}(3,\mathbb{C})$;

\item[(2)] $\mathrm{Aut}(\mathbb{P}^1_\mathbb{C}\times\mathbb{P}^1_\mathbb{C})\simeq\big(\mathrm{PGL}(2,\mathbb{C})\big)^2\rtimes\faktor{\mathbb{Z}}{2\mathbb{Z}}$;

\item[(3)] $\mathrm{Aut}(S)\simeq(\mathbb{C}^*)^2\rtimes\Big(\mathfrak{S}_3\times\faktor{\mathbb{Z}}{2\mathbb{Z}}\Big)$ 
where $S$ is the del Pezzo surface of degree $6$;

\item[(4)] $\mathrm{Aut}(\mathbb{F}_n)\simeq\mathbb{C}^{n+1}\rtimes\faktor{\mathrm{GL}(2,\mathbb{C})}{\mu_n}$
where $\mu_n$ is the $n$-th torsion of the center of $\mathrm{GL}(2,\mathbb{C})$
with $n\geq 2$;

\item[(5)] $\mathrm{Aut}(S,\pi)$ where $(S,\pi)$ is an exceptional conic bundle with 
singular fibers over a set $\Delta\subset\mathbb{P}^1_\mathbb{C}$ of $2n$ distinct
points, $n\geq 2$; the projection of $\mathrm{Aut}(S,\pi)$ onto
$\mathrm{PGL}(2,\mathbb{C})$ gives an exact sequence
\[
1 \longrightarrow \mathbb{C}^*\rtimes\faktor{\mathbb{Z}}{2\mathbb{Z}}\longrightarrow\mathrm{Aut}(S,\pi)\longrightarrow H_\Delta\longrightarrow 1
\]
where $H_\Delta$ is the finite subgroup of $\mathrm{PGL}(2,\mathbb{C})$ formed 
by elements that preserve $\Delta$;

\item[(6)] $\mathrm{Aut}(S)\simeq\mathfrak{S}_5$ where $S$ is the del Pezzo 
surface of degree $5$;

\item[(7)] $\mathrm{Aut}(S)\simeq\Big(\faktor{\mathbb{Z}}{2\mathbb{Z}}\Big)^4\rtimes H_S$ 
where $S$ is a del Pezzo surface of degree $4$ obtained by blowing up $5$ 
points in $\mathbb{P}^2_\mathbb{C}$ and $H_S$ is the group of automorphisms
of $\mathbb{P}^2_\mathbb{C}$ that preserve this set of points;

\item[(8)] $\mathrm{Aut}(S)$ where $S$ is a del Pezzo surface of degree $3$ of the 
following form 
\begin{itemize}
\item[$\diamond$] the triple cover of $\mathbb{P}^2_\mathbb{C}$ ramified 
along a smooth cubic $\Gamma$. If $S$ is the Fermat cubic, then 
$\mathrm{Aut}(S)=\Big(\faktor{\mathbb{Z}}{3\mathbb{Z}}\Big)^3\rtimes\mathfrak{S}_4$,
otherwise we have an exact sequence 
\[
1 \longrightarrow \faktor{\mathbb{Z}}{3\mathbb{Z}}\longrightarrow\mathrm{Aut}(S)\longrightarrow H_\Gamma\longrightarrow 1
\]
where $H_\Gamma$ is the group of automorphisms of $\mathbb{P}^2_\mathbb{C}$
that preserve $\Gamma$, $H_\Gamma$ contains a subgroup isomorphic to 
$\Big(\faktor{\mathbb{Z}}{3\mathbb{Z}}\Big)^2$;

\item[$\diamond$] the Clebsch cubic surface  
whose automorphism group is isomorphic to $\mathfrak{S}_5$;
 
\item[$\diamond$] a cubic surface given by $z_0^3+z_0(z_1^2+z_2^2+z_3^2)+\lambda z_1z_2z_3=0$ 
for some $\lambda\in\mathbb{C}$, $9\lambda^3\not=8\lambda$, 
$8\lambda^3\not=-1$ and whose automorphism group is isomorphic to 
$\mathfrak{S}_4$;
\end{itemize}
\item[(9)] $\mathrm{Aut}(S)\simeq\faktor{\mathbb{Z}}{2\mathbb{Z}}\rtimes H_S$ where $S$
is a del Pezzo surface of degree $2$ which is a double cover of a smooth
quartic $Q_S\subset\mathbb{P}^2_\mathbb{C}$ such that $H_S=\mathrm{Aut}(Q_S)$
acts without fixed point on the quartic without its bitangent points;

\item[(10)] $\mathrm{Aut}(S)$ where $S$ is a del Pezzo surface of degree $1$, 
double cover of a quadratic cone $Q$, ramified along a curve $\Gamma_S$
of degree $6$, complete intersection of $Q$ with a cubic surface of 
$\mathbb{P}^3_\mathbb{C}$. We have the following exact sequence
\[
1 \longrightarrow \faktor{\mathbb{Z}}{2\mathbb{Z}}\longrightarrow\mathrm{Aut}(S)\longrightarrow H_S\longrightarrow 1
\]
where $H_S$ denotes the automorphism group of $Q$ preserving the curve 
$\Gamma_S$;

\item[(11)] $\mathrm{Aut}(S,\pi)$ where $(S,\pi)$ is a 
$\Big(\faktor{\mathbb{Z}}{2\mathbb{Z}}\Big)^2$-conic bundle such that $S$ is not 
a del Pezzo surface. The projection of $\mathrm{Aut}(S,\pi)$ onto 
$\mathrm{PGL}(2,\mathbb{C})$ gives the following exact sequence 
\[
1 \longrightarrow V\longrightarrow\mathrm{Aut}(S,\pi)\longrightarrow H_V\longrightarrow 1
\]
where $V\simeq\Big(\faktor{\mathbb{Z}}{2\mathbb{Z}}\Big)^2$ contains three 
involutions fixing an hyperelliptic
curve ramified over points of $p_1$, $p_2$, $p_3\subset\mathbb{P}^1_\mathbb{C}$
and $H_V\subset\mathrm{Aut}(\mathbb{P}^1_\mathbb{C})$ is the finite 
subgroup preserving the set $\big\{p_1,\,p_2,\,p_3\big\}$.
\end{enumerate}
\end{thm}

The eleven families are disjoint and the conjugacy classes in any family are
parameterized respectively by
\begin{itemize}
\item[$(1)$], $(2)$, $(3)$, $(6)$ the point;

\item[$(4)$] there is only one conjugacy class for any integer $n\geq 2$;

\item[$(5)$] for any integer $n\geq 2$ the set of $2n$ points of 
$\mathbb{P}^1_\mathbb{C}$ modulo the action of 
$\mathrm{Aut}(\mathbb{P}^1_\mathbb{C})=\mathrm{PGL}(2,\mathbb{C})$;

\item[$(7)$] the isomorphism classes of del Pezzo surfaces of degree $4$;

\item[$(8)$] the isomorphism classes of cubic surfaces given respectively 
\begin{itemize}
\item[$\diamond$] by the isomorphism classes of elliptic curves;

\item[$\diamond$] for the Clebsch surface there is only one isomorphism class;

\item[$\diamond$] by the classes of 
$\big\{\lambda\in\mathbb{C}\,\vert\, 9\lambda^3\not=8\lambda,\,8\lambda^3\not=-1\big\}$
modulo the equivalence $\lambda\sim -\lambda$.
\end{itemize}

\item[$(9)$] the isomorphism classes of smoooth quartics of 
$\mathbb{P}^2_\mathbb{C}$ having automorphism groups acting without 
fixed points on the quartic without its bitangent points;

\item[$(10)$] the isomorphism classes of del Pezzo surfaces of degree $1$;

\item[$(11)$] the triplets of ramification 
$\big\{p_1,\,p_2,\,p_3\big\}\subset\mathbb{P}^1_\mathbb{C}$ that determine
$\big(\faktor{\mathbb{Z}}{2\mathbb{Z}}\big)^2$ conic bundles on surfaces that are 
not del Pezzo ones, modulo the action of $\mathbb{P}^1_\mathbb{C}$.
\end{itemize}

 The approach of Blanc used the modern viewpoint of Mori's theory and 
 Sarkisov's 
 program, aiming a generalization in higher dimension:
\begin{itemize}
\item[$\diamond$] he described each maximal 
  algebraic subgroup of the classification as a $\mathrm{G}$-Mori fibration;
\item[$\diamond$] he 
then proved that any algebraic subgroup is contained in one of the groups of 
the classification;
\item[$\diamond$] he also showed that any group of the classification is a 
minimal $\mathrm{G}$-fibration that is furthermore superrigid.
\end{itemize}

Lemmas \ref{lem:te1}, \ref{lem:te2} and Theorem \ref{thm:blanc11cases} 
allow to prove the following statement:

\begin{lem}[\cite{Urech:ellipticsubgroups}]\label{lem:ru4}
Let $\mathrm{G}$ be a subgroup of the plane Cremona group.
Assume that $\mathrm{G}$ is conjugate to an automorphism group of 
a del Pezzo surface~$S$. Then $\mathrm{G}$ can be embedded into 
$\mathrm{GL}(8,\mathbb{C})$.
\end{lem}

\begin{proof}
If $\deg S\leq 5$, then $\mathrm{Aut}(S)$ is finite and Lemma 
\ref{lem:te1} allows to conclude.

If $\deg S=6$, then $\mathrm{Aut}(S)$ can be embedded into 
$\mathrm{GL}(8,\mathbb{C})$ (Lemma \ref{lem:te2}).

If $\deg S=7$, then $\mathrm{Aut}(S)$ is conjugate to a 
subgroup of 
\[
\mathrm{Aut}(\mathbb{P}^1_\mathbb{C}\times\mathbb{P}^1_\mathbb{C})\simeq\big(\mathrm{PGL}(2,\mathbb{C})\times\mathrm{PGL}(2,\mathbb{C})\big)\rtimes\faktor{\mathbb{Z}}{2\mathbb{Z}}\subset \mathrm{GL}(6,\mathbb{C}).
\]

If $\deg S=8$, then $S$ is isomorphic either to 
$\mathbb{F}_0=\mathbb{P}^1_\mathbb{C}\times\mathbb{P}^1_\mathbb{C}$
or to $\mathbb{F}_1$. On the one hand 
\[
\mathrm{Aut}(\mathbb{P}^1_\mathbb{C}\times\mathbb{P}^1_\mathbb{C})\simeq\big(\mathrm{PGL}(2,\mathbb{C})\times\mathrm{PGL}(2,\mathbb{C})\big)\rtimes\faktor{\mathbb{Z}}{2\mathbb{Z}}\subset \mathrm{GL}(6,\mathbb{C}).
\]
and on the other hand $\mathrm{Aut}(\mathbb{F}_1)$ is 
not a maximal algebraic subgroup of 
$\mathrm{Bir}(\mathbb{P}^2_\mathbb{C})$ (Theorem~\ref{thm:blanc11cases}).

If $\deg S=9$, then $S\simeq\mathbb{P}^2_\mathbb{C}$ and 
$\mathrm{Aut}(S)=\mathrm{PGL}(3,\mathbb{C})\subset\mathrm{GL}(8,\mathbb{C})$.
\end{proof}


\section{Closed normal subgroups of the Cremona group}\label{section:closednormalsubgroups}

As we have seen we can endow the Cremona group with a natural 
Zariski topology induced by morphisms 
$V\to\mathrm{Bir}(\mathbb{P}^n_\mathbb{C})$ where $V$ is 
an algebraic variety.

In \cite{Shafarevich} Mumford discussed properties of 
$\mathrm{Bir}(\mathbb{P}^2_\mathbb{C})$ and in particular asked 
if it is a simple group with respect to the Zariski topology, 
{\it i.e.} if every closed normal subgroup of 
$\mathrm{Bir}(\mathbb{P}^2_\mathbb{C})$ is trivial. Blanc 
and Zimmermann provided 
an affirmative answer to Mumford question:

\begin{thm}[\cite{BlancZimmermann}]\label{thm:BlancZimmermann}
Let $n\geq 1$ be an integer. The group $\mathrm{Bir}(\mathbb{P}^n_\mathbb{C})$
is topolo\-gically simple when endowed with the Zariski topology. 
\end{thm}

\begin{rem}
This statement was proved in 
dimension~$2$ by Blanc (\cite{Blanc:simplicite}) using the classical
Noether and Castelnuovo Theorem.
\end{rem}

Let us mention that
$\mathrm{Bir}(\mathbb{P}^n_\mathbb{C})$, 
$n\geq 2$, is not simple as an abstract group
(for $n=2$ \emph{see} \cite{CantatLamy} or \S
\ref{CantatLamy:passimple}, for $n\geq 3$
\emph{see} \cite{BlancLamyZimmermann}).
Furthermore there is an analogue of Theorem
\ref{thm:BlancZimmermann}
when $\mathrm{Bir}(\mathbb{P}^n_\mathbb{C})$ is 
endowed with the Euclidean topology:

\begin{thm}[\cite{BlancZimmermann}]\label{thm:BZ2}
Let $n\geq 2$ be an integer. The topological group 
$\mathrm{Bir}(\mathbb{P}^n_\mathbb{C})$ is topologically simple when 
endowed with the Euclidean topology.
\end{thm}

The proof of Theorem \ref{thm:BZ2} is similar
to the proof of Theorem \ref{thm:BlancZimmermann}, 
so we will just focus on this last one.

\begin{proof}[Sketch of the proof of Theorem \ref{thm:BlancZimmermann}]

Let us first prove the statement for $n=1$. 

\begin{lem}
Let $n\geq 2$ be an integer. The group 
$\mathrm{PSL}(n,\mathbb{C})$ is dense in 
$\mathrm{PGL}(n,\mathbb{C})$ with respect
to the Zariski topology. 

Moreover every non-trivial normal subgroup of
$\mathrm{PGL}(n,\mathbb{C})$ contains 
$\mathrm{PSL}(n,\mathbb{C})$. In particular
$\mathrm{PGL}(n,\mathbb{C})$ does not contain 
any non-trivial normal strict subgroups closed 
for the Zariski topology.
\end{lem}

\begin{proof}
The group morphism $\det\colon\mathrm{GL}(n,\mathbb{C})\to\mathbb{C}^*$
yields a group morphism
\[
\det\colon\mathrm{PGL}(n,\mathbb{C})\to\faktor{\mathbb{C}^*}{\big\{f^n\,\vert\,f\in\mathbb{C}^*\big\}}
\]
whose kernel is the group 
$\mathrm{PSL}(n,\mathbb{C})$. Denote by $\mathrm{id}$
the identity matrix of size $(n-1)\times(n-1)$ and 
consider the morphism 
\begin{align*}
& \rho\colon\mathbb{A}^1_\mathbb{C}\smallsetminus\{0\}\to\mathrm{PGL}(n,\mathbb{C}), && t\mapsto\left(
\begin{array}{cc}
t & 0\\
0 & \mathrm{id}
\end{array}
\right).
\end{align*}
Note that $\rho^{-1}(\mathrm{PSL}(n,\mathbb{C}))$ contains 
$\big\{t^n\,\vert\,t\in\mathbb{A}^1_\mathbb{C}\big\}$ which is 
an infinite subset of $\mathbb{A}^1_\mathbb{C}$ and is 
thus dense in $\mathbb{A}^1_\mathbb{C}$. Therefore the closure 
of $\mathrm{PSL}(n,\mathbb{C})$ contains 
$\rho(\mathbb{A}^1_\mathbb{C}\smallsetminus\{0\})$. 
Any element of $\mathrm{PGL}(n,\mathbb{C})$ is equal to some
$\rho(t)$ modulo $\mathrm{PSL}(n,\mathbb{C})$ hence 
$\mathrm{PSL}(n,\mathbb{C})$ is dense in $\mathrm{PGL}(n,\mathbb{C})$.

\smallskip

Let $\mathrm{N}$ be a non-trivial normal subgroup of 
$\mathrm{PGL}(n,\mathbb{C})$. Let $f$ be a non-trivial 
element of $\mathrm{N}$. Let us prove that $\mathrm{N}$
contains $\mathrm{PSL}(n,\mathbb{C})$. The center of 
$\mathrm{PGL}(n,\mathbb{C})$ being trivial one can 
replace $f$ by 
$\alpha\circ f\circ\alpha^{-1}\circ f^{-1}$ where 
$\alpha\in\mathrm{PGL}(n,\mathbb{C})$ does not 
commute with $f$, and assume that $f$ belongs to 
$\mathrm{N}\cap\mathrm{PSL}(n,\mathbb{C})$. But 
$\mathrm{PSL}(n,\mathbb{C})$ is a simple group 
(\cite[Chapitre II, \S 2]{Dieudonne}) so 
$\mathrm{PSL}(n,\mathbb{C})\subset\mathrm{N}$. 

\smallskip

The first two points imply that 
$\mathrm{PGL}(n,\mathbb{C})$ does not contain any 
non-trivial normal strict subgroup which is closed 
with respect to the Zariski topology.
\end{proof}

We will now focus on $\mathrm{Bir}(\mathbb{P}^n_\mathbb{C})$
as soon as $n\geq 2$. 

\begin{pro}[\cite{BlancZimmermann}]\label{pro:chemin}
Let $\phi$ be an element of $\mathrm{Bir}(\mathbb{P}^n_\mathbb{C})$. 
Let $p$ be a point of $\mathbb{P}^n_\mathbb{C}$ such that~$\phi$ induces 
a local isomorphism at $p$, and fixes $p$. Then there exist morphisms 
$\nu\colon\mathbb{A}^1_\mathbb{C}\smallsetminus\{0\}\to\mathrm{Aut}(\mathbb{P}^n_\mathbb{C})$ 
and 
$\upsilon\colon\mathbb{A}^1_\mathbb{C}\to\mathrm{Bir}(\mathbb{P}^n_\mathbb{C})$ 
such that:
\begin{itemize}
\item[$\diamond$] $\upsilon(t)=\nu(t)^{-1}\circ\phi\circ \nu(t)$ for any $t\in\mathbb{C}$, 
moreover $\nu(1)=\mathrm{id}$, so $\upsilon(1)=\phi$;

\item[$\diamond$] $\upsilon(0)$ belongs to $\mathrm{Aut}(\mathbb{P}^n_\mathbb{C})$ 
and is the identity if and only if the action of $\phi$ on the tangent space is 
trivial.
\end{itemize} 
\end{pro}

\begin{proof}
Up to conjugacy by an element of $\mathrm{Aut}(\mathbb{P}^n_\mathbb{C})$ we can 
assume that $p=(1:0:0:\ldots:0)$. In the affine chart $z_0=1$ one can write 
$\phi$ locally as 
\begin{small}
\[
\left(\frac{p_{1,1}(z_1,\ldots,z_n)+\ldots+p_{1,\ell}(z_1,\ldots,z_n)}{1+q_{1,1}(z_1,\ldots,z_n)+\ldots+q_{1,\ell}(z_1,\ldots,z_n)},\ldots,\frac{p_{n,1}(z_1,\ldots,z_n)+\ldots+p_{n,\ell}(z_1,\ldots,z_n)}{1+q_{n,1}(z_1,\ldots,z_n)+\ldots+q_{n,\ell}(z_1,\ldots,z_n)}\right)
\]
\end{small}
where $p_{i,j}$, $q_{i,j}$ are homogeneous of degree $j$. For each 
$t\in\mathbb{C}\smallsetminus\{0\}$ the element 
\[
\nu_t\colon(z_1,z_2,\ldots,z_n)\mapsto(tz_1,tz_2,\ldots,tz_n)
\]
extends to a linear automorphism of $\mathbb{P}^n_\mathbb{C}$ that fixes $p$. 
Hence the map $t\mapsto\nu_t^{-1}\circ\phi\circ\nu_t$ gives rise to a morphism
$\Theta\colon\mathbb{A}^1_\mathbb{C}\smallsetminus\{0\}\to\mathrm{Bir}(\mathbb{P}^n_\mathbb{C})$ 
and the image of $\Theta$ contains only conjugates of $\phi$ by linear 
automorphisms. Note that 
\begin{small}
\begin{eqnarray*}
\Theta\colon t\mapsto & &\left(\frac{p_{1,1}(z_1,\ldots,z_n)+tp_{1,2}(z_1,\ldots,z_n)+\ldots+t^{\ell-1}p_{1,\ell}(z_1,\ldots,z_n)}{1+tq_{1,1}(z_1,\ldots,z_n)+t^2q_{1,2}(z_1,\ldots,z_n)+\ldots+t^\ell q_{1,\ell}(z_1,\ldots,z_n)},\right.\\
& &\qquad\left.\ldots,\frac{p_{n,1}(z_1,\ldots,z_n)+tp_{n,2}(z_1,\ldots,z_n)+\ldots+t^{\ell-1}p_{n,\ell}(z_1,\ldots,z_n)}{1+tq_{n,1}(z_1,\ldots,z_n)+t^2q_{n,2}(z_1,\ldots,z_n)+\ldots+t^\ell q_{n,\ell}(z_1,\ldots,z_n)}\right)
\end{eqnarray*}
\end{small}
and $\Theta(0)$ corresponds to the linear part of $\Theta$ at $p$ which is locally 
given by 
\[
\big(p_{1,1}(z_1,\ldots,z_n),\ldots,p_{n,1}(z_1,\ldots,z_n)\big).
\]
As $\phi$ is a local isomorphism at $p$, this linear part is an automorphism
of $\mathbb{P}^n_\mathbb{C}$. Furthermore it is trivial if and only if the 
action of $\phi$ on the tangent space is trivial.
\end{proof}

Let $\phi\in\mathrm{Bir}(\mathbb{P}^n_\mathbb{C})\smallsetminus\{\mathrm{id}\}$; 
it induces an isomorphism from $\mathcal{U}$ to $\mathcal{V}$ where 
$\mathcal{U}$, $\mathcal{V}\subset\mathbb{P}^n_\mathbb{C}$ are two non-empty
open subsets. There exist a point $p$ in $\mathcal{U}$ and two 
automorphisms $\alpha_1$, $\alpha_2$ of $\mathbb{P}^n_\mathbb{C}$ such that
\begin{itemize}
\item[$\diamond$] $\psi=\alpha_1\circ \phi\circ\alpha_2$ fixes $p$,
\item[$\diamond$] $\psi=\alpha_1\circ \phi\circ\alpha_2$ is a local isomorphism at $p$, 
\item[$\diamond$] $D_p\psi$ is not trivial.
\end{itemize}
According to Proposition \ref{pro:chemin} there 
exist morphisms 
$\nu\colon\mathbb{A}^1_\mathbb{C}\smallsetminus\{0\}\to\mathrm{Aut}(\mathbb{P}^n_\mathbb{C})$ and 
$\upsilon_1\colon\mathbb{A}^1_\mathbb{C}\to\mathrm{Bir}(\mathbb{P}^n_\mathbb{C})$ 
such that
\begin{itemize}
\item[$\diamond$] $\upsilon_1(t)=\nu(t)^{-1}\circ \psi^{-1}\circ\nu(t)$ for each $t\not=0$,
\item[$\diamond$] $\upsilon_1(0)$ is an automorphism of $\mathbb{P}^n_\mathbb{C}$.
\end{itemize}
Consider the morphism
$\upsilon_2\colon\mathbb{A}_\mathbb{C}^1\to\mathrm{Bir}(\mathbb{P}^n_\mathbb{C})$
defined by 
\[
\upsilon_2(t)=\alpha_1^{-1}\circ \psi\circ\upsilon_1(t)\circ\upsilon_1(0)^{-1}\circ\alpha_2^{-1}.
\]
Since $\alpha_1$, $\alpha_2$, $\upsilon_1(0)$ and $\nu(t)$ are automorphisms of 
$\mathbb{P}^n_\mathbb{C}$ for all $t\not=0$ 
\[
\upsilon_2(t)=\alpha_1^{-1}\circ\big(\psi\circ\nu(t)^{-1}\circ \psi^{-1}\big)\circ\nu(t)\circ\upsilon_1(0)^{-1}\circ\alpha_2^{-1}
\]
belongs for any $t\not=0$ to the normal subgroup of 
$\mathrm{Bir}(\mathbb{P}^n_\mathbb{C})$ generated by 
$\mathrm{Aut}(\mathbb{P}^n_\mathbb{C})$. As a consequence $\phi=\upsilon_2(0)$ belongs
to the closure of the normal subgroup of $\mathrm{Bir}(\mathbb{P}^n_\mathbb{C})$
generated by~$\mathrm{Aut}(\mathbb{P}^n_\mathbb{C})$.
The normal subgroup of
$\mathrm{Bir}(\mathbb{P}^n_\mathbb{C})$ 
generated by 
$\mathrm{Aut}(\mathbb{P}^n_\mathbb{C})$
is dense in 
$\mathrm{Bir}(\mathbb{P}^n_\mathbb{C})$ 
(\emph{see} \cite{BlancFurter}). 

\smallskip

In particular $\mathrm{Bir}(\mathbb{P}^n_\mathbb{C})$
does not contain any non-trivial closed normal 
strict subgroup. Indeed let 
$\{\mathrm{id}\}\not=\mathrm{N}\subset\mathrm{Bir}(\mathbb{P}^n_\mathbb{C})$
be a closed normal subgroup with respect to the Zariski topology.
Then $\mathrm{Aut}(\mathbb{P}^n_\mathbb{C})\subset\mathrm{N}$
(\emph{see} \cite[Prop. 3.3, Lemma 3.4]{BlancZimmermann}). 
Since~$\mathrm{N}$ is closed it contains the closure of the 
normal subgroup generated by 
$\mathrm{Aut}(\mathbb{P}^n_\mathbb{C})$ which is equal to
$\mathrm{Bir}(\mathbb{P}^n_\mathbb{C})$.
\end{proof}

Furthermore, one has:

\begin{thm}[\cite{BlancZimmermann}]
If $n\geq 1$, the group 
$\mathrm{Bir}(\mathbb{P}^n_\mathbb{C})$ 
is connected with respect 
to the Zariski topology.

If $n\geq 2$, the group 
$\mathrm{Bir}(\mathbb{P}^n_\mathbb{C})$
is path-connected, and thus connected with respect 
to the Euclidean topology.
\end{thm}

Let us give an idea of the proof of 
this statement. We start with an 
example.

\begin{eg}\label{eg:utile}
Let $n\geq 2$ and let $\alpha$ be
an element of $\mathbb{C}^*$. Consider 
the birational self map 
of~$\mathbb{P}^n_\mathbb{C}$ given by 
\[
\Phi\colon(z_0:z_1:\ldots:z_n) \dashrightarrow \left(\frac{z_0(z_1+\alpha z_2)+z_1z_2}{z_1+z_2}:z_1:z_2:\ldots:z_n\right).
\]
The points $p=(0:1:0:0:\ldots:0)$ and
$q=(0:0:1:0:0:\ldots:0)$ are fixed
by $\Phi$. Applying Proposition 
\ref{pro:chemin} to the points $p$ and 
$q$ we get two morphisms $\Theta_1$, 
$\Theta_2\colon\mathbb{A}^1_\mathbb{C}\to\mathrm{Bir}(\mathbb{P}^n_\mathbb{C})$ 
such that
\begin{itemize}
\item[$\diamond$] $\Theta_1(0)\colon(z_0:z_1:\ldots:z_n)\mapsto(z_0+z_2:z_1:z_2:\ldots:z_n)\in\mathrm{Aut}(\mathbb{P}^n_\mathbb{C})$,

\item[$\diamond$] $\Theta_2(0)\colon(z_0:z_1:\ldots:z_n)\mapsto(\alpha z_0+z_1:z_1:z_2:\ldots:z_n)\in\mathrm{Aut}(\mathbb{P}^n_\mathbb{C})$,

\item[$\diamond$] $\Theta_1(1)=\Theta_2(1)=\Phi$.
\end{itemize}
\end{eg}

\begin{pro}
Let $n\geq 2$ be an integer. For any $\phi$, 
$\psi\in\mathrm{Bir}(\mathbb{P}^n_\mathbb{C})$ there is a morphism 
$\upsilon\colon\mathbb{P}^1_\mathbb{C}\to\mathrm{Bir}(\mathbb{P}^n_\mathbb{C})$
such that $\upsilon(0)=\phi$ and $\upsilon(1)=\psi$. 
\end{pro}

\begin{proof}
Up to composition with $\phi^{-1}$
one can assume that $\phi=\mathrm{id}$.
Let us consider the subset $S$ of 
$\mathrm{Bir}(\mathbb{P}^n_\mathbb{C})$
given by 
\[
S=\left\{\varphi\in\mathrm{Bir}(\mathbb{P}^n_\mathbb{C})\,\vert\,\exists\,\, \nu\colon\mathbb{A}^1_\mathbb{C}\to\mathrm{Bir}(\mathbb{P}^n_\mathbb{C})\text{ morphism such that }\nu(0)=\mathrm{id}\text{ and }\nu(1)=\varphi \right\}.
\]
Let $\phi$ (resp. $\psi$) be an 
element of $S$; denote by 
$\nu_\phi$ (resp. $\nu_\psi$) the
associated morphism. We define
a morphism 
$\nu_{\phi\circ\psi}\colon\mathbb{A}^1_\mathbb{C}\to\mathrm{Bir}(\mathbb{P}^n_\mathbb{C})$
by 
$\nu_{\phi\circ\psi}(t)=\nu_\phi(t)\circ\nu_\psi(t)$ 
which satisfies
$\nu_{\phi\circ\psi}(0)=~\mathrm{id}$
and $\nu_{\phi\circ\psi}(1)=\phi\circ\psi$.
For any 
$\varphi\in\mathrm{Bir}(\mathbb{P}^n_\mathbb{C})$ 
it is also possible to define a 
morphism 
$t\mapsto\varphi\circ\nu_\phi(t)\circ\varphi^{-1}$. 
Therefore, $S$ is a normal subgroup
of~$\mathrm{Bir}(\mathbb{P}^n_\mathbb{C})$.

\begin{claim}[\cite{BlancFurter}] 
The group $S$ contains
$\mathrm{PSL}(n+1,\mathbb{C})$.
\end{claim}

Take $\alpha$, $\Phi$, $\Theta_1$ and 
$\Theta_2$ as in Example \ref{eg:utile};
for $i\in\{1,\,2\}$ the morphisms
\[
t\mapsto\Theta_i(t)\circ(\Theta_i(0))^{-1}
\]
show that 
$g\circ\big(\Theta_1(0)\big)^{-1}$
and 
$g\circ\big(\Theta_2(0)\big)^{-1}$
belong to $S$, hence 
$\Theta_1(0)\circ\big(\Theta_2(0)\big)^{-1}$
belong to $S$. But $\Theta_1(0)$ 
belongs to 
$\mathrm{PSL}(n+1,\mathbb{C})\subset S$,
so $\Theta_2(0)$ belongs to $S$. 
Thus~$\mathrm{Aut}(\mathbb{P}^n_\mathbb{C})=\mathrm{PGL}(n+1,\mathbb{C})$
is contained in $S$.

\smallskip

Take 
$\phi\in\mathrm{Bir}(\mathbb{P}^n_\mathbb{C})$
of degree $d\geq 2$. Let $p$ be 
a point of $\mathbb{P}^n_\mathbb{C}$
such that $\phi$ induces a local 
isomorphism at $p$. Consider 
an element $A$ of 
$\mathrm{PSL}(n+1,\mathbb{C})$ such
that~$A\circ\phi$ fixes~$p$. There exists a morphism
$\theta\colon\mathbb{A}^1_\mathbb{C}\to\mathrm{Bir}(\mathbb{P}^n_\mathbb{C})$ 
such that $\theta(0)$ belongs 
to~$\mathrm{Aut}(\mathbb{P}^n_\mathbb{C})$ 
and $\theta(1)=A\circ\phi$. 
Let us define 
$\theta'\colon\mathbb{A}^1_\mathbb{C}\to\mathrm{Bir}(\mathbb{P}^n_\mathbb{C})$
by $\theta'(t)=\rho(t)\circ\theta(0)^{-1}$. 
Then 
$\theta'(1)=A\circ\phi\circ\theta(0)^{-1}$.
But $A$ and $\theta(0)$
belong to 
$\mathrm{Aut}(\mathbb{P}^n_\mathbb{C})\subset S$, so $\phi$ belongs to~$S$. 
\end{proof}


\section{Regularization of rational group actions}\label{sec:WeilKraft}

The aim of \cite{Kraft:regularization} is to give
a modern proof of the regularization theorem of
Weil which says:

\begin{thm}[\cite{Weil}]\label{thm:Weil}
For every rational action 
of an algebraic group $\mathrm{G}$ on a variety
$X$ there exist a variety~$Y$ with a regular 
action of $\mathrm{G}$ and a $\mathrm{G}$-equivariant
birational map $X\dashrightarrow Y$.
\end{thm}

In this section a variety is an algebraic
complex variety, and an algebraic group is an
algebraic $\mathbb{C}$-group. 

A rational map $\phi\colon X\dashrightarrow Y$
is called 
\textsl{biregular}\index{defi}{biregular (rational map)} 
in $p\in X$ if there is an open neighborhood
$\mathcal{U}\subset (X\smallsetminus\mathrm{Base}(\phi))$
of $p$ such that 
$\phi_{\vert\mathcal{U}}\colon\mathcal{U} \hookrightarrow Y$ 
is an open immersion. As a result the subset 
\[
X'=\big\{p\in X\,\vert\,\phi\text{ is biregular in }p\big\}
\] 
is open in $X$, and the induced morphism 
$\phi\colon X'\hookrightarrow Y$ is an 
open immersion. We can thus state:

\begin{lem}\label{lem:Kraft1}
Let $X$ and $Y$ be two varieties.
Let $\phi\colon X\dashrightarrow Y$ be a birational
map. Then the set
\[
\mathrm{Breg}(\phi)=\big\{p\in X\,\vert\,\phi\text{ is biregular in }p\big\}
\]\index{not}{$\mathrm{Breg}(\phi)$}
is open and dense in $X$.
\end{lem}

\subsection{Rational group actions}

Let $X$ and $Z$ be two varieties. Let us recall
that a map $\phi\colon Z\to\mathrm{Bir}(X)$ 
is a morphism if there exists an open dense 
set $\mathcal{U}\subset Z\times X$ such that
\begin{itemize}
\item[$\diamond$] the induced map 
$\mathcal{U}\to X$, $(q,p)\mapsto\phi(q)(p)$
is a morphism of varieties;

\item[$\diamond$] for every $q\in Z$ the open
set 
$\mathcal{U}_q=\big\{p\in X\,\vert\,(q,p)\in\mathcal{U}\big\}$ is dense in $X$;

\item[$\diamond$] for every $q\in Z$ the 
birational map $\phi(q)\colon X\dashrightarrow X$
is defined on $\mathcal{U}_q$.
\end{itemize}
Equivalently there is a rational map 
$\phi\colon Z\times X\to X$ such that for 
every $q\in Z$
\begin{itemize}
\item[$\diamond$] the open subset 
$(Z\times X\smallsetminus\mathrm{Base}(\phi))\cap(\{q\}\times X)$ 
is dense in $\{q\}\times X$;

\item[$\diamond$] the induced birational 
map $\phi_q\colon X\dashrightarrow X$, 
$p\mapsto\phi(q,p)$ is birational.
\end{itemize}

Recall that this definition allows to define
the Zariski topology on $\mathrm{Bir}(X)$
(\emph{see}~\S\ref{sec:alg}).

We can now define rational group actions on 
varieties. Let $\mathrm{G}$ be an algebraic 
group, and let~$X$ be a variety. A 
\textsl{rational action}\index{defi}{rational action of a group on a variety} 
of $\mathrm{G}$ on $X$ is a morphism 
$\rho\colon\mathrm{G}\to\mathrm{Bir}(X)$ 
which is a morphism of groups. In 
other words there is a rational map still
denoted $\rho$
\[
\rho\colon\mathrm{G}\times X\dashrightarrow X
\]
such that
\begin{itemize}
\item[$\diamond$] the open set 
$((\mathrm{G}\times X)\smallsetminus\mathrm{Base}(\rho))\cap(\{g\}\times X)$ 
is dense in $\{g\}\times X$ for every 
$g\in\mathrm{G}$;

\item[$\diamond$] the induced map 
$\rho_g\colon X\dashrightarrow X$, 
$p\mapsto\rho(g,p)$ is birational
for every $g\in\mathrm{G}$;

\item[$\diamond$] the map $g\mapsto\rho_g$
is a group morphism.
\end{itemize}

\begin{thm}[\cite{Kraft:regularization}]\label{thm:Kraft2}
Let $\rho\colon\mathrm{G}\to\mathrm{Bir}(X)$ be 
a rational action where $X$ is affine. Assume 
that there exists a dense subgroup 
$\Gamma\subset\mathrm{G}$ such that 
$\rho(\Gamma)\subset\mathrm{Aut}(X)$. 
Then the $\mathrm{G}$-action on $X$ is 
regular.
\end{thm}

\begin{defi}
Let $X$ and $Y$ be two varieties.
Let $\rho$ be a rational $\mathrm{G}$-action
on~$X$. Let $\mu$ be a rational $\mathrm{G}$-action
on $Y$.

A dominant rational map 
$\phi\colon X\dashrightarrow Y$ is 
\textsl{$\mathrm{G}$-equivariant}\index{defi}{$\mathrm{G}$-equivariant (rational map)}
if the following holds: for every 
$(g,p)\in\mathrm{G}\times X$ such 
that 
\begin{itemize}
\item[$\diamond$] $\rho$ is defined in $(g,p)$,

\item[$\diamond$] $\phi$ is defined in $p$ and 
in $\rho(g,p)$,

\item[$\diamond$] $\mu$ is defined in 
$(g,\phi(p))$,
\end{itemize}
we have $\phi(\rho(g,p))=\mu(g,\phi(p))$.
\end{defi}

\begin{rem}
The set of $(g,p)\in\mathrm{G}\times X$ 
satisfying the previous assumptions is
open and dense in $\mathrm{G}\times X$ 
and has the property that it meets 
all $\{g\}\times X$ in a dense open 
set.
\end{rem}

Let $X$ be a variety with a rational action 
$\rho\colon \mathrm{G}\times X\dashrightarrow X$ of 
an algebraic group~$\mathrm{G}$. Consider 
\begin{align*}
&\widetilde{\rho}\colon\mathrm{G}\times X\dashrightarrow\mathrm{G}\times X,&& (g,p)\mapsto (g,\rho(g,p)).
\end{align*}
It is clear that
$(\mathrm{G}\times X)\smallsetminus\mathrm{Base}(\widetilde{\rho})=(\mathrm{G}\times X)\smallsetminus\mathrm{Base}(\rho)$. 
Furthermore $\widetilde{\rho}$ is birational with 
inverse $\widetilde{\rho}^{-1}(g,p)=(g,\rho(g^{-1},p))$, 
that is 
\[
\widetilde{\rho}^{-1}=\tau\circ\widetilde{\rho}\circ\tau
\]
where $\tau$ is the isomorphism
\begin{align*}
& \tau\colon\mathrm{G}\times X\to\mathrm{G}\times X, && (g,p)\mapsto(g^{-1},p).
\end{align*}

\begin{defi}
A point $x\in X$ is called 
\textsl{$\mathrm{G}$-regular}\index{defi}{$\mathrm{G}$-regular (point)} 
for the rational $\mathrm{G}$-action $\rho$ on~$X$ if 
$\mathrm{Breg}(\widetilde{\rho})\cap(\mathrm{G}\times\{p\})$
 is dense in $\mathrm{G}\times\{p\}$. 

In other words a point $p\in X$ is called $\mathrm{G}$-regular 
for the rational $\mathrm{G}$-action $\rho$ on~$X$ if 
$\widetilde{\rho}$ is biregular in $(g,p)$ for all $g$ in 
a dense open set of $\mathrm{G}$.
\end{defi}

Denote by 
$X_{\text{reg}}\subset X$\index{not}{$X_{\text{reg}}$} 
the set of $\mathrm{G}$-regular points.

Let 
$\lambda_g\colon \mathrm{G}\stackrel{\sim}{\longrightarrow}\mathrm{G}$ 
be the left multiplication with $g\in\mathrm{G}$. 
For any $h\in\mathrm{G}$ the diagram
\[
  \xymatrix{
    \mathrm{G}\times X \ar@{-->}[r]^{\widetilde{\rho}} \ar[d]_{\lambda_h\times\mathrm{id}} & \mathrm{G}\times X \ar[d]^{\lambda_h\times\rho_h} \\
    \mathrm{G}\times X \ar@{-->}[r]_{\widetilde{\rho}} & \mathrm{G}\times X
  }
\]
commutes. This implies the following statement:

\begin{lem}[\cite{Kraft:regularization}]\label{lem:Kraft2}
If $\rho$ is defined in $(g,p)$ and if $\rho_h$ is defined in 
$g\cdot p$, then~$\rho$ is defined in $(hg,p)$.

If $\widetilde{\rho}$ is biregular in $(g,p)$ and if $\rho_h$ 
is biregular in $g\cdot p$, then $\widetilde{\rho}$ is 
biregular in~$(hg,p)$.
\end{lem}

\begin{pro}[\cite{Kraft:regularization}]\label{pro:Kraft3}
The set $X_{\text{reg}}$ of $\mathrm{G}$-regular points is 
open and dense in $X$.

If $p$ belongs to $X_{\text{reg}}$ and if $\widetilde{\rho}$ 
is biregular in $(g,p)$, then $g\cdot p$ belongs to 
$X_{\text{reg}}$. 
\end{pro}

\begin{proof}
Let 
$\mathrm{G}=\mathrm{G}_0\cup\mathrm{G}_1\cup\ldots\cup\mathrm{G}_n$ 
be the decomposition into connected components. Then 
$D_i=\mathrm{Breg}(\rho)\cap(\mathrm{G}_i\times X)$ 
is open and dense for all $i$ (Lemma~\ref{lem:Kraft1}); the 
same holds for the image $\mathcal{D}_i\subseteq X$ under 
the projection onto $X$. Since 
$X_{\text{reg}}=~\displaystyle\bigcap_i\mathcal{D}_i$ the 
set $X_{\text{reg}}$ is open and dense in $X$.

If $\widetilde{\rho}$ is biregular in $(g,p)$, then 
$\widetilde{\rho}^{-1}=\tau\circ\widetilde{\rho}\circ\tau$ 
is biregular in $(g,g\cdot p)$. As a consequence $\widetilde{\rho}$ 
is biregular in $\tau(g,g\cdot p)=(g^{-1},g\cdot p)$. If $p$ is 
$\mathrm{G}$-regular, then $\rho_h$ is biregular in $p$ for 
all $h$ in a dense open subset $\mathrm{G}'$ of $\mathrm{G}$. 
According to the second assertion of Lemma \ref{lem:Kraft2} 
the birational map $\widetilde{\rho}$ is biregular in 
$(hg^{-1},g\cdot p)$ for all $h\in \mathrm{G}'$. Hence $g\cdot p$ belongs 
to $X_{\text{reg}}$.
\end{proof}

A consequence of Proposition 
\ref{pro:Kraft3} allows us to only consider the case of a 
rational $\mathrm{G}$-action such every point is 
$\mathrm{G}$-regular.

\begin{cor}[\cite{Kraft:regularization}]
For the rational $\mathrm{G}$-action on $X_{\text{reg}}$ 
every point is $\mathrm{G}$-regular.
\end{cor}

\begin{lem}[\cite{Kraft:regularization}]\label{lem:Kraft3}
Assume that $X=X_{\text{reg}}$. If $\rho_g$ is defined 
in $p$, {\it i.e.} if 
$p\in X\smallsetminus\mathrm{Base}(\rho_g)$, then $\rho_g$ 
is biregular in $p$.
\end{lem}

\begin{proof}
Suppose that $\rho_g$ is defined in $p\in X$. As 
$X=X_{\text{reg}}$ there exists a dense open 
subset $\mathrm{G}'$ of $\mathrm{G}$ such 
that for all $h\in\mathrm{G}'$ 
\begin{itemize}
\item[$\diamond$] $\rho_h$ is biregular in $g\cdot p$,

\item[$\diamond$] $\rho_{hg}$ is biregular in $p$.
\end{itemize}
Since $\rho_{hg}=\rho_h\circ\rho_g$ the map 
$\rho_g$ is biregular in $p$. 
\end{proof}

Let us recall that if $\phi\colon X\dashrightarrow Y$ is a rational
map, its graph $\Gamma(\phi)$ is defined by 
\[
\Gamma(\phi)=\big\{(x,y)\in X\times Y\,\vert\,x\in X\smallsetminus\mathrm{Base}(\phi)\text{ and } \phi(x)=y\big\}.
\]
In particular 
$\mathrm{pr}_1(\Gamma(\phi))=X\smallsetminus\mathrm{Base}(\phi)$
and 
$\mathrm{pr}_2(\Gamma(\phi))=\phi(X\smallsetminus\mathrm{Base}(\phi))$.

\begin{lem}[\cite{Kraft:regularization}]
Let $\rho$ be a rational $\mathrm{G}$-action on a variety $X$.
Suppose that every point of $X$ is $\mathrm{G}$-regular, that
is $X=X_{\text{reg}}$. Then for every $g\in\mathrm{G}$ 
the graph $\Gamma(\rho_g)$ of $\rho_g$ is closed in $X\times X$. 
\end{lem}

\begin{proof}
Denote by $\Gamma$ the closure 
$\overline{\Gamma(\rho_g)}$ of the graph 
of $\rho_g$ in $X\times X$. Let us prove 
that for any $(x_0,y_0)\in\Gamma$
the rational map $\rho_g$ is defined in $x_0$. 
It is equivalent to prove that the morphism 
$\mathrm{pr}_{1\vert\Gamma}\colon\Gamma\to X$ is 
biregular in $(x_0,y_0)$. 

Let $h$ be an element of $\mathrm{G}$ such 
that $\rho_{hg}$ is biregular in $x_0$ and 
$\rho_h$ is biregular in $y_0$. Consider the
birational map
\[
\phi=(\rho_{hg},\rho_h)\colon X\times X\dashrightarrow X\times X.
\]
If $\phi$ is defined in $(x,y)\in\Gamma(\rho_g)$, 
$y=g\cdot x$, then $\phi(x,y)=((hg)\cdot x,(hg)\cdot x)$ belongs 
to the diagonal
\[
\Delta(X)=\big\{(x,x)\in X\,\vert\,x\in X\big\}
\]
of $X\times X$. It follows that 
$\overline{\phi(\Gamma)}\subseteq\Delta(X)$. 
Since $\phi$ is biregular in $(x_0,y_0)$, 
the restriction 
$\varphi=\phi_{\vert\Gamma}\colon\Gamma\dashrightarrow\Delta(X)$ 
of $\phi$ to $\Gamma$ is also biregular in 
$(x_0,y_0)$. By construction 
\[
\rho_{hg}\circ\mathrm{pr}_{1\vert\Gamma}=\mathrm{pr}_{1\vert\Delta(X)}\circ\varphi;
\]
indeed
\[
  \xymatrix{
    X\times X \ar@{-->}[r]^{\rho_{hg}\times\rho_h}  & X\times X \\
    \Gamma \ar@{-->}[r]^{\varphi} \ar[d]_{\mathrm{pr}_{1\vert\Gamma}}\ar@{^{(}->}[u] & \Delta(X) \ar[d]^{\mathrm{pr}_{1\vert\Delta(X)}}\ar@{^{(}->}[u] \\
    X \ar@{-->}[r]^{\rho_{hg}}  & X 
  }
\]
But $\rho_{hg}$ is biregular in 
$\mathrm{pr}_{1\vert\Gamma}(x_0,x_0)$, $\varphi$
is biregular in $(x_0,y_0)$ and 
$\mathrm{pr}_{1\vert\Delta(X)}$ is an isomorphism, 
so $\mathrm{pr}_{1\vert\Gamma}$ is biregular 
in $(x_0,y_0)$.
\end{proof}

\begin{lem}[\cite{Kraft:regularization}]\label{lem:Kraft5}
Let $\rho$ be a rational action of $\mathrm{G}$ 
on a variety $X$. Suppose that there is a dense 
open subset $\mathcal{U}$ of $X$ such that
\begin{align*}
& \widetilde{\rho}\colon\mathrm{G}\times\mathcal{U}\to\mathrm{G}\times X, && (g,p)\mapsto\big(g,\rho(g,p)\big)
\end{align*}
defines an open immersion.
Then the open dense subset 
$Y=\displaystyle\bigcup_{g\in\mathrm{G}}g\cdot \mathcal{U}\subseteq X$ 
carries a regular $\mathrm{G}$-action.
\end{lem}

\begin{proof}
Any $\rho_g$ induces an isomorphism 
$\mathcal{U}\stackrel{\sim}{\longrightarrow} g\cdot\mathcal{U}$. Therefore,
\[
Y=\displaystyle\bigcup_{g\in\mathrm{G}} g\cdot X\subset~X
\]
is stable under all
$\rho_g$. By assumption the induced map on 
$\mathrm{G}\times\mathcal{U}$ is a morphism, so 
the induced map on $\mathrm{G}\times g\cdot \mathcal{U}$
is a morphism for all $g\in\mathrm{G}$. As a result
the induced map $\mathrm{G}\times Y\to Y$ is a 
morphism.
\end{proof}

\subsection{Construction of a regular model}

\begin{thm}[\cite{Kraft:regularization}]\label{thm:Kraft3}
Let $X$ be a variety with a rational action of $\mathrm{G}$.
Suppose that every point of $X$ is $\mathrm{G}$-regular.
Then there exists a variety $Y$ with a regular $\mathrm{G}$-action
and a $\mathrm{G}$-equivariant open immersion.
\end{thm}

Assume now that $X$ is a variety with a rational $\mathrm{G}$-action
$\rho$ such that $X_{\text{reg}}=X$. Consider a finite subset
$S=\big\{g_0=e,\,g_1,\,g_2,\,\ldots,\,g_m\big\}$ of $\mathrm{G}$. Denote
by $X^{(0)}$, $X^{(1)}$, $\ldots$, $X^{(m)}$ some copies of $X$.
Consider the disjoint union
\[
X(S)=X^{(0)}\cup X^{(1)}\cup\ldots\cup X^{(m)}.
\]
Let us define on $X^{(i)}$ the following relations
\[
\left\{
\begin{array}{ll}
\forall\,i\quad p_i\sim p'_i \Longleftrightarrow p_i=p'_i\\
\forall\,i,\,j\quad i\not=j\quad p_i\sim p_j\Longleftrightarrow \rho_{g_j^{-1}g_i} \text{ is defined in $p_i$ and sends $p_i$ to $p_j$}
\end{array}
\right.
\]
This defines an equivalence relation (Lemma \ref{lem:Kraft3} 
is needed to prove the symmetry).
Consider $\widetilde{X}(S)=\faktor{X(S)}{\sim}$ the set of equivalence 
classes endowed with the induced topology.

\begin{lem}[\cite{Kraft:regularization}]
The maps $\iota_i\colon X^{(i)}\to\widetilde{X}(S)$ are 
open immersions and endow $\widetilde{X}(S)$ with the 
structure of a variety.
\end{lem}

Let us fix the open immersion 
$\iota_0\colon X=X^{(0)}\hookrightarrow\widetilde{X}(S)$. 
Then $\mathrm{G}$ acts rationally on $\widetilde{X}(S)$ 
via $\overline{\rho}=\overline{\rho}_S$ such that $\iota_0$ 
is $\mathrm{G}$-equivariant. Consider any $X^{(i)}$ as 
the variety~$X$ with the rational $\mathrm{G}$-action 
\[
\rho^{(i)}(g,p)=\rho(g_igg_i^{-1},p);
\]
by construction of $\widetilde{X}(S)$ the open immersions
\[
\iota_i\colon X^{(i)}\hookrightarrow\widetilde{X}(S)
\]
are all $\mathrm{G}$-equivariant.

\begin{lem}[\cite{Kraft:regularization}]\label{lem:Kraft7}
Let $\widetilde{X}^{(i)}$ be the image of the 
open immersion 
$\iota_i\colon X^{(i)}\hookrightarrow\widetilde{X}(S)$.
For all $i$ the rational map $\overline{\rho}_{g_i}$ is
defined on $\widetilde{X}^{(0)}$.

Furthermore 
$\overline{\rho}_{g_i}\colon\widetilde{X}^{(0)}\stackrel{\sim}{\to}\widetilde{X}^{(i)}$ defines an isomorphism.
\end{lem}

\begin{proof}
Consider the open immersion
\[
\tau_i=\iota_i\circ\iota_0^{-1}\colon\widetilde{X}^{(0)}\hookrightarrow\widetilde{X}(S).
\]
Note that $\mathrm{im}\,\tau_i=\widetilde{X}^{(i)}$.
Let us check that $\tau_i(\overline{p})=g_i\overline{p}$. 
It is sufficient to show that it holds on an open 
dense subset of $\widetilde{X}^{(0)}$. Let 
$\mathcal{U}\subseteq X$ be the open dense set 
where $g_i\cdot p$ is defined. Take $p$ in~$\mathcal{U}$.
On the one hand by definition
\[
\iota_0(g_i\cdot p)=\iota_i(p);
\]
on the other hand
\[
\iota_0(g_i\cdot p)=g_i\cdot\iota_0(p).
\]
As a result $g_i\cdot \iota_0(p)=\iota_i(p)$ and 
\[
\tau_i(\overline{p})=\iota_i(\iota_0^{-1}(\overline{p}))=g_i\cdot \iota_0(\iota_0^{-1}(\overline{p}))=g_i\cdot \overline{p}
\]
for any $\overline{p}\in\iota_0(\mathcal{U})$.
\end{proof}

\begin{proof}[Proof of Theorem \ref{thm:Kraft3}]
Set $D=\mathrm{Breg}(\rho)\cap(\mathrm{G}\times X)$. 
Since $X_{\text{reg}}=X$ for any $p\in X$ there is
an element $g$ in $\mathrm{G}$ such that $(g,p)\in D$.
As a consequence 
$\displaystyle\bigcup_{g\in\mathrm{G}}g\cdot D=\mathrm{G}\times X$
where $\mathrm{G}$ acts on $\mathrm{G}\times X$ by 
left-multiplication on $\mathrm{G}$. Hence 
$\displaystyle\bigcup_ig_iD=\mathrm{G}\times X$
for a suitable finite subset 
\[
S=\big\{g_0=e,\,g_1,\,g_2,\,\ldots,\,g_m\big\}.
\]
Recall that $\widetilde{X}^{(0)}=\mathrm{im}(\iota_0)$.
Let $D^{(0)}\subset\mathrm{G}\times\widetilde{X}^{(0)}$
be the image of $D$. 
Consider the rational map
\begin{align*}
&\widetilde{\rho}_S\colon\mathrm{G}\times\widetilde{X}^{(0)}\dashrightarrow\mathrm{G}\times\widetilde{X}(S).
\end{align*}
The map $(g,p)\mapsto(g,g\cdot p)$ is the composition
of 
$(g,p)\mapsto(g,(g_i^{-1}g)\cdot p)$ 
and 
$(g,y)\mapsto(g,g_i\cdot y)$. The first one is biregular on 
$g_i\cdot D^{(0)}$ and its image is contained in 
$\mathrm{G}\times\widetilde{X}^{(0)}$; the second 
is biregular on $\mathrm{G}\times\widetilde{X}^{(0)}$
(Lemma \ref{lem:Kraft7}). As 
$\mathrm{G}\times\widetilde{X}^{(0)}=\displaystyle\bigcup_ig_i\cdot D^{(0)}$
the map $\widetilde{\rho}_S$ is biregular. As a 
consequence the rational action $\overline{\rho}$
of $\mathrm{G}$ on $\widetilde{X}(S)$ has the 
property that 
\[
\widetilde{\rho}_S\colon\mathrm{G}\times\widetilde{X}^{(0)}\hookrightarrow\mathrm{G}\times\widetilde{X}(S)
\]
defines an open immersion. Lemma \ref{lem:Kraft5}
allows to conclude.
\end{proof}

\subsection{Proof of Theorem \ref{thm:Kraft2}}

Let us start with the following statement:

\begin{lem}[\cite{Kraft:regularization}]\label{lem:Kraft8}
Let $X$, $Y$, $Z$ be varieties. Assume that $Z$ is affine. 
Let $\phi\colon X\times Y\dashrightarrow Z$ be a rational
map. Suppose that 
\begin{itemize} 
\item[$\diamond$] there exists an open dense subset $\mathcal{U}$
of $Y$ such that $\phi$ is defined on $X\times\mathcal{U}$;

\item[$\diamond$] there exists a dense subset $X'$ of $X$ such that
the induced maps $\phi_p\colon\{p\}\times Y\to Z$ are 
morphisms for all $p\in X'$.
\end{itemize}
Then $\phi$ is a regular morphism.
\end{lem}

Consider a rational action 
$\rho\colon\mathrm{G}\to\mathrm{Bir}(X)$ of an algebraic
group on a variety~$X$. Assume that there is a dense 
subgroup $\Gamma$ of $\mathrm{G}$ such that 
$\rho(\Gamma)\subset\mathrm{Aut}(X)$.
\begin{itemize}
\item[$\diamond$] Let us first prove that the rational 
$\mathrm{G}$-action on the open dense set 
$X_{\text{reg}}\subseteq~X$ is regular. For every 
$p\in X_{\text{reg}}$ there is $g\in\Gamma$ such 
that $\widetilde{\rho}$ is biregular in $(g,p)$. By 
assumption for any $h\in\Gamma$ the map 
$\rho_h$ is biregular on $X$, hence the map $\widetilde{\rho}$ is 
biregular in $(h,p)$ for any $h\in\Gamma$ (Lemma 
\ref{lem:Kraft2}).
Furthermore $h\cdot p$ belongs to $X_{\text{reg}}$
(Proposition \ref{pro:Kraft3}), {\it i.e.}
$X_{\text{reg}}$ is stable under $\Gamma$.
According to Theorem \ref{thm:Kraft3} there 
exists a $\mathrm{G}$-equivariant open immersion
\[
X_{\text{reg}}\hookrightarrow Y
\]
where $Y$ is a variety with a regular $\mathrm{G}$-action.
The complement $Y\smallsetminus X_{\text{reg}}$ is 
closed and $\Gamma$-stable, so 
$Y\smallsetminus X_{\text{reg}}$ is stable under 
$\overline{\Gamma}=\mathrm{G}$.

\item[$\diamond$] From the previous point the rational map
\[
\rho\colon\mathrm{G}\times X\dashrightarrow X
\]
has the following properties:
\begin{itemize}
\item[- ] there is a dense open set $X_{\text{reg}}\subseteq X$
such that $\rho$ is regular on 
$\mathrm{G}\times X_{\text{reg}}$;

\item[- ] for every $g\in\Gamma$ the rational map
\begin{align*}
&\rho_g\colon X\to X, && p\mapsto \rho(g,p)
\end{align*}
is a regular isomorphism.
\end{itemize}
Lemma \ref{lem:Kraft8} implies that $\rho$ is a regular
action in case $X$ is affine.
\end{itemize}

\begin{rem}
All the statements of this section hold for an algebraically
closed field.
\end{rem}






\chapter{Generators and relations of the Cremona
group}\label{chapter:gen}

\bigskip
\bigskip

As we already say 

\begin{thm}[\cite{Castelnuovo}]\label{thm:noether}
The group $\mathrm{Bir}(\mathbb{P}^2_\mathbb{C})$ is 
generated by $\mathrm{Aut}(\mathbb{P}^2_\mathbb{C})=\mathrm{PGL}(3,\mathbb{C})$
and the standard quadratic involution 
\[
\sigma_2\colon(z_0:z_1:z_2)\dashrightarrow(z_1z_2:z_0z_2:z_0z_1).
\]
\end{thm}

This result is well-known as the Theorem of Noether and Castelnuovo.
Noether was the first mathematician
to state this result at the end of the 
XIXth century. Nevertheless the first 
exact proof is due to Castelnuovo. 
Noether's idea was the following.
Let us consider a birational self map 
$\phi$ of $\mathbb{P}^2_\mathbb{C}$. Take
a quadratic birational self map $q$ of 
$\mathbb{P}^2_\mathbb{C}$ such that 
the three base-points of $q$ are three
base-point of $\phi$ of highest 
multiplicity. Then $\deg(\phi\circ q)<\deg\phi$.
By induction one gets a birational map
of degree $1$. But such a quadratic 
birational map $q$ may not exist. This is for
instance the case if one starts with 
the polynomial automorphism
\[
(z_0:z_1:z_2)\dashrightarrow(z_1^3-z_0z_2^2:z_1z_2^2:z_2^3).
\]
In \cite{Alexander} Alexander
fixes Noether's proof by 
introducing the notion of complexity
of a map: start with a birational 
self map $\phi$ of the complex
projective plane; one can find 
a quadratic birational self map
$q$ of the complex projective 
plane such that
\begin{itemize}
\item[$\diamond$] either the complexity
of $\phi\circ q$ is strictly less 
that the complexity of $\phi$; 

\item[$\diamond$] or the complexities
of $\phi\circ q$ and $\phi$ are 
equal but 
$\#\mathrm{Base}(\phi\circ q)<\#\mathrm{Base}(\phi)$.
\end{itemize}
Alexander's proof is a proof
by induction on these two integers.

\begin{rem}
One consequence of Noether
and Castelnuovo theorem 
is: the Jonqui\`eres group
and 
$\mathrm{Aut}(\mathbb{P}^2_\mathbb{C})=\mathrm{PGL}(3,\mathbb{C})$ 
generate 
$\mathrm{Bir}(\mathbb{P}^2_\mathbb{C})$.
This result is "weaker" nevertheless it 
has the following nice property:
\end{rem}

\begin{thm}[\cite{AlberichCarraminana}]
Let $\phi$ be an element of 
$\mathrm{Bir}(\mathbb{P}^2_\mathbb{C})$. 
There exist $j_1$, $j_2$, $\ldots$, $j_k$
in $\mathcal{J}$ and~$A$ in 
$\mathrm{PGL}(3,\mathbb{C})$ such that 
\begin{itemize}
\item[$\diamond$] $\phi=A\circ j_k\circ j_{k-1}\circ\ldots\circ j_2\circ j_1$;

\item[$\diamond$] for any $1\leq i\leq n-1$
\[
\deg(A\circ j_k\circ j_{k-1}\circ\ldots\circ j_{i+1}\circ j_i)>\deg(A\circ j_k\circ j_{k-1}\circ\ldots\circ j_{i+2}\circ j_{i+1}).
\]
\end{itemize}
\end{thm}

The first presentation of the plane Cremona group is given
by Giza\-tullin:

\begin{thm}[\cite{Gizatullin:relations}]
  The Cremona group $\mathrm{Bir}(\mathbb{P}^2_\mathbb{C})$ is generated
  by the set $\mathcal{Q}$ of all quadratic maps.

  The relations in $\mathrm{Bir}(\mathbb{P}^2_\mathbb{C})$ are
  consequences of relations of the form $q_1\circ q_2\circ q_3=\mathrm{id}$
  where $q_1$, $q_2$, $q_3$ are quadratic birational 
  self maps of $\mathbb{P}^2_\mathbb{C}$. In other words
  we have the following presentation
  \[
  \mathrm{Bir}(\mathbb{P}^2_\mathbb{C})=\langle\mathcal{Q}\,\vert\,q_1\circ q_2\circ q_3=\mathrm{id}\text{ $\,\,\forall$ $q_1$, $q_2$, $q_3\in\mathcal{Q}$ such that $q_1\circ q_2\circ q_3=\mathrm{id}$ in $\mathrm{Bir}(\mathbb{P}^2_\mathbb{C})$} \rangle
  \]
\end{thm}

Two years later Iskovskikh proved the following statement:

\begin{thm}[\cite{Iskovskikh:generatorsandrelations, Iskovskikh:relations}]\label{thm:Iskovskikh}
The group $\mathrm{Bir}(\mathbb{P}^1_\mathbb{C}\times\mathbb{P}^1_\mathbb{C})$
is generated by the group~$B$ of birational maps preserving the fibration
given by the first projection together with $\tau\colon(z_0,z_1)\mapsto(z_1,z_0)$. 

Moreover the following relations form a complete system of relations:
\begin{itemize}
\item[$\diamond$] relations inside the groups $\mathrm{Aut}(\mathbb{P}^1_\mathbb{C}\times\mathbb{P}^1_\mathbb{C})$ and $B$;

\item[$\diamond$] $\left(\tau\circ\left((z_0,z_1)\mapsto\left(z_0,\frac{z_0}{z_1}\right)\right)\right)^3=\mathrm{id}$;

\item[$\diamond$] $\left(\tau\circ\left((z_0,z_1)\mapsto(-z_0,z_1-z_0)\right)\right)^3=\mathrm{id}$.
\end{itemize}
\end{thm}

In $1994$ Iskovskikh, Kabdykairov and Tregub present a list of
generators and relations of~$\mathrm{Bir}(\mathbb{P}^2_\Bbbk)$
over arbitrary perfect field $\Bbbk$ (\emph{see}
\cite{IskovskikhKabdykairovTregub}).

The group $\mathrm{Bir}(\mathbb{P}^2_\mathbb{C})$ hasn't a
structure of amalgamated product (\cite{Cornulier:amalgamatedproduct}). 
Nevertheless a presentation of the plane Cremona group in the form
of a generalized amalgam was given by Wright:

\begin{thm}[\cite{Wright}]\label{thm:Wright}
The plane Cremona group is the free product of~$\mathrm{PGL}(3,\mathbb{C})$, 
$\mathrm{Aut}(\mathbb{P}^1_\mathbb{C}\times\mathbb{P}^1_\mathbb{C})$
and $\mathcal{J}$ amalgamated along their pairwise 
intersections in~$\mathrm{Bir}(\mathbb{P}^2_\mathbb{C})$.  
\end{thm}

Twenty years later Blanc proved: 

\begin{thm}[\cite{Blanc:relations}]\label{thm:Blancrelations}
The group $\mathrm{Bir}(\mathbb{P}^2_\mathbb{C})$ is the amalgamated 
product of the Jonqui\`eres group with the group  of automorphisms
of the plane along their intersection,
divided by the relation $\sigma_2\circ\tau=\tau\circ\sigma_2$ where
$\sigma_2$ is the standard involution and 
$\tau$ is the involution $(z_0:z_1:z_2)\mapsto(z_1:z_0:z_2)$.
\end{thm}

As we have seen in Chapter \ref{Chapter:algebraicsubgroup} there is 
an Euclidean topology on the Cremona group (\cite{BlancZimmermann}).
With respect to this topology $\mathrm{Bir}(\mathbb{P}^2_\mathbb{C})$
is a Hausdorff topological group. Furthermore the restriction of
the Euclidean topology to any algebraic subgroup is the classical
Euclidean topology. To show that $\mathrm{Bir}(\mathbb{P}^2_\mathbb{C})$
is compactly presentable with respect to the Euclidean topology
Zimmermann established the following statement:

\begin{thm}[\cite{Zimmermann:presentation}]
The group $\mathrm{Bir}(\mathbb{P}^2_\mathbb{C})$ is isomorphic to
the amalgamated product of $\mathrm{Aut}(\mathbb{P}^2_\mathbb{C})$,
$\mathrm{Aut}(\mathbb{F}_2)$,
$\mathrm{Aut}(\mathbb{P}^1_\mathbb{C}\times\mathbb{P}^1_\mathbb{C})$
along their pairwise intersection 
in~$\mathrm{Bir}(\mathbb{P}^2_\mathbb{C})$ modulo the relation
$\tau\circ\sigma_2\circ\tau\circ\sigma_2=\mathrm{id}$ where
$\sigma_2$ is the standard involution and $\tau$ the involution
$\tau\colon(z_0:z_1:z_2)\mapsto(z_1:z_0:z_2)$.
\end{thm}

Urech and Zimmermann got a presentation of the
plane Cremona group with respect to the generators
given by the Theorem of Noether and Castelnuovo:

\begin{thm}[\cite{UrechZimmermann}]\label{thm:UrechZimmermann}
The Cremona group $\mathrm{Bir}(\mathbb{P}^2_\mathbb{C})$ is
isomorphic to
\[
\langle\sigma_2,\,\mathrm{PGL}(3,\mathbb{C})\,\vert\,(\mathcal{R}_1), \,(\mathcal{R}_2), \,(\mathcal{R}_3), \,(\mathcal{R}_4),\, (\mathcal{R}_5)\rangle
\]
where
\begin{eqnarray*}
& (\mathcal{R}_1)& g_1\circ g_2\circ g_3^{-1}=\mathrm{id} 
\text{ for all } g_1,\,g_2,\,g_3\in\mathrm{PGL}(3,\mathbb{C}) 
\text{ such that } g_1\circ g_2=g_3;\\
& (\mathcal{R}_2) & \sigma_2^2=\mathrm{id}\\
& (\mathcal{R}_3)& \sigma_2\circ\eta\circ(\eta\circ\sigma_2)^{-1}=\mathrm{id}\text{ for all } \eta \text{ in the symmetric group } 
\mathfrak{S}_3\subset\mathrm{PGL}(3,\mathbb{C}) \\
& & \text{ of order } 6 \text{ acting on }
\mathbb{P}^2_\mathbb{C} \text{ by coordinate permutations}\\
& (\mathcal{R}_4)& \sigma_2\circ d\circ\sigma_2\circ d=\mathrm{id}\text{ for all diagonal automorphisms }d \text{ in the subgroup} \\
& &\mathrm{D}_2\subset\mathrm{PGL}(3,\mathbb{C})\text{ of diagonal automorphisms};\\
& (\mathcal{R}_5)& (\sigma_2\circ h)^3=\mathrm{id} \text{ where } h\colon(z_0:z_1:z_2)\mapsto(z_2-z_0:z_2-z_1:z_2)
\end{eqnarray*}
\end{thm}

\begin{rems}
\begin{itemize}
\item[$\diamond$] The relations $(\mathcal{R}_2)$, $(\mathcal{R}_3)$ 
and $(\mathcal{R}_4)$ occur in the group 
$\mathrm{Aut}(\mathbb{C}^*\times\mathbb{C}^*)$ which is given by the 
group of monomial maps $\mathrm{GL}(2,\mathbb{Z})\ltimes\mathrm{D}_2$.

\item[$\diamond$] $(\mathcal{R}_5)$ is a relation from the group 
$\mathrm{Aut}(\mathbb{P}^1_\mathbb{C}\times\mathbb{P}^1_\mathbb{C})^0\simeq\mathrm{PGL}(2,\mathbb{C})\times\mathrm{PGL}(2,\mathbb{C})$ 
which is considered as a subgroup of $\mathrm{Bir}(\mathbb{P}^2_\mathbb{C})$
by conjugation with the birational equivalence
\begin{eqnarray*}
\mathbb{P}^1_\mathbb{C}\times\mathbb{P}^1_\mathbb{C}&\dashrightarrow&\mathbb{P}^2_\mathbb{C}\\
\big((u_0:u_1),(v_0:v_1)\big)&\dashrightarrow&(u_1v_0:u_0v_1:u_1v_1)
\end{eqnarray*}
\end{itemize}
\end{rems}

\begin{rem}
All the results are stated on $\mathbb{C}$ but indeed
\begin{itemize}
\item[$\diamond$] \cite{Cornulier:amalgamatedproduct, UrechZimmermann, Gizatullin:relations, Iskovskikh:generatorsandrelations, Iskovskikh:relations, Wright, Blanc:relations} 
work for the plane Cremona group over an algebraically 
closed field,
\item[$\diamond$] \cite{Zimmermann:presentation} works for the plane 
Cremona group over a locally compact local field.
\end{itemize}
\end{rem}

\medskip

In the first section we recall the proof of 
Noether and Castelnuovo 
due to Alexander. 

In the second 
section we give an outline of the proof of 
the result of \cite{Cornulier:amalgamatedproduct} that 
says that the plane Cremona group
does not decompose as a non-trivial amalgam. We also
recall the proof of Theorem \ref{thm:Wright}.

The third section is devoted to generators and relations
in the Cremona group. We first give a sketch
of the proof of Theorem \ref{thm:Blancrelations}. 
We also give
a sketch of the proof of Theorem 
\ref{thm:UrechZimmermann}. We then 
explain why there is no Noether and
Castelnuovo theorem in higher dimension.

\bigskip
\bigskip


\section{Noether and Castelnuovo theorem}

Let us now deal with the proof of Theorem \ref{thm:noether} given by 
Alexander (\cite{Alexander}). 
Recall the two following formulas proved in
\S \ref{sec:geodef}. Consider a birational
self map $\phi$ of $\mathbb{P}^2_\mathbb{C}$ 
of degree $\nu$; denote by $p_1$, $p_2$, 
$\ldots$, $p_k$ the base-points of $\phi$ 
and by $m_{p_i}$ the multiplicity 
of $p_i$. Then 
\begin{equation}\label{eq:cremona3}
\displaystyle\sum_{i=0}^km_{p_i}=3(\nu-1)
\end{equation}
\begin{equation}\label{eq:cremona1}
\displaystyle\sum_{i=0}^km_{p_i}^2=\nu^2-1.
\end{equation}

From (\ref{eq:cremona1}) and (\ref{eq:cremona3})
one gets
\begin{equation}\label{eq:cremona2}
\displaystyle\sum_{i=0}^km_{p_i}\big(m_{p_i}-1\big)=(\nu-1)(\nu-2).
\end{equation}

Consider a birational self map of 
$\mathbb{P}^2_\mathbb{C}$ of degree $\nu$. 
If $\nu=1$, then according to 
(\ref{eq:cremona3}) the map $\phi$ is an 
automorphism of $\mathbb{P}^2_\mathbb{C}$.
So let us now assume that $\nu>1$. Let 
$\Lambda_\phi$ be the linear system associated
to $\phi$. Denote by $p_1$, $p_2$, $\ldots$,
$p_k$ the base-points (in 
$\mathbb{P}^2_\mathbb{C}$ or infinitely near)
of $\phi$ and $m_{p_i}$ their multiplicity. 
Up to reindexation let us assume that 
\[
m_{p_0}\geq m_{p_1}\geq \ldots \geq m_{p_k}\geq 1.
\]
Alexander introduced the notion of 
complexity: the 
\textsl{complexity}\index{defi}{complexity (linear system)} 
of $\Lambda_\phi$
is the integer $2c=\nu-m_{p_0}$. It is the 
number of points except $p_0$ that belong
to the intersection of a general line passing
through $p_0$ and a curve of $\Lambda_\phi$. 
\begin{rems}
One has
\begin{itemize}
\item[$\diamond$]  $2c\geq 0$: the degree
of the hypersurfaces of $\Lambda_\phi$ is 
$\nu$, so a point has multiplicity $\leq\nu$;

\item[$\diamond$] furthermore $2c\geq 1$; indeed
if an homogeneous polynomial of degree $\nu$
has a point of multiplicity $\nu$, then the 
hypersurface given by this polynomial is the union of $\nu$ lines.
\end{itemize}
\end{rems}

Set 
\[
C=\big\{p\in\mathrm{Base}(\phi)\smallsetminus\{p_0\}\,\vert\,m_p>c\big\}
\]
and 
\[
n=\#\,C.
\]
Bezout theorem implies that the line 
through $p_0$ and $p_1$ intersects any curve 
of $\Lambda_\phi$ in $\nu$ points (counted with
multiplicity). Furthermore it intersects any 
curve of $\Lambda_\phi$ at $p_0$ with multiplicity
$m_{p_0}$. Consequently $m_{p_1}\leq \nu-m_{p_0}=2c$
and 
\begin{equation}\label{eq:beurk}
c<m_{p_k}\leq \ldots\leq m_{p_2}\leq m_{p_1}\leq 2c
\end{equation}

\begin{lem}\label{lem:three}
There are at least three base-points of multiplicity
$>c=\frac{\nu-m_{p_0}}{2}$, {\it i.e.} $n\geq 2$;
hence $m_{p_0}>\frac{\nu}{3}$.

Furthermore if $\nu\geq 3$, then $p_1$, $p_2$, 
$\ldots$, $p_n$ are not aligned.
\end{lem}

\begin{proof}
According to (\ref{eq:cremona1}) and (\ref{eq:cremona2}) 
one has on the one hand
\[
c\displaystyle\sum_{i=0}^km_{p_i}(m_{p_i}-1)-(c-1)\displaystyle\sum_{i=0}^km_{p_i}^2=\displaystyle\sum_{i=0}^k m_{p_i}(cm_{p_i}-c-cm_{p_i}+m_{p_i})=\displaystyle\sum_{i=0}^km_{p_i}(m_{p_i}-c)
\]
and on the other hand
\[
(\nu-1)(\nu-2)c-(\nu^2-1)(c-1)=(\nu-1)(\nu c-2c-\nu c+\nu-c+1)=(\nu-1)(\nu-3c+1).
\]
As a result
\begin{equation}\label{eq:bloub}
\displaystyle\sum_{i=0}^km_{p_i}(m_{p_i}-c)=(\nu-1)(\nu-3c+1)
\end{equation}
Since $m_{p_{n+i}}\leq c$ for any $i>0$ one gets
\[
\displaystyle\sum_{i=0}^nm_{p_i}(m_{p_i}-c)\geq\displaystyle\sum_{i=0}^km_{p_i}(m_{p_i}-c).
\]
According to (\ref{eq:bloub}) 
\[
\displaystyle\sum_{i=0}^nm_{p_i}(m_{p_i}-c)\geq(\nu-1)(\nu-3c+1)=\nu(\nu-3c)+3c-1.
\]
But $3c-1\geq\frac{1}{2}>0$, so 
\[
\displaystyle\sum_{i=0}^nm_{p_i}(m_{p_i}-c)>\nu(\nu-3c)=\nu(m_{p_0}-c).
\]
Consequently
\[
\displaystyle\sum_{i=0}^nm_{p_i}(m_{p_i}-c)>\nu(m_{p_0}-c)-m_{p_0}(m_{p_0}-c)=(\nu-m_{p_0})(m_{p_0}-c)=2c(m_{p_0}-c).
\]
As $2c\geq m_{p_i}$ for any $i\geq 1$ (see (\ref{eq:beurk})) 
one gets $2c\displaystyle\sum_{i=1}^n(m_{p_i}-c)>2c(m_{p_0}-c)$ and 
\[
2c\displaystyle\sum_{i=1}^n(m_{p_i}-c)>2c(m_{p_0}-c)
\]
that is  
\begin{equation}\label{eq:bipbip}
m_{p_i}-c>m_{p_0}-c
\end{equation}
since $c>0$. But $m_{p_1}\leq m_{p_0}$, so $n\geq 2$. 
Therefore, 
$m_{p_0}+m_{p_1}+m_{p_2}>3\left(\frac{\nu-m_{p_0}}{2}\right)$ 
and $m_{p_0}>\frac{\nu}{3}$.

Let us assume that $n\geq 3$; then (\ref{eq:bipbip})
can be rewritten 
\[
\displaystyle\sum_{i=1}^n m_{p_i}-nc>m_{p_0}-c=\nu-3c
\]
and $\displaystyle\sum_{i=1}^nm_{p_i}>\nu+(n-3)c\geq\nu$.
\end{proof}

\begin{defi}
A \textsl{general quadratic birational 
self map of 
$\mathrm{Bir}(\mathbb{P}^2_\mathbb{C})$ 
centered at $p$, $q$ $r$}
\index{defi}{general quadratic birational 
self map centered at three points}
is the map, up
to linear automorphism, that blows up
the three distinct points $p$, $q$, $r$
of $\mathbb{P}^2_\mathbb{C}$ and blows 
down the strict transform of the lines 
$(pq)$, $(qr)$ and $(pr)$. These lines 
are thus sent onto points denoted 
$p'$, $q'$ and $r'$. 

The line $(p'q')$ (resp. $(q'r')$, 
resp. $(p'r')$) corresponds to the 
exceptional line of the blow up of
$r$ (resp. $p$, resp. $q$).
\end{defi}

\begin{lem}\label{lem:four}
Compose $\phi$ with a general quadratic
birational self map of 
$\mathbb{P}^2_\mathbb{C}$ centered at 
$p_0$, $q$ and $r$ where $p_0$ is the 
base-point of $\phi$ of maximal 
multiplicity. 

The complexity of the new system is 
equal to the complexity of the old 
system if and only if $p_0'$ is of 
maximal multiplicity. 

If it is not the case, then the 
complexity of the new system is 
strictly less than the complexity
of the old one.
\end{lem}

\begin{proof}
The complexity of the new system is 
$2c'=\nu'-m'_{\text{max}}$ where 
$m'_{\text{max}}$ denotes the 
highest multiplicity of the base-points
of the new system. Then 
\begin{eqnarray*}
2c'&=&\nu'-m'_{\text{max}}\\
&=&2\nu-m_{p_0}-m_q-m_r-m'_{\text{max}}\\
&=&\nu-m_{p_0}+(\nu-m_q-m_r)-m'_{\text{max}}\\
&=&\nu-m_{p_0}+m_{p'_0}-m'_{\text{max}}\\
&=&2c+m_{p'_0}-m'_{\text{max}}.
\end{eqnarray*}
Hence $c'\leq c$ and $c'=c$ if and 
only if $m_{p'_0}=m'_{\text{max}}$. 
\end{proof}

\begin{lem}\label{lem:five}
If there exist two points $p_i$ and $p_j$
in $C=\big\{p_1,\,p_2,\,\ldots,\,p_n\big\}$ such that
\begin{itemize}
\item[$\diamond$] $p_i$ and $p_j$ are not 
infinitely near;

\item[$\diamond$] $p_i$ and $p_0$ are not 
infinitely near;

\item[$\diamond$] $p_j$ and $p_0$ are not 
infinitely near.
\end{itemize}
Then there exists a general quadratic birational
self map of $\mathbb{P}^2_\mathbb{C}$ 
such that after composition with $\phi$
\begin{itemize}
\item[$\diamond$] either the complexity of the 
system decreases,

\item[$\diamond$] or $\#\,C=n$ decreases by $2$.
\end{itemize}
\end{lem}

\begin{proof}
Suppose that there exist two points $p_i$ and $p_j$
in $C=\big\{p_1,\,p_2,\,\ldots,\,p_n\big\}$ such that
\begin{itemize}
\item[$\diamond$] $p_i$ and $p_j$ are not 
infinitely near;

\item[$\diamond$] $p_i$ and $p_0$ are not 
infinitely near;

\item[$\diamond$] $p_j$ and $p_0$ are not 
infinitely near.
\end{itemize}

Let us now compose $\phi$ with a general
quadratic birational self map of 
$\mathbb{P}^2_\mathbb{C}$ centered at 
$p_0$, $p_i$ and $p_j$. The degree of 
the new linear system $\Lambda'_\phi$
is $\nu'=2\nu-m_{p_0}-m_{p_i}-m_{p_j}$. 
Let us remark that 
\begin{eqnarray*}
\nu'&=&2\nu-m_{p_0}-m_{p_i}-m_{p_j}\\
&=&\nu+(\nu-m_{p_0}-m_{p_i}-m_{p_j})\\
&=&\nu+(2c-m_{p_i}-m_{p_j})\\
&<&\nu
\end{eqnarray*}
{\it i.e.} the degree has decreased.

The new linear system $\Lambda'_\phi$
has complexity $c'$ and we denote by 
$C'$ the set of points of multiplicity
$>c'$. 

The points $p_0$, $p_i$ and $p_j$ are 
no more points of indeterminacy; the other
base-points and their multiplicity do not 
change. There are three new base-points 
which are $p'_0$, $p'_i$ and $p'_j$.
By definition the multiplicity of 
$p'_0$ (resp. $p'_i$, resp. $p'_j$) 
is equal to the number of intersection 
points (counted with multiplicity) 
between the corresponding contracted
line and the strict transform of a 
general curve of the linear system.
From Bezout theorem we thus
have 
\[
\left\{
\begin{array}{lll}
m_{p'_0}=\nu-m_{p_i}-m_{p_j}\\
m_{p'_i}=\nu-m_{p_0}-m_{p_j}\\
m_{p'_j}=\nu-m_{p_0}-m_{p_i}
\end{array}
\right.
\]
\begin{itemize}
\item[$\diamond$] If $p'_0$ is not the 
point of highest multiplicity, the 
complexity of the system decreases 
(Lemma \ref{lem:four});

\item[$\diamond$] otherwise if 
$p'_0$ is the point of highest 
multiplicity, then the complexity
remains constant (Lemma 
\ref{lem:four}). Furthermore 
$p'_0$ belongs to $C'$ (Lemma 
\ref{lem:three}). Since $m_{p_i}>c$, 
$m_{p_j}>c$ and $\nu-m_{p_0}=2c$, 
then $m_{p'_i}<c$ and $m_{p'_j}<c$, 
{\it i.e.} $p'_i\not\in C'$ and 
$p'_j\not\in C'$. As a consequence
$n'=n-2$.
\end{itemize}
\end{proof}

\begin{lem}\label{lem:six}
Assume there exists a base-point $p_k$ in $C$
that is not infinitely near $p_0$. Then 
after composition by a general quadratic 
birational map, one can disperse the points
above $p_0$ and $p_k$.

The complexity of the system does not change, 
the cardinal of $C$ does not change. There
is no point infinitely near $p'_0$.
\end{lem}

\begin{proof}
Consider a point $q$ of the complex projective
plane such that 
\begin{itemize}
\item[$\diamond$] the lines $(p_0q)$ and 
$(p_kq)$ contain no base-point; 

\item[$\diamond$] there is no point 
infinitely near $p_0$ in the direction
of the line $(p_0q)$;

\item[$\diamond$] there is no point 
infinitely near $p_k$ in the direction
of the line $(p_kq)$.
\end{itemize}
Compose $\phi$ with a general quadratic 
birational map centered at $p_0$, 
$p_k$ and $q$. The degree of the new
linear system is 
\[
\nu'=2\nu-m_{p_0}-m_{p_k}=\nu+2c-m_{p_k}\geq \nu.
\]
The point $p'_0$ is the point of 
highest multiplicity:
\[
\left\{
\begin{array}{lll}
m_{p'_0}=\nu-m_{p_k}\geq\nu-m_{p_0}=2c\geq m_{p_1}\\
m_{p'_k}=\nu-m_{p_0}=2c>c\\
m_{q'}=\nu-m_{p_0}-m_{p_k}=2c-m_{p_k}<c
\end{array}
\right.
\]
hence the complexity remains constant 
(Lemma \ref{lem:four}). Note that 
$\#\,C'=\#\,C$. 

The assumptions on $q$ allow to say that 
a point infinitely near $p_k$ (resp. $p_0$)
is not transformed in a point 
infinitely near $p'_0$. Similarly a point 
infinitely near $p_k$ (resp. $p_0$)
is not transformed in a point infinitely
near $q'$.
\end{proof}

\begin{lem}\label{lem:seven}
Assume that all the points of $C$ are above
the point of highest multiplicity $p_0$. Then 
one can disperse them with a general 
quadratic birational self map; in other 
words there is no base-point of $C'$ 
infinitely near the point $p'_0$ of 
highest multiplicity of the new system. 
The cardinal $n$ increases by $2$ but 
the complexity of the system remains constant.
\end{lem}

\begin{proof}
Take two points $q$ and $r$ in 
$\mathbb{P}^2_\mathbb{C}$ such that the 
lines $(p_0r)$, $(p_0q)$ and $(rq)$
\begin{itemize}
\item[$\diamond$] do not contain 
base-points;

\item[$\diamond$] are not in the 
direction of the points infinitely 
near $p_0$.
\end{itemize}
The degree of the new linear system
is $\nu'=2\nu-m_{p_0}>\nu$. Since 
the curves of the system do not 
pass through $q$ and $r$ Bezout
theorem implies that $m_{p'_0}=\nu$; 
it is thus the point of highest 
multiplicity. Furthermore
\[
2c'=2\nu-m_{p_0}-\nu=2c.
\]
Any curve of the linear system
intersects $(p_0r)$ and $(p_0q)$
in $\nu-m_{p_0}=2c$ points. As a 
result $m_{r'}=m_{q'}=2c>c=c'$. 
Moreover $r'$ and $q'$ belong
to $C'$ and $n'=n+2$. 

The points infinitely near $p_0$ 
have been dispersed onto the line
$(q'r')$. As there is no base-point 
on the line $(qr)$ there is no 
base-point infinitely near $p'_0$.
\end{proof}

\begin{proof}[Proof of Theorem \ref{thm:noether}]
Let us consider a birational self map $\phi$ of
$\mathbb{P}^2_\mathbb{C}$ of degree $\nu$. Denote
by $p_0$, $p_1$, $\ldots$, $p_k$ its base-points
and by $\Lambda_\phi$ the linear system associated
to $\phi$. Let $m_{p_i}$ be the multiplicity of 
$p_i$ and assume up to reindexation that 
\[
m_{p_0}\geq m_{p_1}\geq \ldots\geq m_{p_k}.
\]
Recall that the complexity of the system 
$\Lambda_\phi$ is $c$ where $2c=\nu-m_{p_0}$, 
that 
\[
C=\big\{p\in\mathrm{Base}(\phi)\smallsetminus\{p_0\}\,\vert\, m_p>c\big\}
\]
and that $n=\#\,C$.
We will now compose $\phi$ with a sequence of 
general quadratic birational maps in order 
to decrease the complexity until the 
complexity equals to $1$.

\subsubsection*{Step 1} If all points of $C$ are above
$p_0$, let us apply Lemma \ref{lem:seven}. One gets 
that $p'_0$ has no more infinitely near base-points
and that $n'=n+2$. Let us now apply Lemma \ref{lem:six}
until the points of $C'$ are all distinct; note that
$C'$ and $n'$ do not change. According to Lemma 
\ref{lem:three} the points of $C'$ are not aligned. 
Let us take two of these points, denoted by $p_i$ 
and $p_j$ such that there exist two base-points 
$p_k$ and $p_\ell$ with the following property:
$p_k$ and $p_\ell$ do not belong to the lines 
$(p'_0p_i)$, $(p'_0p_j)$ and $(p_ip_j)$. Apply two 
times Lemma \ref{lem:five} to the points $p_k$ and 
$p_\ell$. If the complexity decreases (the first
or the second time anyway), then let us start this
process again; otherwise the first application of 
Lemma \ref{lem:five} yields to $n'=n$ and the 
second to $n'=n-2$. Furthermore there is no more 
base-point of $C'$ infinitely near $p'_0$ and 
we go to \textit{Step~2}.

\subsubsection*{Step 2} We distinguish two 
possibilities:

\noindent\textit{Step 2i.} Either there are two base-points
in $C$ that are not infinitely near and one applies
Lemma \ref{lem:five}. If the complexity decreases, 
come back to \textit{Step 1}, otherwise come back to 
\textit{Step 2}.

\noindent\textit{Step 2ii.} Or let us apply Lemma 
\ref{lem:six}, then there are two base-points 
that are not infinitely near and one can 
apply \textit{Step 2i}.

According to Lemma \ref{lem:three} if $\nu>1$, 
then $\#\,C\geq 3$. As a result \textit{Step 1} 
and \textit{Step 2}
allow to decrease the complexity. When the 
complexity is $1$, the point $p'_0$ has the 
highest multiplicity and from Lemmas 
\ref{lem:five}, \ref{lem:six} and \ref{lem:seven}
one gets that $\#\,C$ decreases until $0$. 
In other words our system has at most one 
base-point. From (\ref{eq:cremona3}) and 
(\ref{eq:cremona1}) one gets that $\nu=1$ and 
that there is no base-point.
\end{proof}

\section{Amalgamated product and $\mathrm{Bir}(\mathbb{P}^2_\mathbb{C})$}

\subsection{It is not an amalgamated product of two groups}

Let us recall that the group
$\mathrm{Aut}(\mathbb{A}^2_\mathbb{C})\subset\mathrm{Bir}(\mathbb{P}^2_\mathbb{C})$ 
of polynomial automorphisms of the plane is the amalgamated 
product of the affine group 
$\mathrm{Aff}_2=\mathrm{Aut}(\mathbb{P}^2_\mathbb{C})\cap\mathrm{Aut}(\mathbb{A}^2_\mathbb{C})$\index{not}{$\mathrm{Aff}_2$} 
and the group 
$\mathcal{J}_{\mathbb{A}^2_\mathbb{C}}=\mathcal{J}\cap\mathrm{Aut}(\mathbb{A}^2_\mathbb{C})$\index{not}{$\mathcal{J}_{\mathbb{A}^2_\mathbb{C}}$} 
along their intersection. On the contrary 
$\mathrm{Bir}(\mathbb{P}^2_\mathbb{C})$ is not the amalgamated product 
of $\mathrm{Aut}(\mathbb{P}^2_\mathbb{C})$ and $\mathcal{J}$. Indeed
there exist elements of $\mathrm{Bir}(\mathbb{P}^2_\mathbb{C})$ 
of finite order which are neither conjugate to an element 
of~$\mathrm{Aut}(\mathbb{P}^2_\mathbb{C})$, nor to an element of 
$\mathcal{J}$ (\emph{see} \cite{Blanc:cyclic}), contrary to the 
case of amalgamated pro\-ducts.

More precisely Cornulier proved that the plane Cremona group does not decompose 
as a non-trivial amalgam (\cite{Cornulier:amalgamatedproduct}); we will
give a sketch of the proof in this section.

\medskip

A \textsl{graph}\index{defi}{graph} $\Gamma$ consists of two sets $X$ and $Y$, 
and two applications 
\begin{align*}
& Y\to X\times X,\quad y\mapsto (o(y),t(y)) && Y\to Y,\quad y\mapsto\overline{y}
\end{align*}
such that: 
\[
\forall\, y\in Y \qquad \overline{\overline{y}}=y,\qquad \overline{y}\not=y,\qquad o(y)=t(y).
\]

An element of $X$ is a \textsl{vertex}\index{defi}{vertex (of a graph)}
of $\Gamma$; an element $y\in Y$ is an 
\textsl{oriented edge}\index{defi}{oriented edge (of a graph)}, and 
$\overline{y}$ is the \textsl{reversed edge}\index{defi}{reversed edge 
(of a graph)}. The vertex $o(y)=t(\overline{y})$ is the 
\textsl{origin}\index{defi}{origin (of an edge)} of $y$, and the 
vertex $t(y)=o(\overline{y})$ is the 
\textsl{terminal vertice}\index{defi}{terminal vertice (of an edge)}.
These two vertices are called the 
\textsl{extremities}\index{defi}{extremities (of an edge)} of~$y$.

An \textsl{orientation}\index{defi}{orientation (of a graph)} of a graph
$\Gamma$ is a part $Y_+$ of $Y$ such that $Y$ is the disjoint union of $Y_+$
and $\overline{Y_+}$. An \textsl{oriented graph}\index{defi}{oriented (graph)}
is defined, up to isomorphism, by the data of two sets $X$ and $Y_+$, and 
an application $Y_+\to X\times X$. The set of edges of the corresponding graph
is $Y=Y_+\bigsqcup\overline{Y_+}$. 
 
A graph is \textsl{connected}\index{defi}{connected (graph)} if two
vertices are the extremities of at least one path.

\begin{egs}
\begin{itemize}
\item[$\diamond$] Let $n$ be an integer. Let us consider the oriented graph
\begin{figure}[!ht]
\centering
\includegraphics[height=2.5cm]{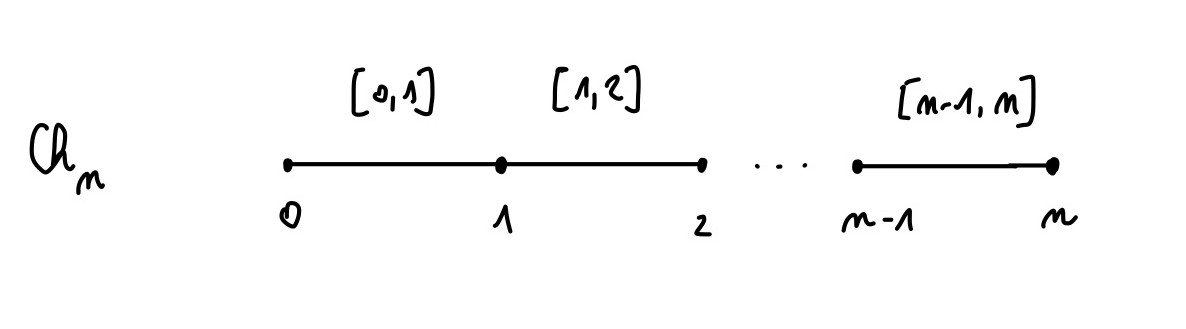}
\end{figure}


It has $n+1$ vertices $0$, $1$, $\ldots$, $n$ and the orientation is 
given by the $n$ egdes $[i,i+1]$, $0\leq i<n$ with $o([i,i+1])=i$ and 
$t([i,i+1])=i+1$.

\item[$\diamond$] Let $n\geq 1$ be an integer. Consider the oriented graph given by
\begin{figure}[!h]
\centering
\includegraphics[height=4.5cm]{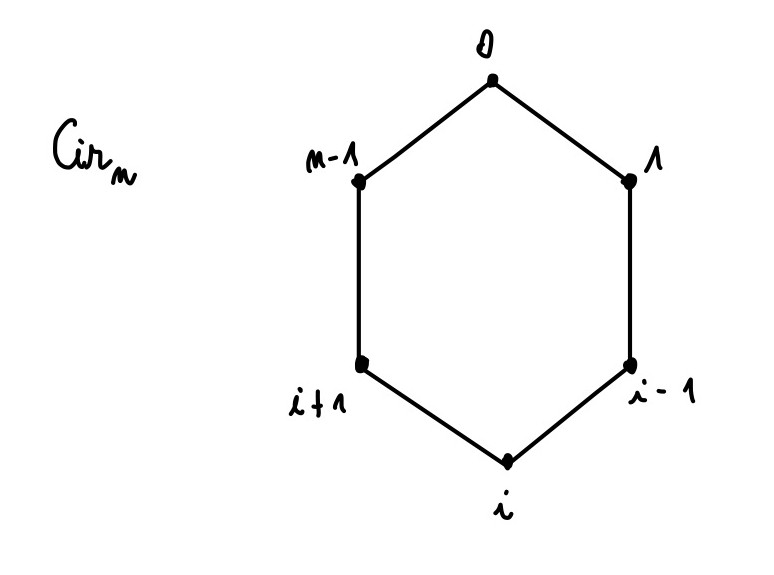}
\end{figure}


The set of vertices is $\faktor{\mathbb{Z}}{n\mathbb{Z}}$, and the orientation
is given by the $n$ edges $[i,i+1]$, $i\in\faktor{\mathbb{Z}}{n\mathbb{Z}}$, with
$o([i,i+1])=i$ and $t([i,i+1])=i+1$.
\end{itemize}
\end{egs}

\begin{defis}
A \textsl{path}\index{defi}{path (of a graph)} of length $n$ in a graph
$\Gamma$ is a morphism from $\mathrm{Ch}_n$ to $\Gamma$.

A \textsl{cycle}\index{defi}{cycle (of a graph)} of length $n$ in a graph
is a subgraph isomorphic to $\mathrm{Cir}_n$.

A \textsl{tree}\index{defi}{tree} is a non-empty, connected graph without cycle.
\end{defis}

\begin{defi}
A group $\mathrm{G}$ is said to have 
\textsl{property (FA)}\index{defi}{property (FA)} if every action of 
$\mathrm{G}$ on a tree has a global fixed point.
\end{defi}

\begin{defis}
A \textsl{geodesic metric space}\index{defi}{geodesic metric space}
is a metric space if given any two points there is a path between 
them whose length equals the distance between the points.

A \textsl{real tree}\index{defi}{real (tree)} can be defined in the 
following equivalent ways (\cite{Chiswell}):
\begin{itemize}
\item[$\diamond$] a geodesic metric space which is $0$-hyperbolic in the
sense of Gromov;

\item[$\diamond$] a uniquely geodesic metric space for which 
$[a,c]\subset[a,b]\cup[b,c]$ for all $a$, $b$ and $c$;

\item[$\diamond$] a geodesic metric space with no subspace homeomorphic
to the circle.
\end{itemize}

In a real tree a \textsl{ray}\index{defi}{ray} is a geodesic embedding 
of the half line. An \textsl{end}\index{defi}{end (of a tree)} is an equivalence 
class of rays modulo being at bounded distance. 

For a group of isometries of a real tree, to \textsl{stably fix an end}
\index{defi}{stably fix an end} means to pointwise stabilize a ray 
modulo eventual coincidence ({\it i.e.} it fixes the end as well
as the corresponding Busemann function\footnote{Let $(X,d)$ 
be a metric space. Given a ray $\gamma$ the Busemann
function $B_\gamma\colon X\to\mathbb{R}$ is defined by 
$B_\gamma(x)=\displaystyle\lim_{t\to +\infty}\big(d(\gamma(t),x)-t\big)$.}).
\end{defis}

\begin{defi}
A group has \textsl{property (FR)}\index{defi}{property (FR)} if for 
every isometric action on a complete real tree every element has a 
fixed point.
\end{defi}

\begin{rem}\label{rem:FRFA}
Property (FR) implies property (FA).
\end{rem}

\begin{lem}[\cite{Cornulier:amalgamatedproduct}]\label{lem:propertyFR}
Let $\mathrm{G}$ be a group. Property (FR) has the following equivalent 
characterizations:
\begin{itemize}
\item[$\diamond$] for every isometric action of $\mathrm{G}$ on a complete
real tree every finitely generated subgroup has a fixed point;

\item[$\diamond$] every isometric action of $\mathrm{G}$ on a complete 
real tree either has a fixed point, or stably fixes a point at infinity.
\end{itemize}
\end{lem}

\begin{defi}
A group $\mathrm{G}$ decomposes as a \textsl{non-trivial amalgam}
\index{defi}{non-trivial amalgam} if 
$\mathrm{G}\simeq\mathrm{G}_1\ast_\mathrm{H}\mathrm{G}_2$ with 
$\mathrm{G}_1\not=\mathrm{H}\not=\mathrm{G}_2$. 
\end{defi}

\begin{thm}[\cite{Serre:arbres}, Chapter 6]\label{thm:serreFA}
A group $\mathrm{G}$ has property (FA) if and only if it does not 
decompose as a non-trivial amalgam.
\end{thm}

In the Appendix of \cite{Cornulier:amalgamatedproduct} the 
author has shown that  $\mathrm{Bir}(\mathbb{P}^2_\mathbb{C})$ 
satisfies the first assertion of Lemma \ref{lem:propertyFR}, hence:

\begin{thm}[\cite{Cornulier:amalgamatedproduct}]\label{thm:notamalgamated}
The Cremona group $\mathrm{Bir}(\mathbb{P}^2_\mathbb{C})$ has 
property (FR).
\end{thm}

According to Remark \ref{rem:FRFA} the group $\mathrm{Bir}(\mathbb{P}^2_\mathbb{C})$ thus has property (FA). 
From Theorem \ref{thm:serreFA} one gets that:

\begin{cor}[\cite{Cornulier:amalgamatedproduct}]
The plane Cremona group does not decompose as a non-trivial
amalgam.
\end{cor}

\smallskip

Let us give the main steps of the proof of Theorem \ref{thm:notamalgamated}.
From now on $\mathcal{T}$ is a complete real tree and all actions on 
$\mathcal{T}$ are isometric.

\subsubsection*{Step 1} Let $p_0$, $p_1$, $\ldots$, $p_k$ be points of $\mathcal{T}$ and $s\geq 0$. Suppose that the following 
equality holds
\[
d(p_i,p_j)=s\vert i-j\vert
\]
for all $i$, $j$ such that $\vert i-j\vert\leq 2$. Then it holds for all
$i$ and $j$.

\subsubsection*{Step 2} If $d\geq 3$, then $\mathrm{SL}(d,\mathbb{C})$ has 
property (FR). In particular if $d\geq 3$, then $\mathrm{PGL}(d,\mathbb{C})$ has property 
(FR).

\subsubsection*{Step 3} Let us recall that a \textsl{torus}\index{defi}{torus} $\mathrm{T}$ in a 
compact Lie group $\mathrm{G}$ is a compact, connected, abelian Lie 
subgroup of $\mathrm{G}$ (and therefore isomorphic to the standard 
torus $\mathbb{T}^n$ for some integer $n$). Given a torus $\mathrm{T}$, the 
\textsl{Weyl group}\index{defi}{Weyl group} of $\mathrm{G}$ with respect
to $\mathrm{T}$ can be defined as the normalizer of $\mathrm{T}$ modulo
the centralizer of $\mathrm{T}$. A 
\textsl{Cartan subgroup}\index{defi}{Cartan subgroup} of
an algebraic group is one of the subgroups whose Lie algebra is a 
Cartan subalgebra. For connected algebraic groups over $\mathbb{C}$ a 
Cartan subgroup is usually defined as the centralizer of a maximal 
torus.

Let $\mathrm{C}$ be the normalizer of the standard Cartan subgroup of 
$\mathrm{PGL}(3,\mathbb{C})$, {\it i.e.} the semi-direct product of the
diagonal matrices by the Weyl group (of order $6$). Set 
$\varsigma\colon (z_0,z_1)\mapsto(1-z_0,1-z_1)$.
 The group generated by $\varsigma$ and 
$\mathrm{C}$ coincides with $\mathrm{PGL}(3,\mathbb{C})$:
\[
\langle \mathrm{C},\varsigma\rangle=\mathrm{PGL}(3,\mathbb{C}).
\]

\subsubsection*{Step 4} Let $\mathrm{Bir}(\mathbb{P}^2_\mathbb{C})$ act on
$\mathcal{T}$ so that $\mathrm{PGL}(3,\mathbb{C})$ has no fixed point and 
has a (unique) stably fixed end. Then $\mathrm{Bir}(\mathbb{P}^2_\mathbb{C})$
fixes this unique end.

\subsubsection*{Step 5} Note that 
$\sigma_2=\big(\varsigma\circ\sigma_2\big)\circ\varsigma\circ\big(\varsigma\circ\sigma_2\big)^{-1}$. 
Since 
$\mathrm{Bir}(\mathbb{P}^2_\mathbb{C})=\langle\sigma_2,\,\mathrm{PGL}(3,\mathbb{C})\rangle$
the groups $\mathrm{H_1}=\mathrm{PGL}(3,\mathbb{C})$ and 
$\mathrm{H}_2=\sigma_2\circ\mathrm{PGL}(3,\mathbb{C})\circ\sigma_2^{-1}$ generate 
$\mathrm{Bir}(\mathbb{P}^2_\mathbb{C})$. Let us consider an action of 
$\mathrm{Bir}(\mathbb{P}^2_\mathbb{C})$ on $\mathcal{T}$. By Steps 2 and 4
it is sufficient to consider the case when $\mathrm{PGL}(3,\mathbb{C})$ has
a fixed point. Let us prove that $\mathrm{Bir}(\mathbb{P}^2_\mathbb{C})$ 
has a fixed point; suppose by contradiction 
that $\mathrm{Bir}(\mathbb{P}^2_\mathbb{C})$ has no fixed point. Denote by 
$\mathcal{T}_i$ the set of fixed points of $\mathrm{H}_i$, $i=1$, $2$.
These two trees are exchanged by $\sigma_2$, and as~$\mathrm{H}_1$ and 
$\mathrm{H}_2$ generate $\mathrm{Bir}(\mathbb{P}^2_\mathbb{C})$ they 
are disjoint. Denote by $\mathcal{S}=[p_1,p_2]$, $p_i\in\mathcal{T}_i$, 
the minimal segment joining the two trees, and by $s>0$ its length. The
segment $\mathcal{S}$ is thus fixed by 
$\mathrm{C}\subset\mathrm{H}_1\cap\mathrm{H}_2$, and reversed by 
$\sigma_2$. Step $1$ implies that for all $k\geq 1$, the distance between the points
$p_1$ and $(\sigma_2\circ\varsigma)^kp_1$ is exactly $sk$. This contradicts
the fact that $(\sigma_2\circ\varsigma)^3=\mathrm{id}$. 

\subsection{It is an amalgamated product of three groups}\label{subsec:wright}

In \cite{Wright} the author shows that 
$\mathrm{Bir}(\mathbb{P}^2_\mathbb{C})=\mathrm{Aut}_{\mathbb{C}}\mathbb{C}(z_0,z_1)$
acts on a two-dimensional simplicial complex~$C$, which has as vertices 
certain models in the function field $\mathbb{C}(z_0,z_1)$ and 
whose fundamental domain consists of one face $F$. This yields a structure 
description of $\mathrm{Bir}(\mathbb{P}^2_\mathbb{C})$ as an amalgamation of 
three subgroups along pairwise intersections. The subgroup 
$\mathrm{Aut}(\mathbb{A}^2_\mathbb{C})$ acts on $C$ by restriction; more precisely the
face $F$ has an edge~$E$ satisfying the following property: the 
$\mathrm{Aut}(\mathbb{A}^2_\mathbb{C})$-translates of $E$ form a tree $T$, and the 
action of $\mathrm{Aut}(\mathbb{A}^2_\mathbb{C})$ on $T$ yields the well-known structure
theory for $\mathrm{Aut}(\mathbb{A}^2_\mathbb{C})$ as an amalgamated product 
(\cite{Jung}).

Let us give some details. Recall that
\[
\mathrm{Aut}(\mathbb{P}^1_\mathbb{C}\times\mathbb{P}^1_\mathbb{C})=\big(\mathrm{PGL}(2,\mathbb{C})\times\mathrm{PGL}(2,\mathbb{C})\big)\rtimes\faktor{\mathbb{Z}}{2\mathbb{Z}}
\]
and
\[
\mathcal{J}=\mathrm{PGL}(2,\mathbb{C})\ltimes\mathrm{PGL}(2,\mathbb{C}(z_0)).
\]

\begin{proof}[Proof of Theorem \ref{thm:Wright}]
It is based on Theorem \ref{thm:Iskovskikh}.
Denote by $\mathrm{G}$ be the group obtained by amalgamating $\mathrm{PGL}(3,\mathbb{C})$, 
$\mathrm{Aut}(\mathbb{P}^1_\mathbb{C}\times\mathbb{P}^1_\mathbb{C})$, 
$\mathcal{J}$ along their pairwise intersections in 
$\mathrm{Bir}(\mathbb{P}^2_\mathbb{C})$. Let $\tau$ 
be the involution $\tau\colon(z_0,z_1)\mapsto(z_1,z_0)$. 
Consider the group homomorphism
$\alpha\colon\mathrm{G}\to\mathrm{Bir}(\mathbb{P}^2_\mathbb{C})$ 
restricting to the identity on 
\[
\mathrm{PGL}(3,\mathbb{C})\cup\mathrm{Aut}(\mathbb{P}^1_\mathbb{C}\times\mathbb{P}^1_\mathbb{C})\cup\mathcal{J}.
\]
As $\mathrm{im}\,\alpha$ contains $\mathcal{J}$ and 
$\tau\in\mathrm{Aut}(\mathbb{P}^1_\mathbb{C}\times\mathbb{P}^1_\mathbb{C})$
Theorem \ref{thm:Iskovskikh} implies that $\alpha$ is surjective.

Since 
$\big\{\mathrm{id},\,\tau\big\}\subset\mathrm{Aut}(\mathbb{P}^1_\mathbb{C}\times\mathbb{P}^1_\mathbb{C})$ 
Theorem 
\ref{thm:Iskovskikh} gives a map $\widetilde{\beta}$ from the free product 
$\big\{\mathrm{id},\,\tau\big\}\ast\mathcal{J}$ to $\mathrm{G}$. Since 
$\mathrm{Aut}(\mathbb{P}^1_\mathbb{C}\times\mathbb{P}^1_\mathbb{C})\subset\mathrm{G}$ the equality
\[
\tau\circ(\varphi_0,\varphi_1)\circ\tau=(\varphi_1,\varphi_0)\qquad\forall\,
(\varphi_0,\varphi_1)
\]
also holds in $\mathrm{G}$. Let us now prove that in $\mathrm{G}$ we have 
$(\tau\circ\varepsilon)^3=\sigma_2$ where 
$\varepsilon\colon(z_0,z_1)\mapsto\left(z_0,\frac{z_0}{z_1}\right)$. First note that the 
equality $\varepsilon=\rho\circ\sigma_2$, where 
$\rho\colon(z_0,z_1)\mapsto\left(\frac{1}{z_0},\frac{z_1}{z_0}\right)$, holds in $\mathcal{J}$
and so in $\mathrm{G}$. On the one hand $\sigma_2$ and $\rho$ commute in 
$\mathcal{J}$ so in $\mathrm{G}$, on the other hand $\sigma_2$ and $\tau$
commute in $\mathrm{Aut}(\mathbb{P}^1_\mathbb{C}\times\mathbb{P}^1_\mathbb{C})$ 
hence in $\mathrm{G}$. Therefore, one has the following equality in $\mathrm{G}$
\begin{equation}\label{eq:rel}
(\tau\circ\varepsilon)^3=(\tau\circ\rho\circ\sigma_2)^3=(\tau\circ\rho)^3\circ\sigma_2^3
\end{equation}
The maps $\tau$ and $\rho$ belong to $\mathrm{PGL}(3,\mathbb{C})$ and 
$(\tau\circ\rho)^3=\mathrm{id}$ in $\mathrm{PGL}(3,\mathbb{C})$; as a consequence 
$(\tau\circ\rho)^3=\mathrm{id}$ in $\mathrm{G}$. One has 
$\sigma_2^3=\sigma_2$ in $\mathrm{Aut}(\mathbb{P}^1_\mathbb{C}\times\mathbb{P}^1_\mathbb{C})$ so in $\mathrm{G}$. 
From (\ref{eq:rel}) one gets $(\tau\circ\varepsilon)^3=\sigma_2$ in $\mathrm{G}$.
Consequently $\widetilde{\beta}$ induces a map 
$\beta\colon\mathrm{Bir}(\mathbb{P}^2_\mathbb{C})\to\mathrm{G}$ with the 
following property: $\beta$ restricts to the identity on $\mathcal{J}$
and $\big\{\mathrm{id},\,\tau\big\}$. According to Theorem 
\ref{thm:Iskovskikh}, $\alpha\circ\beta=\mathrm{id}$. The image of 
$\beta$ contains $\mathcal{J}\subset\mathrm{G}$ and 
$\tau\in(\mathrm{PGL}(3,\mathbb{C})\cap\mathrm{Aut}(\mathbb{P}^1_\mathbb{C}\times\mathbb{P}^1_\mathbb{C}))\subset\mathrm{G}$. 
But both $\mathrm{PGL}(3,\mathbb{C})$ and 
$\mathrm{Aut}(\mathbb{P}^1_\mathbb{C}\times\mathbb{P}^1_\mathbb{C})$ are generated 
by their intersection with $\mathcal{J}$ (in $\mathrm{Bir}(\mathbb{P}^2_\mathbb{C})$) 
together with $\tau$; hence $\mathrm{PGL}(3,\mathbb{C})$ and 
$\mathrm{Aut}(\mathbb{P}^1_\mathbb{C}\times\mathbb{P}^1_\mathbb{C})$ 
belong to $\mathrm{im}\,\beta$. As $\mathrm{G}$ is generated by 
$\mathrm{PGL}(3,\mathbb{C})\cup\mathrm{Aut}(\mathbb{P}^1_\mathbb{C}\times\mathbb{P}^1_\mathbb{C})\cup\mathcal{J}$, $\beta$ is surjective.
Therefore, $\alpha$ is an isomorphism ($\alpha^{-1}=\beta$).
\end{proof}

\smallskip

The amalgamated product group structure of Theorem \ref{thm:Wright} reflects 
the fact that it acts on a simply connected two-dimensional simplicial 
complex. This follows from a higher dimensional analogue of Serre's
tree theory (\emph{see for instance} \cite{Soule, Swan}). Let us detail it.

\begin{defis}
A \textsl{simplicial complex}\index{defi}{simplicial complex} $\mathcal{K}$
is a finite collection of non-empty finite sets such that if $X\in\mathcal{K}$ 
and $\emptyset\not=Y\subseteq X$ then $Y\in\mathcal{K}$. 

The union of all members of $\mathcal{K}$ is denoted by 
$V(\mathcal{K})$\index{not}{$V(\mathcal{K})$}. 

The elements
of $V(\mathcal{K})$ are called the \textsl{vertices}\index{defi}{vertex (of
a simplicial complex)} of $\mathcal{K}$. 

The elements of $\mathcal{K}$ are
called the \textsl{simplices}\index{defi}{simplices (of a simplicial complex)}
of $\mathcal{K}$. 

The \textsl{dimension of a simplex}\index{defi}{dimension
(of a simplex)} $S\in\mathcal{K}$ is $\dim S=\vert S\vert-1$. 

The 
\textsl{dimension}\index{defi}{dimension (of a simplicial complex)} of
$\mathcal{K}$ is the maximum dimension of any simplex in $\mathcal{K}$.
\end{defis}

\textsl{Admissible models}

A \textsl{model}\index{defi}{model} is a reduced, irreducible, separated $\mathbb{C}$-scheme having 
function field $\mathbb{C}(z_0,z_1)$. Consider the set of models $S$ 
satisfying one of the three properties
\begin{itemize}
\item[$\diamond$] $S\simeq \mathbb{P}^2_\mathbb{C}$,

\item[$\diamond$] $S\simeq \mathbb{P}^1_\mathbb{C}\times\mathbb{P}^1_\mathbb{C}$, 

\item[$\diamond$] $S\simeq\mathbb{P}^1_\Bbbk$ for some subfield $\Bbbk$ 
of $\mathbb{C}(z_0,z_1)$ necessarily of pure transcendance degree $1$ 
over~$\mathbb{C}$.
\end{itemize}
Such a $\mathbb{C}$-scheme $S$ will be called an 
\textsl{admissible model}\index{defi}{admissible (model)}. 
In the first (resp. second, resp. third) case, we say that $S$ is 
$\mathbb{P}^2_\mathbb{C}$ (resp. $S$ is a 
$\mathbb{P}^1_\mathbb{C}\times\mathbb{P}^1_\mathbb{C}$, resp. $S$ is a
$\mathbb{P}^1_\Bbbk$).

\smallskip

\textsl{The complex $C$}

It is constructed using as vertices the set of admissible models. The three
models $S$, $V$ and $R$, where $S$ is a $\mathbb{P}^2_\mathbb{C}$, $V$ is a 
$\mathbb{P}^1_\mathbb{C}\times\mathbb{P}^1_\mathbb{C}$ and $R$ is a 
$\mathbb{P}^1_\Bbbk$, determine a face when there exist two distinct points
$p$ and $q$ on $S$ such that
\begin{itemize}
\item[$\diamond$] $V$ is the $\mathbb{P}^1_\mathbb{C}\times\mathbb{P}^1_\mathbb{C}$
($\simeq\mathbb{F}_0$) obtained by blowing up $S$ at $p$ and $q$, then 
blowing down the proper transform of the line in $S$ containing $p$ and $q$;
 
\item[$\diamond$] $R$ is the generic $\mathbb{P}^1_\mathbb{C}$ obtained by blowing 
up $S$ at $p$.
\end{itemize}

If $S$ is the standard $\mathbb{P}^2_\mathbb{C}$, $p=(0:1:0)$ and $q=(1:0:0)$, 
then $V$ is the standard $\mathbb{P}^1_\mathbb{C}\times\mathbb{P}^1_\mathbb{C}$, 
and $R$ the standard $\mathbb{P}^1_{\mathbb{C}(z_0)}$. The standard models form
a face called the standard face in $C$.

\smallskip

\textsl{Fundamental domain}

Note that from the construction of $C$ the group 
$\mathrm{Bir}(\mathbb{P}^2_\mathbb{C})$ acts on $C$ without inverting any edge or
rotating any face. A fondamental domain for the action is given by any one 
face
\begin{figure}[!h]
\centering
\includegraphics[height=2.5cm]{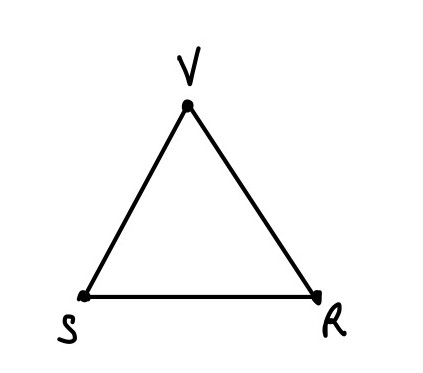}
\end{figure}

If as before we choose $S$ to be the standard $\mathbb{P}^2_\mathbb{C}$, 
$p=(0:1:0)$ et $q=(1:0:0)$ one gets the standard face. For this choice 
the centralizer of $S$, $V$ and $R$ are respectively 
$\mathrm{PGL}(3,\mathbb{C})$, 
$\mathrm{Aut}(\mathbb{P}^1_\mathbb{C}\times\mathbb{P}^1_\mathbb{C})$, 
$\mathcal{J}$.

Let us recall that two simplices $S_0$ and $S_n$ are 
\textsl{$k$-connected}\index{defi}{$k$-connected (simplices)} 
if there is a sequence of
simplices $S_0$, $S_1$, $S_2$, $\ldots$, $S_n$ such that any two 
consecutive ones share a $k$-face, {\it i.e.} they have at least $k+1$
vertices in common. The complex $\mathcal{K}$ is 
\textsl{$k$-connected}\index{defi}{$k$-connected (complex)} if any 
two simplices in $\mathcal{K}$ of dimension $\geq k$ are $k$-connected.

Wright establishes the two following results (\cite{Wright}):
\begin{itemize}
\item[$\diamond$] the simplicial complex $C$ is $1$-connected;

\item[$\diamond$] the complex $C$ contains the $\mathrm{Aut}(\mathbb{A}^2)$-tree.
\end{itemize}

\section{Two presentations of the Cremona group}

\subsection{A simple set of generators and relations for 
$\mathrm{Bir}(\mathbb{P}^2_\mathbb{C})$}\label{subsection:nc1}

In \cite{Blanc:relations} Blanc gives a simple set of generators and 
relations for the plane Cremona group 
$\mathrm{Bir}(\mathbb{P}^2_\mathbb{C})$. 
Namely he shows:

\begin{thm}[\cite{Blanc:relations}]
The group $\mathrm{Bir}(\mathbb{P}^2_\mathbb{C})$ is the amalgamated 
product of the Jonqui\`eres group with the group 
$\mathrm{Aut}(\mathbb{P}^2_\mathbb{C})$ of automorphisms of the plane,
divided by the relation $\sigma_2\circ\tau=\tau\circ\sigma_2$ where
$\tau\colon(z_0:z_1:z_2)\mapsto(z_1:z_0:z_2)$.
\end{thm}

Blanc's proof is inspired by Iskovskikh's proof but Blanc stays on 
$\mathbb{P}^2_\mathbb{C}$. It is clear that $\sigma_2\circ\tau=\tau\circ\sigma_2$,
so it suffices to prove that no other relation holds.

Blanc first establishes the following statement:

\begin{lem}[\cite{Blanc:relations}]\label{lem:tec1}
Let $\varphi$ be an element of $\mathcal{J}$ such that 
$\big\{p_1=(1:0:0),\,q\big\}\subset\mathrm{Base}(\varphi)$ where $q$ 
is a proper point of $\mathbb{P}^2_\mathbb{C}\smallsetminus\{p_1\}$. If 
$\nu\in\mathrm{Aut}(\mathbb{P}^2_\mathbb{C})$ exchanges $p_1$ and $q$,
then
\begin{itemize}
\item[$\diamond$] $\psi=\nu\circ\varphi\circ\nu^{-1}$ belongs to 
$\mathcal{J}$,  
\item[$\diamond$] the relation $\nu\circ\varphi^{-1}=\psi^{-1}\circ\nu$ 
is generated by the relation $\sigma_2\circ\tau=\tau\circ\sigma_2$ in 
the amalgamated product of $\mathcal{J}$ and 
$\mathrm{Aut}(\mathbb{P}^2_\mathbb{C})$. 
\end{itemize}
\end{lem}

Let $\phi$ be an element of 
$\mathrm{Aut}(\mathbb{P}^2_\mathbb{C})\ast_{\mathrm{Aut}(\mathbb{P}^2_\mathbb{C})\cap\mathcal{J}}\mathcal{J}$ 
modulo the relation $\sigma_2\circ\tau=\tau\circ\sigma_2$. Write
$\phi$ as 
\begin{align*}
j_r\circ a_r\circ j_{r-1}\circ a_{r-1}\circ\ldots\circ j_1\circ a_1
\end{align*}
where $j_i\in\mathcal{J}$ and $a_i\in\mathrm{Aut}(\mathbb{P}^2_\mathbb{C})$ 
for $i=1$, $\ldots$, $r$. Note that this decomposition is of course not
unique.

Let $\Lambda_0$ be the linear system of lines of $\mathbb{P}^2_\mathbb{C}$. 
For any $i=1$, $\ldots$, $r$ let us denote by $\Lambda_i$ the linear system
$(j_i\circ a_i\circ\ldots\circ j_1\circ a_1)(\Lambda_0)$, and by $d_i$ the 
degree of $\Lambda_i$. Set
\begin{align*}
& D=\max\big\{d_i\,\vert\,i=1,\,\ldots,\,r\big\}, && n=\max\big\{i\,\vert\,d_i=D\big\}, &&
k=\displaystyle\sum_{i=1}^n\big((\deg j_i)-1\big).
\end{align*}

Recall that $j_i$ belongs to 
$\mathcal{J}\subset\mathrm{Bir}(\mathbb{P}^2_\mathbb{C})$ and satisfies the 
following property: 
\[
\deg j_i=\deg j_i(\Lambda_0)=\deg j_i^{-1}(\Lambda_0). 
\]
In particular $\deg j_i=1$ if and only if $j_i\in\mathrm{Aut}(\mathbb{P}^2_\mathbb{C})$.

Let us give an interpretation of $k$: the number $k$ determines the complexity 
of the word $j_n\circ a_n\circ j_{n-1}\circ a_{n-1}\circ\ldots\circ j_1\circ a_1$ 
which corresponds to the birational self map 
$j_i\circ a_i\circ\ldots\circ j_1\circ a_1$ of the highest degree.

\smallskip

Let us now give the strategy of the proof. If $D=1$, then each $j_i$ is an automorphism of $\mathbb{P}^2_\mathbb{C}$ 
and the word $\phi$ is equal to an element of $\mathrm{Aut}(\mathbb{P}^2_\mathbb{C})$
in the amalgamated product. Since 
$\mathrm{Aut}(\mathbb{P}^2_\mathbb{C})\hookrightarrow\mathrm{Bir}(\mathbb{P}^2_\mathbb{C})$ 
this eventuality is clear. 
Assume now that $D>1$, and prove the result by induction on the pairs
$(D,k)$ (we consider the lexicographic order).

\subsubsection*{Fact} We can suppose that
\begin{align*}
&j_{n+1},\,j_n\in\mathcal{J}\smallsetminus\mathrm{Aut}(\mathbb{P}^2_\mathbb{C}),
&& a_{n+1}\in\mathrm{Aut}(\mathbb{P}^2_\mathbb{C})\smallsetminus\mathcal{J}.
\end{align*}

\subsubsection*{Remark} The point $p=(1:0:0)$ is the base-point of the 
pencil associated to~$\mathcal{J}$. As $a_{n+1}\not\in\mathcal{J}$, one
has $a_{n+1}(p)\not=p$.

\subsubsection*{Properties of the Jonqui\`eres maps}

Since $j_n$, $j_{n+1}$ do not belong to $\mathrm{Aut}(\mathbb{P}^2_\mathbb{C})$,
then $\deg j_n>1$, $\deg j_{n+1}>1$. Set 
$D_L=\deg j_{n+1}$, $D_R=\deg j_n$. The maps $j_{n+1}$ and $j_n$ preserve the 
pencil of lines through $p$. Furthermore $p$ is a base-point of $j_{n+1}$
(resp. $j_n$) of multiplicity $D_L-1$ (resp. $D_R-1$). Since 
$j_{n+1}^\pm(\Lambda_0)$ (resp. $j_n^\pm(\Lambda_0)$) is the image of 
the system $\Lambda_0$ it is a system of rational curves with exactly
one free intersection point. The system $j_{n+1}^\pm(\Lambda_0)$ (resp.
$j_n^\pm(\Lambda_0)$) has $2D_L-2$ (resp. $2D_R-2$) base-points distinct
from $p$, which all have multiplicity $1$.

Set $\Omega_L=(j_{n+1}\circ a_{n+1})^{-1}(\Lambda_0)$ and 
$\Omega_R=(j_n\circ a_n)(\Lambda_0)$. 
Since the automorphisms $a_{n+1}$, $a_n$ are changes of coordinates the 
following properties hold:
\begin{itemize}
\item[$\diamond$] $\deg\Omega_L=D_L$ and $\ell_0=a_{n+1}^{-1}(p)\not=p$ 
is a base-point of $\Omega_L$ of multiplicity $D_L-1$;

\item[$\diamond$] $\deg\Omega_R=D_R$ and $r_0=p$ is a base-point of 
$\Omega_R$ of multiplicity $D_R-1$.
\end{itemize}
The author uses these systems to compute the degrees $d_{n+1}$, resp. 
$d_{n-1}$ of the systems $\Lambda_{n+1}=(j_{n+1}\circ a_{n+1})(\Lambda_n)$, resp. 
$\Lambda_{n-1}=(a_n^{-1}\circ j_n^{-1})(\Lambda_n)$. Indeed for any $i$ the
integer $d_i$ coincides with the degree of $\Lambda_i$ which is 
on the one hand the intersection of $\Lambda_i$ with a general line, 
on the other hand the free intersection of $\Lambda_i$ with $\Lambda_0$.
So $d_{n+1}$ (resp. $d_{n-1}$) is the free intersection of
$\Lambda_{n+1}=(j_{n+1}\circ a_{n+1})(\Lambda_n)$ (resp. 
$\Lambda_{n-1}=(a_n^{-1}\circ j_n^{-1})(\Lambda_n)$) with $\Lambda_0$ 
but also the free intersection of $\Lambda_n$ with $\Omega_L$ (resp. 
$\Omega_R$).

Denote by $m(q)$ the multiplicity of a point $q$ as a base-point of 
$\Lambda_n$. Let $\ell_1$, $\ldots$, $\ell_{2D_L-2}$ (resp. $r_1$, 
$\ldots$, $r_{2D_R-2}$) be the base-points of $\Omega_L$ (resp. $\Omega_R$).
Assume that up to reindexation $m(\ell_i)\geq m(\ell_{i+1})$ (resp.
$m(r_i)\geq m(r_{i+1})$) and if $\ell_i$ (resp. $r_i$) is infinitely near 
to $\ell_j$ (resp. $r_j$), then $i>j$. The following equalities hold:
\begin{equation}\label{eq1}
\left\{
\begin{array}{ll}
d_{n+1}=D_Ld_n-(D_L-1)m(\ell_0)-\displaystyle\sum_{i=1}^{2D_L-2} m(\ell_i)<d_n\\
d_{n-1}=D_Rd_n-(D_R-1)m(r_0)-\displaystyle\sum_{i=1}^{2D_R-2} m(r_i)<d_n
\end{array}
\right.
\end{equation}

Inequalities (\ref{eq1}) imply 
\begin{equation*}\label{eq1b}
\left\{
\begin{array}{ll}
m(\ell_0)+m(\ell_1)+m(\ell_2)>d_n\\
m(r_0)+m(r_1)+m(r_2)\geq d_n
\end{array}
\right.
\end{equation*}

\subsubsection*{First case: $m(\ell_0)\geq m(\ell_1)$ and $m(r_0)\geq m(r_1)$}

Let $q$ be a point in 
$\big\{\ell_1,\,\ell_2,\,r_1,\,r_2\big\}\smallsetminus\{\ell_0,\,r_0\}$ with the 
maximal multiplicity $m(q)$ and so that $q$ is a proper point of 
$\mathbb{P}^2_\mathbb{C}$ or infinitely near to $\ell_0$ or $r_0$.

\begin{itemize}
\item[$\diamond$] Either $\ell_1=r_0$, $m(q)\geq m(\ell_2)$ and 
$m(\ell_0)+m(r_0)+m(q)\geq m(\ell_0)+m(\ell_1)+m(\ell_2)>d_n$ by 
(\ref{eq1b}).

\item[$\diamond$] Or $\ell_1\not=r_0$, $m(q)\geq m(\ell_1)\geq m(\ell_2)$
hence $m(\ell_0)+m(q)>\frac{2d_n}{3}$. The inequalities 
$m(r_0)\geq m(r_1)\geq m(r_2)$ imply $m(r_0)\geq \frac{d_n}{3}$ and then 
$m(\ell_0)+m(r_0)+m(q)>d_n$ holds.

The inequality $m(\ell_0)+m(r_0)+m(q)>d_n$ implies that $\ell_0$, $r_0$ 
and $q$ are not aligned and there exists an element $\theta$ in 
$\mathcal{J}$ of degree $2$ with base
points $\ell_0$, $r_0$, $q$. Note that 
\[
\deg\theta(\Lambda_n)=2d_n-m(\ell_0)-m(r_0)-m(q)<d_n. 
\]
Let us recall that
the automorphism $a_{n+1}$ of $\mathbb{P}^2_\mathbb{C}$ sends $\ell_0$ onto
$r_0=p$. Take $\nu\in\mathrm{Aut}(\mathbb{P}^2_\mathbb{C})\cap~\mathcal{J}$
such that $\nu$ fixes $r_0=p$ and sends $a_{n+1}(r_0)$ onto~$\ell_0$. 
Replace $a_{n+1}$ (resp. $j_{n+1}$) by $\nu\circ a_{n+1}$ (resp. 
$j_{n+1}\circ\nu^{-1}$); we can thus assume that $a_{n+1}$ exchanges $\ell_0$
and $r_0$. As a consequence according to Lemma \ref{lem:tec1} and modulo 
the relation $\sigma_2\circ\tau=\tau\circ\sigma_2$
\begin{small}
\[
j_{n+1}\circ a_{n+1}\circ j_n=j_{n+1}\circ a_{n+1}\circ\theta^{-1}\circ\theta\circ j_n=(j_{n+1}\circ\widetilde{\theta}^{-1})\circ a_{n+1}\circ(\theta\circ j_n)
\]
\end{small}
where $\widetilde{\theta}=a_{n+1}\circ \theta\circ a_{n+1}^{-1}\in\mathcal{J}$. Both 
$j_{n+1}\circ \widetilde{\theta}^{-1}$ and $\theta\circ j_n$ belong to $\mathcal{J}$, but $a_{n+1}$
belongs to $\mathrm{Aut}(\mathbb{P}^2_\mathbb{C})$. Since 
$\theta(\Lambda_n)=(\theta\circ j_n)(\Lambda_{n-1})$ has degree $<d_n$ this 
rewriting decreases the pair $(D,k)$.

\end{itemize}

\subsubsection*{Second case: 
$m(\ell_0)<m(\ell_1)$ or $m(r_0)<m(r_1)$.}

The author comes back to the first case by changing the writing of
$\phi$ in the amalgamated product and modulo the relation
$\sigma_2\circ\tau=\tau\circ\sigma_2$ without changing $(D,k)$
but reversing the inequalities.

\subsection{An other set of generators and relations for 
$\mathrm{Bir}(\mathbb{P}^2_\mathbb{C})$}\label{subsection:nc2}

The idea of the proof of Theorem \ref{thm:UrechZimmermann} is the same as 
in \cite{Iskovskikh:generatorsandrelations, Iskovskikh:relations, Blanc:relations}.
The authors study linear systems of compositions of birational maps of the 
complex projective plane and use the presentation of 
$\mathrm{Bir}(\mathbb{P}^2_\mathbb{C})$ given in Theorem \ref{thm:Blancrelations}.
Before giving the proof of Theorem \ref{thm:UrechZimmermann} let us state the following:

\begin{pro}[\cite{UrechZimmermann}]\label{pro:UrechZimmermann}
Let $\phi_1$, $\phi_2$, $\ldots$, $\phi_n$ be some elements of
$\mathrm{PGL}(3,\mathbb{C})\cap~\mathcal{J}$. Suppose that 
$\phi_n\circ\sigma_2\circ\phi_{n-1}\circ\sigma_2\circ\ldots\circ\sigma_2\circ\phi_1=\mathrm{id}$
as maps. 

Then this expression is generated by relations $(\mathcal{R}_1)$-$(\mathcal{R}_5)$.
\end{pro}

\begin{proof}[Proof of Theorem \ref{thm:UrechZimmermann}]
Let $\mathrm{G}$ be the group generated by $\sigma_2$ and $\mathrm{PGL}(3,\mathbb{C})$ 
divided by the relations $(\mathcal{R}_1)$-$(\mathcal{R}_5)$
\[
\mathrm{G}=\langle\sigma_2,\,\mathrm{PGL}(3,\mathbb{C})\,\vert\,(\mathcal{R}_1)-(\mathcal{R}_5)\rangle.
\]
Denote by $\pi\colon\mathrm{G}\to\mathrm{Bir}(\mathbb{P}^2_\mathbb{C})$ 
the canonical homomorphism that sends generators onto generators. 
Proposition \ref{pro:UrechZimmermann} asserts that sending an element 
of $\mathcal{J}$ onto its corresponding word in $\mathrm{G}$ is 
well defined. Hence there exists a homomorphism 
$w\colon\mathcal{J}\to\mathrm{G}$ such that 
\[
\pi\circ w=\mathrm{id}_{\mathcal{J}}.
\]
In particular $w$ is injective.

The universal property of the amalgamated product implies that there
exists a unique homomorphism
\[
\varphi\colon\mathrm{PGL}(3,\mathbb{C})\ast_{\mathrm{PGL}(3,\mathbb{C})\cap\mathcal{J}}\mathcal{J}\to\mathrm{G}
\]
such that the following diagram commutes 
 \[
  \xymatrix{
     & & &\mathrm{G}    \\
    \mathcal{J}\ar@/^/[rrru] \ar[rr]^w  & & \mathrm{PGL}(3,\mathbb{C})\ast_{\mathrm{PGL}(3,\mathbb{C})\cap\mathcal{J}}\mathcal{J}\ar[ru]^{\varphi} & \\
    \mathrm{PGL}(3,\mathbb{C})\cap\mathcal{J}\ar@{^{(}->}[u] \ar@{^{(}->}[rr] & & \mathrm{PGL}(3,\mathbb{C})\ar@{^{(}->}[u]\ar@/_2pc/[ruu] &
  }
  \]

According to Theorem \ref{thm:Blancrelations} the plane 
Cremona group is isomorphic to 
\[
\mathrm{PGL}(3,\mathbb{C})\ast_{\mathrm{PGL}(3,\mathbb{C})\cap\mathcal{J}}\mathcal{J}
\]
divided by the relation $\tau\circ\sigma_2\circ\tau\circ\sigma_2$ 
where $\tau\colon(z_0:z_1:z_2)\mapsto(z_1:z_0:z_2)$. Note that this 
relation holds as well in $\mathrm{G}$. As a consequence $\varphi$ factors
through the quotient 
\[
\mathrm{PGL}(3,\mathbb{C})\ast_{\mathrm{PGL}(3,\mathbb{C})\cap\mathcal{J}}\mathcal{J}\Big/\langle\tau\circ\sigma_2\circ\tau\circ\sigma_2\rangle.
\]
This yields a homomorphism 
$\overline{\varphi}\colon\mathrm{Bir}(\mathbb{P}^2_\mathbb{C})\to\mathrm{G}$.
More precisely the homomorphisms
$\pi\colon\mathrm{G}\to\mathrm{Bir}(\mathbb{P}^2_\mathbb{C})$ and
$\overline{\varphi}\colon\mathrm{Bir}(\mathbb{P}^2_\mathbb{C})\to\mathrm{G}$
both send
generators to generators
\[
\left\{
\begin{array}{ll}
& \pi(\sigma_2)=\sigma_2\text{ and } \pi(A)=A\,\,\,\,\,\forall\,A\in\mathrm{PGL}(3,\mathbb{C})\\
&\overline{\varphi}(\sigma_2)=\sigma_2\text{ and } \overline{\varphi}(A)=A\,\,\,\,\,\forall\,A\in\mathrm{PGL}(3,\mathbb{C})
\end{array}
\right.
\]
The homomorphisms $\pi$ and $\overline{\varphi}$ are thus isomorphisms that 
are inverse to each other.
\end{proof}

Let us give some Lemmas and Remarks that allow to give a proof of 
Proposition~\ref{pro:UrechZimmermann}.

In \cite{AlberichCarraminana} the author gave a general formula for the 
degree of a composition of two elements of $\mathrm{Bir}(\mathbb{P}^2_\mathbb{C})$
but the multiplicities of the base-points of the composition 
is hard to compute in general. If we impose that one of the two maps 
has degree~$2$ then it is a rather straight forward computation 
(\cite{AlberichCarraminana}). Denote by $m_p(\phi)$\index{not}{$m_p(\phi)$} 
the multiplicity of $\phi$ at the point $p$.
For any $\phi\in\mathcal{J}$ of degree $d$ one has 
\begin{itemize}
\item[$\diamond$] $m_{(1:0:0)}(\phi)=d-1$,
\item[$\diamond$] $m_p(\phi)=1$ $\forall\,p\in\mathrm{Base}(\phi)\smallsetminus\{(1:0:0)\}$,
\end{itemize}
so according to \cite{AlberichCarraminana} one has:

\begin{lem}[\cite{UrechZimmermann}]\label{lem:tecnik1}
Let $\phi$, resp. $\psi$ be a Jonqui\`eres map of degree $2$, 
resp.~$d$. Let $p_1$, $p_2$ be the base-points of $\phi$ different 
from $(1:0:0)$ and $q_1$, $q_2$ be the base-points of $\phi^{-1}$ 
different from $(1:0:0)$ such that the pencil of lines through 
$p_i$ is sent by $\phi$ onto the pencil of lines through $q_i$.

Then 
\begin{itemize}
\item[$\diamond$] $\deg(\psi\circ\phi)=d+1-m_{q_1}(\psi)-m_{q_2}(\psi)$,
\item[$\diamond$] $m_{(1:0:0)}(\psi\circ\phi)=d-m_{q_1}(\psi)-m_{q_2}(\psi)=\deg(\psi\circ\phi)-1$, 
\item[$\diamond$] $m_{p_i}(\psi\circ\phi)=1-m_{q_j}(\psi)$ if $i\not=j$.
\end{itemize}
\end{lem}

\begin{rem}[\cite{UrechZimmermann}]\label{rem:tecnik2}
These equalities can be translate as follows when $\Lambda_\psi$ denotes 
the linear system of $\psi$:
\begin{itemize}
\item[$\diamond$] $\deg(\psi\circ\phi)=\deg(\phi^{-1}(\Lambda_\psi))=d+1-m_{q_1}(\Lambda_\psi)-m_{q_2}(\Lambda_\psi)$,
\item[$\diamond$] $m_{(1:0:0)}\big(\phi^{-1}(\Lambda_\psi)\big)=d-m_{q_1}(\Lambda_\psi)-m_{q_2}(\Lambda_\psi)=\deg\big(\phi^{-1}(\Lambda_\psi)\big)-1$,
\item[$\diamond$] $m_{p_i}\big(\phi^{-1}(\Lambda_\psi)\big)=1-m_{q_j}(\Lambda_\psi)\qquad i\not=j$.
\end{itemize}
But the multiplicity of $\Lambda_\psi$ in a point different from $(1:0:0)$
is $0$ or $1$ so
\[
\left\{
\begin{array}{lll}
\text{eiter }\deg\phi^{-1}(\Lambda_\psi)=\deg(\Lambda_\psi)+1\text{ and }m_{q_1}(\Lambda_\psi)=m_{q_2}(\Lambda_\psi)=0\\
\text{or }\deg\phi^{-1}(\Lambda_\psi)=\deg(\Lambda_\psi)\text{ and }m_{q_1}(\Lambda_\psi)+m_{q_2}(\Lambda_\psi)=1\\
\text{or }\deg\phi^{-1}(\Lambda_\psi)=\deg(\Lambda_\psi)-1\text{ and }m_{q_1}(\Lambda_\psi)=m_{q_2}(\Lambda_\psi)=1
\end{array}
\right.
\]
Furthermore Bezout theorem implies that $(1:0:0)$ and any other base-points
of $\psi$ are not collinear; indeed $(1:0:0)$ is a base-point of multiplicity
$d-1$, all other base-points of multiplicity $1$ (since $\psi$ belongs to 
$\mathcal{J}$) and a general member of $\Lambda_\psi$ intersects a line
in $d$ points counted with multiplicity.
\end{rem}

\begin{lem}[\cite{UrechZimmermann}]\label{lem:tecnik3}
Let $\phi$ be an element of $\mathrm{PGL}(3,\mathbb{C})\cap\mathcal{J}$.
Suppose that $\sigma_2\circ\phi\circ\sigma_2$ is linear. 

Then $\sigma_2\circ\phi\circ\sigma_2$ is generated by 
the relations $(\mathcal{R}_1)$, $(\mathcal{R}_3)$ and $(\mathcal{R}_4)$. 
\end{lem}

\begin{proof}
By Lemma \ref{lem:tecnik1} to say that $\sigma_2\circ\phi\circ\sigma_2$ 
is linear means that 
\[
\mathrm{Base}(\sigma_2\circ\phi)=\mathrm{Base}(\sigma_2)=\big\{(1:0:0),\,(0:1:0),\,(0:0:1)\big\}.
\]
Since $\phi$ belongs to $\mathcal{J}$ it fixes the point $(1:0:0)$,
and so permutes $(0:1:0)$ and $(0:0:1)$. As a result there exist
$\varphi$ in $\mathfrak{S}_3\cap\mathcal{J}$ and $d$ in 
$\mathrm{D}_2$ such that $\phi=d\circ\varphi$. Hence
\[
\sigma_2\circ\phi\circ\sigma_2\stackrel{(1)}{=}\sigma_2\circ d\circ\varphi\circ\sigma_2\stackrel{(3),\,(4)}{=}d^{-1}\circ\varphi.
\]
\end{proof}

\begin{lem}[\cite{UrechZimmermann}]\label{lem:tecnik4}
Let $\phi$ be an element of $\mathrm{PGL}(3,\mathbb{C})\cap\mathcal{J}$.
Suppose that no three of the base-points of $\sigma_2$ and 
$\sigma_2\circ\phi$ are collinear. 

Then there exist $\varphi$, $\psi$ in 
$\mathrm{PGL}(3,\mathbb{C})\cap\mathcal{J}$ such that 
$\sigma_2\circ\phi\circ\sigma_2=\varphi\circ\sigma_2\circ\psi$. 
Furthermore this expression is generated by relations
$(\mathcal{R}_1)$, $(\mathcal{R}_3)$, $(\mathcal{R}_4)$ and
$(\mathcal{R}_5)$.
\end{lem}

\begin{proof}
The assumption $\deg(\sigma_2\circ\phi\circ\sigma_2)=2$ implies
that $\sigma_2$ and $\sigma_2\circ\phi$ have exactly two common
base-points (Lemma \ref{lem:tecnik1}), among them $(1:0:0)$ because
$\sigma_2\circ\phi$ and $\sigma_2$ belong to $\mathcal{J}$.
One can assume up to coordinate permutation that the second 
point is $(0:1:0)$. More precisely there exist $t_1$, $t_2$ in 
$\mathfrak{S}_3\cap\mathcal{J}$ such that $t_1\circ t_2$ fixes
$(1:0:0)$ and $(0:1:0)$. As a result
\[
t_1\circ\phi\circ t_2\colon(z_0:z_1:z_2)\mapsto(a_1z_0+a_2z_2:b_1z_1+b_2z_2:cz_2)
\]
for some complex numbers $a_1$, $a_2$, $b_1$, $b_2$, $c$. Since no
three of the base-points of $\sigma_2$ and $\sigma_2\circ\phi$
are collinear, $a_2b_2$ is non-zero. There thus exist $d_1$, 
$d_2$ in $\mathrm{D}_2$ such that
\[
t_1\circ\phi\circ t_2=d_1\circ\zeta\circ d_2.
\]
We get
\begin{eqnarray*}
 \sigma_2\circ\phi\circ\sigma_2&=&\sigma_2\circ t_1^{-1}\circ t_1\circ\phi\circ t_2\circ t_2^{-1}\circ\sigma_2  \\
 &\stackrel{(1)}{=}& \sigma_2\circ t_1^{-1}\circ d_1\circ\zeta\circ d_2\circ t_2^{-1}\circ\sigma_2 \\
 &\stackrel{(3),(4)}{=}& t_1^{-1}\circ d_1^{-1}\circ\sigma_2\circ\zeta\circ\sigma_2\circ d_2^{-1}\circ t_2^{-1}\\
 &\stackrel{(5)}{=}& t_1^{-1}\circ d_1^{-1}\circ\zeta\circ\sigma_2\circ\zeta\circ d_2^{-1}\circ t_2^{-1} \\
\end{eqnarray*}
Finally $\varphi=t_1^{-1}\circ d_1^{-1}\circ\zeta$ and $\psi=\zeta\circ d_2^{-1}\circ t_2^{-1}$ suit.
\end{proof}

\begin{lem}[\cite{UrechZimmermann}]\label{lem:tecnik5}
Let $\varphi_1$, $\varphi_2$, $\ldots$, $\varphi_n$ be elements 
of $\mathrm{PGL}(3,\mathbb{C})\cap\mathcal{J}$. Then there exist
$\psi_1$, $\psi_2$, $\ldots$, $\psi_n$ in 
$\mathrm{PGL}(3,\mathbb{C})\cap\mathcal{J}$ and $\phi$ in 
$\mathcal{J}$ such that 
\[
\phi\circ\varphi_n\circ\sigma_2\circ\varphi_{n-1}\circ\sigma_2\circ\ldots\circ\sigma_2\circ\varphi_1\circ\phi^{-1}=\psi_n\circ\sigma_2\circ\psi_{n-1}\circ\ldots\circ\sigma_2\circ\psi_1,
\]
and 
\begin{itemize}
\item[$\diamond$] the above relation is generated by relations 
$(\mathcal{R}_1)$-$(\mathcal{R}_5)$, 

\item[$\diamond$] $\deg(\sigma_2\circ\psi_i\circ\sigma_2\circ\ldots\circ\sigma_2\circ\psi_1)=\deg(\sigma_2\circ\phi_i\circ\sigma_2\circ\ldots\circ\sigma_2\circ\phi_1)$ for all $1\leq i\leq n$,

\item[$\diamond$] $(\sigma_2\circ \psi_i\circ\sigma_2\circ \psi_{i-1}\circ\ldots\circ\sigma_2\circ \psi_1)^{-1}$ does not have any infinitely near 
base-points for all $1\leq i\leq n$.
\end{itemize}
\end{lem}

\begin{proof}[Idea of the Proof of Proposition \ref{pro:UrechZimmermann}]
Let us introduce similar notations as in the proof of Theorem 
\ref{thm:Blancrelations}. Let $\Lambda_0$ be the complete linear system
of lines in $\mathbb{P}^2_\mathbb{C}$ and for $1\leq i\leq j$ let 
$\Lambda_i$ be the following linear system
\[
\Lambda_i:=\sigma_2\circ\varphi_{i-1}\circ\sigma_2\circ\ldots\circ\sigma_2\circ\varphi_1(\Lambda_0).
\]
Set $\delta_i:=\deg\Lambda_i$, 
$D_i:=\max\big\{\delta_i\,\vert\,i=1,\,2,\,\ldots,\,j\big\}$, 
$n:=\max\big\{i\,\vert\,\delta_i=D\big\}$. Consider the lexicographic order. Let us 
prove the result by induction on pairs of positive integers $(D,n)$.

If $D=1$, then $j=1$, and there is nothing to prove.

Assume now that $D>1$. We can suppose that for $1\leq i\leq j$ the map
\[
\big(\phi_i\circ\sigma_2\circ\phi_{i-1}\circ\sigma_2\circ\ldots\circ\sigma_2\circ\phi_1\big)^{-1}
\]
does not have any infinitely near base-points (Lemma \ref{lem:tecnik5}). 
Furthermore we can do this without increasing the pair $(D,n)$. Hence any 
$\Lambda_i$, $1\leq i\leq j$, does not have any infinitely near base-points.

The maps $\phi_i$ are Jonqui\`eres ones, so fix $(1:0:0)$. The maps
$\sigma_2\circ\phi_i$ and $\sigma_2$ always have $(1:0:0)$ as common 
base-points. In particular $\deg(\sigma_2\circ\phi_i\circ\sigma_2)\leq 3$
for any $1\leq i\leq j$ (Lemma \ref{lem:tecnik1}). Let us now deal with 
the three distinct cases: $\deg(\sigma_2\circ\phi_n\circ\sigma_2)=1$, 
$\deg(\sigma_2\circ\phi_n\circ\sigma_2)=2$, 
$\deg(\sigma_2\circ\phi_n\circ\sigma_2)=3$.
\begin{itemize}
\item[$\diamond$] First case: $\deg(\sigma_2\circ\phi_n\circ\sigma_2)=1$.
According to Lemma \ref{lem:tecnik3}  the word $\sigma_2\circ\phi_n\circ\sigma_2$
can be replaced by the linear map $\phi'_n=\sigma_2\circ\phi_n\circ\sigma_2$
using relations $(\mathcal{R}_1)$, $(\mathcal{R}_3)$ and $(\mathcal{R}_4)$. 
We thus get a new pair $(D',n')$ with $D'\leq D$; moreover if $D=D'$, then 
$n'<n$.

\item[$\diamond$] Second case: $\deg(\sigma_2\circ\phi_n\circ\sigma_2)=2$.
The maps $\sigma_2$ and $\phi_n\circ\sigma_2$ have exactly two common 
base-points, one of them being $(1:0:0)$. One can assume that the 
other one is $(0:1:0)$. More precisely there are two coordinate 
permutations $t_1$ and $t_2$ in $\mathfrak{S}_3\cap\mathcal{J}$ 
such that $t_1\circ\phi_n\circ t_2$ fixes $(1:0:0)$ and $(0:1:0)$, that is
\[
t_1\circ\phi_n\circ t_2\colon(z_0:z_1:z_2)\mapsto(a_1z_0+a_2z_2:b_1z_1+b_2z_2:cz_2)
\]
for some $a_1$, $a_2$, $b_1$, $b_2$, $c$ in $\mathbb{C}$. Using 
$(\mathcal{R}_1)$ and $(\mathcal{R}_3)$ we get
\begin{small}
\begin{eqnarray*}
& &\phi_j\circ\sigma_2\circ\ldots\circ\phi_{n+1}\circ\sigma_2\circ t_1^{-1}\circ t_1\circ\phi_n\circ t_2\circ t_2^{-1}\circ\sigma_2\circ\ldots\circ\phi_1\\
& & \hspace{2mm}=\phi_j\circ\sigma_2\circ\ldots\circ\sigma_2\circ(\phi_{n+1}\circ t_1^{-1})\circ\sigma_2\circ(t_1\circ\sigma_2\circ t_2)\circ\sigma_2\circ(t_2^{-1}\circ\phi_{n-1})\circ\sigma_2\circ\ldots\circ\phi_1
\end{eqnarray*}
\end{small}
The pair $(D,n)$ is unchanged. Let us thus assume that $t_1=t_2=\mathrm{id}$
and 
\[
\phi_n\colon(z_0:z_1:z_2)\mapsto(a_1z_0+a_2z_2:b_1z_1+b_2z_2:cz_2).
\]
Recall that by assumption for any $1\leq i\leq n$ the maps 
$\phi_i\circ\sigma_2\circ\phi_{i-1}\circ\sigma_2\circ\ldots\circ\sigma_2\circ\phi_1$
have no infinitely other base-points. As a result $\Lambda_n$ has no 
infinitely near base-points. 

\begin{claim}[\cite{UrechZimmermann}]
The product $a_2b_2$ is non-zero.
\end{claim}

\begin{proof}
Assume by contradiction that $a_2b_2=0$. Then $q:=\phi_n^{-1}(0:0:1)$ is 
a base-point of $\sigma_2\circ\phi_n$ that lies on a line contracted 
by $(\sigma_2\circ\phi_{n-1})^{-1}$. By Remark \ref{rem:tecnik2} one has
\[
D-1=\delta_{n+1}=D+1-m_{(0:1:0)}(\Lambda_n)-m_q(\Lambda_n).
\]
In particular $m_q(\Lambda_n)=2-m_{(0:1:0)}(\Lambda_n)=1$. As 
$q\not\in\mathrm{Base}(\sigma_2)$ one has: 
$q\not\in\mathrm{Base}\big((\sigma_2\circ\phi_{n-1})^{-1}\big)$.
Its proper image by $(\sigma_2\circ\phi_{n-1})^{-1}$ is thus a 
base-point of $\Lambda_{n-1}$. But $a_2b_2=0$; as a result $q$
is an infinitely near point: contradiction.
\end{proof}

If $a_2b_2$ is non-zero, then no three of the base-points of $\sigma_2$ and 
$\sigma_2\circ\phi_n$ are collinear. According to Lemma \ref{lem:tecnik5} there 
exist $\psi$ and $\varphi$ in $\mathrm{PGL}(3,\mathbb{C})$ such that the 
word $\sigma_2\circ\phi_n\circ\sigma_2$ can be replaced by the word 
$\psi\circ\sigma_2\circ\varphi$ using $(\mathcal{R}_1)$, $(\mathcal{R}_3)$, 
$(\mathcal{R}_4)$ and $(\mathcal{R}_5)$. We thus get a new pair $(D',n')$
where $D'\leq D$; moreover if $D=D'$, then $n'<n$.

\item[$\diamond$] Third case: $\deg(\sigma_2\circ\phi_n\circ\sigma_2)=3$. See
\cite{UrechZimmermann}.
\end{itemize}
\end{proof}

Let us give an application of this new presentation (\cite{UrechZimmermann}).
In \cite{Gizatullin:rep} Gizatullin has considered the following question:
can a given group homomorphism 
$\varphi\colon\mathrm{PGL}(3,\mathbb{C})\to\mathrm{PGL}(n+1,\mathbb{C})$ be 
extended to a group homomorphism 
$\Phi\colon\mathrm{Bir}(\mathbb{P}^2_\mathbb{C})\to\mathrm{Bir}(\mathbb{P}^n_\mathbb{C})$ ?
He answers yes when $\varphi$ is the projective representation induced by 
the regular action of $\mathrm{PGL}(3,\mathbb{C})$ on the space of plane
conics, plane cubics, or plane quartics. To construct these homomorphisms
Gizatullin uses the following construction. Denote by 
$\mathrm{Sym}(n,\mathbb{C})$ the $\mathbb{C}$-algebra of symmetric $n\times n$
matrices. Define $\mathbb{S}(2,n)$ as the quotient 
$\big(\mathrm{Sym}(n,\mathbb{C})\big)^3\big/\mathrm{GL}(n,\mathbb{C})$ 
where the regular action of $\mathrm{GL}(n,\mathbb{C})$ is given by 
\[
C\cdot(A_0,A_1,A_2)=\big(CA_0{}^{t}\!\,C,\,CA_1{}^{t}\!\,C,\,CA_2{}^{t}\!\,C\big).
\]

\begin{lem}[\cite{UrechZimmermann}]
The variety $\mathbb{S}(2,n)$ is a rational variety, and 
\[
\dim\mathbb{S}(2,n)=\frac{(n+1)(n+2)}{2}-1.
\]
\end{lem}

\begin{rem}
The variety $\mathbb{S}(2,n)$ has thus the same dimension as the space of 
plane curves of degree $n$.
\end{rem}

An element $A=(A_0,A_1,A_2)$ of $\mathrm{PGL}(3,\mathbb{C})$ induces an 
automorphism on $(\mathrm{Sym}(n,\mathbb{C}))^3$ by
\[
\phi(A_0,A_1,A_2):=\big(\phi_0(A_0,A_1,A_2),\phi_1(A_0,A_1,A_2),\phi_2(A_0,A_1,A_2)\big).
\]
This automorphism commutes with the action of $\mathrm{GL}(n,\mathbb{C})$; 
we thus obtain a regular action of $\mathrm{PGL}(3,\mathbb{C})$ on 
$\mathbb{S}(2,n)$. 

Theorem \ref{thm:UrechZimmermann} allows to give a short proof of the 
following statement:

\begin{pro}[\cite{Gizatullin:rep}]
The regular action of $\mathrm{PGL}(3,\mathbb{C})$ extends to a rational
action of $\mathrm{Bir}(\mathbb{P}^2_\mathbb{C})$.
\end{pro}

\begin{proof}
Define the birational action of $\sigma_2$ on $\mathbb{S}(2,n)$ by 
\[
(A_0,A_1,A_2)\dashrightarrow(A_0^{-1},A_1^{-1},A_2^{-1}).
\]
According to Theorem \ref{thm:UrechZimmermann} to see that this indeed 
defines a rational action of $\mathrm{Bir}(\mathbb{P}^2_\mathbb{C})$
on~$\mathbb{S}(2,n)$ it is sufficient to see 
that $(\mathcal{R}_1)$-$(\mathcal{R}_5)$ are satisfied which is the 
case.
\end{proof}

\subsection{Why no Noether and Castelnuovo theorem in higher dimension ?}
\label{subsection:hudsonandpan}

Let us give an idea of the proof of the fact that there is no Noether
and Castelnuovo theorem in higher dimension:

\begin{thm}[\cite{Hudson, Pan:generation}]\label{thm:hudsonpan}
Any set of group generators of $\mathrm{Bir}(\mathbb{P}^n_\mathbb{C})$, $n\geq 3$,
contains uncountably many elements of 
$\mathrm{Bir}(\mathbb{P}^n_\mathbb{C})\smallsetminus\mathrm{PGL}(n+1,\mathbb{C})$.
\end{thm}

Let us first recall the following construction of Pan which given a birational 
self map of~$\mathbb{P}^n_\mathbb{C}$ allows one to construct a birational
self map of $\mathbb{P}^{n+1}_\mathbb{C}$. First introduce some notations: let
$P\in\mathbb{C}[z_0,z_1,\ldots,z_n]_d$, $Q\in\mathbb{C}[z_0,z_1,\ldots,z_n]_\ell$
and $R_0$, $R_1$, $\ldots$, $R_{n-1}\in\mathbb{C}[z_0,z_1,\ldots,z_n]_{d-\ell}$
be some homogeneous polynomials of degree $d$, resp. $\ell$, resp. $d-\ell$.
Consider $\widetilde{\psi}_{P,Q,R}$ and $\widetilde{\psi}_R$ the rational maps given by 
\begin{align*}
& \widetilde{\psi}_{P,Q,R}\colon(z_0:z_1:\ldots:z_n)\dashrightarrow(QR_0:QR_1:\ldots:QR_{n-1}:P),\\
& \widetilde{\psi}_R\colon(z_0:z_1:\ldots:z_n)\dashrightarrow(R_0:R_1:\ldots:R_{n-1}).
\end{align*}

\begin{lem}[\cite{Pan:generation}]\label{lem:construction}
Let $d$ and $\ell$ be some integers such that $d\leq\ell+1\leq 2$. Take 
$Q\in\mathbb{C}[z_0,z_1,\ldots,z_n]_\ell$ and $P\in\mathbb{C}[z_0,z_1,\ldots,z_n]_d$
without common factors. Let $R_1$, $R_2$, $\ldots$, $R_n$ be some elements of 
$\mathbb{C}[z_0,z_1,\ldots,z_{n-1}]_{d-\ell}$. Assume that 
\begin{align*}
& P=z_nP_{d-1}+P_d&& Q=z_nQ_{\ell-1}+Q_\ell 
\end{align*}
with $P_{d-1}$, $P_d$, $Q_{\ell-1}$, $Q_\ell\in\mathbb{C}[z_0,z_1,\ldots,z_{n-1}]$
of degree $d-1$, resp. $d$, resp. $\ell-1$, resp. $\ell$ and such that 
$(P_{d-1},Q_{\ell-1})\not=(0,0)$.

The map $\widetilde{\psi}_{P,Q,R}$ is birational if and only if $\widetilde{\psi}_R$ is.
\end{lem}

This statement allows to prove that given a hypersurface of $\mathbb{P}^n_\mathbb{C}$
one can construct a birational self map of $\mathbb{P}^n_\mathbb{C}$ that blows down 
this hypersurface:

\begin{lem}[\cite{Pan:generation}]\label{lem:prescribed}
Let $n\geq 3$. Let $S$ be an hypersurface of $\mathbb{P}^n_\mathbb{C}$
of degree $\ell\geq 1$ having a point $p$ of multiplicity $\geq \ell-1$. 

Then there exists a birational self map of $\mathbb{P}^n_\mathbb{C}$ of 
degree $d\geq \ell+1$ that blows down $S$ onto a point.
\end{lem}

\begin{proof}
Let us assume without loss of generality that $p=(0:0:\ldots:0:1)$. Suppose
that $S$ is given by $(Q=0)$. Take a generic plane passing through $p$ 
given by $(H=0)$. Choose $P=z_nP_{d-1}+P_d$ such that
\begin{itemize}
\item[$\diamond$] $P_{d-1}\in\mathbb{C}[z_0,z_1,\ldots,z_{n-1}]$ of degree $d-1$ and $\not=0$;

\item[$\diamond$] $P_d\in\mathbb{C}[z_0,z_1,\ldots,z_{n-1}]$ of degree $d$;

\item[$\diamond$] $\mathrm{pgcd}(P,HQ)=1$.
\end{itemize}
Set $\widetilde{Q}=H^{d-\ell-1}q$ and $R_i=z_i$. The statement then follows from 
Lemma \ref{lem:construction}.
\end{proof}

\begin{proof}[Proof of Theorem \ref{thm:hudsonpan}]
Consider the family of hypersurfaces given by $Q(z_1,z_2,z_3)=0$ where 
$(Q=0)$ defines a smooth curve $\mathcal{C}_Q$ of degree $\ell$ on 
$\big\{z_0=z_4=z_5=\ldots=z_n=0\big\}$. Note that $(Q=0)$ is birationally
equivalent to $\mathbb{P}^{n-2}_\mathbb{C}\times\mathcal{C}_Q$. 
Furthermore $(Q=0)$ and $(Q'=0)$ are birationally equivalent if 
and only if $\mathcal{C}_Q$ and $\mathcal{C}_{Q'}$ are isomorphic. 
Take $\ell=2$; the set of isomorphism classes of smooth cubics is a 
$1$-parameter family. For any $\mathcal{C}_Q$ there exists a birational
self map of $\mathbb{P}^n_\mathbb{C}$ that blows down $\mathcal{C}_Q$
onto a point (Lemma \ref{lem:prescribed}). As a result any set of 
group generators of $\mathrm{Bir}(\mathbb{P}^n_\mathbb{C})$, $n\geq 3$, 
has to contain uncountably many elements of 
$\mathrm{Bir}(\mathbb{P}^n_\mathbb{C})\smallsetminus\mathrm{PGL}(n+1,\mathbb{C})$.
\end{proof}

As we have seen one consequence of 
Noether and Castelnuovo 
theorem is that the Jonqui\`eres 
group and 
$\mathrm{Aut}(\mathbb{P}^2_\mathbb{C})=\mathrm{PGL}(3,\mathbb{C})$ 
generate 
$\mathrm{Bir}(\mathbb{P}^2_\mathbb{C})$.
This statement does also not hold 
in higher dimension (\cite{BlancLamyZimmermann}): 
let $n\geq 3$, the $n$-dimensional 
Cremona group is not generated
by $\mathrm{Aut}(\mathbb{P}^n_\mathbb{C})$
and by Jonqui\`eres elements, 
{\it i.e.} elements that preserve a family 
of lines through a given point, which form a
subgroup
\[
\mathrm{PGL}(2,\mathbb{C}(z_2,z_3,\ldots,z_n))\rtimes\mathrm{Bir}(\mathbb{P}^{n-1}_\mathbb{C})\subseteq \mathrm{Bir}(\mathbb{P}^n_\mathbb{C}).
\]
A more precise statement has been 
established in dimension $3$ in 
\cite{BlancYasinsky}: the $3$-dimensional
Cremona group is not generated by
birational maps preserving a linear fibration 
$\mathbb{P}^3_\mathbb{C}\dashrightarrow\mathbb{P}^2_\mathbb{C}$.


\chapter{Algebraic properties of the 
Cremona group}\label{chapter:alg}

\bigskip
\bigskip

The group $\mathrm{Bir}(\mathbb{P}^2_\mathbb{C})$
has many properties of linear groups, so we wonder
if $\mathrm{Bir}(\mathbb{P}^2_\mathbb{C})$ has 
a faithful linear representation; in the first
section we show that the answer
is no (\cite{CerveauDeserti:ptdegre, Cornulier}).
Still in the first section we give the proof of
the following property: the plane 
Cremona group contains non-linear finitely 
generated subgroups (\cite{Cornulier}). 

\medskip

In the second section we give the proof
of the facts that 
\begin{itemize}
\item[$\diamond$] the normal subgroup 
generated by $\sigma_2$ 
in $\mathrm{Bir}(\mathbb{P}^2_\mathbb{C})$
is $\mathrm{Bir}(\mathbb{P}^2_\mathbb{C})$. 
\item[$\diamond$] the normal subgroup, 
generated by a non-trivial element of 
$\mathrm{PGL}(3,\mathbb{C})=\mathrm{Aut}(\mathbb{P}^2_\mathbb{C})$
in $\mathrm{Bir}(\mathbb{P}^2_\mathbb{C})$
is $\mathrm{Bir}(\mathbb{P}^2_\mathbb{C})$. 
\end{itemize}

As a consequence 
$\mathrm{Bir}(\mathbb{P}^2_\mathbb{C})$
is perfect (\cite{CerveauDeserti:ptdegre}), 
that is 
$[\mathrm{Bir}(\mathbb{P}^2_\mathbb{C}),\mathrm{Bir}(\mathbb{P}^2_\mathbb{C})]=\mathrm{Bir}(\mathbb{P}^2_\mathbb{C})$.

\medskip

We finish this chapter by the description of the 
endomorphisms of the plane Cremona
group; as a corollary we get the

\begin{thm}[\cite{Deserti:hopfian}]
The plane Cremona group is hopfian, {\it i.e.} any surjective endomorphism
of $\mathrm{Bir}(\mathbb{P}^2_\mathbb{C})$ is an automorphism.
\end{thm}

We use for that the classification
of the representations of 
$\mathrm{SL}(3,\mathbb{Z})$ in 
$\mathrm{Bir}(\mathbb{P}^2_\mathbb{C})$,
we thus recall and establish it in the third 
section:

\begin{thm}[\cite{Deserti:IMRN}]\label{thm:IMRN}
Let $\Gamma$ be a finite index subgroup of $\mathrm{SL}(3,\mathbb{Z})$. Let $\upsilon$ be 
an injective morphism from $\Gamma$ to $\mathrm{Bir}(\mathbb{P}^2_\mathbb{C})$. 
Then, up to birational conjugacy, either $\upsilon$ is the canonical embedding, 
or $\upsilon$ is the involution $A\mapsto ({}^{t}\!\,A)^{-1}$.
\end{thm}

As a result we obtain the:

\begin{cor}[\cite{Deserti:IMRN}]\label{cor:IMRN}
If a morphism from a subgroup of finite index
of $\mathrm{SL}(n,\mathbb{Z})$ into
$\mathrm{Bir}(\mathbb{P}^2_\mathbb{C})$ has 
infinite image, then $n\leq 3$.
\end{cor}

\bigskip
\bigskip

\section{The group $\mathrm{Bir}(\mathbb{P}^2_\mathbb{C})$ is not linear}\label{section:notlinear}

Cantat and Lamy proved that $\mathrm{Bir}(\mathbb{P}^2_\mathbb{C})$ is not
simple but the non-existence of a faithful representation does not imply 
the non-existence of a non-trivial representation. So let us deal with 
the following statement:

\begin{pro}[\cite{CerveauDeserti:ptdegre}]\label{thm:CerveauDesertinonlinear}
The plane Cremona group has no faithful linear representation
in characteristic zero.
\end{pro}

Before giving the proof let us mention that making an easy refinement of it 
provides the following stronger result:

\begin{pro}[\cite{Cornulier}]
If $\Bbbk$ is an algebraically closed field, then there is no non-trivial
finite dimensional linear representation for $\mathrm{Bir}(\mathbb{P}^2_\Bbbk)$
over any field.
\end{pro}

Let us recall the following statement due to Birkhoff:

\begin{lem}[\cite{Birkhoff}]
Let $\Bbbk$ be a field of characteristic zero. Let $A$, $B$, and $C$ be 
three elements of $\mathrm{GL}(n,\Bbbk)$ such that 
\begin{itemize}
\item[$\diamond$] $[A,B]=C$, $[A,C]=[B,C]=\mathrm{id}$, 

\item[$\diamond$] $C$ has prime order $p$.
\end{itemize}

Then $p\leq n$.
\end{lem}

\begin{proof}
Assume that $\Bbbk$ is algebraically closed.
Since $C$ is of order $p$ its eigenvalues
are $p$-rooth of unity. 

If the eigenvalues of $C$ are all equal to 
$1$, then $C$ is unipotent and $p\leq n$.

Otherwise $C$ admits an eigenvalue
$\alpha\not=1$. Consider the eigenspace
$E_\alpha=\big\{v\,\vert\, Cv=\alpha v\big\}$
of $C$ associated to the eigenvalue 
$\alpha$. By assumption $A$ and $B$ commute to $C$,
so $E_\alpha$ is invariant by $A$ and $B$.
From $[A,B]=C$ we get 
$[A_{\vert E_\alpha},B_{\vert E_\alpha}]=C_{\vert E_\alpha}$;
but
$C_{\vert E_\alpha}=\alpha\mathrm{id}_{\vert E_\alpha}$
hence 
$[A_{\vert E_\alpha},B_{\vert E_\alpha}]=\alpha\mathrm{id}_{\vert E_\alpha}$, 
that is $(B^{-1}AB)_{\vert E_\alpha}=\alpha A_{\vert E_\alpha}$.
Note that $(B^{-1}AB)_{\vert E_\alpha}$ and 
$A_{\vert E_\alpha}$ are conjugate thus 
$(B^{-1}AB)_{\vert E_\alpha}$ and 
$A_{\vert E_\alpha}$ have the same eigenvalues.
Furthermore these eigenvalues are non-zero. 
If $\lambda$ is an eigenvalue of $A_{\vert E_\alpha}$, 
then $\alpha\lambda$, $\alpha^2\lambda$, 
$\ldots$, $\alpha^{p-1}\lambda$ are 
also eigenvalues of $A_{\vert E_\alpha}$.
As $p$ is prime and $\alpha$ distinct from $1$, 
the numbers $\alpha$, $\alpha^2$, $\ldots$, 
$\alpha^{p-1}$ are distinct, 
$\dim E_\alpha\geq p$, and $n\geq p$.
\end{proof}

\begin{proof}[Proof of Proposition \ref{thm:CerveauDesertinonlinear}]
Assume by contradiction that there exists an injective morphism~$\zeta$ from $\mathrm{Bir}(\mathbb{P}^2_\mathbb{C})$ into
$\mathrm{GL}(n,\Bbbk)$. For any prime $p$ let us consider in the 
affine chart $z_2=1$ the group generated by the maps
\begin{align*}
& (z_0,z_1)\mapsto(\mathrm{e}^{-2\mathbf{i}\pi/p}z_0,z_1),&&(z_0,z_1)\dashrightarrow(z_0,z_0z_1), &&(z_0,z_1)\mapsto(z_0,\mathrm{e}^{-2\mathbf{i}\pi/p}z_1). 
\end{align*}
The images of these three elements of 
$\mathrm{Bir}(\mathbb{P}^2_\mathbb{C})$ satisfy the assumptions of
Birkhoff Lemma; therefore, $p\leq n$ for any prime $p$: 
contradiction.
\end{proof}

\medskip

In \cite{Cornulier} Cornulier gives an example of a non-linear finitely 
generated subgroup of the plane Cremona group. The existence of such 
subgroup is not new, for instance it follows from an unpublished construction
of Cantat. The example in \cite{Cornulier} has the additional feature of being
$3$-solvable. To prove its non-linearity Cornulier proves that it contains
nilpotent subgroups of arbitrary large nilpotency length. 

\smallskip

Let $\mathrm{G}$ be a group. Recall that 
$[g,h]=g\circ h\circ g^{-1}\circ h^{-1}$\index{not}{$[g,h]$} 
denotes the commutator of $g$ and~$h$. If $\mathrm{H}_1$ and $\mathrm{H}_2$ 
are two subgroups of $\mathrm{G}$, then 
$[\mathrm{H}_1,\mathrm{H}_2]$\index{not}{$[\mathrm{H}_1,\mathrm{H}_2]$} is 
the subgroup of $\mathrm{G}$ generated by the elements of the form $[g,h]$ 
with $g\in\mathrm{H}_1$ and $h\in\mathrm{H}_2$. We defined the 
\textsl{derived series}\index{defi}{derived series} of 
$\mathrm{G}$ by setting $\mathrm{G}^{(0)}=\mathrm{G}$ and for all $n\geq 0$
\[
\mathrm{G}^{(n+1)}=[\mathrm{G}^{(n)},\mathrm{G}^{(n)}].
\]
The \textsl{soluble length}\index{defi}{soluble length (of a group)}
$\ell(\mathrm{G})$ of $\mathrm{G}$ is defined by 
\[
\ell(\mathrm{G})=\min\big\{k\in\mathbb{N}\cup\{0\}\,\vert\,\mathrm{G}^{(k)}=\{\mathrm{id}\}\big\}
\]
with the convention: $\min\emptyset=\infty$. We say that $\mathrm{G}$
is \textsl{solvable}\index{defi}{solvable (group)} if 
$\ell(\mathrm{G})<\infty$.
The 
\textsl{descending central series}\index{defi}{descending central series} 
of a group $\mathrm{G}$ is defined by $C^0\mathrm{G}=\mathrm{G}$ and for all $n\geq 0$
\[
C^{n+1}\mathrm{G}=[\mathrm{G},C^n\mathrm{G}].
\]
The group $\mathrm{G}$ is \textsl{nilpotent}\index{defi}{nilpotent (group)}
if there exists $j\geq 0$ such that $C^j\mathrm{G}=\{\mathrm{id}\}$. If $j$
is the minimum non-negative number with such a property, we say that 
$\mathrm{G}$ is of 
\textsl{nilpotent class}\index{defi}{nilpotent class (of a group)} $j$.

\smallskip

Take $f$ in $\mathbb{C}(z_0)$ and $g$ in $\mathbb{C}(z_0)^*$; define $\alpha_f$
and $\mu_g$ by 
\begin{align*}
& \alpha_f\colon(z_0,z_1)\dashrightarrow\big(z_0,z_1+f(z_0)\big), && \mu_g\colon(z_0,z_1)\dashrightarrow\big(z_0,z_1g(z_0)\big).
\end{align*}
Note that 
\begin{equation}\label{eq:uneetoile}
\alpha_{f+f'}=\alpha_f\circ\alpha_{f'} \qquad \mu_{gg'}=\mu_g\circ\mu_{g'}\qquad
\mu_g\circ\alpha_f\circ\mu_g^{-1}=\alpha_{fg}
\end{equation}
Take $t\in\mathbb{C}$ and consider $s_t\colon(z_0,z_1)\mapsto(z_0+t,z_1)$. The following equalities 
hold
\begin{equation}\label{eq:deuxetoiles}
s_t\circ\alpha_{f(z_0)}\circ s_t^{-1}=\alpha_{f(z_0-t)}, \qquad s_t\circ\mu_{g(z_0)}\circ s_t^{-1}=\mu_{g(z_0-t)}
\end{equation}
Let $\Gamma_n$ be the subgroup of $\mathrm{Bir}(\mathbb{P}^2_\mathbb{C})$ 
defined for any $n\geq 0$ by 
\[
\Gamma_n=\langle s_1,\,\alpha_{z_0^n}\rangle.
\]
Remark that $\Gamma_n$ is indeed a subgroup of the Jonqui\`eres group.
It satisfies the following properties:

\begin{lem}[\cite{Cornulier}]
The nilpotency length of $\Gamma_n$ is exactly $n+1$, and $\Gamma_n$ is 
torsion free.
\end{lem}

\begin{proof}
Let $\mathrm{A}_n$ be the abelian subgroup of the Jonqui\`eres group consisting 
of all $\alpha_P$ where $P$ ranges over polynomials of degree at most $n$.
The group $\mathrm{A}_n$ is normalized by $s_1$, and $[s_1,\mathrm{A}_n]\subset \mathrm{A}_{n-1}$ for 
$n\geq 1$ while $\mathrm{A}_0=\{\mathrm{id}\}$. Therefore, the largest group 
generated by $s_1$ and $\mathrm{A}_n$ is nilpotent of class at most $n+1$, and 
so is $\Gamma_n$.

Consider now the $n$-iterated group commutator given by
\[
[s_1,[s_1,\ldots,[s_1,\alpha_{z_0^n}]\ldots]
\]
It coincides with $\alpha_{\Delta^nz_0^n}$ where $\Delta$ is the discrete 
differential operator $\Delta P(z_0)=-P(z_0)+P(z_0-1)$. Remark that 
$\Delta^nz_0^n\not=0$ and $\Gamma_n$ is not $n$-nilpotent.

Clearly $\Gamma_n$ is torsion-free.
\end{proof}

The group
\[
\mathrm{G}=\langle s_1,\,\alpha_1,\mu_{z_0}\rangle\subset\mathrm{Bir}(\mathbb{P}^2_\mathbb{Q})
\]
satisfies the following properties:

\begin{pro}[\cite{Cornulier}]
The group $\mathrm{G}\subset\mathrm{Bir}(\mathbb{P}^2_\mathbb{Q})$ is 
solvable of length $3$, and is not linear over any field.
\end{pro}

A consequence of this statement is Proposition \ref{thm:CerveauDesertinonlinear}.

\begin{proof}
Relations (\ref{eq:uneetoile}) and (\ref{eq:deuxetoiles}) imply that 
$\langle s_1,\,\alpha_f,\,\mu_g\,\vert\, f\in\mathbb{C}(z_0),\,g\in\mathbb{C}(z_0)^*\rangle$ 
is solvable of length at most three. The subgroup
\[
\langle s_1,\,\alpha_f,\,\mu_g\,\vert\,f\in\mathbb{C}(z_0),\,g=\displaystyle\prod_{n\in\mathbb{Z}}(z_0-n)^{k_n},\,k_n\text{ finitely supported }\rangle
\]
contains $\Gamma_n$, and is torsion free.

As $\mu_{z_0}^n\circ\alpha_1\circ\mu_{z_0}^{-n}=\alpha_{z_0^n}$, the group $\Gamma_n$ is contained 
in $\mathrm{G}$ for all $n$. But $\Gamma_n$ is nilpotent of length 
exactly $n+1$, hence $\mathrm{G}$ has no linear representation over any field.
\end{proof}

\bigskip


\section{The Cremona group is perfect}\label{section:perfect}

In this section let us prove the following statement

\begin{thm}[\cite{CerveauDeserti:ptdegre}]\label{thm:parfait}
  The plane Cremona group is perfect, {\it i.e.} the
  commutator subgroup of $\mathrm{Bir}(\mathbb{P}^2_\mathbb{C})$
  is $\mathrm{Bir}(\mathbb{P}^2_\mathbb{C})$:
  \[
\big[\mathrm{Bir}(\mathbb{P}^2_\mathbb{C}),\,\mathrm{Bir}(\mathbb{P}^2_\mathbb{C})\big]=\mathrm{Bir}(\mathbb{P}^2_\mathbb{C}).
  \]
\end{thm}

Let $\mathrm{G}$ be a group, and let $g$ be an element of
$\mathrm{G}$. We denote by $\ll g\gg_{\mathrm{G}}$\index{not}{$\ll g\gg_{\mathrm{G}}$}
the \textsl{normal subgroup of $\mathrm{G}$ generated by $g$}\index{defi}{normal 
subgroup generated by an element}:
\[
\ll g\gg_{\mathrm{G}}=\langle h\circ g\circ h^{-1},\,h\circ g^{-1}\circ h^{-1}\,\vert\,h\in \mathrm{G}\rangle.
\]
Since $\mathrm{PGL}(3,\mathbb{C})$ is simple then
\begin{equation}\label{eq:simple}
\ll A\gg_{\mathrm{PGL}(3,\mathbb{C})}=\mathrm{PGL}(3,\mathbb{C})
\end{equation}
for any
non-trivial element $A$ of $\mathrm{PGL}(3,\mathbb{C})$.
Consider now a birational self map $\phi$ of
$\mathrm{Bir}(\mathbb{P}^2_\mathbb{C})$. The Noether and Castelnuovo
Theorem implies that
\begin{equation}\label{eq:noether}
\phi=(A_1)\circ\sigma_2\circ A_2\circ\sigma_2\circ A_3\circ\ldots\circ A_n\circ(\sigma_2)
\end{equation}
with $A_i\in\mathrm{PGL}(3,\mathbb{C})$. The relation (\ref{eq:simple})
implies that
\[
\ll(z_0,z_1)\mapsto(-z_0,-z_1)\gg_{\mathrm{PGL}(3,\mathbb{C})}=\mathrm{PGL}(3,\mathbb{C});
\]
thus any $A_i$ in (\ref{eq:noether}) can be written
\begin{eqnarray*}
& &B_1\circ\big((z_0,z_1)\mapsto(-z_0,-z_1)\big)\circ B_1^{-1}\circ B_2\circ\big((z_0,z_1)\mapsto(-z_0,-z_1)\big)\circ B_2^{-1}\\
& & \hspace{0.3cm}\circ\ldots\circ B_n\circ\big((z_0,z_1)\mapsto(-z_0,-z_1)\big)\circ B_n^{-1}
\end{eqnarray*}
with $B_i\in\mathrm{PGL}(3,\mathbb{C})$. The involutions
$(z_0,z_1)\mapsto(-z_0,-z_1)$ and $\sigma_2$ being conjugate via
$(z_0,z_1)\mapsto\left(\frac{z_0+1}{z_0-1},\frac{z_1+1}{z_1-1}\right)\in\mathrm{PGL}(2,\mathbb{C})\times\mathrm{PGL}(2,\mathbb{C})$
any element of $\mathrm{Bir}(\mathbb{P}^2_\mathbb{C})$ can be written as
a composition of $\mathrm{Bir}(\mathbb{P}^2_\mathbb{C})$-conjugates of $\sigma_2$.
As a consequence one has

\begin{pro}[\cite{CerveauDeserti:ptdegre}]\label{pro:parfait1}
  The normal subgroup of $\mathrm{Bir}(\mathbb{P}^2_\mathbb{C})$
  generated by $\sigma_2$ in $\mathrm{Bir}(\mathbb{P}^2_\mathbb{C})$
  is $\mathrm{Bir}(\mathbb{P}^2_\mathbb{C})$:
  \[
  \ll \sigma_2\gg_{\mathrm{Bir}(\mathbb{P}^2_\mathbb{C})}=\mathrm{Bir}(\mathbb{P}^2_\mathbb{C}).
  \]
\end{pro}

Consider now a non-trivial automorphism $A$ of $\mathbb{P}^2_\mathbb{C}$.
As $\ll A\gg_{\mathrm{PGL}(3,\mathbb{C})}=\mathrm{PGL}(3,\mathbb{C})$
(\emph{see} (\ref{eq:simple})) the involution $(z_0,z_1)\mapsto(-z_0,-z_1)$
can be written as a composition of $\mathrm{PGL}(3,\mathbb{C})$-conjugates of $A$.
Since $(z_0,z_1)\mapsto(-z_0,-z_1)$ and $\sigma_2$ are conjugate via
$(z_0,z_1)\mapsto\left(\frac{z_0+1}{z_0-1},\frac{z_1+1}{z_1-1}\right)\in\mathrm{PGL}(2,\mathbb{C})\times\mathrm{PGL}(2,\mathbb{C})$
one gets
\[
\sigma_2=\varphi_1\circ A\circ\varphi_1^{-1}\circ\varphi_2\circ A\circ\varphi_2^{-1}\circ\ldots\circ\varphi_n\circ A\circ\varphi_n^{-1}
\]
with $\varphi_i\in\mathrm{Bir}(\mathbb{P}^2_\mathbb{C})$. As a consequence
the inclusion
$\ll \sigma_2\gg_{\mathrm{Bir}(\mathbb{P}^2_\mathbb{C})}\subset \ll A\gg_{\mathrm{Bir}(\mathbb{P}^2_\mathbb{C})}$ 
holds. But $\ll \sigma_2\gg_{\mathrm{Bir}(\mathbb{P}^2_\mathbb{C})}=\mathrm{Bir}(\mathbb{P}^2_\mathbb{C})$ (Proposition \ref{pro:parfait1}) hence

\begin{pro}[\cite{CerveauDeserti:ptdegre}]\label{pro:parfait2}
    Let $A$ be a non-trivial automorphism of the complex projective
    plane. Then
    \[
    \ll A\gg_{\mathrm{Bir}(\mathbb{P}^2_\mathbb{C})}=\mathrm{Bir}(\mathbb{P}^2_\mathbb{C}).
    \]
\end{pro}

According to (\ref{eq:noether}) and Proposition \ref{pro:parfait2} one has

\begin{cor}\label{cor:parfait}
  Any birational self map of $\mathbb{P}^2_\mathbb{C}$ can be written as the
  composition of $\mathrm{Bir}(\mathbb{P}^2_\mathbb{C})$-conjugates of the
  translation $(z_0,z_1)\mapsto(z_0,z_1+1)$.
\end{cor}

But the translation $(z_0,z_1)\mapsto(z_0,z_1+1)$ is a commutator
\[
\Big((z_0,z_1)\mapsto(z_0,z_1+1)\Big)=\left[(z_0,z_1)\mapsto(z_0,3z_1),\,(z_0,z_1)\mapsto\left(z_0,\frac{z_1+1}{2}\right)\right]
\]
and Corollary \ref{cor:parfait} thus implies Theorem \ref{thm:parfait}.

\bigskip


\section{Representations of $\mathrm{SL}(n,\mathbb{Z})$ into $\mathrm{Bir}(\mathbb{P}^2_\mathbb{C})$ for $n\geq 3$}

We will now give a sketch of the proofs of 
Theorem \ref{thm:IMRN} and Corollary \ref{cor:IMRN}.

Let us introduce some notations. Given 
$A\in\mathrm{Aut}(\mathbb{P}^2_\mathbb{C})=\mathrm{PGL}(3,\mathbb{C})$ we denote by 
${}^{t}\!\,A$ the linear transpose of $A$\index{not}{${}^{t}$}. 
The involution 
\[
A\mapsto A^\vee=({}^{t}\!\,A)^{-1}
\]\index{not}{$^\vee$}
determines an exterior and algebraic automorphism of the group
$\mathrm{Aut}(\mathbb{P}^2_\mathbb{C})$ (\emph{see} \cite{Dieudonne}).

Let us recall some properties about the 
groups $\mathrm{SL}(n,\mathbb{Z})$ 
(\emph{see for instance} \cite{Steinberg}).
For any integer $q$ let us introduce the 
morphism
\[
\Theta_q\colon\mathrm{SL}(n,\mathbb{Z})\to\mathrm{SL}\left(n,\faktor{\mathbb{Z}}{q\mathbb{Z}}\right)
\]
induced by the reduction modulo $q$ morphism
$\mathbb{Z}\to\faktor{\mathbb{Z}}{q\mathbb{Z}}$. 
Denote by $\Gamma(n,q)$\index{not}{$\Gamma(n,q)$}
the kernel of $\Theta_q$ and by 
$\widetilde{\Gamma}(n,q)$\index{not}{$\widetilde{\Gamma}(n,q)$}
the reciprocical image of the subgroup of
diagonal matrices of 
$\mathrm{SL}\left(n,\faktor{\mathbb{Z}}{q\mathbb{Z}}\right)$ 
by $\Theta_q$. The $\Gamma(n,q)$ are normal 
subgroups called
\textsl{congruence subgroups}\index{defi}{congruence subgroups of $\mathrm{SL}(n,\mathbb{Z})$}.

\begin{thm}[\cite{Steinberg}]\label{thm:structsl}
Let $n\geq 3$ be an integer. Let $\Gamma$ be a 
subgroup of $\mathrm{SL}(n,\mathbb{Z})$. 

If $\Gamma$ has finite index, there exists
an integer $q$ such that the following 
inclusions hold
\[
\Gamma(n,q)\subset\Gamma\subset\widetilde{\Gamma}(n,q).
\]

If $\Gamma$ has infinite index, then 
$\Gamma$ is finite.
\end{thm}

Take $1\leq i,\,j\leq n$, $i\not=j$. 
Let us denote by $\delta_{ij}$ the 
$n\times n$ Kronecker matrix and
set $\mathrm{e}_{ij}=\mathrm{id}+\delta_{i,j}$.
\index{not}{$\mathrm{e}_{ij}$}

\begin{pro}[\cite{Steinberg}]
The group $\mathrm{SL}(3,\mathbb{Z})$ has 
the following presentation
\[
\langle \mathrm{e}_{ij}\,\vert\,
[\mathrm{e}_{ij},\mathrm{e}_{k\ell}]=\left\{\begin{array}{lll}
\mathrm{id}\text{ if $i\not=\ell$ and $j\not=k$}\\
\mathrm{e}_{i\ell}\text{ if $i\not=\ell$ and $j=k$}\\
\mathrm{e}_{kj}^{-1}\text{ if $i=\ell$ and $j\not=k$}
\end{array}\right.,\,(\mathrm{e}_{12}\mathrm{e}_{21}^{-1}\mathrm{e}_{12})^4=\mathrm{id}\rangle.
\]
\end{pro}

\begin{rem}
The $\mathrm{e}_{ij}^q$'s generate $\Gamma(3,q)$ 
and satisfy relations similar to those verified
by the $\mathrm{e}_{ij}$'s except 
$(\mathrm{e}_{12}\mathrm{e}_{21}^{-1}\mathrm{e}_{12})^4=\mathrm{id}$.
\end{rem}

The $\mathrm{e}_{ij}^q$'s are called 
the 
\textsl{standard generators}\index{defi}{standard generators  (congruence group)} of $\Gamma(3,q)$.

\begin{defi}
Let $k$ be an integer. A 
\textsl{$k$-Heisenberg group} 
is a group with the following presentation
\[
\mathcal{H}_k=\langle f,\,g,\,h\,\vert\,[f,g]=h^k,\,[f,h]=[g,h]=\mathrm{id}\rangle.
\]\index{not}{$\mathcal{H}_k$}

We will say that $f$, $g$ and $h$ are the 
\textsl{standard generators} of $\mathcal{H}_k$
\index{defi}{standard generators ($k$-Heisenberg group)}.
\end{defi}

\begin{rems}\label{rem:heis}
\begin{itemize}
\item[$\diamond$] The subgroup of $\mathcal{H}_k$
generated by $f$, $g$ and $h^k$ is a 
subgroup of index~$k$.

\item[$\diamond$] The groups $\Gamma(3,q)$ contain
a lot of $k$-Heisenberg groups;
for instance if $1\leq i\not=j\not=\ell\leq 3$, 
then 
$\langle \mathrm{e}_{ij}^q,\,\mathrm{e}_{i\ell}^q,\,\mathrm{e}_{j\ell}^q\rangle$ 
is a $q$-Heisenberg group of $\Gamma(3,q)$.
\end{itemize}
\end{rems}

Let $\mathrm{G}$ be a finitely generated group, let 
$\big\{a_1,\,a_2,\,\ldots,\,a_n\big\}$ be a generating set of~$\mathrm{G}$, and
let $g$ be an element of $\mathrm{G}$. The \textsl{length} 
$\vert\vert g\vert\vert$ of $g$ is the smallest integer $k$ for which 
there exists a sequence $(s_1,\,s_2,\ldots,\,s_k)$ with 
$s_i\in\big\{a_1,\,a_2,\,\ldots,\,a_n,\,a_1^{-1},\,a_2^{-1},\,\ldots,\,a_n^{-1}\big\}$
for any $1\leq i\leq k$, such that 
\[
g=s_1s_2\ldots s_k. 
\]
We say that 
\[
\displaystyle\lim_{k\to +\infty}\frac{\vert\vert g^k\vert\vert}{k}
\]
is the \textsl{stable length}\index{defi}{stable length (word)}
of $g$. A \textsl{distorted}\index{defi}{distorted (element)}
element of $\mathrm{G}$ is an element of infinite order of~$\mathrm{G}$
whose stable length is zero.

\begin{lem}[\cite{Deserti:IMRN}]\label{lem:disto}
Let 
$\mathcal{H}_k=\langle f,\,g,\,h\,\vert\,[f,g]=h^k,\,[f,h]=[g,h]=\mathrm{id}\rangle$
be a $k$-Heisenberg group.

The element $h^k$ is distorted. 

In particular the standard generators of 
$\Gamma(3,q)$ are distorted.
\end{lem}

\begin{proof}
Since $[f,g]=[g,h]=\mathrm{id}$ on has 
$[f^\ell,g^m]=h^{\ell m}$ for any integer 
$\ell$, $m$. In particular 
$h^{k\ell^2}=[f^\ell,g^\ell]$. As a 
result 
$\vert\vert h^{k\ell^2}\vert\vert\leq 4\ell$.

Each standard generator of $\Gamma(3,q)$
satisfies 
$\mathrm{e}_{ij}^{q^2}=[\mathrm{e}_{i\ell}^q,\mathrm{e}_{\ell j}^q]$.
\end{proof}

\begin{lem}[\cite{Deserti:nilpotent}]\label{lem:nilkeylemma}
Let $\mathrm{G}$ be a finitely generated group. Let $\upsilon$ be a morphism
from $\mathrm{G}$ to $\mathrm{Bir}(\mathbb{P}^2_\mathbb{C})$. Any distorted
element $g$ of $\mathrm{G}$ satisfies $\lambda(\upsilon(g))=1$, {\it i.e.}
$\upsilon(g)$ is an elliptic map or a parabolic one.
\end{lem}

\begin{proof}
Let $\big\{a_1,\,a_2,\,\ldots,\,a_n\big\}$ be a generating set of 
$\mathrm{G}$. The inequalities
\[
\lambda(\upsilon(g))^k\leq \deg(\upsilon(g)^k)\leq\displaystyle\max_i\big(\deg(\upsilon(a_i))\big)^{\vert\vert g^k\vert\vert}
\]
imply the following ones
\[
0\leq\log\big(\lambda(\upsilon(g))\big)\leq\frac{\vert\vert g^k\vert\vert}{k}\log\big(\max_i\big(\deg(\upsilon(a_i))\big)\big).
\]
But since $g$ is distorted 
$\displaystyle\lim_{k\to +\infty}\frac{\vert\vert g^k\vert\vert}{k}=0$ and 
$\log\big(\lambda(\upsilon(g))\big)=0$.
\end{proof}

\begin{rem}
We follow the proof of \cite{Deserti:IMRN};
nevertheless it is possible to "simplify it"
by using the following result: any distorted element of
$\mathrm{Bir}(\mathbb{P}^2_\mathbb{C})$
is algebraic (\cite{BlancFurter:length, CantatCornulier}).

According to Corollary \ref{cor:3plus} we
thus have: any distorted element of
$\mathrm{Bir}(\mathbb{P}^2_\mathbb{C})$
is elliptic.
\end{rem}

\begin{defi}
Let $\phi_1$, $\phi_2$, $\ldots$, $\phi_k$
be some birational self maps of a rational
surface $S$. Assume that $\phi_1$, $\phi_2$, 
$\ldots$, $\phi_k$ are 
virtually isotopic to the identity. 
We say that $\phi_1$, $\phi_2$, $\ldots$, 
$\phi_k$ are 
\textsl{simultaneously virtually isotopic to the identity}
\index{defi}{simultaneously virtually isotopic to the identity} if there exists a surface 
$\widetilde{S}$, a birational map 
$\psi\colon\widetilde{S}\dashrightarrow S$
such that for any $1\leq i\leq k$ the 
map $\psi^{-1}\circ\phi_i\circ\psi$ belongs
to $\mathrm{Aut}(\widetilde{S})$ and 
$\psi^{-1}\circ\phi_i^\ell\circ\psi$
belongs to $\mathrm{Aut}(\widetilde{S})^0$
for some integer $\ell$.
\end{defi}

\begin{pro}[\cite{Deserti:IMRN}]\label{pro:sim}
Let $\upsilon$ be a representation from
\[
\mathcal{H}_k=\langle f,\,g,\,h\,\vert\,[f,g]=h^k,\,[f,h]=[g,h]=\mathrm{id}\rangle 
\]
into $\mathrm{Bir}(\mathbb{P}^2_\mathbb{C})$. 
Assume that any standard generator 
$\upsilon(f)$, $\upsilon(g)$ and $\upsilon(h)$
of $\upsilon(\mathcal{H}_k)$ is virtually 
isotopic to the identity. Then 
$\upsilon(f)$, $\upsilon(g)$ and $\upsilon(h)$
are simultaneously virtually isotopic to the 
identity.
\end{pro}

\begin{proof}
According to Proposition \ref{pro:commas} the 
maps $\upsilon(f)$ and $\upsilon(g)$ are 
simultaneously virtually isotopic to the 
identity. Since $g$ and $h$ commute, 
$\mathrm{Exc}(\upsilon(g))$ and 
$\mathrm{Ind}(\upsilon(g))$ are invariant
by $\upsilon(h)$. The relation 
$[f,g]=h^k$ implies that both 
$\mathrm{Exc}(\upsilon(g))$ and 
$\mathrm{Ind}(\upsilon(g))$ are invariant 
by $\upsilon(f)$. A reasoning analogous 
to that of the proof of Proposition 
\ref{pro:commas} and 
\cite[Lemma 4.1]{DillerFavre} allows us 
to establish the statement.
\end{proof}

The second assertion of Remarks \ref{rem:heis}
leads us to study the representations of 
Heisenberg $k$-groups into 
automorphisms groups of minimal rational 
surfaces. Let us deal with it.

\begin{lem}[\cite{Deserti:IMRN}]\label{lem:notp1p1}
Let $\upsilon$ be a morphism from 
$\mathcal{H}_k$ into 
$\mathrm{Aut}(\mathbb{P}^1_\mathbb{C}\times\mathbb{P}^1_\mathbb{C})$. 

The morphism $\upsilon$ is not an embedding.
\end{lem}

\begin{proof}
We can assume that $\upsilon(f)$, 
$\upsilon(g)$ and $\upsilon(h)$ fixe the
two standard fibrations (if it is not the 
case we can consider $\mathcal{H}_{2k}$
instead of $\mathcal{H}_k$); in other 
words we can assume that 
$\mathrm{im}\,\upsilon$ is contained in 
$\mathrm{PGL}(2,\mathbb{C})\times\mathrm{PGL}(2,\mathbb{C})$.
Denote by $\mathrm{pr}_i$, $i\in\{1,\,2\}$, 
the $i$-th projection. Note that 
$\mathrm{pr}_i(\upsilon(\mathcal{H}_{2k}))$
is a solvable subgroup of 
$\mathrm{PGL}(2,\mathbb{C})$. Furthermore
$\mathrm{pr}_i(\upsilon(h^k))$ is a 
commutator. Hence $\mathrm{pr}_i(\upsilon(h^k))$
is conjugate to the translation 
$z\mapsto z+\beta_i$. Let us prove by 
contradiction that $\beta_i=0$; assume
$\beta_i\not=0$. Then both 
$\mathrm{pr}_i(\upsilon(f))$ and 
$\mathrm{pr}_i(\upsilon(g))$ are
also some translation since they 
commute with $\mathrm{pr}_i(\upsilon(h^k))$.
But then 
$\mathrm{pr}_i(\upsilon(h^k))=[\mathrm{pr}_i(\upsilon(f)),\mathrm{pr}_i(\upsilon(g))]=\mathrm{id}$: 
contradiction with $\beta_i\not=0$. 
As a result $\beta_i=0$ and $\upsilon$
is not an embedding.
\end{proof}

\begin{lem}[\cite{Deserti:IMRN}]\label{lem:heishirz}
Let $\upsilon$ be a morphism from 
$\mathcal{H}_k$ into $\mathrm{Aut}(\mathbb{F}_n)$
with $n\geq 1$.

Then up to birational conjugacy 
$\upsilon(\mathcal{H}_k)$ is a 
subgroup of 
\[
\big\{(z_0,z_1)\mapsto(\alpha z_0+P(z_1),\beta z_1+\gamma)\,\vert\,\alpha,\,\beta\in\mathbb{C}^*,\,\gamma\in\mathbb{C},\,P\in\mathbb{C}[z_1]\big\}.
\]
Moreover up to birational conjugacy
\[
\upsilon(h^{2k})\colon(z_0,z_1)\mapsto(z_0+P(z_1),z_1)
\]
for some $P\in\mathbb{C}[z_1]$.
\end{lem}

\begin{lem}[\cite{Deserti:IMRN}]\label{lem:heispgl3}
Let $\upsilon$ be an embedding of
$\mathcal{H}_k$ into 
$\mathrm{PGL}(3,\mathbb{C})$. Up to 
linear conjugacy 
\begin{align*}
& \upsilon(f)\colon(z_0,z_1)\mapsto(z_0+\zeta z_1,z_1+\beta)&& \upsilon(g)\colon(z_0,z_1)\mapsto(z_0+\gamma z_1,z_1+\delta) \\
& \upsilon(h^k)\colon(z_0,z_1)\mapsto(z_0+k,z_1)&&
\end{align*}
where $\zeta$, $\delta$, $\beta$ $\gamma$ denote
complex numbers such that $\zeta\delta-\beta\gamma=k$.
\end{lem}

\begin{proof}
The Zariski closure 
$\overline{\upsilon(\mathcal{H}_k)}$ of 
$\upsilon(\mathcal{H}_k)$ is an algebraic 
unipotent subgroup of 
$\mathrm{PGL}(3,\mathbb{C})$. By assumption
$\upsilon$ is an embedding, so the 
Lie algebra of 
$\overline{\upsilon(\mathcal{H}_k)}$ is
isomorphic to 
\[
\mathfrak{h}=\left\{\left(\begin{array}{ccc}
0 & \zeta & \beta\\
0 & 0 & \gamma\\
0 & 0 & 0
\end{array}
\right)\,\vert\,\zeta,\,\beta,\,\gamma\in\mathbb{C}\right\}.
\]
Let $\mathrm{pr}$ be the canonical projection
from $\mathrm{SL}(3,\mathbb{C})$ into 
$\mathrm{PGL}(3,\mathbb{C})$. The Lie
algebra of 
$\mathrm{pr}^{-1}(\upsilon(\mathcal{H}_k))$
coincides with $\mathfrak{h}$ up to conjugacy.
Let us recall that the exponential map
sends $\mathfrak{h}$ in the group $\mathrm{H}$
of upper triangular matrices and that 
$\mathrm{H}$ is a connected algebraic group.
As a consequence 
$\big(\mathrm{pr}^{-1}(\overline{\upsilon(\mathcal{H}_k)})\big)^0=\mathrm{H}$. Any element of 
$\mathrm{pr}^{-1}(\overline{\upsilon(\mathcal{H}_k)})$
acts by conjugation on $\mathrm{H}$, so belongs
to 
$\langle \mathrm{H},\,\mathbf{j}\cdot\mathrm{id}\,\vert\,\mathbf{j}^3=1\rangle$. As
$\mathrm{pr}(\mathbf{j}\cdot\mathrm{id})=\mathrm{id}$,
the restriction $\mathrm{pr}_{\vert\mathrm{H}}$ 
of $\mathrm{pr}$ to $\mathrm{H}$ is surjective
on $\overline{\upsilon(\mathcal{H}_k)}$. It is 
also injective. Hence it is an isomorphism. 
Therefore, $\upsilon$ can be lifted to a 
representation $\widetilde{\upsilon}$ from
$\mathcal{H}_k$ into $\mathrm{H}$. The map
$\widetilde{\upsilon}(h^k)$ can be written 
as a commutator; it is thus unipotent. The 
relations satisfied by the generators imply
that up to conjugacy in $\mathrm{SL}(3,\mathbb{C})$
\begin{align*}
& \upsilon(f)\colon(z_0,z_1)\mapsto(z_0+\zeta z_1,z_1+\beta)&& \upsilon(g)\colon(z_0,z_1)\mapsto(z_0+\gamma z_1,z_1+\delta) \\
& \upsilon(h^k)\colon(z_0,z_1)\mapsto(z_0+k,z_1)&&
\end{align*}
with $\zeta\delta-\beta\gamma=k$.
\end{proof}

Let $\rho$ be an embedding of $\Gamma(3,q)$ into
$\mathrm{Bir}(\mathbb{P}^2_\mathbb{C})$. According
to Lemma \ref{lem:disto} and Lemma \ref{lem:nilkeylemma}
for any standard generator $\mathrm{e}_{ij}$ of 
$\mathrm{SL}(3,\mathbb{Z})$ one has 
$\lambda(\rho(e_{ij}))=1$. Theorem 
\ref{thm:dilfav} implies that 
\begin{itemize}
\item[(i)] either one of the 
$\rho(\mathrm{e}_{ij}^q)$ preserves a unique
fibration that is rational or elliptic,

\item[(ii)] or any standard generator of 
$\Gamma(3,q)$ is virtually isotopic to the
identity.
\end{itemize}

Let us first assume that (i) holds. 

\begin{lem}[\cite{Deserti:IMRN}]\label{lem:fleur}
Let $\Gamma$ be a Kazhdan group
that is finitely generated. Let $\rho$
be a morphism from $\Gamma$ into 
$\mathrm{PGL}(2,\mathbb{C}(z_1))$ $($resp.
$\mathrm{PGL}(2,\mathbb{C}))$. Then 
$\rho$ has finite image.
\end{lem}

\begin{proof}
Denote by $\gamma_i$ the generators
of $\Gamma$ and by 
$\left(\begin{array}{cc}
a_i(z_1) & b_i(z_1)\\
c_i(z_1) & d_i(z_1) 
\end{array}\right)$ their image
by $\rho$. A finitely generated
$\mathbb{Q}$-group is isomorphic 
to a subfield of $\mathbb{C}$. 
Hence 
$\mathbb{Q}(a_i(z_0),b_i(z_0),c_i(z_0),d_i(z_0))$
is isomorphic to a subfield of
$\mathbb{C}$ and one can assume 
that 
$\mathrm{im}\,\rho\subset\mathrm{PGL}(2,\mathbb{C})=\mathrm{Isom}(\mathbb{H}_3)$. 
As $\Gamma$ is Kazhdan any 
continuous action of $\Gamma$ by 
isometries of a real or complex 
hyperbolic space has a fixed point.
The image of $\rho$ is thus up to 
conjugacy a subgroup of 
$\mathrm{SO}(3,\mathbb{R})$; 
according to \cite{delaHarpeValette} 
the image of $\rho$ is thus finite. 
\end{proof}

\begin{pro}[\cite{Deserti:IMRN}]\label{pro:2fleurs}
Let $\rho$ be a morphism from $\Gamma(3,q)$ 
to $\mathrm{Bir}(\mathbb{P}^2_\mathbb{C})$.
If one $\rho(\mathrm{e}_{ij}^q)$ preserves
a unique fibration, then $\mathrm{im}\,\rho$
is finite.
\end{pro}

\begin{proof}
Let us assume without loss of generality 
that $\rho(\mathrm{e}_{12}^q)$ preserves
a unique fibration $\mathcal{F}$. The 
relations satisfied by the $\mathrm{e}_{ij}^q$
imply that $\mathcal{F}$ is invariant by
any $\rho(\mathrm{e}_{ij}^{q^2})$. Hence 
for any $\rho(\mathrm{e}_{ij}^{q^2})$ 
there exist 
\begin{itemize}
\item[$\diamond$] $F\colon\mathbb{P}^2_\mathbb{C}\to\mathrm{Aut}(\mathbb{P}^1_\mathbb{C})$ that defines $\mathcal{F}$, 

\item[$\diamond$] and 
$h_{ij}\in\mathrm{PGL}(2,\mathbb{C})$
\end{itemize}
such that $F\circ\rho(\mathrm{e}_{ij}^{q^2})=h_{ij}\circ F$.

Let $\upsilon$ be the morphism defined by 
\begin{align*}
&\upsilon\colon\Gamma(3,q^2)\to\mathrm{PGL}(2,\mathbb{C}), &&  \mathrm{e}_{ij}^{q^2}\mapsto h_{ij}.
\end{align*}
The group $\Gamma(3,q^2)$ is a Kazhdan 
group, so $\Gamma=\ker\upsilon$ is of 
finite index (Lemma \ref{lem:fleur});
as a consequence $\Gamma$ is a Kazhdan 
group. 

Remark that $\mathcal{F}$ can not be 
elliptic; indeed the group of birational 
maps that preserve fiberwise an elliptic
fibration is metabelian and a subgroup
of $\Gamma(3,q^2)$ of finite index can
not be metabelian. 

Let us assume that $\mathcal{F}$ is a 
rational fibration. One can assume 
that $\mathcal{F}=(z_1=\text{ constant})$.
The group of birational maps of the 
complex projective plane that preserves
$\mathcal{F}$ is identified with 
$\mathrm{PGL}(2,\mathbb{C}(z_1))\rtimes\mathrm{PGL}(2,\mathbb{C})$ 
hence 
$\rho_{\vert\Gamma}\colon\Gamma\to\mathrm{PGL}(2,\mathbb{C}(z_1))$
has finite image (Lemma \ref{lem:fleur}).
\end{proof}

Consider now the case (ii), {\it i.e.} assume
that any standard generator of $\Gamma(3,q)$ 
is virtually isotopic to the identity.

\begin{rem}\label{rem:hom}
Two irreducible homologous curves of negative
self-intersection coincide. As a consequence
an automorphism $\varphi$ of a surface $S$ isotopic
to the identity fixes any curve of negative 
self-intersection. Furthermore for any 
sequence of blow-downs $\psi$ from $S$ to a
minimal model $\widetilde{S}$ of $S$ the map
$\psi\circ \varphi\circ\psi^{-1}$ is an
automorphism of $\widetilde{S}$ isotopic to 
the identity. 
\end{rem}

According to Remark \ref{rem:hom}, Proposition
\ref{pro:sim}, Lemma \ref{lem:disto} and Lemma \ref{lem:nilkeylemma} the maps 
$\rho(\mathrm{e}_{12}^{qn})$, 
$\rho(\mathrm{e}_{13}^{qn})$,
$\rho(\mathrm{e}_{23}^{qn})$ are, for some 
integer $n$, some automorphisms of a minimal
rational surface, that is of 
$\mathbb{P}^2_\mathbb{C}$ or of 
$\mathbb{F}_n$, $n\geq 2$. Let us mention 
the case $\mathbb{F}_n$, $n\geq 2$
(\emph{see} \cite{Deserti:IMRN} for more
details) and detail the case 
$\mathbb{P}^2_\mathbb{C}$. 

\begin{lem}[\cite{Deserti:IMRN}]\label{lem:hirzcong}
Let $\rho$ be a morphism from a 
congruence subgroup $\Gamma(3,q)$ 
of $\mathrm{SL}(3,\mathbb{Z})$
in the plane Cremona
group. 

Assume that $\rho(\mathrm{e}_{12}^{q\ell})$,
$\rho(\mathrm{e}_{13}^{q\ell})$ and 
$\rho(\mathrm{e}_{23}^{q\ell})$ belong to
$\mathrm{Aut}(\mathbb{F}_n)$, $n\geq 2$, 
for some integer $\ell$. Then the image of 
$\rho$ is 
\begin{itemize}
\item[$\diamond$] either finite,

\item[$\diamond$] or contained in 
$\mathrm{PGL}(3,\mathbb{C})=\mathrm{Aut}(\mathbb{P}^2_\mathbb{C})$ up to conjugacy.
\end{itemize}
\end{lem}

\begin{lem}[\cite{Deserti:IMRN}]\label{lem:pgl3cong}
Let $\rho$ be an embedding of a 
congruence subgroup $\Gamma(3,q)$ of 
$\mathrm{SL}(3,\mathbb{Z})$ into 
$\mathrm{Bir}(\mathbb{P}^2_\mathbb{C})$.
If $\rho(\mathrm{e}_{12}^{qn})$, 
$\rho(\mathrm{e}_{13}^{qn})$ and 
$\rho(\mathrm{e}_{23}^{qn})$ belong, 
for some integer $n$, to 
$\mathrm{PGL}(3,\mathbb{C})=\mathrm{Aut}(\mathbb{P}^2_\mathbb{C})$, 
then $\rho(\Gamma(3,q^2n^2))$ is a 
subgroup of 
$\mathrm{PGL}(3,\mathbb{C})=\mathrm{Aut}(\mathbb{P}^2_\mathbb{C})$.
\end{lem}

To establish this statement we will need
the two following results; the first one
was obtained by Cantat and 
Lamy when they study the 
embeddings of lattices from simple 
Lie groups into the group 
of polynomial automorphisms 
$\mathrm{Aut}(\mathbb{A}^2_\mathbb{C})$
whereas the second one is a technical
one. 

\begin{thm}[\cite{CantatLamyaut}]\label{thm:CantatLamyaut}
Let $\mathrm{G}$ be a simple real 
Lie group. Let $\Gamma$ 
be a lattice of $\mathrm{G}$. If 
there exists an embedding of 
$\Gamma$ into 
$\mathrm{Aut}(\mathbb{A}^2_\mathbb{C})$, 
then $\mathrm{G}$ is isomorphic to 
either $\mathrm{PSO}(1,n)$ or 
$\mathrm{PSU}(1,n)$ for some integer $n$.
\end{thm}

\begin{lem}[\cite{Deserti:IMRN}]\label{lem:tecn}
Let $\phi$ be an element of the plane
Cremona group. Assume that 
$\mathrm{Exc}(\phi)$ and $\mathrm{Exc}(\phi^2)$
are non-empty and contained in the line at 
infinity. If $\mathrm{Ind}(\phi)$ is also
contained in the line at infinity, then $\phi$
is a polynomial automorphism of 
$\mathbb{A}^2_\mathbb{C}$.
\end{lem}

\begin{proof}[Proof of Lemma \ref{lem:pgl3cong}]
Lemma \ref{lem:heispgl3} allows us to assume that 
\begin{align*}
& \rho(\mathrm{e}_{13}^{qn})\colon(z_0,z_1)\mapsto(z_0+qn,z_1), && \rho(\mathrm{e}_{12}^{qn})\colon(z_0,z_1)\mapsto (z_0+\zeta z_1,z_1+\beta),\\
& \rho(\mathrm{e}_{23}^{qn})\colon(z_0,z_1)\mapsto
(z_0+\gamma z_1,z_1+\delta)&&
\end{align*}
where $\zeta\delta-\beta\gamma=q^2n^2$.
\begin{itemize}
\item[$\diamond$] Let us first suppose that 
$\beta\delta\not=0$. Since 
$[\rho(\mathrm{e}_{13}^{qn}),\rho(\mathrm{e}_{21}^{qn})]=\rho(\mathrm{e}_{23}^{-q^2n^2})$ 
the curves blown down by $\rho(\mathrm{e}_{21}^{qn})$, 
if they exist, are of the type $z_1=$ constant. 
As $\rho(\mathrm{e}_{21}^{qn})$ and 
$\rho(\mathrm{e}_{23}^{qn})$ commute, the sets 
$\mathrm{Exc}(\rho(\mathrm{e}_{21}^{qn}))$ and
$\mathrm{Ind}(\rho(\mathrm{e}_{21}^{qn}))$
are invariant by $\rho(\mathrm{e}_{23}^{qn})$.
As a result 
$\mathrm{Exc}(\rho(\mathrm{e}_{21}^{qn}))$,
$\mathrm{Ind}(\rho(\mathrm{e}_{21}^{qn}))$
and
$\mathrm{Exc}((\rho(\mathrm{e}_{21}^{qn}))^2)$
are contained in the line at infinity. Hence 
$\rho(\mathrm{e}_{21}^{qn})$ belongs to either
$\mathrm{PGL}(3,\mathbb{C})$ or
$\mathrm{Aut}(\mathbb{A}^2_\mathbb{C})$ 
(Lemma \ref{lem:tecn}). Note that if 
$\rho(\mathrm{e}_{21}^{qn})$ belongs to 
$\mathrm{PGL}(3,\mathbb{C})$, then 
$\rho(\mathrm{e}_{21}^{qn})$ preserves 
the line at infinity because 
$[\rho(\mathrm{e}_{21}^{qn}),\rho(\mathrm{e}_{23}^{qn})]=\mathrm{id}$. 
In other words $\rho(\mathrm{e}_{21}^{qn})$ 
also belongs to 
$\mathrm{Aut}(\mathbb{A}^2_\mathbb{C})$. 
Using the relations $[\rho(\mathrm{e}_{13}^{qn}),\rho(\mathrm{e}_{32}^{qn})]=\rho(\mathrm{e}_{12}^{q^2n^2})$ 
and 
$[\rho(\mathrm{e}_{12}^{qn}),\rho(\mathrm{e}_{32}^{qn})]=\mathrm{id}$ 
we get that $\rho(\mathrm{e}_{23}^{qn})$ belongs to 
$\mathrm{Aut}(\mathbb{A}^2_\mathbb{C})$. Finally
any $\rho(\mathrm{e}_{ij}^{q^2n^2})$ is a 
polynomial automorphism of $\mathbb{A}^2_\mathbb{C}$
and $\rho$ is not an embedding (Theorem 
\ref{thm:CantatLamyaut}).

\item[$\diamond$] Assume that $\beta\delta=0$.
Since $\zeta\delta-\beta\gamma=q^2n^2$ one has 
$(\beta,\delta)\not=(0,0)$. 

Suppose that $\beta=0$. The conjugacy by 
\[
(z_0,z_1)\mapsto\left(z_0+\frac{\gamma}{2}z_1-\frac{\gamma}{2\delta}z_1^2,z_1\right)
\]
does change neither $\rho(\mathrm{e}_{13}^{qn})$,
nor $\rho(\mathrm{e}_{12}^{qn})$, and sends 
$\rho(\mathrm{e}_{23}^{qn})$ onto 
$(z_0,z_1)\mapsto(z_0,z_1+\delta)$. One can 
thus assume that 
\begin{align*}
& \rho(\mathrm{e}_{13}^{qn})\colon(z_0,z_1)\mapsto(z_0+qn,z_1), && \rho(\mathrm{e}_{12}^{qn})\colon(z_0,z_1)\mapsto (z_0+\zeta z_1,z_1)\\
& \rho(\mathrm{e}_{23}^{qn})\colon(z_0,z_1)\mapsto
(z_0,z_1+\delta).&&
\end{align*}
The map $\rho(\mathrm{e}_{21}^{qn})$ satisfies
the relations 
$[\rho(\mathrm{e}_{13}^{qn}),\rho(\mathrm{e}_{21}^{qn})]=\rho(\mathrm{e}_{23}^{-q^2n^2})$, 
and $[\rho(\mathrm{e}_{21}^{qn}),\rho(\mathrm{e}_{23}^{qn})]=\mathrm{id}$ 
so does the element 
$\psi\colon(z_0,z_1)\mapsto(z_0,\delta nz_0+z_1)$ 
of $\mathrm{PGL}(3,\mathbb{C})$. Remark that the 
map $\phi=\rho(\mathrm{e}_{21}^{qn})\circ\psi^{-1}$
commute to both $\rho(\mathrm{e}_{13}^{qn})$
and $\rho(\mathrm{e}_{23}^{qn})$. As a consequence
\[
\phi\colon(z_0,z_1)\mapsto(z_0+a,z_1+b)
\]
for some $a$, $b$ in $\mathbb{C}$. Finally up 
to conjugacy by 
$(z_0,z_1)\mapsto\left(z_0+\frac{b}{\delta},z_1\right)$
one has 
\[
\rho(\mathrm{e}_{21}^{qn})\colon(z_0,z_1)\mapsto(z_0+a,\delta z_0+z_1);
\]
in particular $\rho(\mathrm{e}_{21}^{qn})$ belongs 
to $\mathrm{PGL}(3,\mathbb{C})$. Similarly if 
$\varphi$ is the map given by 
\[
(z_0,z_1)\mapsto\left(\frac{z_0}{1+\zeta z_1},\frac{z_1}{1+\zeta z_1}\right)
\]
then the map 
$\rho(\mathrm{e}_{32}^{qn})\circ\varphi^{-1}$ 
commute to both $\rho(\mathrm{e}_{13}^{qn})$
and $\rho(\mathrm{e}_{12}^{qn})$. Therefore
\[
\rho(\mathrm{e}_{32}^{qn})\circ\varphi^{-1}\colon(z_0,z_1)\mapsto(z_0+b(z_1),z_1)
\]
and 
\[
\rho(\mathrm{e}_{32}^{qn})\colon(z_0,z_1)\mapsto\left(\frac{z_0}{1+\zeta z_1}+b\left(\frac{z_1}{1+\zeta z_1}\right),\frac{z_1}{1+\zeta z_1}\right).
\]
Thanks to $[\rho(\mathrm{e}_{23}^{qn}),\rho(\mathrm{e}_{31}^{qn})]=\rho(\mathrm{e}_{21}^{q^2n^2})$, $[\rho(\mathrm{e}_{21}^{qn}),\rho(\mathrm{e}_{31}^{qn})]=\mathrm{id}$ and $[\rho(\mathrm{e}_{12}^{qn}),\rho(\mathrm{e}_{31}^{qn})]=\rho(\mathrm{e}_{32}^{-q^2n^2})$ 
we get 
$\rho(\mathrm{e}_{21}^{qn})\colon(z_0,z_1)\mapsto(z_0,\delta z_0+z_1)$.
Finally since $\rho(\mathrm{e}_{31}^{qn})$
and $\rho(\mathrm{e}_{32}^{qn})$ commute,
$b\equiv 0$ and 
$\mathrm{im}\,\rho\subset\mathrm{PGL}(3,\mathbb{C})$. 

Assume that $\delta=0$;
using a similar reasoning we get a
contradiction.
\end{itemize}
\end{proof}

\begin{proof}[Proof of Theorem \ref{thm:IMRN}]
Any $\rho(\mathrm{e}_{ij})$ is virtually isotopic
to the identity (Lemma \ref{lem:nilkeylemma} and 
Proposition \ref{pro:2fleurs}). The maps 
$\rho(\mathrm{e}_{12}^n)$, $\rho(\mathrm{e}_{13}^n)$
and $\rho(\mathrm{e}_{23}^n)$ are, for some 
integer $n$, conjugate to automorphisms of a 
minimal rational surface (Proposition \ref{pro:sim}
and Remark \ref{rem:heis}). Up to conjugacy 
one can assume that 
$\rho(\Gamma(3,n^2))\subset\mathrm{PGL}(3,\mathbb{C})$
(Lemmas \ref{lem:notp1p1}, \ref{lem:hirzcong} and 
\ref{lem:pgl3cong}). The restriction 
$\rho_{\vert\Gamma(3,n^2)}$ of $\rho$ to
$\Gamma(3,n^2)$ can be extended to an endomorphism
 of $\mathrm{PGL}(3,\mathbb{C})$
(\emph{see} \cite{Steinberg}). But 
$\mathrm{PGL}(3,\mathbb{C})$ is simple, so 
this extension is both injective and surjective. 
The automorphisms of $\mathrm{PGL}(3,\mathbb{C})$
are obtained from inner automorphisms, automorphisms
of the field $\mathbb{C}$ and the involution 
$u\mapsto u^\vee$ (\emph{see}
\cite[Chapter IV]{Dieudonne}). But automorphisms
of the field $\mathbb{C}$ do not act on 
$\Gamma(3,n^2)$; hence up to linear conjugacy
$\rho_{\vert\Gamma(3,n^2)}$ coincides with the 
identity or the involution $u\mapsto u^\vee$.

Let $\phi$ be an element of 
$\rho(\mathrm{SL}(3,\mathbb{Z}))\smallsetminus\rho(\Gamma(3,n^2))$
that blows down at least one curve $\mathcal{C}$. 
The group $\Gamma(3,n^2)$ is a normal subgroup of
$\Gamma$. As a consequence $\mathcal{C}$ is 
invariant by $\rho(\Gamma(3,n^2))$, and so by 
$\overline{\rho(\Gamma(3,n^2))}=\mathrm{PGL}(3,\mathbb{C})$
which is impossible. Finally $\phi$ does not blow
down any curve, and 
$\rho(\mathrm{SL}(3,\mathbb{Z}))\subset\mathrm{PGL}(3,\mathbb{C})$.
\end{proof}

\begin{proof}[Proof of Corollary \ref{cor:IMRN}]
\begin{itemize}
\item[$\diamond$] Let $\Gamma$ be a subgroup of 
finite index of $\mathrm{SL}(4,\mathbb{Z})$, and
let $\rho$ be a morphism from $\Gamma$ into the 
plane Cremona group. We will
prove that $\mathrm{im}\,\rho$ is finite. To 
simplify let us suppose that 
$\Gamma=\mathrm{SL}(4,\mathbb{Z})$. Denote by 
$\mathrm{e}_{ij}$ the standard generators of 
$\mathrm{SL}(4,\mathbb{Z})$. The morphism 
$\rho$ induces a faithful representation 
$\widetilde{\rho}$ from $\mathrm{SL}(3,\mathbb{Z})$
into $\mathrm{Bir}(\mathbb{P}^2_\mathbb{C})$:
\[
\mathrm{SL}(4,\mathbb{Z})\supset\left(
\begin{array}{cc}
\mathrm{SL}(3,\mathbb{Z})& 0\\
0 & 1
\end{array}
\right)\stackrel{\widetilde{\rho}}{\to}
\mathrm{Bir}(\mathbb{P}^2_\mathbb{C})
\]
According to Theorem \ref{thm:IMRN} the map
$\widetilde{\rho}$ is, up to birational 
conjugacy, the identity or the involution 
$u\mapsto u^\vee$.

Let us first assume that up to birational 
conjugacy $\widetilde{\rho}=\mathrm{id}$.
Assume that 
$\mathrm{Exc}(\rho(\mathrm{e}_{34}))\not=~\emptyset$.
Since 
$[\mathrm{e}_{34},\mathrm{e}_{31}]=[\mathrm{e}_{34},\mathrm{e}_{32}]=\mathrm{id}$ 
the map $\rho(\mathrm{e}_{34})$ commutes with
\[
(z_0,z_1,z_2)\mapsto(z_0,z_1,az_0+bz_1+z_2)
\]
where $a$, $b\in\mathbb{C}$ and 
$\mathrm{Exc}(\rho(\mathrm{e}_{34}))$ is 
invariant by 
$(z_0,z_1,z_2)\mapsto(z_0,z_1,az_0+bz_1+z_2)$.
Moreover $\mathrm{e}_{34}$ commutes with 
$\mathrm{e}_{12}$ and $\mathrm{e}_{21}$, 
in other words $\mathrm{e}_{34}$ commutes 
with the following copy of 
$\mathrm{SL}(2,\mathbb{Z})$
\[
\mathrm{SL}(4,\mathbb{Z})\supset\left(
\begin{array}{ccc}
\mathrm{SL}(2,\mathbb{Z})& 0 & 0\\
0 & 1 & 0 \\
0 & 0 & 1
\end{array}
\right)
\]
The action of $\mathrm{SL}(2,\mathbb{Z})$
on $\mathbb{C}^2$ has no invariant curve, 
so $\mathrm{Exc}(\rho(\mathrm{e}_{34}))$ 
is contained in the line at infinity. 
But the image of this line by 
$(z_0,z_1,z_2)\mapsto(z_0,z_1,az_0+bz_1+z_2)$
intersects~$\mathbb{C}^2$: contradiction. Hence
$\mathrm{Exc}(\rho(\mathrm{e}_{34}))=\emptyset$
and $\rho(\mathrm{e}_{34})$ belongs to 
$\mathrm{PGL}(3,\mathbb{C})$. Similarly we
get that $\rho(\mathrm{e}_{43})$ belongs
to $\mathrm{PGL}(3,\mathbb{C})$. The relations 
satisfied by the standard generators thus 
imply that $\rho(\mathrm{SL}(4,\mathbb{Z}))$
is contained in $\mathrm{PGL}(3,\mathbb{C})$.
As a consequence $\mathrm{im}\,\rho$ is
finite.

A similar idea allows to conclude when 
$\widetilde{\rho}$ is, up to conjugacy, 
the involution $u\mapsto u^\vee$.

\item[$\diamond$] Let $n\geq 4$ be an 
integer. Consider a subgroup of finite
index $\Gamma$ of $\mathrm{SL}(n,\mathbb{Z})$.
Let $\rho$ be a morphism from $\Gamma$
to $\mathrm{Bir}(\mathbb{P}^2_\mathbb{C})$.
According to Theorem \ref{thm:structsl}
the group $\Gamma$ contains a congruence
subgroup $\Gamma(n,q)$. The morphism $\rho$
induces a representation $\widetilde{\rho}$
from $\Gamma(4,q)$ to 
$\mathrm{Bir}(\mathbb{P}^2_\mathbb{C})$.
As we just see the kernel of this 
representation is infinite so does 
$\ker\rho$. 
\end{itemize}
\end{proof}

\section{The group $\mathrm{Bir}(\mathbb{P}^2_\mathbb{C})$ is hopfian}

Let $V$ be a projective variety defined over a field $\Bbbk\subset\mathbb{C}$.
The group $\mathrm{Aut}_\Bbbk(\mathbb{C})$ of automorphisms of the field 
extension $\faktor{\mathbb{C}}{\Bbbk}$ acts on $V(\mathbb{C})$, and on 
$\mathrm{Bir}(V)$ as follows
\begin{equation}\label{eq:aut}
{}^{\kappa}\!\,\psi(p)=(\kappa\circ\psi\circ\kappa^{-1})(p)
\end{equation}
for any $\kappa\in\mathrm{Aut}_\Bbbk(\mathbb{C})$, any $\psi\in\mathrm{Bir}(V)$, 
and any point $p\in V(\mathbb{C})$ for which both sides of (\ref{eq:aut})
are well defined. As a consequence $\mathrm{Aut}_\Bbbk(\mathbb{C})$ acts 
by automorphisms on $\mathrm{Bir}(V)$. If $\kappa\colon\mathbb{C}\to\mathbb{C}$
is a field morphism, then this construction gives an injective morphism
\begin{align*}
&\mathrm{Aut}(\mathbb{P}^n_\mathbb{C})\to\mathrm{Aut}(\mathbb{P}^n_\mathbb{C}),
&& g\mapsto {}^{\kappa}\!\,g.
\end{align*}
Write $\mathbb{C}$ as the algebraic closure of a purely transcendental 
extension $\mathbb{Q}(x_i,\,i\in I)$ of $\mathbb{Q}$; if $f\colon I\to I$ 
is an injective map, then there exists a field morphism
\begin{align*}
& \kappa\colon\mathbb{C}\to\mathbb{C}, && x_i\mapsto x_{f(i)}.
\end{align*}
Such a morphism is surjective if and only if $f$ is onto.

The group $\mathrm{Aut}(\mathrm{Bir}(\mathbb{P}^2_\mathbb{C}))$ has been 
described in \cite{Deserti:abelien} and \cite{Deserti:IMRN} via two 
different me\-thods:

\begin{thm}[\cite{Deserti:abelien, Deserti:IMRN}]
Let $\varphi$ be an element of $\mathrm{Aut}(\mathrm{Bir}(\mathbb{P}^2_\mathbb{C}))$. 
Then there exist a birational self map $\psi$ of 
$\mathbb{P}^2_\mathbb{C}$
and an automorphism $\kappa$ of the field $\mathbb{C}$ 
such that 
\[
\varphi(\phi)={}^{\kappa}\!\,(\psi\circ\phi\circ\psi^{-1})\qquad\qquad \forall\,\phi\in\mathrm{Bir}(\mathbb{P}^2_\mathbb{C})
\]
\end{thm}

The proof of \cite{Deserti:abelien} will be deal with in 
\S\ref{sec:autbir}.
The proof of \cite{Deserti:IMRN} can in fact be used to describe the endomorphisms of the plane Cremona group:

\begin{thm}[\cite{Deserti:hopfian}]\label{thm:hopfian}
Let $\varphi$ be a non-trivial endomorphism of  $\mathrm{Bir}(\mathbb{P}^2_\mathbb{C})$. 
Then there exist~$\psi$ in $\mathrm{Bir}(\mathbb{P}^2_\mathbb{C})$
and an immersion $\kappa$ of the field $\mathbb{C}$ such that 
\[
\varphi(\phi)={}^{\kappa}\!\,(\psi\circ\phi\circ\psi^{-1})\qquad\qquad \forall\,\phi\in\mathrm{Bir}(\mathbb{P}^2_\mathbb{C})
\]
\end{thm}

Let us work in the affine chart $z_2=1$. The group of translations is 
\[
\mathrm{T}=\big\{(z_0,z_1)\mapsto(z_0+\alpha,z_1+\beta)\,\vert\,\alpha,\,\beta\in\mathbb{C}\big\}.
\]

\begin{lem}[\cite{Deserti:hopfian}]\label{lem:com}
Let $\varphi$ be a birational self map of $\mathbb{P}^2_\mathbb{C}$. 
Assume that $\varphi$ commutes with both $(z_0,z_1)\mapsto(z_0+1,z_1)$
and $(z_0,z_1)\mapsto(z_0,z_1+1)$. 

Then $\varphi$ belongs to $\mathrm{T}$.
\end{lem}

\begin{proof}
Let $\varphi=(\varphi_0,\varphi_1)$ be an element of 
$\mathrm{Bir}(\mathbb{P}^2_\mathbb{C})$ that commutes with both
$(z_0,z_1)\mapsto(z_0+1,z_1)$ and $(z_0,z_1)\mapsto(z_0,z_1+1)$. 
In particular 
\[
\left\{
\begin{array}{ll}
\varphi_0(z_0+1,z_1)=\varphi_0(z_0,z_1)+1\\
\varphi_1(z_0+1,z_1)=\varphi_1(z_0,z_1)
\end{array}
\right.
\]
From $\varphi_1(z_0+1,z_1)=\varphi_1(z_0,z_1)$ we get that $\varphi_1=\varphi_1(z_1)$. 
The equality $\varphi_0(z_0+1,z_1)=\varphi_0(z_0,z_1)+1$ implies
\[
\frac{\partial\varphi_0}{\partial z_0}(z_0+1,z_1)=\frac{\partial\varphi_0}{\partial z_0}(z_0,z_1);
\]
as a consequence $\frac{\partial\varphi_0}{\partial z_0}=a(z_1)$ 
and $\varphi_0=a(z_1)z_0+b(z_1)$ for some $a$, $b$ in $\mathbb{C}(z_1)$.
Then 
\[
\varphi_0(z_0+1,z_1)=\varphi_0(z_0,z_1)+1
\]
yields $a(z_1)=1$.
In other words $\varphi\colon(z_0,z_1)\dashrightarrow(z_0+b(z_1),\varphi_1(z_1))$. 

Let us now write that $\varphi\circ(z_0,z_1+1)\colon(z_0,z_1)\dashrightarrow(z_0,z_1+1)\circ\varphi$;
we get that $\varphi\colon(z_0,z_1)\dashrightarrow(\varphi_0(z_0),z_1+c(z_0))$.

Finally 
$\varphi\colon(z_0,z_1)\dashrightarrow(z_0+b(z_1),\varphi_1(z_1))$ 
and 
$\varphi\colon(z_0,z_1)\dashrightarrow(\varphi_0(z_0),z_1+c(z_0))$
imply that $\varphi$ belongs to $\mathrm{T}$.
\end{proof}

\begin{proof}[Proof of Theorem \ref{thm:hopfian}]
Since $\mathrm{PGL}(3,\mathbb{C})$ is simple the restriction 
$\varphi_{\vert\mathrm{PGL}(3,\mathbb{C})}$ is either trivial
or injective.

\smallskip

Let us first suppose that $\varphi_{\vert\mathrm{PGL}(3,\mathbb{C})}$ is 
trivial. Consider the element of $\mathrm{PGL}(3,\mathbb{C})$ given by
\[
\ell\colon(z_0,z_1)\mapsto\left(\frac{z_0}{z_0-1},\frac{z_0-z_1}{z_0-1}\right).
\]
According to \cite{Gizatullin:relations} one has $(\ell\circ\sigma_2)^3=\mathrm{id}$.

As a result $\varphi((\ell\circ\sigma_2)^3)=\mathrm{id}$. Since $\varphi(\ell)=\ell$ 
(recall that $\ell$ belongs to $\mathrm{PGL}(3,\mathbb{C})$) one gets that
$\varphi(\sigma_2)=\mathrm{id}$. As the plane Cremona group is
generated by $\mathrm{PGL}(3,\mathbb{C})$ and $\sigma_2$ one gets that
$\varphi=\mathrm{id}$.

\smallskip

Assume now that $\varphi_{\vert\mathrm{PGL}(3,\mathbb{C})}$ is injective. 
According to Theorem \ref{thm:IMRN} the restriction 
$\varphi_{\vert\mathrm{SL}(3,\mathbb{Z})}$ of $\varphi$ to 
$\mathrm{SL}(3,\mathbb{Z})$ is, up to inner conjugacy, the canonical 
embedding or $A\mapsto A^\vee$. 

\begin{itemize}
\item[$\diamond$] Suppose first that $\varphi_{\vert\mathrm{SL}(3,\mathbb{Z})}$
is the canonical embedding. Denote by $\mathcal{U}$ the group of unipotent
upper triangular matrices. Set
\begin{align*}
 & f_\beta=\varphi(z_0+\beta,z_1), && g_\alpha=\varphi(z_0+\alpha,z_1), && h_\gamma=\varphi(z_0,z_1+\gamma). 
\end{align*}
Since $f_\beta$ and $h_\gamma$ commute to both $(z_0,z_1)\mapsto(z_0+1,z_1)$
and $(z_0,z_1)\mapsto(z_0,z_1+1)$ one gets from Lemma \ref{lem:com} that 
\begin{align*}
& f_\beta\colon(z_0,z_1)\mapsto(z_0+\lambda(\beta),z_1+\zeta(\beta)) && h_\gamma\colon(z_0,z_1)\mapsto(z_0+\eta(\gamma),z_1+\mu(\gamma))
\end{align*}
where $\lambda$, $\zeta$, $\eta$ and $\mu$ are additive morphisms from 
$\mathbb{C}$ to $\mathbb{C}$. As $g_\gamma$ commutes with 
$(z_0,z_1)\mapsto(z_0+z_1,z_1)$ and $(z_0,z_1)\mapsto(z_0+1,z_1)$ there 
exists $a_\alpha$ in $\mathbb{C}(y)$ such that 
\[
g_\gamma\colon(z_0,z_1)\mapsto(z_0+a_\alpha(z_1),z_1).
\]
The equality
\[
(z_0+\alpha z_1,z_1)\circ(z_0,z_1+\gamma)\circ(z_0+\alpha z_1,z_1)^{-1}\circ(z_0,z_1+\gamma)^{-1}=(z_0+\alpha z_1,z_1)
\]
implies that $g_\alpha\circ h_\alpha=f_{\alpha\gamma}\circ h_\gamma\circ g_\alpha$ for 
any $\alpha$, $\gamma$ in $\mathbb{C}$. As a consequence 
\begin{align*}
& f_\beta\colon(z_0,z_1)\mapsto(z_0+\lambda(\beta),z_1) && g_\alpha\colon(z_0,z_1)\mapsto(z_0+\theta(\alpha)z_1+\zeta(\alpha),z_1)
\end{align*}
and $\theta(\alpha)\mu(\alpha)=\lambda(\alpha\gamma)$. From 
\[
\Big[\big((z_0,z_1)\mapsto(z_0+\alpha,z_1)\big),\big((z_0,z_1)\mapsto(z_0,z_1+\beta z_0)\big)\Big]=\big((z_0,z_1)\mapsto(z_0,z_1-\alpha)\big)
\]
one gets $h_\gamma\colon(z_0,z_1)\mapsto(z_0,z_1+\mu(\gamma))$. In other words 
for any $\alpha$, $\beta\in\mathbb{C}$ one has 
\[
\varphi(z_0+\alpha,z_1+\beta)=f_\alpha\circ h_\beta=(z_0,z_1)\mapsto(z_0+\lambda(\alpha),z_1+\mu(\beta)).
\]
Therefore, $\varphi(\mathrm{T})\subset\mathrm{T}$ and 
$\varphi(\mathcal{U})\subset\mathcal{U}$. Since 
$\mathrm{PGL}(3,\mathbb{C})=\langle\mathcal{U},\,\mathrm{SL}(3,\mathbb{Z})\rangle$ 
the inclusion $\varphi(\mathrm{PGL}(3,\mathbb{C}))\subset\mathrm{PGL}(3,\mathbb{C})$ 
holds. According to \cite{BorelTits} the action of $\varphi$ on 
$\mathrm{PGL}(3,\mathbb{C})$ comes, up to inner conjugacy, from an embedding of 
the field $\mathbb{C}$ into itself.

\item[$\diamond$] Assume now that $\varphi_{\vert\mathrm{SL}(3,\mathbb{Z})}$
is $A\mapsto A^\vee$. 
Similar computations and \cite{BorelTits} imply that 
$\varphi_{\vert\mathrm{PGL}(3,\mathbb{C})}$ comes, up to inner conjugacy,
from the composition of $A\mapsto A^\vee$ and an embedding of the field 
$\mathbb{C}$ into itself.
\end{itemize}

\smallskip

To finish let us assume for instance that 
$\varphi_{\vert\mathrm{PGL}(3,\mathbb{C})}$ comes, up to inner 
conjugacy, from the composition of $A\mapsto A^\vee$ and an 
embedding of the field $\mathbb{C}$ into itself. Set 
$(\eta_1,\eta_2)=\varphi\left((z_0,z_1)\mapsto\left(z_0,\frac{1}{z_1}\right)\right)$. From 
\begin{small}
\[
\left((z_0,z_1)\dashrightarrow\left(z_0,\frac{1}{z_1}\right)\right)\circ((z_0,z_1)\mapsto(\alpha z_0,\beta z_1))\circ\left((z_0,z_1)\dashrightarrow\left(z_0,\frac{1}{z_1}\right)\right)=\left((z_0,z_1)\mapsto\left(\alpha z_0,\frac{z_1}{\beta}\right)\right)
\]
\end{small}
one gets
\[
\left\{
\begin{array}{ll}
\eta_1\big(\lambda(\alpha^{-1})z_0,\lambda(\beta^{-1})z_1\big)=\lambda(\alpha^{-1})\eta_1(z_0,z_1)\\
\eta_2\big(\lambda(\alpha^{-1})z_0,\lambda(\beta^{-1})z_1\big)=\lambda(\beta)\eta_2(z_0,z_1)
\end{array}
\right.
\]
Hence 
$\varphi\left((z_0,z_1)\mapsto\left(z_0,\frac{1}{z_1}\right)\right)=\left((z_0,z_1)\mapsto\left(\pm z_0,\pm\frac{1}{z_1}\right)\right)$. 
But 
\[
\left(\big((z_0,z_1)\mapsto(z_1,z_0)\big)\circ\left((z_0,z_1)\mapsto\left(z_0,\frac{1}{z_1}\right)\right)\right)^2=\sigma_2, 
\]
so 
$\varphi(\sigma_2)=\pm\sigma_2$. Furthermore $\varphi(\ell)=\big((z_0,z_1)\mapsto(-z_0-z_1-1,z_1)\big)$ 
as $\varphi_{\vert\mathrm{SL}(3,\mathbb{Z})}$ coincides with $A\mapsto A^\vee$. 
Then the second component of $\varphi\big(\ell\circ\sigma_2\big)^3$ is 
$\pm\frac{1}{z_1}$: contradiction with 
$\varphi\big(\ell\circ\sigma_2\big)^3=\mathrm{id}$.

If $\varphi_{\vert\mathrm{PGL}(3,\mathbb{C})}$ comes, up to inner conjugacy, from 
an embedding of $\mathbb{C}$ similar computations imply that 
$\varphi(\sigma_2)=\sigma_2$ and one concludes with Noether and Castelnuovo theorem.
\end{proof}


\chapter{Finite subgroups of the Cremona group}
\label{chapter:finite}

\bigskip
\bigskip

The classification of finite subgroups of 
$\mathrm{Bir}(\mathbb{P}^1_\mathbb{C})=\mathrm{PGL}(2,\mathbb{C})$
is well known and goes back to Klein. It consists of cyclic, 
dihedral, tetrahedral, octahedral and icosahedral groups. 
Groups of the same type and same order constitute a unique 
conjugacy class in $\mathrm{Bir}(\mathbb{P}^1_\mathbb{C})$.

What about the two-dimensional case, {\it i.e.} what about the 
finite subgroups of $\mathrm{Bir}(\mathbb{P}^2_\mathbb{C})$ ?
The story starts a long time ago with Bertini (\cite{Bertini})
who classified conjugacy classes of subgroups of order $2$ 
in $\mathrm{Bir}(\mathbb{P}^2_\mathbb{C})$. Already the answer
is drastically different from the one-dimensional case. The 
set of conjugacy classes is parameterized (\emph{see} 
Theorem \ref{thm:BayleBeauville}) by a disconnected algebraic
variety whose connected components are respectively isomorphic
to 
\begin{itemize}
\item[$\diamond$] either the moduli spaces of hyperelliptic 
curves of genus $g$,

\item[$\diamond$] or the moduli space of canonical curves of 
genus $3$,

\item[$\diamond$] or the moduli space of canonical curves of
genus $4$ with vanishing theta characteristic.
\end{itemize}

Bertini's proof is considered to be incomplete; a complete
and short proof was published only a few years ago by 
Bayle and Beauville (\cite{BayleBeauville}). 

In $1894$ Castelnuovo proved that any element of 
$\mathrm{Bir}(\mathbb{P}^2_\mathbb{C})$ of finite order
leaves inva\-riant either a net of lines, or a pencil of
lines, or a linear system of cubic curves with $n\leq 8$
base-points (\cite{Castelnuovo}). Kantor announced a 
similar result for arbitrary finite subgroups of 
$\mathrm{Bir}(\mathbb{P}^2_\mathbb{C})$; his proof 
relies on a classification of possible groups in each 
case (\cite{Kantor}). Unfortunately Kantor's classification, 
even with some corrections made by Wiman (\cite{Wiman}),
is incomplete in the following sense:
\begin{itemize}
\item[$\diamond$] given some abstract finite group, it is
not possible using their list to say whether this group
is isomorphic to a subgroup of $\mathrm{Bir}(\mathbb{P}^2_\mathbb{C})$;

\item[$\diamond$] the possible conjugation between the 
groups of the list is not considered.
\end{itemize}

The Russian school has made great progress since the $1960$'s: 
Manin and Iskovskikh classified the minimal $G$-surfaces into automorphisms 
of del Pezzo surfaces and of conic bundles 
(\cite{Manin:rational2, Iskovskih:minimal}).
Many years after people come back to this problem. As 
we already mention Bayle and Beauville classified groups 
of order $2$. It is the first example of a precise 
description of conjugacy classes; it is shown that the 
non-rational curves fixed by the groups determine the 
conjugacy classes. Groups of prime order were also 
studied (\cite{BeauvilleBlanc, deFernex, Zhang}). 
Zhang applies Bayle and Beauville strategy to the 
case of birational automorphisms of prime order
$p\geq 3$. It turns out that nonlinear automorphisms 
occur only for $p=3$ and $p=5$; the author describes 
them explicitly. The techniques of \cite{BayleBeauville} 
are also generalized by 
de Fernex to cyclic subgroups of prime order 
(\cite{deFernex}). The list is as precise as one can 
wish, except for two classes of groups of order $5$: 
the question of their conjugacy is not answered. 
Beauville and Blanc completed this classification 
(\cite{BeauvilleBlanc}); they prove in particular 
that a birational self map of the complex 
projective plane of prime order is not conjugate to a linear automorphism if and only if
it fixes some non-rational curve. Beauville 
classified $p$-elementary groups (\cite{Beauville}).
Blanc classified all finite cyclic groups (\cite{Blanc:CRAS}), 
and all finite abelian groups (\cite{Blanc:these}).
The goal of \cite{DolgachevIskovskikh} is to update the 
list of Kantor and Wiman. The authors used the modern theory 
of $G$-surfaces, the theory of elementary links, and 
the conjugacy classes of Weyl groups.

\smallskip

In the first section we recall the definitions of 
Geiser involutions, Bertini involutions
and Jonqui\`eres involutions. We give a 
sketch of the proof of the classification of 
birational involutions of the complex projective
plane due to Bayle and Beauville.

In the second section we deal with finite abelian 
subgroups of the plane Cremona 
group. Results due to Dolgachev and 
Iskovskikh are recalled.

In the last section we state some results of 
Blanc about finite cyclic subgroups 
of $\mathrm{Bir}(\mathbb{P}^2_\mathbb{C})$, 
isomorphism classes of finite abelian subgroups
of $\mathrm{Bir}(\mathbb{P}^2_\mathbb{C})$ 
but also a generalization of a theo\-rem of 
Castelnuovo which states that an 
element of finite order which fixes a curve 
of geometric genus $>1$ has order $2$, $3$ 
or $4$.

\bigskip
\bigskip

\section{Classification of subgroups of order $2$ of $\mathrm{Bir}(\mathbb{P}^2_\mathbb{C})$}\label{sec:ordre2}

\subsection{Geiser involutions}\label{geiser}

Let $p_1$, $p_2$, $\ldots$, $p_7$ be seven points 
of the complex projective plane in general position.
Denote by $L$ the linear system of cubics through 
the $p_i$'s. The linear system $L$ of cubic curves through the $p_i$'s is two-dimensional. 
Take a general point $p$, and 
consider the pencil of curves from $L$ passing through 
$p$. A general pencil of 
cubic curves has nine base-points; let us 
define $\mathcal{I}_G(p)$ as
the ninth base-point of the pencil. 
The map $\mathcal{I}_G$ is a 
\textsl{Geiser involution}\index{defi}{Geiser
involution}\index{not}{$\mathcal{I}_G$} 
(\cite{Geiser}). The algebraic degree of a 
Geiser involution is equal to 
$8$.

One can also see a Geiser involution as follows. 
The linear system $L$ defines a rational map of degree $2$,
\[
\psi\colon \mathbb{P}^2_\mathbb{C}\dashrightarrow\vert L\vert^*\simeq \mathbb{P}^2_\mathbb{C}.
\]
The points $p$ and $\mathcal{I}_G(p)$ lie in the 
same fibre. As a consequence $\mathcal{I}_G$ is a 
birational deck map of this cover. If we blow up 
$p_1$, $p_2$, $\ldots$, $p_7$ we get a 
del Pezzo surface $S$ of degree $2$ and a regular map
of degree $2$ from $S$ to $\mathbb{P}^2_\mathbb{C}$.
Furthermore the Geiser involution becomes an automorphism
of $S$.

Note that the fixed points of $\mathcal{I}_G$ lie
on the ramification curve of $\psi$. It is a curve
of degree~$6$ with double points $p_1$, $p_2$, 
$\ldots$, $p_7$ and is birationally isomorphic to a 
canonical curve of genus~$3$.

A third way to see Geiser involutions is the following.
Let $S$ be a del Pezzo surface of degree $2$. The 
linear system $\vert -K_S\vert$ defines a double 
covering $S\to\mathbb{P}^2_\mathbb{C}$, branched 
along a smooth quartic curve (\cite{Demazure:delPezzo}).
The involution $\iota$ which exchanges the two sheets
of this covering is called a Geiser involution; it 
satisfies
\[
\mathrm{Pic}(S)^\iota\otimes\mathbb{Q}\simeq\mathrm{Pic}(\mathbb{P}^2_\mathbb{C})\otimes\mathbb{Q}=\mathbb{Q}.
\]

The exceptional locus of a Geiser involution is the union 
of seven cubics passing through the seven points of
indeterminacy of $\mathcal{I}_G$ and singular at one of 
these seven points.

\subsection{Bertini involutions}\label{bertini}

Let us fix in $\mathbb{P}^2_\mathbb{C}$ eight 
points $p_1$, $p_2$, $\ldots$, $p_8$ in general 
position. Consider the pencil of cubic curves 
through these points. It has a ninth base-point
$p_9$. For any general point $p$ there is a 
unique cubic curve $\mathcal{C}(p)$ of the 
pencil passing through $p$. Take $p_9$ as the 
zero of the group law of the cubic 
$\mathcal{C}(p)$; define $\mathcal{I}_B(p)$ 
as the negative $-p$ with respect to the group
law. The map $\mathcal{I}_B$ is a birational
involution called \textsl{Bertini involution}\index{defi}{Bertini
involution}\index{not}{$\mathcal{I}_B$} 
(\cite{Bertini}).

The algebraic degree of a Bertini involution 
is equal to $17$. The fixed points of a Bertini
involution lie on a canonical curve of genus $4$
with vanishing theta characteristic isomorphic to
a nonsingular intersection of a cubic surface and
a quadratic cone in $\mathbb{P}^3_\mathbb{C}$.

Another way to see a Bertini involution is the following.
Consider a del Pezzo surface~$S$ of degree $1$. 
The map $S\to\mathbb{P}^3_\mathbb{C}$ defined by
the linear system $\vert -2K_S\vert$ induces 
a degree $2$ morphism of $S$ onto a quadratic
cone $Q\subset\mathbb{P}^3_\mathbb{C}$, branched 
along the vertex of $Q$ and a smooth genus $4$ 
curve (\cite{Demazure:delPezzo}). The corresponding 
involution, the Bertini involution, satisfies 
$\mathrm{rk}\,\mathrm{Pic}(S)^{\mathcal{I}_B}=1$.

\subsection{Jonqui\`eres involutions}

Let $\mathcal{C}$ be an irreducible curve of degree 
$\nu\geq~3$. Assume that $\mathcal{C}$ has a unique
singular point $p$ and that $p$ is an ordinary multiple
point with multiplicity $\nu-2$. To $(\mathcal{C},p)$
we associate a birational involution $\mathcal{I}_J$ 
that fixes pointwise $\mathcal{C}$ and preserves lines
through $p$. Let $m$ be a generic point of 
$\mathbb{P}^2_\mathbb{C}\smallsetminus\mathcal{C}$. 
Let $r_m$, $q_m$ and $p$ be the intersections of the 
line $(mp)$ and $\mathcal{C}$. The point 
$\mathcal{I}_J(m)$ is the point such that the cross 
ratio of $m$, $\mathcal{I}_J(m)$, $q_m$ and $r_m$
is equal to $-1$. The map $\mathcal{I}_J$ is a 
\textsl{Jonqui\`eres involution}
\index{defi}{Jonqui\`eres involution} of degree $\nu$ centered 
at $p$; it preserves $\mathcal{C}$. More precisely
its fixed points are the curve $\mathcal{C}$ of 
genus $\nu-2$ as soon as $\nu\geq 3$.

If $\nu=2$, then $\mathcal{C}$ is a smooth conic ; 
the same construction can be done by choosing a 
point~$p$ that does not lie on $\mathcal{C}$.

\begin{lem}[\cite{DolgachevIskovskikh}]\label{lem:notconj}
Let $\mathrm{G}$ be a finite subgroup of 
$\mathrm{Bir}(\mathbb{P}^2_\mathbb{C})$. 
Let $C_1$, $C_2$, $\ldots$, $C_k$ be  
non-rational irreducible curves on 
$\mathbb{P}^2_\mathbb{C}$ such that each
of them contains an open subset $C_i^0$ 
whose points are fixed under all 
$g\in\mathrm{G}$. 

Then the set of birational isomorphism classes
of the curves $C_i$ is an invariant of the 
conjugacy class of $\mathrm{G}$ in 
$\mathrm{Bir}(\mathbb{P}^2_\mathbb{C})$.
\end{lem}

\begin{proof}
Assume that 
$\mathrm{G}=\psi\circ\mathrm{H}\circ\psi^{-1}$
for some subgroup $\mathrm{H}$ of 
$\mathrm{Bir}(\mathbb{P}^2_\mathbb{C})$ and 
some birational self map $\psi$ of the complex 
projective plane. Replacing $C_i^0$ by a 
smaller open subset if needed we assume that
$\psi^{-1}(C_i^0)$ is defined and consists
of fixed points of $\mathrm{H}$. As $C_i$ is 
not rational, $\psi^{-1}(C_i^0)$ is not a 
point. Its Zariski closure is thus a rational 
irreducible curve $C'_i$ birationally isomorphic
to $C_i$ that contains an open subset of 
fixed points of $\mathrm{H}$.
\end{proof}

\begin{cor}
Jonqui\`eres involutions of degree $\geq 3$ 
are not conjugate to each other, not conjugate 
to projective involutions, not conjugate to 
Bertini involutions, not conjugate to Geiser 
involutions. 

Bertini involutions are not conjugate to Geiser
involutions, not conjugate to projective 
involutions. 

Geiser involutions are not conjugate to 
projective involutions.
\end{cor}

\begin{proof}
The statement follows from Lemma \ref{lem:notconj}
and the above properties:
\begin{itemize}
\item[$\diamond$] a connected component of the 
fixed locus of a projective map is a line or a
point;

\item[$\diamond$] the fixed points of a Geiser
involution lie on a curve birationally
isomorphic to a cano\-nical curve of genus $3$;

\item[$\diamond$] the fixed points of a Bertini
involution lie on a canonical curve of genus $4$
with va\-nishing theta characteristic;

\item[$\diamond$] the set of fixed points of a 
Jonqui\`eres involution of degree $\nu\geq 3$
outside the base locus is an hyperelliptic 
curve of degree $\nu-2$.
\end{itemize}
\end{proof}

We can thus introduce the following definition.

\begin{defi}
An involution is of 
\textsl{Jonqui\`eres type}\index{defi}{Jonqui\`eres type}
if it is birationally conjugate to a Jonqui\`eres involution.

An involution is of 
\textsl{Bertini type}\index{defi}{Bertini type}
if it is birationally conjugate to a Bertini involution.

An involution is of 
\textsl{Geiser type}\index{defi}{Geiser type}
if it is birationally conjugate to a Geiser involution.
\end{defi}

\smallskip

The classification of subgroups of 
$\mathrm{Bir}(\mathbb{P}^2_\mathbb{C})$ of order $2$ is given 
by the following statement:

\begin{thm}[\cite{BayleBeauville}]\label{thm:BayleBeauville}
A non-trivial birational involution of the complex projective
plane is conjugate to one and only one of the following:
\begin{itemize}
\item[$\diamond$] a Jonqui\`eres involution of a given 
degree $\nu\geq 2$;

\item[$\diamond$] a Geiser involution;

\item[$\diamond$] a Bertini involution.
\end{itemize}
\end{thm}

More precisely the parameterization of each conjugacy
class is known. Before stating it let us give some
definitions.

\begin{rems}\label{rem:norm}
Let $S$, $S'$ be two rational surfaces and 
$\iota\in\mathrm{Bir}(S)$, 
$\iota'\in\mathrm{Bir}(S')$ be two involutions.
They are 
\textsl{birationally equivalent}\index{defi}{birationally equivalent (maps)}
if there exists a birational map 
$\varphi\colon S\dashrightarrow S'$ such that 
$\varphi\circ\iota=\iota'\circ\varphi$. Note that
in particular two involutions of 
$\mathrm{Bir}(\mathbb{P}^2_\mathbb{C})$ are 
equivalent if and only if they are conjugate
in $\mathrm{Bir}(\mathbb{P}^2_\mathbb{C})$. Assume 
that $\iota$ fixes a curve $C$. Then 
$\iota'=\varphi\circ\iota\circ\varphi^{-1}$ 
fixes the proper transform of $C$ under $\varphi$ 
which is a curve birational to $C$ except possibly 
if $C$ is rational; indeed, if $C$ is rational it 
may be contracted to a point. The 
\textsl{normalized fixed curve}\index{defi}{normalized fixed curve}
of $\iota$ is the union of the normalizations of 
the non-rational curves fixed by $\iota$. This is
an invariant of the birational equivalence class
of $\iota$.
\end{rems}

\begin{pro}[\cite{BayleBeauville}]
The map which associates to a birational involution
of $\mathbb{P}^2_\mathbb{C}$ its normalized fixed
curve establishes a one-to-one correspondence between
\begin{itemize}
\item[$\diamond$] conjugacy classes of Jonqui\`eres
involutions of degree $\nu$ and isomorphism classes 
of hyperelliptic curves of genus $\nu-2$ $(\nu\geq 3)$;

\item[$\diamond$] conjugacy classes of Geiser 
involutions and isomorphism classes of non-hyperelliptic
curves of genus $3$;

\item[$\diamond$] conjugacy classes of Bertini involutions
and isomorphism classes of non-hyperelliptic curves 
of genus $4$ whose canonical model lies on a singular
quadric.
\end{itemize}
Jonqui\`eres involutions of degree $2$ form 
one conjugacy class.
\end{pro}

The approach of Bayle and Beauville is different from 
the approach of Castelnuovo. It is based on the 
following observation: any birational involution of 
$\mathbb{P}^2_\mathbb{C}$ is conjugate, via an 
appropriate birational isomorphism 
$S\stackrel{\sim}{\dashrightarrow}\mathbb{P}^2_\mathbb{C}$
to a biregular involution $\iota$ of a rational 
surface~$S$. Therefore, the authors are reduced to 
the birational classification of the pairs 
$(S,\iota)$. In \cite{Manin:rational2} Manin classified 
the pairs $(S,\mathrm{G})$ where $S$ is a surface and 
$\mathrm{G}$ a finite group. This question has been 
simplified by the introduction of Mori theory. This 
theory allows Bayle and Beauville to show that the 
minimal pairs $(S,\iota)$ fall into two categories, 
those which admit a $\iota$-invariant base-point 
free pencil of rational curves, and those with 
$\mathrm{rk}\,\mathrm{Pic}(S)^{\iota}=1$. The 
first case leads to the so-called Jonqui\`eres
involutions whereas the second one leads to the 
Geiser and Bertini involutions.

\medskip

Let us now give some details. By a surface we mean
a smooth, projective, connected surface over 
$\mathbb{C}$. We consider pairs $(S,\iota)$ where 
$S$ is a rational surface and $\iota$ a non-trivial
biregular involution of $S$. Recall that the pair 
$(S,\iota)$ is minimal if 
any birational morphism $\psi\colon S\to S'$ such
that there exists a biregular involution~$\iota'$
of $S'$ with $\psi\circ\iota=\iota'\circ\psi$ is
an isomorphism.

\begin{lem}[\cite{BayleBeauville}]
The pair $(S,\iota)$ is minimal if and only if
for any exceptional curve\footnote{Recall that
an exceptional curve $E$ on a surface $S$ is 
a smooth rational curve with $E^2=-1$.} $E$ on 
$S$ the following hold:
\begin{align*}
&\iota(E)\not=E &&E\cap\iota(E)\not=\emptyset.
\end{align*}
\end{lem}

\begin{proof}
Suppose that $(S,\iota)$ is not minimal. Then
there exist a pair $(S',\iota')$ and a birational
morphism $\psi\colon S\to S'$ such that 
$\psi\circ\iota=\iota'\circ\psi$ and $\psi$ 
contracts some exceptional curve $E$. Then~$\psi$
contracts the divisor $E+\iota(E)$. Therefore,
$(E+\iota(E))^2\leq 0$, and so $E\cdot\iota(E)\leq 0$,
{\it i.e.} $\iota(E)=E$ or $E\cap\iota(E)=\emptyset$.

\smallskip

Conversely assume that there exists an exceptional
curve $E$ on $S$ such that $\iota(E)=E$
(resp. $E\cap\iota(E)=\emptyset$). Let $S'$ be the 
surface obtained by blowing down~$E$ (resp. 
$E\cup\iota(E)$). Then $\iota$ induces an involution
$\iota'$ of $S'$ so that $(S,\iota)$ is not minimal.
\end{proof}

The only piece of Mori theory used by Bayle and 
Beauville is the following one:

\begin{lem}[\cite{BayleBeauville}]\label{lem:pencil}
Let $(S,\iota)$ be a minimal pair with 
$\mathrm{rk}\,\mathrm{Pic}(S)^\iota>1$. Then 
$S$ admits a base-point free pencil stable under 
$\iota$.
\end{lem}

It allows them to establish the:

\begin{thm}[\cite{BayleBeauville}]\label{thm:minimalpairs}
Let $(S,\iota)$ be a minimal pair. One of the 
following holds:
\begin{enumerate}
\item[(1)] there exists a smooth $\mathbb{P}^1_\mathbb{C}$-fibration
$f\colon S\to\mathbb{P}^1_\mathbb{C}$ and a non-trivial 
involution $\mathcal{I}$ of $\mathbb{P}^1_\mathbb{C}$ such 
that $f\circ\iota=\mathcal{I}\circ f$;

\item[(2)] there exists a fibration 
$f\colon S\to\mathbb{P}^1_\mathbb{C}$ such that 
$f\circ\iota=f$, the smooth fibres of $f$ are rational curves 
on which $\iota$ induces a non-trivial involution, any 
singular fibre is the union of two rational curves 
exchanged by $\iota$, meeting at one point;

\item[(3)] $S$ is isomorphic to $\mathbb{P}^2_\mathbb{C}$;

\item[(4)] $(S,\iota)$ is isomorphic to 
$\mathbb{P}^1_\mathbb{C}\times\mathbb{P}^1_\mathbb{C}$ with 
the involution $(z_0,z_1)\mapsto(z_1,z_0)$;

\item[(5)] $S$ is a del Pezzo surface of degree $2$ and 
$\iota$ is the Geiser involution;

\item[(6)] $S$ is a del Pezzo surface of degree $1$ and 
$\iota$ is the Bertini involution.
\end{enumerate}
\end{thm}

\begin{proof}
\begin{itemize}
\item[$\diamond$] Assume $\mathrm{rk}\,\mathrm{Pic}(S)^\iota=1$.
As $\mathrm{Pic}(S)^\iota$ contains an ample class, $-K_S$ 
is ample, {\it i.e.} $S$ is a del Pezzo surface. If 
$\mathrm{rk}\,\mathrm{Pic}(S)=1$, then one obtains case $(3)$.

If $\mathrm{rk}\,\mathrm{Pic}(S)>1$, then $-\iota$ is the 
orthogonal reflection with respect to $K_S^\perp$. Such a 
reflection is of the form 
\[
x\mapsto x-2\,\frac{(\alpha\cdot x)}{(\alpha\cdot\alpha)}\,\alpha
\]
with $(\alpha\cdot\alpha)\in\{1,\,2\}$ and $K_S$ proportional
to $\alpha$. If $K_S$ is divisible, then $S$ is isomorphic to 
$\mathbb{P}^1_\mathbb{C}\times\mathbb{P}^1_\mathbb{C}$ and
since $\iota$ must act non-trivially on $\mathrm{Pic}(S)$
we get case $(4)$. The only remaining eventualities are 
$K_S^2\in\{1,\,2\}$. The  Geiser and Bertini involutions have
the required properties (\S\ref{geiser}, \S\ref{bertini}).
An automorphism $\varphi$ of $S$ acting trivially on $\mathrm{Pic}(S)$
is the identity; indeed $S$ is the blow up of 
$\mathbb{P}^2_\mathbb{C}$ at $9-d$ points in general 
position, $\varphi$ induces an automorphism of 
$\mathbb{P}^2_\mathbb{C}$ which must fix these points. 
Hence Geiser and Bertini involutions are the only ones
to have the required properties.

\item[$\diamond$] Suppose now that 
$\mathrm{rk}\,\mathrm{Pic}(S)^\iota>1$. According to 
Lemma \ref{lem:pencil} the surface $S$ admits a 
$\iota$-invariant pencil $\vert F\vert$ of rational 
curves. This defines a fibration 
$f\colon S\to\mathbb{P}^1_\mathbb{C}$ with fibre $F$, 
and an involution $\mathcal{I}$ of $\mathbb{P}^1_\mathbb{C}$
such that $f\circ\iota=\mathcal{I}\circ f$.

If $f$ is smooth, then this gives $(1)$ or a particular
case of $(2)$.

If $f$ is not smooth, let $F_0$ be a singular fibre of $f$. 
It contains an exceptional divisor $E$. Since $(S,\iota)$
is minimal, then $\iota(E)\not=E$ and $E\cdot\iota(E)\geq 1$.
As a result $(E+\iota(E))^2\geq 0$, 
so $F_0=E+\iota(E)$ and $E\cdot\iota(E)=1$. Set 
$p=E\cap\iota(E)$. The involution induced by $\iota$ on 
$T_pS$ exchanges the directions of $E$ and $\iota(E)$; it
thus has eigenvalues $1$ and $-1$. As a consequence 
$\iota$ fixes a curve passing through $p$; this curve
must be horizontal and $\mathcal{I}$ trivial. Furthermore
the fixed curve of $\iota$ being smooth, the involution 
induced by $\iota$ on a smooth fibre cannot be trivial.
We get case $(2)$.
\end{itemize}
\end{proof}

Bayle and Beauville precised which pairs in the list of 
Theorem \ref{thm:minimalpairs} are indeed minimal 
(\cite[Proposition 1.7]{BayleBeauville}).

\smallskip

Let us now give the link between biregular involutions of 
rational surfaces and birational involutions of the 
complex projective plane:

\begin{lem}[\cite{BayleBeauville}]\label{lem:linkbirinvbirinv}
Let $\iota$ be a birational involution of a surface $S_1$. 
There exists a birational morphism $\varphi\colon S\to S_1$
and a biregular involution $\mathcal{I}$ of $S$ such that
$\varphi\circ\mathcal{I}=\iota\circ\varphi$.
\end{lem}

To prove it we need some results, let us state and prove
them.

\begin{thm}[\emph{see for instance} \cite{Beauville:book}, Theorem II.7]\label{thm:morph}
Let $S$ be a surface, and let $X$ be a projective variety.
Let $\phi\colon S\dashrightarrow X$ be a rational map. 

Then there exist
\begin{itemize}
\item[$\diamond$] a surface $S'$, 

\item[$\diamond$] a morphism $\eta\colon S'\to S$ which is 
the composition of a finite number of blow-ups,

\item[$\diamond$] a morphism $\psi\colon S'\to X$
\end{itemize}
such that 
\[
 \xymatrix{
     & S'\ar[rd]^\psi\ar[ld]_\eta & \\
    S\ar@{-->}[rr]_\phi & & X
  }
\]
commutes.
\end{thm}

\begin{proof}
As $X$ lies in some projective space we may assume that 
$X=\mathbb{P}^m_\mathbb{C}$. Furthermore we can suppose
that $\phi(S)$ lies in no hypersurface of 
$\mathbb{P}^m_\mathbb{C}$. As a result $\phi$ corresponds
to a linear system $P\subset\vert D\vert$ of dimension 
$m$ on $S$ without fixed component. 

If $P$ has no base-point, then $\phi$ is a morphism and 
there is nothing to do.

Assume that $P$ has at least one base-point $p$. Consider
the blow up $\varepsilon\colon \mathrm{Bl}_pS\to S$ at 
$p$. Set $S_1=\mathrm{Bl}_pS$. The exceptional curve $E$
occurs in the fixed part of the linear system 
$\varepsilon^*P\subset\vert\varepsilon^*D\vert$ with 
some multiplicity $k\geq 1$; that is, the system 
$P_1=\vert\varepsilon^*P-kE\vert\subset\vert\varepsilon^*D-kE\vert$
has no fixed component. It thus defines a rational map
$\phi_1=\phi\circ\varepsilon\colon S_1\dashrightarrow\mathbb{P}^m_\mathbb{C}$.
If $\phi_1$ is a morphism, then the result is proved. 
If not, we repeat the "same step". We get by induction
a sequence $\varepsilon_n\colon S_n\to S_{n-1}$ of blow
ups and a linear system 
$P_n\subset\vert D_n\vert=\vert\varepsilon_n^*D_{n-1}-k_nE_n\vert$
on $S_n$ with no fixed part. On the one hand 
$D_n^2=D_{n-1}^2-k_n^2<D_{n-1}^2$; on the other hand 
$P_n$ has no fixed part, so $D_n^2\geq 0$ for any $n$. 
Consequently the process must end. More precisely after
a finite number of blow ups we obtain a system $P_n$ 
with no base-points which defines a morphism 
$\psi\colon S_n\to\mathbb{P}^m_\mathbb{C}$ as required.
\end{proof}

\begin{lem}[\emph{see for instance} \cite{Beauville:book}]\label{lem:preim}
Let $S$ be an irreducible surface. Let~$S'$ be a smooth surface.  Let 
$\phi\colon S\to S'$ be a birational morphism.
Assume that the rational map $\phi^{-1}$ is not defined at a point 
$p\in S'$. 

Then $\phi^{-1}(p)$ is a curve on $S$.
\end{lem}

\begin{proof}
We assume that $S$ is affine so that there is an embedding 
$j\colon S\hookrightarrow\mathbb{A}^n_\mathbb{C}$. The 
rational map 
\[
j\circ \phi^{-1}\colon S'\dashrightarrow \mathbb{A}^n_\mathbb{C}
\]
is defined by rational functions $g_1$, $g_2$, $\ldots$ $g_n$.
One of them, say for instance $g_1$ is undefined at~$p$, that is 
$g_1\not\in\mathcal{O}_{S',p}$. Set $g_1=\frac{u}{v}$ with 
$u$, $v\in\mathcal{O}_{S',p}$, $u$ and $v$ coprime and
$v(p)=0$. Consider the curve $D$ on $S$ given by $\phi^*v=0$. 
On $S\subset\mathbb{A}^n_\mathbb{C}$ denote by $z_0$ the first 
coordinate function. We have $\phi^*u=z_0\phi^*v$ on $S$.
Hence $\phi^*u=\phi^*v=0$ on $D$. Consequently $D=\phi^{-1}(Z)$
where 
\[
Z=\big\{u=v=0\big\}\subset S'.
\]
By assumption $u$ and $v$ are coprime, so $Z$ is finite.
Shrinking $S'$ if necessary we can assume that $Z=\{p\}$. 
Finally $D=\phi^{-1}(p)$.
\end{proof}

\begin{lem}[\emph{see for instance} \cite{Beauville:book}]\label{lem:preim2}
Let $S$, $S'$ be two surfaces. Let $\phi\colon S\dashrightarrow S'$ 
be a birational map such that $\phi^{-1}$ is not defined at $p\in S'$.

Then there exists a curve $C$ on $S$ such that $\phi(C)=\{p\}$.
\end{lem}

\begin{proof}
The map $\phi$ corresponds to a morphism $\psi\colon\mathcal{U}\to S'$
for some subset $\mathcal{U}$ of $S$. Denote by 
\[
\Gamma=\big\{\big(u,\psi(u)\big)\,\vert\,u\in\mathcal{U}\big\}\subset\mathcal{U}\times S'
\]
the graph of $\psi$. Let $\overline{\Gamma}$ be the closure of 
$\Gamma$ in $S\times S'$; it is an irreducible surface, 
possibly with singularities. The projections 
\begin{align*}
&\mathrm{pr}_1\colon\overline{\Gamma}\to S, &&
\mathrm{pr}_2\colon\overline{\Gamma}\to S'
\end{align*}
are birational morphisms and the diagram
\[
 \xymatrix{
     & \overline{\Gamma}\ar[ld]_{\mathrm{pr}_1}\ar[rd]^{\mathrm{pr}_2} & \\
    S\ar@{-->}[rr]_\phi & & S'
  }
\]
is commutative. 

By assumption $\phi^{-1}$ is not defined at $p\in S'$, so does 
$\mathrm{pr}_2^{-1}$. There is an irreducible curve $C'$ on
$\overline{\Gamma}$ such that $\mathrm{pr}_2(C')=\{p\}$ 
(Lemma \ref{lem:preim}). As 
$\overline{\Gamma}\subset S\times S'$ the image 
$\mathrm{pr}_1(C')$ of $C'$ by $\mathrm{pr}_1$ is a curve 
$C$ in $S$ such that $\phi(C)=\{p\}$.
\end{proof}

\begin{pro}[\emph{see for instance} \cite{Lamy:jung}]\label{pro:universalpropertyofblowingup}
Let $X$ and $S$ be two surfaces. Let $\phi\colon X\to S$ 
be a birational morphism of surfaces. Suppose that 
the rational map $\phi^{-1}$ is not defined at a point 
$p$ of $S$.

Then
\[
 \xymatrix{
     & \mathrm{Bl}_pS\ar[rd]^\varepsilon & \\
    X\ar[ur]^\psi\ar[rr]_\phi & & S
  }
\]
where $\psi\colon X\to\mathrm{Bl}_pS$ is a birational
map and $\varepsilon\colon\mathrm{Bl}_pS\to S$ is 
the blow up at $p$.
\end{pro}

\begin{proof}
Set $\psi=\varepsilon^{-1}\circ\phi$. Suppose that $\psi$ is not 
a morphism, and let $m$ be a point of $X$ such that $\psi$ is
not defined at $m$. On the one hand $\phi(m)=p$ and $\phi$ is 
not locally invertible at $m$; on the other hand there
exists a curve in $\mathrm{Bl}_pS$ blown down onto $m$
by $\psi^{-1}$ (Lemma \ref{lem:preim2}). This curve has 
to be the exceptional divisor~$E$ associated to 
$\varepsilon$. Let $r$ and $q$ be two distinct points 
of $E$ at which $\psi^{-1}$ is well defined; consider
$C$, $C'$ two germs of smooth curves transverse to 
$E$ at $r$ and $q$ respectively. Then $\varepsilon(C)$
and $\varepsilon(C')$ are two germs of smooth curves
transverse at $p$, which are images by $\phi$ of two 
germs of curves at $m$. The differential of $\phi$ at $m$
has thus rank $2$: contradiction with the fact that 
$\phi$ is not invertible at $m$.

\smallskip

\begin{figure}[!h]
\centering
\includegraphics[height=5cm]{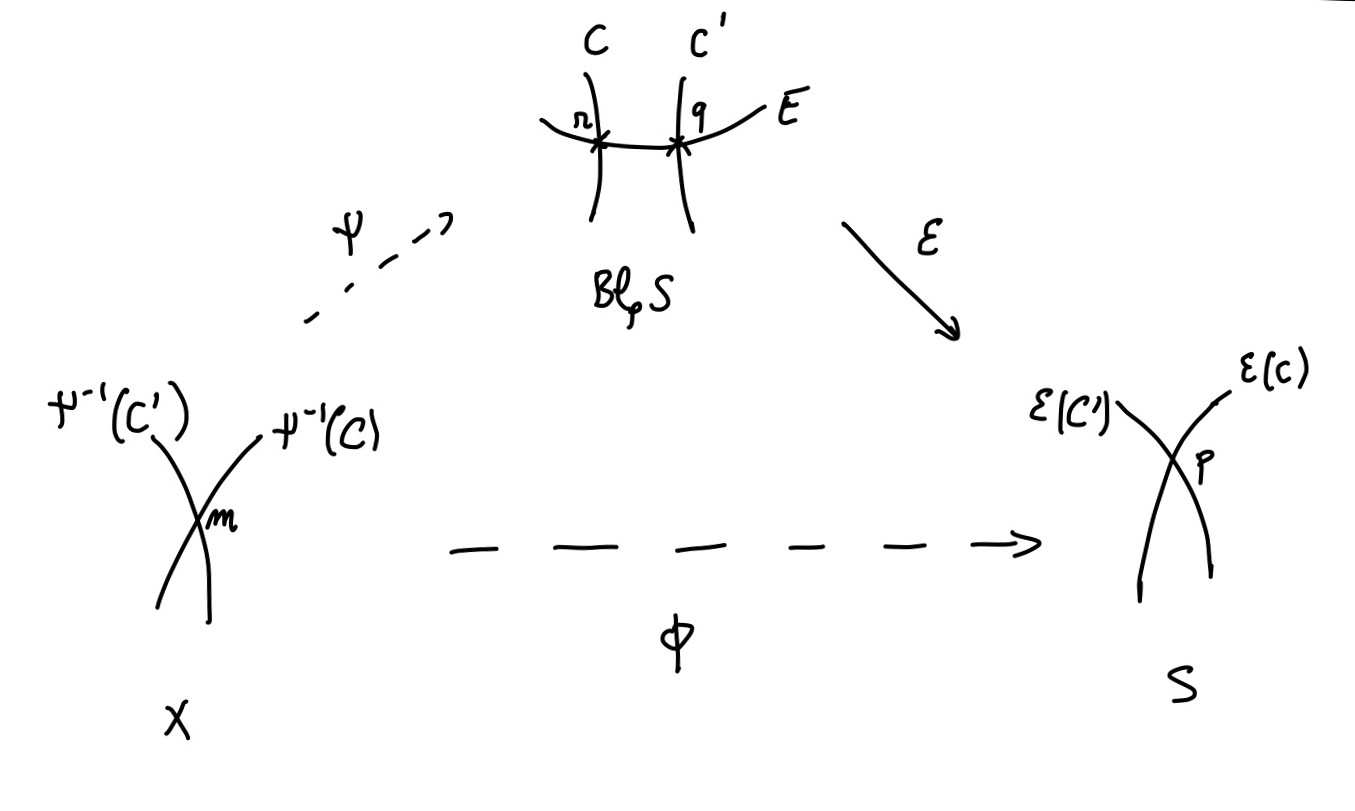}
\end{figure}
\end{proof}

\begin{proof}[Proof of Lemma \ref{lem:linkbirinvbirinv}]
There exists a birational morphism $\varphi\colon S\to S_1$
such that the rational map $\psi=\iota\circ\varphi$ is 
everywhere defined (Theorem \ref{thm:morph}). 
Furthermore~$\varphi$ can be written as 
\[
\varphi=\varepsilon_{n-1}\circ\varepsilon_{n-2}\circ\ldots\circ\varepsilon_1
\]
where $\varepsilon_i\colon S_{i+1}\to S_i$, $1\leq i\leq n-1$, 
is obtained by blowing up a point $p_i\in S_i$ and $S=S_n$. 
The map $\iota$ is not defined at $p_1$, so  
$\psi^{-1}=\varphi^{-1}\circ\iota$ is not defined at $p_1$. 
Proposition \ref{pro:universalpropertyofblowingup} implies 
that~$\psi$ factors as 
  \[
  \xymatrix{
     & S_2\ar[rd]^{\varepsilon_1} & \\
    S\ar[rr]_\psi\ar[ur]^{g_1} & & S_1
  }
  \]
Proceeding by induction we see that $\psi$ factors as
$\varphi\circ\mathcal{I}$ where $\mathcal{I}$ is a 
birational morphism. Since 
$\varphi\circ\mathcal{I}^2=\varphi$, the map 
$\mathcal{I}$ is an involution.
\end{proof}

In other words Lemma \ref{lem:linkbirinvbirinv} says that
any birational involution of a surface is birationally 
equivalent to a 
biregular involution $\iota\colon S\to S$; furthermore
$(S,\iota)$ can be assumed to be minimal. Therefore, the 
classification of conjugacy classes of involutions in 
$\mathrm{Bir}(\mathbb{P}^2_\mathbb{C})$ is equivalent to
the classification of minimal pairs $(S,\iota)$ up to
birational equivalence.

\begin{rem}
Recall that the $\mathbb{P}^1_\mathbb{C}$-bundles over
$\mathbb{P}^1_\mathbb{C}$ are of the form
\[
\mathbb{F}_n=\mathbb{P}_{\mathbb{P}^1_\mathbb{C}}\big(\mathcal{O}_{\mathbb{P}^1_\mathbb{C}}\oplus\mathcal{O}_{\mathbb{P}^1_\mathbb{C}}(n)\big)
\]
for some integers $n\geq 0$ (\emph{see} \S \ref{subsec:hirz}).

For $n\geq 1$ the fibration
\[
f\colon \mathbb{F}_n\to\mathbb{P}^1_\mathbb{C}
\]
has a unique section of self-intersection $-n$.
Consider a fibre $F$ of $f$, and a point $p$ of $F$.
Assume that $\iota$ is a birational involution of 
$\mathbb{F}_n$ regular in a neighborhood of $F$
and fixing $p$. After the elementary transformation
at $p$ we get a birational involution of 
$\mathbb{F}_{n+1}$ regular in a neighborhood of 
the new fibre.
\end{rem}

\begin{proof}[Proof of Theorem \ref{thm:BayleBeauville}]
The unicity assertion follows from Remark \ref{rem:norm}.

Using Lemma \ref{lem:linkbirinvbirinv} we will prove that 
the involutions of Theorem \ref{thm:minimalpairs} are 
birationally equivalent to one of Theorem 
\ref{thm:BayleBeauville}.

\smallskip

Cases $(5)$ and $(6)$ give by definition the Geiser and
Bertini involutions. 

\smallskip

An involution of type $(4)$ is birationally equivalent 
to a Jonqui\`eres involution of degree $2$. Indeed 
let $Q$ be a smooth conic in $\mathbb{P}^2_\mathbb{C}$, 
and let $p\in\mathbb{P}^2_\mathbb{C}\smallsetminus Q$ be a 
point. Consider the birational involution $\iota$ of
$\mathbb{P}^2_\mathbb{C}$ that maps a point $x$ 
to its harmonic conjugate on the line $(px)$ through 
$p$ and $x$ with respect to the two points of 
$(px)\cap Q$. This involution is not defined at the 
following three points: $p$ and the two points 
$q$ and $r$ where the tangent line to $Q$ passes 
through $p$. Set 
$S=\mathrm{Bl}_{p,q,r}\mathbb{P}^2_\mathbb{C}$. The 
involution $\iota$ extends to a biregular involution
$\mathcal{I}$ of $S$, the Jonqui\`eres involution
of degree $2$.

\smallskip

In case $(3)$ take a point $p\in\mathbb{P}^2_\mathbb{C}$
such that $\iota(p)\not=p$. Let us blow up $p$, $\iota(p)$
and then blow down the proper transform of the line 
$(p\iota(p))$ which is a $\iota$-invariant exceptional 
curve. We get a pair $(T,\iota')$ with 
$T\simeq\mathbb{P}^1_\mathbb{C}\times\mathbb{P}^1_\mathbb{C}$
by stereographic projection and 
$\mathrm{rk}\,\mathrm{Pic}(T)^{\iota'}=1$: we are thus 
in case $(4)$, so in the case of a Jonqui\`eres 
involution of degree $2$.

\smallskip

Let us now deal with case $(1)$. The surface $S$ is 
isomorphic to $\mathbb{F}_n$ for some $n\geq 0$. The 
involution $\iota$ has two invariant fibres, any of 
them containing at least two fixed points. One of 
these points does not belong to $s_n$ (section of 
self-intersection $-n$ on $\mathbb{F}_n$), hence 
after a (finite) sequence of elementary transformations
we get $n=1$. Let us thus focus on the case 
$n=1$. Let $\mathbb{F}_1$ be the surface obtained
by blowing up a point $p\in\mathbb{P}^2_\mathbb{C}$. 
Projecting from $p$ defines a $\mathbb{P}^1$-bundle
$f\colon\mathbb{F}_1\to\mathbb{P}^1_\mathbb{C}$.
Any biregular involution $\iota$ of $\mathbb{F}_1$
preserves this fibration hence defines a pair 
$(\mathbb{F}_1,\iota)$ of case $(1)$ or $(2)$. The 
involution $\iota$ preserves the unique exceptional
curve $E_1$ of~$\mathbb{F}_1$; the pair 
$(\mathbb{F}_1,\iota)$ is thus not minimal: 
$\iota$ induces a biregular involution of 
$\mathbb{P}^2_\mathbb{C}$. We finally get a 
Jonqui\`eres involution of degree $2$ as we just
see.

\smallskip

We now consider case $(2)$. Let us distinguish
two possibilities: denote by $F_1$, $F_2$, $\ldots$,
$F_s$ the singular fibres of $f$ and by $p_i$, 
$1\leq i\leq s$, the singular point of $F_i$. The 
fixed locus of $\iota$ is a smooth curve $C$ 
passing through $p_1$, $p_2$, $\ldots$, $p_s$. 
The degree $2$ covering $C\to\mathbb{P}^1_\mathbb{C}$
induced by $f$ is ramified at $p_1$, $p_2$, $\ldots$,
$p_s$. 
\begin{itemize}
\item[(2a)] Either $f$ is smooth, $s=0$ and $C$ is the
union of two sections of $f$ which do not intersect;
 
\item[(2b)] or $f$ is not smooth, $C$ is a hyperelliptic
curve of genus $g\geq 0$ and $s=2g+2$.
\end{itemize}

First assume that we are in case $(2a)$. After elementary
transformations we can suppose that $S=\mathbb{F}_1$.
The fixed locus of $\iota$ is the union of $E_1$ and a 
section which does not meet $E_1$. Blowing down $E_1$ 
one gets case $(4)$.

Finally let us look at case $(2b)$ for $g\geq 0$. Let 
us blow down one of the components in each singular 
fibre. We thus have a birational involution on a 
surface $\mathbb{F}_n$, the fixed curve $C$ embedded
into $\mathbb{F}_n$. After elementary transformations
at general points of $C$ one gets a birational 
involution on a surface $\mathbb{F}_1$, the fixed curve
$C$ embedded into $\mathbb{F}_1$. The genus formula 
implies that $E_1\cdot C=g$. Suppose that $C$ is 
tangent to $E_1$ at some point $q\in\mathbb{F}_1$. 
After an elementary transformation at~$q$ then
 an elementary transformation at some 
general point of $C$ the order of contact of $C$
and~$E_1$ at $q$ decreases by $1$. Proceeding in 
this way we arrive at the following situation: $E_1$
and $C$ meet transversally at $g$ distinct points. 
Let blow down $E_1$ to a point $p$ of 
$\mathbb{P}^2_\mathbb{C}$; the curve $C$ maps 
to a plane curve $\overline{C}$ of degree $g+2$ with an 
ordinary multiple point of multiplicity $g$ at $p$ 
and no other singularity. This yields to a birational 
involution of $\mathbb{P}^2_\mathbb{C}$ which 
preserves the lines through~$p$ and admits 
$\overline{C}$ as fixed
curve, {\it i.e.} a Jonqui\`eres involution 
with center $p$ and fixed curve~$\overline{C}$.
\end{proof}

\section{Finite abelian subgroups of the Cremona group}

Dolgachev and Iskovskikh used a modern approach to the 
problem initiated in the works of Manin and Iskovskikh who
gave a clear understanding of the conjugacy problem via the 
concept of a $\mathrm{G}$-surface 
(\cite{Manin:rational2, Iskovskih:minimal}). 
Let $\mathrm{G}$ be a finite group. A 
\textsl{$\mathrm{G}$-surface}\index{defi}{$\mathrm{G}$-surface}
is a pair $(S,\psi)$ where $S$ is a nonsingular projective
surface and $\psi$ is an isomorphism from $\mathrm{G}$
to $\mathrm{Aut}(S)$. A 
\textsl{morphism of the pairs}\index{defi}{morphism (of $\mathrm{G}$-surfaces)}
$(S,\psi)\to(S',\psi')$ is defined to be a morphism of 
surfaces $\phi\colon S\to S'$ such that 
\[
\psi'(\mathrm{G})=\phi\circ\psi(\mathrm{G})\circ\phi^{-1}.
\]
In particular let us note that two subgroups of 
$\mathrm{Aut}(S)$ define isomorphic $\mathrm{G}$-surfaces
if and only if they are conjugate inside $\mathrm{Aut}(S)$.

Let $(S,\psi)$ be a rational $\mathrm{G}$-surface. Take 
a birational map 
$\phi\colon S\dashrightarrow\mathbb{P}^2_\mathbb{C}$.
For any $g\in\mathrm{G}$ the map $\phi\circ g\circ\phi^{-1}$
belongs to $\mathrm{Bir}(\mathbb{P}^2_\mathbb{C})$. This 
yields to an injective homomorphism
\[
\iota_\phi\colon \mathrm{G}\to\mathrm{Bir}(\mathbb{P}^2_\mathbb{C}).
\]

\begin{lem}[\cite{DolgachevIskovskikh}]
Let $(S,\psi)$ and $(S',\psi')$ be two rational 
$\mathrm{G}$-surfaces. Let 
$\phi\colon S\dashrightarrow\mathbb{P}^2_\mathbb{C}$ 
and 
$\phi'\colon S\dashrightarrow\mathbb{P}^2_\mathbb{C}$
be two birational maps.

The subgroups $\iota_\phi(\mathrm{G})$ and 
$\iota_{\phi'}(\mathrm{G})$ are conjugate if and only if 
there exists a birational map of $\mathrm{G}$-surfaces 
$S'\dashrightarrow S$.
\end{lem}

In other words a birational isomorphism class of 
$\mathrm{G}$-surfaces defines a conjugacy class
of subgroups of $\mathrm{Bir}(\mathbb{P}^2_\mathbb{C})$
isomorphic to $\mathrm{G}$. The following result 
shows that any conjugacy class is obtained in this
way:

\begin{lem}[\cite{DolgachevIskovskikh}]
Let $\mathrm{G}$ be a finite subgroup of 
$\mathrm{Bir}(\mathbb{P}^2_\mathbb{C})$. There exist
a rational $\mathrm{G}$-surface $(S,\psi)$ and a 
birational map 
$\phi\colon S\dashrightarrow\mathbb{P}^2_\mathbb{C}$
such that 
\[
\mathrm{G}=\phi\circ\psi(\mathrm{G})\circ\phi^{-1}.
\]
\end{lem}

\begin{proof}
If $\phi$ belongs to $\mathrm{G}$, we denote by 
$\mathrm{dom}(\phi)$ an open subset on which $\phi$ is defined. Set 
$\mathcal{D}=\displaystyle\bigcap_{\phi\in\mathrm{G}}\mathrm{dom}(\phi)$.
Then 
$\mathcal{U}=\displaystyle\bigcap_{\phi\in\mathrm{G}}g(\mathcal{D})$
is an open invariant subset of $\mathbb{P}^2_\mathbb{C}$ on 
which $\phi$ acts biregularly. Consider 
$\mathcal{U}'=\faktor{\mathcal{U}}{\mathrm{G}}$ the orbit space; it is
a normal algebraic surface. Let us choose any normal projective 
completion $X'$ of $\mathcal{U}'$. Consider $S'$ the 
normalization of $X'$ in the field of rational functions
of $\mathcal{U}$. It is a normal projective surface on which
$\mathrm{G}$ acts by biregular transformations. A 
$\mathrm{G}$-invariant resolution of singularities $S$ of 
$S'$ suits (\cite{DeFernexEin}).
\end{proof}

Hence one has:

\begin{thm}[\cite{DolgachevIskovskikh}]
There is a natural bijective correspondence between 
birational isomorphism classes of rational 
$\mathrm{G}$-surfaces and conjugate classes of 
subgroups of $\mathrm{Bir}(\mathbb{P}^2_\mathbb{C})$
isomorphic to $\mathrm{G}$.
\end{thm}

Therefore, the goal of Dolgachev and 
Iskovskikh is to
classify $\mathrm{G}$-surfaces up to birational
isomorphism of $\mathrm{G}$-surfaces. 

There is a $\mathrm{G}$-equivariant analogue of minimal
surfaces:

\begin{defi}
A 
\textsl{minimal $\mathrm{G}$-surface}\index{defi}{minimal ($\mathrm{G}$-surface)} 
is a $\mathrm{G}$-surface $(S,\psi)$ such that any birational
morphism of $\mathrm{G}$-surfaces 
$(S,\psi)\to(S',\psi')$ is an isomorphism.
\end{defi}

Note that it is enough to classify minimal rational 
$\mathrm{G}$-surfaces up to birational isomorphism
of $\mathrm{G}$-surfaces. The authors can rely 
on the following fundamental result:

\begin{thm}\label{thm:classificationminimalGsurfaces}
Let $S$ be a minimal rational $\mathrm{G}$-surface. Then
\begin{itemize}
\item[$\diamond$] either $S$ admits a structure of 
a conic bundle with 
$\mathrm{Pic}(S)^{\mathrm{G}}\simeq\mathbb{Z}^2$;

\item[$\diamond$] or $S$ is isomorphic to a del Pezzo
surface with $\mathrm{Pic}(S)^{\mathrm{G}}\simeq\mathbb{Z}$.
\end{itemize}
\end{thm}

An analogous result from the classical literature is 
showed by using the method of the termination of 
adjoints, first introduced for linear system of 
plane curves in the work of Castelnuovo. This method
is applied to find a $\mathrm{G}$-invariant linear 
system of curves in the plane in \cite{Kantor};
Kantor essentially stated the result above but without
the concept of minimality. A first modern proof 
can be found in \cite{Manin:rational2} and 
\cite{Iskovskih:minimal}. 
Nowadays Theorem \ref{thm:classificationminimalGsurfaces} 
follows from a $\mathrm{G}$-equivariant version of 
Mori theory (\cite{deFernex}).

As a result to complete the classification Dolgachev 
and Iskovskikh need
\begin{itemize}
\item[$(i)$] to classify all finite groups $\mathrm{G}$ 
that may occur in a minimal $\mathrm{G}$-pair;

\item[$(ii)$] to determine when two minimal 
$\mathrm{G}$-surfaces are birationally isomorphic.
\end{itemize}
 
To achieve $(i)$ the authors computed the full
automorphisms group of a conic bundle surface on a del
Pezzo surface and then made a list of all finite 
subgroups acting minimally on the surface. 

To achieve $(ii)$ the authors used the ideas of Mori theory
to decompose a birational map of rational 
$\mathrm{G}$-surfaces into elementary links.

\section{Finite cyclic subgroups of 
$\mathrm{Bir}(\mathbb{P}^2_\mathbb{C})$}

In \cite{Blanc:CRAS} the author gave
the list of finite cyclic subgroups
of the plane Cremona group, 
up to conjugation. The curves fixed
by one element of the group, and the 
action of the whole group on these curves, 
are often sufficient to distinguish the 
conjugacy classes. It was done in 
\cite{Blanc:these} in many cases, but 
some remain unsolved. In \cite{Blanc:these}
the author completed this classification
with the case of abelian non-cyclic groups.

Its classification implies several results
we will now mention.

\begin{thm}[\cite{Blanc:CRAS}]
For any integer $n\geq 1$ there are infinitely many conjugacy
classes of elements of $\mathrm{Bir}(\mathbb{P}^2_\mathbb{C})$ 
of order $2n$, that are non-conjugate to a linear automorphism.

\smallskip

If $n>15$, a birational map of $\mathbb{P}^2_\mathbb{C}$ 
of order $2n$ is a $n$-th root of a Jonqui\`eres 
involution and preserves a pencil of rational curves.

\smallskip
If an element of $\mathrm{Bir}(\mathbb{P}^2_\mathbb{C})$ 
is of finite odd order and is not conjugate to a linear 
automorphism of $\mathbb{P}^2_\mathbb{C}$, then its order 
is $3$, $5$, $9$ or $15$. In particular any birational 
map of $\mathbb{P}^2_\mathbb{C}$ of odd order $>15$ is 
conjugate to a linear automorphism of the plane.
\end{thm}

Then Blanc generalized a theorem of Castelnuovo which 
states that an element of finite order which fixes a curve 
of geometric genus $>1$ has order $2$, $3$ or $4$ 
(\emph{see}~\cite{Castelnuovo}):

\begin{thm}[\cite{Blanc:CRAS}]
Let $\mathrm{G}$ be a finite abelian group which fixes 
some curve of positive geometric genus. 

Then $\mathrm{G}$ is cyclic, of order $2$, $3$, $4$, 
$5$ or $6$, and all these cases occur.

If the curve has geometric genus $>1$, then $\mathrm{G}$ is
of order $2$ or $3$.
\end{thm}

\begin{thm}[\cite{Blanc:CRAS}]
Let $\mathrm{G}$ be a finite abelian subgroup of 
$\mathrm{Bir}(\mathbb{P}^2_\mathbb{C})$.

The following assertions are equivalent:
\begin{itemize}
\item[$\diamond$] any 
$g\in\mathrm{G}\smallsetminus\big\{\mathrm{id}\big\}$ 
does not fix a curve of positive geometric genus;

\item[$\diamond$] the group $\mathrm{G}$ is birationally 
conjugate to a subgroup of 
$\mathrm{Aut}(\mathbb{P}^2_\mathbb{C})$, or to a subgroup 
of 
$\mathrm{Aut}(\mathbb{P}^1_\mathbb{C}\times\mathbb{P}^1_\mathbb{C})$, 
or to the group isomorphic to 
$\faktor{\mathbb{Z}}{2\mathbb{Z}}\times\faktor{\mathbb{Z}}{4\mathbb{Z}}$ 
generated by the two follo\-wing elements 
\begin{align*}
& (z_0:z_1:z_2)\mapsto\big(z_1z_2:z_0z_1:-z_0z_2\big), \\
&  (z_0:z_1:z_2)\mapsto\big(z_1z_2(z_1-z_2):z_0z_2(z_1+z_2):z_0z_1(z_1+z_2)\big).
\end{align*}
\end{itemize}
Furthermore this last group is conjugate neither to a subgroup of $\mathrm{Aut}(\mathbb{P}^2_\mathbb{C})$,
nor to a subgroup of $\mathrm{Aut}(\mathbb{P}^1_\mathbb{C}\times\mathbb{P}^1_\mathbb{C})$.
\end{thm}

In \cite{Beauville} Beauville gave the isomorphism classes 
of $p$-elementary subgroups of the plane Cremona 
group. Blanc generalized it as follows:

\begin{thm}[\cite{Blanc:CRAS}]
The isomorphism classes of finite abelian subgroups of the
plane Cremona group are the following:
\begin{itemize}
\item[$\diamond$] $\faktor{\mathbb{Z}}{m\mathbb{Z}}\times\faktor{\mathbb{Z}}{n\mathbb{Z}}$ 
for any integers $m$, $n\geq 1$,

\item[$\diamond$] $\faktor{\mathbb{Z}}{2n\mathbb{Z}}\times\Big(\faktor{\mathbb{Z}}{2\mathbb{Z}}\Big)^2$ 
for any integer $n\geq 1$,

\item[$\diamond$] $\Big(\faktor{\mathbb{Z}}{4\mathbb{Z}}\Big)^2\times\faktor{\mathbb{Z}}{2\mathbb{Z}}$,

\item[$\diamond$] $\Big(\faktor{\mathbb{Z}}{3\mathbb{Z}}\Big)^3$,

\item[$\diamond$] $\Big(\faktor{\mathbb{Z}}{2\mathbb{Z}}\Big)^4$.
\end{itemize}
\end{thm}

\medskip

In \cite{Blanc:cyclic} the author finished the classification 
of cyclic subgroups of finite order
of the Cremona group, up to conjugation. He gave natural 
parameterizations of conjugacy classes, related
to fixed curves of positive genus. The classification of 
finite cyclic subgroups that are not of Jonqui\`eres 
type was almost achieved in \cite{DolgachevIskovskikh}.
Let us explain what we mean by "almost":
\begin{itemize}
\item [$\diamond$] a list of representative elements 
is available;

\item[$\diamond$] explicit forms are given;

\item[$\diamond$] the dimension of the varieties which 
parameterize the conjugacy classes are provided. 
\end{itemize}

What is missing ? A finer geometric description of the 
algebraic variety parameterizing conjugacy classes
according to \cite{DolgachevIskovskikh}.

The case of groups conjugate to subgroups of 
$\mathrm{Aut}(\mathbb{P}^2_\mathbb{C})$ was 
studied in \cite{BeauvilleBlanc}: there is 
exactly one conjugacy class for each order $n$, 
representated by 
\[
\langle(z_0:z_1:z_2)\mapsto(z_0:z_1:\mathrm{e}^{2\mathbf{i}\pi/n}z_2)\rangle.
\]

Blanc completed the classification of cyclic 
subgroups of $\mathrm{Bir}(\mathbb{P}^2_\mathbb{C})$ 
of finite order (\cite{Blanc:cyclic}). For groups of 
Jonqui\`eres type he applied cohomology group theory
and algebraic tools to the group $\mathcal{J}$ and 
got:

\begin{thm}[\cite{Blanc:cyclic}]
\begin{itemize}
\item[$\diamond$] For any positive integer $m$, there 
exists a unique conjugacy class of linearisable 
elements of order $n$, represented by the automorphism
\[
(z_0:z_1:z_2)\mapsto(z_0:z_1:\mathrm{e}^{2\mathbf{i}\pi/n}z_2).
\]
\item[$\diamond$] Any non-linearisable Jonqui\`eres 
element of finite order of 
$\mathrm{Bir}(\mathbb{P}^2_\mathbb{C})$ has order $2n$, 
for some positive integer $n$, and is conjugate to 
an element $\phi$, such that $\phi$ and $\phi^n$ are
of the following form
\[
\phi\colon(z_0,z_1)\dashrightarrow\left(\mathrm{e}^{2\mathbf{i}\pi/n}z_0,\frac{a(z_0)z_1+(-1)^\delta p(z_0^n)b(z_0)}{b(z_0)z_1+(-1)^\delta a(z_0)}\right)
\]
\[
\phi^n\colon(z_0,z_1)\dashrightarrow\left(z_0,\frac{p(z_0^n)}{z_1}\right)
\]
where $a$, $b$ belongs to $\mathbb{C}(z_0)$, $\delta$ 
to $\big\{0,\pm 1\big\}$, and $p\in\mathbb{C}[z_0]$ is 
a polynomial with simple roots.

The curve $\Gamma$ of equation $z_1^2=p(z_0^n)$, 
pointwise fixed by $\phi^n$, is hyperelliptic, of 
positive geometric genus, and admits a $(2:1)$-map
$\phi_1^2\colon\Gamma\to\mathbb{P}^1_\mathbb{C}$.
The action of $\phi$ on $\Gamma$ has order $n$, 
and is not a root of the involution associated 
to any $\phi_1^2$.

Furthermore the above association yields a 
parameterization of the conjugacy classes of 
non-linearisable Jonqui\`eres elements of 
order $2n$ of 
$\mathrm{Bir}(\mathbb{P}^2_\mathbb{C})$ by
isomorphism classes of pairs $(\Gamma,\psi)$, 
where 
\begin{itemize}
\item[$\diamond$] $\Gamma$ is a smooth hyperelliptic 
curve of positive genus, 

\item[$\diamond$] $\psi\in\mathrm{Aut}(\Gamma)$
is an automorphism of order $n$, which 
preserves the fibres of the $\phi_1^2$ and is 
not a root of the involution associated to 
the $\phi_1^2$.
\end{itemize}
\end{itemize}
\end{thm}

The analogous result for finite Jonqui\`eres
cyclic groups holds, and follows directly from
this statement.

Note that if the curve $\Gamma$ has geometric genus 
$\geq 2$, the $\phi_1^2$ is unique, otherwise
it is not.

\smallskip

Blanc also dealt with cyclic subgroups of 
$\mathrm{Bir}(\mathbb{P}^2_\mathbb{C})$ that 
are not of Jonqui\`eres type. Using the 
classification of \cite{DolgachevIskovskikh}
and some classical tools on surfaces and 
curves he provided the parameterization
of the $29$ families of such groups. 

The classification is divided in two parts:
\begin{itemize}
\item[$\diamond$] find representative families
and prove that each group is conjugate to 
one of these;

\item[$\diamond$] parameterize the conjugacy
classes in each families by algebraic 
varieties.
\end{itemize}

For cyclic groups of prime order the varieties 
parameterizing the conjugacy classes are the 
moduli spaces of the non-rational curves 
fixed by the groups. Blanc needs to generalize
it, by looking for the non-rational curves 
fixed by the non-trivial elements of the group.
Let us give the definition of this invariant
which provides a simple way to decide whether
two cyclic groups are conjugate. Recall that
a birational map of the complex projective
plane fixes a curve if it restricts
to the identity on the curve.

\begin{defi}
Let $\phi$ be a non-trivial element of 
$\mathrm{Bir}(\mathbb{P}^2_\mathbb{C})$ of 
finite order. 

If no curve of positive geometric genus is
(pointwise) fixed by $\phi$, then 
$\mathrm{NFC}(\phi)=\emptyset$; otherwise
$\phi$ fixes exactly one curve of positive
genus (\cite{BayleBeauville,deFernex}), 
and $\mathrm{NFC}(\phi)$ is then the 
isomorphism class of the normalization of
this curve.
\end{defi}

Two involutions $\phi$, $\psi$ of 
$\mathrm{Bir}(\mathbb{P}^2_\mathbb{C})$
are conjugate if and only if 
$\mathrm{NFC}(\phi)=\mathrm{NFC}(\psi)$
(\emph{see} \S\ref{sec:ordre2}).
If $\phi$, $\psi$ are elements of 
$\mathrm{Bir}(\mathbb{P}^2_\mathbb{C})$
of the same prime order, then 
$\langle\phi\rangle$ and $\langle\psi\rangle$
are conjugate if and only if 
$\mathrm{NFC}(\phi)=\mathrm{NFC}(\psi)$
(\emph{see} \cite{BeauvilleBlanc, deFernex}).
This is no longer the case for cyclic groups
of composite order as observed in 
\cite{BeauvilleBlanc}: the automorphism $\phi$
of the cubic surface $z_0^3+z_1^3+z_2^3+z_3^3=0$
in $\mathbb{P}^3_\mathbb{C}$ given by 
\[
\phi\colon(z_0:z_1:z_2:z_3)\mapsto\big(z_1:z_0:z_2:\zeta z_3\big) 
\]
where $\zeta^3=1$, $\zeta\not=1$ has only four
fixed points while $\phi^2$ 
fixes the elliptic curve $z_3=0$.

\begin{defi}
Let $\phi\in\mathrm{Bir}(\mathbb{P}^2_\mathbb{C})$ 
be a non-trivial element of finite order $n$. 
Then $\mathrm{NFCA}(\phi)$ is the sequence of 
isomorphism classes of pairs 
\[
\Big(\mathrm{NFC}(\phi^k),\phi_{\vert\mathrm{NFC}(\phi^k)}\Big)_{k=1}^{n-1}
\]
where $\phi_{\vert\mathrm{NFC}(\phi^k)}$ is 
the automorphism induced by $\phi$ on the 
curve $\mathrm{NFC}(\phi^k)$ (if 
$\mathrm{NFC}(\phi^k)=\emptyset$, 
then $\phi$ acts trivially on it).
\end{defi}

Let us now give a simple way to decide whether 
two cyclic subgroups of finite order of 
$\mathrm{Bir}(\mathbb{P}^2_\mathbb{C})$ are
conjugate:

\begin{thm}[\cite{Blanc:cyclic}]
Let $\mathrm{G}$ and $\mathrm{H}$ be two 
cyclic subgroups of 
$\mathrm{Bir}(\mathbb{P}^2_\mathbb{C})$ of 
the same finite order. Then $\mathrm{G}$
and $\mathrm{H}$ are conjugate in 
$\mathrm{Bir}(\mathbb{P}^2_\mathbb{C})$
if and only if 
$\mathrm{NFCA}(\phi)=\mathrm{NFCA}(\psi)$
for some generators $\phi$ of $\mathrm{G}$
and $\psi$ of $\mathrm{H}$.
\end{thm}


\chapter{Uncountable subgroups of the Cremona group}\label{chapter:uncountable}

\bigskip
\bigskip

All the results of this Chapter have been proved 
without the construction
of the action of the isometric action of 
$\mathrm{Bir}(\mathbb{P}^2_\mathbb{C})$ on the 
hyperbolic space $\mathbb{H}^\infty$ and 
we keep this point of view. Different 
ideas and tools are used in any section: foliations
and group theory are the main ingredients.

\smallskip

The study of the automorphis groups starts a long time ago. For
instance for classical groups let us see \cite{Dieudonne}.
Consider the automorphism group of the complex projective space
$\mathbb{P}^n_\mathbb{C}$; it is $\mathrm{PGL}(n+1,\mathbb{C})$.
The automorphism group of $\mathrm{PGL}(n+1,\mathbb{C})$ is
generated by the inner automorphisms, the involution
$M\mapsto M^\vee$ and the action of the field automorphisms
of $\mathbb{C}$. In $1963$ Whittaker showed that any isomorphism
between homeomorphism groups of connex topological varieties is
induced by an homeomorphism between the varieties themselves
(\cite{Whittaker}). In $1982$ Filipkiewicz
proved a similar statement for differentiable varieties. 

\begin{thm}[\cite{Filipkiewicz}]\label{thm:Filipkiewicz}
  Let $V$, $W$ be two connected varieties of class $\mathcal{C}^k$,
  resp. $\mathcal{C}^j$. Let $\mathrm{Diff}^k(V)$ be the
group of $\mathcal{C}^k$-diffeomorphisms of $V$.
  Let $\phi\colon\mathrm{Diff}^k(V)\to\mathrm{Diff}^j(V)$ be an
  isomorphism group. Then $k=j$ and there exists a
  $\mathcal{C}^k$-difffeomorphism $\psi\colon V\to W$ such that
  \[
  \phi(\varphi)=\psi\circ\varphi\circ\psi^{-1}\qquad\qquad\forall\,\varphi\in\mathrm{Bir}(\mathbb{P}^2_\mathbb{C}).
  \]
\end{thm}

The description of uncountable maximal abelian subgroups of the
plane Cremona group 
allows to characterize the automorphisms group of
$\mathrm{Bir}(\mathbb{P}^2_\mathbb{C})$:

\begin{thm}[\cite{Deserti:abelien}]\label{thm:autofbir}
  Let $\varphi$ be an automorphism of $\mathrm{Bir}(\mathbb{P}^2_\mathbb{C})$.
  There exist a birational self map $\psi$ of the complex
  projective plane and an automorphism $\kappa$ of the field
  $\mathbb{C}$ such that
  \[
  \varphi(\phi)={}^{\kappa}\!\,(\psi\circ\phi\circ\psi^{-1})\qquad\qquad\forall\,\phi\in\mathrm{Bir}(\mathbb{P}^2_\mathbb{C}).
  \]
  In other words the non-inner automorphism group of
  $\mathrm{Bir}(\mathbb{P}^2_\mathbb{C})$ can be identified
  with the automorphisms of the field $\mathbb{C}$.
\end{thm}

In the first section we study uncountable maximal
abelian subgroups of $\mathrm{Bir}(\mathbb{P}^2_\mathbb{C})$;
let $\mathrm{G}$ be such a group. We give an outline
of the proofs of the following results:
\begin{itemize}
\item[$\diamond$] any element of $\mathrm{G}$ preserves
at least one singular holomorphic foliation;

\item[$\diamond$] either no element of $\mathrm{G}$ 
is torsion-free, or $\mathrm{G}$ leaves invariant 
a holomorphic foliation;

\item[$\diamond$] if $\mathrm{G}$ is torsion-free, 
then $\mathrm{G}$ is conjugate to a subgroup of 
the Jonqui\`eres group.
\end{itemize}

In the second section we describe the automorphism
group of $\mathrm{Bir}(\mathbb{P}^2_\mathbb{C})$.
A study of the torsion-free maximal abelian subgroups
of the Jonqui\`eres group shows that the 
group 
\[
\mathcal{J}_a=\big\{(z_0,z_1)\dashrightarrow(z_0+a(z_1),z_1)\,\vert\,a\in\mathbb{C}(z_1)\big\}
\]
is invariant by any automorphism of
$\mathrm{Bir}(\mathbb{P}^2_\mathbb{C})$. Some 
work on special subgroups of~$\mathcal{J}_a$
achieves the description of 
$\mathrm{Aut}(\mathrm{Bir}(\mathbb{P}^2_\mathbb{C}))$.

\smallskip

In a session problems during the International Congress of Mathematicians 
Mumford proposed the following (\cite{Mumford}):

\begin{quotation}
"Let $\mathrm{G}=\mathrm{Aut}_{\mathbb{C}}\mathbb{C}(z_0,z_1)$ be the Cremona 
group (...) the problem is to topologize $\mathrm{G}$ and associate to
it a Lie algebra consisting, roughly, of those meromorphic vector 
fields $D$ on $\mathbb{P}^2_\mathbb{C}$ which "integrate" into an 
analytic family of Cremona transformations."
\end{quotation}

In the third section we deal with a contribution in 
that direction: the description of $1$-parameter 
subgroups of quadratic birational self maps of 
$\mathbb{P}^2_\mathbb{C}$. 

\smallskip

In \cite{Ghys} Ghys showed that any nilpotent
subgroup of $\mathrm{Diff}^\omega(\mathbb{S}^2)$ 
is metabelian; as a consequence he got that if 
$\Gamma$ is a subgroup of finite index of 
$\mathrm{SL}(n,\mathbb{Z})$, $n\geq 4$, then 
any morphism from $\Gamma$ into 
$\mathrm{Diff}^\omega(\mathbb{S}^2)$ has finite
image. In the same spirit the nilpotent subgroups
of the plane Cremona group
are described in the fourth section: if $\Gamma$
is a strongly nilpotent group of length $>1$, 
then either $\mathrm{G}$ is metabelian up to 
finite index, or $\mathrm{G}$ is a torsion 
group. As a consequence as soon as $n\geq 5$
no subgroup of $\mathrm{SL}(n,\mathbb{Z})$ 
of finite index embeds into 
$\mathrm{Bir}(\mathbb{P}^2_\mathbb{C})$.

\smallskip

The description of centralizers of discrete
dynamical systems is an important problem 
in real/complex dynamics. Julia
(\cite{Julia1, Julia2}) then Ritt
(\cite{Ritt}) show that the set
\[
\mathrm{Cent}(\phi)=\big\{\psi\colon\mathbb{P}^1_\mathbb{C}\to\mathbb{P}^1_\mathbb{C}\,\vert\,\psi\circ \phi=\phi\circ \psi\big\}
\]
of rational functions that commute to a 
rational function $\phi$ coincide in general
\footnote{except monomial maps $z\mapsto z^k$, 
Tchebychev polynomials, 
Latt\`es examples ...} 
with $\big\{\phi_0^n\,\vert\,n\in\mathbb{N}\big\}$ 
where $\phi_0$ is an element of $\mathrm{Cent}(\phi)$.
In the $60$'s Smale considered generic
diffeomorphisms $\phi$ of compact manifolds and 
asked if its centralizer coincides with 
$\big\{\phi^n\,\vert\,n\in\mathbb{Z}\big\}$.
Many mathematicians have considered this 
question (for instance 
\cite{BonattiCrovisierWilkinson, Palis, PalisYoccoz, PalisYoccoz2}). 
The fifth section deals with centralizers of 
elliptic birational maps, Jonqui\`eres 
twists and Halphen twists.

\bigskip
\bigskip


\section{Uncountable maximal abelian subgroups of $\mathrm{Bir}(\mathbb{P}^2_\mathbb{C})$}\label{sec:uncountableabelian}

Let $S$ be a complex compact surface. A \textsl{foliation}\index{defi}{foliation}
$\mathcal{F}$ on $S$ is given by a family $(\chi_i)_i$ of holomorphic vector
fields with isolated zero defined on some open cover $(\mathcal{U}_i)_i$ of 
$S$. The vector fields $\chi_i$ have to satisfy the following conditions: 
there exist $g_{ij}\in\mathcal{O}^*(\mathcal{U}_i\cap\mathcal{U}_j)$ such 
that $\chi_i=g_{ij}\chi_j$ on $\mathcal{U}_i\cap\mathcal{U}_j$. Let us remark
that a non-trivial meromorphic vector field on $S$ defines such a foliation.

\begin{lem}[\cite{Deserti:abelien}]
Let $\mathrm{G}$ be an uncountable abelian subgroup of 
$\mathrm{Bir}(\mathbb{P}^2_\mathbb{C})$. There exists a rational vector 
field $\chi$ such that 
\[ 
\varphi_*\chi=\chi\qquad\qquad\forall\,\varphi\in\mathrm{G}.
\]
In particular $\mathrm{G}$ preserves a foliation.
\end{lem}

\begin{proof}
Since $\mathrm{G}$ is uncountable, there exists an integer $d$ such that 
\[
\mathrm{G}_d=\mathrm{G}\cap\mathrm{Bir}_d(\mathbb{P}^2_\mathbb{C})
\]
is uncountable. Hence the Zariski closure $\overline{\mathrm{G}_d}$ of 
$\mathrm{G}_d$ in $\mathrm{Bir}_{\leq d}(\mathbb{P}^2_\mathbb{C})$ is an 
algebraic set of dimension $\geq 1$. Consider a curve in 
$\overline{\mathrm{G}_d}$, {\it i.e.} a map 
\[
\eta\colon\mathbb{D}\to\overline{\mathrm{G}_d},\qquad\qquad t\mapsto\eta(t).
\]
Remark that elements of $\overline{\mathrm{G}_d}$ are rational maps that commute.
Let us define the rational vector field $\chi$ at any 
$m\in\mathbb{P}^2_\mathbb{C}\smallsetminus\mathrm{Ind}(\eta(0)^{-1})$ by
\[
\chi(m)=\frac{\partial\eta(s)}{\partial s}\Big\vert_{s=0}\big(\eta(0)^{-1}(m)\big).
\]
Let $\varphi$ be an element of $\overline{\mathrm{G}_d}$. If we differentiate
the equality
\[
\varphi\eta(s)\varphi^{-1}(m)=\eta(s)(m)
\]
with respect to $s$, $m$ being fixed, one gets: $\varphi_*\chi=\chi$. In
other words $\chi$ is invariant by the elements of $\overline{\mathrm{G}_d}$,
and so by any element of $\mathrm{G}$.
\end{proof}

As a result for any uncountable abelian subgroup $\mathrm{G}$ of 
$\mathrm{Bir}(\mathbb{P}^2_\mathbb{C})$, there exists a foliation on 
$\mathbb{P}^2_\mathbb{C}$ invariant by $\mathrm{G}$. Brunella, McQuillan 
and Mendes have classified, up to birational equivalence, singular 
holomorphic foliations on projective, compact, complex surfaces 
(\cite{Brunella, McQuillan, Mendes}). If $S$ is a projective surface endowed
with a foliation $\mathcal{F}$, we denote by 
$\mathrm{Bir}(S,\mathcal{F})$\index{not}{$\mathrm{Bir}(S,\mathcal{F})$}
(resp. 
$\mathrm{Aut}(S,\mathcal{F})$\index{not}{$\mathrm{Aut}(S,\mathcal{F})$}) 
the group of birational maps (resp.
holomorphic maps) of $S$ preserving the foliation $\mathcal{F}$. In 
general $\mathrm{Bir}(S,\mathcal{F})$ coincides with 
$\mathrm{Aut}(S,\mathcal{F})$ and is finite. In \cite{CantatFavre} the authors
dealt with the opposite case and got a classification.

\begin{thm}[\cite{CantatFavre}]\label{thm:CantatFavre1}
Let $\mathcal{F}$ be a foliation on $S$ such that 
$\mathrm{Aut}(X,\varphi^*\mathcal{F})\subsetneq\mathrm{Bir}(X,\varphi^*\mathcal{F})$ 
for any birational map $\varphi\colon X\dashrightarrow S$. Then, up to 
conjugacy, there exists an element of infinite order in 
$\mathrm{Bir}(S,\mathcal{F})$ and
\begin{itemize}
\item[$\diamond$] either $\mathcal{F}$ is a rational fibration, 

\item[$\diamond$] or up to a finite cover there exist some integers 
$p$, $q$, $r$, $s$ such that
\[
\mathrm{Bir}(\mathbb{P}^2_\mathbb{C},\mathcal{F})=\big\{(z_0,z_1)\dashrightarrow(z_0^pz_1^q,z_0^rz_1^s),\,(z_0,z_1)\mapsto(\alpha z_0,\beta z_1)\,\vert\,\alpha,\,\beta\in\mathbb{C}^*\big\}.
\]
\end{itemize}
\end{thm}

Before stating the opposite case $\mathrm{Aut}(S,\mathcal{F})$ infinite, 
let us give some definitions. Let $\Lambda$ be a lattice in $\mathbb{C}^2$;
it induces a complex torus $\mathbb{T}=\faktor{\mathbb{C}^2}{\Lambda}$ of dimension $2$.
For instance the product of an elliptic curve by itself is a complex torus.
An affine map $\psi$ that preserves $\Lambda$ induces an automorphism of the
torus $\mathbb{T}$. If the linear part of $\psi$ is of infinite order, then
\begin{itemize}
\item[$\diamond$] either the linear part of $\psi$ is hyperbolic and $\psi$ 
induces an Anosov automorphism that preserves two linear foliations;

\item[$\diamond$] or the linear part of $\psi$ is unipotent and $\psi$ preserves
an elliptic fibration.
\end{itemize}
Sometimes there is a finite automorphism group of $\mathbb{T}$ 
normalized by $\psi$. Denote by $\widetilde{\faktor{\mathbb{T}}{\mathrm{G}}}$ the 
desingularization of $\faktor{\mathbb{T}}{\mathrm{G}}$. The automorphism 
 induced by $\psi$ on $\widetilde{\faktor{\mathbb{T}}{\mathrm{G}}}$ 
preserves
\begin{itemize}
\item[$\diamond$] the foliations induced the stable and unstable foliations 
preserved by $\psi$ when $\psi$ is hyperbolic;

\item[$\diamond$] an elliptic fibration when the linear part of $\psi$ is 
unipotent.
\end{itemize}
If $\mathrm{G}=\big\{\mathrm{id},\,(z_0,z_1)\mapsto(-z_0,-z_1)\big\}$ we 
say that $\widetilde{\faktor{\mathbb{T}}{\mathrm{G}}}$ is a 
\textsl{Kummer surface}\index{defi}{Kummer surface}; otherwise 
$\widetilde{\faktor{\mathbb{T}}{\mathrm{G}}}$ is a 
\textsl{generalized Kummer surface}\index{defi}{generalized Kummer surface}.

\begin{thm}[\cite{CantatFavre}]\label{thm:CantatFavre2}
Let $\mathcal{F}$ be a singular holomorphic foliation on a projective 
surface~$S$.
Assume that $\mathrm{Aut}(S,\mathcal{F})$ is infinite. Then 
$\mathrm{Aut}(S,\mathcal{F})$ contains at least one element $\varphi$ of
infinite order and one of the following holds:
\begin{itemize}
\item[$\diamond$] $\mathcal{F}$ is invariant by an holomorphic vector field;

\item[$\diamond$] $\mathcal{F}$ is an elliptic fibration;

\item[$\diamond$] the surface $S$ is a generalized Kummer surface, $\varphi$
can be lifted to an Anosov automorphism $\widetilde{\varphi}$ of the
torus and $\mathcal{F}$ is the projection on $S$ of the unstable or 
stable foliation of~$\widetilde{\varphi}$.
\end{itemize}
\end{thm}

\begin{rems}
\begin{itemize}
\item[$\diamond$] The foliations invariant by an holomorphic vector field
are described in \cite[Proposition 3.8]{CantatFavre}.

\item[$\diamond$] The last two cases are mutually exclusive.
\end{itemize}
\end{rems}

Using these two statements one can prove the following one:

\begin{thm}[\cite{Deserti:abelien}]\label{thm:abmax}
Let $\mathrm{G}$ be an uncountable maximal abelian subgroup of 
$\mathrm{Bir}(\mathbb{P}^2_\mathbb{C})$. Then:
\begin{itemize}
\item[$\diamond$] either $\mathrm{G}$ has an element of finite order;

\item[$\diamond$] or $\mathrm{G}$ is up to conjugacy a subgroup of 
the Jonqui\`eres group.
\end{itemize}
\end{thm}

\begin{proof}[Idea of the proof]
Assume first that 
$\mathrm{Aut}(X,\phi^*\mathcal{F})\subsetneq\mathrm{Bir}(X,\phi^*\mathcal{F})$
for any birational map $\phi\colon X\dashrightarrow S$. Then according to
Theorem \ref{thm:CantatFavre1} either $\mathrm{G}$ preserves a rational 
fibration, and then $\mathrm{G}$ is up to birational conjugacy contained in 
the Jonqui\`eres group; or $\mathrm{G}$ is up to conjugacy and finite cover a
subgroup of 
\[
\big\{(z_0,z_1)\dashrightarrow(z_0^pz_1^q,z_0^rz_1^s),\,(z_0,z_1)\mapsto(\alpha z_0,\beta z_1)\,\vert\,\alpha,\,\beta\in\mathbb{C}^*,\,\alpha=\alpha^p\beta^q,\,\beta=\alpha^r\beta^s\big\}.
\]
If $\mathrm{G}$ is conjugate to the diagonal group 
$\mathrm{D}=\big\{(z_0,z_1)\mapsto(\alpha z_0,\beta z_1)\,\vert\,\alpha,\,\beta\in\mathbb{C}^*\big\}$, 
then $\mathrm{G}$ contains elements of finite order. Otherwise since 
$\mathrm{G}$ is uncountable it can not be reduced to 
\[
\langle(z_0,z_1)\dashrightarrow(z_0^pz_1^q,z_0^rz_1^s)\rangle.
\]
Therefore, there exists a non-trivial 
element $(z_0,z_1)\mapsto(\lambda z_0,\mu z_1)$ in $\mathrm{G}$ such that 
$\lambda=\lambda^p\mu^q$ and $\mu=\lambda^r\mu^s$. For any $\ell$ 
the map $(z_0,z_1)\mapsto(\lambda^\ell z_0,\mu^\ell z_1)$ satisfies these 
equalities, so belongs to $\mathrm{G}$. Consider $\ell$ such that 
$\lambda^\ell=\mathbf{i}$; then 
$\mu^\ell=\mathrm{e}^{\mathbf{i}\pi \frac{1-p}{2q}}$ is also a root of unity and
$(z_0,z_1)\mapsto(\lambda^\ell z_0,\mu^\ell z_1)$ is thus an element of finite 
order of $\mathrm{G}$. More precisely $\mathrm{G}$ contains periodic elements 
of any order.

\bigskip

Suppose now that there exist a surface $S$ and a birational map 
$\psi\colon S\dashrightarrow\mathbb{P}^2_\mathbb{C}$ such that
$\mathrm{Aut}(S,\psi^*\mathcal{F})=\mathrm{Bir}(S,\psi^*\mathcal{F})$. 
According to Theorem \ref{thm:CantatFavre2}
\begin{itemize}
\item[$\diamond$] either $\psi^*\mathcal{F}$ is invariant by an holomorphic
vector field on $S$;

\item[$\diamond$] or $\psi^*\mathcal{F}$ is an elliptic fibration.
\end{itemize}

Since $\mathrm{G}$ is uncountable the last eventuality can not occur 
(\cite{BHPV}). Let us thus assume that $\psi^*\mathcal{F}$ is invariant
by an holomorphic vector field on $S$. According to \cite{CantatFavre}
one can assume up to conjugacy that $\mathrm{G}$ is a subgroup of 
$\mathrm{Aut}(\widetilde{S})$ where $\widetilde{S}$ is a minimal model
of $S$. But minimal rational surfaces are $\mathbb{P}^2_\mathbb{C}$, 
$\mathbb{P}^1_\mathbb{C}\times\mathbb{P}^1_\mathbb{C}$ and the Hirzebruch
surfaces $\mathbb{F}_n$, $n\geq 2$, and their automorphisms groups are
known (\emph{see} Chapter \ref{Chapter:algebraicsubgroup}). 

The description of the uncountable maximal abelian subgroups of minimal rational
surfaces gives:

\begin{pro}[\cite{Deserti:abelien}]
Let $S$ be a minimal rational surface. Let $\mathrm{G}$ be an uncountable 
abelian subgroup of $\mathrm{Aut}(S)$ maximal in $\mathrm{Bir}(S)$. Then:
\begin{itemize}
\item[$\diamond$] either $\mathrm{G}$ contains an element of finite order,

\item[$\diamond$] or $\mathrm{G}$ coincides with 
$\big\{(z_0,z_1)\mapsto(z_0+P(z_1),z_1)\,\vert\, P\in\mathbb{C}[z_1],\,\deg P\leq n\big\}$,

\item[$\diamond$] or 
$\mathrm{G}=\big\{(z_0,z_1)\mapsto(z_0+\alpha,z_1+\beta)\,\vert\,\alpha,\,\beta\in\mathbb{C}\big\}$.
\end{itemize}
\end{pro}

A study of the uncountable maximal abelian subgroups of the Jonqui\`eres 
group allows to refine Theorem \ref{thm:abmax} as follows:

\begin{thm}[\cite{Deserti:abelien}]
Let $\mathrm{G}$ be an uncountable maximal abelian subgroup of the 
plane Cremona group. Then up to conjugacy:
\begin{itemize}
\item[$\diamond$] either $\mathrm{G}$ contains an element of finite
order,

\item[$\diamond$] or $\mathrm{G}=\big\{(z_0,z_1)\dashrightarrow(z_0+a(z_1),z_1)\,\vert\,a\in\mathbb{C}(z_1)\big\}$,

\item[$\diamond$] or $\mathrm{G}=\big\{(z_0,z_1)\mapsto(z_0+\alpha,z_1+\beta)\,\vert\,\alpha,\,\beta\in\mathbb{C}\big\}$,

\item[$\diamond$] or any subgroup of $\mathrm{Bir}(\mathbb{P}^2_\mathbb{C})$
acting by conjugacy on $\mathrm{G}$ is, up to finite index, solvable.
\end{itemize}
\end{thm}
\end{proof}


\section{Description of the automorphisms group of the Cremona group}\label{sec:autbir}

Let us give an idea of the proof of Theorem \ref{thm:autofbir}.
The description of uncountable maximal abelian subgroups of
$\mathrm{Bir}(\mathbb{P}^2_\mathbb{C})$ yields to

\begin{cor}[\cite{Deserti:abelien}]
  Let $\varphi$ be an automorphism of $\mathrm{Bir}(\mathbb{P}^2_\mathbb{C})$.
  Set 
  \[
  \mathrm{J}_a=\big\{(z_0,z_1)\dashrightarrow(z_0+a(z_1),z_1)\,\vert\,a\in\mathbb{C}(z_1)\big\}.
  \]
  Up to birational conjugacy
$\varphi(\mathrm{J}_a)=\mathrm{J}_a$ and 
$(z_0,z_1)\mapsto(z_0+1,z_1)$ is invariant 
by $\varphi$.
\end{cor}

Let us consider
\begin{align*}
&\mathrm{T}_1=\big\{(z_0,z_1)\mapsto(z_0+\alpha,z_1)\,\vert\,\alpha\in\mathbb{C}\big\}, && \mathrm{T}_2=\big\{(z_0,z_1)\mapsto(z_0,z_1+\beta)\,\vert\,\beta\in\mathbb{C}\big\},
\end{align*}
and
\begin{align*}
& \mathrm{D}_1=\big\{(z_0,z_1)\mapsto(\alpha z_0,z_1)\,\vert\,\alpha\in\mathbb{C}^*\big\}, &&\mathrm{D}_2=\big\{(z_0,z_1)\mapsto(z_0,\beta z_1)\,\vert\,\alpha\in\mathbb{C}\big\}.
\end{align*}

\begin{pro}[\cite{Deserti:abelien}]\label{pro:autbir}
  Let $\varphi$ be an automorphism of $\mathrm{Bir}(\mathbb{P}^2_\mathbb{C})$.
  Assume that $\varphi(\mathrm{J}_a)=\mathrm{J}_a$ and
  $(z_0,z_1)\mapsto(z_0+1,z_1)$ is invariant by $\varphi$. 
  Then up to birational conjugacy:
  \begin{itemize}
  \item[$\diamond$] $\varphi(\mathrm{J}_a)=\mathrm{J}_a$,

  \item[$\diamond$] $(z_0,z_1)\mapsto(z_0+1,z_1)$ is invariant by $\varphi$,

  \item[$\diamond$] $\varphi(\mathrm{T}_1)=\mathrm{T}_1$ and
 $\varphi(\mathrm{T}_2)=\mathrm{T}_2$,

  \item[$\diamond$] $\varphi(\mathrm{D}_1)=\mathrm{D}_1$ and
     $\varphi(\mathrm{D}_2)=\mathrm{D}_2$.
  \end{itemize}
\end{pro}

As a consequence an automorphism of
$\mathrm{Bir}(\mathbb{P}^2_\mathbb{C})$
induces two automorphisms of the group
$\mathrm{Aff}(\mathbb{C})$
\index{not}{$\mathrm{Aff}(\mathbb{C})$} 
of affine maps 
of the complex line.

\begin{lem}\label{lem:autbir}
  Let $\varphi$ be an automorphism of
  $\mathrm{Aff}(\mathbb{C})$. Then $\varphi$ is the composition of
  an inner automorphism and an automorphism of the field $\mathbb{C}$.
\end{lem}

\begin{proof}[Sketch of the proof]
  The maximal abelian subgroups of $\mathrm{Aff}(\mathbb{C})$ are the
  group of translations
  \[
  \mathrm{T}=\big\{z\mapsto z+\beta\,\vert\,\beta\in\mathbb{C}\big\}
  \]
  and the groups of affine maps that preserve a point
  \[
  \mathrm{D}_{z_0}=\big\{z\mapsto \alpha(z-z_0)+z_0\,\vert\alpha\in\mathbb{C}^*\big\}.
  \]
  Since $\mathrm{T}$ does not contain element of finite order,
  $\varphi$ sends $\mathrm{T}$ onto $\mathrm{T}$. In other
  words there exists an additive bijection
  $\kappa_2\colon\mathbb{C}\to\mathbb{C}$ such that
  $\varphi(z+\beta)=z+\kappa_2(\beta)$. Up to conjugacy by an
  element of $\mathrm{T}$ one can assume that
  $\varphi(\mathrm{D}_0)=\mathrm{D}_0$. In other words there
  exists a multiplicative bijection
  $\kappa_1\colon\mathbb{C}^*\to\mathbb{C}^*$ such that
  $\varphi(\alpha z)=\kappa_1(\alpha)z$. On the one hand
  \[
  \varphi\big(z\mapsto\alpha z+\alpha\big)=\varphi\big(z\mapsto z+\alpha\big)\circ\varphi\big(z\mapsto \alpha z\big)=\big(z\mapsto\kappa_1(\alpha)z+\kappa_2(\alpha)\big)
  \]
  and on the other hand
  \[
  \varphi\big(z\mapsto \alpha z+\alpha\big)=\varphi\big(z\mapsto \alpha z\big)\circ\varphi\big(z\mapsto z+1\big)=\big(z\mapsto\kappa_1(\alpha)z+\kappa_1(\alpha)\kappa_2(1)\big).
  \]
  Hence for any $\alpha$ the equality $z\mapsto\kappa_1(\alpha)z+\kappa_2(\alpha)=z\mapsto\kappa_1(\alpha)z+\kappa_1(\alpha)\kappa_2(1)$ holds. Since $\mu=\kappa_2(1)$
  is non-zero, $\kappa_2$ is additive and multiplicative. As a
  result $\kappa_2$ is an isomorphism of the field $\mathbb{C}$ and
  \begin{eqnarray*}
    \varphi\big(z\mapsto \alpha z+\beta)&=& \big(z\mapsto{}^{\kappa_1}\!\,\alpha z+{}^{\kappa_2}\!\,\beta\big)\\
    &=& \big(z\mapsto{}^{\kappa_1}\!\,\alpha z+\mu{}^{\kappa_1}\!\,\beta\big)\\
    &=& \big(z\mapsto{}^{\kappa_1}\!\,\big(\alpha z+{}^{\kappa_1^{-1}}\!\,\mu\beta\big)\big)\\
    &=& \big(z\mapsto{}^{\kappa_1}\!\,\big(({}^{\kappa_1^{-1}}\!\,\mu z)\circ(\alpha z+\beta)\circ({}^{\kappa_1}\!\,\mu z)\big)\big)\\
    &=& \big(z\mapsto{}^{\kappa_1}\!\,\big(({}^{\kappa_1}\!\,\mu z)^{-1}\circ(\alpha z+\beta)\circ({}^{\kappa_1}\!\,\mu z)\big)\big).
  \end{eqnarray*}
\end{proof}

\begin{proof}[Sketch of the proof of Theorem \ref{thm:autofbir}]
Proposition \ref{pro:autbir} and Lemma \ref{lem:autbir} imply that for 
any $\alpha$, $\beta$ in $\mathbb{C}^*$, for any $\gamma$, $\delta$ in 
$\mathbb{C}$ one has
\[
\varphi\big((z_0,z_1)\mapsto(\alpha z_0+\gamma,\beta z_1+\delta)\big)=\big((z_0,z_1)\mapsto({}^{\kappa_1}\!\,\alpha z_0+\mu{}^{\kappa_1}\!\,\gamma,{}^{\kappa_2}\!\,\beta z_1+\eta{}^{\kappa_2}\!\,\delta)\big)
\]
where $\eta$, $\mu$ are two non-zero complex numbers and $\kappa_1$, 
$\kappa_2$ two automorphisms of the field~$\mathbb{C}$. 
Since $(z_0,z_1)\mapsto(z_0+z_1,z_1)$ and $(z_0,z_1)\mapsto(\alpha z_0,\alpha z_1)$
commute their image by $\varphi$ also, and so $\kappa_1=\kappa_2$. As a consequence 
up to conjugacy by an inner automorphism and an automorphism of the field $\mathbb{C}$, 
the groups 
\[
\mathrm{T}=\big\{(z_0,z_1)\mapsto(z_0+\alpha,z_1+\beta)\,\vert\,\alpha,\,\beta\in\mathbb{C}\big\}
\] 
and
\[
\mathrm{D}=\big\{(z_0,z_1)\mapsto(\alpha z_0,\beta z_1)\,\vert\,\alpha,\,\beta\in\mathbb{C}^*\big\}
\]
are pointwise invariant. Then one can check that the involutions 
$(z_0,z_1)\mapsto\left(z_0,\frac{1}{z_1}\right)$ and $(z_0,z_1)\mapsto(z_1,z_0)$
are invariant by $\varphi$. But the group generated by $\mathrm{T}$, $\mathrm{D}$, 
$(z_0,z_1)\mapsto\left(z_0,\frac{1}{z_1}\right)$ and $(z_0,z_1)\mapsto(z_1,z_0)$
contains $\mathrm{PGL}(3,\mathbb{C})$. Furthermore
\[
\sigma_2=\left(\left((z_0,z_1)\mapsto\left(z_0,\frac{1}{z_1}\right)\right)\circ\left((z_0,z_1)\mapsto(z_1,z_0)\right)\right)^2
\]
hence $\varphi(\sigma_2)=\sigma_2$. We conclude thanks to the Noether and 
Castelnuovo Theorem.
\end{proof}

\begin{cor}[\cite{Deserti:abelien}]
  An isomorphism of the semi-group of rational self maps of $\mathbb{P}^2_\mathbb{C}$
  is inner up to the action of an automorphism of the field $\mathbb{C}$.
\end{cor}

In the spirit of the result of Filipkiewicz (Theorem \ref{thm:Filipkiewicz})
one has:

\begin{cor}[\cite{Deserti:abelien}]
Let $S$ be a complex projective surface. Let $\varphi$ be an
isomorphism between $\mathrm{Bir}(S)$ and
$\mathrm{Bir}(\mathbb{P}^2_\mathbb{C})$. There exist a
birational map $\psi\colon S\dashrightarrow\mathbb{P}^2_\mathbb{C}$
and an automorphism of the field $\mathbb{C}$ such that
  \[
  \varphi(\phi)={}^{\kappa}\!\,(\psi\circ\phi\circ\psi^{-1})\qquad\qquad\forall\,\phi\in\mathrm{Bir}(S).
  \]
\end{cor}

\begin{cor}[\cite{Deserti:abelien}]
The automorphism group of $\mathbb{C}(z_0,z_1)$ is
isomorphic to the automorphisms group of
$\mathrm{Bir}(\mathbb{P}^2_\mathbb{C})$.
\end{cor}

\begin{rem}
According to \cite{Beauville} the groups
$\mathrm{Bir}(\mathbb{P}^n_\mathbb{C})$ and
$\mathrm{Bir}(\mathbb{P}^2_\mathbb{C})$ are
isomorphic if and only if $n=2$.
\end{rem}

\bigskip

Note that there is no description of 
$\mathrm{Aut}(\mathrm{Bir}(\mathbb{P}^n_\mathbb{C}))$ for $n\geq 3$.
Nevertheless there are two results in that direction:
\begin{itemize}
\item[$\diamond$] the first one is
\begin{thm}[\cite{Deserti:reg}]
Let $\varphi$ be an automorphism of 
$\mathrm{Bir}(\mathbb{P}^n_\mathbb{C})$; there exist an
automorphism $\kappa$ of the field $\mathbb{C}$, and a 
birational self map $\psi$ of $\mathbb{P}^n_\mathbb{C}$ such
that
\[
\varphi(\phi)={}^{\kappa}\!\,(\psi\circ\phi\circ\psi^{-1})\qquad\qquad \forall\,\phi\in\mathrm{G}(n,\mathbb{C})=\langle\sigma_n,\,\mathrm{PGL}(n+1,\mathbb{C})\rangle.
\]
\end{thm}

\item[$\diamond$] the second one is
\begin{thm}[\cite{Cantat5}]\label{thm:Cantatcompos}
Let $V$ be a smooth connected complex 
projective variety of dimension $n$. Let $r$ be a 
positive integer and let 
$\rho\colon\mathrm{Aut}(\mathbb{P}^r_\mathbb{C})\to\mathrm{Bir}(V)$
be an injective morphism of groups. Then $n\leq r$.

Furthermore if $n=r$, there exist a field morphism
$\kappa\colon\mathbb{C}\to\mathbb{C}$ and a birational
map $\psi\colon V\dashrightarrow\mathbb{P}^n_\mathbb{C}$
such that 
\begin{itemize}
\item[$\diamond$] either $\psi\circ\rho(A)\circ\psi^{-1}={}^{\kappa}\!\, A$ for all $A\in\mathrm{Aut}(\mathbb{P}^n_\mathbb{C})$,

\item[$\diamond$] or $\psi\circ\rho(A)\circ\psi^{-1}=({}^{\kappa}\!\, A)^\vee$ for all $A\in\mathrm{Aut}(\mathbb{P}^n_\mathbb{C})$.
\end{itemize}
In particular $V$ is rational. Moreover $\kappa$ is an 
automorphism of $\mathbb{C}$ if $\rho$ is an isomorphism.
\end{thm}
\end{itemize}

Before giving an idea of the proof of this last result
let us state two corollaries of it. The first shows 
that the Cremona groups 
$\mathrm{Bir}(\mathbb{P}^n_\mathbb{C})$
are pairwise non-isomorphic, thereby solving an open 
problem for $n\geq 4$.

\begin{cor}[\cite{Cantat5}]
Let $n$ and $k$ be natural integers. The group 
$\mathrm{Bir}(\mathbb{P}^n_\mathbb{C})$ embeds into 
$\mathrm{Bir}(\mathbb{P}^k_\mathbb{C})$ if and only if
$n\leq k$.

In particular $\mathrm{Bir}(\mathbb{P}^n_\mathbb{C})$ 
is isomorphic to $\mathrm{Bir}(\mathbb{P}^k_\mathbb{C})$
if and only if $n=k$.
\end{cor}

The second characterizes rational varieties $V$ by 
the structure of $\mathrm{Bir}(V)$, as an abstract 
group:

\begin{cor}[\cite{Cantat5}]
Let $V$ be an irreducible complex projective variety
of dimension~$n$. The following properties are 
equivalent:
\begin{itemize}
\item[$\diamond$] $V$ is rational,

\item[$\diamond$] $\mathrm{Bir}(V)$ is isomorphic
to $\mathrm{Bir}(\mathbb{P}^n_\mathbb{C})$ as an 
abstract group,

\item[$\diamond$] there is a non-trivial morphism
from $\mathrm{PGL}(n+1,\mathbb{C})$ to $\mathrm{Bir}(V)$.
\end{itemize}
\end{cor}

The strategy that leads to the proof of Theorem
\ref{thm:Cantatcompos} is similar to the proof
of Theorem \ref{thm:autofbir} but requires 
several new ideas:
\begin{itemize}
\item[$\diamond$] Weil's regularization Theorem
(Theorem \ref{thm:Weil}), 
that transforms a group of birational maps of $V$ 
with uniformly bounded degrees into a group of 
automorphisms of a new variety by a birational
change of variables;

\item[$\diamond$] Epstein and Thurston work on 
nilpotent Lie subalgebras in the Lie algebra 
of smooth vector fields of a compact manifold
(\cite{EpsteinThurston}).
\end{itemize}


\section{One-parameter subgroups of $\mathrm{Bir}(\mathbb{P}^2_\mathbb{C})$}

\subsection{Description of $1$-parameter subgroups of quadratic 
birational maps of $\mathbb{P}^2_\mathbb{C}$}

A \textsl{germ of flow}\index{defi}{germ (of flow)} in 
$\mathrm{Bir}_{\leq 2}(\mathbb{P}^2_\mathbb{C})$ is a germ of holomorphic 
application $t\mapsto\phi_t\in\mathrm{Bir}_{\leq 2}(\mathbb{P}^2_\mathbb{C})$
such that 
\[
\left\{
\begin{array}{ll}
\phi_{t+s}=\phi_t\circ\phi_s\\
\phi_0=\mathrm{id}
\end{array}
\right.
\]

Since a germ of flow can be generalized we speak about 
\textsl{flow}\index{defi}{flow}. The set of lines blown down by the flow 
$\phi_t$ is a germ of analytic sets in the Grassmaniann of lines in 
$\mathbb{P}^2_\mathbb{C}$, {\it i.e.} in the dual space 
$(\mathbb{P}^2_\mathbb{C})^\vee$. Similarly the set of indeterminacy 
points of the $\phi_t$ is a germ of analytic sets of~$\mathbb{P}^2_\mathbb{C}$.

We call \textsl{family of contracted curves}\index{defi}{family (of contracted
curves)} a continuous
map (indeed an analytic one) defined over a germ of closed sector $\Delta$
of vertex $0$ in $\mathbb{C}$
\[
\mathcal{D}\colon\Delta\to(\mathbb{P}^2_\mathbb{C})^{\vee}
\]
such that for any $t\in\Delta$ the lines $\mathcal{D}_t$ coincide with a 
line $\mathcal{D}(t)$ blown down by $\phi_t$. 

Similarly a \textsl{family of indeterminacy points}\index{defi}{family (of 
indeterminacy points)} is a continuous map $t\mapsto p_t$ defined on a 
sector~$\Delta$ such that any $p_t$ is an indeterminacy point of $\phi_t$.

Let $\phi_t$ be a flow. Let $\mathcal{D}_t$ (resp. $p_t$) be a family of 
curves blown down by $\phi_t$ (resp. a family of indeterminacy points of $\phi_t$).
If $\mathcal{D}_t$ (resp. $p_t$) is independent of $t$, the family is 
\textsl{unmobile}\index{defi}{unmobile (family)}, otherwise it is 
\textsl{mobile}\index{defi}{mobile (family)}.

A rational vector field $\chi$ on $\mathbb{P}^2_\mathbb{C}$ is 
\textsl{rationally integrable}\index{defi}{rationally integrable (flow)} if 
its flow is a flow of birational maps.

A germ of flow in $\mathrm{Bir}_2(\mathbb{P}^2_\mathbb{C})$ is the flow of a rationally integrable
vector field $\chi=\frac{\partial\phi_t}{\partial t}\Big\vert_{t=0}$
called \textsl{infinitesimal generator}\index{defi}{infinitesimal generator}
of $\phi_t$. To this vector field is associated a foliation whose leaves 
are "grosso modo" the trajectories of $\chi$. Recall that a 
\textsl{fibration by lines}\index{defi}{fibration (by lines)} $\mathcal{L}$ 
of $\mathbb{P}^2_\mathbb{C}$ is given by
\[
\lambda\ell_1+\mu\ell_2=0
\]
where $\ell_1$, $\ell_2$ are linear forms that are not proportional. The 
\textsl{base-point}\index{defi}{base-point (fibration by lines)} is 
the intersection point~$p$ of all these lines. We also say that 
$\mathcal{L}$ is a \textsl{pencil of lines}\index{defi}{pencil of lines} through $p$, or $\mathcal{L}$ is a \textsl{foliation by 
lines}\index{defi}{foliation (by lines)} singular at $p$. Recall that a birational self map of $\mathbb{P}^2_\mathbb{C}$ that preserves a rational fibration belongs up to birational conjugacy to $\mathcal{J}$.

Let $\phi_t$ be a germ of flow in $\mathrm{Bir}_2(\mathbb{P}^2_\mathbb{C})$. Then the following 
properties hold:
\begin{itemize}
\item[$\diamond$] assume that $\phi_t$ blows down a mobile line, then 
$\phi_t$ preserves a fibration by lines, more precisely the family of 
contracted lines belongs to a fibration invariant by any element of the 
flow (\cite[Proposition 2.5, Remark 2.6]{CerveauDeserti:ptdegre});

\item[$\diamond$] there is at most one unmobile line blown down by 
$\phi_t$ (\emph{see} \cite[Lemma 2.10]{CerveauDeserti:ptdegre});

\item[$\diamond$] if $\phi_t$ blows down a unique line that is moreover
unmobile, then there exists an invariant affine chart $\mathbb{C}^2$
such that $\phi_{t\vert\mathbb{C}^2}\colon\mathbb{C}^2\to\mathbb{C}^2$ is 
polynomial for any $t$ (\emph{see} 
\cite[Proposition 2.12]{CerveauDeserti:ptdegre});

\item[$\diamond$] assume that there exists an invariant affine chart
$\mathbb{C}^2$ such that 
$\phi_{t\vert\mathbb{C}^2}\colon\mathbb{C}^2\to\mathbb{C}^2$ is polynomial 
for any $t$. Then $\phi_t$ preserves a pencil of lines. Furthermore
either $\phi_t$ is affine, or there exists a normal form for $\phi_t$ 
up to linear conjugacy (\cite[Proposition 2.15]{CerveauDeserti:ptdegre}).
\end{itemize}

Combining all these properties one can state the following result:

\begin{thm}[\cite{CerveauDeserti:ptdegre}]\label{thm:fibrationbis}
A germ of flow in $\mathrm{Bir}_2(\mathbb{P}^2_\mathbb{C})$ preserves 
a fibration by lines.
\end{thm}

Let $\phi_t$ be a quadratic birational flow, and let $\chi$ be its
infinitesimal generator. A \textsl{strong symmetry}\index{defi}{strong 
symmetry} $Y$ of $\chi$ is a rationally integrable vector field of flow 
$\psi_s$ such that
\begin{itemize}
\item[$\diamond$] $\phi_t$ and $\psi_s$ commute, {\it i.e.} $[\chi,Y]=0$,
\item[$\diamond$] $\psi_s\in\mathrm{Bir}_2(\mathbb{P}^2_\mathbb{C})$ for all $s$,
\item[$\diamond$] $\chi$ and $Y$ are not $\mathbb{C}$-colinear.
\end{itemize}

Let $\phi_t$ be a flow in $\mathrm{Bir}_2(\mathbb{P}^2_\mathbb{C})$, 
and let $\chi$ (resp. $\mathcal{F}_\chi$\index{not}{$\mathcal{F}_\chi$}) be 
the associated vector field  (resp. foliation). We denote by 
$\overline{\langle\phi_t\rangle}^Z\subset\mathrm{Bir}_2(\mathbb{P}^2_\mathbb{C})$\index{not}{$\overline{\phi_t}^Z$}
the Zariski closure of $\langle\phi_t\rangle$ in $\mathrm{Bir}_2(\mathbb{P}^2_\mathbb{C})$. Let 
$\mathrm{G}(\chi)$\index{not}{$\mathrm{G}(\chi)$} be the maximal algebraic 
abelian subgroup of $\mathrm{Bir}_2(\mathbb{P}^2_\mathbb{C})$ that contains 
$\overline{\langle\phi_t\rangle}^Z$.

\begin{thm}[\cite{CerveauDeserti:ptdegre}]\label{thm:flotalt}
Let $\phi_t$ be a germ of flow in $\mathrm{Bir}_2(\mathbb{P}^2_\mathbb{C})$, and let $\chi$ be its
infinitesimal generator.
\begin{itemize}
\item[$\diamond$] If $\dim\mathrm{G}(\chi)=1$, then $\mathcal{F}_\chi$ is
a rational fibration.

\item[$\diamond$] If $\dim\mathrm{G}(\chi)\geq 2$, then $\mathcal{F}_\chi$
has a strong symmetry.
\end{itemize}

In both cases $\mathcal{F}_\chi$ is defined by a rational closed $1$-form.
\end{thm}

\begin{proof}
Let us prove the first assertion. If $\dim\mathrm{G}(\chi)=1$, then
$\overline{\langle\phi_t\rangle}^Z$ is the component of $\mathrm{G}(\chi)$
that contains the identity. This group viewed as a Lie group is isomorphic
to $\mathbb{C}$, or $\mathbb{C}^*$, or $\faktor{\mathbb{C}}{\Lambda}$. According to
Theorem \ref{thm:fibrationbis} the group $\overline{\langle\phi_t\rangle}^Z$
preserves a fibration by lines; let us assume that this fibration is given 
by $z_1=$ constant. One yields a morphism
\[
\pi\colon\overline{\langle\phi_t\rangle}^Z\to\mathrm{PGL}(2,\mathbb{C})
\]
that describes the action of $\phi_t$ on the fibers. 

If $\pi$ is trivial ({\it i.e.} if the fibration is preserved fiberwise), 
then $\mathcal{F}_\chi=\big\{z_1=\text{ constant }\big\}$ and the result holds.

Otherwise $\overline{\langle\phi_t\rangle}^Z$ is not isomorphic to 
$\faktor{\mathbb{C}}{\Lambda}$ because there is no $\faktor{\mathbb{C}}{\Lambda}$ among 
the subgroup of $\mathrm{PGL}(2,\mathbb{C})$. Hence the topological closure
of $\overline{\langle\phi_t\rangle}^Z$ in
$\mathbb{P}^{17}_\mathbb{C}\simeq\mathrm{Rat}_2$ is a rational curve. But
according to Darboux a foliation of $\mathbb{P}^2_\mathbb{C}$ whose the 
closure of all leaves are algebraic curves has a non-constant rational
first integral (\cite{Jouanolou}). In our case the curves are rational,
so $\mathcal{F}_\chi$ is a rational fibration.

\medskip

Let us now prove the second assertion. Assume $\dim\mathrm{G}(\chi)\geq 2$.
One can find a germ of $1$-parameter group $\psi_s$ in $\mathrm{G}(\chi)$ 
not contained in $\langle\phi_t\rangle$. Let $Y$ be the infinitesimal 
generator of $\psi_s$. The vector fields $\chi$ and $Y$ commute and are
not $\mathbb{C}$-colinear. Let us consider $\omega$ a rational $1$-form 
that define $\mathcal{F}_\chi$,  {\it i.e.} $i_\chi\omega=0$. If $\chi$ and~$Y$
are generically independent, then $\Omega=\frac{\omega}{i_Y\omega}$ is 
closed and define $\mathcal{F}_\chi$. If $\chi$ and $Y$ are not 
generically independent, then $Y=f\chi$ with $f$ rational and non-constant.
Since $[\chi,Y]=0$ one has $\chi(f)=0$. As a result $\mathrm{d}f$ defines 
$\mathcal{F}_\chi$ and is closed.
\end{proof}

\begin{rem}
The last two statements can be generalized as follows:
\begin{thm}[\cite{CerveauDeserti:ptdegre}]
Let $\phi_t$ be a germ of flow in $\mathrm{Bir}_n(\mathbb{P}^2_\mathbb{C})$, and let $\chi$ be 
its infinitesimal generator. Denote by $\mathrm{G}(\chi)$ the abelian
maximal algebraic group contained in $\mathrm{Bir}_n(\mathbb{P}^2_\mathbb{C})$ and that contains
$\overline{\langle\phi_t\rangle}^Z$. Then 
\begin{itemize}
\item[$\diamond$] if $\dim\mathrm{G}(\chi)=1$, then $\mathcal{F}_\chi$
is either a rational fibration or an elliptic fibration;

\item[$\diamond$] if $\dim\mathrm{G}(\chi)\geq 2$, then $\chi$ has
a strong symmetry.
\end{itemize}

In both cases $\mathcal{F}_\chi$ is defined by a closed rational 
$1$-form.
\end{thm}

\begin{thm}[\cite{CerveauDeserti:ptdegre}]
Any germ of birational flow in $\mathrm{Bir}_n(\mathbb{P}^2_\mathbb{C})$ preserves a rational fibration.
\end{thm}
\end{rem}

\subsection{A few words about the classification of germs of quadratic 
birational flows}

Let $\phi_t$ be a germ of flow in $\mathrm{Bir}_2(\mathbb{P}^2_\mathbb{C})$; 
then $\phi_t$ preserves a fibration by lines 
(\cite[Theorem 2.16]{CerveauDeserti:ptdegre}). In other words up to linear 
conjugacy
\[
\phi_t\colon(z_0,z_1)\dashrightarrow\left(\frac{A(z_1,t)z_0+B(z_1,t)}{C(z_1,t)z_0+D(z_1,t)},\nu(z_1,t)\right)
\]
with 
\begin{itemize}
\item[$\diamond$] $\nu(z_1,t)=z_1$, or $z_1+t$, or $\mathrm{e}^{\beta t}z_1$;

\item[$\diamond$] $A$, $B$, $C$, $D$ are polynomials in $z_1$ and 
$\deg_{z_1}A\leq 1$, $\deg_{z_1}B\leq 2$, $\deg_{z_1}C=0$, $\deg_{z_1}D\leq 1$,

\item[$\diamond$] $B(z_1,0)=C(z_1,0)=0$ and $A(z_1,0)=D(z_1,0)$. 
\end{itemize}

The infinitesimal generator 
$\chi=\frac{\partial\phi_t}{\partial t}\Big\vert_{t=0}$ of $\phi_t$ can be 
written 
\[
\frac{\alpha z_0^2+\ell(z_1)z_0+P(z_1)}{az_1+b}\frac{\partial}{\partial z_0}+\varepsilon (z_1)\frac{\partial}{\partial z_1}
\]
with $\alpha$, $a$, $b\in\mathbb{C}$, $\ell$, $P\in\mathbb{C}[z_1]$, $\deg\ell=1$, $\deg P=2$
and up to linear conjugacy and scalar multiplication 
$\varepsilon\in\{0,\,1,\,z_1\}$.

The above vector fields are classified up to automorphisms of 
$\mathbb{P}^2_\mathbb{C}$ and renormalization in 
\cite[Chapter 2, \S 2]{CerveauDeserti:ptdegre}; such vector
fields are detected via the following methods:

\begin{itemize}
\item[$\diamond$] compute explicitely the flow by integration;

\item[$\diamond$] or degenerate $\chi$ on another vector
field $\chi_0$ that is not rationally integrable;

\item[$\diamond$] or show that a birational model of $\mathcal{F}_\chi$ 
has an isolated degenerate resonnant singular point (one and only 
one non-zero eigenvalue), and so $\mathcal{F}_\chi$ has no rational
first integral. Then prove that there is no strong symmetry
hence $\chi$ is not rationally integrable (Theorem \ref{thm:flotalt}).
\end{itemize}


\section{Nilpotent subgroups of the Cremona group}

In \cite{Deserti:nilpotent} are described the nilpotent subgroups of the 
plane Cremona group:

\begin{thm}[\cite{Deserti:nilpotent}]\label{thm:nilpotent}
Let $\mathrm{N}$ be a nilpotent subgroup of $\mathrm{Bir}(\mathbb{P}^2_\mathbb{C})$.
Assume that, up to finite index, $\mathrm{N}$ is not abelian. Then
\begin{itemize}
\item[$\diamond$] either $\mathrm{N}$ is a torsion group;

\item[$\diamond$] or $\mathrm{N}$ is metabelian up to finite index, {\it i.e.}
$[\mathrm{N},\mathrm{N}]$ is abelian up to finite index.
\end{itemize}
\end{thm}

\begin{egs}
Let $\alpha$ and $\beta$ be two non zero complex numbers; the group
\[
\langle(z_0,z_1)\mapsto(z_0+\alpha\beta,z_1),\,(z_0,z_1)\mapsto(z_0+\alpha z_1,z_1),\,(z_0,z_1)\mapsto(z_0,z_1+\beta)\rangle
\]
is a non-abelian, non-finite 
and nilpotent subgroup of $\mathrm{Bir}(\mathbb{P}^2_\mathbb{C})$.

If $a$ belongs to $\mathbb{C}(z_1)$, then 
\[
\langle(z_0,z_1)\mapsto(z_0+1,z_1),\,(z_0,z_1)\mapsto(z_0+z_1,z_1),\,(z_0,z_1)\mapsto(z_0+a(z_1),z_1-1)\rangle, 
\]
is a non-abelian, non-finite 
and nilpotent subgroup of $\mathrm{Bir}(\mathbb{P}^2_\mathbb{C})$.
\end{egs}

\begin{cor}[\cite{Deserti:nilpotent}]
Let $\mathrm{G}$ be a group. Assume that $\mathrm{G}$ contains a subgroup 
$\mathrm{N}$ such that 
\begin{itemize}
\item[$\diamond$] $\mathrm{N}$ is of nilpotent class $>1$,

\item[$\diamond$] $\mathrm{N}$ has no torsion,

\item[$\diamond$] $\mathrm{N}$ is not metabelian up to finite index.
\end{itemize}
Then there is no faithfull representation of $\mathrm{G}$ into 
$\mathrm{Bir}(\mathbb{P}^2_\mathbb{C})$.
\end{cor}

\begin{rem}\label{rem:distorted}
Let $\mathrm{G}$ be a nilpotent group of nilpotent class $n$. Take $f$ in
$\mathrm{G}$, $g$ in $C^{(n-2)}\mathrm{G}$ and consider $h=[f,g]\in C^{(n-1)}\mathrm{G}$. Since $\mathrm{G}$
is of nilpotent class $n$, then $[f,h]=[g,h]=\mathrm{id}$. In other words
any nilpotent group contains a distorted element.
\end{rem}

According to Remark \ref{rem:distorted} and Lemma \ref{lem:nilkeylemma}
one has:

\begin{pro}
Let $\mathrm{N}$ be a nilpotent subgroup of the plane
Cremona group. It contains a distorted element which is elliptic or 
parabolic.
\end{pro}

\begin{proof}[Idea of the proof of Theorem \ref{thm:nilpotent}]
Take $\mathrm{G}\subset\mathrm{Bir}(\mathbb{P}^2_\mathbb{C})$ a nilpotent 
subgroup of class $k$ which is not up to finite index of nilpotent class
$k-1$. Denote by $\Sigma_\mathrm{G}$ the set of finitely generated nilpotent
subgroups of $\mathrm{G}$ that are, up to finite index, not abelian. Then
\begin{itemize}
\item[$\diamond$] either any element of $\Sigma_\mathrm{G}$ is finite and 
$\mathrm{G}$ is a torsion group;

\item[$\diamond$] or $\Sigma_\mathrm{G}$ contains a non-finite element 
$\mathrm{H}$.
\end{itemize}

\begin{claim}[\cite{Deserti:IMRN}]\label{claim:nil}
The group $\mathrm{H}$ preserves a fibration $\mathcal{F}$ that is rational or
elliptic.
\end{claim}

Any element of $C^{(k-1)}\mathrm{H}$ preserves fiberwise $\mathcal{F}$.
 Let $\phi$ be in $C^{(k-1)}\mathrm{H}$. As $[\phi,\mathrm{G}]=\mathrm{id}$, then
\begin{itemize}
\item[a)] either $\phi$ preserves fiberwise two distinct fibrations; 

\item[b)] of $\mathrm{G}$ preserves fiberwise $\mathcal{F}$.
\end{itemize}

If a) holds, then $\phi$ is of finite order; if it is the case for any 
$\phi\in C^{(k-1)}\mathrm{H}$, then $\mathrm{H}$ is, up to finite index, of
nilpotent class $k-1$: contradiction.

If b) holds, then $\mathrm{G}$ is, up to finite index, metabelian. Let us 
detail why when $\mathcal{F}$ is rational. In that case $\mathrm{G}$ is,
up to conjugacy, a subgroup of the Jonqui\`eres group $\mathcal{J}$. Let 
$\mathrm{pr}_2$ be the projection $\mathcal{J}\to\mathrm{PGL}(2,\mathbb{C})$.
A non-finite nilpotent subgroup of $\mathrm{PGL}(2,\Bbbk)$, where 
$\Bbbk=\mathbb{C}$ or $\mathbb{C}(z_1)$, is up to finite index abelian.
The group $\mathrm{pr}_2(\mathrm{G})$ is thus,
up to finite index, abelian. Consequently we can assume that 
$\mathrm{pr}_2(C^{(i)}\mathrm{G})=\{\mathrm{id}\}$ for $1\leq i\leq k$.
In particular $C^{(1)}\mathrm{G}$ is a nilpotent subgroup of 
$\mathrm{PGL}(2,\mathbb{C}(z_1))$ and as a result is, up to finite 
index, abelian.
\end{proof}

\begin{proof}[Idea of the proof of the Claim \ref{claim:nil}]
Let us recall that $\mathrm{H}$ is a non-finite nilpotent subgroup of
$\mathrm{Bir}(\mathbb{P}^2_\mathbb{C})$ with the following properties:
\begin{itemize}
\item[$\diamond$] $\mathrm{H}$ is finitely generated,

\item[$\diamond$] $\mathrm{H}$ is nilpotent of class $k>0$, 

\item[$\diamond$] $\mathrm{H}$ is not, up to finite index, of nilpotent
class $k-1$.
\end{itemize}

Assume $C^{(k-1)}\mathrm{H}$ is not a torsion group. 
Then $\mathrm{H}$ preserves a fibration that is rational or elliptic.
According to Lemma \ref{lem:nilkeylemma} a non-trivial element of $C^{(k-1)}\mathrm{G}$ either preserves a 
unique fibration $\mathcal{F}$ that is rational or elliptic, or  
is an elliptic birational map. We have the following alternative:
\begin{itemize}
\item[a)] either $C^{(k-1)}\mathrm{G}$ contains an element $h$ 
that preserves a unique fibration $\mathcal{F}$,

\item[b)] or any element of
$C^{(k-1)}\mathrm{G}\smallsetminus\{\mathrm{id}\}$
is elliptic.
\end{itemize}

Let us look at these eventualities:

\begin{itemize}
\item[a)] Since $[h,\mathrm{G}]=\{\mathrm{id}\}$ any element 
of $\mathrm{G}$
preserves $\mathcal{F}$. 

\item[b)] The group $C^{(k-1)}\mathrm{G}$ 
is finitely generated and abelian. Let $\big\{a_1,\,a_2,\,\ldots,\,a_n\big\}$
be a genera\-ting set of $C^{(k-1)}\mathrm{G}$. The $a_i$'s are elliptic maps,
so there exist a surface $S_i$, a birational map 
$\eta_i\colon S_i\dashrightarrow \mathbb{P}^2_\mathbb{C}$ and an integer 
$k_i>0$ such that $\eta_i^{-1}\circ a_i^{k_i}\circ\eta_i$ belongs to the 
neutral component $\mathrm{Aut}(S_i)^0$ of $\mathrm{Aut}(S_i)$. In 
particular the $a_i$'s fix any curve of negative self-intersection, we 
can thus assume that 
$S_i$ is a minimal rational surface. A priori all the $S_i$ are distinct.
Nevertheless according to Proposition \ref{pro:commas} there exist a 
minimal rational surface $S$, a birational 
map $\eta\colon S\dashrightarrow\mathbb{P}^2_\mathbb{C}$ 
and an integer 
$k>0$ such that for any $1\leq i\leq n$ the map
$\eta^{-1}\circ a_i^{k}\circ\eta$ belongs to the 
neutral component $\mathrm{Aut}(S)^0$ of $\mathrm{Aut}(S)$. 

Minimal rational surfaces are
$\mathbb{P}^2_\mathbb{C}$, $\mathbb{P}^1_\mathbb{C}\times\mathbb{P}^1_\mathbb{C}$
and Hirzebruch sur\-faces~$\mathbb{F}_n$, $n\geq 2$. Using 

\begin{itemize}
\item[$\diamond$] on the one hand 
the
description of the automorphisms groups of minimal rational surfaces 
(\emph{see} Chapter \ref{Chapter:algebraicsubgroup}), 
\item[$\diamond$] and on the other hand the fact that if $\mathrm{K}$ is an 
algebraic Lie subgroup of $\mathrm{GL}(n,\mathbb{C})$, then the semi-simple 
and nilpotent parts of any element of $\mathrm{K}$ belong to $\mathrm{K}$,
\end{itemize}
we prove that $\mathrm{G}$ is, up to finite index and up to conjugacy, 
contained in the Jonqui\`eres group $\mathcal{J}$ (\emph{see} 
\cite{Deserti:nilpotent}).
\end{itemize}

\medskip

It remains to consider the case "$C^{(k-1)}\mathrm{G}$ 
is a torsion group"; the ideas are similar 
(\emph{see} \cite[Proposition 4.5]{Deserti:nilpotent}).
\end{proof}

\section{Centralizers in $\mathrm{Bir}(\mathbb{P}^2_\mathbb{C})$}\label{sec:cent}

\subsection{Centralizers of elliptic birational maps}\label{subsec:centrell}

We will focus on the case of birational self maps of 
$\mathbb{P}^2_\mathbb{C}$ of infinite order. Note for 
instance that for birational self map of 
$\mathbb{P}^2_\mathbb{C}$ of finite order the situation 
is wild: consider for instance a birational involution 
$\phi$ of $\mathbb{P}^2_\mathbb{C}$. If $\phi$ is 
conjugate to an automorphism of $\mathbb{P}^2_\mathbb{C}$, 
then the centralizer of $\phi$ in 
$\mathrm{Bir}(\mathbb{P}^2_\mathbb{C})$ is uncountable but 
if~$\phi$ is conjugate to a Bertini (or a Geiser) 
involution, then the centralizer is finite (\cite{BlancPanVust}).

According to \cite{BlancDeserti:degree} an elliptic 
birational self map of $\mathbb{P}^2_\mathbb{C}$ of 
infinite order is conjugate to an automorphism of 
$\mathbb{P}^2_\mathbb{C}$ which restricts to one of
the following automorphisms on some open subset 
isomorphic to $\mathbb{C}^2$:
\begin{itemize}
\item[$\diamond$] $(z_0,z_1)\mapsto(\alpha z_0,\beta z_1)$ 
where $\alpha$, $\beta$ belong to $\mathbb{C}^*$ and where 
the kernel of the group homomorphism
\begin{align*}
& \mathbb{Z}^2\to\mathbb{C}^2 && (i,j)\mapsto \alpha^i\beta^j
\end{align*}
is generated by $(k,0)$ for some $k\in\mathbb{Z}$;
\item[$\diamond$] $(z_0,z_1)\mapsto(\alpha z_0,z_1+1)$
where $\alpha\in\mathbb{C}^*$.
\end{itemize}

We can describe the centralizers of such maps; let us 
start with the centralizer of 
$(z_0,z_1)\mapsto(\alpha z_0,\beta z_1)$ where $\alpha$, 
$\beta$ belong to $\mathbb{C}^*$ and where the kernel 
of the group homomorphism
\begin{align*}
& \mathbb{Z}^2\to\mathbb{C}^2 && (i,j)\mapsto \alpha^i\beta^j
\end{align*}
is generated by $(k,0)$ for some $k\in\mathbb{Z}$. 
Recall that $\mathrm{PGL}(2,\mathbb{C})$ is the group 
of automorphisms of~$\mathbb{P}^1_\mathbb{C}$ or 
equivalently the group of M\"obius 
transformations 
\[
z_0\dashrightarrow\frac{az_0+b}{cz_0+d}
\]
A direct computation implies the following: for 
any $\alpha\in\mathbb{C}^*$
\[
\big\{\eta\in\mathrm{PGL}(2,\mathbb{C})\,\vert\,\eta(\alpha z_0)=\alpha\eta(z_0)\big\}=\left\{
\begin{array}{lll}
\mathrm{PGL}(2,\mathbb{C}) \text{ if $\alpha=1$}\\
\big\{z_0\dashrightarrow \gamma z_0^{\pm 1}\,\vert\,\gamma\in\mathbb{C}^*\big\} \text{ if $\alpha=-1$}\\
\big\{z_0\mapsto\gamma z_0\,\vert\,\gamma\in\mathbb{C}^*\big\} \text{ if $\alpha^2\not=1$}
\end{array}
\right.
\]

\begin{lem}
Let us consider 
$\phi\colon(z_0,z_1)\mapsto(\alpha z_0,\beta z_1)$ 
where $\alpha$, $\beta$ belongs to $\mathbb{C}^*$
and where the kernel of the group homomorphism
\begin{align*}
& \mathbb{Z}^2\to\mathbb{C}^2 && (i,j)\mapsto \alpha^i\beta^j
\end{align*}
is generated by $(k,0)$ for some $k\in\mathbb{Z}$.

The centralizer of $\phi$ in 
$\mathrm{Bir}(\mathbb{P}^2_\mathbb{C})$ is
\[
\big\{(z_0,z_1)\dashrightarrow(\eta(z_0),z_1a(z_0^k))\,\vert\, a\in\mathbb{C}(z_0),\,\eta\in\mathrm{PGL}(2,\mathbb{C}),\,\eta(\alpha z_0)=\alpha\eta(z_0)\big\}.
\]
\end{lem}

\begin{proof}
Let 
$\psi\colon(z_0,z_1)\dashrightarrow(\psi_0(z_0,z_1),\psi_1(z_0,z_1))$
be a birational self map of $\mathbb{P}^2_\mathbb{C}$ 
that commutes with $\phi$. Then 
\begin{equation}\label{eqet}
\psi_0(\alpha z_0,\beta z_1)=\alpha\psi_0(z_0,z_1)
\end{equation}
and 
\begin{equation}\label{eqtri}
\psi_1(\alpha z_0,\beta z_1)=\beta\psi_1(z_0,z_1)
\end{equation}
hold. Denote by $\phi^*$ the linear automorphism of
the $\mathbb{C}$-vector space $\mathbb{C}[z_0,z_1]$
given by 
\[
\phi^*\colon \varphi(z_0,z_1)\mapsto\varphi(\alpha z_0,\beta z_1).
\]
Let us write $\psi_i$ as $\frac{P_i}{Q_i}$ for $i=0$, 
$1$ where $P_i$, $Q_i$ are polynomials without common
factor. Note that $P_0$, $P_1$, $Q_0$, $Q_1$ are 
eigenvectors of $\phi^*$, {\it i.e.} any of the 
$P_i$, $Q_i$ is a product of a monomial in $z_0$, 
$z_1$ with an element of $\mathbb{C}[z_0^k]$. 
Using (\ref{eqet}) and (\ref{eqtri}) we get that
\[
\left\{
\begin{array}{ll}
\psi_0(z_0,z_1)=z_0a_0(z_0^k)\\
\psi_1(z_0,z_1)=z_1a_1(z_0^k)
\end{array}
\right.
\]
But $\psi$ is birational, so $\psi_0$ belongs to $\mathrm{PGL}(2,\mathbb{C})$. Furthemore $\psi_0$
satisfies $\psi_0(\alpha z_0)=\alpha\psi_0(z_0)$.
\end{proof}

Let us now deal with the other possibility:

\begin{lem}
Let $\phi$ be the automorphism of 
$\mathbb{P}^2_\mathbb{C}$ given by 
\[
\phi\colon(z_0,z_1)\mapsto(\alpha z_0,z_1+\beta)
\]
where $\alpha\in\mathbb{C}^*$, $\beta\in\mathbb{C}$. 
The centralizer of $\phi$ in
$\mathrm{Bir}(\mathbb{P}^2_\mathbb{C})$
is 
\begin{small}
\[
\big\{(z_0,z_1)\dashrightarrow(\eta(z_0),z_1+a(z_0))\,\vert\,\eta\in\mathrm{PGL}(2,\mathbb{C}),\,\eta(\alpha z_0)=\alpha\eta(z_0),\,a\in\mathbb{C}(z_0),\,a(\alpha z_0)=a(z_0)\big\}
\]
\end{small}
\end{lem}

\begin{proof}
After conjugacy by $(z_0,z_1)\mapsto(z_0,\beta z_1)$ 
we can assume that $\beta=1$. 

If 
$\psi\colon(z_0,z_1)\dashrightarrow(\psi_0(z_0,z_1),\psi_1(z_0,z_1))$
is a birational map that commutes with $\phi$,
then 
\begin{equation}\label{eq:equ1}
\psi_0(\alpha z_0,z_1+1)=\alpha\psi_0(z_0,z_1)
\end{equation}
and 
\begin{equation}\label{eq:equ2}
\psi_1(\alpha z_0,z_1+1)=\psi_1(z_0,z_1)+1
\end{equation}
From (\ref{eq:equ1}) and \cite{Blanc:manuscripta}
we get that $\psi_0$ only depends on $z_0$. 
Hence $\psi_0$ belongs to $\mathrm{PGL}(2,\mathbb{C})$
and commutes with $z_0\mapsto\alpha z_0$.
From (\ref{eq:equ2}) we get 
\[
\left\{
\begin{array}{ll}
\frac{\partial\psi_1}{\partial z_1}(\alpha z_0,z_1+1)=\frac{\partial\psi_1}{\partial z_1}(z_0,z_1)\\
\frac{\partial\psi_1}{\partial z_0}(\alpha z_0,z_1+1)=\frac{1}{\alpha}\frac{\partial\psi_1}{\partial z_0}(z_0,z_1)
\end{array}
\right.
\]
which again means that both 
$\frac{\partial\psi_1}{\partial z_0}$ and
$\frac{\partial\psi_1}{\partial z_1}$ 
only depend on $z_0$. Therefore, 
$\psi_1\colon(z_0,z_1)\mapsto \gamma z_1+b(z_0)$ 
with $\gamma\in\mathbb{C}^*$ and $b\in\mathbb{C}(z_0)$. 
Then (\ref{eq:equ2}) can be rewritten
\[
b(\alpha z_0)=b(z_0)+1-\gamma
\]
which implies that
\[
\frac{\partial b}{\partial z_0}(\alpha z_0)=\frac{1}{\alpha}\frac{\partial b}{\partial z_0}(z_0)
\]
and that $z_0\frac{\partial b}{\partial z_0}(z_0)$ 
is invariant under $z_0\mapsto\alpha z_0$.

If $\alpha$ is not a root of unity, then 
$\frac{\partial b}{\partial z_0}=\frac{\delta}{z_0}$ 
for some $\delta\in\mathbb{C}$. As $b$ is rational, $\delta$ 
is zero and $b$ is constant. As a consequence 
$b(\alpha z_0)=b(z_0)+1-\gamma$ implies $\gamma=1$, that is
$\psi_1\colon(z_0,z_1)\dashrightarrow z_0+\beta$.

Assume that $\alpha$ is a primitive $k$-th root of 
unity. The map 
$\psi\colon(z_0,z_1)\dashrightarrow(\eta(z_0),\gamma z_1+b(z_0))$
commutes with 
\[
\phi^k\colon(z_0,z_1)\dashrightarrow(z_0,z_1+k)
\]
if and only if $\gamma(z_1+k)+b(z_0)=\gamma z_1+b(z_0)+k$, 
{\it i.e.} if and only if $\gamma=1$. Then 
$b(\alpha z_0)=b(z_0)+1-1$ can be rewritten 
$b(z_0)=b(\alpha z_0)$.
\end{proof}

\subsection{Centralizers of Jonqui\` eres twists}

Recall that the subgroup $\mathcal{J}$
of Jonqui\`eres maps is isomorphic
to 
$\mathrm{PGL}(2,\mathbb{C}(z_1))\rtimes\mathrm{PGL}(2,\mathbb{C})$.
Let us denote by $\mathrm{pr}_2$ the morphism
\[
\mathrm{pr}_2\colon\mathcal{J}\to\mathrm{PGL}(2,\mathbb{C}).
\]
Geometrically it corresponds to look at the action 
of $\phi\in\mathcal{J}$ on the basis of the invariant
fibration $z_1=$ cst. The kernel of $\mathrm{pr}_2$, 
{\it i.e.} the elements of $\mathcal{J}$ which preserve
the fibration $z_1=$ cst fiberwise, is a normal subgroup
$\mathcal{J}_0\simeq\mathrm{PGL}(2,\mathbb{C}(z_1))$ of 
$\mathcal{J}$. Up to a birational conjugacy an element 
$\phi$ of $\mathcal{J}_0$ is of one of the following form
(\cite{Deserti:abelien})
\begin{align*}
& (z_0,z_1)\dashrightarrow(z_0+a(z_1),z_1), && (z_0,z_1)\dashrightarrow(b(z_1)z_0,z_1), \\
& (z_0,z_1)\dashrightarrow\left(\frac{c(z_1)z_0+F(z_1)}{z_0+c(z_1)},z_1\right) &&
\end{align*}
with $a\in\mathbb{C}(z_1)$, $b\in\mathbb{C}(z_1)^*$,
 $c\in\mathbb{C}(z_1)$, $F\in\mathbb{C}[z_1]$ and 
 $F$ not a square. Still according to 
 \cite{Deserti:abelien} the non-finite maximal 
 abelian subgroups of $\mathcal{J}_0$ are
 \begin{eqnarray*}
 \mathcal{J}_a&=&\big\{(z_0,z_1)\dashrightarrow(z_0+a(z_1))\,\vert\,a\in\mathbb{C}(z_1)\big\}\\
 \mathcal{J}_m&=&\big\{(z_0,z_1)\dashrightarrow(b(z_1)z_0,z_1)\,\vert\,a\in\mathbb{C}(z_1)\big\}\\
 \mathcal{J}_F&=&\big\{(z_0,z_1)\dashrightarrow\left(\frac{c(z_1)z_0+F(z_1)}{z_0+c(z_1)},z_1\right)\,\vert\,a\in\mathbb{C}(z_1)\big\}
\end{eqnarray*}
where $F$ denotes an element of 
$\mathbb{C}[z_1]$ which is not a square.
Note that we can assume up to conjugacy that 
$F$ is a polynomial with roots of multiplicity
$1$.

If $\phi$ belongs to $\mathcal{J}_0$, let us denote
by $\mathrm{Ab}(\phi)$ the non-finite maximal 
abelian subgroup of $\mathcal{J}_0$ that contains
$\phi$. Up to conjugacy
\begin{itemize}
\item[$\diamond$] either 
$(z_0,z_1)\dashrightarrow(z_0+a(z_1),z_1)$ and 
$\mathrm{Ab}(\phi)=\mathcal{J}_a$;

\item[$\diamond$] or
$(z_0,z_1)\dashrightarrow(b(z_1)z_0,z_1)$ and 
$\mathrm{Ab}(\phi)=\mathcal{J}_m$;

\item[$\diamond$] or
$(z_0,z_1)\dashrightarrow\left(\frac{c(z_1)z_0+F(z_1)}{z_0+c(z_1)},z_1\right)$ and 
$\mathrm{Ab}(\phi)=\mathcal{J}_F$.
\end{itemize}

\begin{pro}[\cite{CerveauDeserti:centralisateurs}]\label{pro:centrfib}
Let $\phi$ be an element of $\mathcal{J}_0$ that is 
a Jonqui\`eres twist. Then the centralizer
of $\phi$ in $\mathrm{Bir}(\mathbb{P}^2_\mathbb{C})$
is contained in $\mathcal{J}$.
\end{pro}

\begin{proof}
Consider a birational self map 
$\varphi\colon(z_0,z_1)\dashrightarrow(\varphi_0(z_0,z_1),\varphi_1(z_0,z_1))$
of $\mathbb{P}^2_\mathbb{C}$ that commutes with $\phi$.
If $\varphi$ does not belong to $\mathcal{J}$, then 
$\varphi_1=$ cst is a fibration invariant by $\phi$
distinct from $z_1=$ cst. Then $\phi$ is of finite
order (Lemma \ref{lem:2fibr}): contradiction with 
the fact that $\phi$ is a Jonqui\`eres
twist.
\end{proof}

\subsubsection{Centralizers of elements of $\mathcal{J}_a$}

Note that elements of $\mathcal{J}_a$ are not 
Jonqui\`eres twists but elliptic maps.
Hence their centralizers are described in \S
\ref{subsec:centrell}. Let us give some 
details.
Let 
$\phi\colon(z_0,z_1)\dashrightarrow(z_0+a(z_1),z_1)$ 
be a non-trivial element of $\mathcal{J}_a$ 
({\it i.e.} $a\not\equiv 0$). Up to conjugacy
by $(z_0,z_1)\dashrightarrow(a(z_1)z_0,z_1)$ 
one can assume that $a\equiv 1$. The centralizer
of 
$(z_0,z_1)\dashrightarrow(z_0+1,z_1)$ is 
isomorphic to 
$\mathcal{J}_a\rtimes\mathrm{PGL}(2,\mathbb{C})$
(\emph{see} \S\ref{subsec:centrell}). Hence 

\begin{cor}
The centralizer of a non-trivial element of 
$\mathcal{J}_a$ is isomorphic to
$\mathcal{J}_a\rtimes\mathrm{PGL}(2,\mathbb{C})$.
\end{cor}

\subsubsection{centralizers of twists of $\mathcal{J}_m$}

An element $\phi$ of $\mathcal{J}_m$ is a 
Jonqui\`eres twist if and only if up to
birational conjugacy 
\[
(z_0,z_1)\dashrightarrow(a(z_1)z_0,z_1)
\]
with $a\in\mathbb{C}(z_1)\smallsetminus\mathbb{C}^*$.

Remark that if $a$ belongs to $\mathbb{C}^*$, then 
$(z_0,z_1)\dashrightarrow(az_0,z_1)$ is an elliptic
map whose centralizer is described in \S 
\ref{subsec:centrell}. Assume now that 
$\phi\in\mathcal{J}_m$ is a Jonqui\`eres
twist. Let 
$a\in\mathbb{C}(z_1)\smallsetminus\mathbb{C}^*$. 
Denote by 
\[
\mathrm{Stab}(a)=\big\{\nu\in\mathrm{PGL}(2,\mathbb{C})\,\vert\, a(\nu(z_1))=a(z_1)^{\pm 1}\big\}
\]
the subgroup of $\mathrm{PGL}(2,\mathbb{C})$
and by 
\[
\mathrm{stab}(a)=\big\{\nu\in\mathrm{PGL}(2,\mathbb{C})\,\vert\, a(\nu(z_1))=a(z_1)\big\}
\]
the normal subgroup of $\mathrm{Stab}(a)$. Consider 
also 
\[
\overline{\mathrm{stab}(a)}=\big\{(z_0,\nu(z_1))\,\vert\, \nu\in\mathrm{stab}(a)\big\}
\]
and $\overline{\mathrm{Stab}(a)}$ the group generated by 
$\overline{\mathrm{stab}(a)}$ and the elements 
\[
(z_0,z_1)\dashrightarrow\left(\frac{1}{z_0},\nu(z_1)\right)
\]
with $\nu\in\mathrm{Stab}(a)\smallsetminus\mathrm{stab}(a)$.

\begin{pro}[\cite{CerveauDeserti:centralisateurs}]
Let $\phi$ be a Jonqui\`eres twist in 
$\mathcal{J}_m$. The centralizer of~$\phi$ in 
$\mathrm{Bir}(\mathbb{P}^2_\mathbb{C})$ is 
$\mathcal{J}_m\rtimes\overline{\mathrm{Stab}(a)}$; 
in particular it is a finite extension of 
$\mathrm{Ab}(\phi)=\mathcal{J}_m$.
\end{pro}

\begin{rem}
One can write $\phi$ as 
$(z_0,z_1)\dashrightarrow(a(z_1)z_0,z_1)$ with 
$a\in\mathbb{C}(z_1)\smallsetminus\mathbb{C}^*$.
For generic~$a$ the group $\overline{\mathrm{Stab}(a)}$ 
is trivial, so for generic $\phi\in\mathcal{J}_m$ 
the centralizer of~$\phi$ in 
$\mathrm{Bir}(\mathbb{P}^2_\mathbb{C})$ coincides 
with $\mathcal{J}_m=\mathrm{Ab}(\phi)$.
\end{rem}

\begin{proof}
Write $\phi$ as 
$(z_0,z_1)\dashrightarrow(a(z_1)z_0,z_1)$
with $a\in\mathbb{C}(z_1)\smallsetminus\mathbb{C}^*$.
If $\psi$ commutes with $\phi$, then $\psi$
preserves the fibration $z_1=$ cst (Proposition 
\ref{pro:centrfib}), {\it i.e.}
\[
\psi\colon(z_0,z_1)\dashrightarrow\left(\frac{A(z_1)z_0+B(z_1)}{C(z_1)z_0+D(z_1)},\nu(z_1)\right)
\]
with $\left(\begin{array}{cc}
A & B\\
C& D
\end{array}
\right)\in\mathrm{PGL}(2,\mathbb{C}(z_1))$ and
$\nu\in\mathrm{PGL}(2,\mathbb{C})$. Since $\psi$
and $\phi$ commute, the following hold
\[
\left\{
\begin{array}{ll}
A(z_1)C(z_1)\big(1-a(\nu(z_1))\big)=0\\
B(z_1)D(z_1)\big(1-a(\nu(z_1))\big)=0
\end{array}
\right.
\]
Therefore, $AC\equiv0$ and $BD\equiv 0$, 
{\it i.e.} $B=C=0$ or $A=D=0$.

Assume first that $B=C=0$, {\it i.e.} that 
\[
\psi\colon(z_0,z_1)\dashrightarrow(A(z_1)z_0,\nu(z_1)).
\]
The condition $\phi\circ\psi=\psi\circ\phi$ implies
$a(\nu(z_1))=a(z_1)$. As 
$\overline{\mathrm{stab}(a)}$ is contained in the 
centralizer of $\phi$ in 
$\mathrm{Bir}(\mathbb{P}^2_\mathbb{C})$ the map 
$\phi$ belongs to 
$\mathcal{J}_m\rtimes\overline{\mathrm{stab}(a)}$.

Suppose now that $A=D=0$, {\it i.e.} that 
$\psi\colon(z_0,z_1)\dashrightarrow\left(\frac{B(z_1)}{z_0},\nu(z_1)\right)$. The equality $\psi\circ\varphi=\varphi\circ\psi$ implies
that $a(\nu(z_1))=a(z_1)^{-1}$. But
$\overline{\mathrm{Stab}(a)}$ is contained in 
the centralizer of $\phi$ in 
$\mathrm{Bir}(\mathbb{P}^2_\mathbb{C})$, so 
$\psi$ belongs to 
$\mathcal{J}_m\rtimes\overline{\mathrm{Stab}(a)}$.
\end{proof}

\subsubsection{Centralizers of elements of $\mathcal{J}_F$}

Let $\phi$ be a twist in $\mathcal{J}_F$. Let 
us write $\phi$ as 
\[
(z_0,z_1)\dashrightarrow\left(\frac{c(z_1)z_0+F(z_1)}{z_0+c(z_1)},z_1\right)
\]
with $c\in\mathbb{C}(z_1)^*$ and $F\in\mathbb{C}[z_1]$ 
whose roots have multiplicity $1$. The curve
$\mathcal{C}$ of fixed points of $\phi$ is given
by $z_0^2=F(z_1)$. Since $F$ has simple roots one
has
\[
\left\{\begin{array}{lll}
\text{$\mathcal{C}$ is rational when 
$1\leq\deg F\leq 2$;}\\
\text{the genus of $\mathcal{C}$ is 
$1$ when $3\leq\deg F\leq 4$;}\\
\text{the genus of $\mathcal{C}$ is
$\geq 2$ when $\deg F\geq 5$.}
\end{array}
\right.
\]

\begin{itemize}
\item[$\diamond$] Assume first that the genus of 
$\mathcal{C}$ is positive.

\begin{lem}[\cite{CerveauDeserti:centralisateurs}]
Let 
\begin{align*}
&\phi\colon(z_0,z_1)\dashrightarrow\left(\frac{c(z_1)z_0+F(z_1)}{z_0+c(z_1)},z_1\right) && c\in\mathbb{C}(z_1)^*, \,F\in\mathbb{C}[z_1]
\end{align*}
be a twist in $\mathcal{J}_F$. The 
curve $z_0^2=F(z_1)$ and the fibers
$z_1=$cst are invariant and there is
no other invariant curves.
\end{lem}

\begin{proof}
The map $\phi$ has two fixed points on a generic 
fiber which correspond to the intersection of the
fiber with the curve $z_0^2=F(z_1)$. Assume by 
contradiction that there is an other invariant
curve $\mathcal{C}$. The curve $\mathcal{C}$
intersects a generic fiber in a finite number 
of points that are invariant by $\phi$. But a
M\"obius transformation that preserves 
more than three points is periodic: contradiction
with the fact that $\phi$ is a Jonqui\`eres
twist, so of infinite order.
\end{proof}

\begin{pro}[\cite{CerveauDeserti:centralisateurs}]
Let 
\begin{align*}
&\phi\colon(z_0,z_1)\dashrightarrow\left(\frac{c(z_1)z_0+F(z_1)}{z_0+c(z_1)},z_1\right) && c\in\mathbb{C}(z_1)^*, \,F\in\mathbb{C}[z_1]
\end{align*}
be a twist in $\mathcal{J}_F$. Assume that $F$ 
has only simple roots and $\deg F\geq 3$, {\it i.e.}
the curve $z_0^2=F(z_1)$ has genus $\geq 1$. 
Then the centralizer of $\phi$ in 
$\mathrm{Bir}(\mathbb{P}^2_\mathbb{C})$ is 
a finite extension of 
$\mathrm{Ab}(\phi)=\mathcal{J}_F$.
\end{pro}

\begin{proof}
Take $\alpha\in\mathbb{C}$ such that $F(\alpha)\not=0$.
The restriction $\phi_{\vert z_1=\alpha}$ of~$\phi$
on the fiber $z_1=\alpha$ has two fixed points: 
$(\pm\sqrt{F(\alpha)},\alpha)$. Note that the centralizer
$\mathrm{Cent}(\phi)$ of $\phi$ in 
$\mathrm{Bir}(\mathbb{P}^2_\mathbb{C})$ is contained 
in $\mathcal{J}$ (Proposition \ref{pro:centrfib}). We 
will focus on elements $\psi$ of $\mathrm{Cent}(\phi)$ 
that preserve the fibration $z_1=$cst fiberwise, {\it i.e.}
on the kernel of 
\[
\mathrm{pr}_{2\vert\mathrm{Cent}(\phi)}\colon\mathrm{Cent}(\phi)\to\mathrm{PGL}(2,\mathbb{C}).
\]
Remark that any $\psi\in\mathrm{Cent}(\phi)$ preserves
$\mathcal{C}$ and that the automorphism 
$\psi_{\vert\mathcal{C}}$ of $\mathcal{C}$ preserves
$\big\{(\pm\sqrt{F(\alpha)},\alpha)\big\}$. Hence either 
$\psi_{\vert\mathcal{C}}=\mathrm{id}$, that is 
$\psi\in\mathcal{J}_F$, or~$\psi_{\vert\mathcal{C}}$
is the involution $(z_0,z_1)\mapsto(-z_0,z_1)$ of 
$\mathcal{C}$. Note that the restriction of 
$\tau\colon(z_0,z_1)\dashrightarrow\left(-\frac{F(z_1)}{z_0},z_1\right)$ 
to $\mathcal{C}$ is 
$\tau_{\vert\mathcal{C}}\colon(z_0,z_1)\dashrightarrow(-z_0,z_1)$.
Therefore, any birational self map of 
$\mathbb{P}^2_\mathbb{C}$ that preserves both 
$\mathcal{C}$ and the fibration $z_1=$cst fiberwise
belongs either to $\mathcal{J}_F$ or to 
$\tau\circ\mathcal{J}_F$. But 
$\tau\circ\phi\circ\tau^{-1}=\tau\circ\phi\circ\tau=\phi^{-1}$,
so $\tau$ does not belong to $\mathrm{Cent}(\phi)$. As 
a result 
$\ker\mathrm{pr}_{2\vert\mathrm{Cent}(\phi)}=\mathcal{J}_F$.
Any $\varphi\in\mathrm{Cent}(\phi)$ has to preserve 
$\mathcal{C}$ and the fibration $z_1=$cst; the restriction 
$\varphi_{\vert\mathcal{C}}$ of $\varphi$ to $\mathcal{C}$ 
is an automorphism of $\mathcal{C}$ that commutes 
with the involution $\tau_{\vert\mathcal{C}}$. The 
group $\mathrm{Aut}_\tau(\mathcal{C})$ of such 
automorphisms is a finite group (more precisely if $F$
is generic, then 
$\mathrm{Aut}_\tau(\mathcal{C})=\{\mathrm{id},\,\tau_{\vert\mathcal{C}}\}$).
\end{proof}

\item[$\diamond$] Assume that $\mathcal{C}$ is
rational.

\begin{lem}[\cite{CerveauDeserti:centralisateurs}]
Let $\phi\in\mathcal{J}_F$ be a Jonqui\`eres
twist such that the curve $\mathcal{C}$ of fixed points 
of $\phi$ is rational. Any element that commutes with~$\phi$
belongs to $\mathcal{J}$ and preserves $\mathcal{C}$.
\end{lem}

\begin{proof}
The curve of fixed points of $\phi$ is given by 
$z_0^2=F(z_1)$. Let~$\psi$ be a birational self 
map of $\mathbb{P}^2_\mathbb{C}$ such that 
$\phi\circ\psi=\psi\circ\phi$. According to 
Proposition \ref{pro:centrfib} the map $\psi$
preserves the fibration $z_1=$cst. Either $\psi$
contracts $\mathcal{C}$ or $\psi$ preserves 
$\mathcal{C}$. But $\mathcal{C}$ is transverse 
to the fibration $z_1=$cst, so $\psi$ can not 
contract $\mathcal{C}$. As a result $\varphi$ 
is an element of $\mathcal{J}$ that 
preserves~$\mathcal{C}$.
\end{proof}

Note that the case $\deg F\geq 3$ has already 
been studied, so let us assume that $\deg F\leq 2$.
Remark that if 
\[
\phi\colon(z_0,z_1)\dashrightarrow\left(\frac{c(z_1)z_0+z_1}{z_0+c(z_1)},z_1\right)
\]
and if
\[
\varphi\colon(z_0,z_1)\dashrightarrow\left(\frac{z_0}{\gamma z_1+\delta},\frac{\alpha z_1+\beta}{\gamma z_1+\delta}\right)
\]
then $\varphi^{-1}\circ\phi\circ\varphi$ is of the following type
\[
(z_0,z_1)\dashrightarrow\left(\frac{\widetilde{c}(z_1)z_0+(\alpha z_1+\beta)(\gamma z_1+\delta)}{z_0+\widetilde{c}(z_1)},z_1\right).
\]
In other words thanks to 
\[
(z_0,z_1)\dashrightarrow\left(\frac{c(z_1)z_0+z_1}{z_0+c(z_1)},z_1\right)
\]
we obtain all polynomials 
$(\alpha z_1+\beta)(\gamma z_1+\delta)$ of degree
$2$ with simple roots. So one can suppose that 
$\deg F=1$. Note that if $\deg F=1$, {\it i.e.} 
$F(z_1)=\alpha z_1+\beta$, then up to conjugacy 
by 
$(z_0,z_1)\mapsto\left(z_0,\frac{z_1-\beta}{\alpha}\right)$ 
one can assume that $F\colon z_1\mapsto z_1$.

\begin{lem}[\cite{CerveauDeserti:centralisateurs}]\label{lem:stab}
Consider the birational self map of 
$\mathbb{P}^2_\mathbb{C}$ given by 
\[
\phi\colon(z_0,z_1)\dashrightarrow\left(\frac{c(z_1)z_0+z_1}{z_0+c(z_1)},z_1\right)
\]
with $c\in\mathbb{C}(z_1)^*$. If $\psi$ is a 
birational self map of $\mathbb{P}^2_\mathbb{C}$ 
that commutes with~$\phi$, then 
\begin{itemize}
\item[$\diamond$] either 
$\mathrm{pr}_2(\psi)=\frac{\alpha}{z_1}$ with
$\alpha\in\mathbb{C}^*$;

\item[$\diamond$] or $\mathrm{pr}_2(\psi)=\zeta z_1$
with $\zeta$ root of unity.
\end{itemize}

Furthermore $\mathrm{pr}_2(\psi)$ belongs to the 
finite group 
$\mathrm{stab}\left(\frac{4c^2(z_1)}{c^2(z_1)-z_1}\right)$.
\end{lem}

For any $\alpha$ non-zero consider the dihedral 
group
\[
\mathrm{D}_\infty(\alpha)=\langle z_1\mapsto \frac{\alpha}{z_1},\,z_1\mapsto\zeta z_1\,\vert\,\zeta\text{ root of unity}\rangle
\]
Note that all the $\mathrm{D}_\infty(\alpha)$
are conjugate to $\mathrm{D}_\infty(1)$.

\begin{pro}[\cite{CerveauDeserti:centralisateurs}]
Let $\phi\in\mathcal{J}_F$ be a Jonqui\`eres 
twist such that the fixed curve of $\phi$ is 
rational. Up to conjugacy we can assume that 
\[
\phi\colon(z_0,z_1)\dashrightarrow\left(\frac{c(z_1)z_0+z_1}{z_0+c(z_1)},z_1\right)
\]
with $c\in\mathbb{C}(z_1)\smallsetminus\mathbb{C}$.
The centralizer of $\phi$ in 
$\mathrm{Bir}(\mathbb{P}^2_\mathbb{C})$ is 
\[
\mathcal{J}_{z_1}\rtimes\left(\mathrm{stab}\left(\frac{4c^2(z_1)}{c^2(z_1)-z_1}\right)\cap\mathrm{D}_\infty(\alpha)\right)
\]
for some $\alpha\in\mathbb{C}^*$.
\end{pro}

\begin{proof}
Denote by $\mathrm{Cent}(\phi)$ the centralizer of 
$\phi$ in $\mathrm{Bir}(\mathbb{P}^2_\mathbb{C})$,
and by $\mathcal{C}$ the fixed curve of $\phi$.

\begin{itemize}
\item[$\diamond$] Let us first assume that any element
of $\mathrm{Cent}(\phi)$ preserves the fibration
$z_1=$cst fiberwise. Then  
$\mathrm{Cent}(\phi)=\mathcal{J}_{z_1}$.

\item[$\diamond$] Assume now that there exists
an element $\psi$ in $\mathrm{Cent}(\phi)$ that
does not preserve the fibration $z_1=$cst
fiberwise. According to Lemma \ref{lem:stab}
either $\mathrm{pr}_2(\psi)=\zeta z_1$ with
$\zeta$ root of unity, or 
$\mathrm{pr}_2(\psi)=\frac{\alpha}{z_1}$
with $\alpha$ in~$\mathbb{C}^*$.

If $\mathrm{pr}_2(\psi)=\zeta z_1$ with 
$\zeta$ root of unity, then 
\[
\frac{4c^2(\zeta z_1)}{c^2(\zeta z_1)-\zeta z_1}=\frac{4c^2(z_1)}{c^2(z_1)-z_1}
\]
{\it i.e.} $c^2(\zeta z_1)=\zeta c^2(z_1)$.
There exists $\upsilon$ such that 
$\upsilon^2=\zeta$ and 
$c(\upsilon^2 z_1)=\upsilon c(z_1)$.
Note that 
$\varphi\colon(z_0,z_1)\mapsto(\upsilon z_0,\upsilon^2z_1)$
belongs to $\mathrm{Cent}(\phi)$. Remark that 
$\mathrm{pr}_2(\psi\circ\varphi^{-1})=\mathrm{id}$, 
so $\psi\circ\varphi^{-1}$ belongs to 
$\mathcal{J}_{z_1}$. 

If $\mathrm{pr}_2(\psi)=\frac{\alpha}{z_1}$, 
then 
\[
\frac{4c^2\left(\frac{\alpha}{z_1}\right)}{c^2\left(\frac{\alpha}{z_1}\right)-z_1}=\frac{4c^2(z_1)}{c^2(z_1)-z_1}
\]
{\it i.e.} 
$c^2\left(\frac{\alpha}{z_1}\right)=\frac{\alpha}{z_1^2}c^2(z_1)$.
There exists $\beta$ in $\mathbb{C}$ such that 
$\beta^2=\alpha$ and 
$c\left(\frac{\beta^2}{z_1}\right)=\frac{\beta}{z_1}c(z_1)$.
Remark that the map 
$(z_0,z_1)\mapsto\left(\frac{\beta z_0}{z_1},\frac{\beta^2}{z_1}\right)$
commutes with $\phi$. The map $\psi\circ\varphi^{-1}$ 
belongs to $\mathrm{Cent}(\phi)$ and preserves the 
fibration $z_1=$cst fiberwise; hence 
$\psi\circ\varphi^{-1}$ belongs to $\mathcal{J}_{z_1}$.
\end{itemize}
\end{proof}

\end{itemize}

We thus have established:

\begin{pro}[\cite{CerveauDeserti:centralisateurs}]\label{pro:centrfib2}
The centralizer of a Jonqui\`eres twist $\phi$ that 
preserves fiberwise the fibration in the plane
Cremona group is a finite extension of 
$\mathrm{Ab}(\phi)$. 
\end{pro}

\subsubsection{Centralizers of elements of $\mathcal{J}\smallsetminus\mathcal{J}_0$}\label{subsubsec:centrjonq}

The description of the centralizers of elements of 
$\mathcal{J}_0$ (Proposition \ref{pro:centrfib2}) 
allows to describe, up to finite 
index, the centralizer of elements of $\mathcal{J}$.
Generically these maps have a trivial centralizer
(\cite{CerveauDeserti:centralisateurs}). A consequence
of the study of the centralizers of elements of 
$\mathcal{J}$ is:

\begin{cor}[\cite{CerveauDeserti:centralisateurs}]
The centralizer of a Jonqui\`eres twist
is virtually sol\-vable.
\end{cor}

Zhao has refined this statement:

\begin{pro}[\cite{Zhao}]
The centralizer of a Jonqui\`eres twist whose 
action on the basis of the rational fibration 
is of infinite order is virtually abelian.
\end{pro}

\subsection{What about the others ?}

\subsubsection{} Let $\phi$ be an Halphen 
twist. Up to birational conjugacy one can 
assume that $\phi$ is an element of a rational
surface $S$ with an elliptic fibration 
and that this fibration is $\phi$-invariant 
(\S \ref{sec:degreegrowth}). Furthermore
we can assume that there is no smooth 
curve of self-intersection $-1$ in the 
fibers, {\it i.e.} that the fibration 
is minimal, and so that $\phi$ is an 
automorphism. The elliptic fibration
is the unique $\phi$-invariant fibration
(\cite{DillerFavre}). As a result 
the fibration is invariant by all 
elements that commute with $\phi$, and
the centralizer of $\phi$ is contained
in $\mathrm{Aut}(S)$.

Since the fibration is minimal, the surface 
$S$ is obtained by blowing up the complex
projective plane in the nine base-points 
of an Halphen pencil and the 
rank of its N\'eron-Severi group is equal
to $10$. The group $\mathrm{Aut}(S)$ 
can be embedded in the endomorphisms
of $H^2(S,\mathbb{Z})$ for the intersection
form and preserves the class $[K_S]$ of
the canonical divisor, that is the class 
of the fibration. The dimension of the 
orthogonal hyperplane to $[K_S]$ is $9$,
and the restriction of the intersection 
form on its hyperplane is semi-negative:
 its kernel coincides with 
$\mathbb{Z}[K_S]$. As a consequence~$\mathrm{Aut}(S)$ contains an abelian 
group of finite index with rank 
$\leq 8$. We can thus state:

\begin{pro}[\cite{Gizatullin}]
Let $\phi$ be an Halphen twist. The 
centralizer of $\phi$ in~$\mathrm{Bir}(\mathbb{P}^2_\mathbb{C})$ 
contains a subgroup of finite index which is
abelian, free and of rank $\leq 8$.
\end{pro}

\subsubsection{}

We finish the description of the centralizers of 
birational maps with the case of loxodromic maps 
in \S \ref{subsec:centralisateursloxodromic}.


\chapter[Consequences of the action of $\mathrm{Bir}(\mathbb{P}^2_\mathbb{C})$ 
on $\mathbb{H}^\infty$]{Consequences of the action of the Cremona group on an 
infinite dimensional hyperbolic space}\label{chap:hyper}

\bigskip
\bigskip

As we will see in this chapter one of the main techniques to better 
understand infinite subgroups of $\mathrm{Bir}(\mathbb{P}^2_\mathbb{C})$ 
is the construction of the action by isometries of the plane Cremona group 
on an infinite dimensional hyperbolic space detailed in 
Chapter \ref{chap:hyperbolicspace} and the use 
of results from hyperbolic geometry and group theory. 

In the first section we recall results of 
Demazure and Beauville that
suggest that the plane Cremona 
group behaves like a rank $2$ group. We give
an outline of the proof of the description of 
the centralizer of a loxodromic element 
of $\mathrm{Bir}(\mathbb{P}^2_\mathbb{C})$. 
On the one hand it finishes the description of 
the centralizer of the elements of
$\mathrm{Bir}(\mathbb{P}^2_\mathbb{C})$, on 
the other hand it suggests that 
$\mathrm{Bir}(\mathbb{P}^2_\mathbb{C})$ behaves
as a group of rank $1$. We end this section
by recalling the description of the 
morphisms from a countable group with 
Kazhdan property $(T)$ into 
$\mathrm{Bir}(\mathbb{P}^2_\mathbb{C})$
which also insinuates that $\mathrm{Bir}(\mathbb{P}^2_\mathbb{C})$ behaves
as a group of rank $1$.

In the second section we give an 
outline of the proofs of the description of elliptic 
subgroups of $\mathrm{Bir}(\mathbb{P}^2_\mathbb{C})$,
{\it i.e.} the subgroups of 
$\mathrm{Bir}(\mathbb{P}^2_\mathbb{C})$ whose 
all elements are elliptic: if $\mathrm{G}$ is 
such a group, either $\mathrm{G}$ is a bounded
subgroup of $\mathrm{Bir}(\mathbb{P}^2_\mathbb{C})$, 
or $\mathrm{G}$ is a torsion subgroup 
(\cite{Urech:ellipticsubgroups}). It is thus 
natural to describe torsion subgroups of
$\mathrm{Bir}(\mathbb{P}^2_\mathbb{C})$. 
In the third section we give an outline of the 
proof of the fact that if $\mathrm{G}$ is 
a torsion subgroup of 
$\mathrm{Bir}(\mathbb{P}^2_\mathbb{C})$, then 
$\mathrm{G}$ is isomorphic to a bounded 
subgroup of $\mathrm{Bir}(\mathbb{P}^2_\mathbb{C})$; 
furthermore it is isomorphic to a subgroup of 
$\mathrm{GL}(48,\mathbb{C})$. Let us mention
the surprising fact that the proof uses model
theory as Malcev already did in 
\cite{Malcev}.

The fourth section deals with Tits
alternative and Burnside problem. We 
recall the Ping Pong Lemma and give a sketch 
of the proof of the Tits alternative
for the Cremona group, {\it i.e.}
the proof of 

\begin{thm}[\cite{Cantat:annals, Urech:ellipticsubgroups}]
Every subgroup of $\mathrm{Bir}(\mathbb{P}^2_\mathbb{C})$
either is virtually solvable, or contains
a non-abelian free group. 
\end{thm}

One consequence
is a positive answer to the Burnside
problem for the Cremona group: 
every finitely generated torsion subgroup of
$\mathrm{Bir}(\mathbb{P}^2_\mathbb{C})$
is finite.

The study of solvable groups is a very old problem. For instance let us recall
the Lie-Kolchin theorem: any linear solvable subgroup is up to finite index 
triangularizable (\cite{KargapolovMerzljakov}). Note that the 
assumption "up to finite index" is essential: for instance the subgroup
\[
\langle\left(\begin{array}{cc}
1 & 0 \\
1 & -1
\end{array}
\right),\,\left(\begin{array}{cc}
-1 & 1 \\
0 & 1
\end{array}
\right)\rangle
\] 
of $\mathrm{PGL}(2,\mathbb{C})$ is isomorphic 
to $\mathfrak{S}_3$, so is 
solvable but is not triangularizable. 
The fifth section dedicated to a 
sketch of the proof of the characterization
of the solvable subgroups of the plane
Cremona group 
(\cite{Deserti:resoluble, Urech:ellipticsubgroups}).

Let us recall a very old question, already asked in 
$1895$ in \cite{Enriques2}: 
\begin{center}
\begin{fmpage}{10cm}
"Tuttavia altre questioni
d'indole gruppale relative al gruppo Cremona nel 
piano (ed a pi\`u forte ragione in $S_n$, $n>2$) 
rimangono ancora insolute; ad esempio l'importante
questione se il gruppo Cremona contenga alcun 
sottogruppo invariante (questione alla quale
sembra probabile si debba rispondere 
negativamente)".
\end{fmpage}
\end{center}

In $2013$ Cantat and Lamy established that 
$\mathrm{Bir}(\mathbb{P}^2_\Bbbk)$ is not simple as
soon as $\Bbbk$ is algebraically closed 
(\cite{CantatLamy}). Then in $2016$ Lonjou proved
that $\mathrm{Bir}(\mathbb{P}^2_\Bbbk)$ is not 
simple over any field (\cite{Lonjou}). 
The sixth section is devoted to normal 
subgroups of $\mathrm{Bir}(\mathbb{P}^2_\mathbb{C})$ 
and the non-simplicity of
$\mathrm{Bir}(\mathbb{P}^2_\mathbb{C})$.
Strategies of \cite{CantatLamy} and \cite{Lonjou}
are evoked. A consequence of one result of 
\cite{Lonjou} is the following property: 
the Cremona group contains infinitely 
many characteristic subgroups (\cite{Cantat:survey}).

Taking the results of the sixth section
as a starting point Urech gives
a classification of all simple groups
that act non-trivially by birational 
maps on complex compact K\"ahler 
surfaces. In particular he gets the 
two following statements:

\begin{thm}[\cite{Urech:simplesubgroups}]\label{thm:urechsimple1}
A simple group $\mathrm{G}$ acts 
non-trivially by birational maps on a 
rational complex projective surface if 
and only if $\mathrm{G}$ is isomorphic
to a subgroup of $\mathrm{PGL}(3,\mathbb{C})$.
\end{thm}

\begin{thm}[\cite{Urech:simplesubgroups}]\label{thm:urechsimple2}
Let $\mathrm{G}$ be a simple subgroup of 
$\mathrm{Bir}(\mathbb{P}^2_\mathbb{C})$. 
Then 
\begin{itemize}
\item[$\diamond$] $\mathrm{G}$ does not 
contain loxodromic elements;

\item[$\diamond$] if $\mathrm{G}$ contains
a parabolic element, then $\mathrm{G}$ is
conjugate to a subgroup of $\mathcal{J}$;

\item[$\diamond$] if $\mathrm{G}$ is an 
elliptic subgroup, then $\mathrm{G}$
is either a simple subgroup of an algebraic 
subgroup of 
$\mathrm{Bir}(\mathbb{P}^2_\mathbb{C})$, 
or conjugate to a subgroup 
of $\mathrm{G}$.
\end{itemize}
\end{thm}

In the last section we give a sketch of
the proof of these results.

\bigskip
\bigskip


\section{A group of rank $1.5$}

\subsection{Rank $2$ phenomenon}

Let $\Bbbk$ be a field. Consider a connected semi-simple algebraic group 
$\mathrm{G}$ defined over $\Bbbk$. Let 
$\Psi\colon\mathrm{G}\to\mathrm{Aut}(\mathrm{G})$ be the mapping 
$g\mapsto\Psi_g$ where $\Psi_g$ denotes the inner automorphism 
given by 
\begin{align*}
& \Psi_g\colon\mathrm{G}\to\mathrm{G}, && h\mapsto ghg^{-1}.
\end{align*}
For each $g$ in $\mathrm{G}$ one can define $\mathrm{Ad}_g$ to be the 
derivative of $\Psi_g$ at the origin
\[
\mathrm{Ad}_g=(D\Psi_g)_{\mathrm{id}}\colon \mathfrak{g}\to\mathfrak{g}
\]
where $D$ is the differential and $\mathfrak{g}=T_{\mathrm{id}}\mathrm{G}$ is the 
tangent space of $\mathrm{G}$ at the identity element of~$\mathrm{G}$.
The map
\begin{align*}
& \mathrm{Ad}\colon\mathrm{G}\to\mathrm{Aut}(\mathfrak{g}), && 
g\mapsto\mathrm{Ad}_g
\end{align*}
is a group representation called the 
\textsl{adjoint representation}\index{defi}{adjoint representation} 
of $\mathrm{G}$. The \textsl{$\Bbbk$-rank}\index{defi}{$\Bbbk$-rank}
of $\mathrm{G}$ is the maximal dimension of a connected algebraic
subgroup of $\mathrm{G}$ which is diagonalizable over $\Bbbk$ in
$\mathrm{GL}(\mathfrak{g})$. Such a maximal diagonalizable 
subgroup is a \textsl{maximal torus}\index{defi}{maximal torus}.

\begin{thm}[\cite{Demazure:sousgroupesalgebriques, Enriques}]\label{thm:demazure}
Let $\mathbb{G}_m$ be the multiplicative group over 
$\mathbb{C}$. Let $r$ be an integer.

If $\mathbb{G}_m^r$ embeds as an algebraic subgroup in 
$\mathrm{Bir}(\mathbb{P}^n_\mathbb{C})$, then $r\leq n$. If $r=n$, 
then the embedding is conjugate to an embedding into the group
of diagonal matrices in $\mathrm{PGL}(n+1,\mathbb{C})$.
\end{thm}

\begin{rem}
Theorem \ref{thm:demazure} not only holds for $\mathbb{C}$
but also for any algebraically closed field~$\Bbbk$.
\end{rem}

In other words the group of diagonal automorphisms $\mathrm{D}_n$
plays the role of a maximal torus in
$\mathrm{Bir}(\mathbb{P}^n_\mathbb{C})$
and the Cremona group "looks like" a group of rank $n$.

Furthermore Beauville has shown a finite version of Theorem
\ref{thm:demazure} in dimension $2$:

\begin{thm}[\cite{Beauville}]
Let $p\geq 5$ be a prime. 

If the abelian group $\Big(\faktor{\mathbb{Z}}{p\mathbb{Z}}\Big)^r$ 
embeds into $\mathrm{Bir}(\mathbb{P}^2_\mathbb{C})$, then $r\leq 2$.
Moreover if $r=2$, then the image of 
$\Big(\faktor{\mathbb{Z}}{p\mathbb{Z}}\Big)^r$ is conjugate to a 
subgroup of the group $\mathrm{D}_2$ of diagonal automorphisms 
of~$\mathbb{P}^2_\mathbb{C}$.
\end{thm}

\begin{rem}
This statement not only holds for $\mathbb{C}$
but also for any algebraically closed field~$\Bbbk$.
\end{rem}

Let us give an idea of the proof. Consider a finite group 
$\mathrm{G}$ of $\mathrm{Bir}(\mathbb{P}^2_\mathbb{C})$. It can 
be realized as a group of automorphisms of a rational 
surface $S$ (\emph{see for instance} \cite{DeFernexEin}). 
Moreover one can assume that every birational 
$\mathrm{G}$-equivariant morphism of $S$ onto a smooth 
surface with a $\mathrm{G}$-action is an isomorphism.
Then according to \cite{Manin:rational} 
\begin{itemize}
\item[$\diamond$] either $\mathrm{G}$ preserves a fibration
$\pi\colon S\to\mathbb{P}^1$ with rational fibers,

\item[$\diamond$] or $\mathrm{Pic}(S)^{\mathrm{G}}$ has rank $1$.
\end{itemize}

In the first case $\mathrm{G}$ embeds in the group of automorphisms
of the generic fibre $\mathbb{P}^1_{\mathbb{C}(t)}$ of $\pi$ and 
Beauville classified the $p$-elementary subgroups of 
$\mathrm{Aut}(\mathbb{P}^1_{\mathbb{C}(t)})$. 

In the last case $S$ is a del Pezzo surface and the group 
$\mathrm{Aut}(S)$ is well known. Beauville also classified
the $p$-elementary subgroups of such groups.

Combining this result of those recalled in Chapter \ref{chapter:uncountable}, 
\S \ref{sec:cent} Zhao get:

\begin{thm}
Let $\phi\in\mathrm{Bir}(\mathbb{P}^2_\mathbb{C})$ be 
an element of infinite order. If the centralizer of
$\phi$ is not virtually abelian, then either $\phi$
is an elliptic map, or a power of $\phi$ is conjugate to 
an automorphism of $\mathbb{C}^2$ of the form 
$(z_0,z_1)\mapsto(z_0,z_1+1)$ or 
$(z_0,z_1)\mapsto(z_0,\beta z_1)$ with $\beta\in\mathbb{C}^*$.
\end{thm}

\begin{rem}
This statement also holds for $\mathrm{Bir}(\mathbb{P}^2_\Bbbk)$
where $\Bbbk$ is an algebraically closed field 
(\cite{Zhao}).
\end{rem}

\subsection{Rank $1$ phenomenon}\label{subsec:centralisateursloxodromic}

Generic elements of degree $\geq 2$ of 
$\mathrm{Bir}(\mathbb{P}^2_\mathbb{C})$ are loxodromic and 
hence can not be conjugate to elements of the maximal 
torus $\mathrm{D}_2$. The description of their 
centralizer is given by:

\begin{thm}[\cite{Cantat:annals, BlancCantat}]\label{thm:cent}
Let $\phi$ be a loxodromic element
of $\mathrm{Bir}(\mathbb{P}^2_\mathbb{C})$. 

The infinite cyclic subgroup of 
$\mathrm{Bir}(\mathbb{P}^2_\mathbb{C})$
generated by $\phi$ has finite index in the centralizer
\[
\mathrm{Cent}(\phi)=\big\{\psi\in\mathrm{Bir}(\mathbb{P}^2_\mathbb{C})\,\vert\,\psi\circ\phi=\phi\circ\psi\big\}
\]
of $\phi$.
\end{thm}

\begin{rem}
Theorem \ref{thm:cent} holds for any field $\Bbbk$.
\end{rem}

The centralizer of a generic element of $\mathrm{SL}(n+1,\mathbb{C})$
is isomorphic to $(\mathbb{C}^*)^n$; Theorem \ref{thm:cent} suggests
that $\mathrm{Bir}(\mathbb{P}^2_\mathbb{C})$ behaves as a group of
rank $1$.

\begin{proof}[Sketch of the proof]
If $\psi$ commutes to $\phi$, then the isometry $\psi_*$ of 
$\mathbb{H}^\infty$ preserves the axis $\mathrm{Ax}(\phi)$
of $\phi_*$ and its two endpoints. Consider the morphism
$\Theta$ which maps 
$\mathrm{Cent}(\phi)$ to
the group of isometries of $\mathrm{Ax}(\phi)$. View it as
a morphism into the group of translations $\mathbb{R}$ 
of the line. On the one hand the translation lengths are 
bounded from below by $\log(\lambda_L)$ where $\lambda_L$\index{not}{$\Lambda_L$} 
is the Lehmer number, {\it i.e.} the unique root $>1$ of the
irreducible polynomial $x^{10}+x^9-x^7-x^6-x^5-x^4-x^3+x+1$
(\emph{see} \cite{BlancCantat}). On the other hand every
discrete subgroup of $\mathbb{R}$ is trivial or cyclic. 
As a result the image of $\Theta$ is a cyclic group. Its 
kernel is made of elliptic elements of 
$\mathrm{Cent}(\phi)$
fixing all points of $\mathrm{Ax}(\phi)$. Denote by 
$\mathbf{e}_\phi$ the projection of $\mathbf{e}_0$ on 
$\mathrm{Ax}(\phi)$. Since $\ker\Theta$ fixes $\mathbf{e}_\phi$,
the inequality 
\[
\mathrm{dist}(\psi_*\mathbf{e}_0,\mathbf{e}_0)\leq 2\mathrm{dist}(\mathbf{e}_0,\mathbf{e}_\phi)
\]
holds. As a consequence $\ker\Theta$ is a group of birational 
maps of bounded degree. From \cite{BlancFurter} the Zariski 
closure of $\ker\Theta$ in $\mathrm{Bir}(\mathbb{P}^2_\mathbb{C})$ 
is an algebraic subgroup of 
$\mathrm{Bir}(\mathbb{P}^2_\mathbb{C})$. Let us denote by 
$\mathrm{G}$ the connected component of the identity in 
this group. Assume that $\ker\Theta$ is infinite. Then 
$\dim\mathrm{G}$ is positive and $\mathrm{G}$ is 
contained, after conjugacy, in the group of automorphisms
of a minimal, rational surface (\cite{Blanc:ssgpealg, Enriques}).
Therefore, $\mathrm{G}$ contains a Zariski closed abelian 
subgroup whose orbits have dimension $1$. Those orbits 
are organised in a pencil of curves that is invariant 
under the action of $\phi$: contradiction with the fact
that $\phi_*$ is loxodromic. As a result $\ker\Theta$ is 
finite.
\end{proof}

\subsection{Rank $1$ phenomenon}

To generalize Margulis work on linear representations of lattices 
of simple real Lie groups to non-linear 
representations Zimmer proposed to study the actions of lattices
on compact varieties (\cite{Zimmer1, Zimmer2, Zimmer3, Zimmer4}).
One of the main conjectures of the program drawn by Zimmer is: 
let $\mathrm{G}$ be a connex, simple, real Lie group and let 
$\Gamma$ be a lattice of $\mathrm{G}$. If there exists a morphism
from $\Gamma$ into the diffeomorphisms group of a compact 
variety $V$ with infinite image, then the real rank of $\mathrm{G}$
is less or equal to the dimension of $V$.

In the context of birational self maps one 
has the following statement that can be see as another rank one
phenomenum:

\begin{thm}[\cite{Cantat:annals, Deserti:IMRN}]\label{thm:zimmer}
Let $S$ be a complex projective surface. 
Let $\Gamma$ be a countable group with 
Kazhdan property $(T)$. 

If 
$\upsilon\colon\Gamma\to\mathrm{Bir}(\mathbb{P}^2_\mathbb{C})$
is a morphism with infinite image, then 
$\upsilon$ is conjugate to a morphism into
$\mathrm{PGL}(3,\mathbb{C})$.
\end{thm}

\begin{rem}
Theorem \ref{thm:zimmer} indeed holds for any 
algebraically closed field $\Bbbk$.
\end{rem}

\begin{proof}[Sketch of the proof]
The first step is based on a fixed point property: since
$\Gamma$ has Kazhdan property (T), then 
$\upsilon(\Gamma)$ acts by
isometries on $\mathbb{H}^\infty$ and $(\upsilon(\Gamma))_*$ has a
fixed point. Then according to \cite{delaHarpeValette} 
all its orbits have bounded diameter. Hence $\rho(\Gamma)$ has
bounded degree. There thus exists a birational map 
$\pi\colon X\dashrightarrow\mathbb{P}^2_\mathbb{C}$ such that
\begin{itemize}
\item[$\diamond$] $\Gamma_S=\pi^{-1}\circ\Gamma\circ\pi$ is a 
subgroup of $\mathrm{Aut}(S)$;

\item[$\diamond$] $\mathrm{Aut}(S)^0\cap\Gamma_S$ has finite
index in $\Gamma_S$.
\end{itemize}

The classification of algebraic groups of maps of surfaces and 
the fact that every subgroup of $\mathrm{SL}(2,\mathbb{C})$ 
having Kazhdan property (T) is finite allow to prove that: 
since $\mathrm{Aut}(S)^0$ contains an infinite group with 
Kazhdan property (T) the surface $S$ is isomorphic to the 
projective plane $\mathbb{P}^2_\mathbb{C}$.
\end{proof}


\section{Subgroups of elliptic elements of $\mathrm{Bir}(\mathbb{P}^2_\mathbb{C})$}\label{sec:ellipticgroup}

A subgroup $\mathrm{G}$ of the plane Cremona group is 
\textsl{elliptic }\index{elliptic (group)} if any element of 
$\mathrm{G}$ is an elliptic birational map.

Let us give an example: a bounded subgroup of 
$\mathrm{Bir}(\mathbb{P}^2_\mathbb{C})$ is elliptic. But not 
all elliptic subgroups are bounded; indeed for instance
\begin{itemize}
\item[$\diamond$] all elements of 
$\big\{(z_0,z_1+a(z_0))\,\vert\,a\in\mathbb{C}(z_0)\big\}$ are 
elliptic but $\big\{(z_0,z_1+a(z_0))\,\vert\,a\in\mathbb{C}(z_0)\big\}$
contains elements of arbitrarily high degrees;

\item[$\diamond$] Wright gives examples of subgroups of 
$\mathrm{Bir}(\mathbb{P}^2_\mathbb{C})$ isomorphic to a 
subgroup  of roots of unity of $\mathbb{C}^*$ that are not
bounded (\cite{Wright:abelian}). Let us be more precise. 
Set $\psi_0\colon(z_0,z_1)\mapsto(-z_0,-z_1)$ and for 
any $k\geq 1$ 
\begin{align*}
& \alpha_k=\exp\left(\frac{\mathbf{i}\pi}{2^k}\right), && \phi_k\colon(z_0,z_1)\mapsto(z_1,c_kz_1^{2^k+1}+z_0), && \varphi_k=\phi_k^2\circ \phi_{k-1}^2\circ\ldots\circ \phi_1^2
\end{align*}
where $c_k$ denotes an element of $\mathbb{C}^*$. 
Consider 
\[
\psi_k=\varphi_k^{-1}\circ\big((z_0,z_1)\mapsto(\alpha_kz_0,\alpha_k^pz_1)\big)\circ\varphi_k
\]
where $p$ is an odd integer. The group
\[
\mathrm{G}=\displaystyle\bigcup_{k\geq 0}\langle \psi_k\rangle
\]
is an abelian group obtained as a growing union of 
finite cyclic groups that does not preserve any 
fibration (\cite{Lamy:dynamique}).
\end{itemize}

This gives all the possibilities for elliptic subgroups of 
$\mathrm{Bir}(\mathbb{P}^2_\mathbb{C})$:

\begin{thm}[\cite{Urech:ellipticsubgroups}]\label{thm:urechell1}
Let $\mathrm{G}$ be an elliptic subgroup of the plane
Cremona group. Then one of the following holds:
\begin{itemize}
\item[$\diamond$] $\mathrm{G}$ is a bounded subgroup; 

\item[$\diamond$] $\mathrm{G}$ preserves a rational fibration;

\item[$\diamond$] $\mathrm{G}$ is a torsion group.
\end{itemize}
\end{thm}

Furthermore he characterizes torsion subgroups of 
$\mathrm{Bir}(\mathbb{P}^2_\mathbb{C})$:

\begin{thm}[\cite{Urech:ellipticsubgroups}]\label{thm:urechell2}
Let $\mathrm{G}\subset\mathrm{Bir}(\mathbb{P}^2_\mathbb{C})$ be 
a torsion group. Then $\mathrm{G}$ is isomorphic to 
a bounded subgroup of $\mathrm{Bir}(\mathbb{P}^2_\mathbb{C})$.

Furthermore $\mathrm{G}$ is isomorphic to a subgroup of 
$\mathrm{GL}(48,\mathbb{C})$. 
\end{thm}

As a consequence he gets an analogue of the Theorem
of Jordan and Schur:

\begin{cor}[\cite{Urech:ellipticsubgroups}]
There exists a constant $\gamma$ such that every torsion subgroup
of $\mathrm{Bir}(\mathbb{P}^2_\mathbb{C})$ contains a 
commutative normal subgroup of index $\leq \gamma$.
\end{cor}

Theorems \ref{thm:urechell1} and \ref{thm:urechell2} allow to 
refine 
\begin{itemize}
\item[$\diamond$] the result of Cantat about Tits alternative 
(\S \ref{sec:tits});

\item[$\diamond$] the description of solvable subgroups of
$\mathrm{Bir}(\mathbb{P}^2_\mathbb{C})$ 
(\emph{see} \S \ref{sec:solvable}).
\end{itemize}

The aim of the section is to prove Theorem \ref{thm:urechell1}.
We need the following technical lemmas.

\begin{lem}[\cite{CerveauDeserti:centralisateurs}]\label{lem:2fibr}
Let $\phi$ be a birational self map of the complex
projective plane that fixes pointwise two different 
rational fibrations. Then $\phi$ is of finite order.
\end{lem}

\begin{proof}
The intersections of the generic fibres of these 
two fibrations are finite, uniformly bounded. But 
these intersections are invariant by $\phi$ so
$\phi$ is of finite order.
\end{proof}

\begin{lem}[\cite{Urech:ellipticsubgroups}]\label{lem:Urechmoinsde9}
An algebraic subgroup $\mathrm{G}$ of 
$\mathrm{Bir}(\mathbb{P}^2_\mathbb{C})$ of 
dimension $\leq 9$ preserves a unique rational
fibration.
\end{lem}

\begin{proof}
According to Theorem \ref{thm:blanc11cases} the
group $\mathrm{G}$ is conjugate to a subgroup 
of $\mathrm{Aut}(\mathbb{F}_n)$ for some 
Hirzebruch surface $\mathbb{F}_n$, $n\geq 2$.
As a consequence $\mathrm{G}$ preserves 
a rational fibration 
$\pi\colon\mathbb{F}_n\to\mathbb{P}^1_\mathbb{C}$.
The fibres of $\pi$ are permuted by $\mathrm{G}$, 
this yields to a homomorphism
\[
f\colon\mathrm{G}\to\mathrm{PGL}(2,\mathbb{C})
\]
such that $\dim\ker f\geq 6$.

Assume by contradiction that there exists a 
second rational fibration 
$\pi'\colon\mathbb{F}_n\to\mathbb{P}^1_\mathbb{C}$
preserved by $\mathrm{G}$; this yields to a second
homomorphism 
\[
g\colon\mathrm{G}\to\mathrm{PGL}(2,\mathbb{C}).
\]
One has $\dim\ker\pi'_{\vert\ker\pi}>0$; therefore,
$\dim(\ker f\cap\ker g)>0$. In particular 
$\ker f\cap\ker g$ contains an element of 
infinite order: contradiction with Lemma
\ref{lem:2fibr}.
\end{proof}

\begin{lem}[\cite{Urech:ellipticsubgroups}]\label{lem:algcst}
Let 
$\mathrm{G}\subset\mathrm{Bir}(\mathbb{P}^2_\mathbb{C})$ 
be an algebraic subgroup isomorphic as an 
algebraic group to $\mathbb{C}^*$. 

There exists a constant $K(\mathrm{G})$ such
that any elliptic element of 
\[
\mathrm{Cent}(\mathrm{G})=\big\{\varphi\in\mathrm{Bir}(\mathbb{P}^2_\mathbb{C})\,\vert\,\varphi\circ\psi=\psi\circ\varphi\quad\forall\,\psi\in\mathrm{G}\big\}
\]
has degree $\leq K(\mathrm{G})$.
\end{lem}

\begin{proof}
Up to conjugacy by an element 
$\psi\in\mathrm{Bir}(\mathbb{P}^2_\mathbb{C})$
one can assume that 
\[
\mathrm{G}=\big\{(z_0,z_1)\mapsto(\alpha z_0,z_1)\,\vert\,\alpha\in\mathbb{C}^*\big\}.
\]
An elliptic element of 
$\mathrm{Cent}(\mathrm{G})$
is of the following form
\[
\varphi\colon(z_0,z_1)\dashrightarrow(z_0\varphi_1(z_1),\varphi_2(z_1))
\]
where $\varphi_1$, $\varphi_2$ are rational functions.
Since $(\deg\varphi^n)_n$ is bounded, $\varphi_1$
is constant, and so $\varphi_2\colon z_1\mapsto\frac{az_1+b}{cz_1+d}$
for some matrix $\left(
\begin{array}{cc}
a & b \\
c & d
\end{array}
\right)$ of $\mathrm{PGL}(2,\mathbb{C})$. In particular
$\deg\varphi\leq 2$. The constant $K(\mathrm{G})$ thus 
only depends on the degree of $\psi$. 
\end{proof}

\begin{lem}[\cite{Urech:ellipticsubgroups}]\label{lem:Urechmonomial}
A group $\mathrm{G}\subset\mathrm{Bir}(\mathbb{P}^2_\mathbb{C})$ of monomial elliptic elements is bounded.
\end{lem}

\begin{proof}
The group $\mathrm{G}$ is contained in
$\mathrm{GL}(2,\mathbb{Z})\ltimes(\mathbb{C}^*)^2$.
Consider the projection 
$\pi\colon\mathrm{G}\to\mathrm{GL}(2,\mathbb{Z})$. 
On the one hand $\ker\pi$ is bounded, on the other
hand all elements of $\pi(\mathrm{G})$ are bounded.
All elliptic elements in 
$\mathrm{GL}(2,\mathbb{Z})\subset\mathrm{Bir}(\mathbb{P}^2_\mathbb{C})$
are of finite order, so $\pi(\mathrm{G})$ is a 
torsion subgroup of $\mathrm{GL}(2,\mathbb{Z})$. 
Since there are only finitely many
conjugacy classes of finite subgroups in 
$\mathrm{GL}(2,\mathbb{Z})$ the group 
$\pi(\mathrm{G})$ is finite. Therefore, $\mathrm{G}$
is a finite extension of a bounded subgroup hence
$\mathrm{G}$ is bounded.
\end{proof}

\begin{lem}[\cite{Urech:ellipticsubgroups}]\label{lem:Urechsemisimple}
Let $\mathrm{H}$ be a semi-simple algebraic subgroup of 
$\mathrm{Bir}(\mathbb{P}^2_\mathbb{C})$. Let 
$\mathrm{G}\subset\mathrm{Bir}(\mathbb{P}^2_\mathbb{C})$ 
be a group of elliptic elements that normalizes 
$\mathrm{H}$. Then $\mathrm{G}$ is bounded.
\end{lem}

\begin{proof}
The group $\mathrm{H}$ is semi-simple; in particular 
its group of inner automorphisms has finite index in 
its group of algebraic automorphisms. As a result 
there exists $N\in\mathbb{Z}$ such that for any $\phi$
in $\mathrm{G}$ conjugation by $\phi^N$ induces an 
inner automorphism of $\mathrm{H}$. Hence, there exists 
 an element $\psi$ in $\mathrm{H}$ such that 
$\phi^N\circ\psi$ centralizes $\mathrm{H}$. By assumption 
$\mathrm{H}$ is semi-simple, so $\mathrm{H}$ 
contains a closed subgroup $\mathrm{D}$ isomorphic
as an algebraic group to $\mathbb{C}^*$ and this 
group is centralized by $\phi^N\circ\psi$. From Lemma
\ref{lem:algcst} we get that $\deg(\phi^N\circ\psi)$ is 
bounded by a constant that depends neither on 
$\phi$, nor on $N$. As $\mathrm{H}$ is an algebraic
group both $\deg \psi$ and $\deg \phi$ are also 
bounded independently of $\phi$ and $N$. Finally 
$\mathrm{G}$ is bounded.
\end{proof}

\begin{lem}[\cite{Urech:ellipticsubgroups}]\label{lem:Urechtecn1}
Let $\mathrm{G}$ be a subgroup of 
$\mathrm{Bir}(\mathbb{P}^2_\mathbb{C})$ that fixes a 
point of $\mathbb{H}^\infty$. Then 
\begin{itemize}
\item[$\diamond$]  the degree of all elements in $\mathrm{G}$
is uniformly bounded;

\item[$\diamond$] there exist a smooth projective surface $S$
and a birational map 
$\varphi\colon\mathbb{P}^2_\mathbb{C}\dashrightarrow S$
such that 
$\varphi\circ\mathrm{G}\circ\varphi^{-1}\subset\mathrm{Aut}(S)$.
\end{itemize}
\end{lem}

\begin{proof}
Denote by $p\in\mathbb{H}^\infty$ the fixed point of 
$\mathrm{G}$, and by $\mathbf{e}_0\in\mathbb{H}^\infty$ the 
class of a line in~$\mathbb{P}^2_\mathbb{C}$. Take 
an element $\psi$ of $\mathrm{G}$. The action of 
$\mathrm{G}$ on $\mathbb{H}^\infty$ is isometric
hence $d(\psi(\mathbf{e}_0),p)=d(\mathbf{e}_0,p)$, and so 
$d(\psi(\mathbf{e}_0),p)\leq 2d(\mathbf{e}_0,p)$. This implies 
\[
\langle\psi(\mathbf{e}_0),\mathbf{e}_0\rangle\leq \cosh(2d(\mathbf{e}_0,p))\qquad\forall\,\psi\in\mathrm{G}.
\]
Since $\langle\psi(\mathbf{e}_0),\mathbf{e}_0\rangle=\deg\psi$ the 
previous inequality can be rewritten as follows
\[
\deg\psi\leq \cosh(2d(\mathbf{e}_0,p))\qquad\forall\,\psi\in\mathrm{G},
\]
{\it i.e.} the degrees of all elements in $\mathrm{G}$
are uniformly bounded. 

According to Weil $\mathrm{G}$ can be regularized 
(\S \ref{sec:WeilKraft}).
\end{proof}

Let us recall the following statement due to Cantat:

\begin{pro}[\cite{Cantat:annals}]\label{pro:etaumilieuCantat}
Let $\Gamma$ be a finitely generated subgroup of 
$\mathrm{Bir}(\mathbb{P}^2_\mathbb{C})$ of 
elliptic elements. Then 
\begin{itemize}
\item[$\diamond$] either $\Gamma$ is bounded, 

\item[$\diamond$] or $\Gamma$ preserves a 
rational fibration, {\it i.e.} $\Gamma\subset\mathcal{J}$
up to birational conjugacy.
\end{itemize}
\end{pro}

\begin{lem}[\cite{Urech:ellipticsubgroups}]\label{lem:Urechtecalt}
Let $\mathrm{G}\subset\mathrm{Bir}(\mathbb{P}^2_\mathbb{C})$
be a group of elliptic elements. Then one of the 
following holds:
\begin{itemize}
\item[$\diamond$] $\mathrm{G}$ preserves a fibration, 
and so up to birational conjugacy either 
$\mathrm{G}\subset\mathcal{J}$, or 
$\mathrm{G}\subset\mathrm{Aut}(S)$ where $S$ is 
a Halphen surface.

\item[$\diamond$] every finitely generated subgroup 
of $\mathrm{G}$ is bounded. 
\end{itemize}

Furthermore if $\mathrm{G}$ fixes a point 
$p\in\partial\mathbb{H}^\infty$ that does not 
correspond to the class of a rational fibration, 
then the second assertion holds.
\end{lem}

\begin{proof}
The group $\mathrm{G}$ fixes a point 
$p\in\mathbb{H}^\infty\cup\partial\mathbb{H}^\infty$
(Theorem \ref{thm:weak}).

If $p$ belongs to $\mathbb{H}^\infty$, then $\mathrm{G}$
is bounded. 

Let us now assume that $p$ belongs to 
$\partial\mathbb{H}^\infty$. Then either $p$ 
corresponds to the class of a general fibre of 
some fibration, or not.
\begin{itemize}
\item[$\diamond$] If $p$ corresponds to the class 
of a general fibre of some fibration 
$\pi\colon Y\to\mathbb{P}^1_\mathbb{C}$
where $Y$ is a rational surface, then $\mathrm{G}$ 
preserves this fibration and is thus conjugate to a
subgroup of~$\mathcal{J}$ (if the fibration is 
rational), or to a subgroup of $\mathrm{Aut}(S)$ 
where $S$ is a Halphen surface (if the fibration 
consists of curves of genus $1$). 

\item[$\diamond$] Suppose now that
$p$ does not correspond to the class of a fibration.
Let $\Gamma$ be a finitely generated subgroup of 
$\mathrm{G}$. Then either $\Gamma$ is bounded, or
$\Gamma$ preserves a rational fibration 
(Proposition \ref{pro:etaumilieuCantat}). 
If $\Gamma$ preserves a rational fibration 
$\mathcal{F}$, then $\Gamma$ fixes a point 
$q\in\partial\mathbb{H}^\infty$ that corresponds
to the class of $\mathcal{F}$. Hence $p$ and $q$
are two distinct points preserved by~$\mathrm{G}$ 
and $\mathrm{G}$ fixes the geodesic
line through $p$ and $q$. In particular 
$\mathrm{G}$ fixes a point in $\mathbb{H}^\infty$
and according to Lemma \ref{lem:Urechtecn1} 
the degrees of all elements in $\mathrm{G}$ are 
uniformly bounded.
\end{itemize}
\end{proof}

\begin{proof}[Proof of Theorem \ref{thm:urechell1}]
Consider a subgroup $\mathrm{G}$ of 
$\mathrm{Bir}(\mathbb{P}^2_\mathbb{C})$ of elliptic
elements. According to Lemma \ref{lem:Urechtecn1}
either $\mathrm{G}$ preserves a rational fibration,
or any finitely generated subgroup of $\mathrm{G}$
is bounded. 

Assume that any finitely generated subgroup of 
$\mathrm{G}$ is bounded. Set
\[
n:=\sup\big\{\dim\overline{\Gamma}\,\vert\,\Gamma\subset\mathrm{G}\text{ finitely generated}\big\}.
\]
\begin{itemize}
\item[$\diamond$] If $n=0$, then $\mathrm{G}$ is a 
torsion group.

\item[$\diamond$] If $n=+\infty$, then take 
$\Gamma$ a finitely generated subgroup of
$\mathrm{G}$ such that $\dim\overline{\Gamma}\geq 9$.
By Lemma~\ref{lem:Urechmoinsde9} the group 
$\Gamma$ preserves a unique fibration and this 
fibration is, again by Lemma \ref{lem:Urechmoinsde9}, 
preserved as well by $\langle\Gamma,\,\phi\rangle$ 
for any $\phi$ in $\mathrm{G}$.

\item[$\diamond$] Assume now $n\in\mathbb{N}^*$.
Let $\Gamma$ be a finitely generated subgroup of 
$\mathrm{G}$ such that $\dim\overline{\Gamma}~=~n$. 
Let~$\overline{\Gamma}^0$ be the neutral component
of $\Gamma$. For any $\varphi\in\mathrm{G}$
the group 
$\langle\overline{\Gamma}^0,\varphi\circ\overline{\Gamma}^0\circ\varphi^{-1}\rangle$ 
is connected and contained in 
$\langle\overline{\Gamma,\varphi\circ\Gamma\circ\varphi^{-1}}\rangle$
which is finitely generated and thus of dimension
less or equal to $n$. As a consequence
$\langle\overline{\Gamma}^0,\varphi\circ\overline{\Gamma}^0\circ\varphi^{-1}\rangle=\overline{\Gamma}^0$
for any $\varphi\in\mathrm{G}$ and 
$\overline{\Gamma}^0$ is normalized by 
$\mathrm{G}$. If $\overline{\Gamma}^0$ is 
semi-simple, Lemma \ref{lem:Urechsemisimple} 
allows to conclude. Assume that 
$\overline{\Gamma}^0$ is not semi-simple. 
Denote by $R$ the radical of $\overline{\Gamma}^0$, 
{\it i.e.} $R$ is the maximal
connected normal solvable subgroup of 
$\overline{\Gamma}^0$. Since $\overline{\Gamma}^0$ is 
semi-simple the inequality $\dim R>0$ holds.
The radical is unique hence preserved by 
$\mathrm{Aut}(\overline{\Gamma}^0)$ and in 
particular normalized by $\mathrm{G}$. Denote
by 
\[
R^{(\ell+1)}=\big\{\mathrm{id}\big\}\subsetneq R^{(\ell)}\subset\ldots\subset R^{(2)}\subset R^{(1)}\subset R^{(0)}=R
\]
the derived series of $R$ ({\it i.e.} 
$R^{(k+1)}=[R^{(k)},R^{(k)}]$). Note that $\dim R^{(\ell)}>0$
and $R^{(\ell)}$ is abelian. This series is 
invariant under $\mathrm{Aut}(\overline{\Gamma}^0)$,
and so invariant under conjugation by elements 
of $\mathrm{G}$. In particular $\mathrm{G}$
normalizes $R^{(\ell)}$. Since $R^{(\ell)}$ is bounded, $R^{(\ell)}$ is 
conjugate to one of the groups of Theorem
\ref{thm:blanc11cases}; in particular $R^{(\ell)}$ can 
be regularized. In other words, up to birational
conjugacy, $\mathrm{G}$ is a subgroup of $\mathrm{Bir}(S)$ 
for some smooth projective surface $S$ on which
$R^{(\ell)}$ acts regularly. If all the orbits of $R^{(\ell)}$
have dimension $\leq 1$, then $\mathrm{G}$
preserves a rational fibration. Assume that
$R^{(\ell)}$ has an open orbit $\mathcal{O}$. 
The group~$\mathrm{G}$ normalizes $R^{(\ell)}$, so $\mathrm{G}$
acts regularly on $\mathcal{O}$. The action of 
$R^{(\ell)}$ is faithful; as a result $\dim R^{(\ell)}=1$ 
and $R^{(\ell)}\simeq \mathbb{C}^2$, or 
$R^{(\ell)}\simeq\mathbb{C}^*\times\mathbb{C}$, or 
$R^{(\ell)}\simeq\mathbb{C}^*\times\mathbb{C}^*$. 
If $R^{(\ell)}\simeq\mathbb{C}^2$, then $\mathcal{O}$
is isomorphic to the affine plane, and the 
action of $R^{(\ell)}$ on $\mathcal{O}$ is given by 
translations. But the normalizer of $\mathbb{C}^2$
in $\mathrm{Aut}(\mathbb{A}^2_\mathbb{C})$ is the group of
affine maps 
$\mathrm{GL}(2,\mathbb{C})\ltimes\mathbb{C}^2$ 
hence $\mathrm{G}$ is bounded. If 
$R^{(\ell)}\simeq\mathbb{C}^*\times\mathbb{C}$, then we 
similarly get the inclusion 
$\mathrm{G}\subset\mathrm{Aut}(\mathbb{C}^*\times\mathbb{C})$.
The $\mathbb{C}$-fibration of 
$\mathbb{C}^*\times\mathbb{C}$ is given by the 
invertible functions; it is thus preserved by 
$\mathrm{Aut}(\mathbb{C}^*\times\mathbb{C})$.
In particular $\mathrm{G}$ preserves a rational
fibration. If 
$R^{(\ell)}\simeq\mathbb{C}^*\times\mathbb{C}^*$, then
elements of $\mathrm{G}$ are monomial maps, 
and Lemma \ref{lem:Urechmonomial} allows to conclude.
\end{itemize}
\end{proof}


\section{Torsion subgroups of the Cremona group}

As we have seen at the beginning of 
\S \ref{sec:ellipticgroup} some torsion groups
can be embedded into 
$\mathrm{Bir}(\mathbb{P}^2_\mathbb{C})$ in such a 
way that they neither are bounded, nor preserve 
any fibration. However the group structure of 
torsion subgroups can be specified:

\begin{thm}[\cite{Urech:ellipticsubgroups}]\label{thm:urechtorsion}
A torsion subgroup $\mathrm{G}$ of 
$\mathrm{Bir}(\mathbb{P}^2_\mathbb{C})$ is
isomorphic to a bounded subgroup of 
$\mathrm{Bir}(\mathbb{P}^2_\mathbb{C})$. 

Furthermore $\mathrm{G}$ is isomorphic to 
a subgroup of $\mathrm{GL}(48,\mathbb{C})$.
\end{thm}

Malcev used model theory to prove 
that if for a given group $\mathrm{G}$ 
every finitely generated subgroup can be 
embedded into $\mathrm{GL}(n,\Bbbk)$ 
for some field $\Bbbk$, then there exists
a field $\Bbbk'$ such that $\mathrm{G}$ 
can be embedded into $\mathrm{GL}(n,\Bbbk')$. 
Let us briefly introduce the compactness
theory from model theory; it states that a 
set of first order sentences has a model 
if and only if any of its finite subsets
has a model.

\begin{defi}
Let $\{x_i\}_{i\in I}$ be a set of 
variables. A 
\textsl{condition}\index{defi}{condition}
is an expression of the form
\[
P(x_{i_1},x_{i_2},\ldots,x_{i_k})=0
\]
or an expression of the form
\[
\big(P_1(x_{i_1},x_{i_2},\ldots,x_{i_k})\not=0\big) \vee \big(P_2(x_{i_1},x_{i_2},\ldots,x_{i_k})\not=0\big)\vee
\ldots\vee
\big(P_\ell(x_{i_1},x_{i_2},\ldots,x_{i_k})\not=0\big)
\]
where $P$ and the $P_i$'s are polynomials
with integer coefficients.
\end{defi}

\begin{defi}
A \textsl{mixed system}\index{defi}{mixed system}
is a set of conditions.
\end{defi}

\begin{defi}
A mixed system $S$ is 
\textsl{compatible}\index{defi}{compatible (mixed system)}
if there exists a field $\Bbbk$ which 
contains values $\{y_i\}_{i\in I}$ that 
satisfy $S$.
\end{defi}

\begin{thm}[\cite{Malcev}]\label{thm:Malcev}
If every finite subset of a mixed system $S$
is compatible, then $S$ is compatible.
\end{thm}

Let us now explain the proof of Theorem
\ref{thm:urechtorsion}. Let $\mathrm{G}$
be a torsion subgroup of 
$\mathrm{Bir}(\mathbb{P}^2_\mathbb{C})$.
If~$\mathrm{G}$ is finite, then $\mathrm{G}$
is bounded; we can thus assume that 
$\mathrm{G}$ is infinite. Following
Theorem \ref{thm:blanc11cases} we will 
deal with different cases.

\begin{itemize}
\item[$\diamond$] First assume that every
finitely generated subgroup of $\mathrm{G}$
is isomorphic to a subgroup of 
$\mathrm{PGL}(3,\mathbb{C})$. Consider the 
closed embedding $\rho$ of 
$\mathrm{PGL}(3,\mathbb{C})$ into 
$\mathrm{GL}(8,\mathbb{C})$ given by the 
adjoint representation. Let $P_1$, $P_2$, 
$\ldots$, $P_n$ be polynomials in the set 
of variables $\{x_{ij}\}_{1\leq i,\,j\leq 8}$
such that 
$\rho(\mathrm{PGL}(3,\mathbb{C}))\subset\mathrm{GL}(8,\mathbb{C})$
is the zero set of $P_1$, $P_2$, $\ldots$,
$P_n$. To any element $g\in\mathrm{G}$ we
associate a $8\times 8$ matrix of 
variables $(x_{ij}^g)$. Consider the 
following mixed system $S$ defined by
\begin{enumerate}
\item[(1)] the equations 
$(x_{ij}^f)(x_{ij}^g)=(x_{ij}^h)$ for all $f$,
$g$, $h\in\mathrm{G}$ such that $f\circ g=h$;

\item[(2)] the conditions 
$\big(\displaystyle\bigvee_i x_{ii}^g-1\not=0\big)\vee\big(\displaystyle\bigvee_{i\not=j} x_{ij}^g-1\not=0\big)$; 

\item[(3)] $x_{ii}^{\mathrm{id}}-1=0$ and
$x_{ij}^{\mathrm{id}}=0$ for all 
$1\leq i\not=j\leq N$;

\item[(4)] $P_k(\{x_{ij}\})=0$ for all 
$1\leq k\leq n$, for all $g\in\mathrm{G}$;

\item[(5)] $p\not=0$ for all 
$p\in\mathbb{Z}^+$ primes.
\end{enumerate}

\begin{lem}[\cite{Urech:ellipticsubgroups}]
The system $S$ is compatible.
\end{lem}

\begin{proof}
According to Theorem \ref{thm:Malcev} it
suffices to show that every finite subset
of $S$ is compatible. Let $c_1$, $c_2$, 
$\ldots$, $c_n\in S$ be finitely many 
conditions. Only finitely many variables
$x_{ij}^g$ appear in $c_1$, $c_2$,
$\ldots$, $c_n$. Let 
$\big\{g_1,\,g_2,\,\ldots,\,g_\ell\big\}\subset\mathrm{G}$ 
be the finite set of all elements 
$g\in\mathrm{G}$ such that for some 
$1\leq i,\,j\leq 8$ the variable $x_{ij}^g$ 
appears in one of the conditions $c_1$, $c_2$, 
$\ldots$, $c_n$.

Consider the finitely generated subgroup
$\Gamma=\langle g_1,\,g_2,\,\ldots,\,g_\ell\rangle$
of $\mathrm{G}$. By Theorem~\ref{thm:burnside}
the group $\Gamma$ is finite. Therefore, by 
assumption $\Gamma$ has a faithful representation
to $\mathrm{PGL}(3,\mathbb{C})$. This 
representation implies that $\mathbb{C}$ 
contains values that satisfy the conditions
$c_1$, $c_2$, $\ldots$, $c_n$. In other 
words $S$ is compatible.
\end{proof}

As a result there exists a field $\Bbbk$ such
that $\Bbbk$ contains values $y_{ij}^g$ for
all $1\leq i,\,j\leq 8$ and all $g\in\mathrm{G}$
satisfying conditions $(1)$ to $(5)$. Condition
$(5)$ asserts that the characteristic of $\Bbbk$ is~$0$. 
The group $\mathrm{G}$ has at most the 
cardinality of the continuum since 
$\mathrm{G}\subset\mathrm{Bir}(\mathbb{P}^2_\mathbb{C})$;
the values $\{y_{ij}^g\}$ are thus contained 
in a subfield $\Bbbk'$ of $\Bbbk$ that has 
the same cardinality as $\mathbb{C}$. Hence
$\Bbbk'$ can be embedded into $\mathbb{C}$ 
as a subfield. Hence we may suppose that 
$\Bbbk=\mathbb{C}$. Consider the map
\begin{align*}
&\varphi\colon\mathrm{G}\to\mathrm{PGL}(3,\mathbb{C}),
&& g\mapsto (y_{ij}^g)_{i,j}.
\end{align*}
Note that 
\begin{itemize}
\item[- ] conditions $(1)$ imply that the image of any element 
of~$\mathrm{G}$ is an invertible matrix and that $\varphi$ is
a group automorphism;

\item[- ] conditions $(2)$ lead that this automorphism 
is injective;

\item[- ] conditions $(3)$ imply $\varphi(\mathrm{id})=\mathrm{id}$;

\item[- ] conditions $(4)$ lead that 
$\varphi(\mathrm{G})\subset\mathrm{PGL}(3,\mathbb{C})\subset\mathrm{GL}(8,\mathbb{C})$.
\end{itemize}

\item[$\diamond$] Denote by $S_6$ the del Pezzo
surface of degree $6$. If any finitely generated 
subgroup of $\mathrm{G}$ can be embedded into 
$\mathrm{Aut}(S_6)\simeq\mathrm{D}_2\rtimes\Big(\faktor{\mathbb{Z}}{2\mathbb{Z}}\times\mathfrak{S}_3\Big)$ 
a similar reasoning leads to: $\mathrm{G}$ is 
isomorphic to a subgroup of $\mathrm{Aut}(S_6)$.

\item[$\diamond$] If any finitely generated 
subgroup of $\mathrm{G}$ can be embedded 
into 
\[
\mathrm{Aut}(\mathbb{P}^1_\mathbb{C}\times\mathbb{P}^1_\mathbb{C})\simeq\big(\mathrm{PGL}(2,\mathbb{C})\times\mathrm{PGL}(2,\mathbb{C})\big)\rtimes\faktor{\mathbb{Z}}{2\mathbb{Z}},
\]
then $\mathrm{G}$ is isomorphic to a subgroup of
$\mathrm{Aut}(\mathbb{P}^1_\mathbb{C}\times\mathbb{P}^1_\mathbb{C})$.  

\item[$\diamond$] If any finitely generated subgroup
of $\mathrm{G}$ can be embedded into 
\[
\mathrm{Aut}(\mathbb{F}_{2n})\simeq\mathbb{C}[z_0,z_1]_{2n}\rtimes\faktor{\mathrm{GL}(2,\mathbb{C})}{\mu_{2n}}
\] 
for some $n>0$ (and not necessarily the same for 
all finitely generated subgroups of $\mathrm{G}$), 
then $\mathrm{G}$ is isomorphic to a subgroup of 
$\mathrm{GL}(2,\mathbb{C})$ and thus can be 
embedded in $\mathrm{PGL}(3,\mathbb{C})$.

\item[$\diamond$] It remains to consider the case 
where $\mathrm{G}$ contains
\begin{itemize}
\item[-] a finitely generated subgroup $\Gamma_1$
that can not be embedded into 
$\mathrm{Aut}(\mathbb{P}^2_\mathbb{C})$,

\item[-] a finitely generated subgroup $\Gamma_2$
that can not be embedded into 
$\mathrm{Aut}(S_6)$,

\item[-] a finitely generated subgroup $\Gamma_3$
that can not be embedded into 
$\mathrm{Aut}(\mathbb{P}^1_\mathbb{C}\times\mathbb{P}^1_\mathbb{C})$,

\item[-] a finitely generated subgroup $\Gamma_4$
that can not be embedded into 
$\mathrm{Aut}(\mathbb{F}_{2n})$ for all $n>0$.
\end{itemize}
The finitely generated subgroup 
$\Gamma=\langle\Gamma_1,\,\Gamma_2,\,\Gamma_3,\,\Gamma_4\rangle$
is not isomorphic to any subgroup of infinite
automorphisms group of a del Pezzo
surface. Adding finitely many elements if 
needed we may assume that $\Gamma$ has order 
$>648$; as a consequence $\Gamma$ is isomorphic
neither to any subgroup of an automorphisms 
group of a del Pezzo surface 
(Theorem~\ref{thm:648}), nor to a subgroup of 
$\mathrm{Aut}(\mathbb{F}_{2n})$ for all $n>0$. 
Consider a finitely generated 
subgroup~$\mathrm{H}$ of $\mathrm{G}$. The finitely 
generated subgroup 
$\langle\Gamma,\,\mathrm{H}\rangle$, and in
particular $\mathrm{H}$, is isomorphic to 
a subgroup of (Theorem \ref{thm:blanc11cases})
\begin{itemize}
\item either $\mathrm{Aut}(S,\pi)$ where 
$\pi\colon S\to\mathbb{P}^1_\mathbb{C}$ is 
an exceptional conic bundle, 

\item or $\mathrm{Aut}(S,\pi)$ where $(S,\pi)$ 
is a $\Big(\faktor{\mathbb{Z}}{2\mathbb{Z}}\Big)^2$-conic bundle
and $S$ is not a del Pezzo surface,

\item or $\mathrm{Aut}(\mathbb{F}_{2n+1})$ 
for some $n>0$.
\end{itemize}
According to Lemmas \ref{lem:ru1}, \ref{lem:ru2} 
and \ref{lem:ru3} the group $\mathrm{H}$ is 
isomorphic to a subgroup of 
$\mathrm{PGL}(2,\mathbb{C})\times\mathrm{PGL}(2,\mathbb{C})$.
Therefore, every finitely generated subgroup of 
$\mathrm{G}$ is isomorphic to a subgroup of 
$\mathrm{PGL}(2,\mathbb{C})\times\mathrm{PGL}(2,\mathbb{C})$.
The group $\mathrm{G}$ is thus isomorphic to a 
subgroup of 
$\mathrm{PGL}(2,\mathbb{C})\times\mathrm{PGL}(2,\mathbb{C})$ 
(Theorem \ref{thm:Malcev}) and hence to a 
subgroup of 
$\mathrm{Aut}(\mathbb{P}^1_\mathbb{C}\times\mathbb{P}^1_\mathbb{C})$. 
\end{itemize}

\begin{lem}[\cite{Urech:ellipticsubgroups}]
Every torsion subgroup of 
$\mathrm{Bir}(\mathbb{P}^2_\mathbb{C})$ is 
isomorphic to a subgroup of 
$\mathrm{GL}(48,\mathbb{C})$.
\end{lem}

\begin{proof}
Let $\mathrm{G}$ be a torsion group of 
$\mathrm{Bir}(\mathbb{P}^2_\mathbb{C})$.

\begin{itemize}
\item[$\diamond$] Assume that $\mathrm{G}$ is 
infinite. As we just see $\mathrm{G}$ is 
isomorphic to a subgroup of 
$\mathrm{Aut}(\mathbb{P}^2_\mathbb{C})$, 
$\mathrm{Aut}(\mathbb{P}^1_\mathbb{C}\times\mathbb{P}^1_\mathbb{C})$, 
$\mathrm{Aut}(S_6)$ or 
$\mathrm{Aut}(\mathbb{F}_n)$ for some $n\geq 2$.
According to the structure of 
$\mathrm{Aut}(\mathbb{F}_n)$ and Lemma 
\ref{lem:ru1} all torsion subgroups of 
$\mathrm{Aut}(\mathbb{F}_n)$ are isomorphic
to a subgroup of $\mathrm{GL}(2,\mathbb{C})$
or $\mathrm{PGL}(2,\mathbb{C})\times\mathbb{C}^*$.
But $\mathrm{PGL}(2,\mathbb{C})$ can be 
embedded into $\mathrm{GL}(3,\mathbb{C})$
and $\mathrm{PGL}(3,\mathbb{C})$ into 
$\mathrm{GL}(8,\mathbb{C})$, and 
$\mathrm{Aut}(S_6)$ into 
$\mathrm{GL}(6,\mathbb{C})$ (Lemma 
\ref{lem:te2}); the group $\mathrm{G}$ is
thus isomorphic to a subgroup of 
$\mathrm{GL}(8,\mathbb{C})$. 

\item[$\diamond$] Suppose that $\mathrm{G}$
is finite and not contained in an infinite 
bounded subgroup. Then $\mathrm{G}$ is 
contained in the automorphism group 
(Theorem \ref{thm:blanc11cases})
\begin{itemize}
\item either of a del Pezzo
surface,

\item or of an exceptional
fibration, 

\item or of a 
$(\faktor{\mathbb{Z}}{2\mathbb{Z}})^2$-fibration.
\end{itemize}
In the first case we get from Lemma 
\ref{lem:ru4} that $\mathrm{G}$ is 
isomorphic to a subgroup of 
$\mathrm{GL}(8,\mathbb{C})$. 

In the second case the group $\mathrm{G}$ 
can be embedded into 
$\mathrm{PGL}(2,\mathbb{C})\times\mathrm{PGL}(2,\mathbb{C})$
(Lemma \ref{lem:ru2}).

In the last case $\mathrm{G}$ is isomorphic
to a subgroup of $\mathrm{GL}(48,\mathbb{C})$
according to \cite[Lemma 6.2.12]{Urech:phd}.
\end{itemize}
\end{proof}


\section{Tits alternative and Burnside problem}\label{sec:tits}

A group $\mathrm{G}$ is \textsl{virtually solvable}\index{defi}{virtually solvable (group)}
if $\mathrm{G}$ contains a finite index solvable subgroup.

A group $\mathrm{G}$ 
\textsl{satisfies Tits alternative}\index{defi}{Tits alternative} 
if every subgroup of $\mathrm{G}$ either is virtually solvable 
or contains a non-abelian free subgroup.

A group $\mathrm{G}$ \textsl{satisfies Tits alternative for finitely 
generated subgroups}\index{defi}{Tits alternative for finitely 
generated subgroups} if every finitely gene\-rated subgroup of 
$\mathrm{G}$ either is virtually solvable or contains a non-abelian 
free subgroup.

Tits showed that linear groups over fields of characteristic zero 
satisfy the Tits alternative and that linear groups over fields 
of positive characteristic satisfy the Tits alternative for 
finitely generated subgroups (\cite{Tits}). Other well-known 
examples of groups that satisfy  Tits alternative include 
mapping class groups of surfaces (\cite{Ivanov}), the outer 
automorphisms group of the free group of finite rank~$n$
(\cite{BFH}), or hyperbolic groups in the sense of Gromov 
(\cite{Gromov:hypgroups}). Lamy studied the group 
$\mathrm{Aut}(\mathbb{A}^2_\mathbb{C})$; 
in particular using its amalgamated product structure he 
showed that Tits alternative holds for 
$\mathrm{Aut}(\mathbb{A}^2_\mathbb{C})$ (\emph{see} \cite{Lamy}). 
In \cite{Cantat:annals} Cantat established that 
$\mathrm{Bir}(\mathbb{P}^2_\mathbb{C})$ satisfies Tits 
alternative for finitely generated subgroups. Then 
Urech proved that $\mathrm{Bir}(\mathbb{P}^2_\mathbb{C})$ 
satisfies Tits alternative (\cite{Urech:simplesubgroups}).

On the contrary the group of $\mathcal{C}^\infty$-diffeomorphisms 
of the circle does not satisfy Tits alternative 
(\cite{BrinSquier, GhysSergiescu}). 

Note that since solvable subgroups
have either polynomial or exponential growth, if $\mathrm{G}$ sa\-tisfies 
Tits alternative, $\mathrm{G}$ does not contain groups with intermediate 
growth.

\smallskip

The main technique to prove that a group contains a non-abelian 
free group is the ping-pong Lemma (\emph{for instance}
\cite{delaHarpe}):

\begin{lem}\label{lem:pingpong}
Let $S$ be a set.
Let $g_1$ and $g_2$ be two bijections of $S$. 
Assume that $S$ contains two non-empty disjoint subsets $S_1$ and 
$S_2$ such that 
\begin{align*}
& g_1^m(S_2)\subset S_1&& g_2^m(S_1)\subset S_2&& \qquad\forall\,m\in\mathbb{Z}\smallsetminus\{0\}. 
\end{align*}

Then $\langle g_1,\,g_2\rangle$ is a free group on two generators.
\end{lem}

\begin{proof}[Sketch of the Proof]
Let $w=w(a,b)$ be a reduced word that represents a non-trivial
element in the free group $\mathbb{F}_2=\langle a,\,b\rangle$.
Let us prove that $w(g_1,g_2)$ is a non-trivial map of $S$. 
Up to conjugacy by a power of $g_1$ assume that $w(g_1,g_2)$
starts and ends with a power of $g_1$:
\[
w(g_1,g_2)=g_1^{\ell_n}g_2^{m_n}\ldots g_2^{m_1}g_1^{\ell_0}.
\]
One checks that $g_1^{\ell_0}$ maps $S_2$ into $S_1$, then 
$g_2^{m_1}g_1^{\ell_0}$ maps $S_2$ into $S_2$, $\ldots$
and $w$ maps~$S_2$ into $S_1$. As $S_2$ is disjoint from
$S_1$ one gets that $w(g_1,g_2)$ is non-trivial.
\end{proof}

Consider a group $\Gamma$ that acts on a hyperbolic space 
$\mathbb{H}^\infty$ and that contains
two loxodromic isometries $\psi_1$ and $\psi_2$ whose fixed
points in $\partial\mathbb{H}^\infty$
form two disjoint pairs. Let us take disjoint neighborhoods
$S_i\subset\overline{\mathbb{H}^\infty}$ 
of the fixed point sets of $\psi_i$, $i=1$,~$2$. Then 
Lemma \ref{lem:pingpong} applied to sufficiently high powers 
$\psi_1^n$ and $\psi_2^n$ of $\psi_1$ and~$\psi_2$ 
respectively produces a free subgroup of $\Gamma$. This 
strategy can be used for the Cremona group acting by 
isometries on $\mathbb{H}^\infty(\mathbb{P}^2_\mathbb{C})$.
More precisely Cantat obtained the following result:

\begin{thm}[\cite{Cantat:annals}]\label{thm:cantattits}
Let $S$ be a projective surface $S$ over a field $\Bbbk$.
The group $\mathrm{Bir}(S)$ satisfies Tits alternative
for finitely generated subgroups.
\end{thm}

Then Urech proves:

\begin{thm}[\cite{Urech:ellipticsubgroups}]\label{thm:urechtits}
Let $S$ be a complex K\"ahler surface.
Then $\mathrm{Bir}(S)$ sa\-tisfies Tits alternative.
\end{thm}

Let us now give a sketch of the proof of this result
in the case $S=\mathbb{P}^2_\mathbb{C}$.

\subsection{Subgroups of $\mathrm{Bir}(\mathbb{P}^2_\mathbb{C})$ that contain a loxodromic element} 

Recall that
\begin{itemize}
\item[$\diamond$] the subgroup of diagonal automorphisms
\[
\mathrm{D}_2=\big\{(z_0,z_1)\mapsto(\alpha z_0,\beta z_1)\,\vert\,\alpha,\,\beta\in\mathbb{C}^*\big\}\subset\mathrm{PGL}(3,\mathbb{C})=\mathrm{Aut}(\mathbb{P}^2_\mathbb{C})
\]
is a real torus of rank $2$;

\item[$\diamond$] a matrix $A=(a_{ij})\in\mathrm{GL}(2,\mathbb{Z})$
determines a birational map of $\mathbb{P}^2_\mathbb{C}$
\[
(z_0,z_1)\dashrightarrow\big(z_0^{a_{00}}z_1^{a_{01}},z_0^{a_{10}}z_1^{a_{11}}\big)
\]
\end{itemize}
The normalizer of $\mathrm{D}_2$ in 
$\mathrm{Bir}(\mathbb{P}^2_\mathbb{C})$ is the 
semidirect product 
\[
\mathrm{Norm}\big(\mathrm{D}_2,\mathrm{Bir}(\mathbb{P}^2_\mathbb{C})\big)=\big\{\phi\in\mathrm{Bir}(\mathbb{P}^2_\mathbb{C})\,\vert\,\phi\circ\mathrm{D}_2\circ\phi^{-1}=\mathrm{D}_2\big\}=\mathrm{GL}(2,\mathbb{Z})\ltimes\mathrm{D}_2.
\]
If $M\in\mathrm{GL}(2,\mathbb{Z})$ has spectral radius
strictly larger than $1$, the associated birational 
map is loxodromic. In particular there exist 
loxodromic elements that normalize an infinite 
elliptic subgroup. Up to conjugacy these are the 
only examples with this property:

\begin{thm}[\cite{DelzantPy}]\label{thm:DelzantPy}
Let $\mathrm{G}$ be a subgroup of 
$\mathrm{Bir}(\mathbb{P}^2_\mathbb{C})$ containing
at least one loxodromic element. Assume that 
there exists a short exact sequence 
\[
1\longrightarrow \mathrm{A}\longrightarrow \mathrm{G}\longrightarrow\mathrm{B}\longrightarrow 1
\]
where $\mathrm{A}$ is infinite and of bounded degree.

Then $\mathrm{G}$ is conjugate to a subgroup of 
$\mathrm{GL}(2,\mathbb{Z})\ltimes\mathrm{D}_2$.
\end{thm}

Urech generalizes this result to the case where 
$\mathrm{A}$ is an infinite group of elliptic 
elements (\cite{Urech:ellipticsubgroups}):

\begin{thm}[\cite{Urech:ellipticsubgroups}]\label{thm:urechnorm}
Let $\mathrm{G}$ be a subgroup of 
$\mathrm{Bir}(\mathbb{P}^2_\mathbb{C})$ 
containing at least one loxodromic element.
Suppose that there exists a short sequence
\[
1\longrightarrow \mathrm{A}\longrightarrow \mathrm{G}\longrightarrow\mathrm{B}\longrightarrow 1
\]
where $\mathrm{A}$ is an infinite group of 
elliptic elements.

Then $\mathrm{G}$ is conjugate to a subgroup of 
$\mathrm{GL}(2,\mathbb{Z})\ltimes\mathrm{D}_2$.
\end{thm}

In order to give the proof of Theorem 
\ref{thm:urechnorm} we need to establish
some results.

\begin{lem}[\cite{Urech:ellipticsubgroups}]\label{lem:densezar}
Let $\phi$ be a loxodromic monomial map
of the complex projective plane. Let
$\Delta_2$ be an infinite subgroup of
$\mathrm{D}_2$ normalized by $\phi$.

Then $\Delta_2$ is dense in $\mathrm{D}_2$
with respect to the Zariski
topology.
\end{lem}

\begin{proof}
Denote by $\overline{\Delta_2}^{\,0}$ the 
neutral component of the Zariski
closure of $\Delta_2$.

If $\overline{\Delta_2}^{\,0}$ has a dense
orbit on $\mathbb{P}^2_\mathbb{C}$, 
then $\Delta_2$ is dense in 
$\mathrm{D}_2$. Otherwise the dimension
of the generic orbits of $\overline{\Delta_2}^{\,0}$
is $1$. But $\phi$ normalizes 
$\overline{\Delta_2}^{\,0}$, so preserves its
orbits. In particular $\phi$ thus 
preserves a rational fibration: 
contradiction with the fact that 
$\phi$ is loxodromic.
\end{proof}

In \cite{ShepherdBarron} the classification
of tight elements of
$\mathrm{Bir}(\mathbb{P}^2_\mathbb{C})$
is given:

\begin{thm}[\cite{ShepherdBarron}]\label{thm:ShepherdBarron}
Every loxodromic element of the plane
Cremona group is rigid.

Let $\phi$ be a loxodromic birational
self map of the complex projective 
plane; then
\begin{itemize}
\item[$\diamond$] if $\phi$ is conjugate to a 
monomial map, no power of $\phi$
is tight; 
\item[$\diamond$] otherwise $\phi^n$ is tight
for some integer $n$.
\end{itemize}
\end{thm}

Consider a subgroup $\mathrm{G}$ of 
$\mathrm{Bir}(\mathbb{P}^2_\mathbb{C})$. 
Let $\phi\in\mathrm{G}$ be a rigid 
element; then $\phi$ is also a rigid
element in $\mathrm{G}$. The same 
holds for tight elements but the 
converse does not: there exist 
loxodromic maps $\phi\in\mathrm{G}$
such that $\phi$ is tight in 
$\mathrm{G}$ but not in 
$\mathrm{Bir}(\mathbb{P}^2_\mathbb{C})$. 

Proof of Theorem \ref{thm:ShepherdBarron}
and Lemma \ref{lem:densezar} imply the 
following:

\begin{thm}[\cite{Urech:ellipticsubgroups}]\label{thm:pfou}
Let $\mathrm{G}$ be a subgroup of 
$\mathrm{Bir}(\mathbb{P}^2_\mathbb{C})$.
Let $\phi$ be a loxodromic element. 
The following assertions are equivalent:
\begin{itemize}
\item[$\diamond$] no power of $\phi$ is 
tight in $\mathrm{G}$;

\item[$\diamond$] there is a subgroup 
$\Delta_2\subset\mathrm{G}$ that is 
normalized by $\phi$ and a birational 
self map $\psi$ of $\mathbb{P}^2_\mathbb{C}$
such that 
$\psi\circ\Delta_2\circ\psi^{-1}$
is a dense subgroup of $\mathrm{D}_2$ and
$\psi\circ\varphi\circ\psi^{-1}$ belongs
to 
$\mathrm{GL}(2,\mathbb{Z})\ltimes\mathrm{D}_2$.
\end{itemize}
\end{thm}

\begin{proof}[Proof of Theorem \ref{thm:urechnorm}]
The group $\mathrm{A}$ fixes a point 
$p\in\partial\mathbb{H}^\infty\cup\mathbb{H}^\infty$
(Theorem \ref{thm:weak}). Note that 
if $p$ belongs to $\mathbb{H}^\infty$, 
then $\mathrm{A}$ is bounded and 
Theorem \ref{thm:DelzantPy} allows
to conclude. Let us assume that $p$
belongs to $\partial\mathbb{H}^\infty$.
Remark that if $\mathrm{A}$ fixes an 
other point $q$ on 
$\partial\mathbb{H}^\infty$, then 
$\mathrm{A}$ fixes the geodesic 
between $p$ and $q$, and so $\mathrm{A}$
would be bounded again. Suppose 
thus that $p$ is the only fixed point
of $\mathrm{A}$ in 
$\partial\mathbb{H}^\infty$. Consider
a loxodromic map $\phi$ of $\mathrm{N}$.
It normalizes $\mathrm{A}$ and so
$\phi$ fixes~$p$. As $\phi$ is loxodromic, 
$\phi$ does not preserve any fibration;
consequently $p$ does not correspond
to the class of a fibration. From
Lemma \ref{lem:Urechtecalt} any finitely
generated group of elliptic elements 
that fixes $p$ is bounded. Let 
$\mathrm{G}$ be the subgroup of birational
self maps of $\mathbb{P}^2_\mathbb{C}$
that fix $p$. Denote by $L$ the 
one-dimensional subspace of 
$\mathrm{Z}(\mathbb{P}^2_\mathbb{C})$
corresponding to $p$. The group
$\mathrm{G}$ fixes $p$; hence its 
linear action on 
$\mathrm{Z}(\mathbb{P}^2_\mathbb{C})$
acts on $L$ by automorphisms 
preserving the orientation. This implies
the existence of a group homomorphism
$\rho\colon\mathrm{G}\to\mathbb{R}^*_+$.
Note that $\mathrm{G}$ does not contain
any parabolic element because $p$ 
does not correspond to the class of a 
fibration and that loxodromic elements
do not fix any vector in 
$\mathrm{Z}(\mathbb{P}^2_\mathbb{C})$.
As a result $\ker\rho$ consists of 
elliptic elements. But $1$ is the only
eigenvalue of a map of 
$\mathrm{Z}(\mathbb{P}^2_\mathbb{C})$
induced by an elliptic birational 
self map (\cite{Cantat:annals}); 
as a consequence any 
elliptic birational map of 
$\mathrm{G}$ is contained in 
$\ker\rho$.

Take a loxodromic map $\phi$ in 
$\mathrm{G}$. Let us show by 
contradiction that no power of
$\phi$ is tight in~$\mathrm{G}$.
So assume that there exists 
$n\in\mathbb{Z}$ such that 
$\phi^n$ is tight in 
$\mathrm{G}$. The subgroup
$\mathrm{N}$ of $\mathrm{G}$ 
is infinite and 
$\langle\phi\rangle$ has finite
index in 
$\mathrm{Cent}(\phi)$ (Theorem \ref{thm:cent}); 
there thus exists 
$\psi\in\mathrm{G}$ that do not
commute with $\phi^n$. Since 
all non trivial elements of 
$\ll\phi^n\gg$ are loxodromic
(\cite{CantatLamy}) the map 
$\psi\circ\phi^n\circ\psi^{-1}\circ\phi^{-n}$ 
is loxodromic. But 
$\rho\big(\psi\circ\phi^n\circ\psi^{-1}\circ\phi^{-n}\big)=1$, {\it i.e.} 
$\psi\circ\phi^n\circ\psi^{-1}\circ\phi^{-n}$
is elliptic: contradiction. Finally
no power of $\phi$ is tight in 
$\mathrm{G}$. According to 
Theorem \ref{thm:pfou} there 
exist 
$\varphi\in\mathrm{Bir}(\mathbb{P}^2_\mathbb{C})$
and $\Delta_2$ an algebraic 
subgroup of $\mathrm{G}$ such 
that 
\begin{itemize}
\item $\varphi\circ\phi\circ\varphi^{-1}$ is monomial; 

\item $\varphi\circ\Delta_2\circ\varphi^{-1}=\mathrm{D}_2$.
\end{itemize}

Consider a finitely generated subgroup
$\Gamma$ of $\ker\rho$. The Zariski
closure $\overline{\Gamma}$ of 
$\Gamma$ is an algebraic subgroup of 
$\mathrm{G}$ because $\Gamma$ is 
bounded. Set
\[
d=\sup\{\dim\overline{\Gamma}\,\vert\,\Gamma\subset\ker\rho\text{ finitely generated}\}
\]
We will distinguish the cases $d$ is finite
and $d$ is infinite.
\begin{itemize}
\item First consider the case $d<\infty$.
Note that $\ker\rho$ contains 
a subgroup conjugated to $\mathrm{D}_2$,
so $d\geq 2$. Take $\Gamma$ a finitely 
generated subgroup of $\ker\rho$ such 
that $\dim\overline{\Gamma}=d$. 
Let $\overline{\Gamma}^0$ be the 
neutral component of the algebraic 
group $\overline{\Gamma}$. Let $\phi$
be an element of $\mathrm{G}$. The 
group 
$\phi\circ\overline{\Gamma}^0\circ\phi^{-1}$
is again an algebraic subgroup and 
$\langle\overline{\Gamma}^0,\,\phi\circ\overline{\Gamma}^0\circ\phi^{-1}\rangle$ 
is contained in 
$\overline{\langle\Gamma,\,\phi\circ\Gamma\circ\phi^{-1}\rangle}$. According to 
\cite{Humphreys} the group 
$\langle\overline{\Gamma}^0,\,\phi\circ\overline{\Gamma}^0\circ\phi^{-1}\rangle$
is closed and connected. On the one hand 
$\dim \langle\overline{\Gamma}^0,\,\phi\circ\overline{\Gamma}^0\circ\phi^{-1}\rangle\leq d$
and on the other hand 
$\overline{\Gamma}^0\subset\langle\overline{\Gamma}^0,\,\phi\circ\overline{\Gamma}^0\circ\phi^{-1}\rangle$.
As a consequence $\langle\overline{\Gamma}^0,\,\phi\circ\overline{\Gamma}^0\circ\phi^{-1}\rangle=\overline{\Gamma}^0$.
In other words $\phi$ 
normalizes~$\overline{\Gamma}^0$. But 
$\Gamma\cap\overline{\Gamma}^0$ is infinite,
so there exists a birational self map
$\psi$ of $\mathrm{Bir}(\mathbb{P}^2_\mathbb{C})$
such that 
$\psi\circ\mathrm{G}\circ\psi^{-1}\subset\mathrm{GL}(2,\mathbb{Z})\ltimes\mathrm{D}_2$ (Theorem
\ref{thm:DelzantPy}) and hence
$\psi\circ\mathrm{N}\circ\psi^{-1}\subset\mathrm{GL}(2,\mathbb{Z})\ltimes\mathrm{D}_2$.

\item Now assume $d=\infty$. Let $\Gamma$
be a finitely generated subgroup of 
$\ker\rho$ such that 
$\dim\overline{\Gamma}\geq~9$. The closure 
$\overline{\Gamma}$ of $\Gamma$ preserves
a unique rational fibration given by a 
rational map 
$\pi\colon\mathbb{P}^2_\mathbb{C}\dashrightarrow\mathbb{P}^1_\mathbb{C}$ (Lemma 
\ref{lem:Urechmoinsde9}). Consider an 
element $\phi$ of $\ker\rho$. The 
algebraic group 
$\overline{\langle\Gamma,\,\phi\rangle}$
also preserves a unique rational 
fibration; since 
$\overline{\Gamma}\subset\langle\overline{\Gamma},\,\phi\rangle$
this fibration is given by $\pi$. As
a result $\ker\rho$ preserves a rational
fibration. Hence $\ker\rho$ is bounded
and the group 
$\phi\circ\mathrm{G}\circ\phi^{-1}$
is contained in
$\mathrm{GL}(2,\mathbb{Z})\ltimes\mathrm{D}_2$
(Theorem \ref{thm:DelzantPy});
in particular 
$\phi\circ\mathrm{N}\circ\phi^{-1}$
is a subgroup of
$\mathrm{GL}(2,\mathbb{Z})\ltimes\mathrm{D}_2$.
\end{itemize}
\end{proof}

\begin{lem}[\cite{Urech:ellipticsubgroups}]\label{lem:withoutaxe}
Let $\phi$ and $\psi$ be two loxodromic 
elements of 
$\mathrm{Bir}(\mathbb{P}^2_\mathbb{C})$
such that 
$\mathrm{Ax}(\phi)\not=\mathrm{Ax}(\psi)$. 
Then 
\begin{itemize}
\item[$\diamond$] either $\phi$ and $\psi$
have not a common fixed point on 
$\partial\mathbb{H}^\infty$,

\item[$\diamond$] or $\langle\phi,\,\psi\rangle$
contains a subgroup $\mathrm{G}$ and there
exists a birational self map $\varphi$ of 
the complex projective plane such that 
\begin{itemize}
\item[- ] $\varphi\circ\langle\phi,\,\psi\rangle\circ\varphi^{-1}\subset\mathrm{GL}(2,\mathbb{Z})\ltimes\mathrm{D}_2$, 

\item[- ] $\varphi\circ\mathrm{G}\circ\varphi^{-1}$ 
is a dense subgroup of $\mathrm{D}_2$.
\end{itemize}
\end{itemize}
\end{lem}

\begin{proof}
Suppose that $\phi$ and $\psi$ have a
common fixed point $p\in\partial\mathbb{H}^\infty$.
Denote by $L$ the one-dimensional subspace
of $\mathrm{Z}(\mathbb{P}^2_\mathbb{C})$
corresponding to $p$. The group 
$\langle\phi,\,\psi\rangle$ generated by
$\phi$ and $\psi$ fixes $p$, so its 
linear action on 
$\mathrm{Z}(\mathbb{P}^2_\mathbb{C})$
acts on $L$ by automorphisms preserving
the orientation. A reasoning analogous 
to that of the proof of Theorem 
\ref{thm:urechnorm} implies the 
existence of a group homomorphism
\[
\rho\colon\langle\phi,\,\psi\rangle\to\mathbb{R}^*_+
\]
whose kernel consists of elliptic
birational maps (\emph{see} Proof
of Theorem \ref{thm:urechnorm}). 

Assume that $\phi^n$ is tight for 
some $n$. Since 
$\mathrm{Ax}(\phi)\not=\mathrm{Ax}(\psi)$
the maps $\phi^n$ and $\psi$ do not 
commute. According to 
\cite{CantatLamy} any non trivial 
element of $\ll\phi^n\gg$ is 
loxodromic. Therefore, on the one 
hand 
$\psi\circ\phi^n\circ\psi^{-1}\circ\phi^{-n}$
is loxodromic, and on the other
hand 
$\rho(\psi\circ\varphi^n\circ\psi^{-1})=1$ 
hence $\psi\circ\varphi^n\circ\psi^{-1}$
is elliptic: contradiction. As a 
result for any $k$ the map $\phi^k$
is not tight in 
$\langle\phi,\,\psi\rangle$. Theorem
\ref{thm:pfou} implies that there 
exist a birational self map 
$\varphi$ of $\mathbb{P}^2_\mathbb{C}$
and a bounded subgroup 
$\Delta_2\subset\langle\phi,\,\psi\rangle$
such that 
\begin{itemize}
\item[$\diamond$] $\varphi\circ\phi\circ\varphi^{-1}$ is monomial;

\item[$\diamond$] $\varphi\circ\Delta_2\circ\varphi^{-1}$ 
is a dense subgroup of $\mathrm{D}_2$.
\end{itemize}
In particular $\ker\rho\supset\Delta_2$
is thus infinite. Theorem 
\ref{thm:urechnorm} allows to 
conclude.
\end{proof}

\begin{lem}[\cite{Urech:ellipticsubgroups}]\label{Urech:without}
Let $\phi$ and $\psi$ be two loxodromic 
elements of 
$\mathrm{Bir}(\mathbb{P}^2_\mathbb{C})$ 
such that 
$\mathrm{Ax}(\phi)\not=\mathrm{Ax}(\psi)$.
Then $\phi$ and $\psi$ have not a common 
fixed point on $\partial\mathbb{H}^\infty$.
\end{lem}

\begin{proof}
Assume by contradiction that $\phi$ and 
$\psi$ have no common fixed point on 
$\partial\mathbb{H}^\infty$. 
Lemma~\ref{lem:withoutaxe} thus implies that 
up to birational conjugacy 
\begin{itemize}
\item[$\diamond$] $\varphi\circ\langle\phi,\,\psi\rangle\circ\varphi^{-1}\subset\mathrm{GL}(2,\mathbb{Z})\ltimes\mathrm{D}_2$, 

\item[$\diamond$] $\varphi\circ\mathrm{G}\circ\varphi^{-1}\subset\mathrm{D}_2$ is a dense subgroup.
\end{itemize}
Let us write
\begin{align*}
& \phi=d_1\circ m_1 && \psi=d_2\circ m_2
\end{align*}
with $d_i\in\mathrm{D}_2$ and 
$m_i\in\mathrm{GL}(2,\mathbb{Z})$.
The group $\mathrm{D}_2$ fixes the 
axes of all the monomial loxodromic
elements; in particular $m_1$ and 
$m_2$ have the same fixed points on 
$\partial\mathbb{H}^\infty$ as $\phi$
and $\psi$. But the group 
$\langle m_1,\,m_2\rangle$ does not 
contain any infinite abelian group, 
so according to Lemma \ref{lem:withoutaxe}
the birational maps $m_1$ and $m_2$
have not a common fixed point on 
$\partial\mathbb{H}^\infty$: contradiction.
\end{proof}

\begin{lem}[\cite{Urech:ellipticsubgroups}]\label{lem:Urechlox}
Let $\mathrm{G}$ be a subgroup of the 
plane Cremona group that 
contains a loxodromic element. Then
one of the following holds:
\begin{itemize}
\item[$\diamond$] $\mathrm{G}$ is conjugate
to a subgroup of $\mathrm{GL}(2,\mathbb{Z})\ltimes\mathrm{D}_2$;

\item[$\diamond$] $\mathrm{G}$ contains a 
subgroup of index at most $2$ that is 
isomorphic to $\mathbb{Z}\ltimes\mathrm{H}$
where $\mathrm{H}$ is a finite group;

\item[$\diamond$] $\mathrm{G}$ contains a 
non-abelian free subgroup.
\end{itemize}
\end{lem}

\begin{proof}
Let $\phi$ be a loxodromic map of $\mathrm{G}$. 
\begin{itemize}
\item[$\diamond$] Assume first that all elements
in $\mathrm{G}$ preserve the axis $\mathrm{Ax}(\phi)$
of $\phi$. The group $\mathrm{G}$ contains a 
subgroup $\mathrm{H}$ of index at most $2$ 
with the following property: $\mathrm{H}$ 
preserves the orientation of the axis. As 
a result any 
element $\psi\in\mathrm{H}$ translates
the points on $\mathrm{Ax}(\phi)$ by a 
constant $c_\psi$. This yields a group 
morphism
\begin{align*}
&\pi\colon\mathrm{H}\to\mathbb{R}, && \psi\mapsto c_\psi
\end{align*}
such that $\ker\pi$ is a bounded group. From
Theorem \ref{thm:DelzantPy} either $\ker\pi$ 
is finite, or $\mathrm{G}$ is conjugate to a 
subgroup of 
$\mathrm{GL}(2,\mathbb{Z})\ltimes\mathrm{D}_2$.

\item[$\diamond$] Suppose that there is an 
element $\psi\in\mathrm{G}$ that does not 
preserve $\mathrm{Ax}(\phi)$. Denote by 
$\alpha(\phi)$ and $\omega(\phi)$ 
(resp. $\alpha(\psi)$ and $\omega(\psi)$)
the attractive and repulsive fixed points
of $\phi_\bullet$ (resp. $\psi_\bullet$).
Let $\mathcal{U}_1^+$ (resp. $\mathcal{U}_1^-$,
resp. $\mathcal{U}_2^+$, resp. 
$\mathcal{U}_2^-$) be a small neighborhood
of $\alpha(\phi)$ (resp. $\omega(\phi)$, 
resp. $\alpha(\psi)$, resp. $\omega(\psi)$)
in $\partial\mathbb{H}^\infty$. We can 
assume that $\mathcal{U}_1^+$, $\mathcal{U}_1^-$, 
$\mathcal{U}_2^+$ and $\mathcal{U}_2^-$
are pairwise disjoint. Set 
$\mathcal{U}_1=\mathcal{U}_1^+\cup\mathcal{U}_1^-$
and 
$\mathcal{U}_2=\mathcal{U}_2^+\cup\mathcal{U}_2^-$.
There exist $n_1$, $n_2$, $n_3$, $n_4$ some 
positive integers such that 
\begin{align*}
&\phi^{n_1}(\mathcal{U}_2)\subset\mathcal{U}_1^+, && \phi^{-n_2}(\mathcal{U}_2)\subset\mathcal{U}_1^-,
&& \psi^{n_3}(\mathcal{U}_1)\subset\mathcal{U}_2^+,&&\psi^{-n_4}(\mathcal{U}_1)\subset\mathcal{U}_2^-.
\end{align*}
Set $n=\max(n_1,n_2,n_3,n_4)$. Since 
\begin{align*}
&\phi(\mathcal{U}_1^+)\subset\mathcal{U}_1^+ && \phi^{-1}(\mathcal{U}_1^-)\subset\mathcal{U}_1^-&& \psi(\mathcal{U}_2^+)\subset\mathcal{U}_2^+ && \psi^{-1}(\mathcal{U}_2^-)\subset\mathcal{U}_2^-
\end{align*}
one gets that for any $k\leq n$
\begin{align*}
&\phi^k(\mathcal{U}_2)\subset\mathcal{U}_1 &&\phi^{-k}(\mathcal{U}_2)\subset\mathcal{U}_1 &&\psi^k(\mathcal{U}_1)\subset\mathcal{U}_2 && \psi^{-k}(\mathcal{U}_1)\subset\mathcal{U}_2
\end{align*}
\end{itemize}
According to Ping Pong Lemma applied to $\phi^n$, 
$\psi^n$ together with $\mathcal{U}_1$ and 
$\mathcal{U}_2$ we get that 
$\langle\psi^n,\,\phi^n\rangle$ generates a 
non-abelian free subgroup of $\mathrm{G}$.
\end{proof}

\subsection{Tits alternative for finitely generated 
subgroups for automorphisms groups and Jonqui\`eres group}

\begin{lem}[\cite{Cantat:annals}]\label{lem:extsolv}
Let $\mathrm{G}$ be a finitely generated group. Assume 
that $\mathrm{G}$ is an extension of a virtually solvable
group $\mathrm{R}$ of length $r$ by an other virtually 
solvable group $\mathrm{Q}$ of length $q$
\[
1\longrightarrow \mathrm{R}\longrightarrow\mathrm{G}\longrightarrow\mathrm{Q}\longrightarrow 1.
\]
Then $\mathrm{G}$ is virtually solvable of length 
$\leq q+r+1$.
\end{lem}

Hence one has the following statement:

\begin{pro}[\cite{Cantat:annals}]\label{pro:titsext}
Let $\mathrm{G}_1$ and $\mathrm{G}_2$ be two 
groups that satisfy Tits alternative. 

If $\mathrm{G}$ is an extension of $\mathrm{G}_1$
by $\mathrm{G}_2$, then $\mathrm{G}$ satisfies 
Tits alternative.
\end{pro}

\begin{proof}
Let $\Gamma$ be a subgroup of $\mathrm{G}$ that 
does not contain a non abelian free subgroup. 
For $i\in\{1,\,2\}$ denote by 
$\mathrm{pr}_i\colon\mathrm{G}\to\mathrm{G}_i$ the 
canonical projection. Since $\mathrm{pr}_i(\mathrm{G})$
does not contain a non abelian free subgroup
$\mathrm{pr}_i(\Gamma)=\Gamma\cap\mathrm{G}_i$ 
is virtually solvable ($\mathrm{G}_i$
satisfies Tits alternative). Then according 
to Lemma \ref{lem:extsolv} the group
$\Gamma$ is virtually solvable. 
\end{proof}

A first consequence of this result is the following 
one:

\begin{thm}[\cite{Cantat:annals}]\label{thm:titsaut}
Let $V$ be a K\"ahler compact manifold. Its 
automorphism group satisfies Tits alternative.
\end{thm}

\begin{proof}
The group $\mathrm{Aut}(V)$ acts on the 
cohomology of $V$. This yields to a morphism
$\rho$ from $\mathrm{Aut}(V)$ to
$\mathrm{GL}(H^*(V,\mathbb{Z}))$ where 
$H^*(V,\mathbb{Z})$ denotes the direct sum
of the cohomology groups of $V$.
According to \cite{Lieberman} 
\begin{itemize}
\item[$\diamond$] the kernel of $\rho$ is a 
complex Lie group with a finite number of 
connected components;

\item[$\diamond$] its connected component 
$\mathrm{Aut}^0(V)$ is an extension of a 
compact complex torus by a complex algebraic
group. We get the result from 
Proposition \ref{pro:titsext} and classical 
Tits alternative.
\end{itemize}
\end{proof}

A direct consequence of Proposition \ref{pro:titsext}
and Tits alternative for linear groups is:

\begin{pro}[\cite{Cantat:annals}]\label{pro:Titsjonq}
The Jonqui\`eres group 
\[
\mathcal{J}\simeq\mathrm{PGL}(2,\mathbb{C})\rtimes\mathrm{PGL}(2,\mathbb{C}(z_0))
\]
satisfies Tits alternative.
\end{pro}

\subsection{"Weak alternative" for isometries of $\mathbb{H}^\infty$}

Let us recall some notations and definitions introduced 
in Chapter \ref{chap:hyperbolicspace}.

Let $\mathcal{H}$ be a seperable Hilbert space. Let us 
fix a Hilbert basis $\mathcal{B}=(\mathbf{e}_i)_i$ on $\mathcal{H}$. 
Consider the scalar product defined on $\mathcal{H}$ by
\[
\langle v,\,v\rangle=v_0^2-\displaystyle\sum_{i=1}^\infty v_i^2
\]
where the coordinates $v_i$ are the coordinates of 
$v$ in $\mathcal{B}$. The 
\textsl{light cone of 
$\mathcal{H}$}\index{defi}{light cone (of a seperable Hilbert space)} 
is the set 
\[
\mathcal{L}(\mathcal{H})=\big\{v\in\mathcal{H}\,\vert\,\langle v,\,v\rangle=0\big\}.
\]
Let $\mathbb{H}^\infty$ be the connected component of 
the hyperboloid
\[
\big\{v\in\mathcal{H}\,\vert\,\langle v,\,v\rangle=1\big\}
\]
that contains $\mathbf{e}_0$. Consider the metric 
defined on $\mathbb{H}^\infty$ by 
\[
d(u,v):=\mathrm{arccos}(\langle u,\,v\rangle).
\]
The space $\mathbb{H}^\infty$ is a complete
CAT$(-1)$ space, so is hyperbolic (Chapter
\ref{chap:hyperbolicspace}). Its boundary 
$\partial\mathbb{H}^\infty$ can be identified
to $\mathbb{P}(\mathcal{L}(\mathcal{H}))$.

\begin{thm}[\cite{Cantat:annals}]\label{thm:weak}
Let $\Gamma$ be a subgroup of $\mathrm{O}(1,\infty)$.
\begin{enumerate}
\item If $\Gamma$ contains a loxodromic isometry $\psi$, 
then one of the following properties holds:
\begin{itemize}
\item[$\diamond$] $\Gamma$ contains a non-abelian free
group,

\item[$\diamond$] $\Gamma$ permutes the two fixed points
of $\psi$ that lie on $\partial\mathbb{H}^\infty$.
\end{itemize}

\item If $\Gamma$ contains no loxodromic isometry, 
then $\Gamma$ fixes a point of
$\mathbb{H}^\infty\cup\partial\mathbb{H}^\infty$.
\end{enumerate}
\end{thm}

\begin{proof}
\begin{itemize}
\item[$\diamond$] Assume first that $\Gamma$ contains 
two loxodromic isometries $\phi$ and $\psi$ such that the 
fixed points of $\phi$ and $\psi$ on $\partial\mathbb{H}^\infty$
are pairwise distinct. According to the ping-pong Lemma
(Lemma \ref{lem:pingpong}) there are two integers $n$
and $m$ such that $\phi^n$ and $\psi^m$ generate a subgroup
of~$\Gamma$ isomorphic to the free group $\mathbb{F}_2$.

\item[$\diamond$] Suppose that $\Gamma$ contains at 
least one loxodromic isometry $\phi$. Let $\alpha(\phi)$ 
and $\omega(\phi)$ be the fixed points of $\phi$ on 
$\partial\mathbb{H}^\infty$. If $\Gamma$ contains
an element $\psi$ such that 
\[
\big\{\alpha(\phi),\,\omega(\phi)\big\}\cap\big\{\alpha(\psi),\,\omega(\psi)\big\}=\emptyset
\]
then $\phi$ and $\psi\circ\phi\circ\psi^{-1}$ are two loxodromic 
isometries to which we can apply the previous argument. 
Otherwise $\Gamma$ fixes either $\big\{\alpha(\phi),\,\omega(\phi)\big\}$,
or $\{\alpha(\phi)\}$, or $\{\omega(\phi)\}$. Then $\Gamma$
contains a subgroup of index $2$ that fixes 
$\alpha(\phi)$ of $\omega(\phi)$.

\item[$\diamond$] Assume that $\Gamma$ contains two 
parabolic isometries $\phi$ and $\psi$ whose fixed points 
$\alpha(\phi)\in\partial\mathbb{H}^\infty$ and 
$\alpha(\psi)\in\partial\mathbb{H}^\infty$ are 
distinct. Take two elements of $\mathcal{L}(\mathcal{H})$
still denoted $\alpha(\phi)$ and $\alpha(\psi)$ that represent
these two points of $\partial\mathbb{H}^\infty$. Let 
$\ell$ be a point of $\mathcal{H}$ such that 
\begin{align*}
\langle\alpha(\phi),\,\ell\rangle<0 && \langle\alpha(\psi),\,\ell\rangle>0 .
\end{align*}
The hyperplane of $\mathcal{H}$ orthogonal to $\ell$
intersects $\mathbb{H}^\infty$ in a subspace $L$ 
that "separates" $\alpha(\phi)$ and $\alpha(\psi)$. As a 
result there exist integers $n$ and $m$ such that
$\phi^m(L)$, $\phi^{-m}(L)$, $\psi^n(L)$ and $\psi^{-n}(L)$ don't
pairwise intersect. The isometry $\phi^m\circ\psi^n$ has thus 
two distinct fixed points on $\partial\mathbb{H}^\infty$;
hence it is a loxodromic one. Applying the above 
argument we get that $\langle\phi,\,\psi\rangle$ contains
a free group. Therefore, if $\Gamma$ contains at 
least one parabolic isometry, then
\begin{itemize}
\item[-] either $\Gamma$ contains a non-abelian free 
group;

\item[-] or $\Gamma$ fixes a point of 
$\partial\mathbb{H}^\infty$ that is the unique 
fixed point of the parabolic isometries of $\Gamma$.
\end{itemize}

\item[$\diamond$] Let us finish by assuming that
all elements of $\Gamma$ are elliptic ones.
According to \cite[Chapter $8$, Lemma $35$ and Corollary $36$]{delaHarpeGhys} 
\begin{itemize}
\item[-] either the orbit of any point of $\mathbb{H}^\infty$ is bounded;

\item[-] or the limit set of $\Gamma$ is a point.
\end{itemize}
From \cite[Chapter $2$, \S b.8]{delaHarpeValette} one gets 
the following alternative: $\Gamma$ fixes 
\begin{itemize}
\item[-] either a point of $\mathbb{H}^\infty$;

\item[-] or a point of $\partial\mathbb{H}^\infty$.
\end{itemize}
\end{itemize}
\end{proof}

\subsection{Proof of Theorem \ref{thm:urechtits}}

\subsubsection{Assume that $\mathrm{G}$ contains a 
loxodromic element}

Let $\mathrm{G}$ be a subgroup of 
$\mathrm{Bir}(\mathbb{P}^2_\mathbb{C})$
that contains a loxodromic element.  
According to Lemma \ref{lem:Urechlox} we have to 
consider the three following cases:
\begin{itemize}
\item[$\diamond$] $\mathrm{G}$ is conjugate
to a subgroup of $\mathrm{GL}(2,\mathbb{Z})\ltimes\mathrm{D}_2$
and Tits alternative holds by 
Proposition~\ref{pro:titsext};

\item[$\diamond$] $\mathrm{G}$ contains a 
subgroup of index at most $2$ that is 
isomorphic to $\mathbb{Z}\ltimes\mathrm{H}$
where $\mathrm{H}$ is a finite group, in 
other words $\mathrm{G}$ is cyclic up to 
finite index, so Tits alternative holds;

\item[$\diamond$] $\mathrm{G}$ contains a 
non-abelian free subgroup, and Tits 
alternative holds.
\end{itemize}

We can thus state:

\begin{cor}[\cite{Cantat:annals, Urech:ellipticsubgroups}]
Let $\mathrm{G}$ be a subgroup of 
$\mathrm{Bir}(\mathbb{P}^2_\mathbb{C})$
that contains a loxodromic element. 
Then~$\mathrm{G}$ satisfies Tits 
alternative.
\end{cor}

\subsubsection{Assume that $\mathrm{G}$ contains a parabolic element but no loxodromic element}

\begin{lem}[\cite{Urech:ellipticsubgroups}]\label{lem:jonqhalph}
Let $\mathrm{G}$ be a subgroup of 
$\mathrm{Bir}(\mathbb{P}^2_\mathbb{C})$ that does 
not contain any loxodromic element but contains a 
parabolic element. Then $\mathrm{G}$ is conjugate
to a subgroup of $\mathcal{J}$ or $\mathrm{Aut}(S)$, 
where $S$ is a Halphen surface.
\end{lem}

\begin{proof}
By Theorem \ref{thm:weak} the group $\mathrm{G}$
fixes a point 
$p\in\mathbb{H}^\infty\cup\partial\mathbb{H}^\infty$.
Consider a parabolic element $\varphi$ of $\mathrm{G}$;
then $\varphi$ has no fixed point in 
$\mathbb{H}^\infty$ and a unique fixed point $q$ 
in $\partial\mathbb{H}^\infty$. As a consequence
$p=q$. According to Theorem \ref{thm:dilfav}
there exist a surface $S$, a birational map 
$\psi\colon\mathbb{P}^2_\mathbb{C}\dashrightarrow S$, 
a curve $C$, and a fibration $\pi\colon S\to C$ 
such that $\psi\circ\varphi\circ\psi^{-1}$ 
permutes the fibres of $\pi$. In particular
$\psi\circ\varphi\circ\psi^{-1}$ preserves the 
divisor class of a fibre $F$ of $\pi$. Since $F$
is a class of a fibre, $F\cdot F=0$. The point 
$m\in\mathrm{Z}(\mathbb{P}^2_\mathbb{C})$ corresponding 
to $F$, so satisfies $m\cdot m=0$. Therefore,
$q\in\partial\mathbb{H}^\infty$ corresponds to
the line passing through the origin and $m$. 
It follows that any element in $\mathrm{G}$
fixes $m$, and so preserves the divisor class of
$F$. In other words any element in 
$\varphi\circ\mathrm{G}\circ\varphi^{-1}$ 
permutes the fibres of the fibration 
$\pi\colon S\to C$. If the fibration is 
rational, then up to birational conjugacy
$\mathrm{G}\subset\mathcal{J}$. If the fibration
is a fibration of genus $1$ curves, there exists
a Halphen surface $S'$ such that up to birational
conjugacy $\mathrm{G}$ is contained in 
$\mathrm{Bir}(S')$ and preserves the Halphen 
fibration. By Lemma \ref{lem:UrechHalphen} the
group $\mathrm{G}$ is contained in 
$\mathrm{Aut}(S')$.
\end{proof}

Assume first that up to birational conjugacy
$\mathrm{G}\subset\mathcal{J}\simeq\mathrm{PGL}(2,\mathbb{C}(z_1))\rtimes\mathrm{PGL}(2,\mathbb{C})$.
Tits alternative for linear groups in 
characteristic $0$ and Proposition 
\ref{pro:titsext} imply Tits alternative 
for~$\mathrm{G}$.

Finally suppose that 
$\mathrm{G}\subset\mathrm{Aut}(S)$ where
$S$ is a Halphen surface. The automorphisms
groups of Halphen surfaces have been studied
(\cite{Gizatullin, CantatDolgachev, Grivaux}). 
In particular Cantat and Dolgachev prove

\begin{thm}[\cite{CantatDolgachev}]\label{thm:candol}
Let $S$ be a Halphen surface. There 
exists a homomorphism 
$\rho\colon\mathrm{Aut}(S)\to\mathrm{PGL}(2,\mathbb{C})$
with finite image such that $\ker\rho$ is an 
extension of an abelian group of rank $\leq 8$ 
by a cyclic group of order dividing $24$.
\end{thm}

In other words the automorphism group of 
a Halphen surface is virtually abelian 
hence~$\mathrm{G}$ is solvable up to finite
index.

\subsubsection{Assume that $\mathrm{G}$ is a group of elliptic elements}

According to Theorems \ref{thm:urechell1} and \ref{thm:urechell2} 
one of the following holds:
\begin{itemize}
\item[$\diamond$] $\mathrm{G}$ is isomorphic to a 
bounded subgroup;

\item[$\diamond$] $\mathrm{G}$ preserves a rational
fibration.
\end{itemize}

Suppose that $\mathrm{G}$ is isomorphic to a bounded
subgroup; in particular $\mathrm{G}$ is isomorphic
to a subgroup of linear groups, and so satisfies Tits
alternative. 

If $\mathrm{G}$ preserves a rational fibration, then
$\mathrm{G}$ satisfies Tits alternative (Proposition
\ref{pro:Titsjonq}).

\subsection{A consequence of Tits alternative: the Burnside problem}

The Burnside problem posed by Burnside in $1902$ asks 
whether a finitely generated torsion group is finite.
Schur showed in $1911$ that any finitely generated
torsion group that is a subgroup of invertible 
$n\times n$ complex matrices is finite (\cite{Schur}).
One of the tool of the proof is the Jordan-Schur 
Theorem. 

In the $1930$'s Burnside asked another related question
called the restricted Burnside problem: if it is known
that a group $\mathrm{G}$ with $m$ generators and 
exponent $n$ is finite, can one conclude that the 
order of $\mathrm{G}$ is bounded by some constant 
depending only on $n$ and $m$ ? In other words 
are there only finitely many finite groups with 
$m$ generators of exponent $n$ up to isomorphism ?

In $1958$ Kostrikin was able to prove that among 
the finite groups with a given number of generators
and a given prime exponent, there exists a largest
one: this provides a solution for the restricted
Burnside problem for the case of prime exponent
(\cite{Kostrikin}).

Later Zelmanov solved the restricted Burnside 
problem for an arbitrary exponent 
(\cite{Zelmanov1991, Zelmanov1992}).

Golod gave a negative answer to the Burnside 
problem for groups that have a complete system
of linear representations (\cite{Golod}).

Later many examples of infinite, finitely 
generated and torsion groups with even bounded
ordres were exhibited (\cite{AdjanNovikov1, 
AdjanNovikov2, AdjanNovikov3, Olshanskii, Ivanov:Burnside, Lysenok}).

The problem raised by Burnside is still open
for homeomorphism (resp. diffeomorphism) groups 
on closed manifolds. Very few examples are 
known.

Cantat gave a positive answer to the Burnside 
problem for the Cremona group:

\begin{thm}[\cite{Cantat:annals}]\label{thm:burnside}
Every finitely generated torsion subgroup of 
$\mathrm{Bir}(\mathbb{P}^2_\mathbb{C})$ is 
finite.
\end{thm}

\begin{proof}
Let $\mathrm{G}$ be a finitely generated 
torsion subgroup of 
$\mathrm{Bir}(\mathbb{P}^2_\mathbb{C})$.
From Tits alternative (Theorem \ref{thm:cantattits})
$\mathrm{G}$ is solvable up to finite
index. Since any torsion, solvable, 
finitely generated group is finite, 
$\mathrm{G}$ is finite.
\end{proof}


\section{Solvable subgroups of $\mathrm{Bir}(\mathbb{P}^2_\mathbb{C})$}\label{sec:solvable}

The study of the solvable subgroups of the plane Cremona 
group starts in \cite{Deserti:resoluble} and goes on in 
\cite{Urech:ellipticsubgroups}. 

\begin{thm}[\cite{Urech:ellipticsubgroups}]\label{thm:resoluble}
Let $\mathrm{G}$ be a solvable subgroup of
$\mathrm{Bir}(\mathbb{P}^2_\mathbb{C})$.
Then one of the following holds:
\begin{itemize}
\item[$\diamond$] $\mathrm{G}$ is a subgroup of elliptic elements, in 
particular $\mathrm{G}$ is isomorphic either to a solvable subgroup 
of $\mathcal{J}$, or to a solvable subgroup of a bounded group;

\item[$\diamond$] $\mathrm{G}$ is conjugate to a subgroup of 
$\mathcal{J}$;

\item[$\diamond$] $\mathrm{G}$ is conjugate to a subgroup of 
the automorphism group of a Halphen surface;

\item[$\diamond$] $\mathrm{G}$ is conjugate to a subgroup of
$\mathrm{GL}(2,\mathbb{Z})\ltimes\mathrm{D}_2$ where
\[
\mathrm{D}_2=\big\{(z_0,z_1)\mapsto(\alpha z_0,\beta z_1)\,\vert\,\alpha,\,\beta\in\mathbb{C}^*\big\};
\]

\item[$\diamond$] $\mathrm{G}$ contains a loxodromic element 
and there exists a finite subgroup
$\mathrm{H}$ of $\mathrm{Bir}(\mathbb{P}^2_\mathbb{C})$ such
that $\mathrm{G}=\mathbb{Z}\ltimes H$. 
\end{itemize}
\end{thm}

\begin{rem}
A solvable subgroup of a bounded group is a solvable 
subgroup from one of the groups that appear in Theorem
\ref{thm:blanc11cases}.
\end{rem}

\begin{rem}
The centralizer of a birational self map of 
$\mathbb{P}^2_\mathbb{C}$ that preserves a unique fibration that is 
rational is virtually solvable (\S\ref{subsubsec:centrjonq}); 
this example illustrates the second case.
\end{rem}

Before giving the proof let us state some consequences. 

The soluble length of a nilpotent subgroup of 
$\mathrm{Bir}(\mathbb{P}^2_\mathbb{C})$ can be bounded by the 
dimension of~$\mathbb{P}^2_\mathbb{C}$ as Epstein and Thurston 
did in the context of Lie algebras and rational vector fields on
a connected complex manifold (\cite{EpsteinThurston}):

\begin{cor}[\cite{Deserti:resoluble}]
Let $\mathrm{G}$ be a nilpotent subgroup of 
$\mathrm{Bir}(\mathbb{P}^2_\mathbb{C})$ that is not a torsion 
group. The soluble length of $\mathrm{G}$ is bounded by $2$.
\end{cor}

Theorem \ref{thm:blanc11cases} allows to prove:

\begin{cor}[\cite{Urech:ellipticsubgroups}]
The derived length of a bounded solvable subgroup of 
$\mathrm{Bir}(\mathbb{P}^2_\mathbb{C})$ is $\leq 5$.

The derived length of a solvable subgroup of 
$\mathrm{Bir}(\mathbb{P}^2_\mathbb{C})$ is at most $8$.
\end{cor}

\begin{proof}[Proof of Theorem \ref{thm:resoluble}]
It decomposes into three parts:
$\mathrm{G}$ contains a loxodromic element;
$\mathrm{G}$ does not contain a loxodromic 
element but $\mathrm{G}$ contains a parabolic
element;
$\mathrm{G}$ is a group of elliptic elements.

\begin{enumerate}
\item Assume first that $\mathrm{G}$ contains a 
loxodromic element. Then Tits alternative and 
Lemma \ref{lem:Urechlox} imply the following 
alternative
\begin{itemize}
\item[$\diamond$] either $\mathrm{G}$ is conjugate
to a subgroup of 
$\mathrm{GL}(2,\mathbb{Z})\ltimes\mathrm{D}_2$,

\item[$\diamond$] or $\mathrm{G}$ contains a 
subgroup of index at most two that is 
isomorphic to $\mathbb{Z}\ltimes\mathrm{H}$
where $\mathrm{H}$ is a finite group.
\end{itemize}

\item Suppose now that $\mathrm{G}$ does not
contain a loxodromic element but $\mathrm{G}$
contains a parabolic element $\phi$. 
The map $\phi$ preserves a unique fibration $\mathcal{F}$
that is elliptic or rational. Let us prove that any element
of $\mathrm{G}$ preserves $\mathcal{F}$.
Denote by $\alpha(\phi)\in\partial\mathbb{H}^\infty$ the 
fixed point of $\phi_*$. Take one element in the light cone
\[
\mathcal{L}\mathrm{Z}(\mathbb{P}^2_\mathbb{C})=\big\{d\in\mathrm{Z}(\mathbb{P}^2_\mathbb{C})\,\vert\, d\cdot d=0\big\}
\]
of $\mathrm{Z}(\mathbb{P}^2_\mathbb{C})$ still denoted
by $\alpha(\phi)$ that represents $\alpha(\phi)$. Assume
by contradiction that there exists $\varphi$ in $\mathrm{G}$ 
such that $\varphi(\alpha(\phi))\not=\alpha(\phi)$. 
The map $\psi=\varphi\circ\phi\circ\varphi^{-1}$ is parabolic
and fixes the unique element $\alpha(\psi)$
of $\mathcal{L}\mathrm{Z}(\mathbb{P}^2_\mathbb{C})$ 
proportional to $\varphi(\alpha(\phi))$. If $\varepsilon>0$ 
let us denote by $\mathcal{U}\big(\alpha,\varepsilon\big)$ 
the set
\[
\mathcal{U}\big(\alpha,\varepsilon\big)=\big\{\ell\in\mathcal{L}\mathrm{Z}(\mathbb{P}^2_\mathbb{C})\,\vert\,\alpha\cdot\ell<\varepsilon\big\}.
\]
Take $\varepsilon>0$ such that 
$\mathcal{U}\big(\alpha(\phi),\varepsilon\big)\cap\mathcal{U}\big(\alpha(\psi),\varepsilon\big)=\emptyset$.
Since $\psi_*$ is parabolic, 
$\psi_*^n\big(\mathcal{U}\big(\alpha(\phi),\varepsilon\big)\big)$ 
is contained in $\mathcal{U}\big(\alpha(\psi),\varepsilon\big)$ 
for $n$ large enough. 
For $m$ sufficiently large the following inclusions hold
\[
\phi_*^m\circ\psi_*^n\big(\mathcal{U}\big(\alpha(\phi),\varepsilon\big)\big)\subset\mathcal{U}\left(\alpha(\phi),\frac{\varepsilon}{2}\right)\subsetneq\mathcal{U}\big(\alpha(\phi),\varepsilon\big).
\]
This implies that $\phi_*^m\circ\psi_*^n$ is loxodromic: contradiction. So 
$\alpha(\phi_*)=\alpha(\varphi_*)$ for any $\varphi\in\mathrm{G}$.
Finally $\mathrm{G}$ is a subgroup either of $\mathcal{J}$, or
of the automorphism group of a Halphen surface.

\item If $\mathrm{G}$ is a group of elliptic elements, then
according to Theorems \ref{thm:urechell1} and 
\ref{thm:urechell2} either $\mathrm{G}$ is a bounded subgroup, 
or $\mathrm{G}$ preserves a rational fibration.
\end{enumerate}
\end{proof}


\section{Normal subgroups of the Cremona group}\label{CantatLamy:passimple}

The strategy of Cantat and Lamy 
to produce strict, non-trivial, normal subgroups of $\mathrm{Bir}(\mathbb{P}^2_\Bbbk)$  is to let 
$\mathrm{Bir}(\mathbb{P}^2_\Bbbk)$ 
act on the hyperbolic space 
$\mathbb{H}^\infty(\mathbb{P}^2_\Bbbk)$. 
In the first part of their paper they define
the notion of tight element: an element 
$\phi$ of $\mathrm{Bir}(\mathbb{P}^2_\Bbbk)$
is \textsl{tight}\index{defi}{tight (birational map)} if it 
satisfies the following three properties:
\begin{itemize}
\item[$\diamond$] $\phi_*\in\mathrm{Isom}(\mathbb{H}^\infty)$
is hyperbolic;

\item[$\diamond$] there exists a positive number 
$\varepsilon$ such that: if $\psi$ belongs to 
$\mathrm{Bir}(\mathbb{P}^2_\Bbbk)$ and 
$\psi_*(\mathrm{Ax}(\phi))$ contains two points 
at distance $\varepsilon$ which are at 
distance at most $1$ from $\mathrm{Ax}(\phi)$,
then $\psi_*(\mathrm{Ax}(\phi))=\mathrm{Ax}(\phi)$;

\item[$\diamond$] if $\psi$ belongs to 
$\mathrm{Bir}(\mathbb{P}^2_\Bbbk)$
and $\psi_*(\mathrm{Ax}(\phi))=\mathrm{Ax}(\phi)$, 
then $\psi\circ\phi\circ\psi^{-1}=\phi$ or
$\psi\circ\phi\circ\psi^{-1}=~\phi^{-1}$.
\end{itemize}

The second property is a rigidity property
of $\mathrm{Ax}(\phi)$ with respect to 
isometries $\psi_*$ for 
$\psi\in\mathrm{Bir}(\mathbb{P}^2_\Bbbk)$;
we say that $\mathrm{Ax}(\phi)$ is 
\textsl{rigid}\index{defi}{rigid (axis)}
under the action of 
$\mathrm{Bir}(\mathbb{P}^2_\Bbbk)$.
The third property means that the 
stabilizer of $\mathrm{Ax}(\phi)$ coincides
with the normalizer of the cyclic group
$\langle\phi\rangle$.

Here since there is no confusion we write $\ll \phi\gg$
for $\ll \phi\gg_{\mathrm{Bir}(\mathbb{P}^2_\Bbbk)}$.

Cantat and Lamy established the 
following statement:

\begin{thm}[\cite{CantatLamy}]\label{thm:canlam}
Let $\Bbbk$ be an algebraically closed field.
If $\phi\in\mathrm{Bir}(\mathbb{P}^2_\Bbbk)$ is 
tight, then there exists a non-zero integer
$n$ such that for any non-trivial element $\psi$ of 
$\ll \phi^n\gg$
\[
\deg \psi\geq\deg(\phi^n).
\]

In particular 
$\ll \phi^n\gg$ is
a proper subgroup of 
$\mathrm{Bir}(\mathbb{P}^2_\Bbbk)$.
\end{thm}

In the second part of their article Cantat 
and Lamy showed that 
$\mathrm{Bir}(\mathbb{P}^2_\Bbbk)$ contains
tight elements. They distinguished two cases:
$\Bbbk=\mathbb{C}$ and $\Bbbk\not=\mathbb{C}$.
Let us focus on the case $\Bbbk=\mathbb{C}$.
They proved that an element $\phi$ of 
$\mathrm{Bir}(\mathbb{P}^2_\mathbb{C})$
of the form $a\circ j$, where $a$ is a 
general element of 
$\mathrm{PGL}(3,\mathbb{C})$ and $j$ is
a Jonqui\`eres twist, is tight. 
Let us explain what general means in this 
context: any element of 
$\mathrm{PGL}(3,\mathbb{C})$ suits after
removing a countable number of Zariski
closed subsets of $\mathrm{PGL}(3,\mathbb{C})$.
More precisely they needed the two following
conditions: 
\begin{itemize}
\item[$\diamond$] the base-points of $\phi$
and $\phi^{-1}$ belong to $\mathbb{P}^2_\mathbb{C}$;

\item[$\diamond$] 
$\mathrm{Base}(\phi^k)\cap\mathrm{Base}(\phi^{-i})=\emptyset$
for any $k$, $i>0$. 
\end{itemize}

In \cite{Lonjou} Lonjou proved the following
statement:

\begin{thm}[\cite{Lonjou}]\label{thm:lonjou}
For any field $\Bbbk$ the plane
Cremona group 
$\mathrm{Bir}(\mathbb{P}^2_\Bbbk)$ is not 
simple.
\end{thm}

She did not use the notion of tight element
but uses the WPD (weakly properly discontinuous)
property. This property was proposed in the 
context of the mapping class group in
\cite{BestvinaFujiwara}. An element $g$
of a group $\mathrm{G}$ satisfies the WPD
property if for any $\varepsilon \geq 0$
for any point $p\in\mathbb{H}^\infty$ there 
exists a positive integer $N$ such that the
set\index{not}{$S(\varepsilon,p;N)$} 
\[
S(\varepsilon,p;N)=\big\{h\in\mathrm{G}\,\vert\,\mathrm{dist}(h(p),p)\leq\varepsilon,\,\mathrm{dist}(h(g^N(p)),g^N(p))\leq\varepsilon\big\}
\]
is finite.
Since the elements studied by Lonjou have 
an axis she followed the terminology
introduced in \cite{Coulon} and said that
the group $\mathrm{G}$ 
\textsl{acts discretely along the axis of $g$}\index{defi}{discrete action along an axis}.

In \cite{DahmaniGuirardelOsin} the authors 
generalized the small cancellation theory 
for groups acting by isometries on 
$\delta$-hyperbolic spaces.

Small cancellation theory and the WPD property
are connected:
\begin{itemize}
\item[$\diamond$] in the normal group generated 
by a family satisfying the small cancellation 
property elements have a large translation 
length (\cite{Guirardel});

\item[$\diamond$] if some element $g$ satisfies
WPD property then the conjugates of 
$\langle g^n\rangle$ form a family satisfying
the small cancellation property.
\end{itemize}

Combining these two statements the following 
holds:

\begin{thm}[\cite{DahmaniGuirardelOsin}]
Let $\varepsilon$ be a positive real number. Let 
$\mathrm{G}$ be a group acting by isometries
on a $\delta$-hyperbolic space $X$. Let $g$
be a loxodromic element of $\mathrm{G}$.
If $\mathrm{G}$ acts discretely along the axis
of $g$, then there exists $n\in\mathbb{N}$
such that for any 
$h\in\ll g^n\gg\smallsetminus\{\mathrm{id}\}$
the translation length $L(h)$ of $h$ satisfies
$L(h)>\varepsilon$. 

In particular, for $n$ big enough 
$\ll g^n\gg$ is a proper
subgroup of $\mathrm{G}$. Furthermore this 
subgroup is free.
\end{thm}

As a result to prove Theorem \ref{thm:lonjou} 
Lonjou needed to exhibit elements satisfying 
the WPD property:

\begin{pro}[\cite{Lonjou}]\label{pro:lonjou}
Let $n\geq 2$ and let $\Bbbk$ be a field of 
characteristic which does not divide $n$. 
Consider the action of 
$\mathrm{Bir}(\mathbb{P}^2_\Bbbk)$ on 
$\mathbb{H}^\infty(\mathbb{P}^2_{\overline{\Bbbk}})$
where $\overline{\Bbbk}$ is the algebraic 
closure of $\Bbbk$. The group 
$\mathrm{Bir}(\mathbb{P}^2_\Bbbk)$ acts 
discretely along the axis of the loxodromic
map
\[
h_n\colon(z_0:z_1:z_2)\dashrightarrow\big(z_1z_2^{n-1}:z_1^n-z_0z_2^{n-1}:z_2^n\big).
\]
\end{pro}

\begin{rem}
If $\Bbbk$ is an algebraically closed field
of characteristic $p>0$, then for any $\ell\geq 1$
one has (\cite{CerveauDeserti:ptdegre})
\[
\ll h_p^\ell \gg=\mathrm{Bir}(\mathbb{P}^2_\Bbbk).
\]
Let us explain why when $\Bbbk=\mathbb{C}$.

Let us first establish that
\begin{equation}\label{relation1}
\ll\sigma_2\gg=\mathrm{Bir}(\mathbb{P}^2_\mathbb{C}).
\end{equation}
Let $\phi$ be a birational self map of the 
complex projective plane. According to the 
Noether and Castelnuovo Theorem 
\[
\phi=(A_1)\circ\sigma_2\circ A_2\circ\sigma_2\circ A_3\circ\ldots\circ A_n\circ(\sigma_2)
\]
where the $A_i$'s belong to 
$\mathrm{PGL}(3,\mathbb{C})$.
The group $\mathrm{PGL}(3,\mathbb{C})$ is 
simple; as a result any $A_i$ can be written 
as
\begin{small}
\[
B_1\circ\big((z_0,z_1)\mapsto(-z_0,-z_1)\big)\circ B_1^{-1}\circ B_2\circ\big((z_0,z_1)\mapsto(-z_0,-z_1)\big)\circ B_2^{-1}\circ\ldots\circ B_n\circ\big((z_0,z_1)\mapsto(-z_0,-z_1)\big)\circ B_n^{-1}
\]
\end{small}
with $B_i$ in $\mathrm{PGL}(3,\mathbb{C})$. 
The involutions $(z_0,z_1)\mapsto(-z_0,-z_1)$ and $\sigma_2$ 
are conjugate; therefore, $\phi$ can be written 
as a composition of conjugates of $\sigma_2$.

\smallskip

Since $\mathrm{PGL}(3,\mathbb{C})$ 
is simple, for any non-trivial element $A$ of 
$\mathrm{PGL}(3,\mathbb{C})$ the involution 
$\iota\colon(z_0,z_1)\mapsto(-z_0,z_1)$ can 
be written as a composition of conjugates of 
$A$. The involutions $\iota$ and $\sigma_2$ 
being conjugate one has
\[
\sigma_2=\varphi_1\circ A\circ \varphi_1^{-1}\circ \varphi_2\circ A\circ \varphi_2^{-1}\circ\ldots\circ \varphi_n\circ A\circ \varphi_n^{-1}
\]
where the $\varphi_i$'s are some elements of 
$\mathrm{Bir}(\mathbb{P}^2_\mathbb{C})$. As a 
result 
$\ll\sigma_2\gg\subset\ll A\gg$. 
But 
$\ll\sigma_2\gg=\mathrm{Bir}(\mathbb{P}^2_\mathbb{C})$ 
(see (\ref{relation1})),
so 
\begin{equation}\label{relation2}
\ll A\gg=\mathrm{Bir}(\mathbb{P}^2_\mathbb{C}).
\end{equation}

\smallskip

If $\phi$ belongs to $\mathrm{PGL}(2,\mathbb{C}(z_1))$, then
\begin{equation}\label{relation3}
\ll\phi\gg=\mathrm{Bir}(\mathbb{P}^2_\mathbb{C})
\end{equation}
Indeed since $\mathrm{PGL}(2,\mathbb{C}(z_1))$
is simple, the involution $\iota$ can be written 
as a composition of conjugates of $\phi$. But
according to (\ref{relation2}) one has
$\ll\iota\gg=\mathrm{Bir}(\mathbb{P}^2_\mathbb{C})$ 
hence 
$\ll\phi\gg=\mathrm{Bir}(\mathbb{P}^2_\mathbb{C})$.

\smallskip

If $\phi$ belongs to $\mathcal{J}$, then
\begin{equation}\label{relation4}
\ll\phi\gg=\mathrm{Bir}(\mathbb{P}^2_\mathbb{C})
\end{equation}
Indeed up to birational conjugacy 
$\phi\colon(z_0,z_1)\dashrightarrow\Big(\phi_1(z_0,z_1),\gamma(z_1)\Big)$ where 
$\gamma$ is an homothety or a translation. Consider 
an element $\psi\colon(z_0,z_1)\dashrightarrow\Big(\psi_1(z_0,z_1),z_1\Big)$ of 
$\mathrm{PGL}(2,\mathbb{C}(z_1))$. The map 
$\varphi=[\phi,\psi]$ belongs to 
\[
\ll\phi\gg\cap\,\mathrm{PGL}(2,\mathbb{C}(z_1)).
\]
If $\psi$ is well chosen, then $\varphi$ is non 
trivial and from (\ref{relation3}) one gets
\[
\ll\phi\gg=\mathrm{Bir}(\mathbb{P}^2_\mathbb{C}).
\]

\smallskip

As a result if $\phi$ is a 
birational self map of the complex projective 
plane such that there exists 
$\psi\in\mathrm{Bir}(\mathbb{P}^2_\mathbb{C})$
for which $[\phi,\psi]$ preserves a rational
fibration, then from (\ref{relation4})
\begin{equation}\label{relation5}
\ll\phi\gg=\mathrm{Bir}(\mathbb{P}^2_\mathbb{C})
\end{equation}

\smallskip

Let 
$\phi\colon(z_0,z_1)\mapsto(z_1,P(z_1)-\delta z_0)$, 
$\delta\in\mathbb{C}^*$, $P\in\mathbb{C}[z_1]$, 
$\deg P\geq 2$, be a H\'enon map. Then 
$\ll\phi\gg=\mathrm{Bir}(\mathbb{P}^2_\mathbb{C})$.
Indeed if $\psi\colon(z_0,z_1)\mapsto(z_0,2z_1)$, 
then $[\phi,\psi]$ preserves the rational fibration
$z_0=$cst; one concludes with (\ref{relation5}).
\end{rem}

More generally over any infinite field of 
characteristic which does not divide $n$
the map $h_n$ does not satisfy the WPD property:
this explains the assumptions of Proposition
\ref{pro:lonjou}.

Let us mention that Lonjou got not only the 
non-simplicity of the plane Cremona 
group from \cite{DahmaniGuirardelOsin} but 
also the following result:

\begin{thm}[\cite{Lonjou}]
Let $\Bbbk$ be a field. The plane 
Cremona group 
\begin{itemize}
\item[$\diamond$] contains free normal subgroups;

\item[$\diamond$] is $SQ$-universal, that is any countable
subgroup embeds in a quotient of 
$\mathrm{Bir}(\mathbb{P}^2_\Bbbk)$.
\end{itemize}
\end{thm}

In \cite{ShepherdBarron} the author proved
that any loxodromic element in the Cremona
group over any field $\Bbbk$ generates a 
proper normal subgroup; as a result the group 
$\mathrm{Bir}(\mathbb{P}^2_\Bbbk)$ is not 
a simple group. He also gave a criterion 
in terms of the translation length of a 
loxodromic map $\phi$ to know if $\phi$ is tight
and hence if $\ll\phi^n\gg$ is a 
proper subgroup of 
$\mathrm{Bir}(\mathbb{P}^2_\Bbbk)$ for some 
$n$.

\begin{rem}
Let us give the relationship between tight 
element and element that satisfies WPD
property.
When we study the action of the Cremona group
on 
$\mathbb{H}^\infty(\mathbb{P}^2_\Bbbk)$ the
axis of any loxodromic element $\phi$ is rigid
and the stabiliser 
\[
\mathrm{Stab}(\mathrm{Ax}(\phi))=\big\{\psi\in\mathrm{Bir}(\mathbb{P}^2_\Bbbk)\,\vert\,\psi(\mathrm{Ax}(\phi))=\mathrm{Ax}(\phi)\big\}
\]
of the axis $\mathrm{Ax}(\phi)$ is virtually
cyclic if and only if some positive iterate 
of $\phi$ is tight 
(\cite{CantatLamy, Lonjou, ShepherdBarron}).
As a result for $N$ large the set 
$S(\varepsilon,p;N)$ is contained in 
$\mathrm{Stab}(\mathrm{Ax}(\phi))$. The map
$\phi$ thus satisfies the WPD property if 
and only if some positive iterate of $\phi$
is tight.
\end{rem}

\begin{rem}
Let us recall that a subgroup 
$\mathrm{H}$ of a group $\mathrm{G}$
is called a 
\textsl{characteristic subgroup}\index{defi}{characteristic (group)}
of $\mathrm{G}$ if for every automorphism
$\varphi$ of $\mathrm{G}$ the inclusion 
$\varphi(\mathrm{H})\subset\mathrm{H}$
holds.

Recall that the examples of elements having 
the WPD property given by Lonjou are the 
H\'enon maps
\[
h_n\colon(z_0:z_1:z_2)\dashrightarrow\big(z_1z_2^{n-1}:z_1^n-z_0z_2^{n-1}:z_2^n\big)
\]
of degree $n$ which is not divisible by the 
characteristic of $\Bbbk$. The group of 
automorphisms of 
$\mathrm{Bir}(\mathbb{P}^2_\mathbb{C})$ is 
generated by inner automorphisms and the 
action of $\mathrm{Aut}(\mathbb{C},+,\cdot)$
(\emph{see} \S \ref{sec:autbir}). 
As $h_n$ is defined over $\mathbb{Z}$ the 
subgroup $\ll h^m\gg$
is a characteristic subgroup of 
$\mathrm{Bir}(\mathbb{P}^2_\mathbb{C})$.
One has the following result:

\begin{pro}[\cite{Cantat:survey}]
The plane Cremona group contains
infinitely many characteristic subgroups.
\end{pro}
\end{rem}


\section{Simple groups of $\mathrm{Bir}(\mathbb{P}^2_\mathbb{C})$}

This section is devoted to the classification of 
simple subgroups of $\mathrm{Bir}(\mathbb{P}^2_\mathbb{C})$
(Theorems \ref{thm:urechsimple1} and 
\ref{thm:urechsimple2}) but also to 
the proof of the following statement:

\begin{thm}[\cite{Urech:simplesubgroups}]
Let $S$ be a complex surface.

If $\mathrm{G}$ is a finitely generated 
simple subgroup of $\mathrm{Bir}(S)$,
then $\mathrm{G}$ is finite.
\end{thm}

\subsection{Simple subgroups of 
$\mathrm{Bir}(\mathbb{P}^2_\mathbb{C})$}

Let us first prove Theorem \ref{thm:urechsimple1}.
Consider a simple group acting non-trivially
on a rational complex surface. Then 
according to Theorems 
\ref{thm:urechsimple2}
and 
\ref{thm:blanc11cases}
the group $\mathrm{G}$ is isomorphic to a 
subgroup of $\mathrm{PGL}(3,\mathbb{C})$.

Conversely the group
$\mathrm{PGL}(3,\mathbb{C})=\mathrm{Aut}(\mathbb{P}^2_\mathbb{C})$
acts by birational maps on $S$. 

\bigskip

Let us now deal with the proof of 
Theorem \ref{thm:urechsimple2}. Let 
$\mathrm{G}$ be a simple subgroup of the 
plane Cremona group.
We distinguish three cases:
\begin{itemize}
\item[(i)] $\mathrm{G}$ contains no 
loxodromic element but a parabolic one;

\item[(ii)] $\mathrm{G}$ is an elliptic
group;

\item[(iii)] $\mathrm{G}$ contains a 
loxodromic element.
\end{itemize}

\begin{itemize}
\item[(i)] Assume that $\mathrm{G}$ contains no 
loxodromic element but a parabolic one.

\begin{lem}[\cite{Urech:simplesubgroups}]
Consider a simple subgroup $\mathrm{G}$
of $\mathrm{Bir}(\mathbb{P}^2_\mathbb{C})$
that contains no loxodromic element but
a parabolic element. 

Then $\mathrm{G}$ is conjugate to a 
subgroup of $\mathcal{J}$ and is 
isomorphic to a subgroup of 
$\mathrm{PGL}(2,\mathbb{C})$.
\end{lem}

\begin{proof}
According to Lemma \ref{lem:jonqhalph}
one has the following alternative:
$\mathrm{G}$ is conjugate
\begin{itemize}
\item either to a subgroup of the 
automorphisms group of a Halphen
surface,

\item or to a subgroup of $\mathcal{J}$.
\end{itemize}

But automorphisms groups of Halphen
surfaces are finite extensions of abelian
subgroups (Theorem \ref{thm:candol}), so 
do not contain infinite simple subgroups. 
As a result $\mathrm{G}$ is conjugate to 
a subgroup of $\mathcal{J}$. The short
exact sequence from the semi-direct product
of~$\mathcal{J}$ is 
\[
1\longrightarrow\mathrm{PGL}(2,\mathbb{C}(z_1))\longrightarrow\mathcal{J}\stackrel{f}{\longrightarrow}\mathrm{PGL}(2,\mathbb{C})\longrightarrow 1
\]
The group $\mathrm{G}$ is simple thus 
contained in the kernel of the image 
of $f$. In both cases $\mathrm{G}$ is 
isomorphic to a subgroup of 
$\mathrm{PGL}(2,\mathbb{C})$.
\end{proof}

\item[(ii)] Suppose that $\mathrm{G}$ is an 
elliptic group.

\begin{lem}[\cite{Urech:simplesubgroups}]
Let $\mathrm{G}$ be a simple subgroup
of the plane Cremona
group of elliptic elements. Then 
\begin{itemize}
\item[$\diamond$] either $\mathrm{G}$
is a subgroup of an algebraic group of
$\mathrm{Bir}(\mathbb{P}^2_\mathbb{C})$,

\item[$\diamond$] or $\mathrm{G}$ is 
conjugate to a subgroup of 
$\mathcal{J}$.
\end{itemize}
\end{lem}

\begin{proof}
According to Theorems \ref{thm:urechell1}
and \ref{thm:urechell2} 
one of the following holds:
\begin{itemize}
\item[$\diamond$] $\mathrm{G}$ is conjugate
to a subgroup of an algebraic group;

\item[$\diamond$] $\mathrm{G}$ preserves a 
rational fibration; 

\item[$\diamond$] $\mathrm{G}$ is a torsion
group and $\mathrm{G}$ is isomorphic to
a subgroup of an algebraic group. 
\end{itemize}
In the first two cases we are done. Let us
assume that we are in the third one. Then 
$\mathrm{G}$ is a linear group and according 
to the Theorem of Jordan and 
Schur $\mathrm{G}$ has a normal 
abelian subgroup of finite index. As a 
consequence $\mathrm{G}$ is finite, and
so algebraic.
\end{proof}

\item[(iii)] Finally we give a sketch
of the proof of 

\begin{thm}[\cite{Urech:simplesubgroups}]
A simple subgroup of 
$\mathrm{Bir}(\mathbb{P}^2_\mathbb{C})$
does not contain any loxodromic element.
\end{thm}

Let $\mathrm{G}$ be a simple subgroup 
of $\mathrm{Bir}(\mathbb{P}^2_\mathbb{C})$.
Assume by contradiction that 
$\mathrm{G}$ contains a loxodromic map
$\phi$. Theorems \ref{thm:canlam} and 
\ref{thm:pfou} imply that $\phi$ is a
monomial map up to birational conjugacy.
Looking at the curves contracted by 
elements of $\mathrm{G}$ Urech
gets that all loxodromic elements of 
$\mathrm{G}$ are contained in 
$\mathrm{GL}(2,\mathbb{Z})\ltimes\mathrm{D}_2$
(\cite[Lemmas 3.17. and 3.18.]{Urech:ellipticsubgroups}).
Consider $\psi$ in $\mathrm{G}$. As 
$\psi\circ\phi\circ\psi^{-1}$ is loxodromic
it is monomial. The axis of 
$\psi\circ\phi\circ\psi^{-1}$ is fixed
pointwise by both 
$\psi\circ\mathrm{D}_2\circ\psi^{-1}$
and $\mathrm{D}_2$. The group $\mathrm{H}$
generated by $\psi\circ\phi\circ\psi^{-1}$
and $\mathrm{D}_2$ is thus bounded and 
according to Theorem \ref{thm:urechnorm}
conjugate to a subgroup of $\mathrm{D}_2$. 
Hence $\psi\circ\mathrm{D}_2\circ\psi^{-1}$
is contained in $\mathrm{D}_2$ and 
$\psi$ belongs to 
$\mathrm{GL}(2,\mathbb{Z})\ltimes\mathrm{D}_2$. 
Consequently we have the inclusion 
$\mathrm{G}\subset\mathrm{GL}(2,\mathbb{Z})\ltimes\mathrm{D}_2$
and get a non trivial morphism 
$\upsilon\colon\mathrm{G}\to\mathrm{GL}(2,\mathbb{Z})$. 
The kernel of $\upsilon$ contains an 
infinite subgroup of $\mathrm{D}_2$ 
normalized by $\phi$ (Lemma \ref{lem:densezar}): 
contradiction with the fact that $\mathrm{G}$
is simple.
\end{itemize}

\subsection{Finitely generated simple 
subgroups of 
$\mathrm{Bir}(\mathbb{P}^2_\mathbb{C})$}

We finish the chapter by giving a sketch of the 
proof of the following statement:

\begin{thm}[\cite{Urech:simplesubgroups}]\label{thm:finitelysimple}
Any finitely generated simple subgroup of the 
plane Cremona group is 
finite.
\end{thm}

This result and the classification of finite
subgroups of 
$\mathrm{Bir}(\mathbb{P}^2_\mathbb{C})$
(\emph{see} \cite{DolgachevIskovskikh}) 
imply:

\begin{cor}[\cite{Urech:simplesubgroups}]
A finitely generated simple subgroup of
$\mathrm{Bir}(\mathbb{P}^2_\mathbb{C})$
is isomorphic~to 
\begin{itemize}
\item[$\diamond$] either $\faktor{\mathbb{Z}}{p\mathbb{Z}}$
for some prime $p$;

\item[$\diamond$] or $\mathcal{A}_5$;

\item[$\diamond$] or $\mathcal{A}_6$;

\item[$\diamond$] or $\mathrm{PSL}(2,\mathbb{C})$.
\end{itemize}
\end{cor}

Note that the conjugacy classes of these finite
groups are also described in \cite{DolgachevIskovskikh}.

\begin{rem}
Theorem \ref{thm:finitelysimple} also holds for 
the group of birational self maps of a surface
over a field $\Bbbk$. 
\end{rem}

Let $\mathrm{G}$ be a finitely generated subgroup 
of $\mathrm{Bir}(\mathbb{P}^2_\mathbb{C})$. Let
first see that $\mathrm{G}$ does not contain 
loxodromic elements:

\begin{pro}[\cite{Urech:simplesubgroups}]\label{pro:finitelysimple}
Let $\mathrm{G}$ be a finitely generated 
subgroup of $\mathrm{Bir}(\mathbb{P}^2_\mathbb{C})$.
If $\mathrm{G}$ contains a loxodromic 
element, then $\mathrm{G}$ is not simple.
\end{pro}

To prove it we need the following statement.

\begin{pro}[\cite{Urech:simplesubgroups}]\label{pro:degxie}
Let $\mathrm{G}$ be a finitely generated 
subgroup of $\mathrm{Bir}(\mathbb{P}^2_\mathbb{C})$.
There exist a finite field $\Bbbk$ and 
a non trivial morphism
$\upsilon\colon\mathrm{G}\to\mathrm{Bir}(\mathbb{P}^2_\Bbbk)$
such that for any $\phi$ in~$\mathrm{G}$
the following inequality holds:
$\deg\upsilon(\phi)\leq\deg\phi$.
\end{pro}

\begin{proof}[Proof of Proposition \ref{pro:finitelysimple}]
Let $\phi$ be a loxodromic element of 
$\mathrm{G}$.

If $\phi^n$ is tight in $\mathrm{G}$ for some 
integer $n$, then Theorem \ref{thm:canlam}
allows to conclude.

If no power of $\phi$ is tight, then $\mathrm{G}$
contains an infinite subgroup $\Delta_2$ that is 
normalized by $\phi$ and that is conjugate 
either to a subgroup of $\mathrm{D}_2$, 
or to a subgroup of $\mathbb{C}^2$ (Theorem
\ref{thm:pfou}). In particular the degrees of 
the elements of $\Delta_2$ are uniformly 
bounded by an integer $N$. According to 
Proposition \ref{pro:degxie} there exist
a finite field $\Bbbk$ and a non trivial 
morphism 
$\upsilon\colon\mathrm{G}\to\mathrm{Bir}(\mathbb{P}^2_\Bbbk)$
such that for all $\phi$ in $\mathrm{G}$
\[
\deg\upsilon(\phi)\leq\deg\phi.
\]
In $\mathrm{Bir}(\mathbb{P}^2_\Bbbk)$ there 
exist only finitely many elements of degree 
$\leq N$. As a result $\upsilon(\Delta_2)$ 
is finite. The morphism $\upsilon$ has thus 
a proper kernel and $\mathrm{G}$ is not 
simple: contradiction.
\end{proof}

We now have the following alternative
\begin{itemize}
\item[(i)] $\mathrm{G}$ contains a parabolic element,

\item[(ii)] $\mathrm{G}$ is an elliptic subgroup.
\end{itemize}

Let us look at these two possibilities.

\begin{itemize}
\item[(i)] If $\mathrm{G}$
contains a parabolic element,  
then $\mathrm{G}$ is conjugate
either to a subgroup of the automorphism group
$\mathrm{Aut}(S)$ of a Halphen surface, 
or to a subgroup of the Jonqui\`eres
group $\mathcal{J}$.

\begin{itemize}
\item[$\diamond$] Assume first that, up to 
conjugacy, 
$\mathrm{G}\subset\mathrm{Aut}(S)$ where 
$S$ is a Halphen surface. 

Recall that a group $\mathrm{G}$ satisfies
\textsl{Malcev property}\index{defi}{Malcev property} if 
every finitely gene\-rated subgroup 
$\Gamma$ of $\mathrm{G}$ is residually finite, 
{\it i.e.} for any $g\in\Gamma$ there exist
a finite group $\mathrm{H}$ and a morphism 
$\upsilon\colon\Gamma\to\mathrm{H}$ such
that $g$ does not belong to $\ker\upsilon$.

Malcev showed that linear groups 
satisfy this property (\cite{Malcev}). 
In \cite{BassLubotzky} the authors proved
that automorphism groups of scheme over
any commutative ring also satisfy this
property. Consequently if $\mathrm{G}$ 
contains a parabolic element, then 
$\mathrm{G}$ is, up to conjugacy, a 
subgroup of $\mathcal{J}$.

\item[$\diamond$] Suppose that 
$\mathrm{G}\subset\mathcal{J}$ up to 
birational conjugacy. Then $\mathrm{G}$
is finite. Indeed:

\begin{lem}[\cite{Urech:simplesubgroups}]\label{lem:mot}
Let $\mathcal{C}$ be a curve and let 
$\mathrm{G}\subset\mathrm{Bir}(\mathbb{P}^1_\mathbb{C}\times\mathcal{C})$ 
be a finitely generated simple subgroup
that preserves the 
$\mathbb{P}^1_\mathbb{C}$-fibration given
by the projection to $\mathcal{C}$.

Then $\mathrm{G}$ is finite.
\end{lem}

\begin{proof}
The group $\mathrm{G}$ being simple, 
$\mathrm{G}$ is isomorphic either to 
a subgroup of $\mathrm{PGL}(2,\mathbb{C})$, 
or to a subgroup of $\mathrm{Aut}(\mathcal{C})$.
But both $\mathrm{PGL}(2,\mathbb{C})$ and 
$\mathrm{Aut}(\mathcal{C})$ satisfy 
Malcev property, so $\mathrm{G}$ 
is finite.
\end{proof}
\end{itemize}

\item[(ii)] It remains to look at 
$\mathrm{G}$ when $\mathrm{G}$ is a 
finitely generated, simple, elliptic 
subgroup of~$\mathrm{Bir}(\mathbb{P}^2_\mathbb{C})$. 
Proposition \ref{pro:etaumilieuCantat}
asserts that either $\mathrm{G}$ is 
conjugate to a subgroup of 
$\mathcal{J}$, or $\mathrm{G}$ is 
contained in an algebraic subgroup 
of 
$\mathrm{Bir}(\mathbb{P}^2_\mathbb{C})$. 
In the first case Lemma \ref{lem:mot}
allows to conclude. Let us focus on the 
last case: algebraic subgroups of 
$\mathrm{Bir}(\mathbb{P}^2_\mathbb{C})$ 
are linear hence~$\mathrm{G}$ is linear
and therefore finite since linear 
groups satisfy Malcev property.
\end{itemize}


\chapter{Big subgroups of automorphisms "of positive entropy"}\label{chapter:dyn}

\bigskip
\bigskip

In this chapter we will focus on automorphisms of surfaces 
with positive entropy. Recall that a
\textsl{K$3$ surface}\index{defi}{K$3$ surface}\footnote{"so named in 
honor of Kummer, K\"ahler, Kodaira 
and of the beautiful mountain K$2$ in Kashmir" (\cite{Weil:collected}).} 
is a complex,
compact, simply connected surface $S$ with a trivial canonical
bundle. Equivalently there exists a holomorphic $2$-form $\omega$
on~$S$ which is never zero; $\omega$ is unique modulo multiplication
by a scalar. Let $S$ be a K$3$ surface with a holomorphic involution
$\iota$. If $\iota$ has no fixed point, the quotient of~$S$ by
$\langle\iota\rangle$ is an
\textsl{Enriques surface}\index{defi}{Enriques surface}, otherwise
it is a rational surface. Recall that every non-minimal
rational surface can be obtained by repeatedly blowing
up a minimal rational surface. The \textsl{minimal rational surfaces}\index{defi}{minimal (rational surface)} are 
the complex projective plane, 
$\mathbb{P}^1_\mathbb{C}\times\mathbb{P}^1_\mathbb{C}$
and the Hirzebruch surfaces $\mathbb{F}_n$, 
$n\geq 2$. If $S$ is a complex, compact surface
carrying a biholomorphism of positive topological entropy, then 
$S$ is either a complex torus, or a K$3$ surface,
or an Enriques surface, or a non-minimal rational surface
(\cite{Cantat2}). Although automorphisms of complex
tori are easy to describe, it is rather difficult to construct
automorphisms on K$3$ surfaces or rational surfaces. Constructions
and dynamical properties of automorphisms of K$3$ surfaces can be
found in \cite{Cantat3} and \cite{McMullen2}.
The first examples of rational surfaces endowed with biholomorphisms
of positive entropy are due to Coble and Kummer (\cite{Coble}):
\begin{itemize}
\item[$\diamond$] the Coble surfaces are obtained by blowing up
  the ten nodes of a nodal sextic in $\mathbb{P}^2_\mathbb{C}$;
  
\item[$\diamond$] the Kummer surfaces are desingularizations of
  quotients of complex $2$-tori by involutions with fixed
  points.
\end{itemize}

Obstructions to the existence of such biholomorphisms on rational
surfaces are also known: if $\phi$ is a biholomorphism of a rational
surface $S$ such that $\mathrm{h_{top}}(\phi)>0$, then the representation
\begin{align*}
&\mathrm{Aut}(S)\to\mathrm{GL}(\mathrm{Pic}(S)) && g\mapsto g^*
\end{align*}
has infinite image. Hence according to \cite{Harbourne} its kernel is 
finite so that $S$ has no non-zero holomorphic vector field. A second
obstruction follows from \cite{Nagata}:
the surface $S$ has to be obtained by successive blowups from the
complex projective plane and the number of blowups must be at least ten.
The first infinite families of examples have been constructed
independently in \cite{McMullen} and \cite{BedfordKim1} by different methods.
Since then many constructions have emerged (\emph{see for instance} 
\cite{BedfordKim2, BedfordKim3, Diller, DesertiGrivaux, Uehara, McMullen}).

In the first section we give three answers to the question "When is 
a birational self map of a complex projective surface birationally 
conjugate to an automorphism~?" In the second section we deal with 
constructions of automorphisms of rational surfaces with positive
entropy. In the last section we explain how 
$\mathrm{SL}(2,\mathbb{Z})$ is realized as a subgroup of automorphisms
of a rational surface with the property that every element of 
infinite order has positive entropy. 


\section{Birational maps and automorphisms}

\subsection{Definitions}

Given a birational map $\phi\colon S\dashrightarrow S$
of a projective complex surface its dynamical 
degree $\lambda(\phi)$ is a positive real number
that measures the complexity of the dynamics of 
$\phi$ (\emph{see} \S \ref{sec:degreegrowth}).
The neperian logarithm $\log\lambda(\phi)$
provides an upper bound for the topological 
entropy of $\phi\colon S\dashrightarrow S$ 
and is equal to it under natural 
assumptions (\cite{BedfordDiller, DinhSibony}).
Let us give an alternative but 
equivalent definition to 
that of \S\ref{sec:degreegrowth}.
A birational map $\phi\colon S\dashrightarrow S$ of 
a projective complex surface determines 
an endomorphism 
$\phi_*\colon\mathrm{NS}(S)\to\mathrm{NS}(S)$; 
the dynamical degree $\lambda(\phi)$ of 
$\phi$ is defined as the spectral radius of 
the sequence of endomorphisms $(\phi^n)_*$ as
$n$ goes to infinity:
\[
\lambda(\phi)=\displaystyle\lim_{n\to +\infty}\vert\vert(\phi^n)_*\vert\vert^{1/n}
\]
where $\vert\vert\cdot\vert\vert$ denotes a 
norm on the real vector space 
$\mathrm{End}(\mathrm{NS}(S))$. This limit
exists and does not depend on the choice 
of the norm. For any ample divisor 
$D\subset S$
\[
\lambda(\phi)=\displaystyle\lim_{n\to +\infty}(D\cdot (\phi^n)_*D)^{1/n}.
\]
The N\'eron-Severi group 
of $\mathbb{P}^2_\mathbb{C}$ coincides with 
the Picard group of 
$\mathbb{P}^2_\mathbb{C}$, has rank $1$, 
and is generated by the class 
$\mathbf{e}_0$ of a line
\[
\mathrm{NS}(\mathbb{P}^2_\mathbb{C})=\mathrm{Pic}(\mathbb{P}^2_\mathbb{C})=\mathbb{Z}\mathbf{e}_0.
\]
A map 
$\phi\in\mathrm{Bir}(\mathbb{P}^2_\mathbb{C})$
acts on $\mathrm{Pic}(\mathbb{P}^2_\mathbb{C})$
by multiplication by $\deg\phi$.

\subsection{Pisot and Salem numbers}

We will give the definitions of 
Pisot and Salem
numbers, for more details \emph{see}
\cite{BDGGHPDS}.

A \textsl{Pisot number}\index{defi}{Pisot number} is
an algebraic integer $\lambda\in]1,+\infty[$ 
whose other Galois conjugates lie 
in the unit disk. Let us denote by 
$\mathrm{Pis}$\index{not}{$\mathrm{Pis}$}
the set of Pisot numbers. It 
includes all integers $\geq 2$ as well
as all reciprocical quadratic integers 
$\lambda>1$. The set $\mathrm{Pis}$ is a 
closed subset of the real line; its infimum
is equal to the unique root
$\lambda_P>1$\index{not}{$\lambda_P$}
of the cubic equation $x^3=x+1$. The 
smallest accumultation point of 
$\mathrm{Pis}$ is the golden mean 
$\lambda_G=\frac{1+\sqrt{5}}{2}$\index{not}{$\lambda_G$}.
Note that all Pisot numbers between 
$\lambda_P$ and $\lambda_G$ have been listed.

A \textsl{Salem number}\index{defi}{Salem number}
is an algebraic integer $\lambda\in]1,+\infty[$
whose other Galois conjugates are in the closed 
unit disk with at least one on the boundary.
The minimal polynomial of $\lambda$ has thus
at least two complex conjugate roots on the 
unit circle, its roots are permuted by the 
involution $z\mapsto \frac{1}{z}$ and has 
degree at least $4$. 
Let $\mathrm{Sal}$\index{not}{$\mathrm{Sal}$}
be the set of Salem numbers. The 
unique root $\lambda_L>1$\index{not}{$\lambda_L$} 
of the irreducible polynomial 
$x^{10}+x^9-x^7-x^6-x^5-x^4-x^3+x+1$ is a
Salem number. Conjecturally the 
infimum of $\mathrm{Sal}$ is larger than 
$1$ and should be equal to $\lambda_L$.

Remark that $\mathrm{Pis}$ is contained 
in the closure of $\mathrm{Sal}$.

\subsection{Dynamical degrees and Pisot and Salem numbers}

Let us recall that a birational map 
$\phi\colon S\dashrightarrow S$ of a 
compact complex surface is algebraically
stable if $(\phi^*)^n=(\phi^n)^*$ for
all $n\geq 0$ (\emph{see} \S 
\ref{sec:degreegrowth}). If $\phi$ is
algebraically stable, then so does
$\phi^{-1}$ and $\lambda(\phi)$ is 
an algebraic integer. Any 
birational map of a compact complex 
surface is conjugate by a birational 
morphism to an algebraically stable 
map (Proposition \ref{pro:DillerFavre}). From this fact and the Hodge
index theorem according to which the intersection
form has signature $(1,r_S-1)$, where 
$r_S$ denotes the rank of~$S$, Diller and 
Favre get the following 
statement:

\begin{thm}[\cite{DillerFavre}]
Let $\phi$ be a birational self map of a
complex projective surface. If 
$\lambda(\phi)$ is distinct from $1$, 
{\it i.e.} if $\phi$ is loxodromic, then
$\lambda(\phi)$ is a Pisot
or a Salem number.
\end{thm}

\subsection{When is a birational map conjugate to an automorphism ?}

A natural question is the following one;
when is a birational self map of a complex
projective surface birationally conjugate
to an automorphism ? There are three 
answers to this question and we will detail
it. 

\subsubsection{A first answer}

Diller and Favre give
the first characterization of loxodromic
birational maps which are conjugate to 
an automorphism of a projective surface:

\begin{thm}[\cite{DillerFavre}]\label{thm:blabla}
Let $\phi\in\mathrm{Bir}(\mathbb{P}^2_\mathbb{C})$
be a loxodromic map. Assume that $\phi$ is 
algebraically stable. The action of $\phi$ on 
$\mathrm{H}^{1,1}(\mathbb{P}^2_\mathbb{C})$
admits the eigenvalue $\lambda(\phi)>1$ with
eigenvector $\Theta(\phi)$. 

The map $\phi$ is birationally conjugate to
an automorphism if and only if 
$\Theta(\phi)\cdot\Theta(\phi)=0$.
\end{thm}

When $\phi$ is an automorphism, it is easy
to check that $\Theta(\phi)\cdot\Theta(\phi)=0$.
We will thus deal with the reciprocical 
property. 
Let $\phi$ be a birational self map of 
a complex projective surface $S$. Assume
that $\phi$ is algebraically stable. Hence
$\lambda(\phi)$ is equal to the spectral
radius of 
$\phi_*\in\mathrm{End}(\mathrm{NS}(\mathbb{R},S))$
but also to the spectral radius of 
$\phi^*=(\phi^{-1})_*$; indeed these 
endomorphisms are adjoint for the 
intersection form: 
\[
\phi_*C\cdot D=C\cdot\phi^*D
\]
for all $C$, $D$ divisor classes. One can
factorize $\phi$ as $\phi=\eta\circ\pi^{-1}$
where $\eta\colon Z\to S$ and $\pi=\pi_1\circ\ldots\circ\pi_m\colon Z\to S$
are two sequences of point blowups. Denote
by $F_j\subset Z$ the total transform of the 
indeterminacy point of $\pi_j^{-1}$ under
the map $\pi_j\circ\ldots\circ\pi_m$. For
$1\leq j\leq m$ let $E_j$ be the direct
image of $F_j$ by $\eta$. Each $E_j$, if 
not zero, is an effective divisor.
According to \cite{DillerFavre} we get the 
following formula called push-pull formula
\begin{equation}\label{eq:square}
\phi_*\phi^*C=C+\displaystyle\sum_{j=1}^m(C\cdot E_j)E_j
\end{equation}
for all curves (resp. divisor classes) 
$C$ in $S$. Since $\phi_*$ and $\phi^*$ 
are adjoint endomorphisms of 
$\mathrm{NS}(\mathbb{R},S)$ for the 
intersection form we get 
\begin{equation}\label{eq:losange}
\phi^*C\cdot\phi^*C=C\cdot C\displaystyle\sum_{j=1}^m(E_j\cdot C)^2
\end{equation}
This formula and the Hodge index 
theorem  imply
that $\lambda(\phi)$ is a Pisot
number or Salem number.

The endomorphisms $\phi^*$ and $\phi_*$
preserve both the pseudo effective and
nef cones of $\mathrm{NS}(\mathbb{R},S)$.
Suppose that $\lambda(\phi)>1$. According
to the Perron-Frobenius
theorem there exists an eigenvector 
$\Theta(\phi)$ for $\phi_*$ in the nef
cone of $\mathrm{NS}(S)$ such that 
\begin{equation}\label{eq:doublelosange}
\phi^*\Theta(\phi)=\lambda(\phi)\Theta(\phi)
\end{equation}
Note that furthermore this vector is 
unique up to scalar form (\cite{DillerFavre}).
Both (\ref{eq:losange}) and 
(\ref{eq:doublelosange}) imply that 
\[
(\lambda(\phi)^2-1)\Theta(\phi)\cdot\Theta(\phi)=\displaystyle\sum_{j=1}^m(E_j\cdot\Theta(\phi))^2.
\]
As a result for all $E_j$
\[
\Theta(\phi)\cdot\Theta(\phi)=0\qquad\Longleftrightarrow
\qquad \Theta(\phi)\cdot E_j=0.
\]
Assume now that $\Theta(\phi)\cdot\Theta(\phi)=0$; 
then $\Theta(\phi)\cdot E_j=0$ for all $E_j$. As
the $E_j$'s are effective and $\Theta(\phi)$
is nef the $\mathbb{Q}$-vector subspace of
$\mathrm{NS}(\mathbb{Q},S)$ generated by the
irreducible components of the divisors $E_j$
is contained in $\Theta(\phi)^\perp$. On the 
orthogonal complement $\Theta(\phi)^\perp$ 
of the isotropic vector $\Theta(\phi)$ the 
intersection form is negative and its kernel
is the line generated by $\Theta(\phi)$. 
Equation (\ref{eq:square}) implies
\[
\phi_*^k\Theta(\phi)=\frac{1}{\lambda(\phi)^k}\Theta(\phi).
\]
But $\lambda(\phi)>1$ and $\phi_*$ preserves
the lattice $\mathrm{NS}(\mathbb{Z},S)$, so 
$\Theta(\phi)$ is irrational. Consequently
the intersection form is negative definite
on the $\mathbb{Q}$-vector space generated
by all classes of irreducible components of
the divisors $E_j$. According to the 
Grauert-Mumford contraction
theorem (\cite{BHPV}) there exists
a birational morphism $\eta\colon S\to Y$
that contracts simultaneously all these 
components. Set 
$\varphi=\eta\circ\phi\circ\eta^{-1}$. As 
$\Theta(\phi)$ does not intersect the 
curves contracted by $\eta$ the class 
$\eta_*\Theta(\phi)\in\mathrm{NS}(\mathbb{R},Y)$ is 
\begin{itemize}
\item[$\diamond$] isotropic, and

\item[$\diamond$] an eigenvector for 
$\varphi_*$ with eigenvalue $\lambda(\phi)$.
\end{itemize}

Let us iterate this process until $\varphi^{-1}$
does not contract any curve, that is 
$\varphi\in\mathrm{Aut}(Y)$. If $Y$ is 
singular, then consider the minimal 
desingularization $\widetilde{Y}$ of $Y$; 
the automorphism $\varphi$ lifts to an 
automorphism $\widetilde{\varphi}$ of 
$\widetilde{Y}$. 

As a result one can state

\begin{thm}[\cite{DillerFavre}]\label{thm:diamond}
Let $S$ be a complex projective surface.
Let $\phi$ be a loxodromic birational 
self map of $S$. Then 
\begin{itemize}
\item[$\diamond$] all divisors $E_j$ are
orthogonal to $\Theta(\phi)$ if and only 
if $\Theta(\phi)$ is an isotropic vector;

\item[$\diamond$] if $\Theta(\phi)$ is an
isotropic vector, then there exists a 
birational morphism $\eta\colon S\to Y$
such that $\eta\circ\phi\circ\eta^{-1}$ 
is an automorphism of $Y$.
\end{itemize}
\end{thm}

Then Diller and Favre
prove the following statement:

\begin{thm}[\cite{DillerFavre}]\label{thm:blablabla}
Let $\phi\in\mathrm{Bir}(S)$ $($resp.  
$\psi\in\mathrm{Bir}(\widetilde{S}))$
be an algebraically stable map of 
a complex projective surface 
$S$ $($resp. $\widetilde{S})$. Assume
that $\phi$ and $\psi$ are conjugate
via a proper modification. Suppose 
that $\lambda(\phi)>1$ $($or equivalently
that $\lambda(\psi)>1)$. Then 
$\Theta(\phi)\cdot\Theta(\phi)=0$
if and only if 
$\Theta(\psi)\cdot\Theta(\psi)=0$.
\end{thm}

Theorem \ref{thm:blabla} follows from
Theorems \ref{thm:diamond} and \ref{thm:blablabla}.

\subsubsection{A second answer}

The following statement gives another 
characterization of birational maps
conjugate to an automorphism of a 
smooth projective rational surface:

\begin{thm}[\cite{DillerFavre, BlancCantat}]\label{thm:fleur}
Let $\phi$ be a birational map of a 
complex projective surface~$S$.
\begin{itemize}
\item[$\diamond$] If $\lambda(\phi)$ is a 
Salem number, then there exists a 
birational map 
$\psi\colon\widetilde{S}\dashrightarrow S$
that conjugates~$\phi$ to an automorphism
of $\widetilde{S}$;

\item[$\diamond$] if $\phi$ is conjugate to
an automorphism, then $\lambda(\phi)$ is a 
quadratic integer or a Salem 
number.
\end{itemize}
\end{thm}

Assume that $\lambda(\phi)$ is a Salem
number. Denote by $P(t)\in\mathbb{Z}[t]$ 
the minimal polynomial of~$\lambda(\phi)$.
But $\lambda(\phi)$ is a Salem
number, so there exists a root of $P$ with 
modulus $1$, denote it $\alpha$. Hence fix
an automorphism $\kappa$ of the field 
$\mathbb{C}$ such that 
$\kappa(\lambda(\phi))=\alpha$. According to
Proposition \ref{pro:DillerFavre} we can 
suppose that $\phi$ is algebraically stable 
up to birational conjugacy. The eigenvector 
$\Theta(\phi)$ thus corresponds to the 
eigenvalue $\lambda(\phi)$, and so may be taken 
in $\mathrm{NS}(L,S)$ where $L$ is the 
splitting field of~$P$. The automorphism
$\kappa$ acts on $\mathrm{NS}(\mathbb{C},S)$
preserving $\mathrm{NS}(S)$ pointwise.
Since $\phi^*$ is defined over $\mathbb{Z}$
and $\phi^*\Theta(\phi)=\lambda(\phi)\Theta(\phi)$
one obtains 
\[
\phi^*(\kappa(\Theta(\phi))=\kappa(\lambda(\phi))\kappa(\Theta(\phi))=\alpha\kappa(\Theta(\phi))
\]
that is 
$\phi^*\widetilde{\Theta}=\alpha\widetilde{\Theta}$
where $\widetilde{\Theta}=\kappa(\Theta(\phi))$. 
The divisor classes of the $E_j$'s belong to 
$\mathrm{NS}(S)$, so they are $\kappa$-invariant.
As a consequence (\ref{eq:square}) implies
\begin{equation}\label{eq:doublesquare}
\phi_*\phi^*\widetilde{\Theta}=\widetilde{\Theta}+\displaystyle\sum_{j=1}^m(\widetilde{\Theta}\cdot E_j)E_j
\end{equation}
Denote by $\overline{\widetilde{\Theta}}$ 
the conjugate of $\widetilde{\Theta}$ 
and by $\overline{\alpha}$ the conjugate
of $\alpha$; from (\ref{eq:doublesquare})
one gets
\[
(\alpha\overline{\alpha})\widetilde{\theta}\cdot\overline{\widetilde{\theta}}=\phi^*\widetilde{\Theta}\cdot\phi^*\overline{\widetilde{\Theta}}
\]
As $\vert\alpha\vert=\alpha\overline{\alpha}=1$
one gets that $E_j\cdot\widetilde{\Theta}=0$ 
for any $1\leq j\leq m$ and $E_j\cdot\Theta(\phi)=0$ 
for any $1\leq j\leq m$.

Theorem \ref{thm:fleur} follows from Theorem
\ref{thm:diamond}.

\begin{rem}
Theorem \ref{thm:fleur} does not extend to 
quadratic integers (\emph{see} \cite{BlancCantat}).
\end{rem}

\subsubsection{A third answer}

As we have seen in \S \ref{sec:firstdef} 
if $S$ is a projective
smooth surface, then every $\phi\in\mathrm{Bir}(S)$
admits a minimal resolution, {\it i.e.} 
there exist $\pi_1\colon Z\to S$, 
$\pi_2\colon Z\to S$ two sequences of blow ups
such that
\begin{itemize}
\item[$\diamond$] no $(-1)$-curves of $Z$ is 
contracted by both $\pi_1$ and $\pi_2$;

\item[$\diamond$] $\phi=\pi_2\circ\pi_1^{-1}$.
\end{itemize}

Denote by 
$\mathfrak{b}(\phi)$\index{not}{$\mathfrak{b}(\phi)$}
the number of base points of $\phi$; note that 
$\mathfrak{b}(\phi)$ is equal to the difference
of the ranks of $\mathrm{Pic}(Z)$ and 
$\mathrm{Pic}(S)$; thus $\mathfrak{b}(\phi)$ is equal to 
$\mathfrak{b}(\phi^{-1})$. Let us introduce the 
\textsl{dynamical number of the base-points of 
$\phi$}\index{defi}{dynamical number of the base-points}:
it is\index{not}{$\mu(\phi)$}
\[
\mu(\phi)=\displaystyle\lim_{k\to +\infty}\frac{\mathfrak{b}(\phi^k)}{k}
\]
Since 
$\mathfrak{b}(\phi\circ\psi)\leq\mathfrak{b}(\phi)+\mathfrak{b}(\psi)$ 
for any $\phi$, $\psi$ in $\mathrm{Bir}(S)$, 
$\mu(\phi)$ is a non-negative real number.
As $\mathfrak{b}(\phi)=\mathfrak{b}(\phi^{-1})$
one gets $\mu(\phi^k)=\vert k\mu(\phi)\vert$ 
for any $k\in\mathbb{Z}$. Furthermore if 
$\psi\colon S\dashrightarrow Z$ is a birational
map between smooth projective surfaces and
if $\phi\in\mathrm{Bir}(S)$, then for all 
$n\in\mathbb{Z}$
\[
-2\mathfrak{b}(\psi)+\mathfrak{b}(\phi^n)\leq\mathfrak{b}(\psi\circ\varphi^n\circ\psi^{-1})\leq 2\mathfrak{b}(\psi)+\mathfrak{b}(\phi^n);
\]
hence $\mu(\phi)=\mu(\psi\circ\phi\circ\psi^{-1})$. 
One can thus state the following 
result:

\begin{lem}[\cite{BlancDeserti:degree}]\label{lem:losange}
The dynamical number of base-points is an 
invariant of conjugation. In particular if 
$\phi$ is conjugate to an automorphism of
a smooth projective surface, then 
$\mu(\phi)=0$.
\end{lem}

A base-point $p$ of $\phi$ is a 
\textsl{persistent base-point}
\index{defi}{persistent base-point}
if there exists an integer $N$ such 
that for any $k\geq N$
\[
\left\{
\begin{array}{ll}
p\in\mathrm{Base}(\phi^k)\\
p\not\in\mathrm{Base}(\phi^{-k})
\end{array}
\right.
\]
Let $p$ be a point of $S$ or a 
point infinitely near $S$ such that 
$p\not\in\mathrm{Base}(\phi)$. Consider
a minimal resolution of $\phi$
\[
 \xymatrix{
     & Z\ar[rd]^{\pi_1}\ar[ld]_{\pi_2} & \\
    S\ar@{-->}[rr]_\phi & & S
  }
\]
Because $p$ is not a base-point of $\phi$
it corresponds via $\pi_1$ to a point 
of $Z$ or infinitely near; using $\pi_2$
we view this point on $S$ again maybe 
infinitely near and denote 
it~$\phi^\bullet(p)$\index{not}{$\phi^\bullet$}.
For instance if $S=\mathbb{P}^2_\mathbb{C}$, 
$p=(1:0:0)$ and $\phi$ is the birational 
self map of~$\mathbb{P}^2_\mathbb{C}$ given by
\[
(z_0:z_1:z_2)\dashrightarrow(z_1z_2+z_0^2:z_0z_2:z_2^2)
\]
the point $\phi^\bullet(p)$ is not equal
to $p=\phi(p)$ but is infinitely near to 
it. Note that if $\phi$, $\psi$ are 
two birational self maps of $S$ and 
$p$ is a point of $S$ such that 
$p\not\in\mathrm{Base}(\phi)$, 
$\phi(p)\not\in\mathrm{Base}(\psi)$, 
then 
$(\psi\circ\phi)^\bullet(p)=\psi^\bullet(\phi^\bullet(p))$. One can put an equivalence relation on 
the set of points of $S$ or infinitely 
near $S$: the point $p$ is 
\textsl{equivalent}\index{defi}{equivalent (point)}
to the point~$q$ if there exists an 
integer $k$ such that $(\phi^k)^\bullet(p)=q$;
in particular $p\not\in\mathrm{Base}(\phi^k)$
and $q\not\in\mathrm{Base}(\phi^{-k})$. 
Note that the equivalence class is the 
generalization of set of orbits for 
birational maps.

A base-point is \textsl{periodic}\index{defi}{periodic base-point} if 
\begin{itemize}
\item either $(\phi^k)^\bullet(q)=q$ for some $k\geq 0$,

\item or $q\in\mathrm{Base}(\phi^k)$ for any $k\in\mathbb{Z}\smallsetminus\{0\}$ 
(in particular $(\phi^k)^\bullet(p)$ is 
never defined for $k\not=0$).
\end{itemize}
Let $\mathcal{P}$ be the set of periodic 
base-points of $\phi$. Denote by 
$\widehat{\mathcal{P}}$ the finite set of 
points equivalent to a point of $\mathcal{P}$.
Both $\mathfrak{b}(\phi)$ and 
$\mathfrak{b}(\phi^{-1})$ are finite, so there 
exists $n\in\mathbb{N}$ such that for 
any $p\in\mathrm{Base}(\phi)$ non periodic
and for any $j$, $\ell\geq N$
\[
\left\{
\begin{array}{ll}
 p\in\mathrm{Base}(\phi^j)\quad\Longleftrightarrow \quad p\in\mathrm{Base}(\phi^\ell)\\
  p\in\mathrm{Base}(\phi^{-j})\quad\Longleftrightarrow \quad p\in\mathrm{Base}(\phi^{-\ell}) 
\end{array}
\right.
\]

Let us decompose $\mathrm{Base}(\phi)$ into 
five disjoint sets:
\begin{eqnarray*}
& & \mathcal{B}_{++}=\big\{p\,\vert\,p\not\in\mathcal{P},\,p\in\mathrm{Base}(\phi^j),\,p\in\mathrm{Base}(\phi^{-j})\quad\forall\,j\geq N\big\}\\
& & \mathcal{B}_{+-}=\big\{p\,\vert\,p\not\in\mathcal{P},\,p\in\mathrm{Base}(\phi^j),\,p\not\in\mathrm{Base}(\phi^{-j})\quad\forall\,j\geq N\big\}\\
& & \mathcal{B}_{-+}=\big\{p\,\vert\,p\not\in\mathcal{P},\,p\not\in\mathrm{Base}(\phi^j),\,p\in\mathrm{Base}(\phi^{-j})\quad\forall\,j\geq N\big\}\\
& & \mathcal{B}_{--}=\big\{p\,\vert\,p\not\in\mathcal{P},\,p\not\in\mathrm{Base}(\phi^j),\,p\not\in\mathrm{Base}(\phi^{-j})\quad\forall\,j\geq N\big\}
\end{eqnarray*}
and $\mathcal{P}$.

\begin{rems}
Note that:
\begin{itemize}
\item[$\diamond$] $\mathcal{B}_{+-}$ is the set of 
persistent base-points of $\phi$;

\item[$\diamond$] $\mathcal{B}_{-+}$ is the set of 
persistent base-points of $\phi^{-1}$;

\item[$\diamond$] two equivalent base-points of 
$\phi$ belong to the same subsets of 
$\mathrm{Base}(\phi)$.
\end{itemize}
\end{rems}

Take $k\geq 2N$ an integer. Let us compute 
$\mathfrak{b}(\phi^k)$. Any base-point of 
$\phi^k$ is equivalent to a base-point of 
$\phi$. Let us thus consider a base-point 
$p$ of $\phi$ and determine the number 
$m_{p,k}$ of base-points of $\phi^k$ which
are equivalent to $p$.

\begin{enumerate}
\item[a)] If $p$ belongs to $\mathcal{P}$, 
then the number of points equivalent to $p$
is less than $\#\mathcal{P}$ 
and $m_{p,k}\leq\#\mathcal{P}$.

\item[b)] If $p$ does not belong to $\mathcal{P}$, 
then any point equivalent to $p$ is equal
to $(\phi^i)^\bullet(p)$ for some $i$; 
furthermore these points all are distinct.
Hence $m_{p,k}=\# I_{p,k}$ where 
\[
I_{p,k}=\big\{i\in\mathbb{Z}\,\vert\,p\not\in\mathrm{Base}(\phi^i),\,p\in\mathrm{Base}(\phi^{i+k})\big\}.
\]
\begin{enumerate}
\item[b)i)] Suppose that $p$ belongs to 
$\mathcal{B}_{++}$. Since $p$ does not 
belong to $\mathrm{Base}(\phi^i)$, the 
following inequalities hold: $-N<i<N$,
and so $m_{p,k}<2N$.

\item[b)ii)] If $p$ belongs to $\mathcal{B}_{--}$, 
then $p$ belongs to $\mathrm{Base}(\phi^{i+k})$
hence $-N<i+k<N$ and $m_{p,k}<2N$.

\item[b)iii)] Assume that $p$ belongs 
to $\mathcal{B}_{-+}$. As 
$p\not\in\mathrm{Base}(\phi^i)$ 
(resp. $p\in\mathrm{Base}(\phi^{i+k})$), 
one has $-N<i$ (resp. $i+k\leq N$). 
These two conditions imply $-N<i\leq N-k$.
But $k>2N$, so $m_{p,k}=0$.

\item[b)iv)] Finally consider a point 
$p$ in $\mathcal{B}_{+-}$. The fact that
$p\not\in\mathrm{Base}(\phi^i)$ 
(resp. $p\in\mathrm{Base}(\phi^{i+k})$)
yields $i<N$ (resp. $-N<i+k$). As a 
result $-N-k<i<N$ and $m_{p,k}\leq 2N+k$.
Conversely if $i\leq -N$ and $i+k\geq N$, 
then $p\not\in\mathrm{Base}(\phi^i)$ 
and $p\in\mathrm{Base}(\phi^{i+k})$, 
{\it i.e.} $i\in I_{p,k}$. As a
consequence $m_{p,k}\geq \#[N-k,-N]=k-2N+1$.
Finally 
\[
-2N\leq m_{p,k}-k\leq 2N.
\]
\end{enumerate}
\end{enumerate}

Consequently there exist
two constants $\alpha$, $\beta$ (independent
on $k$) such that for all $k\geq 2N$
\[
\nu k+\alpha\leq \mathfrak{b}(\phi^k)\leq\nu k+\beta
\]
where $\nu$ is the number of 
equivalence classes of persistent base-points
of $\phi$ (recall that $\mathcal{B}_{+-}$
is the set of persistent base-points of 
$\phi$). But 
$\mu(\phi)=\displaystyle\lim_{k\to +\infty}\frac{\mathfrak{b}(\phi^k)}{k}$,
so $\mu(\phi)=\nu$. One can thus state:

\begin{pro}[\cite{BlancDeserti:degree}]
Let $S$ be a smooth projective surface.
Let $\phi$ be a birational self map 
of $S$.

Then $\mu(\phi)$ coincides with the 
number of equivalence classes of 
persistent base-points of $\phi$. 
In particular $\mu(\phi)$ is an 
integer.
\end{pro}

The following statement gives another
characterization of birational maps 
which are conjugate to an automorphism
of a projective surface; contrary to 
the two previous one it works for all
maps of $\mathrm{Bir}(S)$.

\begin{thm}[\cite{BlancDeserti:degree}]\label{thm:doublelosange}
Let $\phi$ be a birational self map of a smooth 
projective surface. Then $\phi$ is conjugate
to an automorphism of a smooth projective 
surface if and only if $\mu(\phi)=0$. 
\end{thm}

\begin{rem}
This characterization was implicitely 
used in \cite{BedfordKim1, BedfordKim2, 
BedfordKim3, DesertiGrivaux}.

Let us give an example of \cite{DesertiGrivaux}. 
Consider the birational self map
of $\mathbb{P}^2_\mathbb{C}$ given by 
\[
\psi\colon (z_0:z_1:z_2)\dashrightarrow(z_0z_2^2+z_1^3:z_1z_2^2:z_2^3);
\]
it has five base-points: $p=(1:0:0)$ and 
four points infinitely near. Denote by 
$\widehat{P_1}$ the collection of these 
points. Similarly $\psi^{-1}$ has five
base-points: $(1:0:0)$ and four points
infinitely near; let $\widehat{P_2}$
be the collection of these points. 
Consider the automorphism $A$ given by
\begin{small}
\[
A\colon(z_0:z_1:z_2)\mapsto\big(\alpha z_0+2(1-\alpha)z_1+(2+\alpha-\alpha^2)z_2:-z_0+(\alpha+1)z_2:z_0-2z_1+(1-\alpha)z_2\big)
\]
\end{small}
with $\alpha\in\mathbb{C}\smallsetminus\{0,\,1\}$. 
Then
\begin{itemize}
\item[$\diamond$] $\widehat{P_1}$, 
$A(\widehat{P_2})$, and 
$(A\circ\psi\circ A)(\widehat{P_2})$
have distinct supports; 

\item[$\diamond$] 
$\widehat{P_1}=(A\circ\psi)^2\circ A(\widehat{P_2})$.
\end{itemize}
As a result the base-points of 
$\phi=A\circ\psi$ are non-persistent,
so $\phi$ is conjugate to an automorphism
of a rational surface; this rational 
surface is $\mathbb{P}^2_\mathbb{C}$
blown up in $\widehat{P_1}$, 
$A(\widehat{P_2})$, and 
$(A\circ\psi\circ A)(\widehat{P_2})$.
Furthermore $\lambda(A\circ\psi)>1$. 
\end{rem}

\begin{proof}[Proof of Theorem \ref{thm:doublelosange}]
Lemma \ref{lem:losange} shows that if 
$\phi$ is conjugate to an automorphism 
of a smooth projective surface, then 
$\mu(\phi)=0$. 

Let us prove the converse. Assume 
that $\mu(\phi)=0$. One can 
suppose that by blowing-up points 
$\phi$ is algebraically stable 
(Proposition \ref{pro:DillerFavre}).
Therefore, $\phi$ has no periodic 
base point and $\mathcal{B}_{++}=\emptyset$.
Furthermore $\mu(\phi)=0$ corresponds to
$\mathcal{B}_{+-}=\mathcal{B}_{-+}=\emptyset$.
All base-points thus belong to 
$\mathcal{B}_{--}$. Assume that $\phi$ is
not an automorphism of $S$. Let 
$\tau\colon Z\to S$ be the blow-up of the 
base-points of $\phi$. The morphism
$\chi=\phi\circ\tau\colon Z\to S$ is 
the blow-up of the base-points of $\phi^{-1}$.
Consider a $(-1)$-curve $E\subset Z$
contracted by $\chi$. The image $\chi(E)$
of $E$ is a proper point of $S$ that
belongs to $\mathrm{Base}(\phi^{-1})$. 
Since $\phi$ is algebraically stable, 
then for all $k\geq 0$
\[
\chi(E)\not\in\mathrm{Base}(\phi^k).
\]
As a result 
$\phi^k\circ\chi\colon Z\dashrightarrow S$
is well-defined at any point of $E$. The 
curve $C=\tau(E)$ is thus an irreducible
curve of $S$ contracted by $\phi^{k+1}$; 
any base-point of $\phi^{k+1}$ that 
belongs to $C$ as proper of infinitely 
near point is also a base-point of $\phi$.
This finite set of points is contained
in $\mathcal{B}_{--}$; so there is $n>0$
such that no base-point of $\phi^n$
belongs to $C$. Since $C$ is blown 
down by $\phi^n$, $C$ is a $(-1)$-curve
of $S$. Contracting $C$ conjugates $\phi$
to an algebraically stable birational 
map whose all base-points are in 
$\mathcal{B}_{--}$. The rank of the 
Picard group of this new surface
is strictly less than the rank of 
$\mathrm{Pic}(S)$. Consequently if we repeat 
this process, it has to stop. In other
words $\phi$ is conjugate to an automorphism
of a smooth projective surface.
\end{proof}


\section{Constructions of automorphisms with positive entropy}

\subsection{McMullen's idea}

In \cite{McMullen} McMullen establishes a result 
similar to Torelli's  theorem for K$3$ surfaces: 
he constructs automorphisms on some rational 
surfaces prescribing the action of the automorphisms
on cohomological groups of the surface.

The relationship between the Coxeter group
and the birational geometry of the plane, used
by McMullen, is discussed since $1895$ (\emph{see}
\cite{Kantor}) and has been much developed since
then (\emph{see for instance} 
\cite{Coble, DolgachevOrtland, DolgachevZhang, Harbourne2, Gizatullin}).

A rational surface $S$ is a 
\textsl{marked blow-up}\index{defi}{marked blow-up} of
$\mathbb{P}^2_\mathbb{C}$ if it is presented as a 
blow-up $\pi\colon S\to\mathbb{P}^2_\mathbb{C}$ of 
$\mathbb{P}^2_\mathbb{C}$ at $n$ distinct points
$p_1$, $p_2$, $\ldots$, $p_n$. The marking 
determines the basis for $\mathrm{Pic}(S)$
given by the hyperplane bundle and the classes 
of the exceptional curves over the $p_j$. The 
first step toward finding an automorphism $\phi$
of $S$ is to construct a plausible candidate for 
its linear action $\phi^*$ on the Picard group. 
Note that candidate actions must preserve the 
intersection form, the class of the canonical 
divisor, and the set of effective classes. Let 
us mention two sorts of involutions on 
$\mathrm{Pic}(S)$ that satisfy these restrictions:
\begin{itemize}
\item[$\diamond$] an abstraction of the involution 
$\sigma_2$,

\item[$\diamond$] the involution that swaps the 
basis elements corresponding to two different
exceptional curves.
\end{itemize}

If we compose such involutions one gets a Coxeter
group $\mathrm{W}_n$
\begin{itemize}
\item[$\diamond$] that is infinite as soon as $n\geq 9$,

\item[$\diamond$] has elements with positive spectral
radius when $n\geq 10$.
\end{itemize}

Furthermore except in some degenerate situations 
an element $w\in\mathrm{W}_n$ transforms the 
basis of $\mathrm{Pic}(S)$ corresponding to 
the given marking into a basis corresponding  
to some other marking 
$\varphi'\colon S\to\mathbb{P}^2_\mathbb{C}$. If 
the base-points of the new marking coincide, up 
to an element of $\mathrm{Aut}(\mathbb{P}^2_\mathbb{C})$,
with those of the original, then one obtains 
an automorphism $\phi=\varphi^{-1}\circ\varphi'$
of $S$ with $\phi^*=w$. The main problem with 
this approach is that it is not easy, given 
$w\in\mathrm{W}_n$, to see how the base-points
of the two markings are related. The problem 
is easier if the base-points of the original 
marking lie along an elliptic curve; indeed 
in that case the new base-points also lie 
on this elliptic curve. Computations are thus 
computations on a curve so simpler. The 
best case is the case of a cuspidal cubic 
as there is a one-parameter subgroup of 
$\mathrm{Aut}(\mathbb{P}^2_\mathbb{C})$
fixing such a curve. McMullen proved 

\begin{thm}[\cite{McMullen}]
For any $n\geq 10$ the standard element $w$ of 
$\mathrm{W}_n$ may be realized by an automorphism 
$\phi$ of a marked blow-up $S$ with an invariant
cuspidal anticanonical curve. The entropy 
of $\phi$ is the spectral radius of $w$ which 
is positive.
\end{thm}

\medskip

The following question "What are the elements
of $w\in W_n$ which may be realized by rational
surface automorphisms ?" was also considered
in \cite{Diller} and \cite{Uehara}. Diller
gave a rather thorough enumeration of the 
possibilities for quadratic birational maps
which have an invariant curve. Such maps 
are determined by the data consisting of 
three orbit lengths $(n_1,n_2,n_3)$ and 
a permutation of $\{1,\,2,\,3\}$. Diller 
also showed that not all orbit data, and 
not all $w\in\mathrm{W}_n$, are realizable 
by maps with invariant curve. Uehara 
established the following statement:

\begin{thm}[\cite{Uehara}]
For every $w\in\mathrm{W}_n$ with spectral 
radius $>1$ there is a rational surface
automorphism $\phi$ such that the spectral 
radius of $\phi^*$ is the same as the spectral
radius of $w$.
\end{thm}

Uehara's method combines elements of 
McMullen's and Diller's approaches. Given 
$w\in\mathrm{W}_n$ he prescribed a set of 
orbit data and proved that these orbit
data can be realized by an automorphism
$\phi$. The induced $\phi^*$ has the same
spectral radius as $w$, although the 
two may not be conjugate. 

\begin{rem}
While McMullen's and Diller's constructions
involve automorphisms with invariant curves
note that in \cite{BedfordKim1} the 
authors showed that rational surface 
automorphisms of positive entropy do not
necessarily possess invariant curves.
\end{rem}

\subsection{Bedford and Kim construction}

In \cite{BedfordKim4} and \cite{BedfordKim1} 
the authors found automorphisms within a 
specific two-parameters family of plane 
birational maps. The initial observation 
in the two papers is the same: for certain 
parameter pairs all points of indeterminacy 
for all iterates of the map in question can 
be eliminated by performing finitely many
point blow-ups. The map then lifts to an 
automorphism of the resulting rational 
surface. This idea was "systematized" in 
\cite{DesertiGrivaux}. 

In \cite{BedfordKim1} the authors prove that 
essentially all examples of rational surfaces
automorphisms associated to Coxeter elements 
can be found within the two-parameter 
birational family $(f_{a,b})_{(a,b)}$ given 
by $f_{a,b}(z_0,z_1)=\left(z_1,\frac{z_1+a}{z_0+b}\right)$.


\section{Automorphisms are pervasive}

\subsection{Automorphisms of del Pezzo surfaces}

Any del Pezzo surface $S$ contains a finite number of $(-1)$-curves 
(\emph{i.e.} smooth curves isomorphic to $\mathbb{P}^1_\mathbb{C}$ and of 
self-intersection $-1$). Each of them can be contracted to get another 
del Pezzo surface of degree $(K_S)^2+1$. There are, moreover, the only 
reducible curves of $S$ of negative self-intersection. If 
$S\not=\mathbb{P}^2_\mathbb{C}$, then there is a finite number of conic 
bundles $S\to\mathbb{P}^1_\mathbb{C}$ up to automorphism of 
$\mathbb{P}^1_\mathbb{C}$ and each of them has exactly $8-(K_S)^2$ 
singular fibers.

This latter fact can be found by contracting one component in each singular
fiber which is the union of two $(-1)$-curves, obtaining a line bundle 
on a del Pezzo surface, isomorphic 
to~$\mathbb{P}^1_\mathbb{C}\times\mathbb{P}^1_\mathbb{C}$ or to the 
Hirzebruch surface $\mathbb{F}_1$ and having degree $8$.

For more details see \cite{Demazure:sousgroupesalgebriques, Manin}.

\subsubsection*{Automorphisms of del Pezzo surfaces of order $4$}

Set
\[
S=\big\{(z_0:z_1:z_2:z_3)\in\mathbb{P}(2,1,1,1)\,\vert\,z_0^2-z_1^4=z_2z_3(z_2+z_3)(z_2+\mu z_3)\big\}
\]
where $\mu$ belongs to $\mathbb{C}\smallsetminus\{0,\,1\}$. The surface
$S$ is a del Pezzo one of degree $2$. The automorphism $\beta$ given by
\[
\beta\colon(z_0:z_1:z_2:z_3)\mapsto(z_0:\mathbf{i}z_1:z_2:z_3)
\]
fixes pointwise the elliptic curve given by $z_0=0$. When $\mu$ varies all
possible elliptic curves are obtained. Moreover 
$\mathrm{rk}\,\mathrm{Pic}(S)^{\beta}=1$. 

There are other automorphisms $\beta$ of order $4$ of rational surfaces $S$
such that $\beta^2$ fixes an elliptic curve but none for which 
$\mathrm{rk}\,\mathrm{Pic}(S)^{\beta}=1$ (\emph{see} \cite{Blanc:cyclic}).

\subsubsection*{Automorphisms of del Pezzo surfaces of order $6$}

Let us give explicit possibilities for  automorphisms of order $6$.

\begin{itemize}
\item[i)] Set 
\[
S=\big\{(z_0:z_1:z_2:z_3)\in\mathbb{P}(3,1,1,2)\,\vert\,z_0^2=z_3^3+\mu z_1^4z_3+z_1^6+z_2^6\big\}
\]
for some general $\mu\in\mathbb{C}$ such that $S$ is smooth. The 
surface $S$ is a del Pezzo surface of degree $1$. Consider on $S$
\[
\alpha\colon(z_0:z_1:z_2:z_3)\mapsto(z_0:z_1:-\mathbf{j}z_2:z_3)
\]
where $\mathbf{j}=\mathrm{e}^{2\mathbf{i}\pi/3}$.

The automorphism $\alpha$ fixes pointwise the elliptic curve given 
by $z_2=~0$. When $\mu$ varies all possible elliptic curves are obtained. 
The equality $\mathrm{rk}\,\mathrm{Pic}(S)^\alpha=1$ holds (\emph{see}
\cite[Corollary 6.11]{DolgachevIskovskikh}).

\item[ii)] Set 
\[
S=\big\{(z_0:z_1:z_2:z_3)\in\mathbb{P}^3_\mathbb{C}\,\vert\,z_0z_1^2+z_0^3+z_2^3+z_3^3+\mu z_0z_2z_3=0\big\}
\]
where $\mu$ is such that the cubic surface is smooth. The surface is a
del Pezzo surface of degree $3$. Consider on $S$ the automorphism $\alpha$
given by 
\[
\alpha\colon(z_0:z_1:z_2:z_3)\mapsto(z_0:-z_1:\mathbf{j}z_2:\mathbf{j}^2z_3). 
\]
Remark that
$\alpha^3$ fixes pointwise the elliptic curve $z_1=0$ and $\alpha$ acts on
it via a translation of order $3$. When $\mu$ varies all possible elliptic
curves are obtained. The equality $\mathrm{rk}\,\mathrm{Pic}(S)^{\alpha}=1$ 
holds (\cite{DolgachevIskovskikh}).

\item[iii)] Set
\[
S=\big\{(z_0:z_1:z_2:z_3)\in\mathbb{P}^3_\mathbb{C}\,\vert\, z_0^3+z_1^3+z_2^3+(z_1+\mu z_2)z_3^2=0\big\}
\]
where $\mu\in\mathbb{C}$ is such that the cubic surface is smooth. It is 
a del Pezzo surface of degree $3$. Consider $\alpha$ defined by
$\alpha\colon(z_0:z_1:z_2:z_3)\mapsto(\mathbf{j}z_0:z_1:z_2:z_3)$. The 
automorphism $\alpha^3$ fixes pointwise the elliptic curve $z_3=0$ and 
$\alpha$ acts on it via an automorphism of order $3$ with three fixed 
points. When $\mu$ varies the birational class of $\alpha$ changes but 
not the isomorphism class of the elliptic curve fixed by $\alpha^3$. 
\end{itemize}

\subsection{Outline of the construction}

\subsubsection[The central involution of $\mathrm{SL}(2,\mathbb{Z})$ and its 
image]{The central involution of $\mathrm{SL}(2,\mathbb{Z})$ and its 
image into $\mathrm{Bir}(\mathbb{P}^2_\mathbb{C})$}

Set $A=\left(
\begin{array}{cc}
1 & 1 \\
0 & 1
\end{array}
\right)$ and 
$B=\left(
\begin{array}{cc}
0 & 1 \\
-1 & 0
\end{array}
\right).$
A presentation of $\mathrm{SL}(2,\mathbb{Z})$ is given by (\emph{see}
\cite{Newman})
\[
\langle A,\, B\,\vert\,B^4=(AB)^3=1,\,B^2(AB)=(AB)B^2\rangle. 
\]
As as result the quotient of $\mathrm{SL}(2,\mathbb{Z})$ by its center is
a free product of $\faktor{\mathbb{Z}}{2\mathbb{Z}}$ 
and~$\faktor{\mathbb{Z}}{3\mathbb{Z}}$
generated by the classes $[B]$ of $B$ and $[AB]$ of $AB$
\[
\mathrm{PSL}(2,\mathbb{Z})=\langle [B],\, [AB]\,\vert\, [B]^2=
[AB]^3=\mathrm{id}\rangle.
\]
Recall that $\mathrm{SL}(2,\mathbb{R})$ acts on the upper half plane
\[
\mathbb{H}=\big\{x+\mathbf{i}y\in\mathbb{C}\,\vert\, x,\, y\in\mathbb{R},\, y>0\big\}
\]
by M\"obius transformations
\begin{align*}
&\mathrm{SL}(2,\mathbb{R})\times\mathbb{H}\to\mathbb{H}, && 
\left(\left(
\begin{array}{cc}
a & b\\
c & d
\end{array}
\right),z\right)\mapsto\frac{az+b}{cz+d}
\end{align*}
the hyperbolic structure of $\mathbb{H}$ being preserved. This yields 
to a natural notion of elliptic, parabolic and loxodromic elements of 
$\mathrm{SL}(2,\mathbb{R})$. If $M$ belongs to $\mathrm{SL}(2,\mathbb{Z})$
one can be more precise and check the following observations:
\begin{itemize}
\item[$\diamond$] $M$ is elliptic if and only if $M$ has finite order;

\item[$\diamond$] $M$ is parabolic if and only if $M$ has infinite order
and its trace is $\pm 2$;

\item[$\diamond$] $M$ is loxodromic if and only if $M$ has infinite order and 
its trace is $\not=\pm 2$.
\end{itemize}

Up to conjugacy the elliptic elements of $\mathrm{SL}(2,\mathbb{Z})$ are
\begin{align*}
& \left(
\begin{array}{cc}
-1 & 0 \\
0 & -1
\end{array}
\right),
&& \left(
\begin{array}{cc}
0 & 1\\
-1 & -1
\end{array}
\right),
&& \left(
\begin{array}{cc}
0 & 1 \\
-1 & 0
\end{array}
\right),
&& \left(
\begin{array}{cc}
0 & -1\\
1 & 0
\end{array}
\right),
&& \left(
\begin{array}{cc}
0 & -1 \\
1 & 1
\end{array}
\right).
\end{align*}

In particular an element of finite order is of order $2$, $3$, $4$ or $6$. 

A parabolic element of $\mathrm{SL}(2,\mathbb{Z})$ is up to conjugacy one 
of the following one 
\begin{align*}
& \left(
\begin{array}{cc}
1 & a \\
0 & 1
\end{array}
\right) 
&& \left(
\begin{array}{cc}
-1 & a \\
0 & -1
\end{array}
\right)
\end{align*}
with $a\in\mathbb{Z}$.

Since $B^2\in\mathrm{SL}(2,\mathbb{Z})$ is an involution its image by any
embedding 
$\theta\colon\mathrm{SL}(2,\mathbb{Z})\to\mathrm{Bir}(\mathbb{P}^2_\mathbb{C})$
is a birational involution. As we have seen in \S
\ref{sec:ordre2} an element of order $2$ of the Cremona group
is up to conjugacy one of the following
\begin{itemize}
\item[$\diamond$] an automorphism of $\mathbb{P}^2_\mathbb{C}$,

\item[$\diamond$] a Jonqui\`eres involution of degree $\geq 2$,

\item[$\diamond$] a Bertini involution, 

\item[$\diamond$] a Geiser involution.
\end{itemize}

Since $B^2$ commutes with $\mathrm{SL}(2,\mathbb{Z})$ the group 
$\theta\big(\mathrm{SL}(2,\mathbb{Z})\big)$ is contained in the centralizer
of $\theta(B^2)$. But if $\theta(B^2)$ is a Bertini involution or a Geiser involution, 
then the centralizer of $\theta(B^2)$ is finite (\cite{BlancPanVust}).
As a result $\theta(B^2)$ is conjugate either to an automorphism of 
$\mathbb{P}^2_\mathbb{C}$, or to a Jonqui\`eres involution. Assume that 
$\theta(B^2)$ is not linearisable; $\theta(B^2)$ fixes thus pointwise
a unique irreducible curve $\Gamma$ of genus $\geq 1$. Denote by 
$\mathrm{G}$ the image of $\theta$. The group $\mathrm{G}$ 
preserves~$\Gamma$ and the action of $\mathrm{G}$ on $\Gamma$ gives the exact
sequence
\[
1 \longrightarrow\mathrm{G}'\longrightarrow\mathrm{G}\longrightarrow
\mathrm{H}\longrightarrow 1
\]
where $\mathrm{H}$ is a subgroup of $\mathrm{Aut}(\Gamma)$ and 
$\mathrm{G}'$ contains $\theta(B^2)$ and fixes $\Gamma$. The genus 
of~$\Gamma$ is positive; hence $\mathrm{H}$ cannot coincide with 
$\faktor{\mathrm{G}}{\langle\theta(B^2)\rangle}$, a free product of 
$\faktor{\mathbb{Z}}{2\mathbb{Z}}$ and~$\faktor{\mathbb{Z}}{3\mathbb{Z}}$. As a consequence
$\mathrm{G}'\triangleleft\mathrm{G}$ strictly contains 
$\langle\theta(B^2)\rangle$; thus $\mathrm{G}'$ is infinite and not 
abelian. In particular the group of birational maps fixing pointwise 
$\Gamma$ is infinite and not abelian. So according to 
\cite{BlancPanVust2} the curve $\Gamma$ has genus $1$. One can now state:

\begin{lem}[\cite{BlancDeserti:embeddings}]
Let $\theta$ be an embedding of $\mathrm{SL}(2,\mathbb{Z})$ into the 
plane Cremona group. Then up to birational conjugacy
\begin{itemize}
\item[$\diamond$] either $\theta(B^2)$ is an automorphism of 
$\mathbb{P}^2_\mathbb{C}$ of order $2$, 

\item[$\diamond$] or $\theta(B^2)$ is a Jonqui\`eres involution of degree $3$
fixing $($pointwise$)$ an elliptic curve.
\end{itemize}
\end{lem}

\subsubsection{Existence of infinitely many loxodromic embeddings
of $\mathrm{SL}(2,\mathbb{Z})$ into $\mathrm{Bir}(\mathbb{P}^2_\mathbb{C})$}

Let us consider the standard embedding 
\begin{align*}
& \theta_e\colon\mathrm{SL}(2,\mathbb{Z})\to\mathrm{Bir}(\mathbb{P}^2_\mathbb{C}) && \left(
\begin{array}{cc}
a & b \\
c & d
\end{array}
\right)\mapsto \Big((z_0:z_1:z_2)\mapsto(az_0+bz_1:cz_0+dz_1:z_2)\Big).
\end{align*}
Note that $\theta_e\big(\mathrm{SL}(2,\mathbb{Z})\big)$ is a subgroup of
$\mathrm{PGL}(3,\mathbb{C})$ that preserves the line $L_{z_2}$ of equation 
$z_2=0$ and acts on it via the maps
\[
\mathrm{SL}(2,\mathbb{Z})\to\mathrm{PSL}(2,\mathbb{Z})\subset\mathrm{PSL}(2,\mathbb{C})=\mathrm{Aut}(L_{z_2}).
\]
Pick $\mu\in\mathbb{C}^*$ such that the point $p=(\mu:1:0)\in L_{z_2}$ has a 
trivial isotropy group under the action of $\mathrm{PSL}(2,\mathbb{Z})$. 
Fix an even integer $k>0$; consider $\psi$ the conjugation of 
\[
\psi'\colon(z_0:z_1:z_2)\dashrightarrow(z_0^k:z_0^{k-1}z_1+z_2^k:z_0^{k-1}z_2)
\]
by $(z_0:z_1:z_2)\mapsto(z_0+\mu z_1:z_1:z_2)$. 
Then define the morphism 
$\theta_k\colon\mathrm{SL}(2,\mathbb{Z})\to\mathrm{Bir}(\mathbb{P}^2_\mathbb{C})$
as follows
\begin{align*}
&\theta_k(B)=\theta_e(B)\colon(z_0:z_1:z_2)\mapsto(z_1:-z_0:z_2)
&&\theta_k(AB)=\psi\circ\theta_e(AB)\circ\psi^{-1}.
\end{align*}
The map $\psi'$ restricts to an automorphism of the affine plane where 
$z_0\not=0$, commutes with 
$\theta_k(B^2)=\theta_e(B^2)=(z_0:z_1:-z_2)\in\mathrm{Aut}(\mathbb{P}^2_\mathbb{C})$
and acts trivially on~$L_{z_2}$. Since $\psi$ commutes with $\theta_k(B^2)$ 
the map $\theta_k(AB)$ commutes with $\theta_k(B^2)$. As a result $\theta_k$ 
is a well-defined morphism. As $\psi_{\vert L_{z_2}\smallsetminus\{p\}}=\mathrm{id}$
the actions of $\theta_e$ and $\theta_k$ on $L_{z_2}$ are the same; $\theta_k$ is
thus an embedding.

\begin{lem}[\cite{BlancDeserti:embeddings}]\label{lem:dynamicaldegree}
Let $n$ be a positive integer. Let $a_1$, $\ldots$, $a_n$, $b_1$, $\ldots$, 
$b_n$ be~$2n$ elements in $\{-1,\,1\}$. The birational self map of $\mathbb{P}^2_\mathbb{C}$
\[
\theta_k\big(B^{b_n}(AB)^{a_n}B^{b_{n-1}}(AB)^{a_{n-1}}\ldots B^{b_1}(AB)^{a_1}\big)
\]
has degree $k^{2n}$ and has exactly $2n$ proper base-points, all lying on $L_{z_2}$.

More precisely the base-points are
\[
p,\,\big((AB)^{a_1}\big)^{-1}(p),\,\big(B^{b_1}(AB)^{a_1}\big)^{-1}(p),
\]
\[
\big((AB)^{a_2}B^{b_1}(AB)^{a_1}\big)^{-1}(p),\,\ldots,\,\big((AB)^{a_n}B^{b_{n-1}}(AB)^{a_{n-1}}\ldots B^{b_1}(AB)^{a_1}\big)^{-1}(p),
\]
\[
\big(B^{b_n}(AB)^{a_n}B^{b_{n-1}}(AB)^{a_{n-1}}\ldots B^{b_1}(AB)^{a_1}\big)^{-1}(p).
\]
\end{lem}

This result implies the existence of infinitely many loxodromic embeddings
of $\mathrm{SL}(2,\mathbb{Z})$ into $\mathrm{Bir}(\mathbb{P}^2_\mathbb{C})$:

\begin{cor}[\cite{BlancDeserti:embeddings}]
Let $n$ be a positive integer. Let $a_1$, $a_2$, $\ldots$, $a_n$, $b_1$, 
$b_2$, $\ldots$, $b_n$ be $2n$ elements in $\{-1,\,1\}$. The birational self map of $\mathbb{P}^2_\mathbb{C}$
\[
\theta_k\big(B^{b_n}(AB)^{a_n}B^{b_{n-1}}(AB)^{a_{n-1}}\ldots B^{b_1}(AB)^{a_1}\big)
\]
has dynamical degree $k^{2n}$.

In particular, $\theta_k$ is a loxodromic embedding and 
\[
\big\{\lambda(\phi)\,\vert\,\phi\in\theta_k\big(\mathrm{SL}(2,\mathbb{Z})\big)\big\}=\big\{1,\,k^2,\,k^4,\,k^6,\,\ldots\big\}.
\]
\end{cor}

\begin{proof}
Let us consider an element of infinite order of $\mathrm{SL}(2,\mathbb{Z})$;
it is conjugate to 
\[
\varphi=B^{b_n}(AB)^{a_n}B^{b_{n-1}}(AB)^{a_{n-1}}\ldots B^{b_1}(AB)^{a_1}
\]
where $a_1$, $a_2$, $\ldots$, $a_n$, $b_1$, $b_2$, $\ldots$, $b_n\in\{-1,\,1\}$.
According to Lemma \ref{lem:dynamicaldegree} the degree of $\theta_k(\varphi^r)$
is equal to $k^{2nr}$. As a consequence 
$\lambda\big(\theta_k(\varphi)\big)=k^{2n}$.
\end{proof}

\begin{proof}[Idea of the proof of Lemma \ref{lem:dynamicaldegree}]
We proceed by induction on $n$. Let us detail the case $n=~1$. The birational
map $\psi$ has degree $k$ and has a unique proper base-point 
$p=(\mu:1:0)\in L_{z_2}$. The same holds for $\psi^{-1}$. Moreover 
$\psi_{\vert L_{z_2}\smallsetminus\{p\}}=\psi^{-1}_{\vert L_{z_2}\smallsetminus\{p\}}=\mathrm{id}$. 
Since $\theta_e(AB)^{a_1}\in\mathrm{Aut}(\mathbb{P}^2_\mathbb{C})$ moves the 
point $p$ onto another point of $L_{z_2}$, the map $\theta_k\big((AB)^{a_1}\big)$
has degree $k^2$ and exactly two proper base-points which are $p$ and 
$\big((AB)^{a_1}\big)^{-1}(p)=(\psi\circ\theta_e)(AB)^{-a_1}$. As $\theta_k(S)$ belongs to 
$\mathrm{Aut}(\mathbb{P}^2_\mathbb{C})$, $\theta_k(B^{b_1}(AB)^{a_1})$ has also
degree $k^2$ and two proper base-points which are $p$ and 
$\big((AB)^{a_1}\big)^{-1}(p)$.
\end{proof}

\subsubsection[Description of loxodromic embeddings]{Description of loxodromic embeddings for which the central element fixes (pointwise) an elliptic curve}

Let us note that 
\[
\mathrm{SL}(2,\mathbb{Z})=\langle\alpha,\,\beta\,\vert\,\beta^4=\mathrm{id},\,\alpha^3=\beta^2\rangle
\] 
(take the presentation we gave before and set $\alpha^2=AB$, $\beta=B$) and
that 
\[
\mathrm{SL}(2,\mathbb{Z})=\langle\alpha,\,\beta\,\vert\,\alpha^6=\beta^4=\alpha^3\beta^2=\mathrm{id}\rangle.
\]
In this section we will use this last presentation.

We say that a curve is \textsl{fixed}\index{defi}{fixed (curve)} by a birational
map if it is pointwise fixed, and say that a curve is 
\textsl{invariant}\index{defi}{invariant (curve)} or 
\textsl{preserved}\index{defi}{preserved (curve)} if the map induces a 
birational action on the curve.

All conjugacy classes of elements of order $4$ and $6$ 
in~$\mathrm{Bir}(\mathbb{P}^2_\mathbb{C})$ have been classified in 
\cite{Blanc:commentarii}. Many of them can act on del Pezzo surfaces 
of degree $1$, $2$, $3$ or $4$. 

del Pezzo surfaces $X$, $Y$
of degree $\leq 4$ and automorphisms $\alpha\in\mathrm{Aut}(X)$, resp. 
$\beta\in\mathrm{Aut}(Y)$ of order $6$, resp. $4$ so that 
\begin{itemize}
\item[$\diamond$] $\alpha^3$ and $\beta^2$ fix pointwise an elliptic curve, 
\item[$\diamond$] and that $\mathrm{Pic}(X)^{\alpha}$,
$\mathrm{Pic}(Y)^{\beta}$ both have rank $1$
\end{itemize}
are defined to create the embedding. Contracting $(-1)$-curves 
invariant by the involutions $\alpha^3$ and $\beta^2$ 
(but not by $\alpha$, $\beta$ which act
minimally on $X$ and $Y$) we get rational morphisms $X\to X_4$ and 
$Y\to Y_4$ where $X_4$, $Y_4$ are del Pezzo surfaces on which 
$\alpha^3$ and $\beta^2$ act minimally. Furthermore $X_4$ and $Y_4$ are
del Pezzo surfaces of degree $4$, both $\mathrm{Pic}(X_4)^{\alpha^3}$ 
and $\mathrm{Pic}(Y_4)^{\beta^2}$ have rank $2$ and are generated by the 
fibers of the two conic bundles on $X_4$ and~$Y_4$. Choosing a birational
map $X_4\dashrightarrow Y_4$ conjugating $\alpha^3$ to $\beta^2$ (which 
exists if and only if the elliptic curves are isomorphic), which is general
enough, we obtain a loxodromic embedding 
\[
\mathrm{SL}(2,\mathbb{Z})\to\mathrm{Bir}(\mathbb{P}^2_\mathbb{C}).
\]
To prove that there is no other relation in $\langle\alpha,\,\beta\rangle$
and that all elements of infinite order are loxodromic the 
morphisms $X\to X_4$ and $Y\to Y_4$ and the actions of $\alpha$ and $\beta$ 
on $\mathrm{Pic}(X)^{\alpha^3}$ and~$\mathrm{Pic}(Y)^{\beta^2}$ are 
described ; furthermore the composition of the elements does what is expected.

\backmatter

\chapter*{Index}

\printindex{defi}{Index}
\printindex{not}{Index notations}

\bibliographystyle{alpha}
\bibliography{biblio}

\nocite{}

\end{document}